\documentclass[11pt, reqno]{amsart}

\providecommand{\MR}{\relax\ifhmode\unskip\space\fi MR }

\providecommand{\href}[2]{#2}

\usepackage{hhline}
\usepackage{textcmds}  
\usepackage{mathtools}
\usepackage{tensor}
\usepackage{amssymb}
\usepackage[centertags]{amsmath}
\usepackage[mathcal]{eucal}
\usepackage{microtype}
\usepackage[modulo]{lineno}
\usepackage[usenames]{color}
\usepackage{aliascnt}
\usepackage{enumitem}
\usepackage{xspace}
\usepackage{bm}
\usepackage{mathrsfs}

\usepackage{tikz}
\usetikzlibrary{matrix}
\usetikzlibrary{arrows}
\usetikzlibrary{calc}
\usetikzlibrary{shapes.geometric}
\usetikzlibrary{decorations.pathmorphing,shapes}
\usetikzlibrary{shadows}

\newsavebox{\twosubbox}

\newcounter{sarrow}
\newcommand\xrsquigarrow[1]{%
\stepcounter{sarrow}%
\begin{tikzpicture}[decoration=snake]
\node (\thesarrow) {\strut#1};
\draw[->,decorate] (\thesarrow.south west) -- (\thesarrow.south east);
\end{tikzpicture}%
}
\usepackage{scalerel}

\usepackage{amsthm}
\usepackage[margin=3cm]{geometry}
\usepackage{dsfont}
\usepackage{bm}
\usepackage{color}
\usepackage{amsmath}
\usepackage{array}
\usepackage[all]{xy}
\usepackage{import}
\usepackage{faktor}

\usepackage[list=true]{subcaption}

\usepackage{nameref}
\usepackage{varioref}

\usepackage[pagebackref,colorlinks=true,citecolor=green!50!black,linkcolor=blue,filecolor=blue,urlcolor=red]{hyperref}

 \newcommand{\raisedlink}[1]{\Hy@raisedlink{\hypertarget{#1}{}}}
\usepackage[backrefs,alphabetic,abbrev]{amsrefs} 
\DeclarePairedDelimiter{\ceil}{\lceil}{\rceil}
\DeclareMathOperator{\codim}{codim}
\DeclareMathOperator{\virdim}{vir.dim}
\DeclareMathOperator{\vircodim}{vir.codim}
\newcommand{\noproof}{\hfill \qed}

\newcommand{\xdasharrow}[2][->]{
\tikz[baseline=-\the\dimexpr\fontdimen22\textfont2\relax]{
\node[anchor=south,font=\scriptsize, inner ysep=1.5pt,outer xsep=2.2pt](x){#2};
\draw[shorten <=3.4pt,shorten >=3.4pt,dashed,#1](x.south west)--(x.south east);
}
}

\newtheorem{thm}{Theorem}[section]

\newtheorem{theorem}[thm]{Theorem}

\newtheorem{proposition}[thm]{Proposition}

\newtheorem{lemma}[thm]{Lemma}

\newtheorem{corollary}[thm]{Corollary}

\newtheorem{claim}[thm]{Claim}

\newtheorem{definition}[thm]{Definition}
\newtheorem{definitionlemma}[thm]{Definition/Lemma}

\newtheorem{example}[thm]{Example}

\newtheorem{remark}[thm]{Remark}

\newtheorem*{thm*}{Theorem}
\newtheorem*{cor*}{Corollary}
\newtheorem*{prop*}{Proposition}

\newtheorem*{morselemma}{The Morse lemma}

\newtheorem*{homologicalperturbationlemma}{The Homological Perturbation Lemma}
\newtheorem{kindoftheorem}[thm]{''Theorem"}
\newtheorem{observation}[thm]{Observation}
\newcommand{\on}{\operatorname}

\newcommand{\z}{\textbf{z}}
\newcommand{\w}{\textbf{w}}
\renewcommand{\u}{\textbf{u}}

\newcommand{\st}{{\rm st}}
\newcommand{\cUU}{\overline{\UU}}
\newcommand{\A}{\textbf{A}}
\renewcommand{\b}{\textbf{b}}

\newcommand{\sheafHom}{\mathscr{H}\text{\kern -3pt {\calligra\large om}}\,}
\newcommand{\C}{\mathbb{C}}

\newcommand{\N}{\mathbb{N}}
\renewcommand{\P}{\mathbb{P}}
\newcommand{\Q}{\mathbb{Q}}
\newcommand{\R}{\mathbb{R}}
\renewcommand{\S}{\mathbb{S}}

\newcommand{\Z}{\mathbb{Z}}

\renewcommand{\AA}{\mathcal{A}}
\newcommand{\BB}{\mathcal{B}}
\newcommand{\CC}{\mathcal{C}}
\newcommand{\DD}{\mathcal{D}}
\newcommand{\EE}{\mathcal{E}}
\newcommand{\FF}{\mathcal{F}}
\newcommand{\GG}{\mathcal{G}}
\newcommand{\II}{\mathcal{I}}
\newcommand{\JJ}{\mathcal{J}}
\newcommand{\HH}{\mathcal{H}}
\newcommand{\LL}{\mathcal{L}}
\newcommand{\cMM}{\overline{\mathcal{M}}}
\newcommand{\MM}{\mathcal{M}}
\newcommand{\NN}{\mathcal{N}}
\newcommand{\KK}{\mathcal{K}}
\newcommand{\OO}{\mathcal{O}}
\newcommand{\TT}{\mathcal{T}}
\newcommand{\UU}{\mathcal{U}}
\newcommand{\PP}{\mathcal{P}}
\newcommand{\QQ}{\mathcal{Q}}
\newcommand{\WW}{\mathcal{W}}
\newcommand{\XX}{\mathcal{X}}
\newcommand{\YY}{\mathcal{Y}}
\newcommand{\Id}{\mathrm{Id}}
\newcommand{\Spec}{\mathrm{Spec}}

\newcommand{\oo}{\mathfrak{o}}

\newcommand{\ul}{\underline}
\newcommand{\orn}{\mathfrak{o}}
\newcommand{\morseLabel}{\vec{\mathfrak{p}}}

\newcommand{\Morse}{\mathcal{M}}
\newcommand{\cMorse}{\overline{\mathcal{M}}}

\def\ev{{\rm ev}}

\newcommand{\Edge}{\on{Edge}}
\renewcommand{\Vert}{\on{Vert}}
\newcommand{\Flag}{\on{Flag}}

\newcommand{\Stashefftree}{\mathcal{T}}
\newcommand{\cStashefftree}{\overline{\Stashefftree}}
\newcommand{\Stasheffdisc}{\mathcal{R}}
\newcommand{\cStasheffdisc}{\overline{\Stasheffdisc}}
\newcommand{\Pearltree}{\mathcal{Q}}
\newcommand{\cPearltree}{\overline{\Pearltree}}

\newcommand{\Label}{\Lambda}

\newcommand{\complexJ}{\mathcal{J}}

\newcommand{\Diff}{\mathrm{Diff}}
\newcommand{\Symp}{\mathrm{Symp}}
\newcommand{\Ham}{\mathrm{Ham}}
\newcommand{\Auteq}{\mathrm{Auteq}}
\newcommand{\Fuk}{\mathrm{Fuk}}

\newcommand{\ModuliofStableCurves}{\overline{\mathcal{M}}}

\newcommand{\familyEmbeddedCurves}{\mathcal{C}} 
\newcommand{\familyExceptionalDivisors}{\mathcal{E}}
\newcommand{\familyFanoThreeFolds}{\mathcal{Y}}

\newcommand{\NormalBundle}{\mathcal{N}}

\newcommand{\bdm}{\begin{displaymath}}
\newcommand{\edm}{\end{displaymath}}

\newcommand{\bq}{\begin{equation}}
\newcommand{\eq}{\end{equation}}

\numberwithin{equation}{section}

\newcommand{\LHF}{(E \stackrel{\pi}{\rightarrow} B,\Omega)}

\newcommand{\iso}{\cong}           
\newcommand{\CP}[1]{\C {\mathrm P}^{#1}}

\newcommand{\im}{\mathrm{im}}
\renewcommand{\ker}{\mathrm{ker}}

\newcommand{\crit}{\mathfrak{crit}}

\newcommand{\bigslant}[2]{{\raisebox{.2em}{$#1$}\left/\raisebox{-.2em}{$#2$}\right.}}
\def\mc#1{\mathcal{#1}}

\newcommand{\delbar}{\overline{\partial}}

\newcommand{\field}{\mathbf{k}}

\newcommand{\ainf}{A_\infty}

\definecolor{Noir}{RGB}{0,0,0}
\definecolor{Green}{RGB}{0,255,0}
\definecolor{Blanc}{RGB}{255,255,255}
\definecolor{Bleu}{RGB}{2,60,195}
\definecolor{Vert}{RGB}{23,163,1}
\definecolor{Violet}{RGB}{181,18,225}
\definecolor{Orange}{RGB}{255,113,15}

\tikzstyle{Black} = [Noir!80,thick,draw,line width=4pt]
\tikzstyle{Blue} = [Bleu!80,thick,draw,line width=4pt]
\tikzstyle{Green} = [Green!80,thick,draw,line width=4pt]
\tikzstyle{Violet} = [Violet!80,thick,draw,line width=4pt]
\tikzstyle{Feuille} = [rectangle,draw=Noir!70,fill=Noir!20,thick,
inner sep=0pt,minimum size=3mm,line width=2pt]
\tikzstyle{Boite} = [rectangle,rounded corners,draw=Bleu!100,fill=Bleu!10,
thick,inner sep=0pt,minimum size=10mm,line width=3pt,font=\Huge]

\newcommand{\cupdot}{\mathbin{\mathaccent\cdot\cup}}

\tikzstyle{decision} = [diamond, draw, fill=blue!20, 
    text width=4.5em, text badly centered, node distance=3cm, inner sep=0pt]
\tikzstyle{block} = [rectangle, draw, fill=blue!20, 
    text width=5em, text centered, rounded corners, minimum height=4em]
\tikzstyle{line} = [draw, -latex']
\tikzstyle{cloud} = [draw, ellipse,fill=red!20, node distance=3cm,
    minimum height=2em]

\title{The quantum Johnson homomorphism and symplectomorphism of 3-folds}
\author{Netanel Rubin-Blaier}
\thanks{Partially supported by Simons HMS Collaboration Grant no. 385579.}
\date{\today} 

\begin{document}
\begin{abstract} 
The (conjectured) existence of open--closed and closed--open isomorphisms \cite{MR1403918,MR2553465,MR2573826,MR3578916,MR3595902}
$$HH_{\bullet-n}(Fuk(M,\omega)) \stackrel{\mathcal{OC}}{\longrightarrow} QH^{\bullet}(M,\omega) \stackrel{\mathcal{CO}}{\longrightarrow} HH^{\bullet}(Fuk(M,\omega)),$$ 
combined with the dg-PROP action of chains on Kimura--Stasheff--Voronov moduli spaces on the left-hand side (seen as the closed state space of an open-closed TCFT) \cite{MR1341693,MR2298823,MR2596638,MR3565970}, as well as the fundamental Koszul duality in genus zero between the Gravity and Deligne-Mumford operads \cite{MR1256989,MR1284793,MR1363058} all point to the fact that quantum cohomology and Gromov-Witten invariants are just the tip of a much richer chain-level iceberg. In the first part of this paper, we consider the setting of closed monotone symplectic manifolds (where Floer and Gromov-Witten theories can be done using only classical transversality techniques). We begin by recalling the first (and easiest) strata of chain-level Gromov-Witten theory: the $A_\infty$-enhancement of the usual star product $\star$. This is classical, and was originally considered as far back as \cite{MR1470740,MR1480992}. We proceed to explicitly constructing a family version of it (which should morally correspond to Hochschild cohomology of a ''family of Fukaya categories over the circle" via some parametrized version of the closed--open map) and show that it already has some non-trivial applications to symplectic isotopy problems: Following Kaledin, we look at the obstruction class of the resulting $A_\infty$-structure, and argue that it can be related to a quantum version of Massey products on the one hand (which, in nice cases, can be related to actual counts of rational curves) and to the classical Andreadakis-Johnson theory of the Torelli group on the other hand. \\

In the second part of the paper, we go hunting for \textbf{exotic symplectomorphism}: these are elements of infinite order in the kernel 
$$\mathcal{K}(M,\omega) := \pi_0 \Symp(M,\omega) \to \pi_0 \Diff^+(M,\omega)$$
of the forgetful map from the symplectic mapping class group to the ordinary MCG. We demonstrate how we can apply the theory above to prove the existence of such elements $\psi_Y$ for certain a Fano 3-fold obtained by blowing-up $\P^3$ at a genus 4 curve. \\

Unlike the four-dimensional case, no power of a Dehn twist around Lagrangian 3-spheres can be exotic (because they have infinite order in smooth MCG). In the final part of the paper, the classical connection between our Fano varieties and cubic 3-folds allows us to prove the existence of a new phenomena: ''\textbf{exotic relations}" in the subgroup generated by all Dehn twists. Namely, it turns out we can factor some power of $[\psi_Y]$ in $\pi_0 \Symp(Y,\omega)$ into 3-dimensional Dehn twists. So the isotopy class of the product in the ordinary MCG is torsion, but of infinite order in the symplectic MCG. 
\end{abstract}

\maketitle

\tableofcontents


\section{Introduction}
\label{sec:introduction}
Let $(M,\omega)$ be a closed, $2n$-dimensional symplectic manifold. We define $\Diff^+(M)$ as the group of all orientation preserving diffeomorphism with the $C^\infty$-topology. A symplectomorphism $\phi : M \to M$ is a diffeomorphism such that $\phi^* \omega = \omega$. We equip the group $\Symp(M,\omega)$ of all such maps with the Hamiltonian topology\footnote{When $H^1(M;\mathbb{R})=0$, this is just the usual $C^\infty$-topology. The relation between the naive definition of the symplectic mapping class group and the one below is 
related to the image of the Flux homomorphism. See \cite[Chapter 10]{MR1373431} or \cite{MR1826128} for a detailed discussion.} \cite[Remark 0.1]{MR2441414}. Denote $\Diff^+_0(M)$ for the connected component of the identity inside the group of orientation preserving diffeomorphism. The identity component of $\Symp(M,\omega)$ is the group of all Hamiltonian symplectomorphism $\Ham(M,\omega)$.
\begin{definition}
We call $\pi_0(\Diff^+(M)) = \Diff^+(M)/\Diff_0^+(M)$ the (ordinary, or smooth) \textbf{mapping class group}, and 
\begin{equation}
\pi_0(\Symp(M,\omega)) = \Symp(M,\omega)/\Ham(M,\omega)
\end{equation}
the \textbf{symplectic mapping class group}. 
\end{definition}
Many fundamental problems in symplectic topology can be rephrased as questions about the geometry and topology of the symplectic mapping class group, and there are many open questions regarding its structure (see the opening of \cite{preprint2} for a list of them). The \textit{symplectic isotopy problem} asks the natural question of how much of the topology of $\Symp(M,\omega)$ is already visible in $\Diff^+(M)$, i.e., what is the kernel of the induced map on homotopy groups
\begin{equation} \label{eq:homotopysequence1}
\pi_k(\Symp(M,\omega)) \to \pi_k(\Diff^+(M)). 
\end{equation}
We denote $\mathcal{K}(M, \omega)$ for the kernel of \eqref{eq:homotopysequence1} when $k=0$. Even in this case, describing this group is already a hard problem. Moser's trick solves the problem immediately for any closed surfaces, but already when $n=2$, the right-hand side of \eqref{eq:homotopysequence1} is already notoriously hard to compute\footnote{Currently, the answer is unknown even for $\R^4$!}. Surprisingly, it turns out the symplectomorphism groups are more tractable. The seminal work of Gromov (\cite{MR809718}), continued by Abreu, Anjos and McDuff (\cite{MR1489893},\cite{MR1775741},\cite{MR1914568}) give a detailed picture of the behavior of $\Symp(M)$ for ruled surfaces, as well as the dependence on the symplectic form. In particular, it follows from their work that there are no exotic symplectomorphism for these manifolds. In the opposite direction, P. Seidel showed in \cite{Paulthesis} by a Floer theoretic argument that every squared generalized Dehn twist $\tau_V^2$ in dimension four is smoothly isotopic to the identity, but not symplectically so. A degeneration argument then shows the  existence of exotic symplectomorphism for all complete intersection algebraic surfaces which are not $\mathbb{CP}^2$ or $\mathbb{CP}^1 \times \mathbb{CP}^1$. For K3 surfaces, Seidel \cite{MR1743463} has shown in some cases the kernel $\mathcal{K}(M,\omega)$ contains a subgroup isomorphic to the pure braid group $PB_m$ for $m \leq 15$, and in a recent preprint Sheridan and Smith showed that in certain important cases (including the mirror quartic and the mirror double plane, see \cite[Theorem 1.1]{preprint2}), the entire group $\mathcal{K}(M,\omega)$ surjects into an infinitely infinite-rank free group, and therefore is infinitely generated. Other interesting recent developments in the four-dimensional case include \cite{MR2422351}, \cite{MR2787361}, \cite{MR2946805},\cite{MR3359028},\cite{MR3356772}.

However, higher dimensions remain elusive. This can be attributed to the well known fact that holomorphic curves behave much better in dimension $4$ (automatic transversality, positivity of intersections,...) then in higher dimensions. Some previous results include \cite{MR1736220} (which is our motivation), \cite{MR2929086}, \cite{MR3302963}, \cite{MR3447106} and the preprint \cite{eprint9}. \\

One common thread in most of these papers is that they usually consider the action of the symplectic monodromy at the level of open string invariants. These approaches necessarily require heavy Floer-theoretic machinery. Unfortunately, we currently have only a tenuous grasp of Lagrangians in high-dimensional algebraic varieties and the Fukaya categories involved. On the other hand, algebraic geometers know a great deal about rational curves in these cases, both from classical enumerative tools (degeneration, specialization) and modern techniques (Bend-and-Break, Gromov-Witten). All this motivates us to look at the action of monodromy on closed strings.

\begin{remark}
In more detail, the derived Fukaya category $\mathcal{D}^\pi \Fuk(M,\omega)$ is an example of such weak proper n-Calabi-Yau $A_\infty$-category (for closed monotone symplectic manifold this was proved in \cite{MR3578916}) and as explained in \cite[section 10(j)]{MR2441780} and \cite[p.~18--20]{preprint2} there is an action 
\begin{equation} \label{eq:actionHH1}
\pi_0 \Symp(M,\omega) \to \Auteq_{CY}(\mathcal{D}^\pi\Fuk(M,\omega))/[2], \\
\end{equation}
where $[k]$ for any $k \in \mathbb{Z}$ denote the shift functor. Moreover, if the open-closed map is an isomorphism, then if we denote $\Auteq^0(\mathcal{D}^\pi \Fuk(M,\omega))$ for the subgroup of auto-equivalences which act trivially on Hochschild homology, then we have another induced action
\begin{equation} \label{eq:actionHH2}
\mathcal{K}(M,\omega) \to \Auteq^0(\mathcal{D}^\pi \Fuk(M,\omega))/[2]
\end{equation}
which we can use to detect non-trivial elements in the kernel. The actions \eqref{eq:actionHH1} and \eqref{eq:actionHH2} act on Hochschild homology as well because of functoriality. In this paper we are essentially considering a way to attach a computable invariant to the induced action on the $A_\infty$-structure of Hochschild cohomology. More generally, it is in expected that the open-closed and closed-open isomorphism will intertwine the dg-PROP action on $HH_\bullet(\mathcal{D}^\pi \Fuk(M,\omega))$ (which exist for purely algebraic reasons) with operations defined in terms of the closed invariants of the target manifolds and the monodromy should act on them as well.
\end{remark}

In this paper, we want to suggest a new, dedicated tool to be added to the existing arsenal of techniques. 


\begin{kindoftheorem} \label{thm:kindof}
Let $(M,\omega)$ be a closed, monotone and simply connected symplectic manifold. \vspace{0.2em}
\begin{itemize} 
\item
We associate to every symplectomorphism $\phi : M \to M$ an $A_\infty$-algebra ''mapping torus" 
\begin{equation}
\phi \mapsto \mathfrak{X}_\phi 
\end{equation}
and to every symplectic isotopy $\phi_t$, an $A_\infty$-quasi-isomorphism \vspace{0.2em}
\begin{equation} \label{eq:naturality}
\left\{\phi_t\right\} \mapsto \left(\Phi : \mathfrak{X}_{\phi_0} \to \mathfrak{X}_{\phi_1}\right)
\end{equation} 
\item
Given a spherical homology class $A \in H_2^S(M;\Z)$, and for suitable $\phi$, we assign to it a ''characteristic class" 
\begin{equation} \label{eq:naturality}
q\tau_2 : \phi \mapsto \mathfrak{X}_\phi  \mapsto \oo^{3,A}_\phi \in HH^\bullet(...,...)
\end{equation} 
in the Hochschild cohomology of a certain algebra which is: 
\begin{enumerate} \vspace{0.2em}
\item
Natural w.r.t to the quasi-isomorphism \eqref{eq:naturality}; \vspace{0.2em}
\item
Vanishes for $\phi = id$, and \vspace{0.2em}
\item
Can be ''evaluated" on suitable cohomology classes. \vspace{0.2em}
\end{enumerate}
\end{itemize}
This assignment is called a \textbf{quantum Johnson homomorphism} -- for reasons explained in subsection \ref{subsec:qjohnson}. The ''values" of $q\tau_2$ are triple quantum (matrix) Massey products, the formal definition of which can be found in subsection \ref{subsec:quantummatrixmasseyproduct} and expanded upon in \ref{subsubsec:parametrizedqcoh}. 
\end{kindoftheorem}
\begin{remark}
The content of Theorem \ref{thm:kindof} is made precise in Proposition \ref{prop:existencepearlMphi}. 
\end{remark}
One major advantage of $q \tau_2$ is that it is fairly ''low-tech" in comparison to previous methods. As a result, the values these classes take can be \textbf{computed explicitly} in good cases. \\

To demonstrate this, consider the smooth Fano 3-fold $Y$ obtained by blowing up $\P^3$ at the complete intersection of a quadric and a cubic. We show that

\begin{theorem} \label{thm:1}
There exists Lagrangian spheres $\left\{V_1',\cdots,V_5'\right\}$ and $\left\{V_1'',\cdots,V_5''\right\}$ in $Y$ and an integer $N>0$, such that the symplectic isotopy class $\kappa$ of 
\begin{equation} \label{eq:factorization}
\psi_Y^N := (\tau_{V_1'} \circ \cdots \circ \tau_{V_5'})^{6N} \circ  (\tau_{V_1''} \circ \cdots \circ \tau_{V_5''})^{6N} : Y \to Y
\end{equation}
is of infinite order in $\mathcal{K}(Y,\omega_Y)$
\end{theorem}

\begin{proof}
First, $Y$ is a simply connected 6-dimensional manifold, and $\psi_Y$ acts trivially on cohomology. In addition, if we pick an almost complex structure $J : TY \to TY$ which is compatible with $\omega$, then we can introduce a Riemannian metric $g_J(\bullet,\bullet) := \omega(\bullet,J\bullet)$ on $Y$. The symplectomorphism $\psi_Y$ is now an isometry of $(Y,g_J)$, and the fact that for there exists $N>0$, $\psi_Y^N$ is smoothly isotopic to the identity is classical and follows from Sullivan's rational homotopy theory, see \cite[p.~322--323]{MR0646078}. 

There are two things which remain to show: that there indeed exist Lagrangians such that factorization \eqref{eq:factorization} exists (see \ref{subsubsec:blowupinfamilies}); and the main statement, $q\tau_2(\psi_Y^N) \neq 0$ for every N-th iterate -- see \ref{subsubsec:mainstatement}. 
\end{proof}

\subsection{Organization of the paper} 
Due to the considerable length of this paper (and the many technicalities involved), we have taken the following approach, so as to not obscure the main line of reasoning: The current section contains the motivation, the entire proof of the Main Theorem \ref{thm:1}, as well as an accurate statement of Theorem \ref{thm:kindof}. It \emph{does not} include the construction of $q \tau_2$ (although see \ref{subsec:asketch}), and the various individual computations -- these are treated as a black box. 

Most of the algebraic and geometric aspects of the definition of $q \tau_2$ (along with all necessary preliminaries) are explained in length in sections \ref{sec:definitionsandstatements} and \ref{sec:algebraic}, respectively; We tried (as much as possible) to sequester all functional-analytic and operadic technicalities needed for the construction of the $A_\infty$-pearl complex. These take two separate sections (\ref{sec:masseyproductquantum},\ref{sec:parametrized}) all on their own. 

Each type of computations we need as input for our machine (Morse-theoretic, Enumerative geometry, Parametrized Quantum cohomology, and finally -- the Quantum Massey product) receives a section, numbered: \ref{sec:compute3dim}, \ref{sec:cohomologylevelcomputations}, \ref{sec:ambiguity}, and \ref{sec:compute4}. The K\"{a}hler degeneration to a maximal $2A_5$-curve (which allows us to factor the monodromy into Dehn twists, as well as having the fortunate side effect of making many of the computations easier) appears in section \ref{sec:familiesofcurves}, preceded by a short introduction to singularity theory in section \ref{sec:singularitytheory}. 

Finally, section \ref{sec:toymodel} is dedicated to survey of Morse theory from the $A_\infty$-perspective. None of the material is new, but much of it is scattered in different corners of the literature. This will serve three purposes: as a toy-model for the $A_\infty$-pearl complex, a tie-in with the classical theory of $\tau_2$ (explained in section \ref{subsec:qjohnson}), and also as a necessary ingredient for the computations done in sections \ref{sec:compute3dim} and \ref{sec:compute4}. The inter-dependence is summarized in the chart below.
\begin{figure}
\centering
\begin{tikzpicture}[node distance = 3cm, auto]

		\node [cloud] (compute1) {Section \ref{sec:compute3dim}};
	    \node [cloud, right of=compute1] (compute2) {Section \ref{sec:cohomologylevelcomputations}};
    \node [cloud, right of=compute2] (compute3) {Section \ref{sec:ambiguity}};		
	    \node [cloud, right of=compute3] (compute4) {Section \ref{sec:compute4}};		
			\node [decision, above of=compute3,xshift=-1.3cm,yshift=1.5cm] (morse) {section \ref{sec:toymodel}};
		\node [block, below of=compute2, xshift=1.5cm] (intro) {Section \ref{sec:introduction}};

		\node [block, below of=intro, xshift=-1cm] (definitions) {Section \ref{sec:definitionsandstatements}};	
		\node [block, left of=definitions] (algebraic) {Section \ref{sec:algebraic}};	
		\node [decision, right of=definitions, xshift=-0.49cm] (singularity) {Section \ref{sec:singularitytheory}};	
		\node [block, right of=singularity] (families) {Section \ref{sec:familiesofcurves}};	
		
		\node [block, below of=definitions] (coherence) {Section \ref{sec:masseyproductquantum}};	
		\node [block, right of=coherence] (analytic) {Section \ref{sec:parametrized}};	
		\path [line] (compute1) -- (intro);
		\path [line] (compute2) -- (intro);
		\path [line] (compute3) -- (intro);
		\path [line] (compute4) -- (intro);
		\path [line] (definitions) -- (algebraic);
		\path [line] (algebraic) -- (intro);
		\path [line] (definitions) -- (intro);
		
		\path [line] (singularity) -- (families);
		\path [line] (families) -- (intro);
		\path [line] (families) -- (compute4);
		
		\path [line] (analytic) -- (coherence);
		\path [line] (morse) -- (compute1);
		\path [line] (morse) -- (compute4);
		\path [line] (coherence) -- (definitions);
		\path [line] (compute2) -- (compute3);
		
\end{tikzpicture}
\label{fig:placement}
\end{figure}

\subsection{Background: Quantum cohomology, Massey products} \label{subsec:masseyprodinquantumcohomology}

Let $(M,\omega)$ be a closed, $2n$-dimensional symplectic manifold. In this setting, there is a well-known deformation of the usual cup product in cohomology encoding the intersection patterns of $J$-holomorphic spheres. We recall the basic structure and properties. We will assume that $M$ is \textbf{monotone}: $[\omega]=\kappa \cdot c_1(M)$ in $H^2(M;\R)$ for some positive constant $\kappa \in \R$. This greatly simplifies things and we can use classical pseudo-holomorphic curve theory (see e.g., \cite{MR1366548},\cite{MR2954391}) to define Gromov-Witten invariants, without appealing to the heavy analytic machinery required by virtual perturbation techniques such as Polyfolds or Kuranishi structures. As a result, all our invariants would be defined on the chain-level by actual pseudo-cycles and not just as cohomology classes, which would be important later on. 

\subsubsection{Coefficients}
\begin{definition}
Let $H_2(M)$ be the torsion-free part of $H_2(M;\Z)$. We denote by 
\begin{equation}
\mathfrak{G} \subset H_2(M)
\end{equation}
the submonoid which consist of $A=0$ and all homology classes such that $\omega(A) > 0$. 
\end{definition}

Let $\Q[H_2(M)]$ be the group algebra of $H_2(M)$ with rational coefficients, i.e., the algebra whose elements are finite formal sums of the form 
\begin{equation}
\lambda = \sum_{A \in H_2(M)} \lambda(A) e^A \:  \: , \:  \: \deg(e^A)=2c_1(A) 
\end{equation} 
and where multiplication is defined by the rule $e^{A+B} = e^A e^B$, and $e^0=1$. 

\begin{definition}
We define our \textbf{quantum coefficient ring} to be the $2\Z$-graded $\Q$-subalgebra 
\begin{equation}
\Gamma := \left\{ \gamma = \sum_{A \in \mathfrak{G}} \gamma(A) e^A \right\} \subset \Q[H_2(M)]. 
\end{equation}
\end{definition}
We will often consider deformations and algebraic structures only up to some finite order, energy level or both. 
\begin{definition}
Given a graded algebra $C^\bullet$ over $\Gamma$ (later - an $A_\infty$-algebra), and an ideal $\II \subset \Gamma$ we define the \textbf{truncation at $\II$} to be
\begin{equation}
C_\II^\bullet = C^\bullet \otimes_\Gamma (\Gamma/\II). 
\end{equation}
\end{definition}
\begin{example}
Given an integer $c \in \N_{\geq 0}$, we can define an ideal $\II_c$ which is generated by all symbols $e^B$ with $c_1(B)>c$.
\end{example}
\begin{example} \label{example:1.2}
Given a homology class $A \in H_2(M)$, we denote by 
\begin{equation}
\II_A \subset \Gamma 
\end{equation}
the ideal which consists of all finite formal sums 
\begin{equation}
\lambda = \sum_{B \in H_2(M)} \lambda(B) e^B 
\end{equation} 
where the homology classes satisfy one of the following conditions:
\begin{itemize}
\item
$c_1(B) = c_1(A)$ and $B \neq A$.
\item
$c_1(B) > c_1(A)$.
\end{itemize}
\end{example}
\begin{remark} As opposed to the general case, we do not use a Novikov ring. This is possible by monotonicity: the infinite sums (which the Novikov ring is meant to encode)  are actually finite in our setting. We could have just decided to work over the complex numbers, or maybe a polynomial ring which encodes only the energy, but instead we have chosen to work over a subring of the group ring, because the extra information would make things a little nicer later on.  
\end{remark}
\subsubsection{The small Quantum ring} \label{subsubsec:smallquantumring}
An almost complex structure is an automorphism of the tangent bundle $J : TM \to TM$ which satisfies the identity $J^2 = -id$. We say that the form $\omega$ tames $J$ if $\omega(v, Jv) > 0$ for all nonzero tangent vectors. \\

Additively, one can define the \textbf{(small) quantum cohomology ring}, denoted $QH^\bullet(M,\omega)$ as the graded $\Q$-vector space
\begin{equation}
QH^\bullet(M,\omega) := H^\bullet(M) \otimes_\Q \Gamma. 
\end{equation}
To define the quantum product, we first define the genus zero, GW invariant
\begin{equation}
GW^M_{k,A} : QH^\bullet(M,\omega)^{\otimes k} \to \Q
\end{equation}
as follows. Consider the moduli space of $\textbf{J}$-holomorphic spheres in homology class $A \in H_2(M; \Z)$, with $k$ marked points, 
\begin{equation} \label{eq:moduliofcurves}
\MM_{0,k}(A;M,\omega;\textbf{J}) 
\end{equation}
where $\textbf{J} = (J_z)_{z \in S^2}$ is a domain-dependent $\omega$-tame almost-complex structure. For generic $\textbf{J}$, it is a smooth manifold of dimension $2n + 2c_1(A) +2k-6$. Let $\alpha_i \in H^{d_i}(M)$ be cohomology classes and let 
\begin{equation}
f_i : V_i \to M \: , \: i=1,\ldots,k
\end{equation}
be pseudo-cycles which represent the homology classes $PD(\alpha_i)$. Then for generic choices of almost complex structure and pseudo-cycles, the moduli space of $\textbf{J}$-holomorphic spheres with marked points constrained to lie on the pseudo-cycles $f_1, \ldots, f_k$ respectively, is an oriented smooth manifold of dimension
\begin{equation}
d := 2n + 2c_1(A) +2k-6 - (\sum_{i=1}^k d_i)
\end{equation}
If $d=0$, this manifold is compact, and we define $GW^M_{k,A}(\alpha_1,\ldots,\alpha_k)$
to be the signed count of its points (if $d \neq 0$ we define it to be zero). It is independent of $\textbf{J}$ and the choice of representing pseudocycles by a cobordism argument.

\begin{proposition}[Kontsevich-Manin axioms]
Gromov-Witten invariants satisfy the following list of formal properties: \vspace{0.5em}

\begin{itemize}
\item
\textbf{Effective.} If $\omega(A) < 0$ then $GW^M_{k,A} = 0$. \vspace{0.5em}
\item
\textbf{Symmetry.} $GW^M_{k,A}$ is graded symmetric. \vspace{0.5em}
\item
\textbf{Grading.} If $GW^M_{k,A}(\alpha_k,\ldots,\alpha_1) \neq 0$, then 
\begin{equation}
\deg(\alpha_k) + \ldots + \deg(\alpha_1)  = 2n + 2c_1(A). \vspace{0.5em} 
\end{equation}
\item
\textbf{Divisor.} If $(k,A) \neq (3,0)$, and $\deg(\alpha_1) = 2$ then 
\begin{equation}
GW^M_{k,A}(\alpha_k,\ldots,\alpha_1) =  GW^M_{k-1,A}(\alpha_k,\ldots,\alpha_{2}) \int_A \alpha_1
\end{equation}
\item
\textbf{Zero.} Let $A=0$. If $k \neq 3$ then $GW^M_{0,k} = 0$. If $k = 3$ then 
\begin{equation}
GW^M_{0,3}(\alpha_3,\alpha_2,\alpha_1) = \int_M \alpha_3 \cup \alpha_2 \cup \alpha_1. 
\end{equation}
\end{itemize}
\end{proposition}
\begin{definition}
The product structure on $QH^\bullet(M,\omega)$ is given by the formula
\begin{equation}
\alpha \star \beta := \sum_{A \in \mathfrak{G}} (\alpha \star_A \beta) \otimes e^A
\end{equation}
\end{definition}
In the same way that the structure coefficients of the cup product encode intersections of their Poincare dual homology classes, the structure coefficients 
\begin{equation}
\alpha \star_A \beta \in H^{\deg(\alpha) + \deg(\beta) - 2c_1(A)}(M)
\end{equation}
of the quantum product encode ''fuzzy intersections". They are defined by requiring that
\begin{equation}
\int_M (\alpha \star_A \beta) \cup \gamma = GW^M_{3,A} (\alpha,\beta,\gamma) 
\end{equation}
for any $\gamma \in H^\bullet(M)$. 

\begin{remark} \label{rem:1.14} More algebro-geometrically, when $(M,\omega,\textbf{J})$ is K\"{a}hler and certain technical conditions are met, $GW^M_{A,3}(\alpha_3,\alpha_2,\alpha_1)$ counts the number of simple (i.e., non multiply-covered) rational curves in the homology class $A$ that pass through cycles in ''generic position" whose homology classes are Poincare dual to the cohomology classes $\alpha_3$, $\alpha_2$, and $\alpha_1$. 
\end{remark}

\begin{proposition} \label{thm:propquantumcohomology}
\begin{enumerate}
\item
The quantum cup product is distributive over addition and graded commutative in the sense that 
\begin{equation}
\beta \star \alpha = (-1)^{\deg(\alpha) \deg(\beta)} \alpha \star \beta
\end{equation}
for any $\alpha,\beta \in QH^\bullet(M;\omega)$ of pure degree.
\item
The quantum cup product is associative and commutes with the action of $\Gamma$. 
\item
There is an embedding of the ordinary cohomology $H^\bullet(M)$ into $QH^\bullet(M,\omega)$ given by $\alpha \mapsto \alpha \otimes 1$.
\item
The leading term in the quantum cup product is the standard cup product, i.e. 
\begin{equation}
\alpha \star_0 \beta = \alpha \cup \beta.
\end{equation}
Moreover the higher terms 
\begin{equation}
\alpha \star_A \beta \: \: , \: \: A \neq 0
\end{equation}
vanish whenever $\alpha \in H^0(M)$ or $\alpha \in H^1(M)$, thus the canonical generator $1 \in H^0(M)$ is the unit in quantum cohomology. 
\end{enumerate}
\end{proposition}

\subsubsection{Massey products} \label{subsubsec:masseyproducts}
Motivated by the existence of the quantum product, we would like to define quantum Massey products and study their properties. To do so, we must first describe the classical theory in some amount of detail. 

Let $R$ be a unital commutative ring. Recall that a differential graded algebra over $R$ is a graded $R$-algebra 

\begin{equation}
\mathcal{C}^\bullet = \oplus_{k\geq 0} \mathcal{C}^k = \mathcal{C}^0 \oplus \mathcal{C}^1\oplus \mathcal{C}^2 \oplus \cdots 
\end{equation}

with product $\cdot$ and equipped with a differential $d \colon \mathcal{C}^\bullet\to \mathcal{C}^{\bullet+1}$ such that 
\begin{enumerate}
\item $d$ is a derivation, i.e.,
\[
d(a \cdot b) = (d a) \cdot b +(-1)^k a \cdot (d b) \quad (a\in \mathcal{C}^k \: , \:b\in \mathcal{C}^\ell);
\]
\item $d^2=0$.
\end{enumerate}
The cohomology is defined as usual by 
\begin{equation}
C^\bullet := H^\bullet(\mathcal{C}):=\ker(d)/\im(d), 
\end{equation}
and it inherits an algebra structure $\cup$ from $\mathcal{C}$. 

\begin{definition} \label{def:masseyproduct}
Let $\alpha$, $\beta$, and $\gamma$ be the cohomology classes of closed elements $a \in \mathcal{C}^p$, $b \in \mathcal{C}^q$, and $c \in \mathcal{C}^r$. The triple Massey product 
\begin{equation} \label{eq:1}
\langle \alpha,\beta,\gamma\rangle
\end{equation}
is \textbf{defined} if $\alpha \cup \beta = 0$, and $\beta \cup \gamma = 0$. If that is the case, let us define the \textbf{ambiguity} to be the ideal
\begin{equation}
I = \alpha \cup C^{q+r} + C^{p+q} \cup \gamma.
\end{equation}
We then choose cochain-level representatives 
\begin{equation}
dx = (-1)^p a \cdot b, dy = (-1)^q b \cdot c.
\end{equation}
and define the triple Massey product to be the coset 
\begin{equation}
\langle \alpha,\beta,\gamma\rangle = [(-1)^{p+1} a \cdot x+ (-1)^{p+q} y \cdot c] \in C^{p+q+r-1}  / I. 
\end{equation}
\end{definition}

As an illustrative example, consider the link formed by the \emph{Borromean rings} $B \subset S^3$ (see Figure \ref{fig:2}.) 

\begin{figure}[htb] 
\begin{center}
\centering
 \includegraphics[scale=0.3]{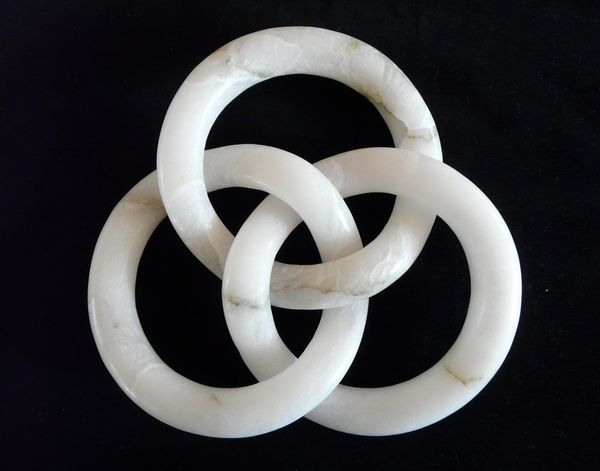}
\end{center}
  \caption{The Borromean rings}
	\label{fig:2}
\end{figure}
Let $N$ denote a tubular neighbourhood of $B$. Then by Poincare duality 
\begin{equation}
		H^1(S^3-B;\mathbb{Z}) \cong H^1(S^3-N;\mathbb{Z}) \cong H_2(S^3-N,\partial N;\mathbb{Z}).
\end{equation}
We choose flat 2-discs $a,b,c$, such that each disk is filling a ring. The situation is depicted in figure \ref{fig:3} below. It is easy to see that the Poincare dual cohomology classes 
\begin{equation}
 \alpha= PD([a]) \: , \: \beta= PD([b]) \: , \: \gamma= PD([c]) 
\end{equation}
generate $H^1(S^3-B;\mathbb{Z})$. 
\begin{figure}[htb] 
\begin{center}
\centering
		\includegraphics[scale=0.4]{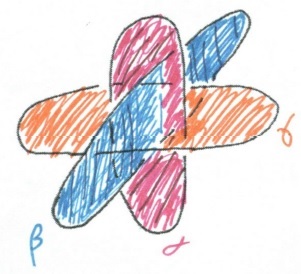}
\end{center}		
\caption{Relative cycles $\alpha,\beta,\gamma \in C_2(S^3-N,\partial N;\mathbb{Z})$}
\label{fig:3}
\end{figure}
Arguing dually in terms of intersections, we see that the cup product of any two generators vanish
\begin{equation}
\alpha \cup \beta = \beta \cup \gamma = \beta \cup \gamma = 0,
\end{equation}
which reflects the fact that no two rings are \emph{pairwise linked}. However, the entire configuration is \emph{triply linked} which is reflected in a nontrivial triple Massey product (see Figure \ref{fig:4} for the Poincare dual of the bounding cochains, and Figure \ref{fig:5} for the Poincare dual of $\zeta$)


\begin{figure} 
  \begin{subfigure}[b]{.28\linewidth}
    \centering
		\includegraphics[scale=0.4]{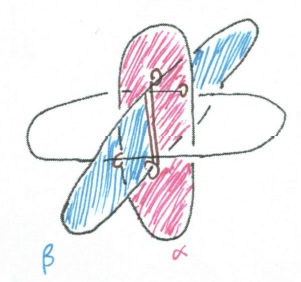}
  \end{subfigure}\hfill
  \begin{subfigure}[b]{.28\linewidth}
    \centering
		\includegraphics[scale=0.4]{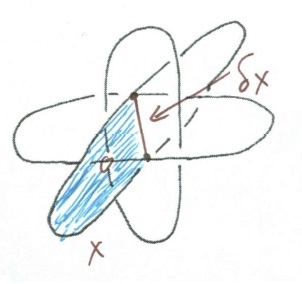}
  \end{subfigure}\hfill
  \begin{subfigure}[b]{.28\linewidth}
    \centering
		\includegraphics[scale=0.4]{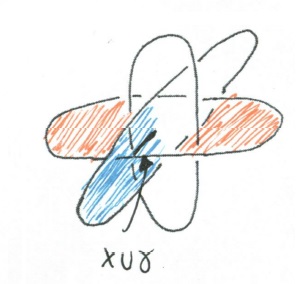}
  \end{subfigure}	\hfill
  \begin{subfigure}[b]{.28\linewidth}
    \centering
		\includegraphics[scale=0.4]{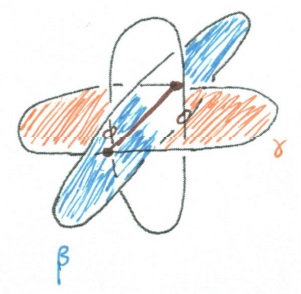}
  \end{subfigure}\hfill
  \begin{subfigure}[b]{.28\linewidth}
    \centering
		\includegraphics[scale=0.4]{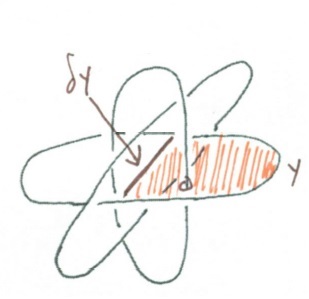}
  \end{subfigure}\hfill
  \begin{subfigure}[b]{.28\linewidth}
    \centering
		\includegraphics[scale=0.4]{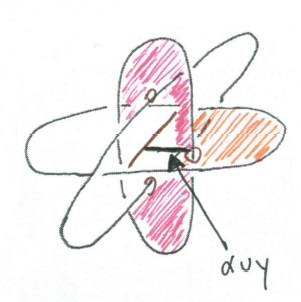}
  \end{subfigure}	\hfill
	\caption{The bounding chains $x,y$ must intersect the Poincare dual of $\alpha,\beta,\gamma$}
	\label{fig:4} 
\end{figure}	
\begin{figure} 
    \centering
		\includegraphics[scale=0.4]{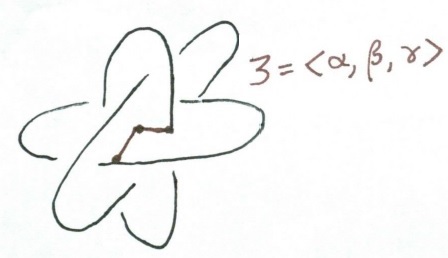}
\caption{A non-trivial element in the Massey product}
	\label{fig:5}
\end{figure}

\begin{equation}
0 \neq \zeta \in \langle \alpha,\beta,\gamma \rangle \subseteq H^2(S^3-B;\mathbb{Z})/I,
\end{equation}
where 
\begin{equation}
I = \alpha \cup H^1(S^3-B;\mathbb{Z}) + H^1(S^3-B;\mathbb{Z}) \cup \gamma.
\end{equation}
We see that if we think of the cup product in terms of transverse intersections of cycles, then Massey products measure ''higher-order" intersections which involve both the cycles and the bounding chains.  

\subsection{Quantum (matrix) Massey products} \label{subsec:quantummatrixmasseyproduct}
We introduce the definition and (some of) the properties Quantum Massey products and Quantum (matrix) Massey products. \\

Given an ideal $\II_A \subset \Gamma$, we denote $\Gamma_A := \Gamma / \II_A$. Note that the quantum multiplication descends to $QH^{\bullet} := QH^\bullet(M,\omega;\Gamma_A)$. In a slight abuse of notation, we will denote it in the same way as before, namely: $\star$ and $\star_B$. 
\begin{definition} 
A \textbf{defining system} (with respect to the ideal $\II_A$) is a triple of cohomology classes
\begin{equation}
x_i \in QH^{\bullet} := QH^\bullet(M,\omega;\Gamma_A ) \: , \: i=1,2,3
\end{equation}
which satisfy the \textbf{vanishing condition}: 
\begin{equation} \label{eq:vanishingcondition}
x_3 \star_B x_2 = 0 \: , \: x_2 \star_B x_1 = 0 
\end{equation}
for all $B \in \Gamma_A$.
\end{definition}

\begin{definition}
Given a defining system, let us denote $\II$ for \textbf{the ambiguity ideal} generated by all
\begin{equation}
x_3 \star_{B} QH^\bullet +  QH^\bullet \star_{B'} x_1
\end{equation}
where $B,B' \in \Gamma_A$. 
\end{definition}

Formally speaking, a (triple) \textbf{Quantum Massey product} 
\begin{equation}
\langle x_3,x_2,x_1 \rangle_A
\end{equation}
with respect to the ideal $\II_A$ (or \textbf{QMP} for short) is an assignment that takes as input a defining system, and whose output is a coset in $(e^A \cdot QH^{p+q+r-1}) / \II$. \\

Like Massey products, QMP's are multilinear operations that measure higher-order ''fuzzy intersections"; and like $QH^\bullet$, QMP contain a new feature controlling the homology class (or energy). The analogy to keep in mind is 
\begin{equation*}
\biggl(\text{ Quantum Massey products w.r.t. $\II_A$ : $\star_A$ }\biggr) \longleftrightarrow \biggl(\text{ Classical Massey products : $\cup$ }\biggr)
\end{equation*}

\begin{proposition} \label{prop:multilinear}
Such an assignment exists and gives a partially defined, multi-linear operation
\begin{equation}
\langle x_3,x_2,x_1 \rangle : QH^{p} \otimes QH^{q} \otimes QH^{r} \to (e^A \cdot QH^{p+q+r-1})/\II.  \vspace{0.5em}
\end{equation} 
\end{proposition}



In a similar vain to Remark \ref{rem:1.14}, Quantum Massey products count higher order ''fuzzy intersections" between cycles in the manifold, and spaces of $\textbf{J}$-holomorphic curves (this can often by obscured by a more complicated perturbation scheme.) Intuitively, the triple Massey product count of configurations of the form that appears in Figure \ref{fig:mu3a}.

\begin{figure} 
  \begin{subfigure}[b]{.28\linewidth}
    \centering
		\fontsize{0.3cm}{1em}
		\def\svgwidth{5cm}
		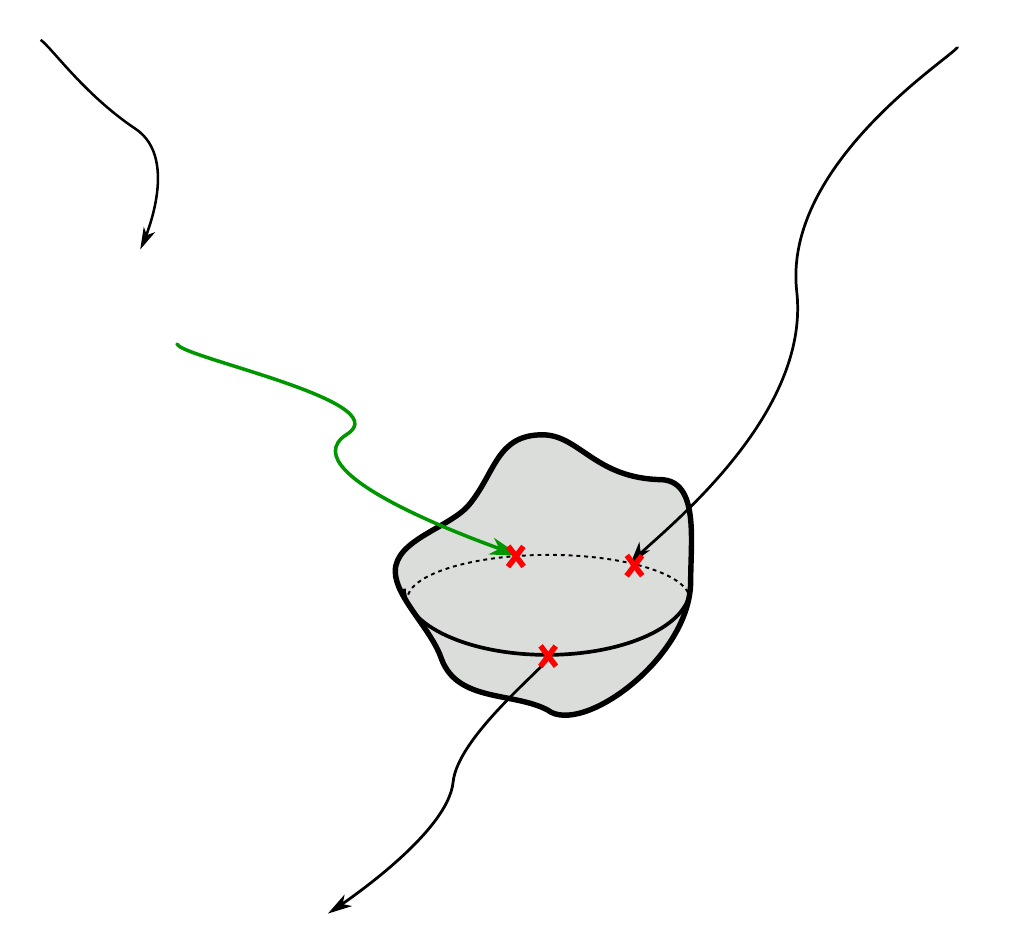
  \end{subfigure}\hfill
  \begin{subfigure}[b]{.28\linewidth}
    \centering
		\fontsize{0.3cm}{1em}
		\def\svgwidth{5cm}
		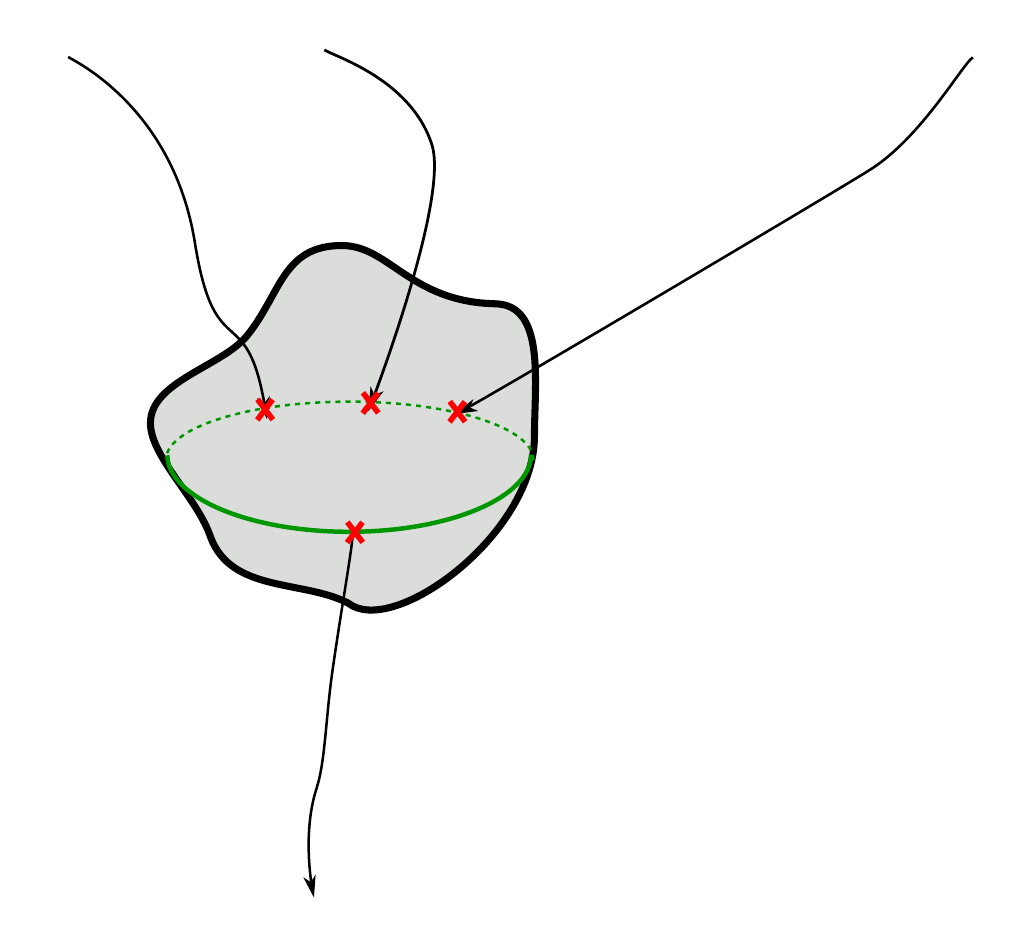
  \end{subfigure}\hfill
  \begin{subfigure}[b]{.28\linewidth}
    \centering
		\fontsize{0.3cm}{1em}
		\def\svgwidth{5cm}
		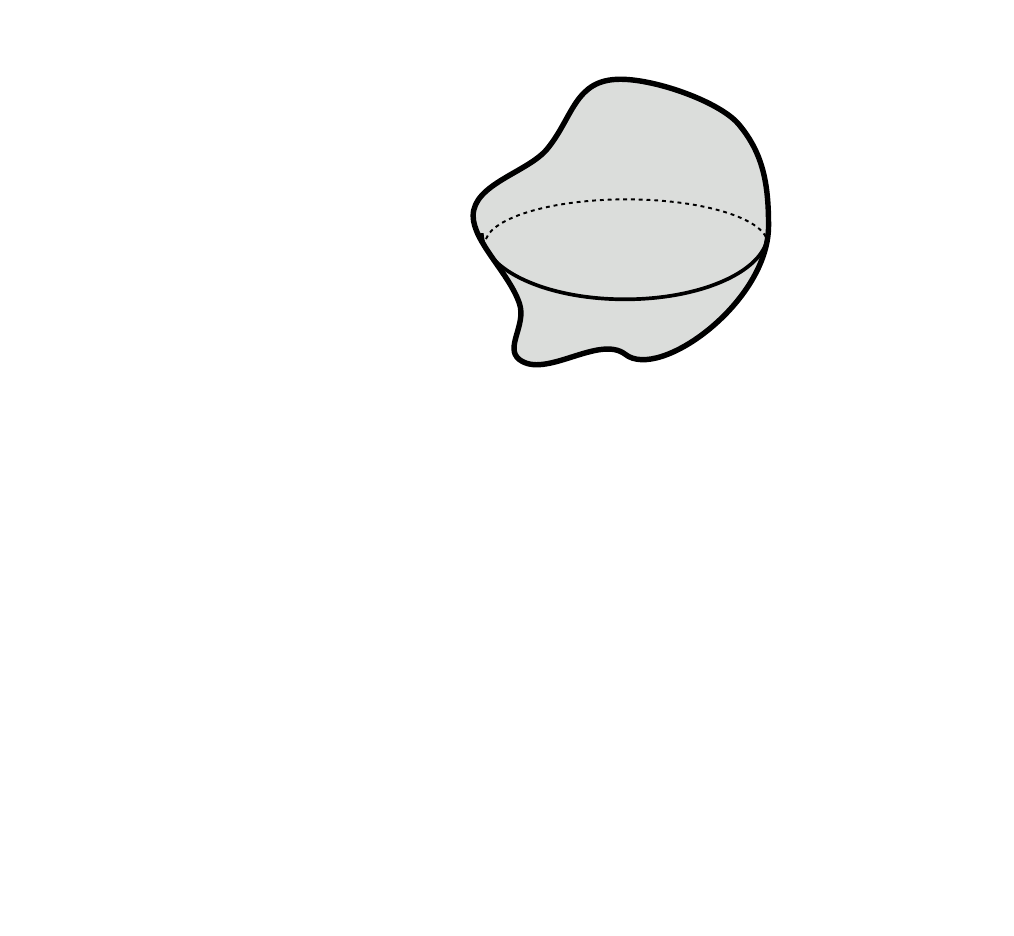
  \end{subfigure}	
	\caption{The main term(s) in a QMP (later called $\mu^3_A$)} 
	\label{fig:mu3a}
\end{figure}	

\begin{remark} Note that the terms to the left and right correspond to what might one expect from a Morse theoretic version of Massey products (see subsection \ref{subsec:qjohnson}), but the central term is a completely new contribution, coming from the fact that four points on the equator of a marked sphere have one moduli parameter.  
\end{remark} 

Unfortunately, in real-world applications, finding reasonable cohomology classes that satisfy the vanishing condition \eqref{eq:vanishingcondition} is quite the onerous task. With that in mind, we introduce the following, less intuitive but technically superior, version: 

\begin{definition}
Let $n \geq 1$. A (triple) \textbf{Quantum matrix Massey product} with respect to the ideal $\II_A$ (\textbf{matrix QMP}) is an assignment that takes as input triples of n-by-n matrices of cohomology classes in 
\begin{equation}
QH_{n \times n}^{\bullet} := M_{n \times n}(QH^\bullet(M,\omega;\Gamma_A)) 
\end{equation}
that satisfy 
\begin{equation} \label{eq:matrixvanishing}
[X_3] \star_B [X_2] = 0 \: , \: [X_2] \star_B [X_1] = 0 
\end{equation}
for every $B \in \Gamma_A$. The output, denoted, 
\begin{equation}
\langle X_3,X_2,X_1 \rangle_A
\end{equation}
is a coset in $(e^A \cdot QH_{n \times n}^{p+q+r-1}) / \II$, where we defined the \textbf{ambiguity ideal} to be the ideal generated by all
\begin{equation}
[X_3] \star_{B} QH_{n \times n}^\bullet + QH_{n \times n}^\bullet \star_{B'} [X_1]
\end{equation}
where $B,B' \in \Gamma_A$.
\end{definition}

\begin{remark}
Of course, just like in the case of classical Massey product, one can consider many various improvements: defining systems that consist k-tuples of matrices, elements of non-uniform degrees, ideals other then $\II_A$,... etc. The definition we give here is based on the ease of presentation (and what is needed for the proof of Theorem \ref{thm:1}) and is by no means the most general. 
\end{remark}

\subsection{The pearl complex}
Much like Definition \ref{def:masseyproduct}, Massey products are actually defined in terms of an underlying algebraic structure and have an intimate connection with its deformation theory. However this time instead of a DG-algebra, the underlying cochain complex is endowed with the \textbf{pearl $A_\infty$-algebra structure} 
\begin{equation}
QP^\bullet(M,\omega) = (CM^\bullet(M;\Gamma),\mu^d_A).
\end{equation}
Let us recall the definition and some key properties, which would serve as a basis for our construction in \ref{subsec:asketch}. \\

Let $(f,g)$ be a Morse-Smale pair on $M$. We write $\nabla f$ for the gradient vector field of $f$ with respect to $g$. We will exclusively work with the negative gradient flow of $f$, namely the flow of $-\nabla f$. To keep grading cohomological, we define $\deg(p) = n - |p|$ where $|p|$ is the Morse index of a critical point $p \in \crit(f)$ (i.e., $\deg(p) := \dim  W^s(p)$ the dimension of the stable manifold.) Additively, we define $QP^\bullet(M,\omega)$ to be an extension of scalars of the usual Morse complex $CM^\bullet(M;f,g) \otimes_\Z \Gamma$. As usual, the generators are critical points of $f$ and the differential $\mu^1$ can be defined by counting Morse trajectories between critical points of degree difference 1. As is well-known, there is a chain-refinement of the familiar cup product, in the form of a bilinear operation $\mu_0^2 : QP^i \otimes QP^j \to QP^{i+j}$. Given any three generators $p_2,p_1,p_0 \in \crit(f)$ with $\deg(p_0) = \deg(p_1) + \deg(p_2)$, the coefficient of $p_0$ in $\mu^2(p_2,p_1)$ counts (perturbed) Y-shaped graphs with asymptotics conditions on each leg corresponding to one of the critical points (see Figure \ref{fig:morseoperationsB}.) 

\begin{figure} 
  \begin{subfigure}[b]{.45\linewidth}
		\centering
				\fontsize{0.25cm}{1em}
			\def\svgwidth{4cm}
			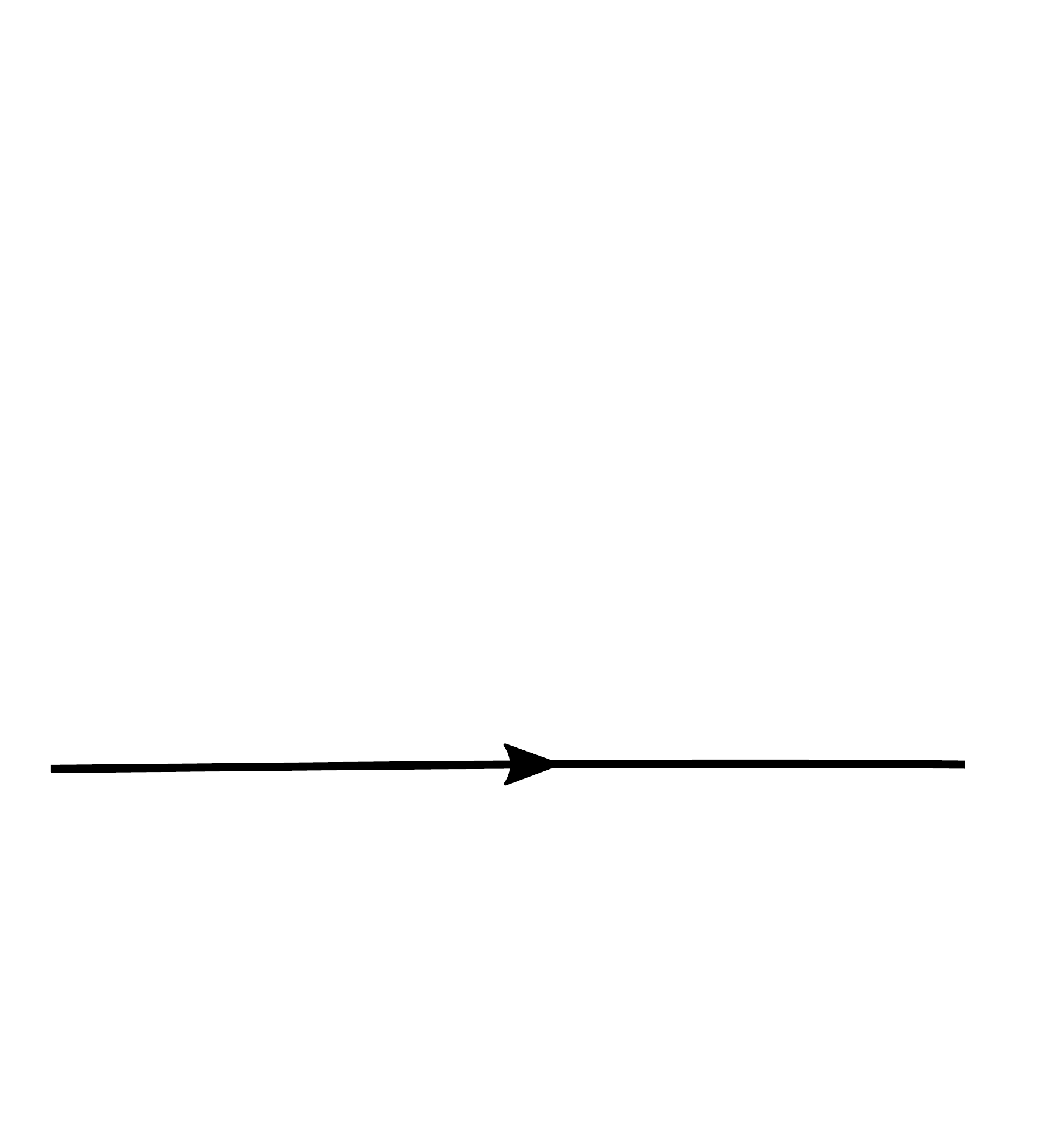
			\caption{The differential $\mu^1$}
			\label{fig:morseoperationsA}
			\vspace*{8mm}
  \end{subfigure}\hfill
  \begin{subfigure}[b]{.45\linewidth}
	\centering
					\fontsize{0.25cm}{1em}
			\def\svgwidth{4cm}
			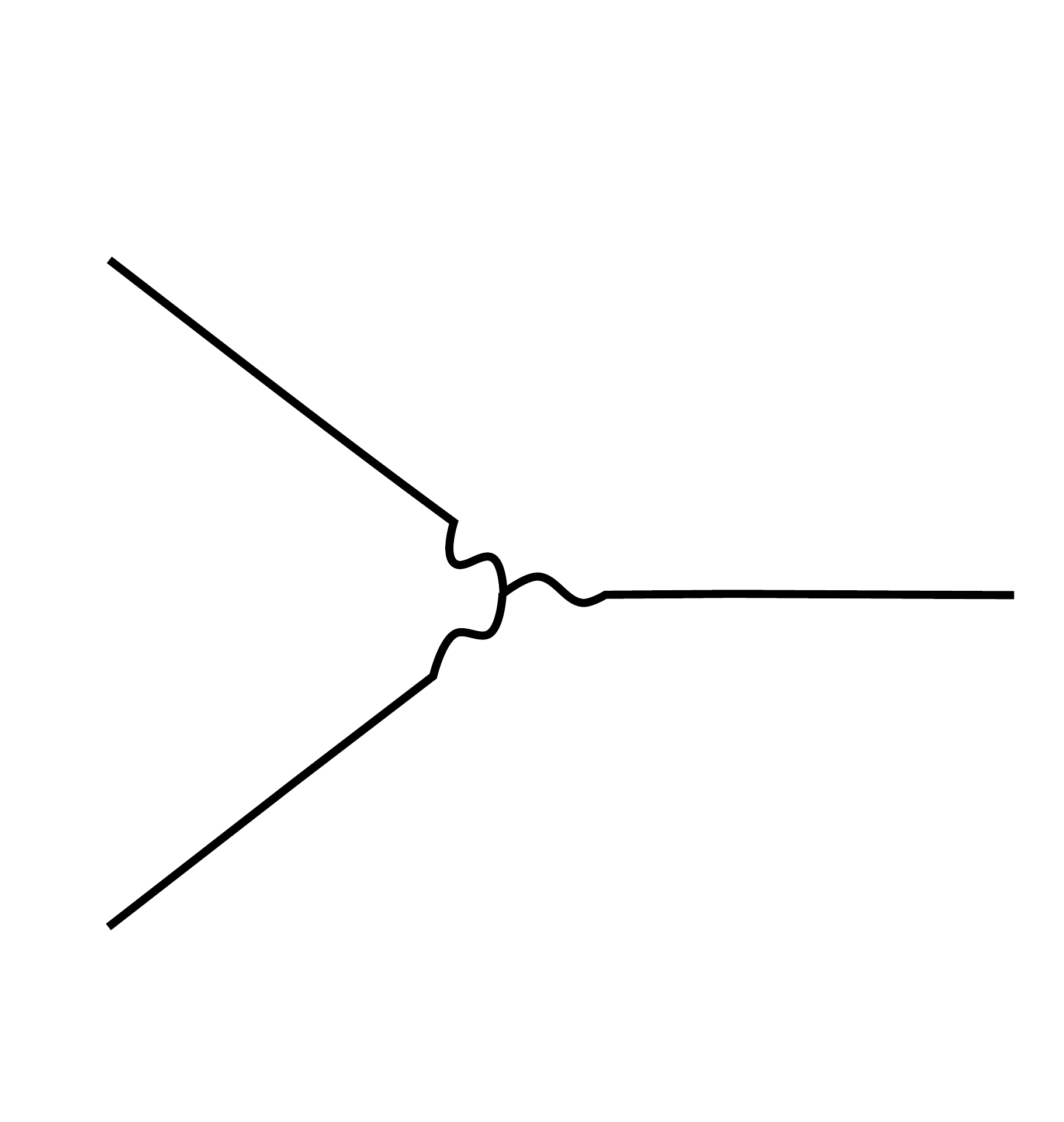 
			\caption{The product $\mu_0^2$}
			\label{fig:morseoperationsB}
			\vspace*{8mm}
  \end{subfigure} 
	  \begin{subfigure}[b]{.45\linewidth}
	\centering
					\fontsize{0.25cm}{1em}
			\def\svgwidth{4cm}
			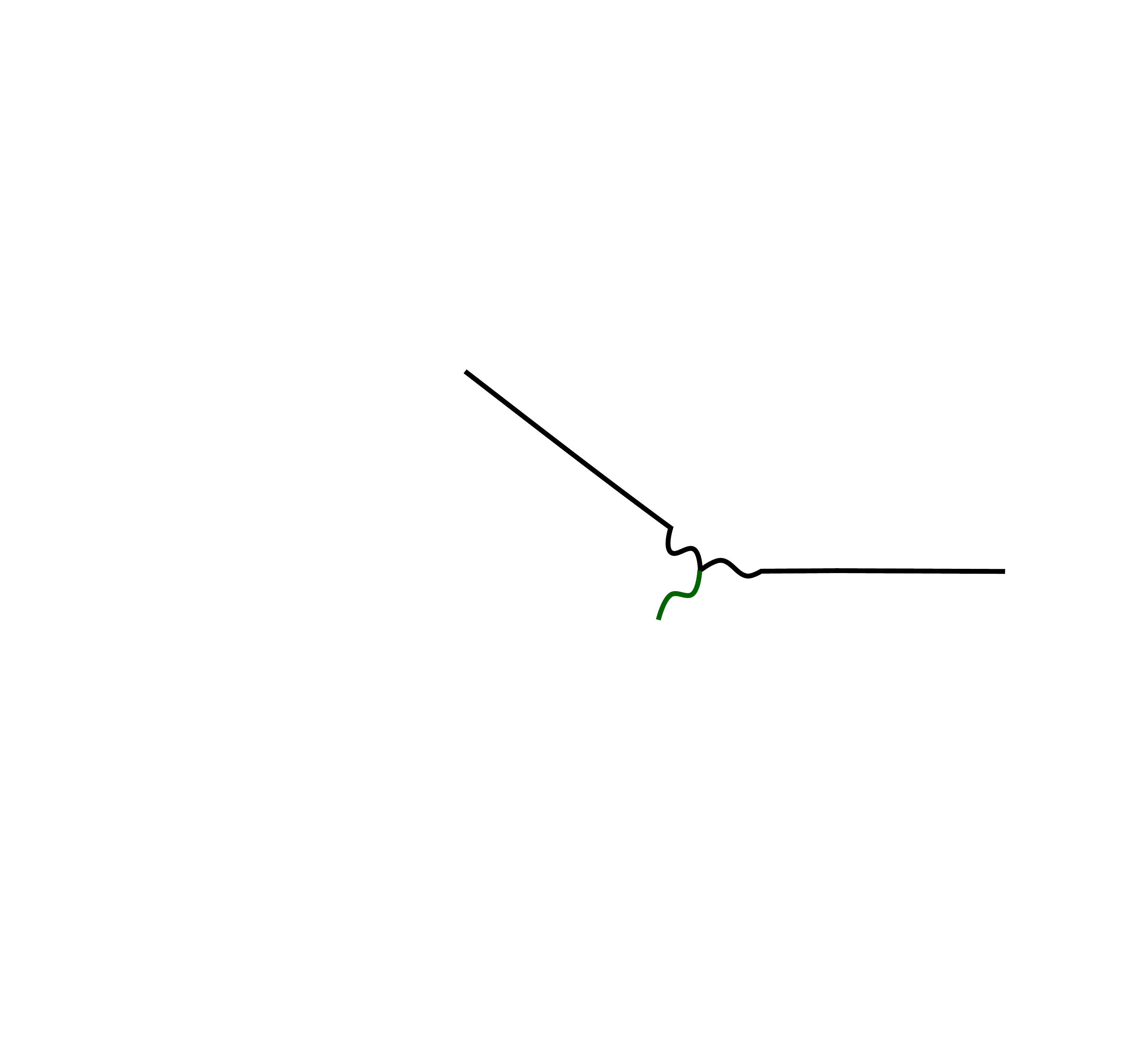
			\caption{Higher operation $\mu_0^3$}
			\label{fig:morseoperationsC}
			\vspace*{2mm}
  \end{subfigure} \hfill
	  \begin{subfigure}[b]{.45\linewidth}
	\centering
					\fontsize{0.25cm}{1em}
			\def\svgwidth{4cm}
			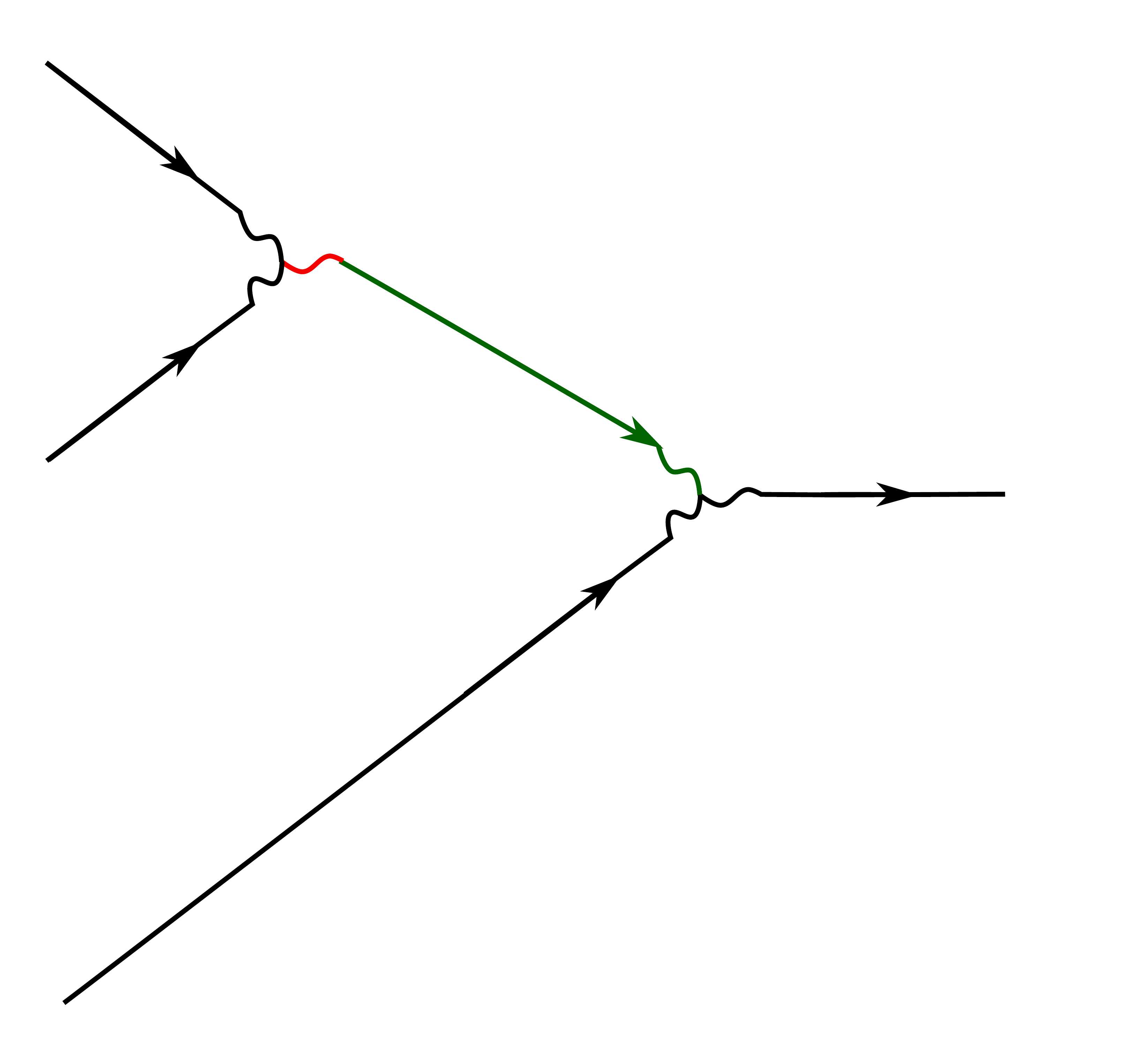
			\caption{Higher operation $\mu_0^3$}
			\label{fig:morseoperationsD}
			\vspace*{2mm}
  \end{subfigure}
	\caption{The Morse complex}
	\label{fig:morseoperations} 
\end{figure}	
Inspired by Feynman diagrams in physics, Fukaya (\cite[Section 1]{Fukaya19971}) and Betz-Cohen (see the original \cite{MR1270436} and \cite{MR2276951},\cite{MR2555936},\cite{MR3004282}) realized that this is just the tip of the proverbial iceberg. As it turns out, many homotopy-theoretic operations (e.g. string topology, Steenrod squares,...) have a simple realization in Morse theory by integrating over suitable \emph{families} of graphs. For instance, there is a infinite sequence $\mu_0^d : QP^{\otimes d} \to QP$ of multi-linear maps of degree $2-d$ that extends the differential and product. As before, on generators $\morseLabel = (p_d,\ldots,p_1,p_0) \subset \crit(f)$ with degree $\sum_{i=1}^d \deg(p_i) - \deg(p_0) = d-2$, the coefficients of $p_0$ in $\mu^d(p_d,\ldots,p_1)$ is defined by counting (perturbed) gradient system modelled after Stasheff trees with asymptotics conditions on each leaf corresponding to one of the input critical points $p_i$, $1 \leq i \leq d$, and on the root $p_0$. 
\begin{example}
The $\mu_0^3$-operation is defined by integrating over the two families of graphs in Figures \ref{fig:morseoperationsC}, \ref{fig:morseoperationsD}. They have the same familiar features as before with the addition of a Morse trajectory of finite length (in green). The length parameter is allowed to vary, which results in two 1-dimensional families meeting along the codimension one boundary corresponding to the ''corolla": the unique quadvalent tree.  
\end{example}
The moduli space of all such trees with $(d+1)$-edges is one of the most familiar versions of the Associahedron, so it should come as no suprise that $\mu_0^d$ satisfy the $A_\infty$-axioms. 
\begin{remark}
If so inclined, the reader may choose at this point to skip ahead to Section \ref{sec:toymodel}, which contains a technical discussion of the Morse $A_\infty$-algebra, as well as a description of the coherent perturbation scheme (one of many possible) that we use.
\end{remark}
In \cite{Fukaya19971} and \cite{MR1324703}, the authors note that just as the quantum product can be thought of as deforming the usual cup product, we can deform the Morse $A_\infty$-algebra by constructing a suitable bounding cochains with new terms $\mu_A^d$, $A \in H^S_2(M)$ coming from \textbf{pseudoholomorphic pearl trees} of homology $A$: these are configurations of gradient trajectories similar to Stasheff trees, but instead of meeting at a single vertex, they meet at pseduo-holomorphic spheres \emph{along the equator}. The homology classes of the holomorphic curves is required to sum up to $A$. For example, Figure \ref{fig:mu3a} should be thought of as ''quantization" of Figures \ref{fig:morseoperationsC} and \ref{fig:morseoperationsD}. 

\begin{theorem} \label{thm:existencepearlM}
Let $(M,\omega)$ be a monotone symplectic manifold. Fix a Morse-Smale pair $(f,g)$ on $M$. 
Then for any $d \geq 2$ and any $A \in H^S_2(M)$, there exists a sequence of multi-linear maps 

\begin{equation}
\mu^d_A : CP^\bullet(M;\omega)^{\otimes d} \to CP^\bullet(M;\omega) \vspace{0.5em}
\end{equation}

of degree $2-d-2c_1(A)$ which satisfies (declare $\mu_1^0  =\mu^1$): \vspace{0.5em}
\begin{enumerate}
\item
$(CP^\bullet(M;\omega),\mu^d_0)$ is the Morse $A_\infty$-algebra. \vspace{0.5em}
\item
If we define $\mu^d := \sum_A \mu^d_A \cdot e^A$ then $(CP^\bullet(M;\omega),\mu^d)$ satisfies the $A_\infty$-axioms. \vspace{0.5em}
\item
The cohomology algebra of the projective $A_\infty$-algebra from (2) is isomorphic to $QH^\bullet(M,\omega)$.
\end{enumerate}
\end{theorem}
\begin{proof}
This is a special case of \cite[Theorem 1.9]{eprint3}. 
\end{proof}
\begin{remark}
We will provide a proof of a generalization of this statement (still in the monotone case) in two dedicated Sections: \ref{sec:parametrized}, and \ref{sec:masseyproductquantum}.
\end{remark}

Proposition \ref{prop:multilinear} follow immediately from Theorem \ref{thm:existencepearlM}, coupled with the actual algebraic Definition of Quantum Massey products (which can easily be deduced from the parametrized one that appears in section \ref{sec:algdefinition}).  \\

\begin{remark}
An alternative point of view which is very popular in Floer-theoretic literature is to think of this as representing a Morse-Bott type complex for the self-Floer homology of the diagonal Lagrangian $\Delta \subset M \times M$ which is equipped with an anti-symplectic involution. This type of construction has appeared in (partial list): Biran-Cornea-Lalonde \cite{MR2555932},\cite{MR2200949}; Woodward-Charest \cite{eprint1},\cite{eprint2}; Fukaya-Oh \cite{Fukaya19972} and Fukaya-Oh-Ohta-Ono \cite{eprint3}; Jiayong Li PhD Thesis \cite{MR3517817}; as well as many others.
\end{remark}

We state without a proof the fact that the formality of this $A_\infty$-algebra implies the vanishing of all Massey products and quantum Massey products (essentially the same argument as Lemma \ref{lem:independence}.)

\begin{remark}
There is an unfortunate clash in the nomenclature among different authors (which is to be expected). We wish to remark upon one distinction between our notation and existing conventions that might be confusing: the $\mu^k_A$ (or homological-perturbation thereof) of our $QP^\bullet$ complex (and maybe ought to call a \textbf{quantum rational homotopy type} or something similar) is called quantum Massey products in \cite{Fukaya19971}. 
\end{remark}

\subsection{Fibrations over the circle}
We generalize the results of the previous subsections to families of symplectic manifolds. 

\subsubsection{Mapping tori and deformations}

By a graded $\Q$-algebra 
\begin{equation}
R = \bigoplus_{i \in \Z} R^i
\end{equation}
we mean one which is finite-dimensional, graded-commutative and unital. All homomorphisms are presumed unital as well.

\begin{definition}
A first order infinitesimal deformation of $R$ of dimension $d > 0$ consists of
\begin{itemize}
\item[(a)]
A graded $\Q$-algebra $\tilde{R}$,
\item[(b)]
An element $t \in \tilde{R}$ with $t^2=0$, such that the sequence
\begin{equation} \label{eq:sequence}
0 \to \tilde{R} / t \cdot \tilde{R} \stackrel{\cdot t}{\longrightarrow}  \tilde{R} \longrightarrow \tilde{R} / t \cdot \tilde{R} \to 0
\end{equation}
is exact, and
\item[(c)]
A homomorphism of graded algebras $j : \tilde{R} \to R$ which is surjective with kernel $t \cdot R$.
\end{itemize}
\end{definition}
The exactness of the sequence \eqref{eq:sequence} is equivalent to the flatness of $\tilde{R}$ as a
module over $\field[\epsilon]/(\epsilon^2)$, where $\epsilon$ acts by multiplication with $t$.

\begin{definition}
Given two graded algebras $R_1,R_2$ and a homomorphism $f : R_1 \to R_2$, one defines a morphism over $f$ from a $d$-dimensional deformation $(\tilde{R}_1,t_1,j_1)$ of $R_1$ to a d-dimensional deformation $(\tilde{R}_2,t_2,j_2)$ of $R_2$ to be a homomorphism of graded algebras $\tilde{f} : \tilde{R}_1 \to \tilde{R}_2$ such that $\tilde{f}(t_1) = t_2$ and $j_2 \tilde{f} = f j_1$.
\end{definition}

In the special case $R_1 = R_2 = R$ and $f = id$, $\tilde{f}$ is called a morphism of deformations of $R$. All morphisms of deformations of $R$ are isomorphisms. It is possible to use $Def_d(R)$ to associate a natural invariant to any diffeomorphism that acts trivially on cohomology.


\begin{definition} 
Let $\phi \in \Diff^+(M)$. The \textbf{mapping torus} $M_\phi$ is defined as the quotient of $\R \times M$ by the free $\Z$-action 
\begin{equation}
n \cdot (t,x) = (t-n,\phi^n(x)). 
\end{equation}
It is naturally a smooth fiber bundle $\pi : M_\phi \to S^1$
\end{definition} 

\begin{lemma} \label{lem:acttrivialy}
Assume that $\phi$ acts trivially on cohomology. Then $H^\bullet(M_\phi;\Q)$ is a one-dimensional deformation of $H^\bullet(M;\Q)$. 
\end{lemma}
\begin{proof}
Immediate from triangle \eqref{eq:triangle3} in the appendix.  
\end{proof}

Simply put, we think of the monodromy $\phi$ as deforming the cohomology algebra of the trivial fibration $H^\bullet(M \times S^1)$, and obtain a class in second Harrison cohomology group which classifies the deformation. 

\subsubsection{Parametrized Gromov-Witten invariants} \label{subsubsec:parametrizedqcoh}
A \textbf{locally Hamiltonian fibration} (LHF) is a triple $(E, \pi,\Omega)$ where $\pi : E \to B$ is a smooth fibre bundle over a compact connected manifold $B$, with a closed 2-form $\Omega \in Z^2(E)$ on the total space, such that $\Omega_b := \Omega|_{\pi^{-1}(b)}$ is non-degenerate on $E_b := \pi^{-1}(b)$, for each $b \in B$. Sometimes, when the map is clear from the context, we will just write $(E, \Omega)$. A symplectic fibration with typical fibre $(M, \omega)$ underlies a locally Hamiltonian fibration if and only if the cohomology class $[\omega] \in H^2(M;\mathbb{R})$ extends to a class on the total space. Note that this is always the case when the base $B = S^1$. We will always assume that all of our LHF's are \textbf{monotone}, which means the fiber $(M,\omega)$ is a monotone symplectic manifold. \\

Let $V = T^v E$ the vertical tangent bundle. A domain-dependent almost complex structure on $\pi$ is a section $\textbf{J} = (J_{z,b}) : S^2 \times B \to Aut(V)$ such that $\textbf{J}^2= -id$. A vertical $\textbf{J}$-holomorphic map is a pair $(b,u)$ such that $ b \in B$ and $u : S^2 \to E_b$ is $\textbf{J}^b$-holomorphic in the usual sense (where $\textbf{J}^b$ denotes the family $(J_{z,b})_{z \in S^2}$ with a fixed $b = b_0$.) Then we can consider the moduli space of all vertical $\textbf{J}$-holomorphic spheres in homology class $A \in H_2(M; \Z)$ (we assume the action of $\pi_1(B)$ on the second homology of the fiber is trivial, so as to make the identification unproblematic), with $k$ marked points. As a set, this is just the union
\begin{equation}
\MM_{0,k}(A;\pi,\Omega;\textbf{J}) = \bigcup_{b} \MM_{0,k}(A;E_b,\omega_b;\textbf{J}^b)
\end{equation}
\begin{theorem}
For generic $\textbf{J}$, it is a smooth manifold of dimension $\dim(E) + 2c_1(A) +2k-6$, and the image of the k-point evaluation map 
\begin{equation}
ev_k : \MM^*_{0,k}(A;\pi,\Omega;\textbf{J}) \to E \times_\pi \ldots \times_\pi E
\end{equation}
is a pseudocycle. 
\end{theorem}
\begin{proof}
See Theorem \ref{thm:pseudocycle} and its proof. 
\end{proof}
\begin{remark}
Note that the meaning of generic here is a generic family, so in each fiber the almost complex structure is non-generic!
\end{remark}
So we can define the parametrized Gromov-Witten invariant $GW^\pi_{k,A}$ in the same way as before. 

\begin{remark} 
These invariants and their basic properties were studied (in one form or another) in the papers \cite{MR2218350}, \cite{MR2509705},\cite{MR2680275} specifically geared toward detecting exotic symplectomorphism. This is a very natural object to consider in this context: Denote $\Symp_0(M,\omega)$ for the intersection $\Symp(M,\omega) \cap \Diff^+_0(M)$, and 
\begin{equation}
\Omega = \left\{\omega' \in \Omega^2(M) \: \big| \: \omega' \text{ is symplectic and }[\omega] = [\omega']\right\} 
\end{equation}
Define $\Omega_{\omega}$ to be the path component of $\Omega$ containing $\omega$ (i.e., the subset of all $\omega' \in \Omega$ which can be connected to $\omega$ via a path of cohomologous symplectic forms). Following \cite{unpublishedpreprint}, we consider the action of $\Diff^+_0(M)$ on $\Omega$ by pullback 
\begin{equation}
\phi \cdot \omega  = (\phi^{-1})^*(\omega).
\end{equation}
The stablizer of the action is $\Symp_0(M,\omega)$, and Moser theorem tells us that the action is transitive on $\Omega_{\omega}$ which is just the homogeneous space $\Diff^+_0(M) / \Symp_0(M,\omega)$. Let $\JJ_{\omega}$ be the space of almost complex structures that are tamed by some symplectic form in $\Omega_{\omega}$. From \cite[Lemma 2.1]{MR1929338}, we know that $\JJ_{\omega}$ and $\Omega_{\omega}$ are homotopy equivalent. So essentially, the same argument that tells the ordinary GW invariants are independent of the choice of almost complex structures shows that parametrized GW invariants depend only on the homotopy class of the family of almost complex structures. 
\end{remark}

The obvious analogues of Kontsevich-Manin axioms hold including WDVV which means that we can define an associative, graded-commutative algebra 
\begin{equation}
QH^\bullet(\pi,\Omega)
\end{equation}

as in Proposition \ref{thm:propquantumcohomology}. We denote the product as $\star$ again. Note that the $A=0$ term $\star_0$ is the usual cup product on the total space $E$.

Of course, quantum Massey products and quantum matrix Massey products generalize immediately to this setting as well.  

\subsubsection{Quantum cohomology of mapping tori} 
Let $(M,\omega)$ be a (monotone) symplectic manifold. From now on, we also assume that $\pi_1(M)=\left\{0\right\}$. Let $\phi : M \to M$ a symplectomorphism. We can equip the mapping torus $\pi : M_\phi \to S^1$ with a closed two-form $\Omega := \omega_\phi$, induced from the pullback of $\omega$ under $\R \times M \to M$, which makes it into a monotone LHF. Then $QH^\bullet(M_\phi)$ has all the usual properties outlined in the previous subsection. We will often use the shorthand $QH^\bullet(M_\phi)$ for $QH^\bullet(\pi,\Omega)$ and $\tilde{GW}$ instead of $GW^\pi$ whenever it is understood from the context we are referring to the LHF defined as the mapping tori of a symplectomorphism. \\

Slightly more interesting, denote $R = H^\bullet(M;\Q)$. Then we can consider the term $\star_A$ of the quantum product as giving a bilinear map 
\begin{equation}
\psi_A : R \times R \to R
\end{equation}
which is (graded) commutative, of degree $-2c_1(A)$ and has the following formal properties: \vspace{0.5em}
\begin{enumerate}
\item
$\psi_A(x,y) = x \cup y$ when $A=0$.  \vspace{0.5em}
\item
For every $A,A',B \in \Gamma$, we have an associtivity relation
\begin{equation}
\sum_{A + A' = B} \psi_{A'}(\psi_{A}(x,y),z) = \sum_{A + A' = B} \psi_A(x,\psi_{A'}(y,z))
\end{equation}
\vspace{0.5em}
\item
$\psi_A(1,x) = 0$ for all x, and  \vspace{0.5em}
\item
If $u \in R^2$ satisfies $\langle u , A \rangle = 0$ then $\psi_A(u,x) = 0$ for all $x$. 
\end{enumerate}
Similarly, denote $\tilde{\psi}_A$ for the bilinear operation defined on $H^\bullet(M_\phi;\Q)$ in the analogues manner.

\begin{definition}[4.1 in \cite{MR1736220}] \label{def:extension}
Let $(\tilde{R},t,j )$ be a deformation of $R$. An \textbf{extension} of $\psi_A$ to $\tilde{R}$ is a bilinear map $\tilde{\psi}_A : \tilde{R} \times \tilde{R} \to \tilde{R}$ of degree $-2c_1(A)$, which is (graded) commutative and has the following properties: \vspace{0.5em}
\begin{enumerate}
\item
$\tilde{\psi}_A(x,y) = x \cup y$ when $A=0$.  \vspace{0.5em}
\item
For every $A,A',B \in \Gamma$, we have an associtivity relation
\begin{equation}
\sum_{A + A' = B} \tilde{\psi}_{A'}(\tilde{\psi}_{A}(x,y),z) = \sum_{A + A' = B} \tilde{\psi}_A(x,\tilde{\psi}_{A'}(y,z))
\end{equation}
 \vspace{0.5em}
\item
If $u \in \tilde{R}^2$ satisfies $\langle u , A \rangle = 0$ then $\tilde{\psi}_A(u,x) = 0$ for all $x$.  \vspace{0.5em}
\item
$j(\tilde{\psi}_A(x,y)) = \psi_A(j(x),j(y))$. 
\end{enumerate}
\end{definition}
Of course, the same claim holds if we replace $\Gamma$ with $\Gamma/\II$ \emph{mutatis mutandi}. \\

As the notation suggests, a minor modification of Proposition 4.2 in \cite{MR1736220} leads to 
\begin{proposition} \label{prop:4.2}
Given a symplectomorphism $\phi : M \to M$ that acts trivially on cohomology, $\tilde{\psi}_A$ is an extension of $\psi_A$. \noproof
\end{proposition}

However, this approach has limited scope, because we should not expect ''most" interesting symplectomorphism to act non-trivially on quantum cohomology. So we should ask ourselves the question: what happens when the deformation class is trivial?

\subsection{The definition of $q\tau_2$: a sketch}  \label{subsec:asketch}
The idea for proving that the monodromy in \ref{thm:1} has infinite order in the kernel is to take a chain-level enhancement of the previous subsection. \\

We associate to every symplectic automorphism $\phi: M \to M$ a cochain complex
\begin{equation}
\Symp(M,\omega) \: \: \xrsquigarrow{\tiny{\: mapping torus \;}} \: \: (M_\phi , \Omega) \: \: \xrsquigarrow{\tiny{\: pearl complex \: }} \: \: \mathfrak{X}_\phi \vspace{0.3em}
\end{equation}

by taking the parametrized pearl complex construction $\mathfrak{X}_\phi := QP^\bullet(\pi,\omega_{\phi})$ of the mapping torus $M_\phi \stackrel{\pi}{\rightarrow} S^1$. 

\begin{proposition} \label{prop:existencepearlMphi}
Let $(M,\omega)$ be a closed, monotone and simply connected symplectic manifold and $\phi : M \to M$ a symplectomorphism. Then there exists a sequence of multi-linear maps 

\begin{equation}
\mu^d_A : \mathfrak{X}_\phi^{\otimes d} \to \mathfrak{X}_\phi \: , \: (d,A) \in \N \times H_2^S(M)\vspace{0.5em}
\end{equation}

of degree $2-d-2c_1(A)$ which satisfies (declare $\mu_1^0  =\mu^1$): \vspace{0.5em}
\begin{enumerate}
\item
$(\mathfrak{X}_\phi,\mu^d_0)$ is the Morse $A_\infty$-algebra of $M_\phi$. \vspace{0.5em}
\item
If we define $\mu^d := \sum_A \mu^d_A \cdot e^A$ then $(\mathfrak{X}_\phi,\mu^d)$ is a projective $A_\infty$-algebra. \vspace{0.5em}
\item
The cohomology algebra of (2) is isomorphic to $QH^\bullet(\pi , \omega_\phi)$.
\end{enumerate}
Moreover, 
\begin{enumerate} \addtocounter{enumi}{3}
\item
An equivalence between $M_{\phi_0}$ and $M_{\phi_1}$ defines an $A_\infty$-quasi-isomorphism 
\begin{equation}
\Phi : \mathfrak{X}_{\phi_0} \to \mathfrak{X}_{\phi_1}.
\end{equation}
\end{enumerate}
\end{proposition}
\begin{proof}
For (1) and (2) see Proposition \ref{prop:ainftypearl} and the discussion preceding it. (3) follows immediately from the way we define the pearl complex. See Section \ref{subsec:transversality}, and Theorem \ref{thm:pseudocycle}. Finally, (4) is proved in Section \ref{sec:independence}. 
\end{proof}

This makes the first statement in Theorem \ref{thm:kindof} precise, once we know that any symplectic isotopy $\left\{\phi_t\right\}_t$ defines an equivalence between $M_{\phi_0}$ and $M_{\phi_1}$ (Lemma \ref{lem:simplyconnectedequivalence}.) We now make a short, and somewhat impressionistic digression into homological algebra. 
\begin{remark}
A detailed exposition of these ideas is given in Section \ref{sec:algebraic}. 
\end{remark}
Let 
\begin{equation}
Q := QH^\bullet(M;\Gamma) \otimes H^\bullet(S^1)= QH^\bullet(M;\Gamma)[\mathfrak{t}]/(\mathfrak{t}^2) 
\end{equation}
considered as a graded algebra ($\mathfrak{t}$ is a formal variable of degree one).
\begin{definition}
A \textbf{1-dimensional deformation of quantum cohomology} is a triple 
\begin{equation}
(\BB,\mu^d,F) 
\end{equation}
which consists of: \vspace{0.2em}
\begin{itemize}
\item
A projective, graded, $\Gamma$-module $\BB$. \vspace{0.2em}
\item
An $A_\infty$-structure $\left\{\mu_A^d\right\}$ on $\BB$. \vspace{0.2em}
\item
A fixed isomorphism (of graded algebras) $F : Q \to H^\bullet(\BB)$. \vspace{0.2em}
\end{itemize}
\end{definition}

\begin{example} 
One possible refinement is the ''trivial"\footnote{Note that it was recently discovered that quantum cohomology of K\"{a}hler varieties does not have to be formal \cite{evans2015generating}! Therefore $QH^\bullet(M,\omega)$ might have interesting deformations. Here we mean trivial with respect to the bimodule structure on $H^\bullet(S^1)$.} $A_\infty$-structure $\BB = QH^\bullet(M;\Gamma) \otimes H^\bullet(S^1)$. 
\end{example}

There is an equivalence relation between such refinements called \textbf{compatible $A_\infty$-quasi-isomorphism}. 

\begin{remark}
Note that for an arbitrary $\phi :M \to M$, there is no reason to expect that the quantum cohomology of a mapping torus $QH^\bullet(M_\phi)$ is a 1-parameter deformation of $QH^\bullet(M;\omega)$ (in fact, they might even have different ordinary cohomology.) This fits well with the general picture of the higher Johnson homomorphism -- which is defined only for automorphism that lie in the kernel of the previous Johnson homomorphisms. See Section \ref{subsec:qjohnson} for more. 
\end{remark}

We think of the infinite dimensional space of all chain-level refinement of quantum cohomology modulo equivalence as moduli space, and consider the pearl structure as a \textbf{deformation} of the canonical base point, represented by the trivial $A_\infty$-structure. Kadeishvili observed that there is a $\Q^*$-action on this moduli space contracting everything to the formal base point. This implies that the tangent space to the base point already contains a lot of information about the moduli space. Indeed, it follows from this line of logic that any chain level refinement $\AA = (Q,\mu^d)$ has a well defined \emph{obstruction} or \emph{Kaledin} class
\begin{equation}
\oo_\AA^3 \in HH^3(Q,Q)^{-1}.
\end{equation}
The right-hand side is a certain group called \textbf{Hochschild cohomology}, see \cite{MR2372207,MR2578584}. Any compatible $A_\infty$-quasi-isomorphism induces a map on $HH^3(Q,Q)^{-1}$ which sends the respective obstruction classes to each other (Lemma \ref{lem:homological perturbation}). When $\AA$ is \textbf{formal}, which means equivalent to the trivial structure, $\oo_\AA^3=0$. The obstruction class is also known as \textbf{the universal Massey product} because the value of any triple Massey product defined for an $A_\infty$-structure $\AA = (Q,\mu^d)$, depends only on the obstruction class (see Lemma \ref{lem:independence}) so we can think of them as giving a way to evaluate the $\oo_\AA^3$ class on suitable collections of cohomology classes.  \\

Back to our setting, given any spherical homology class $A \in H_2^S(M)$ we can repeat the same story above with $\Gamma$ replaced by $\Gamma / \II_A$. Then for any $\phi$ that satisfies 
\begin{equation}  \label{eq:conditionstar}
QH^\bullet(M_\phi,\omega_\phi; \Gamma/\II_A) \iso Q,  \tag{$\bigstar$}
\end{equation} 
it follows that $\mathfrak{X}_{\phi}$ is a 1-dimensional deformation of quantum cohomology. Thus, there is a universal Massey products
\begin{equation}
\oo_\phi^{3,A} \in HH^3(Q,Q)^{-1}. 
\end{equation}

Properties (1),(2) and (3) of the second part of \ref{thm:kindof} are immediate. It will be clear from the Definition of the complex that $\phi = id$ satisfies the assumption and $\oo_id^{3,A} = 0$. We also remark that condition \eqref{eq:conditionstar} is invariant under isotopy (via the standard continuation argument for parametrized quantum cohomology.)

\begin{definition}
We define $\KK_{2,A} \subset \pi_0 \Symp(M,\omega)$ to be subset of all isotopy classes of symplectomorphism that satisfy \eqref{eq:conditionstar}. 
\end{definition}

\begin{lemma}
Assume that $[\phi] \in \KK_{2,A}$ and $\phi$ is symplectically isotopic to the identity. Then $\oo_\phi^{3,A} = 0$. 
\end{lemma}
\begin{proof}
Any symplectic isotopy $\left\{\phi_t\right\}$ determines an equivalence between $M_{\phi_0}$ and $M_{\phi_1}$. But given any such equivalence, $\Phi$ from Proposition \ref{prop:existencepearlMphi} is a compatible $A_\infty$-quasi-isomorphism. Since $\oo_{id}^{3,A} = 0$ this implies that $\oo_{\phi}^{3,A}$ vanishes as well. 
\end{proof}

\begin{corollary}
Under the same assumptions, the existence of a non-trivial Massey product in $QH^\bullet(Y_\phi)$ immediately implies that $0 \neq [\phi]$ in the symplectic mapping class group.
\end{corollary}

\subsection{Outline of the main argument} \label{subsec:mainargument}

\subsubsection{Construction of the family} \label{subsubsec:blowupinfamilies}

\begin{definition} 
A \textbf{nodal family of Riemann surfaces} is a pair $(\pi : \mathcal{C} \to B)$, where: \vspace{0.5em}
\begin{itemize}
\item[(a)]
$\mathcal{C}$ and $B$ are complex analytic manifolds of dimension 
\begin{equation}
\dim_\C(\mathcal{C}) = \dim_\C(B)+ 1;
\end{equation}
Denote $s := \dim_\C(B) = \dim_\C(\mathcal{C})-1$.  \vspace{0.5em}
\item[(b)]
$\pi : \mathcal{C} \to B$ is a proper holomorphic map,   \vspace{0.5em}
\end{itemize}
such that for every $z \in \mathcal{C}$, there exist holomorphic coordinates $(t_0,\ldots , t_s)$, around $z$ in $\mathcal{C}$ and $(v_1,\ldots , v_s)$ around $w = \pi(z)$ in $B$, mapping $z$ to $0 \in \C^{s+1}$ and $\pi(z)$ to $0 \in \C^s$, respectively, and such that in these coordinates, $\pi$ is given one of two local models
\begin{equation} \label{eq:regular}
(t_0 , \ldots , t_s) \to (t_1, \ldots , t_s) 
\end{equation}
or
\begin{equation} \label{eq:node}
(t_0 , \ldots , t_s) \to (t_0 t_1, \ldots , t_s)
\end{equation}
In \eqref{eq:regular}, $w$ is called a \textbf{regular point}, while in the case that \eqref{eq:node} holds, $w$ is a \textbf{nodal point}.
\end{definition}

We say that the family has genus $g$ if every fiber has arithmetic genus $g$.

\begin{definition}
A \textbf{nodal family of space curves} is a nodal family of Riemann surfaces $(\pi : \mathcal{C} \to B)$ of genus $g$ with the additional data of a holomorphic embedding 
\begin{equation}
\mathcal{C} \hookrightarrow  \P^{3}_B := \P^3 \times B
\end{equation}
such that $\pi : \mathcal{C} \to B$ factors as the composition of the embedding and the projection onto the second factor.
\end{definition}
Let $\Delta \subset \C$ be the unit disk, $\Delta^* = \Delta \setminus \left\{0\right\}$. We denote 
\begin{itemize}
\item 
$\overline{\mathcal{P}}$ for the product $\P^{3} \times \Delta$.
\item 
$\overline{\mathcal{Q}}$ for the product $\P^{1} \times \P^{1} \times \Delta$.
\end{itemize}
We consider them as fiber bundles over the disk by projection to the second component. Their restriction to the punctured disk would be denoted as $\mathcal{P}$ and $\mathcal{Q}$ respectively.
\begin{proposition} \label{prop:blowup1}
There exists a nodal family of space curves, denoted $ \overline{\mathcal{C}} \to \Delta$, such that:
\begin{enumerate}
\item 
There are inclusions of closed subvarieties
\begin{equation}
\overline{\mathcal{C}} \hookrightarrow \overline{\mathcal{Q}} \hookrightarrow \overline{\mathcal{P}}
\end{equation}
which respect the bundle structure over the disc (i.e., take fiber to fiber). \vspace{0.5em}
\item 
The restriction to the punctured disk $\overline{\mathcal{C}}^* \to \Delta^*$ is a smooth family, and the fiber $C_s$ over each $s \neq 0$ is a non-hyperelliptic, genus $4$ curve. \vspace{0.5em}
\item 
The isotopy class of the monodromy of $\overline{\mathcal{C}}^*$ is a Dehn twist around a separating curve of genus $2$. \vspace{0.5em}
\end{enumerate}
\end{proposition}
\begin{proof}
See Section \ref{subsec:remarkonbirationalmodels}. 
\end{proof}

The fiberwise normal bundles of embedded family fit together to form a rank $2$ complex vector bundle over the total space, 
\begin{equation}
\NormalBundle^* \to \overline{\mathcal{C}}^*.
\end{equation}
We denote the projectivization of this bundle as
\begin{equation}
\familyExceptionalDivisors = \P_{\familyEmbeddedCurves}(\NormalBundle^*). 
\end{equation}
Let $\overline{\mathcal{Y}}^*$ be the 4-fold obtained by the algebro-geometric blowup of $\overline{\mathcal{C}}^*$ in $\overline{\mathcal{P}}^*$. This is a blowup of a smooth variety at a smooth center, thus $\overline{\familyExceptionalDivisors}^*$ is a smooth divisor in $\overline{\mathcal{Y}}^*$, and since blowup commute with flat base change, the fiber over every $s \in \Delta^*$ is a smooth 3-fold which contains the exceptional divisor $E_s$. 

Fix once and for all a loop $\gamma : S^1 \to \Delta^\ast$ winding positively once around the origin. 
\begin{definition}
We call $\familyEmbeddedCurves := \gamma^* \familyEmbeddedCurves^\ast$ the \textbf{curve fibration}, $\familyExceptionalDivisors := \gamma^* \familyExceptionalDivisors^\ast$ the \textbf{exceptional fibration} and $\familyFanoThreeFolds := \gamma^* \familyFanoThreeFolds^\ast$ the \textbf{3-fold fibration}. 
\end{definition}
They are $S^1$-locally Hamiltonian fibrations. We consider $\gamma(0)$ as our base point and denote the fiber of $\familyEmbeddedCurves$ (respectively, $\familyFanoThreeFolds$) over it as $C$ (resp. $Y$). Then the parallel transport in $\familyFanoThreeFolds$ defines a symplectomorphism which is our $\psi_Y$ from Theorem \ref{thm:1}. Note that the restriction of $\psi_Y$ the curve fibration defines a diffeomorphism of the suface $C$, which we denote as $\phi_C$. By construction, it is a separating Dehn twist of genus $2$. \\

In fact, we will see in subsection \ref{sec:familiesofcurves} that the monodromy $\psi$ can be obtained from a K\"{a}hler degeneration with interesting connections to cubic 3-folds. This would also us to prove

\begin{proposition} \label{prop:blowup2}
There exists a family of 3-folds $\overline{\mathcal{Y}}'$ over $\Delta$ such that: \vspace{0.5em}
\begin{enumerate}
\item
The pullback by $\gamma$ coincides with that of $\familyFanoThreeFolds$. \vspace{0.5em}
\item
In a general fiber, there exist vanishing cycles $\left\{V_1,\ldots,V_5\right\}$ and $\left\{V_1',\ldots,V_5'\right\}$ such that the monodromy is a product of Dehn twist about them. \vspace{0.5em}
\end{enumerate}
\end{proposition}
\begin{proof}
See Corollary \ref{cor:vanishingcycles2} and the discussion preceding it.
\end{proof}

\begin{remark}
The construction in Proposition \ref{prop:blowup2} comes from the theory of ball quotient models, see \cite{MR1949641} and \cite{MR2895186}. It is related to (the more obvious) smoothing which we use to construct the family $\overline{\mathcal{C}}$ of \ref{prop:blowup1} via the Hassett-Keel program. We briefly remark on that in \ref{subsec:remarkonbirationalmodels}. 
\end{remark}

\subsubsection{Proof of Theorem \ref{thm:1}}  \label{subsubsec:mainstatement}

As we have mentioned in the end of \ref{subsec:asketch}, the Theorem is proven if we can show the existence of a nontrivial Massey product in parametrized quantum cohomology. The computation itself is broken into a sequence of steps. \\

Let $C$ be a genus $4$ smooth curve, $Y = Bl_C \P^3$ with blowdown map $b : Y \to \P^3$, and $E$ is the exceptional divisor of the blowdown map $b$. We introduce the following additional cycles: 
\begin{itemize}
\item
$L$ is the preimage of a generic line under $b : Y \to \P^3$.
\item
$H$ is the preimage of a generic hyperplane under $b : Y \to \P^3$.
\item
$F$ is a fixed exceptional fiber.
\item
$\left\{A_i,B_i\right\}$ where $i=1,\ldots,4$ are fixed cycles defining a standard symplectic basis of $H^1(C,\Z) \iso \Z^{\oplus 8}$ equipped with the intersection form. 
\end{itemize}

\textbf{Notation.} We use symbols cycles and their homology classes (this should be clear from context) and denote the Poincare dual cohomology classes by the corresponding small letter e.g. $PD(A) = a$ etc. Finally, as in the local model, we denote $u = -PD(F) = -f$. We treat this cohomology class as a formal variable of degree 2. 

\begin{lemma} \label{lem:cohomologyofY}
We have the following presentation for the cohomology of $Y$:
\begin{itemize}
\item
$H^0 = \Z\left\langle y\right\rangle$, where the class $y = PD(Y)$.
\item
$H^1=0$.
\item
$H^2 = \Z\left\langle h,u \right\rangle$, with $h = PD(H)$ and $u  = -PD(E)$.
\item
$H^3 = \Z\left\langle u a_i,u b_i\right\rangle$ where $i=1,\ldots,4$. Each class is Poincare dual to the total transform of one of the standard cycles in $C$. 
\item
$H^4 = \Z \left\langle \ell,f \right\rangle$, with $\ell = PD(L)$ and $f  = PD(F)$.
\item
$H^5=0$. 
\item
$H^6 = \Z\left\langle pt\right\rangle$, the Poincare dual of the point. 
\end{itemize}
The ring structure is 
\begin{table}[h!]
  \centering
  \resizebox{0.6\textwidth}{!}{ 
	\centering
	\begin{tabular}{| >{$}c<{$} || >{$}c<{$} | >{$}c<{$} | >{$}c<{$} | >{$}c<{$} | >{$}c<{$} | >{$}c<{$} | >{$}c<{$} | >{$}c<{$} |}
	\hline
		\cup & y & h & u & u \cdot a_i & u \cdot b_j  & \ell & f & pt\\ \hhline{|=||=|=|=|=|=|=|=|=|}
		y & y & h & u & u \cdot a_i & u \cdot b_j  & \ell & f & pt\\ \hline
		h & h  & \ell & -6f & 0 & 0 & pt & 0 & 0\\ \hline
		u & u  & -6f & 30f - 6 \ell & 0 & 0 & 0 & pt & 0 \\ \hline
		u \cdot a_i & u \cdot a_i & 0 & 0 & 0 & -\delta_{ij} \cdot pt & 0 & 0 & 0\\ \hline
		u \cdot b_j & u \cdot b_j  & 0 & 0 & \delta_{ij} \cdot pt & 0 & 0 & 0 & 0\\ \hline
		\ell & \ell & pt & 0 & 0 & 0 & 0 & 0 & 0\\ \hline
		f & f & 0 & pt & 0 & 0 & 0 & 0 & 0\\ \hline
		pt & pt & 0 & 0 & 0 & 0 & 0 & 0 & 0\\ \hline
	\end{tabular}}
	\caption{The cohomology ring}
  \label{tab:testtab1}
\end{table}
\end{lemma}

The proof of the Lemma above (as well as some additional claims related to presentation of $Y$ coming from the K\"{a}hler degneration of \ref{prop:blowup2}) appear in Section \ref{sec:cohomologylevelcomputations}. \\

The only classes $A \in H_2(Y;\Z)$ are: $F , R = L-3F$ and $2F, T = L-2F$ (so these are the only classes of interest for computing our Massey product at homology $2F$).

\begin{lemma} \label{lem:3pointexceptioalclass}
Up to permutation and sign, the only nontrivial 3-point Gromov-Witten invariants in the class $F$ are
\begin{equation}
\begin{split}
GW_{F,3}(f,u,u) &= -1,\\
GW_{F,3}(u \cdot a_i,u \cdot b_j,u) &= \delta_{ij}. \\
\end{split}
\end{equation}
The contribution of this class to the quantum product is given by the following table: 
\begin{table}[h!]
  \centering
  \resizebox{0.6\textwidth}{!}{ 
	\centering
	\begin{tabular}{| >{$}c<{$} || >{$}c<{$} | >{$}c<{$} | >{$}c<{$} | >{$}c<{$} | >{$}c<{$} | >{$}c<{$} | >{$}c<{$} | >{$}c<{$} |}
	\hline
		\star_F & y & h & u & u \cdot a_i & u \cdot b_j  & \ell & f & pt\\ \hhline{|=||=|=|=|=|=|=|=|=|}
		y & 0 & 0 & 0 & 0 & 0  & 0 & 0 & 0\\ \hline
		h & 0  & 0 & 0 & 0 & 0 & 0 & 0 & 0\\ \hline
		u & 0  & 0 & -u & u \cdot a_i & u \cdot b_j & 0 & -f & 0 \\ \hline
		u \cdot a_i & 0 & 0 & u \cdot a_i & 0 & \delta_{ij} \cdot f & 0 & 0 & 0\\ \hline
		u \cdot b_j & 0  & 0 & u \cdot b_j & -\delta_{ij} \cdot f & 0 & 0 & 0 & 0\\ \hline
		\ell & 0 & 0 & 0 & 0 & 0 & 0 & 0 & 0\\ \hline
		f & 0 & 0 & -f & 0 & 0 & 0 & 0 & 0\\ \hline
		pt & 0 & 0 & 0 & 0 & 0 & 0 & 0 & 0\\ \hline
	\end{tabular}}
	\caption{Quantum multiplication -- the exceptional fiber}
  \label{tab:testtab2}
\end{table}
\end{lemma}
\begin{proof}
All the entries in the table are computed by Poincare duality from the Gromov-Witten computation in subsection \ref{subsec:F}. 
\end{proof}
\begin{remark}
In fact, we can say a little more in this case -- see Lemma \ref{lem:kpointexceptioalclass}. 
\end{remark}

\begin{lemma} \label{lem:toproveinsubsecT}
Up to permutation and sign, the only nontrivial 3-point Gromov-Witten invariants in the class $T = L-2F$ are
\begin{equation}
\begin{split}
GW_{T,3}(pt,h,h) &= 1 \cdot 1 \cdot 6 =6,\\
GW_{T,3}(pt,h,u) &= -2 \cdot 1 \cdot 6 = -12, \\
GW_{T,3}(pt,u,u) &= (-2)^2 \cdot 6 = 24, \\
GW_{T,3}(\ell, \ell,h)  &= 1 \cdot 21 = 21,\\
GW_{T,3}(\ell, \ell,u)  &= (-2) \cdot 21 = -42,\\
GW_{T,3}(\ell, f,h)  &=  1 \cdot 5  =5,\\
GW_{T,3}(\ell, f,u)  &= (-2) \cdot 5 = -10,\\
GW_{T,3}(f,f,h)  &= 1 \cdot 1 = 1,\\
GW_{T,3}(f,f,u)  &= (-2) \cdot = -2,\\
GW_{T,3}(f, u \cdot a_i , u \cdot a_j) &=0 ,\\
GW_{T,3}(f, u \cdot a_i , u \cdot b_j)  &= \delta_{ij},   \\
GW_{T,3}(f, u \cdot b_i , u \cdot a_j) &= -\delta_{ij}, \\
GW_{T,3}(f, u \cdot b_i , u \cdot b_j) &=0 , \\
GW_{T,3}(\ell, u \cdot a_i , u \cdot a_j)  &=0 , \\
GW_{T,3}(\ell, u \cdot a_i , u \cdot b_j) &= 5\delta_{ij}, \\
GW_{T,3}(\ell, u \cdot b_i , u \cdot a_j) &= -5\delta_{ij}, \\
GW_{T,3}(\ell, u \cdot b_i , u \cdot b_j) &= 0.
\end{split}
\end{equation}
The contribution of this class to the quantum product is given by the following table: 
\begin{table}[h!]
  \centering
  \resizebox{0.9\textwidth}{!}{ 
	\centering
	\begin{tabular}{| >{$}c<{$} || >{$}c<{$} | >{$}c<{$} | >{$}c<{$} | >{$}c<{$} | >{$}c<{$} | >{$}c<{$} | >{$}c<{$} | >{$}c<{$} |}
	\hline
		\star_T & y & h & u & u \cdot a_i & u \cdot b_j  & \ell & f & pt\\ \hhline{|=||=|=|=|=|=|=|=|=|}
		y & 0 & 0 & 0 & 0 & 0  & 0 & 0 & 0\\ \hline
		h & 0  & 6 & -12 & 0 & 0 & 5u+21h & u+5h & 6\ell -12 f\\ \hline
		u & 0  & -12 & 24 &  0 & 0 & 5\ell -  10u & \ell -2u & -12\ell + 24f \\ \hline
		u \cdot a_i & 0 & 0 & 0 & 0 & \delta_{ij}(u + 5h)  & -5u \cdot a_i & -u \cdot a_i & 0\\ \hline
		u \cdot b_j & 0  & 0 & 0 & -\delta_{ij}(u +5h)& 0 & -5u \cdot b_j & -u \cdot b_j & 0\\ \hline
		\ell & 0 & 5u+21h & 5\ell - 10u & 5u \cdot a_i & 5u \cdot b_j & 21\ell -42 f & 5\ell-10f & 0\\ \hline
		f & 0 & u+5h & \ell -2u & u \cdot a_i & u \cdot b_j & 5\ell-10f & \ell - 2f & 0\\ \hline
		pt & 0 & 6\ell -12 f &-12\ell + 24f  & 0 & 0 & 0 & 0 & 0\\ \hline
	\end{tabular}}
	\caption{Quantum multiplication -- secant lines}
  \label{tab:testtab3}
\end{table}
\end{lemma}
\begin{proof}
See subsection \ref{subsec:T}. 
\end{proof}

\begin{lemma} \label{lem:rulinglinestable}
Up to permutation and sign, the only nontrivial 3-point Gromov-Witten invariants in the class $R = L-3F$ are
\begin{equation}
\begin{split}
GW_{R,3}(f,u,u) &= (-3)^2 \cdot 4 = 36,\\
GW_{R,3}(f,h,u) &= 1 \cdot (-3) \cdot 4 = -12 ,\\
GW_{R,3}(f,h,h) &= 1 \cdot 1 \cdot 4 = 4,\\
GW_{R,3}(\ell,u,u) &= (-3)^2 \cdot 2 =18,\\
GW_{R,3}(\ell,h,u) &= 1 \cdot (-3) \cdot 2 = -6,\\
GW_{R,3}(\ell,h,h) &= 1^2 \cdot 2 = 2,\\
GW_{R,3}(u \cdot a_i,u \cdot b_i,h) &= 1 \cdot (-1) = 1,\\
GW_{R,3}(u \cdot a_i,u \cdot b_i,u) &= (-3) \cdot (-1) = 3 .\\
\end{split}
\end{equation}
The contribution of this class to the quantum product is given by the following table: 
\begin{table}[h!]
  \centering
  \resizebox{0.8\textwidth}{!}{ 
	\centering
	\begin{tabular}{| >{$}c<{$} || >{$}c<{$} | >{$}c<{$} | >{$}c<{$} | >{$}c<{$} | >{$}c<{$} | >{$}c<{$} | >{$}c<{$} | >{$}c<{$} |}
	\hline
		\star_R & y & h & u & u \cdot a_i & u \cdot b_j  & \ell & f & pt\\ \hhline{|=||=|=|=|=|=|=|=|=|}
		y & 0 & 0 & 0 & 0 & 0  & 0 & 0 & 0\\ \hline
		h & 0  & 2h+4u & -6h - 12u & u\cdot a_i & u\cdot b_j & 2\ell-6f & 4\ell-12f & 0\\ \hline
		u & 0  & -6h - 12u & 18h + 36u &  3u\cdot a_i & 3u\cdot b_j & -6\ell + 18f & -12\ell + 36f & 0 \\ \hline
		u \cdot a_i & 0 & -u\cdot a_i & -3u\cdot a_i & 0 & \delta_{ij}(\ell+3f)  & 0 & 0 & 0\\ \hline
		u \cdot b_j & 0  & -u\cdot b_j & -u\cdot b_j & -\delta_{ij}(\ell+3f) & 0 & 0 & 0 & 0\\ \hline
		\ell & 0 & 2\ell-6f & -6\ell + 18f & 0 & 0 & 0 & 0 & 0\\ \hline
		f & 0 & 4\ell-12f & -12\ell + 36f & 0 & 0 & 0 & 0 & 0\\ \hline
		pt & 0 & 0 & 0 & 0 & 0 & 0 & 0 & 0\\ \hline
	\end{tabular}}
	\caption{Quantum multiplication -- ruling lines}
  \label{tab:testtab4}
\end{table} 
\end{lemma} 
\begin{proof}
See subsection \ref{subsec:R}. 
\end{proof}

As usual, let $\mathfrak{t} \in H^1(S^1)$ be the pullback of the fundamental class of the circle. As in Lemma \ref{lem:acttrivialy}, we have an additive isomorphism
\begin{equation} \label{eq:additivedecomposition}
H^\bullet(\YY) = H^\bullet(Y) \oplus \mathfrak{t}H^\bullet(Y).
\end{equation}

\begin{lemma} \label{lem:decomposecohomology}
The direct sum decomposition in equation \eqref{eq:additivedecomposition} is an algebra isomorphism. 
\end{lemma}
\begin{proof}
In Section \ref{subsec:proofoflemma}. 
\end{proof}

We denote the parametrized quantum product on $\YY$ as $\tilde{\star}$ and the quantum product on the fiber as $\star$.

\begin{proposition}[quantum product] \label{prop:quantumproduct}
The algebra isomorphism \eqref{eq:additivedecomposition} extends to an isomorphism of quantum rings (truncated at $\II_{2F}$.) That is, the quantum product in

\begin{equation}
QH^\bullet(\YY;\Gamma/\II_{2F}) 
\end{equation}

decomposes as 
\begin{equation}
(w_1' + \mathfrak{t}w_1'') \tilde{\star} (w_2' + \mathfrak{t}w_2'') = (w_1' \star w_2') + \mathfrak{t}(w_1' \star w_2'' + w_1'' \star w_2'). 
\end{equation}
\end{proposition}
\begin{proof}
In section \ref{sec:ambiguity}.
\end{proof}

As a corollary, we can define $q\tau_2(\phi)$: by making suitable auxiliary choices we obtain an $A_\infty$-algbera $(\tilde{\mathcal{C}},\tilde{\mu})$ which satisfies \eqref{eq:conditionstar}. The precise definition (including all the choice involved) is the subject of Section \ref{sec:definitionsandstatements}. It has the following properties:

\begin{lemma} \label{lem:minimalmorsefunction}
$\tilde{\mu}^1 = 0$ and the cohomology algebra of $H(\tilde{\mathcal{C}})$ is isomorphic to $QH^\bullet(\YY;\Gamma/ \II_{2F})$. 
\end{lemma}
\begin{proof}
See Section \ref{subsec:relativemorsemodel}. 
\end{proof}

\begin{proposition}[The main term]  \label{prop:mainterm}
Consider the $\mathfrak{t} \cdot ...$ term of $\tilde{\mu}^3_{2F}(\mathfrak{u}z_3,\mathfrak{u}z_2,\mathfrak{u}z_1)$ where $|z_i|=3$. 
Then it is always zero except in the following cases: 

\begin{equation}
\begin{split}
\tilde{\mu}^3_{2F}(\mathfrak{u}a_i,\mathfrak{u}b_i,\mathfrak{u}a_j) = \mathfrak{t} \mathfrak{u}b_j, \\
\tilde{\mu}^3_{2F}(\mathfrak{u}a_i,\mathfrak{u}b_i,\mathfrak{u}b_j) = -\mathfrak{t} \mathfrak{u}a_j, \\
\tilde{\mu}^3_{2F}(\mathfrak{u}a_j,\mathfrak{u}a_i,\mathfrak{u}b_i) = \mathfrak{t} \mathfrak{u}b_j, \\
\tilde{\mu}^3_{2F}(\mathfrak{u}b_j,\mathfrak{u}b_i,\mathfrak{u}a_i) = -\mathfrak{t} \mathfrak{u}a_j, \\
\end{split}
\end{equation}
where $1 \leq i \leq 2$ and $3 \leq j \leq 4$.
\end{proposition} 
\begin{proof}
See Section \ref{subsec:computemainterm} and also the discussion on the Johnson homomorphism later in this Section (especially Proposition \ref{prop:mainterm2}).  
\end{proof}

Finally, 

\begin{proposition} 
The $\mathfrak{t}$-part of the quantum matrix Massey product \vspace{0.5em}
\begin{equation}
\langle 
\begin{pmatrix}
 a_1 \mathfrak{u} &  a_2 \mathfrak{u} \\
0 & 0
\end{pmatrix}
,
\begin{pmatrix}
b_1 \mathfrak{u} & 0 \\
-b_2 \mathfrak{u} & 0
\end{pmatrix}
,
\begin{pmatrix}
b_1 \mathfrak{u} & b_3 \mathfrak{u} \\
0 & 0
\end{pmatrix}
\rangle_{2F} = 
\begin{pmatrix}
0 & \mathfrak{t} b_3 \mathfrak{u} \\
0 & 0
\end{pmatrix} 
\end{equation} \vspace{0.5em}
is nontrivial. 
\end{proposition}

\begin{proof}
First we must verify the vanishing condition. The relevant homology classes are $A \in H_2(Y)$ with $c_1(A) < c_1(2F)=2$ or $A=2F$. 
But $\star_{2F}$ is identically zero, so this leaves only $A \in \left\{0,F,R\right\}$, and \vspace{0.5em}
\begin{equation}
\begin{pmatrix}
\mathfrak{u} a_1  & \mathfrak{u} a_2  \\
0 & 0
\end{pmatrix} \star_0 \begin{pmatrix}
\mathfrak{u}b_1   & 0 \\
-\mathfrak{u} b_2   & 0
\end{pmatrix} = \begin{pmatrix}
pt - pt & 0 \\
0 & 0
\end{pmatrix} = \begin{pmatrix}
0 & 0 \\
0 & 0
\end{pmatrix}, 
\end{equation}

\begin{equation}
\begin{pmatrix}
\mathfrak{u} a_1  & \mathfrak{u} a_2  \\
0 & 0
\end{pmatrix} \star_F \begin{pmatrix}
\mathfrak{u}b_1   & 0 \\
-\mathfrak{u} b_2   & 0
\end{pmatrix} = \begin{pmatrix}
f - f & 0 \\
0 & 0
\end{pmatrix} = \begin{pmatrix}
0 & 0 \\
0 & 0
\end{pmatrix}, 
\end{equation}

\begin{equation}
\begin{pmatrix}
\mathfrak{u} a_1  & \mathfrak{u} a_2  \\
0 & 0
\end{pmatrix} \star_R \begin{pmatrix}
\mathfrak{u}b_1   & 0 \\
-\mathfrak{u} b_2   & 0
\end{pmatrix} = \begin{pmatrix}
(\ell+3f) - (\ell+3f) & 0 \\
0 & 0
\end{pmatrix} = \begin{pmatrix}
0 & 0 \\
0 & 0
\end{pmatrix}, 
\end{equation}

\begin{equation}
\begin{pmatrix}
\mathfrak{u} b_1   & 0 \\
-\mathfrak{u} b_2   & 0
\end{pmatrix} \star_0 \begin{pmatrix}
\mathfrak{u} b_1   & \mathfrak{u} b_3   \\
0 & 0
\end{pmatrix} = \begin{pmatrix}
0 & 0 \\
0 & 0
\end{pmatrix}, 
\end{equation}

\begin{equation}
\begin{pmatrix}
\mathfrak{u} b_1   & 0 \\
-\mathfrak{u} b_2   & 0
\end{pmatrix} \star_F \begin{pmatrix}
\mathfrak{u} b_1   & \mathfrak{u} b_3   \\
0 & 0
\end{pmatrix} = \begin{pmatrix}
0 & 0 \\
0 & 0
\end{pmatrix},  
\end{equation}

\begin{equation}
\begin{pmatrix}
\mathfrak{u} b_1   & 0 \\
-\mathfrak{u} b_2   & 0
\end{pmatrix} \star_R \begin{pmatrix}
\mathfrak{u} b_1   & \mathfrak{u} b_3   \\
0 & 0
\end{pmatrix} = \begin{pmatrix}
0 & 0 \\
0 & 0
\end{pmatrix}, \vspace{0.5em}
\end{equation}
as 2-by-2 cohomology-valued matrices. Thus, the Massey product can indeed be defined. Second, by Lemma \ref{lem:minimalmorsefunction} above, $\tilde{\mu}^1=0$ and there are no bounding cochains to take into account. So we only need to compute the main term, 

\begin{equation}
\Theta := \mu_{2F}^3(\begin{pmatrix}
\mathfrak{u} a_1  &  \mathfrak{u} a_1 \\
0 & 0
\end{pmatrix}
,
\begin{pmatrix}
\mathfrak{u} b_1   & 0 \\
-\mathfrak{u} b_2   & 0
\end{pmatrix}
,
\begin{pmatrix}
 \mathfrak{u} b_1  &  \mathfrak{u} b_3  \\
0 & 0
\end{pmatrix})
\end{equation}
which is a 2-by-2 matrix whose entries are cohomology classes of degree $4 = 3+3+3-1-2 \cdot 2$. Unwinding the definition, 
\begin{equation}
\begin{split}
&\Theta_{i_3 i_0} = \sum^{2}_{i_1,i_{2}=1} \mu_{2F}^3 \left( \begin{pmatrix}
\mathfrak{u} a_1  & \mathfrak{u}a_2 \\
0 & 0
\end{pmatrix}_{i_3 i_{2}} ,\begin{pmatrix}
\mathfrak{u}b_1   & 0 \\
-\mathfrak{u} b_2   & 0
\end{pmatrix}
_{i_{2} i_{1}}, \begin{pmatrix}
\mathfrak{u}b_1   & \mathfrak{u}b_3   \\
0 & 0
\end{pmatrix}_{i_1 i_0} \right)
\end{split}
\end{equation}
for $1 \leq i_3,i_0 \leq 2$. Note that we can immediatly disregard all sequences $(i_3,i_2,i_1,i_0)$ where $i_3=2$ or $i_1=2$ because the result would be zero. In particular, 
\begin{equation}
\begin{split}
\Theta_{21} &= 0. \\
\Theta_{22} &= 0. \\
\end{split}
\end{equation}
Applying Proposition \ref{prop:sameas3dim}, we see that 
\begin{equation}
\begin{split}
\Theta_{11} &= \tilde{\mu}^3_{2F}(\mathfrak{u}a_1 , \mathfrak{u} b_1   , \mathfrak{u} b_1 )  - \tilde{\mu}^3_{2F}(\mathfrak{u} a_1 , \mathfrak{u} b_2   ,\mathfrak{u} b_1  ) = 0. \\
\Theta_{12} &= \tilde{\mu}^3_{2F}(\mathfrak{u} a_1 ,  \mathfrak{u} b_1  ,\mathfrak{u} b_3  )  - \tilde{\mu}^3_{2F}( \mathfrak{u} a_1, \mathfrak{u} b_2   ,\mathfrak{u} b_3  ) =  \mathfrak{t} \mathfrak{u} b_3 .
\end{split}
\end{equation}
It remains to show that this is a nontrivial coset. For that, we must compute the ambiguity ideal $\II$. By definition, the degree four piece is generated by 
\begin{equation} \label{eq:239}
\begin{pmatrix}
\mathfrak{u} a_1  & \mathfrak{u} a_2  \\
0 & 0
\end{pmatrix} \star_0 \left(M_{2 \times 2}(\Q) \otimes H^1(\YY) \right) 
\end{equation}

\begin{equation} \label{eq:240}
\begin{pmatrix}
\mathfrak{u} a_1  & \mathfrak{u} a_2  \\
0 & 0
\end{pmatrix} \star_F \left(M_{2 \times 2}(\Q) \otimes H^3(\YY) \right)
\end{equation}

\begin{equation} \label{eq:241}
\begin{pmatrix}
\mathfrak{u} a_1  & \mathfrak{u} a_2  \\
0 & 0
\end{pmatrix} \star_R \left(M_{2 \times 2}(\Q) \otimes H^3(\YY) \right)
\end{equation}

\begin{equation} \label{eq:242}
\left(M_{2 \times 2}(\Q) \otimes H^1(\YY) \right) 
\star_0  \begin{pmatrix}
\mathfrak{u}b_1   & \mathfrak{u}b_3   \\
0 & 0
\end{pmatrix} 
\end{equation}

\begin{equation} \label{eq:243}
\left(M_{2 \times 2}(\Q) \otimes H^3(\YY) \right) 
\star_F  \begin{pmatrix}
\mathfrak{u}b_1   & \mathfrak{u}b_3   \\
0 & 0
\end{pmatrix} 
\end{equation}

\begin{equation} \label{eq:244}
\left(M_{2 \times 2}(\Q) \otimes H^3(\YY) \right) 
\star_R \begin{pmatrix}
\mathfrak{u}b_1   & \mathfrak{u}b_3   \\
0 & 0
\end{pmatrix} 
\end{equation}
We argue by contradiction: Assume the Proposition is false and $\Theta$ is actually a trivial coset (i.e., $\Theta \in \II$). Therefore (by definition), $\Theta$ can be written as a sum of matrices of the form \eqref{eq:239}-\eqref{eq:244}. Note that any such equation actually means that $\Theta$ can be written as the sum of matrices
\begin{equation}
\begin{pmatrix}
x & y \\
z & w
\end{pmatrix} \in \II
\end{equation}
with entries $x,y,z,w \in \mathfrak{t}H^3(Y)$. The possible summands in such an expression take the form ($q'_i,q''_j \in \Q$ are scalars):
\begin{equation} \label{eq:246}
\begin{pmatrix}
\mathfrak{u} a_1  & \mathfrak{u} a_2  \\
0 & 0
\end{pmatrix} \star_F \begin{pmatrix}
q'_1 u + q''_1 h & q'_2 u + q''_2 h \\
q'_3 u + q''_3 h & q'_4 u + q''_4 h
\end{pmatrix} = \begin{pmatrix}
q'_1 \mathfrak{t} \mathfrak{u} a_1 + q'_3 \mathfrak{t} \mathfrak{u} a_2 & q'_2 \mathfrak{t} \mathfrak{u} a_1 + q'_4 \mathfrak{t} \mathfrak{u} a_2 \\
0 & 0
\end{pmatrix},
\end{equation}
\begin{equation} \label{eq:247}
\begin{split}
&\begin{pmatrix}
\mathfrak{u} a_1  & \mathfrak{u} a_2  \\
0 & 0
\end{pmatrix} \star_R \begin{pmatrix}
q'_1 u + q''_1 h & q'_2 u + q''_2 h \\
q'_3 u + q''_3 h & q'_4 u + q''_4 h
\end{pmatrix} \\
&= \begin{pmatrix}
-(3q'_1 +q''_1) \mathfrak{t} \mathfrak{u} a_1 -(3q'_3 +q''_3)  \mathfrak{t} \mathfrak{u} a_2 & -(3q'_2 +q''_2)  \mathfrak{t} \mathfrak{u} a_1 -(3q'_4 +q''_4)  \mathfrak{t} \mathfrak{u} a_2 \\
0 & 0
\end{pmatrix},
\end{split}
\end{equation}

\begin{equation} \label{eq:248}
\begin{pmatrix}
q'_1 u + q''_1 h & q'_2 u + q''_2 h \\
q'_3 u + q''_3 h & q'_4 u + q''_4 h
\end{pmatrix} \star_F \begin{pmatrix}
\mathfrak{u} b_1  & \mathfrak{u} b_3  \\
0 & 0
\end{pmatrix} = \begin{pmatrix}
q'_1 \mathfrak{t} \mathfrak{u} b_1  & q'_1  \mathfrak{t} \mathfrak{u} b_3 \\
q'_3  \mathfrak{t} \mathfrak{u} b_1 & q'_3  \mathfrak{t} \mathfrak{u} b_3
\end{pmatrix},
\end{equation}

\begin{equation} \label{eq:249}
\begin{pmatrix}
q'_1 u + q''_1 h & q'_2 u + q''_2 h \\
q'_3 u + q''_3 h & q'_4 u + q''_4 h
\end{pmatrix} \star_R \begin{pmatrix}
\mathfrak{u} b_1  & \mathfrak{u} b_3  \\
0 & 0
\end{pmatrix} = \begin{pmatrix}
-(3q'_1 +q''_1) \mathfrak{t} \mathfrak{u} b_1  & -(3q'_1 +q''_1) \mathfrak{t} \mathfrak{u} b_3 \\
-(3q'_3 +q''_3)  \mathfrak{t} \mathfrak{u} b_1 & -(3q'_3 +q''_3) \mathfrak{t} \mathfrak{u} b_3
\end{pmatrix}.
\end{equation}
But looking at the equations carefully, it is clear that for any linear combination $\Theta'$ of \eqref{eq:246}-\eqref{eq:249}, the rational coefficient of $\mathfrak{t} \mathfrak{u} b_1$ at $\Theta'_{11}$ must be the same as rational coefficient of $\mathfrak{t} \mathfrak{u} b_3$ at $\Theta'_{12}$. However, $\Theta_{11}=0$ and $\Theta_{12} = \mathfrak{t} \mathfrak{u} b_3$. Thus we get a contradiction to our original assumption and $\Theta \notin \II$.
\end{proof}

\subsection{Connection with the Johnson homomorphism} \label{subsec:qjohnson}
In this subsection, we sketch a geometric version of the mapping torus construction of $\tau_k$, and explain how it related to our construction. All results are taken from \cite{MR2850125},\cite{MR3053012} and Andrew Putnam 2008 Master class lecture notes \cite{Aarhus} unless stated to the contrary. \\

Throughout this entire section, we assume $g \geq 2$. Denote by $\Sigma_g$, the compact oriented surface obtained by identifying sides of a solid, $4g$-gon in the standard way. We denote $\Sigma^*_{g,n}$ and $\Sigma^{\partial}_{g,n}$ for the $n$-punctured, and $n$-holed versions. The mapping class groups $Mod_{g}$, $Mod^*_{g,n}$ and $Mod^\partial_{g,n}$ are defined to be the set of appropriate orientation-preserving self-diffeomorphisms of the surface of interest modulo those isotopic to the identity. Specifically, we define the three mapping class groups
\begin{align*}
Mod_g &= \pi_0(\Diff^+(\Sigma_{g})\\
Mod_{g,n}^{*} &= \pi_0(\Diff^+(\Sigma^*_{g,n})\\
Mod_{g,n}^{\partial} &= \pi_0(\Diff(\Sigma^{\partial}_{g,n},\partial \Sigma^{\partial}_{g,n}))\\
\end{align*}
where $\pi_0(\Diff(\Sigma^{\partial}_{g,n},\partial S^{\partial}_{g,n}))$ is the group of diffeomorphisms of $\Sigma^{\partial}_{g,n}$ which preserve the boundary pointwise. Note that they are automatically orientation preserving. 

\begin{remark} There are many variants of the definition of mapping class group, which may or may not give the same object. To mention just a few:
\begin{itemize} \vspace{0.5em}
\item
We could replace diffeomorphisms by homeomorphisms, or alternatively, we could fix a triangulation and consider PL-homeomorphisms. However, in dimension two, this would not affect the definition. \vspace{0.5em}
\item
There is no difference between punctures and marked points (with all the diffeomorphism and isotopies assumed to preserve them). \vspace{0.5em}
\item
More relevant to us, we can fix a symplectic form $\omega$ on the surface $\Sigma_g$. Denote $\Symp^{s}(\Sigma_g,\omega)$ for the group of all symplectomorphism considered with the $C^\infty$-topology. Then as explained in \cite[Appendix B]{MR2383898} and \cite{MR2104005}, it is clear that any Dehn twist can be represented by an area-preserving diffeomorphism. Since the mapping class group is generated by Dehn twists, it follows that the map 
\begin{equation} \label{eq:reallyboringmap}
\Symp^{s}(\Sigma_g,\omega) \to \Diff^+(\Sigma_g)
\end{equation}
is surjective. In fact, Moser's lemma implies that \eqref{eq:reallyboringmap} is a weak homtopy equivalence. If we denote $\Symp_0^{s}(\Sigma_g,\omega) = \Symp^{s}(\Sigma_g,\omega) \cap \Diff_0^+(\Sigma_g)$ then there is a short exact sequence
\begin{equation}
1 \to \Symp_0^s(M,\omega) \to \Symp^s(M,\omega) \to Mod_{g} \to 1. 
\end{equation}
Taking the quotient of the first two groups by $Ham(\Sigma_g,\omega)$, and using the Flux homomorphism gives the short exact sequence
\begin{equation}
1 \to H^1(\Sigma;\mathbb{R}) \to \pi_0 \Symp(M,\omega) \to Mod_{g} \to 1. 
\end{equation}
\item
We could consider homeomorphisms up to homotopy relative to the boundary instead of considering them up to isotopy. This would not affect the definition of the mapping class group since an old result of Baer asserts that two homeomorphisms are homotopic relative to the boundary if and only if they are isotopic. \vspace{0.5em}
\end{itemize}
\end{remark}
We denote $\Sigma^\partial_g = \Sigma^\partial_{g,1}$ for the genus $g$ surface with one boundary component, and $Mod_g^\partial = Mod_{g,1}^\partial$ for its mapping class group. 
\begin{definition} 
A curve in $\Sigma^\partial_g$ is a \textbf{bounding separating closed curve} of genus $h$ (abbreviated to BSCC for brevity) if it bounds a subsurface homeomorphic to $\Sigma^\partial_h$, otherwise it is called non-separating. We define a \textbf{bounding pair} (BP) of genus $h$ to be a pair of simple closed curves which together bound a subsurface homeomorphic to a genus $h$ surface with two boundary components. See Figures \ref{fig:BSCC} and \ref{fig:BP} for illustration. The same definition applies to $\Sigma_g$ with the boundary puncture, except now there is ambiguity between $h$ and $g-h$.
\end{definition}

\begin{figure} 
  \begin{subfigure}[b]{.44\linewidth}
		\centering
				\fontsize{0.25cm}{1em}
			\def\svgwidth{7cm}
			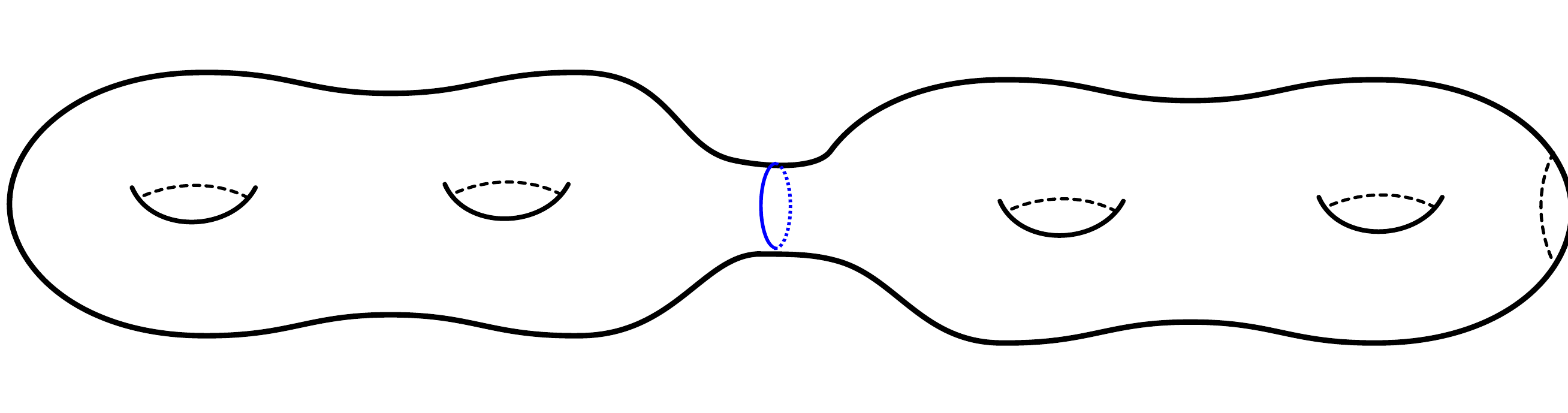
			\caption{A separating curve of genus $h=2$}
			\label{fig:BSCC}
  \end{subfigure}\hfill
  \begin{subfigure}[b]{.44\linewidth}
	\centering
					\fontsize{0.25cm}{1em}
			\def\svgwidth{7cm}
			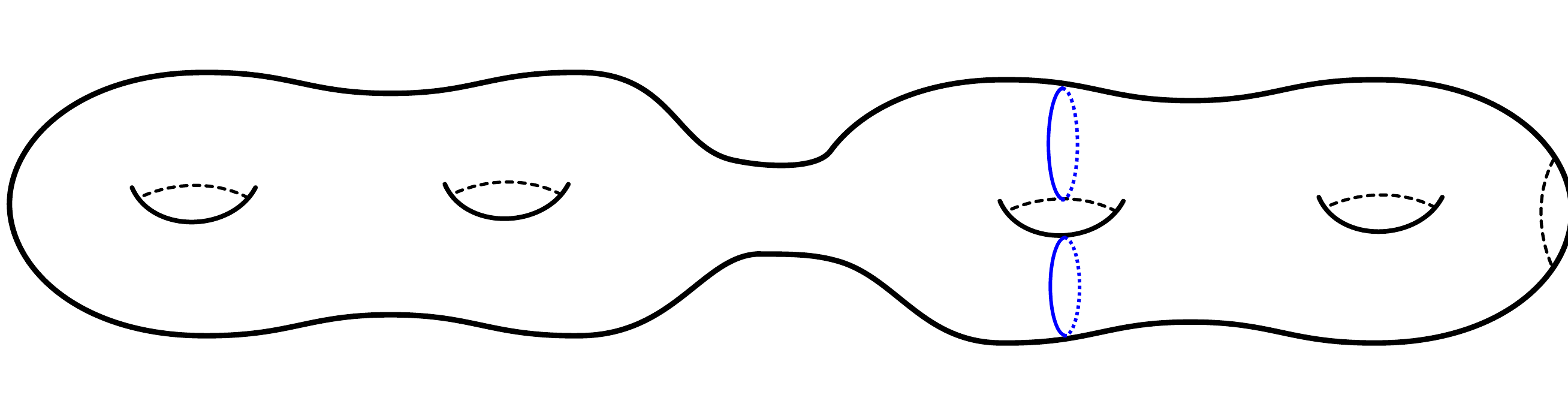 
			\caption{A bounding pair of genus $h=3$}
			\label{fig:BP}
  \end{subfigure} 
	\caption{The generators of the Torelli group}
	\label{fig:illustration} 
\end{figure}	

By definition, the mapping class groups act on simple closed curves $\Sigma^{\partial}_{g}$, which induces an action on the fundamental group of the surface. The action takes any curve homotopic to the boundary to another curve homotopic to the boundary. Thus we have a well defined map
\begin{equation} \label{eq:actionofnielsen}
Mod^\partial_{g} \to \left\{\phi \in Aut(\pi_1 \Sigma^{\partial}_{g}) \left|\right. \phi [\partial \Sigma^{\partial}_{g}] = [\partial \Sigma^{\partial}_{g}]\right\}
\end{equation}
\begin{theorem}[Nielsen]
The map \eqref{eq:actionofnielsen} is an isomorphism.
\end{theorem}
This allows us to import many tools and results from group theory to the study of $Mod^{\partial}_{g}$. Let 
\begin{equation}
F =\pi_1 \Sigma^{\partial}_{g} 
\end{equation}
be the fundamental group of the bordered genus $g$ surface.It is a free group on $2g$ generators. Denote by $\left(\Gamma_k\right)_k$ the lower central series of $F$ i.e. $\Gamma_1 = F$ and $\Gamma_{k+1} = [F, \Gamma_k]$. Let $N_k$ then be the k-th nilpotent quotient group $N_k = \Gamma_1 / \Gamma_{k+1}$. The subgroups $\Gamma_k$ are characteristic, they are preserved by any automorphism, so that the mapping class group in fact acts by automorphisms on $N_k$. Note that $N_1 = H := H_1(\Sigma_g^{\partial}, \mathbb{Z})$ the first homology group. So in the case $k = 1$, this is just the familiar action of the mapping class group on the homology of a surface. Since any such action must respect the intersection pairing of the surface, we see that it induces a surjective homomorphism called the \textbf{symplectic representation} 
\begin{equation}
\tau_0 : Mod^{\partial}_{g} \to Sp_{2g}(\mathbb{Z}).
\end{equation}
\begin{definition}
The \textit{Torelli group} is defined as the kernel 
\begin{equation}
\mathcal{I}^{\partial}_{g} := Ker(\tau_0). 
\end{equation}
It is the subgroup of the mapping class group that acts trivially on the homology of the surface.
\end{definition}
In fact, Torelli group is only the first in a series of nested subgroups of the mapping class group
\begin{definition}
Let $k \geq 0$. We define the \textbf{Johnson-Andreadakis filtration} to be the kernels 
\begin{equation}
Mod^{\partial}_{g,1}[k] = Ker(Mod^{\partial}_{g,1} \to Aut(N_k)).
\end{equation}
\end{definition}
\begin{remark}
$Mod^{\partial}_{g}[0] = Mod^{\partial}_{g}$ and $Mod^{\partial}_{g}[1] = \mathcal{I}^{\partial}_{g}$. 
\end{remark}
Note that a Dehn twist on a BSCC or the composition of opposite twists on a BP lies in the Torelli group $\mathcal{I}^{\partial}_{g}$. 
\begin{definition}
Let 
\begin{equation}
\mathcal{K}^{\partial}_{g} \subset Mod^{\partial}_{g}
\end{equation}
denote the \textbf{Johnson kernel}: the subgroup generated by all Dehn twists on BSCC's.  
\end{definition}
The results of Powell \cite{Powell} and Johnson \cite{Torelli1,Torelli2}, identify it as a stage in the filtration, 
\begin{theorem}. For $g \geq 3$, the Torelli group $\mathcal{I}^{\partial}_{g}$ is generated by a finite number of Dehn twists on genus $1$
bounding pairs. Moreover, for $g \geq 2$, $Mod^{\partial}_{g}[2] = Ker(\tau_1) = \mathcal{K}^{\partial}_{g}$.
\end{theorem}
\textbf{Notation.} We follow the standard convention of using capital letters to denote the homology classes of elements of $F$,
which themselves are denoted by the corresponding lowercase roman or greek letters. Cohomology and homology have integer coefficients.
\begin{definition}
Denote $\mathcal{L}^{\partial}_k = \Gamma_k / \Gamma_{k+1}$ the quotient group, and let 
\begin{equation}
\LL^\partial := \coprod_{k \geq 1} \LL^\partial_k
\end{equation}
Using the Witt-Hall identities, one can show that $[\Gamma_n,\Gamma_m] \subseteq \Gamma_{n+m}$ and that the induced multiplication on $\Gamma_n / \Gamma_{n+1}$ is abelian. Given cosets $x = a \Gamma_{n+1}$ and $x = b \Gamma_{m+1}$ in $\LL^\partial_n$ and $\LL^\partial_m$ respectively, we define the
bracket to be the coset
\begin{equation}
[x, y] = [a,b] \Gamma_{n+m+1}. 
\end{equation}
We then extend this bracket to all of $\LL^\partial$ by linearity. 
\end{definition}
\begin{proposition}
$\LL$ is the free Lie algebra $\LL^\partial = \LL^\partial(H)$ generated by $H = H_1 = \Gamma_1/\Gamma_2$. 
\end{proposition}

\begin{definition} 
The Johnson homomorphism $\tau_0$ is just the symplectic representation. For $k > 0$, the \textbf{k-th Johnson homomorphism}
\begin{align*}
\tau_k : Mod^{\partial}_{g}[k] \to Hom(H,\mathcal{L}^{\partial}_{k+1})
\end{align*}
is defined by
\begin{equation}
\tau_k(\phi)(X) = \phi(x)x^{-1} 
\end{equation}
for any $\phi \in Mod^{\partial}_{g}[k]$, $X \in H$, and $x \in N_{k+1}$ with $[x] = X$.
\end{definition}
Note that the definition makes sense because $\phi(x)x^{-1}$ is the identity element in $N_k$, meaning it lies in $\Gamma_{k+1}$, so that it can be considered as an element of $\mathcal{L}^{\partial}_{k+1}$. Some combinatorical manipulation shows that 
\begin{proposition}
The maps $\tau_k$ are homomorphisms; moreover they are \textbf{equivariant} in the following sense: Since $\LL^\partial$ is generated by $H$, each $\phi \in Mod^{\partial}_{g}$ acts on $\LL^\partial$ via the symplectic representation. Thus $Mod^{\partial}_{g}$ also acts on $Hom(H,\mathcal{L}^{\partial}_{k+1})$. We denote this action by $|\phi|$. Then for any $\psi \in Mod^{\partial}_{g}[k]$ we have
\begin{equation}
\tau_k(\phi \psi \phi^{-1}) = |\phi|\tau_k(\psi).
\end{equation}
\end{proposition}
In particular, 
\begin{corollary}
$\tau_k$ is invariant under conjugation by $\phi \in \II^\partial_g$. 
\end{corollary}
\textbf{Notation.} By use of the Poincar\'{e} duality of $\Sigma_g^\partial$, we have $H_* \iso H^*$. In terms of the standard symplectic basis $\left\{A_i,B_i\right\}$ , we use the convention\footnote{This seems slightly odd, but is in fact very natural when considering the representation theory of $Sp_{2g}(\mathbb{Z})$.} that 
\begin{equation}
A_i \to B_i \to -A_i. 
\end{equation}
So, we may identify $Hom(H,\mathcal{L}^{\partial}_{k+1})$ with $H \otimes \mathcal{L}^{\partial}_{k+1}$. We also note that $\Gamma_2/\Gamma_3 = \LL^\partial_2$ can be canonically identified with the $Gl(H)$-module $\bigwedge^2 H \subset H^{\otimes 2}$ by
\begin{equation}
[x,y] \mapsto [X,Y] \longleftrightarrow X \wedge Y \mapsto X \otimes Y - Y \otimes X
\end{equation}
for $X = [x]$ and $Y = [y]$; and there is an embedding of $\bigwedge^3 H \hookrightarrow H \otimes \bigwedge^2 H$
given by
\begin{equation} \label{eq:embedding3}
A \wedge B \wedge C \mapsto A \otimes (B \wedge C) + B \otimes (C \wedge A) + C \otimes (A \wedge B). 
\end{equation}
We define $D_k(H)$ to be the kernel of the bracket
\begin{equation}
[\bullet,\bullet] : H \otimes \LL^\partial_k \to \LL^\partial_{k+1}.
\end{equation}
Note that the image of $\bigwedge^3 H$ is exactly $D_2(H)$. 
\begin{proposition}[\cite{MR579103},\cite{MR1193604}]
The image of $\tau_k$ is contained in $D_k(H)$. When $k=1$, the image of $\tau_1$ is $D_2(H)$. 
\end{proposition}
In this way, we get a map
\begin{equation}
\bigslant{\left(H \otimes \wedge^2 H\right)}{\wedge^3 H} \stackrel{[\bullet,\bullet]}{\rightarrow} \LL^\partial_3. 
\end{equation}
Consider the submodule $\bigwedge^2 H \otimes \bigwedge^2 H \hookrightarrow H \otimes H \otimes \bigwedge^2 H$. By composing this with the bracket map, we get a map
\begin{equation} \label{eq:weirdmap117}
\bigwedge^2 H \otimes \bigwedge^2 H \to H \otimes \LL^\partial_3. 
\end{equation}
As a final piece of notation, we let: $T = S^2(\bigwedge^2 H)$ denote the symmetric submodule of $\bigwedge^2 H \otimes \bigwedge^2 H$
generated by elements $u^{\otimes 2} := u \otimes u$ and $u \leftrightarrow v := u \otimes v + v \otimes u$ with $u, v \in \bigwedge^2 H$. \begin{definition} Let 
\begin{equation}
\bar{T} \subset H \otimes \LL^\partial_3
\end{equation}
denote the image of $T$ under \eqref{eq:weirdmap117}. 
\end{definition}
\begin{example}
The value of $\tau_1$ on ''the standard" genus h bounding pair twist\footnote{By which we mean: the BP $\left\{\delta,\gamma\right\}$ bounds a subsurface with symplectic basis $\left\{A_i,B_i\right\}_{i=1}^h$ and the homology classes of $\delta$ (oriented with the bounded surface on the right) and $\gamma$ (oriented with the bounded surface on the left) are both equal to $B_{h+1}$.} can be written
\begin{equation}
\tau_1(\tau_\gamma \tau_\delta^{-1}) = \sum_{i=1}^h A_i \wedge B_i \wedge B_{h+1}
\end{equation}
This is easy to see: fix a base point $* \in \partial \Sigma_g^\partial$ on the boundary, and consider the standard set of simple closed curves $\left\{\alpha_i, \beta_i\right\}$ and the genus $h=1$ bounding pair $\left\{\delta,\gamma\right\}$ as shown in Figure \ref{fig:BP2}. We define a representative of the BP twist as the composition $T_\gamma T^{-1}_{\delta}$. It is obvious that the action of these twists
leaves fixed all $\alpha_i$ for $i > 2$ and all $\beta_j$ for $j > 1$. Moreover, the action on the remaining generators is given by:
\begin{equation}
\begin{split}
\alpha_1 &\mapsto d \alpha_1 d^{-1} = [d,\alpha_1] \alpha_1, \\
\beta_1 &\mapsto d \beta_1 d^{-1} = [d,\beta_1] \beta_1, \\
\alpha_2 &\mapsto [\alpha_1,\beta_1] \alpha_2.
\end{split}
\end{equation}
where $d$ is the curve which goes from the basepoint $p$ to the curve $\delta$, then around $\delta$ to the left, then back to $p$. In particular, $d$ is homologous to $\beta_2$. Denoting $A_i = [\alpha_i]$ and $B_j = [\beta_j]$, we get by definition (and using the dualities $A_i^* = B_i$ and $B_j^* = -A_j$):
\begin{equation}
\tau_1(T_\gamma T^{-1}_{\delta}) = B_1 \otimes (B_2 \wedge A_1) - A_1 \otimes( B_2 \wedge B_1) + B_2 \otimes( A_1 \wedge B_1).
\end{equation}
Compare to \eqref{eq:embedding3}. The general case is proved in the same way.
\end{example}
Using slightly more sophisticated machinery (e.g., Fox free calculus and Magnus expansions), we can evaluate $\tau_k$ on various elements of the Torelli group. 
\begin{example} \label{exmp:bscc4}
Let $\phi \in \KK^\partial_g$ be the Dehn twist $T_{c_h}$ on the ''standard"\footnote{By which we mean the subsurface bounded by $c_h$ has a symplectic basis $\left\{A_i,B_i\right\}_{i=1}^h$ with respect to the symplectic basis $\left\{A_i,B_i\right\}_{i=1}^g$.} genus $h$ BSCC $c_h$ as in Figure \ref{fig:BSCC2}. Then $\tau_1(\phi) = 0$ and
\begin{equation}
\tau_2(\phi) = \sum_{i=1}^h \bigg( B_i \otimes [A_i,\omega_h] - A_i \otimes [B_i,\omega_h] \bigg) \in H \otimes \LL^\partial_3 \: , \: \omega_h := \sum_{i=1}^h A_i \wedge B_i
\end{equation}
which is exactly the image of 
\begin{equation}
-\omega_h \otimes \omega_h \in (\bigwedge^2 H) \otimes (\bigwedge^2 H) 
\end{equation}
under the projection $T \to \bar{T}$ from \eqref{eq:weirdmap117}. 
\end{example}

\begin{figure} 
  \begin{subfigure}[b]{.44\linewidth}
	\centering
					\fontsize{0.25cm}{1em}
			\def\svgwidth{7cm}
			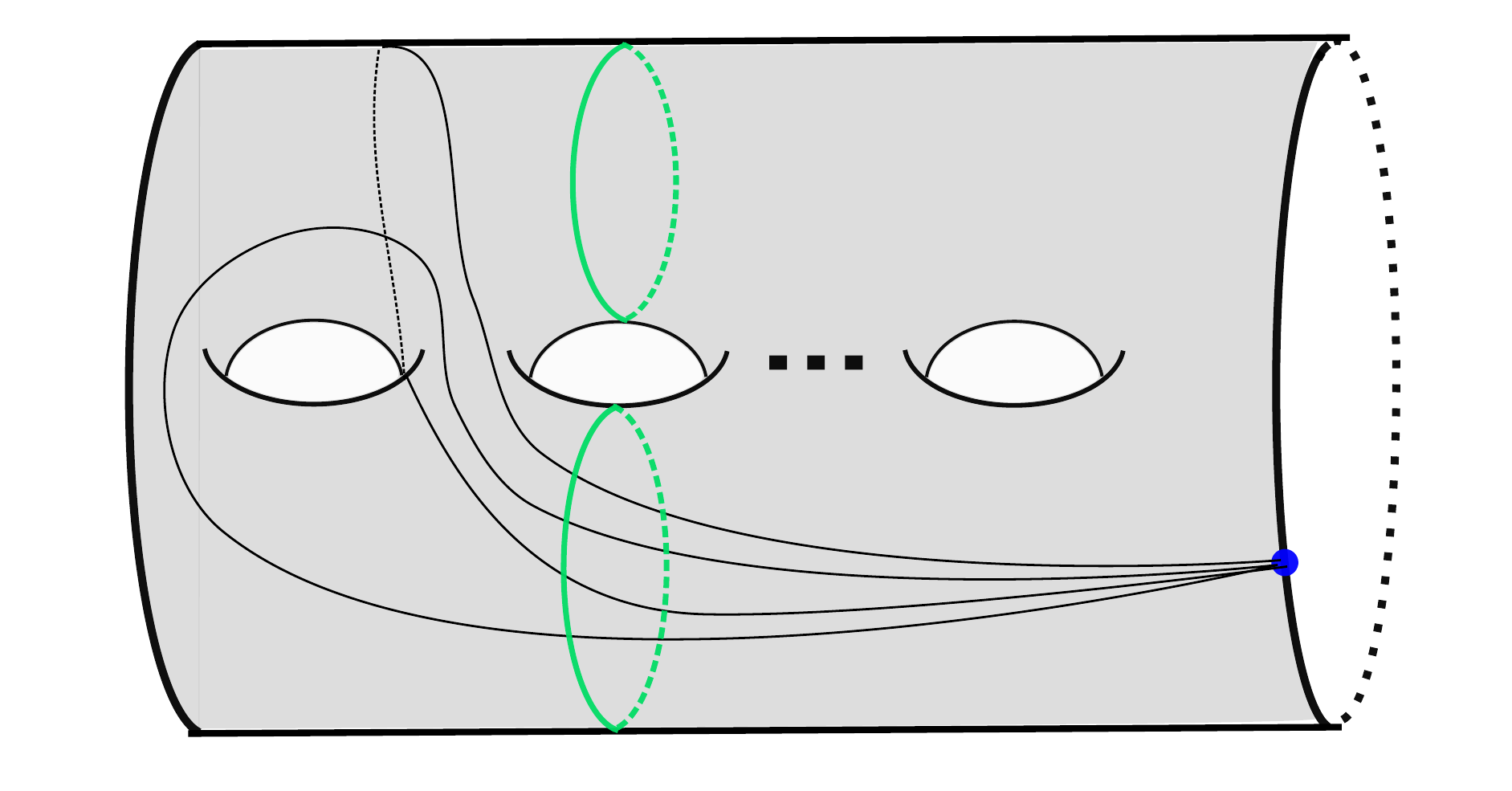 
			\caption{$\tau_1(T_\gamma T_\delta^{-1})$ with genus $h=2$}
			\label{fig:BP2}
  \end{subfigure} 
  \begin{subfigure}[b]{.44\linewidth}
		\centering
				\fontsize{0.25cm}{1em}
			\def\svgwidth{7cm}
			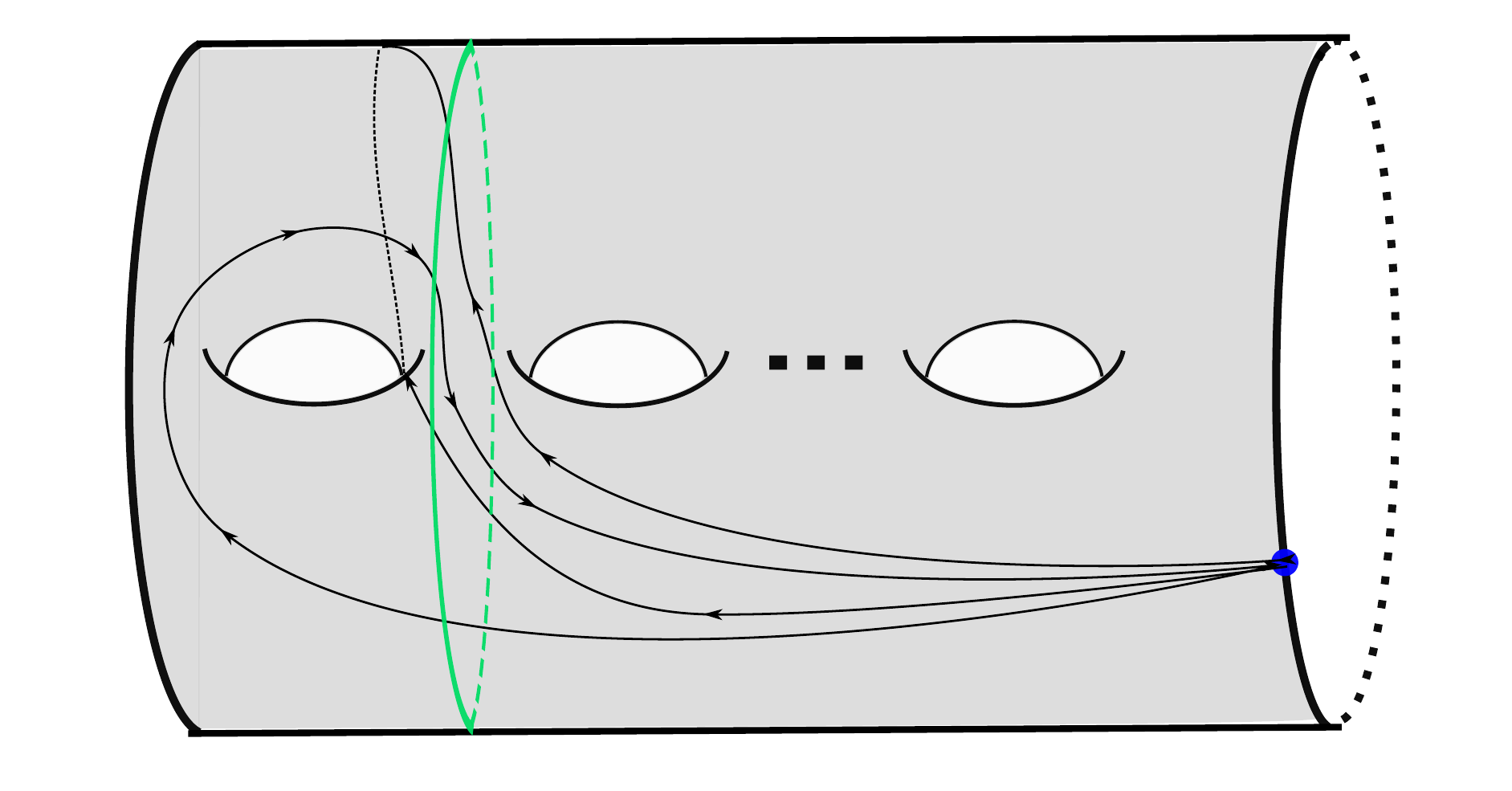
			\caption{$\tau_2(T_{c_h})$ with genus $h=1$}
			\label{fig:BSCC2}
  \end{subfigure}\hfill
	\caption{Computing the Johnson homomorphism}
	\label{fig:illustration2} 
\end{figure}	

Aside from the combinatorical group theory approach (which is by far the most computationally efficient), the theory of the Johnson homomorphism has many interesting connections to other fields in mathematics e.g. Hodge theory, Teichm\"{u}ller theory... and more. As a result, there are many alternative definitions $\tau_k$, see \cite{MR1397061}, \cite{MR1734418} for a survey. We will be interested in the following version conjectured by Johnson and later proved by Kitano in \cite{MR1381688}: the Johnson homomorphisms $\tau_k$ with $k \geq 1$ are given by $(k+1)$-fold Massey products on the mapping torus of $\phi : \Sigma_g^\partial \to \Sigma_g^\partial$. Since the cohomology algebra of a surface is always \textbf{intrinsically formal} (\cite[Proposition 8.4]{MR3379890}), this means that the Johnson homomorphism captures the extra deformations introduced by the nontrivial monodromy. \\

Expanding a little on this idea, we recall the nice geometric construction of $\tau_2$ for $\Sigma_g$ which appeared in \cite[Section 5]{MR3450773}. Let $\phi \in \II_g$ be a Torelli mapping class, and build the mapping torus $\Sigma_\phi$. As $\phi \in \II_g$, for any curve $\gamma \subset \Sigma_g$, the homology class $[\gamma] - \phi_*[\gamma]$ is zero. Thus there exists a map of a surface $i : S \to \Sigma_g$ which cobounds $\gamma \cup \phi(\gamma)$. In fact, there exists an embedded surface $i: S \hookrightarrow \Sigma_g$ with boundary
\begin{equation} \label{eq:boundaryofS}
\partial S \iso \gamma \times \left\{1\right\} \cup \phi(\gamma) \times \left\{0\right\}
\end{equation}
See the argument there. But the choice of $S$ is not unique, because given any map of a closed surface $i' : S \to \Sigma_\phi$, the chain $S + S'$ satisfies $\partial(S+S') = \gamma - \phi_*(\gamma)$ as well. Nonetheless, 
\begin{definition}
Given any $S$ satisfying $\partial S =  \gamma - \phi_*(\gamma)$, we can form a closed submanifold of $\Sigma_\phi$ called a \textbf{tube surface} in the following way. 
\begin{itemize}
\item
The \textbf{tube} of a tube surface is the cylinder 
\begin{equation}
\left(\phi(\gamma) \times [0,1/3]\right) \cup \left(\gamma \times [2/3,1]\right). 
\end{equation}
\item
The \textbf{cap} is the subsurface $S$, which we glue to $\Sigma_g \times [1/3,2/3]$. 
\end{itemize}
The result is a smoothly-embedded oriented\footnote{With a choice of orientation induced by those on $\gamma$ and $S^1$.} submanifold $\Sigma_\gamma \subset \Sigma_\phi$, which will descend to a homology class $\Sigma_z$, where we denoted $z = [\gamma] \in H_1(\Sigma_g;\Z)$.
\end{definition}

This defines a homomorphism $H_1(\Sigma_g) \to H_2(\Sigma_\phi)/[F]$, where $[F]$ is the fundamental class of the fiber (because the choice of $S$ is unique only up to maps from closed surfaces.) 
\begin{remark}
If the bundle has a section $\sigma : S^1 \to \Sigma_\phi$, then we can choose $S$ so that $Im(\sigma)$ and $\Sigma_z$ have zero algebraic intersection, which gives a canonical lift $H_1(\Sigma_g) \to H_2(\Sigma_\phi)$. The same for the case of one boundary component (which we considered before.)
\end{remark}

Having chosen an embedding $i : H_1(\Sigma_g) \to H_2(\Sigma_\phi)$ such that $z \mapsto \Sigma_z$, there is an associated
direct sum decomposition of $H_2(\Sigma_\phi)$, namely
\begin{equation}
H_2(\Sigma_\phi) = H_1(\Sigma_g) \oplus [F]. 
\end{equation}
From now on, we will consider the first homology and cohomology groups of the surface as sitting inside the respective group in the mapping torus. \\

The two following definition are equivalent to the combinatorial ones.
\begin{definition}
Given $\phi \in Mod_{g,1}[1]$. we choose a representative $\phi : \Sigma \to \Sigma$ and form the mapping torus $\Sigma_\phi$. Then $\Sigma_\phi$ is a homology $\Sigma_g \times S^1$, but it is not necessarily an intersection $\Sigma_g \times S^1$. The triple intersection of three classes in $H_2(\Sigma_\phi)$ gives a map 
\begin{equation}
\tau_1(\phi) : \bigwedge^3 H_1(\Sigma)  \iso \bigwedge^3 H_2(\Sigma_\phi) \to \Z.
\end{equation}
\end{definition}

Of course, we could also redefined $\tau_1$ in terms of the map on cohomology

\begin{equation}
\bigwedge^3 H^1(\Sigma) \hookrightarrow \bigwedge^3 H^1(\Sigma_\phi) \to \Z \: , \: \alpha \wedge \beta \wedge \gamma \mapsto \alpha \cup \beta \cup \gamma \in \Z. 
\end{equation}

\begin{definition}
Suppose that $\phi \in Mod_{g,1}[2]$. Since $\tau_1(\phi)=0$, a little Poincar\'{e} duality shows that this implies $\alpha \cup \beta = 0$ for all $\alpha,\beta \in H^1(\Sigma)$ and hence all triple Massey products $(\alpha,\beta,\gamma)$ are defined and in $H^2(\Sigma_\phi,\Z)$. The expression
\begin{equation}
(\alpha,\beta,\gamma) \cup \delta \in H^3(\Sigma_\phi) \iso \Z 
\end{equation}
is quadrilinear on $H^1$ and anti-commutes in the first two and last variables. Considered as a tensor and using Poincare duality, 
\begin{equation}
\tau_2(\phi) : \bigwedge^2 H^1(\Sigma) \otimes \bigwedge^2 H^1(\Sigma)  \to \Z \: , \: (\alpha \wedge\beta) \otimes (\gamma \wedge \delta) \mapsto (\alpha,\beta,\gamma) \cup \delta. 
\end{equation}
\end{definition}

Let $\phi = T_{\gamma}$ be a genus two BSCC in the Riemann surface $\Sigma = \Sigma_4$. In Section \ref{sec:compute3dim}, we provide a chain-level refinement of the above story in Morse theory. \\

Given any Morse-Smale pair $(f,g)$ on $\Sigma$, we can define the Morse cochain complex 
\begin{equation}
CM^\bullet(f,g)
\end{equation}
whose cohomology is 
\begin{equation}
H^n(\Sigma) = \begin{cases}
\Z \cdot \left\langle m_+ \right\rangle \: &, \:n=0, \\
\Z \left\langle a_1,b_1,\ldots,a_g,b_g\right\rangle  \: &,\: n=1, \\
\Z \cdot \left\langle m_- \right\rangle \: &, \:n=2
.
\end{cases}
\end{equation}
When $f$ is a minimal Morse function, there is a canonical identification between the complex and its cohomology. It follows that 

\begin{lemma} \label{lem:tau1vanish}
If $(\tilde{f},\tilde{g})$ is any minimal Morse-Smale pair on $\Sigma_\phi$, 
\begin{equation} \label{eq:isomorphismoftildef}
CM^\bullet(\tilde{f},\tilde{g}) \iso HM^\bullet(\tilde{f},\tilde{g}) \iso HM^\bullet(f,g)[\mathfrak{t}] / (\mathfrak{t}^2) \iso CM^\bullet(f,g)[\mathfrak{t}] / (\mathfrak{t}^2).
\end{equation}
where $\mathfrak{t}$ is a formal variable of degree one. 
\end{lemma}
\begin{proof}
$\tau_1$ of every mapping class in the Johnson kernel vanishes. 
\end{proof}
In a slight abuse of notation, we denote the critical points of a minimal Morse function by the same letter used for their cohomology classes. \\

In Section \ref{sec:toymodel}, we will explain how every Morse-Smale pair (coupled with a suitable choice of perturbation datum) yield an $A_\infty$-algebra structure on the Morse-Smale complex. Assuming this fact for the time being, we denote $\tilde{\mu}^d$ for the higher multiplications taking place in the mapping torus, and $\mu^d$ for the ones computed in the surface. The standard independence-of-choices argument (see Corollary \ref{cor:morseindependence}) c that up to quasi-isomorphism, the $A_\infty$-structure on $H^\bullet(\Sigma_\phi)$ depends only on the isotopy class of $\phi$. 

\begin{proposition} \label{prop:mainterm2}
There exists a specific choice of a minimal Morse-Smale pairs $(f,g)$ on $\Sigma$, and $(\tilde{f},\tilde{g})$ on $\Sigma_\phi$, with the following property: Consider the $\mathfrak{t} \cdot ...$ term of $\tilde{\mu}^3(z_3,z_2,z_1)$ where $z_i$ are critical points of degree 1 coming from the surface. Then it is always zero except in the following cases: 
\begin{equation} \label{eq:listofpossiblemu3}
\begin{split}
\tilde{\mu}^3(a_i,b_i,a_j) = \mathfrak{t} b_j, \\
\tilde{\mu}^3(a_i,b_i,b_j) = -\mathfrak{t} a_j, \\
\tilde{\mu}^3(a_j,a_i,b_i) = \mathfrak{t} b_j, \\
\tilde{\mu}^3(b_j,b_i,a_i) = -\mathfrak{t} a_j, \\
\end{split}
\end{equation}
where $1 \leq i \leq 2$ and $3 \leq j \leq 4$.
\end{proposition}

The construction of Morse-Smale pairs with these properties is the content of Section \ref{sec:compute3dim}. These new $\mu^3$-terms imply the existence of a non-trivial Massey product. 

\begin{claim} \label{claim:nontrivialmatrixmasseyproduct}
The ordinary triple matrix Massey product
		\begin{equation}
\langle 
\begin{pmatrix}
 a_1 &  a_2 \\
0 & 0
\end{pmatrix}
,
\begin{pmatrix}
 b_1 & 0 \\
- b_2 & 0
\end{pmatrix}
,
\begin{pmatrix}
 b_1 &  b_3 \\
0 & 0
\end{pmatrix}
\rangle = 
\begin{pmatrix}
0 & \mathfrak{t}b_3 \\
0 & 0
\end{pmatrix} 
		\end{equation}
is nontrivial. 
\end{claim}

\begin{proof}
In Section \ref{subsec:highermultiplicationinmorsetheory}.
\end{proof}

Thus

\begin{corollary}
$\phi$ is not isotopic to the identity.
\end{corollary}

We remark that Proposition \ref{prop:mainterm2} above is the basis for our computation of $\tilde{\mu}_{2F}^3$ in Section \ref{sec:compute4}.

\subsection{Acknowledgments}
The paper is a somewhat updated version of my PhD dissertation. First, I would like to thank my advisor, Paul Seidel, for suggesting this research project, and for being extremely generous with his time, knowledge, and ideas. 

Second, I would also like to thank Baris Kartal, Emmy Murphy, Tom Mrowka, Mohammed Abouzaid, and Umut varolgunes for reviewing an earlier version of this manuscript and suggesting numerous improvements. 

I am grateful to Atanas Atanasov (Nasko), Alexander Perry, Anand Patel, John Lesieutre, Tiankai Liu, and Roberto Svaldi for explanations regarding birational geometry and deformation theory, and Thomas Church regarding different facets of Johnson-Morita theory. Finally, while working on this project I had several beneficial conversations with Ailsa Keating, Francesco Lin, Nate Bottman, and Nick Sheridan. My thanks goes to all of them. 

\section{The pearl complex (definitions and statements)} \label{sec:definitionsandstatements}

The purpose of this section is to introduce all the definitions and concepts required to clearly state the properties of the parametrized pearl complex that we wish to prove. All the analytical details related to the regularization of the moduli spaces involved, as well as the question of coherence and the various consistent choices required to achieve it are treated as a black box or differed to later sections.

\subsection{Locally Hamiltonian fibrations} \label{subsec:locallyhamiltonianfibrations}

\begin{definition}
A \textbf{locally Hamiltonian fibration} (LHF) is a triple $(E, \pi,\Omega)$ where $\pi : E \to B$ is a smooth fibre bundle over a compact manifold $B$, with a closed 2-form $\Omega \in Z^2(E)$ on the total space, such that $\Omega_b := \Omega|_{\pi^{-1}(b)}$ is non-degenerate on $E_b := \pi^{-1}(b)$, for each $b \in B$. 
\end{definition}

Sometimes, when the map is clear from the context, we will just write $(E, \Omega)$. We will always assume that all of our LHF's are \textbf{monotone}, which means the fiber $(M,\omega)$ is a monotone symplectic manifold.

\begin{definition} \label{def:isotopyofLHFs} Let $(E, \pi,\Omega_0)$ and $(E, \pi,\Omega_1)$ be two locally Hamiltonian fibrations with the same underlying topological fiber bundle $\pi : E \to B$. We will say that they are \textbf{isotopic} if there exists a closed form $\Omega \in Z^2([0,1] \times E)$ making $([0,1] \times E, id \times \pi, \Omega)$ into a locally Hamiltonian fibration with $\Omega|_{\left\{0\right\} \times E} = \Omega_0$ and $\Omega|_{\left\{1\right\} \times E} = \Omega_1$. 
\end{definition}
\begin{definition}
Two locally Hamiltonian fibrations are called \textbf{equivalent} if they are related under the equivalence relation generated by isotopy and 2-form-preserving bundle isomorphism.
\end{definition}

Any locally Hamiltonian fibration $\LHF$ defines a subbundle $T^h X$ of $TE$ complementary to $T^v_x E = ker(D_x \pi)$: the fibre $T_e^h E$ is the $\Omega_e$-annihilator of $T^v_x E$. The horizontal subbundle $T^h X$ defines an Ehresmann connection on $E \to B$. As usual, Horizontal lifting
\begin{equation}
TB \to T^h E \: , \: v \mapsto \tilde{v}
\end{equation}
gives rise to a parallel transport system of maps
\begin{equation}
m_\gamma : E_{\gamma(a)} \to E_{\gamma(b)} \: , \: \gamma : [a,b] \to B
\end{equation}
that satisfy the obvious composition rule. It is well-known that $m_\gamma$ are symplectomorphism (this is just Moser's Lemma in disguise. See e.g., \cite{MR1373431} or \cite{MR1978046}). We will often refer to the following result, proved by radial parallel transport: 

\begin{lemma}[2.2.2 in \cite{perutzthesis}] \label{lem:trivialization}
Let $D^m$ denote the open unit disk with coordinates $x_1, \ldots , x_m$. Fix $b_0 \in B$ and an identification 
\begin{equation}
i_{b_0} : (M,\omega) \stackrel{\iso}{\longrightarrow} (E_{b_0},\Omega_{b_0})
\end{equation}
Then given a chart on the base
\begin{equation}
\chi : (D^{m},0) \stackrel{\iso}{\rightarrow} (U_{b_0},b_0) \hookrightarrow (B,b_0)
\end{equation}
There exists a \textbf{trivializing chart}
\begin{equation} \label{eq:trivializingchart}
\tilde{\chi} : (D^{m},0) \times (M,\omega) \stackrel{\iso}{\rightarrow} \chi^* E
\end{equation}
with $\tilde{\chi} \big|_{\left\{0\right\} \times M} = i_{b_0}$ such that
\begin{equation}
\chi^* \Omega = \omega + \sum_i \eta_i \wedge dx_j + \sum_{i<j} (\zeta_{ij} dx_i \wedge dx)_j
\end{equation}
for forms $\eta_i \in \Omega^1(M)$ and $\zeta_{ij} \in \Omega^0(M)$ (these are not intrinsic to $\Omega$). Any two such charts differ by a map
\begin{equation}
(D^{m},0) \to (\Symp(M,\omega),id_M).
\end{equation} \noproof
\end{lemma}

We will also need the following ''family version" of the Moser Lemma. 
\begin{lemma}[2.2.3 in \cite{perutzthesis}]
Isotopic locally Hamiltonian fibrations are isomorphic: In $E \times [0, 1] \to B \times [0, 1]$, parallel transport along the paths $[0, 1] \times \left\{b\right\}$ in $[0, 1] \times B$ induces an isomorphism between $E \times \left\{0\right\}$ and $E \times \left\{1\right\}$. \noproof
\end{lemma}

A fundamental example of an LHF over $B = S^1$ is,

\begin{definition} 
Let $\phi \in \Diff^+(M)$. The \textbf{mapping torus} $M_\phi$ is defined as the quotient of $\R \times M$ by the free $\Z$-action 
\begin{equation}
n \cdot (t,x) = (t-n,\phi^n(x)). 
\end{equation}
It is naturally a smooth fiber bundle $\pi : M_\phi \to S^1$. If in addition, $(M,\omega)$ is symplectic and $\phi^* \omega = \omega$, then the closed two-form $\omega_\phi$, defined as the pullback of $\omega$ under $\R \times M \to M$, descends to the quotient and gives $(M_\phi,\pi,\omega_\phi)$ the structure of a locally Hamiltonian fibration.
\end{definition}

In fact, it is the only example, since for any locally Hamiltonian fibration $(E, \pi, \Omega)$ over $S^1$, the map 
\begin{equation}
F : E \to \R \times E_0  \: \: , \: \: F(x) = (\pi(x), m_{[0,\pi(x)]}^{-1}(x))
\end{equation}
induces an isomorphism of $(E, \pi, \Omega)$ with the mapping torus of its monodromy $m_{[0,1]} \in Symp(E_0,\Omega_0)$. 

\begin{lemma} \label{lem:simplyconnectedequivalence}
Let $(M,\omega)$ be a simply connected, closed symplectic manifold, and let $\phi_0,\phi_1 : M \to M$ be two symplectomorphism. If $[\phi_0]=[\phi_1] \in \pi_0 \Symp(M,\omega)$ then $(M_{\phi_0}, \omega_{\phi_0})$ and $(M_{\phi_1}, \omega_{\phi_1})$ are equivalent.
\end{lemma}

\textbf{Assumptions and Notation.} Throughout this section and the next, the base $B$ would be either $S^1$ or $S^1 \times [0,1]$. Fix a base point $b_0 \in B$. Let $\LHF$ be a locally Hamiltonian fibration with fiber a \emph{monotone} symplectic manifold $(M,\omega)$. For any symplectic fiber bundle, there is a local system of fiberwise homology and cohomology groups, respectively denoted by $\mathcal{H}_\bullet(\pi)$ and $\mathcal{H}^\bullet(\pi)$. We assume that $\mathcal{H}_2(\pi),\mathcal{H}^2(\pi)$ are \emph{trivial local systems}, i.e., the action of the fundamental group $\pi_1(B)$ on the local system is trivial (for example, that is always the case if the symplectic fiber bundle itself is trivial.) In this case, for a class $A \in H_2(M,\Q)$, there is a global locally constant section, denoted
\begin{equation}
s_A : B \to \mathcal{H}_2(\pi)
\end{equation}
whose value is $A$ at a reference fiber. In particular, there exists a section 
\begin{equation}
s^{[\omega]} : B \to \mathcal{H}^2(\pi)
\end{equation}
which takes the value $[\omega] \in H^2(M;\Q)$, as well as a section
\begin{equation}
s^{[c_1]} : B \to \mathcal{H}^2(\pi),
\end{equation}
and they are proportional. We also fix a finite open cover
\begin{equation}
\mathfrak{U} = \left\{\smash{\chi}_\kappa,U_\kappa\right\} 
\end{equation}
of the base by trivializing neighbourhoods as in Lemma \ref{lem:trivialization}. 

\subsection{Almost complex structures}
Recall,

\begin{definition}
Let $V$ be a vector space. An almost complex structure is an endomorphism $J :V \to V$ such that $J^2 = -id$. The space of all almost complex structures is denoted $\complexJ(V) \subset End(V)$. 
\end{definition}

\begin{definition}
Let $(V, \omega)$ be a symplectic vector space. We say that an almost complex structure $J$ is $\omega$-tame if $\omega(v, Jv) > 0$ for all vectors $0 \neq v \in V$, and $\omega$-compatible if it is $\omega$-tame and in addition $\omega(Jv, Jw) = \omega(v,w)$ for all vectors $v,w \in V$. We denote the open subspace of all $\omega$-tame almost complex structures as $\complexJ_\tau(V, \omega) \subset \complexJ(V)$, and use the notation $\complexJ(V,\omega) \subseteq \complexJ_\tau(V, \omega)$ for the subspace of all $\omega$-compatible ones.
\end{definition}

Equip $V$ with a euclidean structure $g$. Then $\complexJ_\tau(V, \omega)$ is a Riemannian manifold with tangent space
\begin{equation} \label{eq:variationofalmostcomplexstructure}
T_{J_0} \complexJ_\tau(V,\omega) = \{Y\in End(V) \: \big| \:  YJ_0 + J_0 Y  =0\},
\end{equation}
and a norm given by $Y \mapsto tr(Y^t Y)$ (for instance). The tangent vectors to $\complexJ(V, \omega)$ has to satisfy in addition

\begin{equation} \label{eq:variationofalmostcomplexstructure}
\omega (Y v,w)+\omega(v,Y w) = 0 \: , \: v,w \in  V
\end{equation}

which is an infinitesimal version of the compatibility condition. The exponential map defines a diffeomorphism
\begin{equation}
\exp_{J_0}: D(0,\rho(g,J_0)) \subseteq T_{J_0} \complexJ_\tau (V,\omega) \longrightarrow \complexJ_\tau(V,\omega) 
\end{equation}
where $D(0,\rho(g,J_0))$ denotes the open disk of radius $\rho(g, J_0) > 0$ which is function that continuously depends on $g$ and $J_0$. \\

More generally, 

\begin{definition}
Let $X$ be a manifold and $V \to X$ a vector bundle. We define $\textbf{J}(V) \to X$ as the bundle whose fibers over $x \in X$ is $\complexJ(V_x)$. If in addition there is a smoothly varying family of symplectic forms $\omega = (\omega_x)_{x \in X}$ on the fibers, then we denote $\textbf{J}_\tau(V, \omega)$ for the bundle of all almost complex structure that are fiberwise tame; and $\textbf{J}_c(V, \omega)$ for the bundle of all almost complex structure that are fiberwise compatible. 
\end{definition}

\begin{remark} Note that every tame almost complex structure $J$ gives a fiberwise metric by the formula
\begin{equation} \label{eq:fiberwise_metrics}
g_{J}(v_x,w_x) := \frac{1}{2}(\omega_x(v_x,J_x w_x) + \omega_x( w_x,J_x v_x)) \: , \: v,w \in T_x V. 
\end{equation} 
\end{remark}

\textbf{Notation.} An almost complex structure $\textbf{J} = (J_x)_{x \in X}$ is a smooth section of one of these bundles. Given an almost complex structure $J_0$, define
\begin{equation}
\begin{split}
T_{J_0} \textbf{J}(V) &:= C^\infty(X,\textbf{J}(V)), \\
\JJ(V) &:= exp_{J_0}({Y \in T_{J_0} \textbf{J}(V) \: \big| \: Y (x) \in B(0, \rho(g(x), J_0(x)))}).
\end{split}
\end{equation} 

Similarly one can define $\JJ_\tau(V,\omega)$ and $\JJ_c(V,\omega)$; as is customary, $\complexJ(V, \omega)$ is used as notation for either $\complexJ_\tau(V, \omega)$ or $\complexJ_c(V, \omega)$, depending on context. Both spaces can be equipped with the $C^\infty$-topology, and are contractible topological spaces (as well as infinite dimensional Fr\'{e}chet manifolds). 

\begin{definition}
Given a locally Hamiltonian fibration $\LHF$, an almost complex structure is an almost complex structure for the vertical tangent bundle $V = T^v E$, and we use the shorthand $\JJ(\pi)$ (resp. $\JJ_\tau(\pi,\Omega),\JJ_c(\pi,\Omega),\JJ(\pi,\Omega)$.)
\end{definition}

\begin{definition}
Let $X$ and $N$ be manifolds and $V \to X$ a vector bundle.
An almost complex structure parametrized by $N$ is a smooth section in the pullback bundle $\textbf{J}(V) \to N \times X$. Fix $J_0$ as above. Let T$_{J_0} \textbf{J}_N (V) \to N \times X$ be the vector bundle with fibres $T_{J_0(n,x)}\textbf{J}(T_x X)$ and set
\begin{equation}
\JJ_N(V) := exp_{J_0}({Y \in T_{J_0} \textbf{J}(V) \: \big| \: Y (n,x) \in B(0, \rho(g(x), J_0(x)))}).
\end{equation} 
the space of all $N$-parametrized families of almost complex structures. Note that we may think of $J in \JJ_N(V)$ as a map $N \to \JJ(V)$. The same procedure can be carried as above in the presense of a symplectic form, and all the notation extends in the obvious way.
\end{definition}
This would mostly be useful for $N = S^2$ (domain-dependent almost complex structures) or when $N$ is a smooth family of nodal maps of a fixed type. Finally, in the case of a symplectic manifold $(M,\omega)$ we consider $V = TM$ and write $\JJ(X)$ (resp. $\JJ_\tau(X,\omega),\JJ_c(X,\omega),\JJ(X,\omega)$) instead of $\JJ(V)$ (...); Similarly, given a locally Hamiltonian fibration $\LHF$, an almost complex structure on $\pi$ is an almost complex structure for the vertical tangent bundle $V = T^v E$, and we use the shorthand $\JJ(\pi)$ (resp. $\JJ_\tau(\pi,\Omega),\JJ_c(\pi,\Omega),\JJ(\pi,\Omega)$.) 

\subsection{Holomorphic maps} \label{subsec:holomorphicmaps1}
This subsection concentrates standard facts about holomorphic curves, and how they should be adapted to the case where the target is a fibration, and the domain is disconnected.

\subsubsection{Basic definition}
Let $(M,\omega)$ be a symplectic manifold. Fix a tame almost complex structure $\textbf{J} = (J_z) \in \complexJ_{S^2}(M,\omega)$ and $(\Sigma,j,dvol_{\Sigma})$ a genus zero Riemann surface (possibly disconnected.) \\

Let $u : \Sigma \to M$ be a smooth map. Then we can associate to it a $(0,1)$--form, called the \textbf{delbar operator},
\begin{equation} \label{eq:delbar}
\delbar_{\textbf{J}} u := \frac{1}{2}(du + J_z \circ du \circ j) \in \Omega^{1}(\Sigma,u^* TM).
\end{equation}
We say that $u$ is $\textbf{J}$-holomorphic if it satisfies the partial differential equation $\delbar_{\textbf{J}} u = 0$. \\

Now let $\LHF$ be a locally Hamiltonian fibration with fiber $(M,\omega)$, $\textbf{J} \in \complexJ^\tau(\pi,\Omega)$ a tame almost complex structure, and $\Sigma$ as before. 

\begin{definition}
A \textbf{vertical map} is a pair $(b, u)$ where $b \in B$ is a point in the base and $u : \Sigma \to E_b$ is a smooth map. 
\end{definition}

Denote $V = T^v E$. We can define a parametrized delbar operator in this setting by the formula
\begin{equation} \label{eq:verticalJholo}
\delbar_{\textbf{J}} u := \frac{1}{2}(du + J_{b,z} \circ du \circ j) \in \Omega^{1}(\Sigma,u^* V_b) 
\end{equation}
for every $b \in B$ and $u : \Sigma \to E_b$.
\begin{definition}
If $\delbar_{\textbf{J}} u = 0$, we say that $(b,u)$ is \textbf{pseudoholomorphic} or \textbf{J-holomorphic}. 
\end{definition}
This is equivalent to $u : \Sigma \to E_b$ being pseudoholomorphic in the sense of \eqref{eq:delbar} with respect to the $z$-dependent almost complex structure $\textbf{J}^b = (\textbf{J}_{b,z})_{z \in \Sigma}$.

\begin{definition}
A pseudoholomorphic map $(b,u)$ is said to be \textbf{multiply covered} if there exists \vspace{0.5em}
\begin{itemize}
\item[(a)]
A compact Riemann surface $(\Sigma',j',dvol_{\Sigma'})$ with a holomorphic map $\phi : \Sigma' \to \Sigma$ which is not a biholomorphism. \vspace{0.5em}
\item[(b)]
Another pseudoholomorphic vertical map $u'$, where $u' : \Sigma' \to E_b$ is such that 
\begin{equation}
u' = u \circ \phi.
\end{equation} 
\end{itemize}
The map $u$ is \textbf{simple} if it is not multiply covered. \vspace{0.5em}
\end{definition}
\begin{remark}
Note that any degree $1$ holomorphic map between connected Riemann surfaces is necessarily biholomorphic. Thus, if $u$ is multiply covered, then the map $\phi$ from (b) has degree $\geq 2$ on some component or $u$ has a constant component. Equivalently, there exists a connected component $\Sigma_v$ such that the restriction $u_v : \Sigma_v \to E_b$ is multiply covered or two different connected components $\Sigma_v$ and $\Sigma_w$ whose image is the same: $u(\Sigma_v) = u(\Sigma_w)$. 
\end{remark}
A point $z \in \Sigma$ with this property that $du|_z \neq 0$ and $u^{-1}(u(z)) \neq \left\{0\right\}$ is called an \textbf{injective point}. We denote by 
\begin{equation}
Z(u) = \left\{z \in \Sigma \: \big| \: du|_z \neq 0 \text{ and } \# u^{-1}(u(z)) > 1\right\}\subseteq \Sigma
\end{equation}
the complement of the set of injective points. 
\begin{definition}
A vertical map $(b,u)$ is \textbf{somewhere injective} if it has an injective point on every connected component of the domain. 
\end{definition}
Clearly any curve which is somewhere injective is simple. \\

All the standard local properties from \cite[Chapter 2]{MR2954391} generalize immediately including: Carleman Similarity Principle, Aronszajn theorem (\cite[Section 2.3]{MR2954391}) and Micallef-White theorem (\cite[Theorem E.1.1]{MR2954391}). As a consequence, we obtain

\begin{theorem}[Unique continuation]
If $\textbf{J}$ is $C^1$ and $(b,u_0),(b,u_1)$ are two vertical maps with a connected domain which agree to infinite order at a point then $u_0 = u_1$.
\end{theorem}
\begin{proof}
Apply Aronszajn's theorem as in \cite[Theorem 2.3.2]{MR2954391}. 
\end{proof}
and the equivalence between the notions of somewhere injective and simple:
\begin{theorem}
If $J$ is a domain-independent $C^2$ almost complex structure and $(b,u)$ is a simple $J$-holomorphic curve, then $u$ is somewhere injective. Moreover, the set $Z(u)$ of noninjective points is finite. 
\end{theorem}
\begin{proof}
Follows from Unique continuation and Micallef-White as in \cite[Proposition 2.5.1]{MR2954391} and \cite[Theorem E.1.2]{MR2954391}.
\end{proof}
In fact, a small modification of the first proof of \cite[Proposition 2.5.1]{MR2954391} shows that if $(b,u)$ is any $J$-holomorphic curve, and we denote $u_{red} : \Sigma_{red} \to E_b$ for the map obtained by removing the components on which $u$ is constant, then we can factor as a composition $u' = u_{red} \circ \phi$ of a simple $J$-holomorphic curve $(b,u')$ without any constant components, and a holomorphic map $\phi : \Sigma' \to \Sigma_{red}$ between the domains.
\subsubsection{Energy and homology class}

\begin{definition} \label{def:homologydecomposition}
Given $A \in H_2(M)$, a \textbf{homology decomposition} of $A$, denoted $\underline{A}$, is a finite collection of  spherical homology classes $\left\{A_v\right\} \in H^S_2(M)$ such that $\sum_v A_v = A$.
\end{definition}

\begin{definition}
The \textbf{homology class} of a vertical map from a connected domain is defined to be the (unique) homology class $A \in H_2(M;\Z)$ which satisfies
\begin{equation}
u_*([\Sigma]) = s_A(b) \in H_2(E_b;\Z),
\end{equation}
and write $[u] = A$. More generally, given a vertical $(b,u)$ with a possibly disconnected domain, we denote $[u]$ for the homology decomposition whose v-th component is the homology class of $u_v : \Sigma_v \to E_b$ the restriction 
to the v-th connected component of the domain.
\end{definition}
\begin{remark}
Note that this Definition only makes sense in light of our standing assumption about $\HH_2(\pi)$ (see the end of Section \ref{subsec:locallyhamiltonianfibrations}.)
\end{remark}

The \textbf{energy} of a vertical map $(b, u)$ is defined as the $L^2$-norm of the one-form $du \in \Omega^{1}(\Sigma,u^* V_b)$:
\begin{equation} \label{eq:Jholo}
E(u) := \frac{1}{2} \int_{\Sigma} |du|^2_{J_b} dvol_\Sigma
\end{equation}
where the norm of a real linear map $L := du|_z : T_z \Sigma \to V_{u(z)} = T^v_{u(z)} E_b $ is defined by choosing a vector $0 \neq \zeta \in T_z \Sigma$ and setting\footnote{This is independent of the choice of $\zeta$ by direct computation.}
\begin{equation}
|L|_J := |\zeta|^{-1} \sqrt{|L(\zeta)|^2_{J_b} + |L(j_z(\zeta))|^2_{J_b}}. 
\end{equation}

\begin{lemma}[The energy identity]  
If $(b,u)$ is $\textbf{J}$-holomorphic then 
\begin{equation}
E(u) = \int_\Sigma u^* \Omega_b. 
\end{equation}
\end{lemma}
\begin{proof}
See e.g. the proof in \cite[Lemma 2.2.1]{MR2954391}. 
\end{proof}
\subsubsection{Strip-like and cylindrical ends}
Let $\hat{S}$ be a compact Riemann surface (possibly with boundary) with $k$ interior marked point $p_i$ and $l$ boundary marked points $q_j$. We equip the surface with a choice of sign for every marked point. Let $S$ denote the Riemann surface obtained by removing the marked point (which are retained as part of the data). Note that $\hat{S}$ can be recovered canonically from $S$. The negative points (respectively positive) are called incoming (resp. outgoing) points at infinity of $S$. We will call $S$ (or $\hat{S}$) a \textbf{pointed Riemann surface}. Isomorphisms of pointed Riemann surfaces are biholomorphic maps which preserve the distinction between incoming and outgoing points at infinity. \vspace{0.1em}
\begin{itemize}
\item
$D$ is the closed unit disc in $\C$; \vspace{0.1em}
\item
$\mathbb{H}$ is the upper half plane; \vspace{0.1em}
\item
$Z =\R \times [-1,1]$ is the infinite strip with conformal coordinates $(s,t)$, an incoming point $s = -\infty$ and an outgoing point $s = +\infty$. For every $l_e>0$ we denote $Z^+_{l_e} = [l_e,+\infty) \times [0,1]$ and $Z^+_{l_e} = (-\infty,-l_e] \times [0,1]$. When $l_e=0$, we write $Z^\pm  =\R^\pm \times [-1,1]$ for the semi-infinite strips.

\vspace{0.1em}
\item
$A =\R \times S^1$ is the infinite cylinder with conformal coordinates $(s,t)$, an incoming point $s = -\infty$ and an outgoing point $s = +\infty$. We denote by $A^\pm  =\R^\pm \times S^1$ the positive and negative semi-infinite cylinders. \vspace{0.1em}
\end{itemize}

\begin{definition}
A collection of \textbf{strip and cylinder data} for $S$ is a choice of consists of proper holomorphic embeddings $\epsilon_\zeta$, one for each marked point $\zeta \in S$ such that 
\begin{itemize}
\item
If $\zeta$ is interior and positive, we assign a \textbf{positive cylindrical end}, 
\begin{equation}
\epsilon_\zeta : A^+ \to S,
\end{equation}
satisfying 
\begin{equation}
\epsilon_\zeta^{-1}(\partial S) = \phi \: , \: \lim_{s \to +\infty}\epsilon_\zeta(s,\cdot) = \zeta. 
\end{equation}
\item
If $\zeta$ is interior and negative, we assign a \textbf{negative cylindrical end}, 
\begin{equation}
\epsilon_\zeta : A^- \to S,
\end{equation}
satisfying 
\begin{equation}
\epsilon_\zeta^{-1}(\partial S) = \phi \: , \: \lim_{s \to -\infty}\epsilon_\zeta(s,\cdot) = \zeta. 
\end{equation}
\item
If $\zeta$ is boundary and positive, we assign a \textbf{positive strip-like end}, 
\begin{equation}
\epsilon_\zeta : Z^+ \to S,
\end{equation}
satisfying 
\begin{equation}
\epsilon_\zeta^{-1}(\partial S) = \R^+ \times \left\{-1,1\right\} \: , \: \lim_{s \to +\infty}\epsilon_\zeta(s,\cdot) = \zeta. 
\end{equation}
\item
If $\zeta$ is boundary and negative, we assign a \textbf{negative strip-like end}, 
\begin{equation}
\epsilon_\zeta : Z^- \to S,
\end{equation}
satisfying 
\begin{equation}
\epsilon_\zeta^{-1}(\partial S) = \R^- \times \left\{-1,1\right\} \: , \: \lim_{s \to -\infty}\epsilon_\zeta(s,\cdot) = \zeta. 
\end{equation}
\end{itemize}
We require that the images of all the ends are pairwise disjoint. 
\end{definition}
\begin{remark}
Note that the Riemann surface structure on $Z^\pm$ and $A^\pm$ extends to their one-point compactifications, and each strip-like or cylinderical end extends to a holomorphic embedding $\hat{\epsilon_\zeta}$ into $\hat{S}$ which takes $\pm \infty$ to $\zeta$.
\end{remark}

\subsection{On pearls and pearl trees}
Let $d \geq 2$. We will (mostly) be interested in two types of Riemann surfaces.
\begin{definition} \label{def: markeddisc}
A \textbf{$(d+1)$-pointed disc} $D$, is a pointed Riemann surface whose compactification $\hat{D}$ is the closed unit disc, and which has one incoming point at infinity and $d$ outgoing ones on the boundary. Our convention is to number the points at infinity, so that the incoming point is $z_0$ and the outgoing points are numbered $\left\{z_1,\ldots,z_d\right\}$ according to their counterclockwise order around the boundary $\partial D$. 
\end{definition}
Recall that for every Riemann surface $\Sigma$ with a nonempty boundary, there exists a double cover
\begin{equation}
\pi : \Sigma_\C \to \Sigma
\end{equation}
by a compact Riemann surface $\Sigma_\C$ (called the \textbf{complex double}) and an antiholomorphic involution $\sigma : \Sigma_\C \to \Sigma_\C$ such that $\pi \circ \sigma = \pi$; as well as a holomorphic embedding $\iota : \Sigma \to \Sigma_\C$ such that $\pi \circ \iota = id$. Moreover, the triple $(\Sigma_\C,\sigma,\iota)$ is unique up to isomorphism. 
\begin{definition}
The result of doubling $D$, denoted 
\begin{equation}
(C,\underline{z} = \left\{z_0,\ldots,z_d\right\})
\end{equation}
is a punctured $\CP{1}$ which we call a \textbf{pearl}.
\end{definition}
Note that the marked points of a pearl inherit a designation as incoming/outgoing from the disc, and any choice of strip-like ends for $D$ doubles to give cylindrical coordinates around the punctures. A small generalization of the above is \textbf{broken} or \textbf{nodal} pearls which come from doubling nodal discs (see Section \ref{subsec:consistent_choice_of_cylindircal_ends} for more). 

\begin{definition}
A \textbf{d-leafed pearl tree} $P = (T,\underline{C})$ consists of \vspace{0.5em}
\begin{itemize}
\item
A metric tree 
\begin{equation}
T = (T,g_T) 
\end{equation}
with $(d+1)$-external edges of infinite type. \vspace{0.5em}
\item
A collection of $|v|$-marked pearls
\begin{equation}
\underline{C} = \left\{(C_v,\underline{\smash{z}}_v)\right\}_{v \in \Vert^{finite}(T)},
\end{equation}
indexed the set of all vertices of finite type. \vspace{0.5em}
\end{itemize}
the marked points on $C_v$ are ordered in such a way that they are in order-preserving bijection with the flags adjacent to $v$. We require that if we were to collapse all edges of the tree in the obvious way and glue all the pearls $C_v$ together, the result would be a broken pearl (i.e., the involution action on different pearls is compatible.)  We call $\underline{C}$ the \textbf{sphere part} and $T$ the \textbf{tree part}. 
\end{definition}
The combinatorial type of any pearl tree is the underlying trees with the additional information of which edges have zero length. An \textbf{isomorphism} of pearl trees is an automorphism of the underlying metric tree together with an isomorphism of collections of nodal curves. The \textbf{topological realization} $|P|$ of a pearl tree $P$ is obtained by removing the vertices from the tree and gluing in the nodal curves by attaching the marking to the edges of tree, according to the bijection between the marked points and the flags. A pearl tree $P$ is \textbf{stable} if every nodal pearl it contains is stable. \\

\begin{definition}
A \textbf{Floer datum} 
\begin{equation}
(f^{base},g^{base},J^{base}) 
\end{equation}
consists of: a Morse-Smale pair $(f^{base},g^{base})$ on the total space $E$, as well as a tame almost complex structure $J^{base} \in \complexJ(\pi,\Omega)$. 
\end{definition}

We will usually denote 
\begin{equation}
X^{base} = \nabla_{g^{base}} f^{base} 
\end{equation}
for the gradient field. \\

Let $P = (T,\underline{C})$ be a pearl tree. Naively, a pseudo-holomorphic pearl tree map based on $P$ should be a continuous map $u : |P| \to E$ whose restriction to the tree satisfies the negative gradient trajectory equation with respect to $X^{base}$, and whose restriction to the sphere is a vertical (and possibly nodal) $J$-holomorphic map. However, just like in the Morse $A_\infty$-case, to achieve transversality we need to perturb the PDE's involved in the definition.

\begin{definition}
A \textbf{perturbation data for a pearl tree} $P = (T,\underline{C})$ is a pair 
\begin{equation}
\textbf{Y}^P = (\textbf{X}^P,\textbf{J}^P),
\end{equation}
where $\textbf{X}^P$ is a choice, for each edge $e \in \Edge(T)$ of a family of vector fields
\begin{equation}
X^P_e :e \to C^\infty(TE)
\end{equation}
which vanish away from a bounded subset; and  
\begin{equation}
\textbf{J}^P = (J^P_{\underline{b},x,z}) \in \complexJ_{\underline{C}}(\pi^{\otimes ...},\Omega^{\boxtimes ...}) 
\end{equation}
is a domain-dependent almost complex structure. This notation requires some explanation: denote $v = |\Vert^{int}(T)|$. Then here $\textbf{J}^P$ actually takes values in the LHF
\begin{equation}
\begin{split}
\pi^{\otimes v} &:= \pi \times \ldots \times \pi : E \times \ldots \times E \to B \times \ldots \times B, \\
\Omega^{\boxtimes v} &:= p_1^* \Omega + \ldots + p_{v} \Omega
\end{split}
\end{equation}
but the restriction of $\textbf{J}^P$ to the i-th connected component of the domain $\underline{C}$ depends on: the tuple of base coordinates $\underline{b}$ for all components $\underline{C}$, the fiber coordinate $x$ of the i-th factor of $E^{\times}$ \emph{only}, and $z \in P$. In addition, we require $\textbf{J}^P$ to be compatible on the cylindrical ends with the Floer data, in the sense that
\begin{equation}
J^P(\epsilon_{z_i}(s,t)) = J^{base}
\end{equation}
for each marked point $z_i$ and $(s,t) \in A^\pm$.
\end{definition}

\begin{remark}
Caution: this is actually a crucial point. Our almost complex structures depend on the simultaneous position of all sphere components with respect to the base projection. 
\end{remark}

\begin{definition} \label{def:pseudoholomorphicpearltreemap}
Let $P = (T,\underline{C})$ be a pearl tree with a choice of perturbation data $\textbf{Y}^P$. A \textbf{pseudo-holomorphic pearl tree map} $\textbf{u} = (\underline{b},u)$ based on $P$ is: a collection of points in the base, indexed by the internal vertices of $T$; and a continuous map $u : |P| \to E$ such that
\begin{itemize}
\item[(a)]
For every internal vertex $v \in \Vert^{int}(T)$, the restriction 
\begin{equation}
(b_v,u_v)
\end{equation}
to $C_v$ is a vertical (possibly nodal) $J$-holomorphic map. That is, $u_v : C_v \to E_{b_v}$ satisfies 
\begin{equation}
\delbar_{\textbf{J}_v}(u) = \frac{1}{2}(du + \textbf{J}_v \circ du \circ j) = 0
\end{equation}
\item[(b)]
Restricted to every edge $e \in \Edge(T)$, the map $u$ is a negative gradient trajectory. That is, writing $t_e$ for the induced coordinate
\begin{equation}
d u (\partial_{t_e}) = (-X^{base} + X_e)\big|_{t_e}
\end{equation}
\end{itemize}
\end{definition}
\begin{remark}
We will occasionally denote $\underline{u}$ when we want to think of the restriction of $u$ to $\underline{C}$, which is a collection of nodal pseudo-holomorphic maps.
\end{remark}

As a set, we define 
\begin{equation}
\PP_d 
\end{equation}
to be the moduli space of all d-leafed pearl trees up to isomorphism. There is an obvious metrizable topology on $\PP_d$ coming from a combination of the topologies on the moduli spaces of Stasheff trees and marked discs ($\cStashefftree_d$, and $\cStasheffdisc_{|v|}$ respectively.) We can extend $\PP_d$ by allowing the length of edges of the underlying tree to go to infinity. The resulting space, denoted $\bar{\PP}_d$, is a compact, Hausdorff, metrizable space with a natural stratification according to combinatorial type. We can endow each strata with a smooth structure. We remark that unlike the case of Stasheff trees, $\PP_d$ has multiple cells in the top dimension. 
\begin{remark}
This subject would be taken up again, in more detail, in Section \ref{sec:masseyproductquantum}. 
\end{remark}
An \textbf{isomorphism} of pearl tree maps is an automorphism of the underlying metric tree together with an isomorphism of collections of nodal vertical maps. Note that two pearls tree maps $\textbf{u},\textbf{u}'$ with $b_v \neq b'_v$ can never be isomorphic. A pearl tree map $\textbf{u}$ is \textbf{stable} if every pearl map it contains is stable. If the underlying pearl tree is stable we say that $\textbf{u}$ is \textbf{domain-stable}. Domain-stablity implies stability, but the converse is not true.

\begin{definition}
As a set, let
\begin{equation} \label{eq:modulispaceofpearltrees}
\PP_{\Upsilon}(\underline{A},\morseLabel)
\end{equation}
denote the space of all isomorphism classes of maps based on pearl trees $P$ of combinatorial type $\Upsilon$, with homology decomposition $[\textbf{u}] = \underline{A}$ and asymptotics $\morseLabel$. 
\end{definition}

\begin{definition} 
The \textbf{expected codimension} of $(\Upsilon, \underline{A})$ is given by a formula involving the number of zero-length or broken edges, and the \textbf{defect} of $\underline{A}$ (which counts the excess marked points in ghost spheres):
\begin{equation}
\begin{split}
\vircodim(\Upsilon, \underline{A}) &= |\Edge^0_f(\Upsilon)| + |\Vert^{\infty}(\Upsilon_P) - (d+1)| \\
& + \sum_{A_v=0} (|v|-3). 
\end{split}
\end{equation}
The \textbf{expected dimension} of $(A,\morseLabel)$ is 
\begin{equation}
\virdim(A, \morseLabel) = \deg(p_0) - \sum_{i=1}^d \deg(p_i) + (d -2) + 2 c_1(A).
\end{equation}
Finally, the \textbf{expected dimension} of $\PP_{\Upsilon}(\underline{A},\morseLabel)$ is defined as the difference
\begin{equation} \label{eq:expecteddimforpearl}
\virdim(\PP_{\Upsilon}(\underline{A},\morseLabel)) = \virdim(A, \morseLabel)  - \vircodim(\Upsilon, \underline{A}).
\end{equation}
\end{definition}

\begin{definition}
Given an integer $d \geq 2$, a homology class $A \in H_2(M)$, and a sequence of critical points $\morseLabel = (p_d,\ldots,p_1,p_0)$ we define
\begin{equation} \label{eq:eq:modulispaceofpearltrees}
\PP_{d}(A,\morseLabel) = \bigcup_{\Upsilon} \bigcup_{\underline{A}} \PP_{\Upsilon}(\underline{A},\morseLabel)
\end{equation}
where the union is taken over all d-leafed combinatorial types of expected codimension zero.  
\end{definition}

The following is the main result proved in Section \ref{sec:masseyproductquantum}:
\begin{theorem} \label{thm:mainfacts}
For a generic choice of universal perturbation data, each moduli spaces \eqref{eq:modulispaceofpearltrees} can be given the structure of a smooth oriented manifold of the expected dimension \eqref{eq:expecteddimforpearl}. Moreover, if $\virdim(A, \morseLabel) = 0$ then $\PP_{d}(A,\morseLabel)$ is a finite union of points, and if $\virdim(A, \morseLabel) = 1$ it is a 1-manifold with ends. 
\end{theorem}

It is worth to pause for a moment and reflect upon what equation \ref{eq:expecteddimforpearl} means. Given any convergent sequence of pearl tree maps $\textbf{u}^\nu \to \textbf{u}^\infty$, the domains must converge as well. It can happen that a pearl $C_v$ in the domain of the $\textbf{u}^\nu$ breaks in the limit into $C_{w_0}$ and $C_{w_1}$. This is an \emph{internal point} in the moduli space. But if one of the bubbles, say $C_{w_0}$, is a ghost bubble then \emph{it has the same local contribution to the dimension formula as a tree vertex}. A nice way to visualize it is that the moment such a ghost bubble forms it instantly ''deflates". We can pretend that in fact we have two types of ''combinatorial sub-structures" inside each pearl map:   \vspace{0.5em}
\begin{itemize}
\item
\textbf{Null clusters}. These is a collection of pearls (always with positive energy) connected by edges of zero combinatorial length. The restriction of a pearl map to a null cluster can be thought of as maps from a single, nodal pearl. \vspace{0.5em}
\item
\textbf{Metric trees}. Since spheres of zero energy display the same behavior we have seen in the Stasheff tree case, we might as well treat any sub-pearl tree that contains only vertices with zero energy essentially as a metric trees.  \vspace{0.5em}
\end{itemize}
See Figure \ref{fig:combinatoricsofpearltreemaps} for an illustration. \\

\begin{figure}  
  \begin{subfigure}[b]{.45\linewidth}
		\centering
				\fontsize{0.25cm}{1em}
			\def\svgwidth{4cm}
			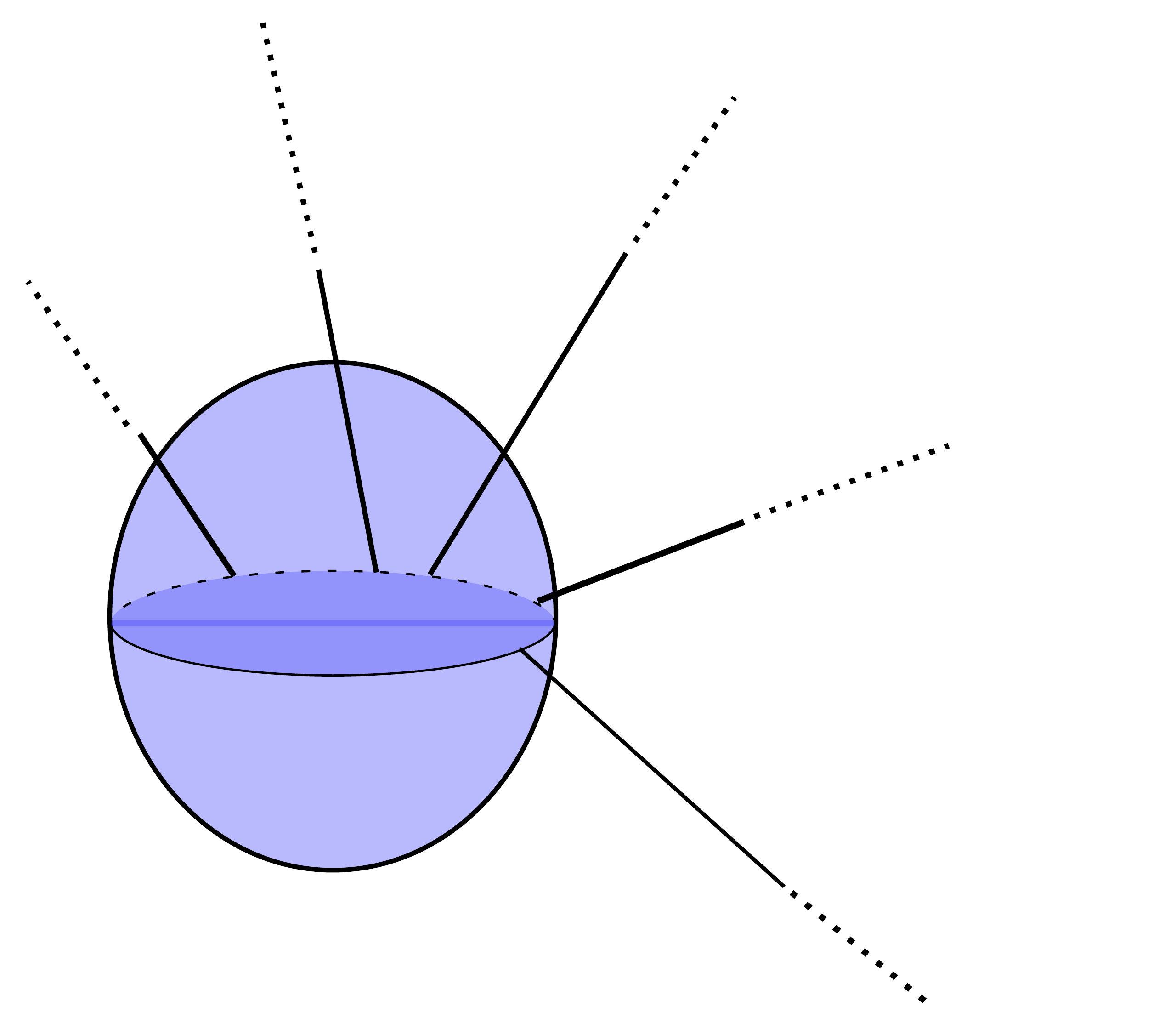
			\caption{The pearl $C_v$}
			\label{fig:morseoperationsA}
			\vspace*{8mm}
  \end{subfigure}\hfill
  \begin{subfigure}[b]{.45\linewidth}
	\centering
					\fontsize{0.25cm}{1em}
			\def\svgwidth{4cm}
			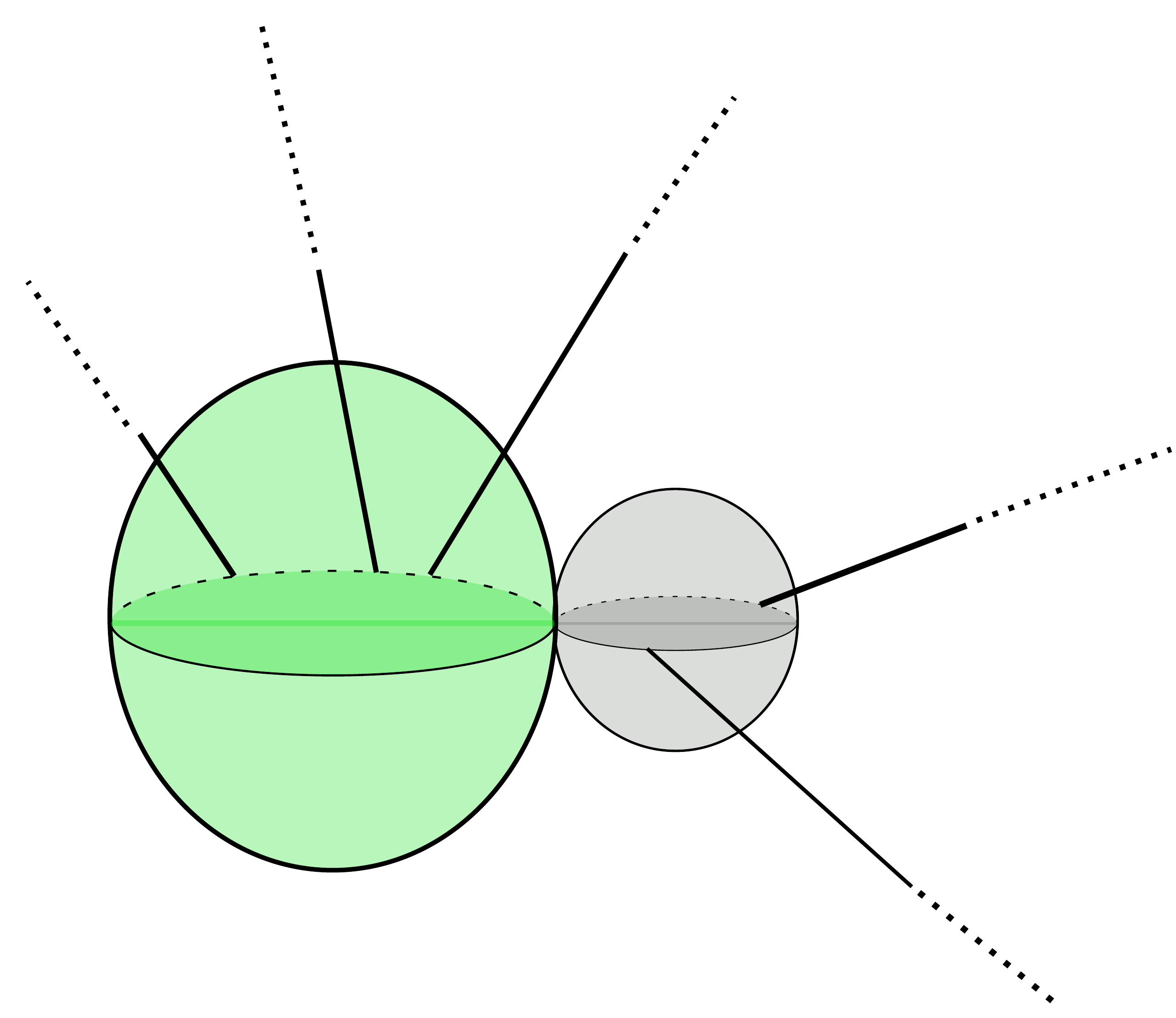 
			\caption{Domain degeneration into $C_{w_1}$ (green) and $C_{w_0}$ (grey)}
			\label{fig:morseoperationsB}
			\vspace*{8mm}
  \end{subfigure} 
	  \begin{subfigure}[b]{.45\linewidth}
	\centering
					\fontsize{0.25cm}{1em}
			\def\svgwidth{4cm}
			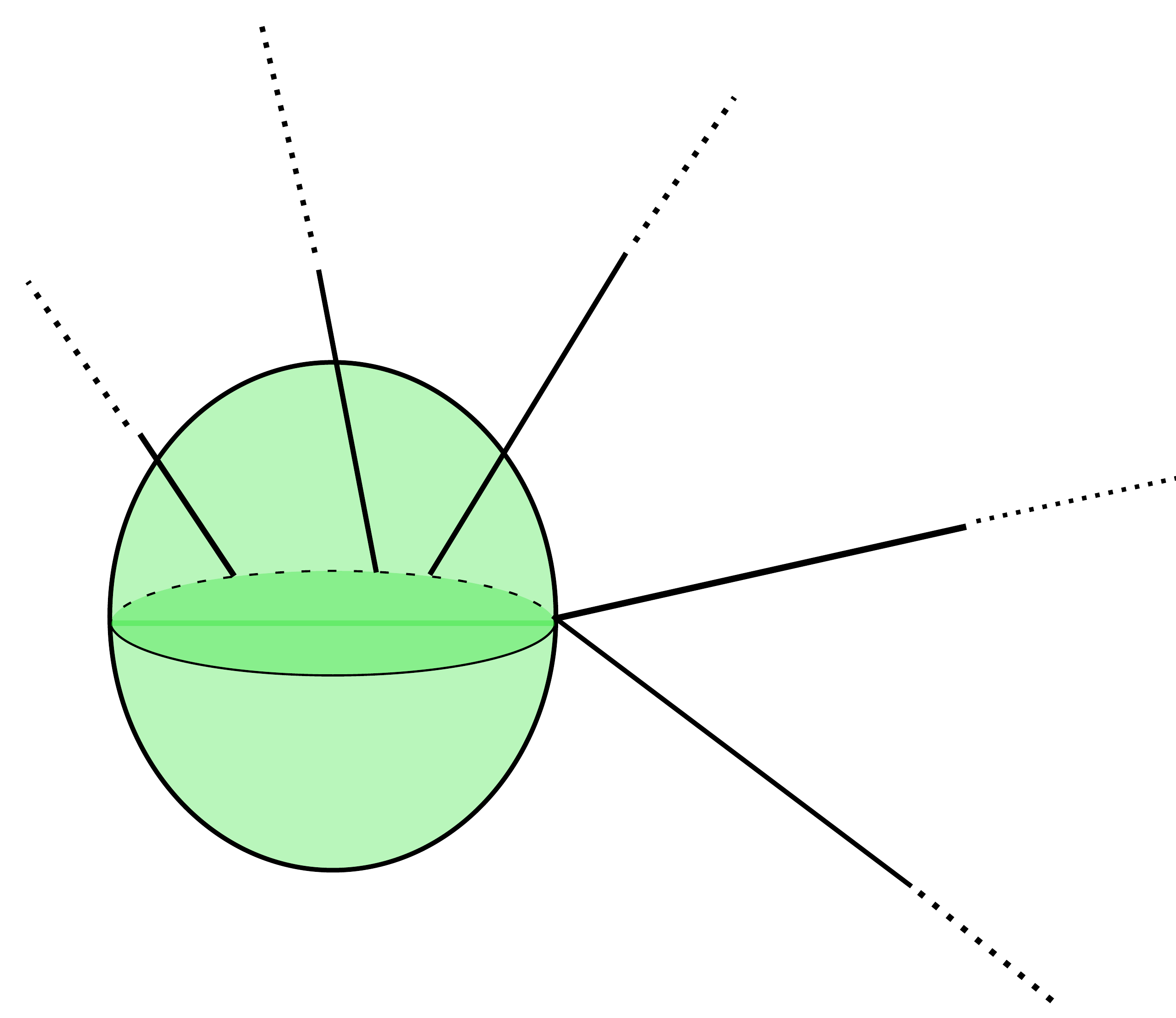
			\caption{$C_{w_0}$ ''instantly deflates"}
			\label{fig:morseoperationsC}
			\vspace*{2mm}
  \end{subfigure} \hfill
	  \begin{subfigure}[b]{.45\linewidth}
	\centering
					\fontsize{0.25cm}{1em}
			\def\svgwidth{4cm}
			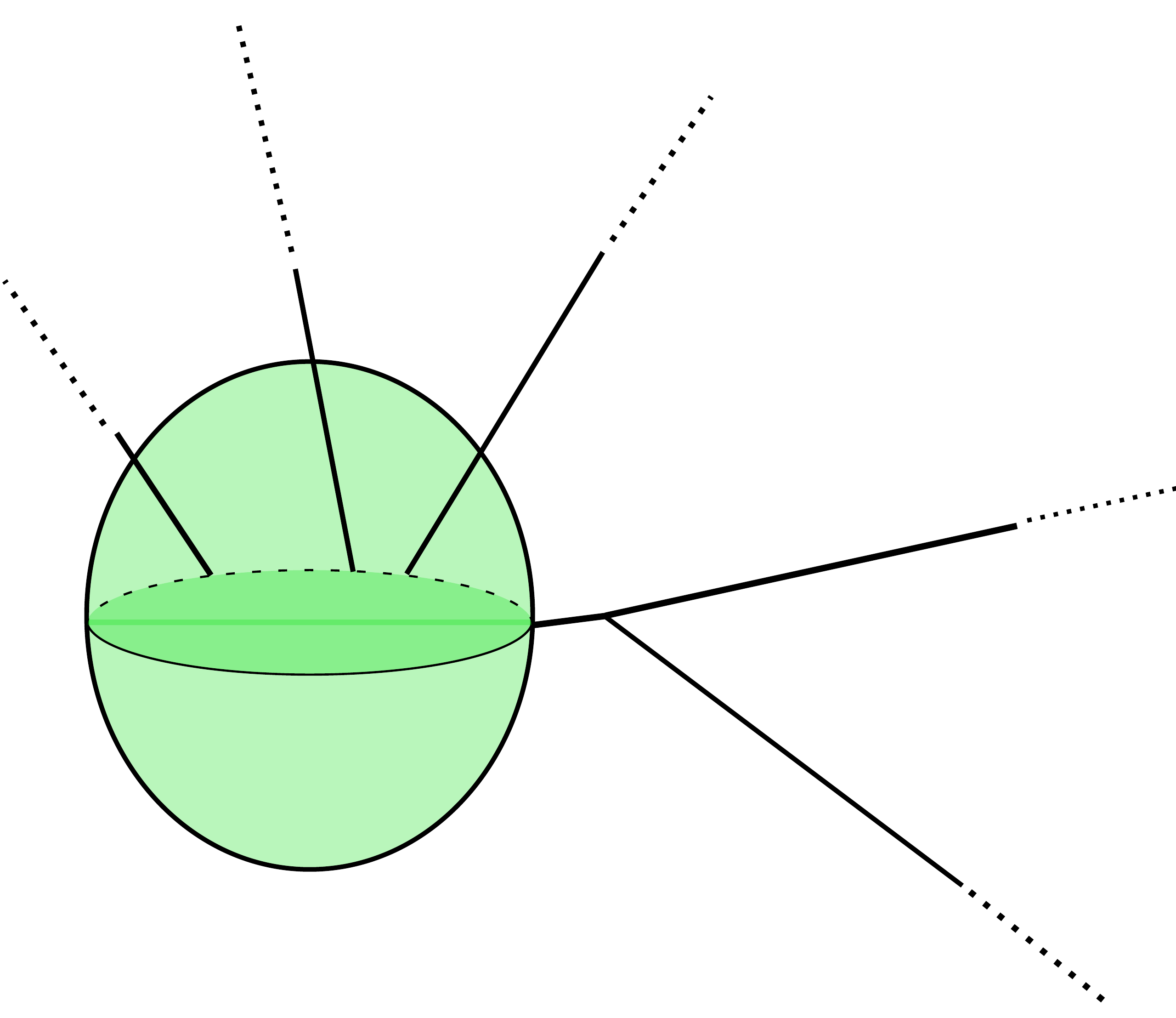
			\caption{We collapse all pearls with zero energy }
			\label{fig:morseoperationsD}
			\vspace*{2mm}
  \end{subfigure}
	\caption{Combinatorics of pearl tree maps}
	\label{fig:combinatoricsofpearltreemaps}
\end{figure}

Theorem \ref{thm:mainfacts} allows us to make the following definition. 

\begin{definition} \label{def:pearlcomplexainfty}
Let $A \in H_2(M)$ be a homology class. We define a degree $2-d-2c_1(A)$ multilinear operation on generators by
\begin{equation}
\mu^d_A([p_d],\ldots,[p_1]) = \sum_{[\textbf{u}] \in \PP_d(A,\morseLabel)} (-1)^{(\dim(E)+1)(\deg(p_0) + \maltese_d)} \mu_\psi([p_d],\ldots,[p_1])
\end{equation}
where the sum is taken over all $p_0$ such that the expected dimension of $(A,\morseLabel)$ is zero, the induced map on orientation lines $\mu_\psi$ is the one given in Definition \ref{def:inducedmaponorientationlines} and 
\begin{equation}
\maltese_d = \sum^d_{k=1} k \deg(p_k).
\end{equation}
Extending $\mu^d_A$ by linearity to $CM^{\bullet}(f^{base},g^{base};\Gamma)$, this defines a map 
\begin{equation}
\mu^d_A : CM(f^{base},g^{base})^{\otimes d} \to CM(f^{base},g^{base}).
\end{equation}
\end{definition}

\begin{definition} \label{def:pearlcomplexainfty2}
We define the \textbf{d-th order higher multiplication map} as
\begin{equation}
\mu^d([p_d],\ldots,[p_1]) = \sum_{A \in H_2(M)} \mu^d_A ([p_d],\ldots,[p_1]).
\end{equation}
\end{definition}

\begin{proposition} \label{prop:ainftypearl}
The maps $(\mu^d)_{d \geq 1}$ satisfy the $A_\infty$-axiom, and $\mu^d_0$ is the d-th higher multiplication from Section \ref{sec:toymodel}. 
\end{proposition}
\begin{proof}
In Section \ref{subsec:ainftypearl}. 
\end{proof}

\section{Homological algebra: Definitions and Preliminaries}  \label{sec:algebraic}
In this section, we define the quantum Johnson homomorphism and the quantum Massey products, as well as recall some homological algebra needed for the definition thereof. \\

\textbf{Conventions and Notations.} For now, we work over a ground field $\field$ (of characteristic zero), but we shall discuss the adjustment to more general ground rings soon. Everything is \textit{cohomologically graded}, which means that the shift operator acts by shifting the grading \textit{down}, i.e., 
\begin{equation}
V^\bullet[n] := V^{\bullet+n}. 
\end{equation}
We follow the \textit{Koszul sign rule} which means that if $B^\bullet,C^\bullet,D^\bullet$ are graded $\field$-vector spaces and $x,y \in C^\bullet$, $ f,g\in Hom^\bullet(C,D)$, $h,k  \in Hom^\bullet(B,C)$ are homogeneous elements then
\begin{equation}
\begin{split}
(f \otimes g)(x \otimes y) &= (-1)^{|g| \cdot |x|} f(x) \otimes g(y),  \\
(f \otimes g)(h \otimes k) &= (-1)^{|g| \cdot |h|} (f \circ h) \otimes (g \circ k).
\end{split}
\end{equation}
When considering formulas for multilinear operations (e.g., Hochschild cochains) 
\begin{equation}
x_s \otimes \ldots x_1 \mapsto f(x_s,\ldots,x_1)
\end{equation}
it will often be convenient to use the short hand $|x|' = |x|-1$ and $|f|' = |f|-1$ to denote the shifted-by-1 degree of an element or a morphism; we define 
\begin{equation}
    \maltese_i := |x_1|' + \cdots + |x_i|'.
\end{equation}
Occasionally, we will use \textit{Sweedler's notation} for the iterated coproduct (See Definition \ref{def:iteratedcoproduct} below) and write
\begin{equation} 
\begin{split}
\Delta(x) &= \sum_{(x)} x_{(2)} \otimes x_{(1)}, \\
				&\vdots \\
\Delta^n(x) &= \sum_{(x)} x_{(n)} \otimes \ldots \otimes x_{(1)}. \\
\end{split}
\end{equation}

\subsubsection{Co-algebras} 

\begin{definition}
A (coassociative, graded) \textbf{coalgebra} $(C,\Delta)$ over $\field$ is a graded $\field$-vector space $C^\bullet$ with a degree zero $\field$-linear map 
\begin{equation}
\Delta : C \to C \otimes C
\end{equation}
called the \textbf{coproduct}, such that coassociativity equation (obtained by simply inverting all the arrows in the associtivity diagram) holds:
\begin{equation}
(\Delta \otimes Id_C) \circ \Delta = (Id_C \otimes \Delta) \circ \Delta : C \to C \otimes C \otimes C.
\end{equation}
\end{definition}
\begin{example} \label{example:1} Let $C = \Q[t]$ be the polynomial ring in one variable $t$ of degree zero. The linear map
\begin{equation}
\Delta : \Q[t] \to \Q[t] \otimes \Q[t] \: \: , \:\: \Delta(t^n) = \sum_{i=0}^{n} t^i \otimes t^{n-i}, 
\end{equation}
makes $(C,\Delta)$ into a coalgebra. The dual of $(C,\Delta)$ is exactly the algebra of formal power series $A = \Q[[t]]$ with the standard multiplication. 
\end{example}

\begin{remark}
It is important to note that the dual of an algebra is usually \textbf{not} a coalgebra! Heuristically, this asymmetry comes from the fact that for an infinite dimensional vector space $V$, there always exists a map $V^\wedge \otimes V^\wedge \to (V \otimes V)^\wedge$, but there isn't neccesarily a map $(V \otimes V)^\wedge \to V^\wedge \otimes V^\wedge$ going in the other direction. 
\end{remark}
More generally,
\begin{definition} \label{def:iteratedcoproduct}
We define the \textbf{iterated coproduct} 
\begin{equation}
\Delta^n : C \to C^{\otimes (n+1)} 
\end{equation}
recursively by setting $\Delta^0 = Id_C$, $\Delta^1 = \Delta$, and for every $n \geq 2$
\begin{equation}
\Delta^n := (\Delta \otimes \ldots \otimes Id_C) \circ \Delta^{n-1}. 
\end{equation}
Note that since $\Delta$ is coassociative, we have 
\begin{equation} \label{eq:iteratedproduct}
\Delta^n := (Id_C \otimes \ldots \otimes \Delta \otimes \ldots \otimes Id_C) \circ \Delta^{n-1}
\end{equation}
as well. 
\end{definition}
There are simple combinatorial formulas for $\Delta^n$, namely, for every $s \geq 1$, and every $a_s,\ldots,a_0 \geq 0$, we have
\begin{equation}
(\Delta^{a_s} \otimes \ldots \otimes \Delta^{a_0}) \Delta^s = \Delta^{s + \sum_{i=0}^{s} a_i}
\end{equation}
and for every $n \geq 0$, we have
\begin{equation}
ker(\Delta^{n+1}) = \left\{x \in C \: \big| \: \Delta(x) \in ker(\Delta^n) \otimes ker(\Delta^n) \right\}. 
\end{equation}

\begin{definition} A coalgebra $(C,\Delta)$ is said to be \textbf{counital} if it is equipped with a \textbf{counit}: a linear map $\epsilon : C \to \field$ which satisfies 
\begin{equation}
(\epsilon \otimes Id_C) \circ \Delta = (Id_C \otimes \epsilon) \circ \Delta.
\end{equation}
\end{definition}
That is,
\begin{equation} \tag{$1.17^\star$}
\Delta(x) = \sum_{(x)} \epsilon(x_{(2)}) \otimes x_{(1)} = \sum_{(x)} x_{(2)} \otimes \epsilon(x_{(1)}). 
\end{equation}

In a slight abuse of notation, we abbreviate (counital) coassociative coalgebra into coalgebra if no confusion can arise, and we shorten the notation $(C,\Delta, \epsilon)$ into $(C,\Delta)$ or even just $C$.

\begin{definition} Let $(C , \Delta)$ and $(B , \Gamma)$ be coalgebras. A degree $j$ \textbf{morphism} of coalgebras 
\begin{equation}
f : (C , \Delta) \to (B , \Gamma)
\end{equation}
is a degree $j$ morphism of graded vector spaces that commutes with coproducts, i.e.
\begin{equation} \label{eq:defmorphism}
\Gamma \circ f = (f \otimes f) \circ \Delta : C \to B[j] \otimes B[j]. 
\end{equation}
In Sweedler's notation
\begin{equation} \tag{$1.19^\star$}
\sum_{(f(x))} {f(x)}_{(2)} \otimes {f(x)}_{(1)} = \sum_{(x)} f(x_{(2)}) \otimes f(x_{(1)}).
\end{equation}
If $(C , \Delta)$ and $(B , \Gamma)$ are counital we require the morphism to commute with the counits as well. 
\end{definition}

We denote the graded $\field$-vector spaceof coalgebra morphisms as $Hom^\bullet(C,B)$.

\begin{lemma} If $f : (C , \Delta) \to (B , \Gamma)$ is a morphism of coalgebras then, for every $n \geq 1$, we have
\begin{equation}
\Gamma^n f = (\otimes^{n+1}f)\Delta^n : C \to \otimes^{n+1} B. 
\end{equation}
\end{lemma}
\begin{proof} For $n=1$ this is the equation \ref{eq:defmorphism}, now assume that $n \geq 2$ and the claim holds for $n-1$. Then
\begin{equation}
\Gamma^n f = (id_B \otimes \Gamma^{n-1}) \Gamma f = (f \otimes \Gamma^{n-1} f)\Delta = (f \otimes (\otimes^n f) \Delta^{n-1})\Delta = (\otimes^{n+1}f)\Delta^n.
\end{equation}
\end{proof}

\begin{definition}
Let $(C,\Delta)$ be a graded coalgebra and $p: C \to V$ a linear map of graded vector spaces. We shall say that $p$ is a \textbf{system of cogenerators} for $C$ if for every $x \in C$ there exists $n \geq 0$ such that 
\begin{equation}
(\otimes^{n+1} p) \circ \Delta^n(x) \neq 0
\end{equation}
in $V^{\otimes {n+1}}$. Equivalently, $p: C \to V$ is a system of cogenerators if the map
\begin{equation}
C \to \coprod_{n \geq 0} \otimes^{n+1} V \: \: , \: \: c \mapsto (c,\Delta c , \Delta^2 c, \ldots)
\end{equation}
is injective. 
\end{definition}

\begin{lemma} \label{lem:cogenerator}
Let $p : B \to V$ be a system of cogenerators for a graded coalgebra $(B,\Gamma)$. Then every morphism of graded coalgebra $f : (C,\Delta) \to (B,\Gamma)$ is completely determined by the composition $p \circ f : C \to V$. 
\end{lemma}
\begin{proof}
Let $f,g : (C,\Delta) \to (B,\Gamma)$ be two morphism of graded coalgebra such that $p \circ f = p \circ g$. It is sufficient to show that for every $x \in C$ and every $n \geq 0$ we have 
\begin{equation}
(\otimes^{n+1} p) \circ \Delta^n(f(x)) = (\otimes^{n+1} p) \circ \Delta^n(g(x)).
\end{equation}
But $\Gamma^n(f(x)) = (\otimes^{n+1} f ) \circ \Delta^n(x)$ and $\Gamma^n(g(x)) = (\otimes^{n+1} g ) \circ \Delta^n(x)$, so
\begin{equation}
\begin{split}
(\otimes^{n+1} p) \circ \Delta^n(f(x)) &= (\otimes^{n+1} p) \circ (\otimes^{n+1} f ) \circ \Delta^n(x) = (\otimes^{n+1} (p \circ f) ) \circ \Delta^n(x)  \\
&= (\otimes^{n+1} (p \circ g) ) \circ \Delta^n(x) = (\otimes^{n+1} p) \circ (\otimes^{n+1} g ) \circ \Delta^n(x) \\
&= (\otimes^{n+1} p) \circ \Delta^n(g(x)).
\end{split}
\end{equation}
\end{proof}

\begin{example}
For the coalgebra $(C = \Q[t],\Delta)$ defined in \ref{example:1}, the natural projection $\Q[t] \to \Q \oplus \Q \cdot t$ is a system of cogenerators.
\end{example}

Observe that $\field$ with coproduct given by $\Delta(1) = 1 \otimes 1$ is a counital coassociative coalgebra over itself. This leads to the following definition 
\begin{definition}
A coalgebra $(C,\Delta)$ is \textbf{coaugmented} if there is given a morphism of coalgebras $u : \field \to \Delta$; In particular if $C$ is counital then $\epsilon \circ u = id_\field$. If $C$ is coaugmented, then $C$ is canonically isomorphic to $(\field \cdot 1) \oplus Ker(\epsilon)$. 
\end{definition}
The kernel is often denoted by $\overline{C}$, so
\begin{equation}
C \cong (\field \cdot 1) \oplus \overline{C}.
\end{equation}
\begin{definition}
Let $(C,\Delta)$ be a coaugmented coalgebra. The \textbf{reduced coproduct} is the map 
\begin{equation}
\overline{\Delta} : \overline{C} \to \overline{C} \otimes \overline{C}
\end{equation}
given by
\begin{equation}
\overline{\Delta}(x) := \Delta(x) - x \otimes 1 - 1 \otimes x.
\end{equation}
The iterated reduced coproduct is denoted by 
\begin{equation}
\overline{\Delta}^n : \overline{C} \to \overline{C}^{\otimes (n+1)}.
\end{equation}
The \textbf{coradical filtration} on $C$ is defined as follows: $F_0 C = \field \cdot 1$, and for $r \geq 1$ 
\begin{equation}
F_r C := (\field \cdot 1) \oplus \left\{\: x \in C \: \big| \: \overline{\Delta}(x)^n=0 \text{ for any }n \geq r \: \right\}.
\end{equation}
\end{definition}
\begin{definition}
A coalgebra is \textbf{conilpotent} if it is coaugmented and the filtration is exhaustive, that is 
\begin{equation}
C =\bigcup_r F_r C.
\end{equation}
\end{definition}
\begin{example} The vector space
\begin{equation}
\overline{\Q[t]} = \left\{p(t) \in \Q[t] \: \big| \: p(0)=0 \right\} = \bigoplus_{n>0} \Q t^n
\end{equation}
with the coproduct
\begin{equation}
\Delta : \overline{\Q[t]} \to \overline{\Q[t]} \otimes \overline{\Q[t]} \: \: , \: \: \Delta(t^n) = \sum_{i=1}^{n-1}t^i \otimes t^{n-i}. 
\end{equation}
is a conilpotent coalgebra. The projection $\Q[t] \to \overline{\Q[t]}$ which sends $p(t) \mapsto p(t)-p(0)$, is a morphism of coalgebras. 
\end{example}

\begin{definition}
Let $(C, \Delta)$ be a coalgebra. A degree $j$ linear map  
\begin{equation}
d \in Hom^j(C,C)
\end{equation}
is called a \textbf{coderivation} if it satisfies the (graded) coLeibniz rule
\begin{equation}
\Delta \circ d = (d \otimes id_C + id_C \otimes d) \circ \Delta : C \to C[j] \otimes C[j].
\end{equation}
\end{definition}

We denote the graded $\field$-vector space of coderivation of a coalgebra $(C,\Delta)$ as $Coder^\bullet(C)$.  

\begin{definition}
A degree $j = -1$ coderivation $d : C^\bullet \to C^{\bullet-1} $ is called a \textbf{codifferential} if 
\begin{equation}
d^2 := d \circ d = 0. 
\end{equation}
A differential graded coalgebra $(C^\bullet,\Delta,d)$ is a graded coalgebra with coderivation.
\end{definition}

Let $A^\bullet$ be a graded vector space. 
\begin{definition}
The \textbf{cofree coalgebra} over $A$ is a conilpotent coalgebra $\FF^c(A)$ equipped with a linear map $p_1 : \FF^c(A) \to A$, such
that $1 \mapsto 0$ and which satisfies the following universal condition: any $\field$-linear map 
\begin{equation} 
\phi : C \to A \: \: , \: \: \phi(1) = 0
\end{equation}
where $C$ is a conilpotent coalgebra, extends uniquely into a coaugmented coalgebra morphism 
\begin{equation}
\hat{\phi}: C \to \FF^c(A )
\end{equation}
called the \textbf{hat extension} (see diagram~\ref{fig:commdiag1}). Observe that the cofree coalgebra over $V$ is well-defined up to a unique isomorphism. 
\begin{figure}
\centering
\begin{tikzpicture} 
\matrix (m) [matrix of math nodes, row sep=3em,
column sep=3em]
{ C &  \\
\FF^c(A) & A \\ };
\path[->,font=\scriptsize,dashed]
(m-1-1) edge node[left] {$ \exists! \: \hat{\phi} $} (m-2-1);
\path[->,font=\scriptsize]
(m-1-1) edge node[above] {$ \phi $} (m-2-2)
(m-2-1) edge node[below] {$ p_1 $} (m-2-2);
\end{tikzpicture}
\caption{Hat extension}
\label{fig:commdiag1}
\end{figure}
\end{definition}
Categorically, $\FF_C$ is a functor from the category of vector spaces to the category of conilpotent coalgebras. It is right adjoint to the forgetful functor which assigns to $(C,\Delta)$ the underlying vector space of $C$:
\begin{equation}
Hom_{\field-vector-space}(C,A) \cong Hom_{conil-coalg}((C,\Delta), \FF^c(A )).
\end{equation}

There is an explicit construction of the cofree coalgebras: 

\begin{definition}
The \textbf{tensor co-algebra} of $A^\bullet$ is the graded coalgebra defined as the direct sum
\begin{equation}
T^c(A) = \field \oplus A \oplus A^{\otimes 2} \oplus \ldots
\end{equation}
equipped with the \textbf{deconcatenation} coproduct
\begin{equation}
(x_s \otimes \ldots \otimes x_1) \mapsto \sum_{i=0}^{s} (x_s \otimes \ldots \otimes x_{i+1}) \otimes (x_i \otimes \ldots \otimes x_1),
\end{equation}
where $(x_s \otimes \ldots \otimes x_{i+1}) \otimes (x_i \otimes \ldots \otimes x_1) \in T^c(A)$ and the former with $i = s$ and the latter with $i = 0$ are understood as $1 \in T^c(A)$. 
We denote $p_i : T^c(A) \to A^{\otimes i}$ he projection onto the i-th graded piece. In particular, $p_0 : T^c(A) \to \field$ is the counit. 
\end{definition}

\begin{lemma} \label{lem:propoftensorcoalgebra}
$T^c(A)$ is a coassociative, counital, conilpotent coalgebra. The coradical filtration is 
\begin{equation}
F_r(T^c(A)) := \bigoplus_{i \leq r} A^{\otimes i}, 
\end{equation} 
and the projection 
\begin{equation}
p_1 : T^c(A) \to A
\end{equation}
is a system of cogenerators. \noproof
\end{lemma}

\subsubsection{$A_\infty$-algebras} \label{subsec:ainftyalgebradefinition}
Let $A^\bullet$ be a graded vector space. Write
\begin{equation}
CC^{s+t}(A,A)^t := Hom(T^c(A[1]),A) = Hom^t(A^{\otimes s},A)
\end{equation}
for the space of Hochschild elements of length $s$ and degree $t$.
\begin{definition}
 The \textbf{Hochschild space} of $A^\bullet$ is defined to be 
\begin{equation} 
CC^t(A,A) := \prod_{s \geq 0} CC^{s+t}(A,A)^t
\end{equation}
(more precisely, the right hand side is the completion of the direct sum in the category of graded vector spaces, with respect to the filtration by length $s$.) It admits a \textbf{Gerstenhaber product}:
\begin{equation} 
(\phi \circ \psi)^s(x_s,\ldots,x_1) = \sum_{i,j} (-1)^{|\psi|' \circ \maltese_i} \phi^{s-j+1}(x_d,\ldots,\psi^j(x_{i+j},\ldots,x_{i+1}),\ldots,x_1)
\end{equation}
\end{definition}
As a corollary of Lemma \ref{lem:propoftensorcoalgebra}, this has a simple interpretation in terms of coderivations, namely, 
\begin{equation} \label{eq:isowithcoderivation}
CC^\bullet(A,A) \iso Coder^\bullet(T^c(A[1]))[-1],
\end{equation}
Explicitly, given any coderivation $\hat{\phi} : T^c(A[1]) \to T^c(A[1])$, we can obtain a map $\phi: T^c(A[1]) \to A$ by pre-composing with $p_1$ and changing the grading. Note that $\phi$ completely determines $\hat{\phi}$ by Lemma \ref{lem:cogenerator}. Conversely, given a (degree shifted) homogeneous element $\phi : T^c(A[1]) \to A[1]$, the so called hat extension
\begin{equation}
\hat{\phi} :  T^c(A[1]) \to T^c(A[1])
\end{equation}
is specified as follows:
\begin{equation}    \label{eq:extendphi}
    \hat{\phi}(x_s\otimes \cdots \otimes x_1) := \sum_{i,j} (-1)^{|\phi|' \cdot \maltese_i} x_s
    \otimes \cdots \otimes x_{i+j+1} \otimes \phi^{j}(x_{i+j}, \ldots,
    x_{i+1})\otimes x_{i}\otimes \cdots \otimes x_1.  
\end{equation}
One can check that it indeed satisfies the graded co-Leibniz rule with respect to deconcatenation. 
\begin{definition}
Obviously, the space of coderivations naturally has a Lie bracket. The corresponding structure on the left hand side of \eqref{eq:isowith} is called the \textbf{Gerstenhaber bracket}, and
\begin{equation} \label{eq:isowith}
[\phi , \psi] = (\phi \circ \psi) - (-1)^{|\phi|'|\psi|'}(\psi \circ \phi).
\end{equation}
\end{definition}

There is a complete decreasing filtration on $CC^\bullet(A, A)$, called \textbf{the length filtration}, given by the subspaces $F^p CC^\bullet(A, A)$
of maps which vanish on $A^{ \otimes s}$, $s < p$. This is compatible with the Lie bracket up to a shift by 1:
\begin{equation}
[F^p CC^\bullet(A, A),F^q CC^\bullet(A, A)] \subset F^{p+q-1} CC^\bullet(A, A).
\end{equation}

\begin{definition}
An \textbf{$A_\infty$-structure} is an element 
\begin{equation}
\mu = (\mu^s) \in F^1 CC^2(A,A)
\end{equation}
satisfying $\mu \circ \mu = 0$. 
\end{definition}
Explicitly, that means that for every integer $s \geq 1$, there is a multilinear map
\begin{equation}
\mu^s : A^{\otimes s} \to A[2-s]
\end{equation}
and that these maps are required to satisfy the $\ainf$-equations: 
    \begin{align}
				&\mu^1(\mu^1(x_1)) = 0  \tag{Differential} \\
				&\mu^2(x_2,\mu^1(x_1)) + (-1)^{|x_1|'}\mu^2(\mu^1(x_2),x_1) + \mu^1(\mu^2(x_2,x_1))= 0 \tag{Leibniz rule}\\
				&\vdots \tag*{}\\
        &\sum_{n,m} (-1)^{\maltese_n} \mu^{s-m+1}(x_s, \ldots, x_{n+m+1}, \mu^{m}(x_{n+m}, \ldots, x_{n+1}), x_n, \ldots, x_1) = 0.
    \end{align}
		where the sum is taken over all pairs $n,m$ such that $1 \leq m \leq s$, $0 \leq n \leq s-m$. The pair $\AA = (A^\bullet,\mu^s)$ is called an $A_\infty$-algbera.

\begin{example}
If $(A^\bullet,d,\cdot)$ is a dg-algebra, we can define an $A_\infty$-structure with the same underlying graded vector space, by setting
\begin{eqnarray} 
\mu^1(x_1) & := & d(x_1).\\
\mu^2(x_2,x_1) & := & (-1)^{|x_1|} x_2 \cdot x_1.\\
\mu^{\geq 3} & := & 0.
\end{eqnarray}
\end{example}

\begin{definition}
Given an $A_\infty$-algebra $\AA$, the $\ainf$ equations can be re-expressed
\begin{equation}
\mu \circ \hat{\mu} = 0,
\end{equation}
which is equivalent to the requirement that $\hat{\mu}^2=0$. We call the resulting dg-coalgebra
\begin{equation}
B(\AA) := (T^c(A[1]) \: , \: \hat{\mu}) 
\end{equation}
the \textbf{bar complex} of $\AA$.
\end{definition}

\begin{definition} 
Let $\AA$ be an $\ainf$-algebra. The \textbf{cohomology algebra} $H^\bullet(\AA)$ is the algebra whose underlying graded vector space is
\begin{equation}
H^\bullet(\AA) = H^\bullet(\AA,\mu^1_\AA) 
\end{equation}
and has a product
\begin{equation}
[x_2] \cdot [x_1] = (-1)^{|x_1|}(\mu^2_\AA(x_2,x_1))
\end{equation}
It is associative by the $\ainf$-equations. 
\end{definition} 
We assume that all of our $A_\infty$-algebras are \textbf{cohomologically unital}, meaning that $H^\bullet(\AA)$ has a unit. 
\begin{definition} 
    An {\bf morphism of $\ainf$-algebras} 
    \begin{equation}
        \FF: (\AA,\mu_{\AA}^{\bullet}) \rightarrow  (\mc{B},\mu_{B}^{\bullet})
    \end{equation}
    is the data of, for each $s\geq 1$, maps of graded vector spaces
    \begin{equation}
        \FF^s : \AA^{\otimes s} \rightarrow \BB[1-s]
    \end{equation}
    satisfying the following polynomial equations:
    \begin{equation}\label{eq:functoreqn}
        \begin{split}
        \sum_{r \geq 1} \sum_{s_1 + \cdots + s_r = s} \mu_{\BB}^r(\FF^{s_r}(x_s, \ldots, x_{s-s_r+1}),\ldots,\FF^{s_1}(x_{s_1},\ldots, x_{1})) =\\
        \sum_{k,\ell} (-1)^{\maltese_k} \FF^{s-\ell+1}(x_s,\ldots, x_{k+\ell+1}, \mu_{\AA}^\ell(x_{k+\ell}, \ldots, x_{k+1}), x_k, \ldots, x_1).
    \end{split}
    \end{equation}
\end{definition} 
Again, the equation above has a simple explanation in terms of the bar complex: $\FF$ is an $\ainf$-morphism if and only if the hat extension $\hat{\FF} : B(\AA) \to B(\BB)$ is a chain map. \\

\begin{example}
Let $F: A^\bullet \to B^\bullet$ be a map of dg-algebras. As we've seen, we can think of $A^\bullet$ and $B^\bullet$ as $A_\infty$-algebras $\AA$ and $\BB$ with vanishing $\mu{\geq 3}$. Then $F$ can be regarded as an $A_\infty$-morphism $\FF$ by setting $\FF^1 := F$ and $\FF^{\geq 2} := 0$.
\end{example}

The equations \eqref{eq:functoreqn} imply that the first-order term of any morphism or functor descends to a cohomology level algebra morphism $[\FF^1]$. We always assume $A_\infty$-morphisms $\FF : \AA \to \BB$ to be cohomologically unital as well, which means the induced map on cohomology $H(\FF) : H(\AA) \to H(\BB)$ preserves the unit. 
\begin{definition}
We say that a morphism $\FF : \AA \to \BB$ of $\ainf$-algebras is a {\bf quasi-isomorphism} if $[\FF^1]$ is an isomorphism. 
\end{definition}

\subsubsection{Hochschild cohomology and deformations}

\begin{definition}
Given an $A_\infty$-algebra $\AA = (A^\bullet,\mu^s)$, we define the \textbf{Hochschild differential} 
\begin{equation}
\partial_A: CC^\bullet(\AA) \to CC^\bullet(\AA)[1] 
\end{equation}
by $\partial_A(\phi) := [\mu,\phi]$. Because $[\mu,\mu] = 0$, $\partial_A$ is a differential, i.e., $(\partial_A)^2 = 0$. 
Its cohomology is called the \textbf{Hochschild cohomology}, denoted $HH^\bullet(\AA)$.
\end{definition}

Let $A^\bullet$ be a fixed graded algebra. We write $\mu_A$ for the \textbf{trivial $A_\infty$-structure} on $A$: that is the $A_\infty$-algebra whose underlying graded vector space is $A$, and where
\begin{equation} \label{eq:definitionfirstmu}
\mu^1_A = 0 \: , \: \mu^2_A(a_2,a_1) = (-1)^{|x_1|} a_2 \cdot a_1 
\end{equation}
and $\mu^d_A$ vanishes for any $d \geq 3$. Note that the Hochschild differential $\partial_A = [\mu_A,\cdotp]$ preserves \emph{both} the length $s$ and the degree $t$; As a result, the Hochschild cohomology of $(A,\mu_A)$ is a bigraded vector space, and we denote 
\begin{equation}
HH^{s+t}(A,A)^{t} = H(CC^{s+t}(A,A)^{t},\partial_A)
\end{equation}
where the elements of $CC^{s+t}(A,A)^{t}$ are multilinear maps $\tau : A^{\otimes s} \to A$ of degree $t$.

\begin{definition}
The set $\mathfrak{U}(A)$ is defined to be those $A_\infty$-algebras whose underlying graded vector space is $A$, and where the first two multiplication maps agree with those given in equation \eqref{eq:definitionfirstmu}, i.e., $\mu^1_\AA = \mu^1_A = 0$ and $\mu^2_\AA = \mu^2_A$. 
\end{definition}
In particular, this means that for any $\AA \in \mathfrak{U}(A)$ the cohomology algebra $H(\AA)$ is $A$ itself. 
\begin{definition}
Suppose that for $\AA \in \UU(A)$, we have
\begin{equation} \label{eq:assumption}
\mu^s_\AA = 0 \: , \: \text{ for all } \: 2 <s< d. 
\end{equation}
Then 
\begin{equation}
0 = \frac{1}{2}[\mu_\AA,\mu_\AA] = \frac{1}{2}\left([\mu_\AA^2 , \mu_\AA^d] + [\mu_\AA^d , \mu_\AA^2] + \ldots \right) = [\mu_A,\mu_\AA^d ] + (\text{ \footnotesize{higher length terms ...} }) 
\end{equation}
so $\mu^d_\AA$ is a Hochschild cocycle; We call the class 
\begin{equation}
\oo^d_\AA = [\mu^d_\AA] \in HH^2(A,A)^{2-d}
\end{equation}
the \textbf{d-th universal Massey product}. More generally, it can be defined as the obstruction to making $\mu^d_\AA$ trivial by a gauge transformation which is itself equal to the identity to all orders $< d - 1$ in the following sense: the length filtration of $CC^\bullet(\AA,\AA)$ yields a spectral sequence with
\begin{equation}
E_2^{s,t} = HH^{s+t}(A,A)^t
\end{equation}
that converges to $HH^\bullet(\AA,\AA)$ under suitable technical assumptions and if \eqref{eq:assumption} holds, then the first potentially nontrivial
differential in the spectral sequence is
\begin{equation}
\delta_{d-1} = [\mu^d_\AA,\cdot] : E_2^{s,t} \to E_2^{s+d-1,t+2-d}.
\end{equation}
Then, either $\AA$ is equivalent to the trivial $A_\infty$-structure on $A$ (and the spectral sequence degenerates); or, after using gauge transformations to make as many higher order terms zero as possible, one eventually encounters a nontrivial obstruction class $o^d_\AA$ for some $d > 2$.
\end{definition}
In particular, note that condition \eqref{eq:assumption} holds automatically for $d=3$, so
\begin{equation}
\oo^3_\AA = [\mu^3_\AA] \in HH^2(A,A)^{-1}
\end{equation}
can always be defined. This is the only case we will need in this paper, so when we talk about a universal Massey product without an additional moniker, we will always mean the triple universal Massey product. 
\begin{definition}
We define the \textbf{group of gauge transformations} as arbitrary sequences of multilinear maps of the form
\begin{equation}
\mathfrak{G}(A) = \left\{\GG^1 = id_A \: , \: \GG^2 = A \otimes A \to A[-1]  \: , \: \ldots  \: , \:  \GG^d = A^{\otimes d}  \to A[1-d] \:, \: \ldots \right\}. 
\end{equation}
\end{definition}
Two $\AA, \AA' \in \mathfrak{U}(A)$ are said to be \textbf{equivalent} if there exists an $A_\infty$-homomorphism
\begin{equation}
\GG : \AA \to \AA'
\end{equation}
with $\GG^1 = id$ (i.e., if $\GG$ is a gauge transformation.) Conversely, given any $\AA \in \mathfrak{U}(A)$ and $\GG \in \mathfrak{G}(A)$, we can define an equivalent $\AA' \in \mathfrak{U}(A)$ by solving the $A_\infty$-homomorphism equations recursively. It is easy to check that this determines $\AA'$ completely and there are no constraints on $\GG$. In other words, there is an action of $\mathfrak{G}(A)$ on $\mathfrak{U}(A)$, and the quotient is exactly the set of equivalence classes. 

\begin{lemma} \label{lem:gaugeinvariance}
The assignment 
\begin{equation}
\begin{split}
\mathfrak{U}(A) &\to HH^2(A,A)^{-1} \\
\AA &\mapsto \oo^3_\AA
\end{split}
\end{equation}
descends to $\mathfrak{U}(A)/\mathfrak{G}(A)$. 
\end{lemma}
\begin{proof}
Let $\AA'$ be equivalent to $\AA$ by a gauge transformation $\GG$. Then 
\begin{equation} \label{eq:pullback}
\begin{split}
 \mu^3_{\AA'}(a_3,a_2,a_1)
 & = \mu^3_\AA(a_3,a_2,a_1) \\ &
 + \GG^2(a_3,\mu^2_\AA(a_2,a_1)) + (-1)^{|a_1|'} \GG^2(\mu^2_\AA(a_3,a_2),a_1) \\
 & - \mu^2_\AA(a_3,\GG^2(a_2,a_1))
 - \mu^2_\AA(\GG^2(a_3,a_2),a_1).
\end{split}
\end{equation}
But the last four terms are a Hochschild cocycle because 
\begin{equation}
\begin{split} \label{eq:coboundary}
\partial_A(\GG^2) &= [\mu_A,\GG^2] = \mu_A \circ \GG^2 - (-1)^{|\mu_A|' |\GG|'}\GG^2 \circ \mu_A = \\
&=  \mu^2_\AA(a_3,\GG^2(a_2,a_1))  +(-1)^{-|a_1|'}  \mu^2_\AA(\GG^2(a_3,a_2),a_1) \\
&- \GG^2(a_3,\mu^2_\AA(a_2,a_1)) - \GG^2(\mu^2_\AA(a_3,a_2),a_1). \\
\end{split}
\end{equation}
\end{proof}

\subsubsection{Universal Massey products}
Now we move on to consider \textbf{$A_\infty$-refinements} of $A$. This are defined as pairs $(\BB,F)$ where $\BB$ is an $A_\infty$-algebra and $F : A \to H^\bullet(\BB)$ an isomorphism of graded algebras. We say that two such refinements $(\BB_2,F_2)$ and $(\BB_1,F_1)$ are \textbf{equivalent} if there exists an $A_\infty$-quasi isomorphism $\GG : \BB_2 \to \BB_1$ such that the induced map on cohomology $G : H^\bullet(\BB_2) \to H^\bullet(\BB_1)$ fits into a commutative triangle
\begin{equation}
\begin{tikzpicture}
\matrix(m)[matrix of math nodes,
row sep=2.6em, column sep=2.8em,
text height=1.5ex, text depth=0.25ex]
{H^\bullet(\BB_2)&&H^\bullet(\BB_1) \\
&A&\\};
\path[->,font=\scriptsize]
(m-1-1) edge node[above] {$G$} (m-1-3)
(m-2-2) edge node[left] {$F_1$} (m-1-1)
(m-2-2) edge node[right] {$F_2$} (m-1-3);
\end{tikzpicture}
\end{equation}
\begin{homologicalperturbationlemma} \label{lem:homological perturbation}
Let $(\BB,F)$ be an $A_\infty$-refinement of $A$. Then there exists a $\AA \in \mathfrak{U}(A)$ and an $A_\infty$-quasi-isomorphism $\FF : \AA \to \BB$ such that $H(\FF) = F$.
\end{homologicalperturbationlemma}
This establishes a bijection \vspace{0.5em}
\begin{equation} \label{eq:bijection}
 \bigslant{\mathfrak{U}(A)}{\mathfrak{G}(A)} \iso \frac{
 \{\text{pairs $(\BB,F: A \rightarrow H(\BB))$}\}
 }
 {
 \{\text{$A_\infty$-quasi-isomorphisms compatible with $F$}\}.
 } \vspace{0.5em}
\end{equation}
and we define the universal Massey product of a pair $(\BB,F)$ by choosing a lift $(\AA,\FF)$ and setting
\begin{equation}
\oo^3_{(\BB,F)} := \oo^3_\AA  =[\mu^3_\AA]. 
\end{equation}
Note that if $( \BB_1,\FF_1), ( \BB_2,\FF_2)$ differ by a quasi-isomorphism that is compatible with $F$, then any choice of lifts $(\AA_1,\FF_1)$ and $(\AA_2,\FF_2)$ differ by an element of $\mathfrak{G}(A)$. So this is well defined and descends to the quotient in the right hand side \eqref{eq:bijection}. More generally, 
\begin{lemma}
Let $( \BB_1,\FF_1), ( \BB_2,\FF_2)$ be $A_\infty$-refinements of $A$. Given a quasi-isomorphism 
\begin{equation}
\GG : \BB_1 \to \BB_2, 
\end{equation}
there is a (strictly) commutative diagram 
\begin{equation}
\begin{tikzpicture}
\matrix(m)[matrix of math nodes,
row sep=2.6em, column sep=2.8em,
text height=1.5ex, text depth=0.25ex]
{H(\BB_2) & H(\BB_1) \\
A& A\\};
\path[->,font=\scriptsize]
(m-1-1) edge node[above] {$G $} (m-1-2)
(m-2-1) edge node[left] {$F_1$} (m-1-1)
(m-2-2) edge node[right] {$F_2$} (m-1-2)
(m-2-1) edge node[dashed,above] {$G'$} (m-2-2);
\end{tikzpicture}
\end{equation}
where $G = H(\GG)$ is the induced maps on cohomology. Thus, we can define an automorphism of Hochschild cohomology 
\begin{equation}
HH^{2}(A,A)^{2-d} \longrightarrow HH^{2}(A,A)^{2-d}
\end{equation}
by pre- and post-composition 
\begin{equation} \label{eq:preandpost}
\tau \mapsto G^{-1} \circ \tau \circ G^{\otimes d}
\end{equation}
where $\tau : A^{\otimes d} \to A$ is any multilinear map of degree $2-d$. Under \eqref{eq:preandpost}, the class $\oo^3_{(\BB_1,F_1)}$ is sent to $\oo^3_{(\BB_2,F_2)}$.
\end{lemma}
\begin{proof}
The construction in Lemma \ref{lem:homological perturbation} is a special case of the homotopy transfer formula. That is, if we choose \emph{homotopy data} for the chain complexes
\begin{eqnarray}
&\xymatrix{     *{ \quad \ \  \quad (\BB, \mu^1_\BB)\ } \ar@(dl,ul)[]^{h}\ \ar@<0.5ex>[r]^{\pi} & *{\
(A,0)\quad \ \  \ \quad }  \ar@<0.5ex>[l]^{\lambda}}
\end{eqnarray}
where the the map induced by $\lambda$ on cohomology being $F$, and
\begin{equation} \label{eq:hpt-data}
\begin{split}
\pi \mu^1_\BB = 0 \: , \: \mu^1_\BB \lambda &= 0 \: ,\: \pi \lambda = \Id_A, \\
\lambda \pi - \Id_\BB &=  \mu^1_\BB h + h \mu^1_\BB \ ,
\end{split}
\end{equation}
then $\mu_\AA^d$ is a sum over planar trees with $d$ leaves and one root. The contribution from each tree involves only the $\mu^k_\BB$ for $2 \leq k \leq d$, and the auxiliary data \eqref{eq:hpt-data}. For instance, writing $b_k = \lambda(a_k)$ one has
\begin{equation} \label{eq:hptmu3}
\mu^3_\AA(a_3,a_2,a_1) = \pi \mu_\BB^3(b_3,b_2,b_1) + \pi \mu_\BB^2(h \mu^2_\BB(b_3,b_2),b_1) + \pi \mu_\BB^2(b_3,h \mu^2_\BB(b_2,b_1))
\end{equation}
In our case, we can simply choose homotopy data $(\lambda,\pi,h)$ for $(\BB_2,F_2)$ and note that $(\lambda \circ G,G^{-1} \circ \pi,h)$ is actually a choice of homotopy data for $(\BB_2,F_2 \circ G^{-1})$. But this differs from $(\BB_1,\FF_1)$ by a gauge transform as in \eqref{eq:bijection}, so the universal Massey products are the same. 
\end{proof}
\begin{corollary}
Suppose that the $A_\infty$-algebra $\BB$ is formal, i.e., quasi-isomorphic to its cohomology algebra with trivial differential. Then $\oo^3_{(\BB,F)}  =0 $ for any $F : A \to H^\bullet(\BB)$.
\end{corollary}
Let us assume for now that we have an $A_\infty$-algebra $\BB$ and take for simplicity $A = H^\bullet(\BB)$, $F = \Id: A \to H^\bullet(\BB)$. We want to find a computationally efficient way to detect the non-vanishing of $\oo^3_{(\BB,F)}$, and through it the non-formality of $\BB$. Fortunantly, there is a reasonable way to ''evaluate" $\oo^3_{(\BB,F)}$ on cohomology classes (which is also the reason for the name ''universal Massey product".) 

\begin{definition} \label{def:tripleMasseyproduct}
Let $a_3,a_2,a_1 \in A$ be cohomology classes of pure degree which satisfy the \textbf{vanishing condition}:
\begin{equation}
\mu^2_A(a_3 ,a_2) = 0 \:, \: \mu^2_A(a_2 ,a_1) = 0. 
\end{equation}
Then we can choose representating cocycles $b_i \in \BB$, $i=1,2,3$ with $[b_i] = a_i$ and homotopies $h_k \in \BB$, $k=1,2$ such that
\begin{equation}
\mu_\BB^1(h_1) = \mu_\BB^2(b_2,b_1) \:, \: \mu_\BB^1(h_2) = \mu_\BB^2(b_3,b_2)
\end{equation}
Then 
\begin{equation}
c := \mu_\BB^3(b_3,b_2,b_1) - \mu_\BB^2(h_2,b_1) - \mu^2_\BB(b_3,h_1) 
\end{equation}
is a cocycle of degree $|a_3| + |a_2| + |a_1|-1$. We define the \textbf{ambiguity ideal} to be 
\begin{equation}
\II := \mu^2_A(a_3 , A) + \mu^2_A(A , a_1).
\end{equation}
and the \textbf{triple Massey product} as the coset
\begin{equation}
\langle a_3,a_2,a_1\rangle = (-1)^{|a_2|} [c] \in A / \II. 
\end{equation}
Note that the triple Massey product depends only on the original original $a_i$ and not on the choice of representating cocycles or homotopies. 
\end{definition}

\begin{lemma} \label{lem:independence}
The triple Massey product depend only on $\oo^3_{(\BB,F)}$, and not on the entire $A_\infty$-structure $\mu_\BB$. 
\end{lemma}
\begin{proof}
If we choose homotopy transfer data and set $b_i = \lambda(a_i)$ and $h_k = -h(b_{k+1},b_k)$ , then $c = \mu_\AA^3(a_3,a_2,a_1)$ by \eqref{eq:hptmu3}; Changing $\mu_\AA^3$ by a Hochschild coboundary has no effect, because the first two terms in equation \eqref{eq:coboundary} are in the ambiguity ideal and the last two are identically zero. 
\end{proof}

\subsubsection{Parametrized QMP} \label{sec:algdefinition}
Finally, we can now specialize to our specific setup and define the quantum Johnson homomorphism and matrix Massey products. There are a few minor changes to make to the general algebraic framework and we comment on them briefly. \vspace{0.5em}

\begin{itemize}
\item
Instead of a field of characteristic zero, our $A_\infty$-structures are defined over the $\Q$-algebra $\Gamma$. 
We define $A_\infty$-structures over ring in the same way as before, replacing graded vector spaces with graded $\Gamma$-modules. In fact, we define a \textbf{projective $A_\infty$-structures} to be an $A_\infty$-structure where the underlying space is a projective graded $\Gamma$-module. In this situation, $Ext(V_1, V_2) = Hom(V_1, V_2)$, $Tor(V_1, V_2) = V_1 \otimes V_2$ and all constructions in the $A_\infty$-world are automatically ''derived" (as is the case for fields.) In particular, all the naive definitions of $A_\infty$-homomorphism and their natural transformations, and of Hochschild cohomology and the length spectral sequence work well for projective $A_\infty$-structures. Finally, note that a sufficient condition for projectivity is the existence of finite bases for the underlying graded module. Since we work exclusively with Morse cochains this is satisfied automatically \emph{and we will deal with this situation alone}. \vspace{0.5em}
\item
Let $\CC$ be a graded $\Q$-vector space, and $\CC_\Gamma = \CC \otimes_\Q \Gamma$ an extension of scalars (an example to keep in mind is $\CC = H^\bullet(Y)$ and $\CC_\Gamma = QH^\bullet(Y)$.) The ring $\Gamma$ has an additional \textbf{energy filtration} (in fact, a direct sum decomposition) so the data of an $A_\infty$-structure on $(\CC_\Gamma,\mu^d)$, which can be written as 
\begin{equation}
\begin{split}
\mu^d(x_d e^{A_d},\ldots,x_1 e^{A_1}) &= e^{A_d + \ldots + A_1}\mu^d(x_d,\ldots,x_1), \\
\mu^d(x_d,\ldots,x_1) &= \sum_{A \in \Gamma} \mu^d_A(x_d,\ldots,x_1) e^A.
\end{split}
\end{equation}
is equivalent to specifying a collection of multilinear operations 
\begin{equation}
\mu^d_A : C^{i_d} \otimes \ldots \otimes C^{i_1} \to C^{i_d + \ldots + i_1 + 2-d-2c_1(A)} \: , \: A \in \Gamma
\end{equation}
on the original graded $\Q$-vector space which satisfy the obvious analogue of the $A_\infty$-equations: 
\begin{equation}
\begin{split}
\sum_{A + A' = B} \sum_{n,m} (-1)^{\maltese_n} &\mu_{A}^{s-m+1}(x_s, \ldots, x_{n+m+1}, \\
&\mu_{A'}^{m}(x_{n+m}, \ldots, x_{n+1}), x_n, \ldots, x_1) = 0
\end{split}
\end{equation}
for any $A,A',B \in \Gamma$. \vspace{0.5em}
\item
In practice, the vanishing condition severely restricts the usefulness of Definition \ref{def:tripleMasseyproduct} in detecting nonformality. However, this problem can be easily circumvented. Let $n \geq 1$ be an integer and let $M_{n \times n}(\Gamma)$ be the algebra of n-by-n matrices with coefficients in $\Gamma$. Given an $A_\infty$-algebra $(\CC_\Gamma,\mu^d)$ we can form the tensor product 
\begin{equation}
\CC_{\Gamma,n \times n} := \CC_\Gamma \otimes_\Gamma M_{n \times n}(\Gamma) 
\end{equation}
The $A_\infty$-operations on the new complex (also denoted $\mu^d$) are defined as follows: let $X_d,\ldots,X_1$ be 
n-by-n matrices of homogeneous elements in $\CC^\bullet$. We define 
\begin{equation}
\mu^d(X_d,\ldots,X_1) 
\end{equation}
to be an n-by-n matrix $Y$, whose $(i,j)$-entry for every $1 \leq i ,j \leq n$ is given by the formula $Y_{ij} := \mu^d(z_{ij})$ where
\begin{equation}
\begin{split}
z_{ij} := \sum_{1 \leq k_{d-1},\ldots,k_1 \leq n} &(X_d)_{ i k_{d-1} } \otimes (X_{d-1})_{k_{d-1} k_{d-2}} \otimes \\
&\ldots \otimes (X_2)_{k_2 k_1} \otimes (X_1)_{k_1 j}.
\end{split}
\end{equation}
It is easy to see that this definition satisfies the $A_\infty$-equations (in fact, it is a special, non-messy case of a more general definition for tensor product of $A_\infty$-algebras.) A Massey product in $\CC_{\Gamma,n \times n}$ will be called a \textbf{matrix Massey product}. Note that $\CC_{\Gamma,n \times n}$ is formal if and only if $\CC_\Gamma$ is. \vspace{0.5em}
\item
Another source of inconvenience is that while non-formality is ''finitely determined" which means it can already be established by looking at the coefficient of $e^A$ for some fixed $A$, the triple Massey products as we have defined it involve $\mu^d$ which is the sum of all of $\mu^d_A$ for every $A \in \Gamma$ (and this translated to computing more parametrized gromov-witten invariants down the line.) To make the most out of the energy filtration, we fix a homology class $A \in \Gamma$ and take the tensor product with the ideal $\II_A$ defined in Example \ref{example:1.2}. The resulting $A_\infty$-algebra is denoted 
\begin{equation}
\CC_{\Gamma,n \times n,\II_A} := \CC_{\Gamma,n \times n} \otimes_\Gamma (\Gamma/\II_A). 
\end{equation}
It is clear that if $\CC_{\Gamma,n \times n}$ is formal, then $\CC_{\Gamma,n \times n,\II_A}$ must be formal as well.
\end{itemize}
Fix a homology class $A \in \Gamma$. \vspace{0.5em}
\begin{enumerate} 
\item
Let $(M,\omega)$ be a monotone symplectic manifold. We denote by $(\CC,\mu^d)$ the Morse $A_\infty$-algebra over $\Q$. This is the rational homotopy type of $M$. Let $(\CC_\Gamma,\mu^d)$ be the pearl complex. \vspace{0.5em}
\item
The cohomology algebra of $(\CC,\mu^d)$ is isomorphic to the singular cohomology of $M$, and we denote it $(H^\bullet(M),\cdot)$. The cohomology algebra of $(\CC_\Gamma,\mu^d)$ is the quantum cohomology $(QH^\bullet(M),\star)$. \vspace{0.5em}
\item
Let $\phi : M \to M$ be a symplectomorphism. Denote by $(M_\phi,\Omega)$ the locally Hamiltonian fibration over $S^1$ which is the mapping torus of $\phi$. We denote by $(\tilde{\CC},\tilde{\mu}^d)$ the Morse $A_\infty$-algebra over $\Q$. Let $(\widetilde{\CC_\Gamma},\tilde{\mu}^d)$ be the ''quantization" from Definition \ref{def:pearlcomplexainfty2}.\vspace{0.5em}
\item
The cohomology algebra of $(\tilde{\CC},\tilde{\mu}^d)$ is denoted $(H^\bullet(M_\phi),\cdot)$. The cohomology algebra of $(\tilde{\CC},\tilde{\mu}^d)$ is the \textbf{parametrized quantum cohomology} and we denote it as $(QH^\bullet(M_\phi),\star)$. \vspace{0.5em}
\end{enumerate}

\begin{definition} \label{def:algebraicdefinitionofmasseyproducts}
Assume that we have fixed an isomorphism of algebras
\begin{equation}
F : \bigslant{H^\bullet(\CC_{\Gamma,\II_A})[\mathfrak{t}]}{(\mathfrak{t}^2)} \to H^\bullet(\widetilde{\CC_{\Gamma,\II_A}})
\end{equation}
The quantum Johnson homomorphism, denoted $\oo^{\phi}_{A}$, is defined as the universal Massey product 
\begin{equation}
\oo^3_{(\widetilde{\CC_{\Gamma,\II_A}},F)} \in HH^2( \: H^\bullet(\CC_{\Gamma,\II_A})[\mathfrak{t}]/(\mathfrak{t}^2)\: ,\: H^\bullet(\CC_{\Gamma,\II_A})[\mathfrak{t}]/(\mathfrak{t}^2) \:)^{-1}.
\end{equation} 
\end{definition}
Let $n \geq 1$ be an integer and $X_3,X_2,X_1$ be n-by-n matrices. A triple quantum matrix Massey product is simply the coefficient of $e^A$ in the triple Massey product in $(\widetilde{\CC_{\Gamma,n \times n,\II_A}},\tilde{\mu}^d)$. More concretely, we can use $\Gamma$-linearity to reduce to the case where all elements in $X_i$, $i=1,2,3$ belong to $\CC^\bullet$ and are of pure degree: $p$, $q$ and $r$ respectively. Then,
\begin{definition}
A \textbf{triple quantum matrix Massey products} is a partially defined multilinear operation
\begin{equation}
\bigg\langle [X_3],[X_2],[X_1] \bigg\rangle_A : M_{n \times n}(\tilde{\CC}^p) \otimes M_{n \times n}(\tilde{\CC}^q) \otimes M_{n \times n}(\tilde{\CC}^r) \to M_{n \times n}(\tilde{\CC}^{p+q+r+2c_1(A)-1})/\II.  \vspace{0.5em}
\end{equation}

\begin{description}
\item[Vanishing condition] Assume that 
\begin{equation}
[X_3] \star_B [X_2] = 0 \: , \: [X_2] \star_B [X_1] = 0 
\end{equation}
for every $B \in \Gamma \setminus \II_A$ (that is, $B = A$ or $c_1(B)<c_1(A)$.) \\

\item[Ambiguity ideal] 
We define $\II$ to be the ideal generated by
\begin{equation}
[X_3] \star_{B} M_{n \times n}(QH^\bullet(M_\phi)) + M_{n \times n}(QH^\bullet(M_\phi)) \star_{B'} [X_1]
\end{equation}
where $B,B' \in \Gamma \setminus \II_A$. \\

\item[Bounding cochains] Assuming that the ambiguity condition holds. We choose matrices of cochain representative $X_i$, $i=1,2,3$ and  bounding cochains, i.e., n-by-n matrices $H_k^B$ of homogeneous elements in $\CC^\bullet$ such that 
\begin{equation}
\tilde{\mu}^1(H_k^B) = \tilde{\mu}_B^2(X_{k+1},X_k)
\end{equation}
for every $B \in B \in \Gamma \setminus \II_A$ and $k=1,2$. \\

\item[Definition] We set
\begin{equation}
\begin{split}
\bigg\langle [X_3],[X_2],[X_1] \bigg\rangle_A = &[\tilde{\mu}_A^3(X_3,X_2,X_1)] \\
- &\sum_{B + B' = A} [H_2^{B'}] \star_B [X_1]\\
- &\sum_{B + B' = A} [X_3] \star_{B'} [H_1^{B}]
\end{split}
\end{equation}
where both $B,B' \in \Gamma \setminus \II_A$. 
\end{description}
\end{definition}

\begin{remark}
Note that when the underlying pearl complex is minimal, not only are there no bounding cochains to take into account, but we also know that $\mu^2_A$ coincides with the product $\star_A$ of parametrized quantum cohomology, which can be defined over a smaller ring. This is a great computational advantage: it means we apriori know that any $\mu^2_A = 0$ for any homology class $A$ which is not in the cone of curves. See Section \ref{sec:coneofcurvescomputation}. 
\end{remark}

\section{Background on singularity theory} \label{sec:singularitytheory}
In this section, we gather some basic facts related to singularities and deformations that we will need to prove Proposition \ref{prop:blowup2}. In particular, we recall some elements of Picard-Lefschetz theory and the properties of Milnor fibers which would allow us to conclude that the monodromy is a product of generalized Dehn twists. None of the material in this section is new: all classical facts are taken from the beautiful survey \cite{MR1660090}, and their symplectic counterparts appeared before in \cite{MR1862802},\cite{MR1978046},\cite{MR2441780} and \cite{MR3432159}. The connection with algebraic geometry and representation theory (coupled with an extensive historical discussion -- including many references to the original papers) is summarized in \cite{MR2290112}, \cite{MR3020098}. 
 
\begin{remark}
This section is ''the local theory'', in order to reduce the proof of Proposition \ref{prop:blowup2} to the situation described here, we will also need some theorems from algebraic geometry regarding the moduli space of cubic 3-folds. Those would be discussed in the next section (along with the proof itself.) 
\end{remark}

The singularities we are interested in are of a specific type:

\begin{definition} 
On the set $\GG'$ of pairs $(X,x)$ consisting of an analytic space $X$ and a point $x \in X$, we say that $(X,x)$ is equivalent to $(Y,y)$ and write $(X,x) \sim (Y,y)$ if there exist a neighborhood $U \subset X$ of $x$, a neighborhood $V \subset Y$ of $y$ and an biholomorphism $f : U \to V$ such that $f(x)=y$. This is an equivalence relation; let the quotient set be denoted $\GG := \GG' / \sim$. An element of $\GG$ is called a \textbf{germ} of the analytic space and we denote it by a representative $(X,x)$.
\end{definition} 

\begin{definition} Let $f : \C^{m+1} \to \C$ be a holomorphic function such that $f(0) = 0$, $df|_0 = 0$ and $df \neq 0$ on the punctured ball $B^\ast_\delta (0)$, some sufficiently small $\delta>0$. An \textbf{isolated hypersurface singularity} is the germ
\begin{equation}
X = (V(f),0).
\end{equation}
\end{definition}

From now on, whenever we say ''singularity" we mean ''an isolated hypersurface singularity". We also assume that $ m \geq 2$. 

Finally, one point we wish to remark upon is that we will not keep track of the difference between the algebraic and analytic realms. This is reasonable because,
\begin{theorem}
Let $(X,x), (Y,y)$ be germs of analytic spaces, then the following are equivalent:
\begin{enumerate}
\item
There is an equivalence 
\begin{equation}
(X,x) \sim (Y,y).
\end{equation}
\item
There exists an isomorphism of $\C$-algebras 
\begin{equation}
\OO_{X,x} \cong \OO_{Y,y}.
\end{equation}
\item
There exists an isomorphism of $\C$-algebras 
\begin{equation}
\hat{\OO}_{X,x} \cong \hat{\OO}_{Y,y}, 
\end{equation}
where $\hat{R}$ denotes the completion of a Noetherian local ring $R$ by its maximal ideal $\mathfrak{m} \subset R$ in the $\mathfrak{m}$-adic topology.
\end{enumerate}
\end{theorem}

\begin{theorem}[Artin's Algebraization Theorem] For a germ $(X,x)$ of analytic space, if $x$ is an isolated singularity, there exists a complex algebraic variety $Y$ and a point $y \in Y$ such that:
\begin{equation}
(X,x) \sim (y,Y)
\end{equation}
\end{theorem}

\subsubsection{Miniversal deformations}
To every singularity, we can associate a number of algebraic gadgets:
\begin{definition}
The \textbf{Jacobian ideal} is the ideal generated by all the partial derivatives:
\begin{equation}
J(f) := (\frac{\partial f}{\partial z_1}, \ldots ,\frac{\partial f}{\partial z_{m+1}}) \subset \C[[z_1, \ldots , z_{m+1}]].
\end{equation}
\end{definition}
\begin{definition}
We define the \textbf{Milnor algebra} as the quotient
\begin{equation}
\AA_f := \C[[z_1, \ldots , z_{m+1}]] / J(f).
\end{equation}
The algebra $\AA_f$ does not depend on the choice of local coordinate system. The length of this algebra 
\begin{equation}
\mu(f) = \dim_\C \AA_f 
\end{equation}
is exactly the Milnor number from Definition \ref{def:milnornumber}. 
\end{definition}

\begin{definition}
The \textbf{Tjurina algebra}, denoted $T^1_f$, is defined to be
\begin{equation}
\C[[z_1, \ldots , z_{m+1}]] / (f,\frac{\partial f}{\partial z_1}, \ldots ,\frac{\partial f}{\partial z_{m+1}}).
\end{equation}
The length of these algebra is called the \textbf{Tjurina number} and we denote it 
\begin{equation}
\tau(f) = \dim_\C T^1_f.
\end{equation}
\end{definition}
Let $X = V(f)$ be a hypersurface. Then there is an exact sequence
\begin{equation}
0 \to \OO_X \stackrel{(\frac{\partial f}{\partial z_1},\ldots,\frac{\partial f}{\partial z_{m+1}})}{\xrightarrow{\hspace*{3cm}}} \OO_X^{ \otimes (m+1)} \to \Omega^1_X \to 0
\end{equation}
which yields
\begin{equation}
Ext^j(\Omega_X^1,\OO_X) = \begin{cases}
    T^1_f,& \text{if } j = 1\\
    0,              & \text{if } j \geq 2
\end{cases}
\end{equation} 
so by Grauert's theorem, there exists a miniversal deformation whose basis is just the Tjurina algebra, and the obstruction space vanishes. \\

Clearly $\mu \geq \tau$, with equality iff $f$ belongs to the Jacobian ideal. The Milnor algebra and the Tjurina algebra coincide in the case of a weighted homogeneous singularity by the Euler identity.

\subsubsection{Milnor fibers} \label{subsec:milnorfiber}

Let $f$ be a singularity.

\subsubsection{As smooth manifolds} 
Let $h: \C^{m+1} \to \R$ be the function given by
\begin{equation}
h(x_0, \ldots, x_m) = |x_0|^2 + \ldots + |x_{m+1}|^2.
\end{equation}
The restriction $h: \,  f^{-1}(0) \to \R$ is a real algebraic function. By the curve selection lemma (see Chapter 3 in \cite{MR0239612}), it has isolated critical values. Let $\delta_f$ be the smallest positive one. For any $\delta < \delta_f$ and sufficiently small $\epsilon_\delta$, we have $f^{-1} (\epsilon_\delta) \pitchfork S_{\sqrt{\delta}}(0)$, where $S_{\sqrt{\delta}} (0)$ is the sphere of radius $\sqrt{\delta}$ about the origin. 

\begin{definition} \label{def:Milnorfibre}
The \textbf{Milnor fibre} of $f$ is the smooth manifold with boundary $f^{-1}(\epsilon_\delta) \cap \overline{B}_{\sqrt{\delta}}(0)$, for any such $\delta < \delta_f$ and $\epsilon_\delta \neq 0$. 
\end{definition} 
Topologically, the Milnor fiber is homeomorphic to the wedge sum of half-dimensional spheres
\begin{equation} 
M_f \iso \bigvee_\mu S^m
\end{equation}

\begin{definition} \label{def:milnornumber}
The number $\mu$ is called the \textbf{Milnor number}. For any hypersurface singularity $f$, the number of critical points in a Morsification $\tilde{f}$ equals the Milnor number $\mu(f)$. In particular, it is independent of the choice of perturbation.
\end{definition}


Despite all the choices we made in the construction, the differentiable structure (as a smooth manifold) of the Milnor fiber is canonical. In fact, instead of standard height function $h$, we can consider any real algebraic function
\begin{equation}
\tilde{h}: \C^{m+1} \to [0 , \infty)
\end{equation}
such that $\tilde{h}^{-1} (0) = 0$. The curve selection lemma still applies: the restriction of $\tilde{h}$ to $f^{-1}(0)$ also has isolated critical values. For sufficiently small $\delta$ and $\epsilon_\delta$, $f^{-1}(\epsilon_\delta)$ is transverse to the $\delta$--level set of $\tilde{h}$, and the manifolds  
\begin{equation}
  f^{-1}(\epsilon_\delta) \cap \{ \tilde{h} (x_0 ,\ldots, x_{m+1}) \leq \delta  \}
\end{equation}
are also diffeomorphic copies of the Milnor fibre. This can be shown by linearly interpolating between $h$ and $\tilde{h}$, and noting that all the intermediate functions are real algebraic, with the same properties as used above.
  
\subsubsection{As symplectic manifolds} 
The affine space $\C^{m+1}$ comes with a standard exact symplectic form, which is the usual K\"{a}hler form on $\C^{n+1}$: $\omega_{std} = d \theta_{std} $, where 
\begin{equation}
\theta_{std}  = \frac{i}{4} \sum_{i} x_i d \bar{x}_i -  \bar{x}_i d x_i. 
\end{equation}
This restricts to an exact symplectic form on any of the Milnor fibres. By construction, the associated negative Liouville flow is the gradient flow of $h(x)$ with respect to the standard K\"{a}hler metric. Suppose we use the cut-off function
\begin{equation}
  h_A (\mathbf{x} ) = || A \mathbf{x} ||^2
\end{equation}
for some $A \in GL_{m+1}(\C)$. On $\C^{m+1}$, the negative gradient flow of $h$ points strictly inwards at any point of the real hypersurface $|| A \mathbf{x} ||^2 = \delta$. (This is of course true for any $\delta$, though we only need it for $\delta < \delta_{f,A}$.) In particular, 

\begin{definition}
For sufficiently small $\delta$ and $\epsilon_\delta$, the Milnor fibre is the exact symplectic manifold with contact type boundary (i.e.~a Liouville domain) $(M_f,\omega_f,\theta_f,\alpha_f)$ where
\begin{equation} \label{eq:Milnor fiber}
\begin{split}
  M_f &:= f^{-1}(\epsilon_\delta) \cap \{  h_A ( \mathbf{x} ) \leq  \delta  \}, \\
	\omega_f &:= \omega_{std}  \big|_{M_f},\\
	\theta_f &:= \theta_{std}  \big|_{M_f},\\
	\alpha_f &:= \omega_f \big|_{\partial M_f},
\end{split}
\end{equation}
\end{definition}
As usual, we can attach cylindrical ends to it using the Liouville flow on a collar neighbourhood of the boundary.  
\begin{definition}
We call this the \textbf{completed Milnor fibre} of $f$, and denote it 
\begin{equation} \label{eq:complMilnor fiber}
(\widehat{M}_f,\widehat{\omega}_f,\widehat{\theta}_f,\widehat{\alpha}_f).
\end{equation} 
\end{definition}
To make the notation less cumbersome, we usually shorten \eqref{eq:Milnor fiber} and \eqref{eq:complMilnor fiber} to 
$M_f$ and $\widehat{M}_f$. 

\begin{lemma}[2.6 in \cite{MR3432159}]\label{th:Milnorfibreindepchoices}
Given a holomorphic function $f: \C^{m+1} \to \C$, the completed Milnor fibre $\widehat{M}_f$ is independent of the choice of $A$, $\delta$ and $\epsilon$ up to exact symplectomorphism.
\end{lemma}

\begin{lemma}[2.7 in \cite{MR3432159}]\label{th:Milnorfibreindepreparametrisation}
The completed Milnor fibre $\widehat{M}_f$ does not depend on the choice of holomorphic representative of $f$. More precisely, suppose  $f = g \circ \rho$, some holomorphic change of coordinates $\rho$. Then there is an exact symplectomorphism from $\widehat{M}_f$ to $\widehat{M}_g$. 
\end{lemma} 



\subsubsection{In families} 
Assume that we are given an $(m+1)$-parameter \textbf{deformation} 
\begin{equation}
(f_w)_{w=(w_0,\ldots,w_m) \in \C^{m+1}} 
\end{equation}
of a singularity $f$. That is, a family of polynomials of the form 
\begin{equation}
f_w(\mathbf{x}) = f(\mathbf{x}) + \sum w_j \tilde{f}_j(\mathbf{x}), 
\end{equation}
where $\tilde{f}_0 \equiv 1$. A Theorem of Milnor states that if we take a sufficiently small $\delta>0$, then 
\begin{equation} \label{eq:Fdelta}
f_0^{-1}(0) \cap \{  h_A ( \mathbf{x} ) = \delta  \}
\end{equation}
is transverse. We fix $\delta_f>0$ such that the above holds for every $\delta <  \delta_f$. Then for every $w \in \C^{m+1}$ with $|w|<r$ (here $r$ is a fixed, sufficiently small positive number), the intersection 
\begin{equation} \label{eq:Fdelta}
f_w^{-1}(0) \cap \{  h_A ( \mathbf{x} ) = \delta  \}
\end{equation}
remains transverse.  %

\begin{lemma} \label{lem:trivialization}
The smooth family of contact manifolds 
\begin{equation}
f_w^{-1}(0) \cap \{  h_A ( \mathbf{x} ) =  \delta  \}, 
\end{equation}
admits a trivialization, which is unique up to homotopy.
\end{lemma} 
\begin{proof}
Follows from Gray's stability, because $B_{r}(0) \subset \C^{m+1}$ is contractible.
\end{proof}

\begin{definition}
The \textbf{Milnor fibration} of the deformation $(f_w)$ is symplectic fibration $E$ over $(W, \left\{\hat{w}\right\})$ defined as follows: \vspace{0.5em}
\begin{itemize}
\item[(a)]
Let $W \subset B_{r}(0)$ be the subset of those $w$ such that $f_w^{-1}(0) \cap \{  h_A ( \mathbf{x} ) \leq  \delta  \}$ is smooth. It is open and connected, because its complement is a complex hypersurface (the discriminant). 
\vspace{0.5em}
\item[(b)]
We define 
\begin{equation}
E = \left\{(\mathbf{x},w) \in \C^{m+1} \times W \: \big| \: |\mathbf{x}| \leq r , f_w(\mathbf{x}) = 0 \right\}
\end{equation}
and $\pi : E \to W$ the projection onto the second factor. We give $\pi^{-1}(w)$ a Liouville domain structure in the same way as we did in \eqref{eq:Milnor fiber}. In fact, after we attach conical ends, $\pi^{-1}(w)$ is exact symplectomorphic to the completed Milnor fibre of $f$ as defined in Lemma \ref{th:Milnorfibreindepchoices} (see Lemma 2.18 in \cite{MR3432159}). \vspace{0.5em}
\item[(c)]
We choose some base point $\hat{w} \in W$ and set
\begin{equation}
(M, \omega,\alpha) = (E_{f_{\hat{w}}},\omega_{f_{\hat{w}}},\alpha_{f_{\hat{w}}}).
\end{equation}
\item[(d)]
By restriction, any trivialization as in Lemma \ref{lem:trivialization} defines
a contact trivialization $\tau : \partial M \times W \to \partial E$. \vspace{0.5em}
\item[(e)]
Take $\eta : M \to E_{\hat{w}}$ to be the identity map. \vspace{0.5em}
\end{itemize}
\end{definition}
We caution that different fibres of the Milnor fibration are not necessarily symplectically isomorphic, but they become isomorphic after attaching an infinite cone. \\

By Proposition 6.2 in \cite{MR1862802} one can associate to it a homomorphism
\begin{equation}
\rho_s = \rho_s^E : \pi_1(W,\hat{w}) \to \pi_0(Symp(M,\partial M,\omega))
\end{equation}

\begin{definition}
We call $\rho_s$ the \textbf{symplectic monodromy map} associated to the deformation. The composition 
\begin{equation}
\rho_g : \pi_1(W,\hat{w}) \to \pi_0(Symp(M,\partial M,\omega)) \to \pi_0(\Diff^+(M,\partial M))
\end{equation}
is called the \textbf{geometric monodromy}. Finally, the induced action by automorphism on the middle dimensional homology of $M$ is denoted $\rho_H$.
\end{definition}

We have obviously made many choices in the process of defining the monodromy map, but as remarked in \cite[p.~63]{MR1862802}, $\tau$ is unique up to homotopy. Therefore $\rho_s$ is independent of the choice of $\tau$ (up to symplectic isotopy.) Changing the value of $r$ can modify $W$, but the fundamental group remains the same for all sufficiently small $r$, and we will assume from now on that $r$ has been chosen in that range. As for the dependence on the choice of base point $\hat{w}$: given a path in $W$ from $w_0$ to $w_1$, one can identify
\begin{equation}
\pi_0(\Symp(E_{w_0}, \partial E_{w_0},\omega_{w_0})) \iso \pi_0(\Symp(E_{w_1}, \partial E_{w_1},\omega_{w_1}))
\end{equation}
and this fits into a commutative diagram with the corresponding isomorphism $\pi_1(W,w_0) \iso \pi_1(W,w_1)$ and with the symplectic monodromy maps at these base points. For a similar reason, making $\delta$ or $\epsilon_\delta$ smaller does not affect $\pi_0(\Symp(M,\partial M,\omega))$ or $\rho_s$.  Finally one can choose $(f_w)$ to be a miniversal deformation of $f$, and then the resulting symplectic monodromy map is really an invariant of the singularity. We call the \textbf{monodromy representation of the singularity}.

\subsubsection{Artin and Coxeter groups}
We fix a Dynkin diagram $\Gamma$ of type $T_k \in \left\{A_k, D_k, E_6,E_7,E_8\right\}$ and recall some facts concerning the Artin groups and finite reflection groups. \\ 

$\Gamma$ is a weighted graph with vertices indexed by a set $I$ and with an edge for each pair of vertices ${i, j}$ with $i \neq j$ labeled with an integral weight $m_{ij} \geq 2$ (and possibly $\infty$). We can associate to it two groups: \vspace{0.5em}
\begin{definition}
\begin{enumerate}
\item
The \textbf{Artin group} $Art = Art(\Gamma)$ is the group generated by $\left\{t_i\right\}_{i \in I}$ subject to the relations:
\begin{equation} \label{eq:artinrelation}
t_i t_j t_i \ldots = t_j t_i t_j \ldots
\end{equation}
where both sides have $m_{ij}$ letters. \vspace{0.5em}
\item
The associated \textbf{Coxeter group} $W = W(\Gamma)$ (of Weyl type) as the generators $\left\{s_i\right\}_{i \in I}$, satisfying the same relations \eqref{eq:artinrelation}, with the additional condition that each generator is an involution, i.e.,
\begin{equation} \label{eq:involution}
s_i^2=1. 
\end{equation}
\end{enumerate}
\end{definition}
In particular, note that there is a natural epimorphism 
\begin{equation}
Art \to W.
\end{equation}
Since we have fixed a set of generators $\left\{t_i\right\}_{i \in I}$ for $Art$ and their projections $\left\{s_i\right\}_{i \in I}$ in $W$, there is a given Weyl chamber in the real vector space $V$, where $V^\vee$ contains the associated \textbf{root system} $R = R(\Gamma)$.

\begin{definition}
The \textbf{Coxeter element} of $W$ is the product of the generators 
\begin{equation}
s_1 \cdot \ldots \cdot s_n \in W, 
\end{equation}
The order of the Coxeter element, denoted $h$, is called the \textbf{Coxeter number}.
\end{definition}

It is well known that $h$ is always even, except in the $A_{even}$ case.
\begin{definition}
The corresponding product 
\begin{equation}
\Pi = t_1 \cdot \ldots \cdot t_n  \in Art 
\end{equation}
of the generators will be called the \textbf{Artin-Coxeter element}. 
\end{definition}

We set $\mathfrak{D} \in Art$ to be the \textbf{Garside element} which is a specific lift of the longest element $w_0 \in W$ (determined by the choice of Weyl chamber).
\begin{lemma}[5.8 in \cite{MR0323910}]
Up to conjugacy,
\begin{equation}
D^2 = \Pi^h
\end{equation}
\end{lemma}
In fact, except in the $A_{2k}$ case, one has (up to conjugacy)
\begin{equation}
\mathfrak{D} = \Pi^{h/2}.
\end{equation}

Artin groups appear as fundamental groups in the following way: let $V^\vee$ be the real vector space containing the root system $R$ and let $C\subseteq V$ be the Weyl chamber associated to the generators of $W$ determined by $\Gamma$. Let $\KK'=V\otimes_{\mathbb R}\mathbb C$, $\KK=\KK'/W$, and $\Delta'=\mathscr H=\bigcup_{\alpha\in R} (H_{\alpha}\otimes_{\mathbb R}\mathbb C)$ be the associated hyperplane arrangement. Let $(\KK')^\circ=\KK' \setminus \Delta'$ be the complement of the hyperplane arrangement. Similarly, set $\KK^\circ=(\KK')^\circ/W$; i.e. the complement of the discriminant. Note that the longest element $w_0\in W$  is also distinguished by the fact that it sends $C$ to $-C$. Identifying $\KK'=V\otimes_{\mathbb R}\mathbb C$ as $V\oplus \mathbf{i} V$, we have $C\subseteq \KK'$. \\

Fix a base point $*' \in C \subset (\KK')^\circ$ with image $* \in \KK^\circ$. We are interested in understanding the fundamental groups of $\KK^\circ$. As described in \cite[p.195]{MR2365657}, there exist contractible (analytic) open subsets 
\begin{equation}
\begin{split}
\mathbb U^+ &\subseteq (\KK')^\circ \\
\mathbb U^- &\subseteq (\KK')^\circ 
\end{split}
\end{equation}
such that for each $w\in W$, we have $w(\ast')\in w\cdot C\subset \mathbb U^+\cap \mathbb U^-$. Thus for each $w\in W$, there is a path $\gamma_w^+$ (resp. $\gamma_w^-$)  contained in $\mathbb U^+$ (resp. $\mathbb U^-$) connecting $\ast'$ to $w(\ast')$, unique up to an end-point fixing homotopy. 

\begin{definition} We define maps
\begin{equation}
\begin{split} 
t^+:W &\to \pi_1(\KK^\circ,\ast)  \\
t^-:W &\to \pi_1(\KK^\circ,\ast)
\end{split}
\end{equation}
by sending $w$ to the class of the image of the path $\gamma^+_w$ (resp. $\gamma_w^-$).
\end{definition}

\begin{proposition}
The assignment 
\begin{equation}
t_i\mapsto t^+(s_i)=\left(t^-(s_i)\right)^{-1}
\end{equation}
induces an isomorphism
\begin{equation}\label{eq:nagiso}
\operatorname{Art} \to \pi_1(\KK^\circ,\ast).
\end{equation} 
Under this isomorphism $t^+$ and $t^-$ are sections (as maps of sets) of the homomorphism $\operatorname{Art}\to W$, and 
\begin{equation}\label{eq:nmlgarside}
t^+(w_0)=\left(t^-(w_0)\right)^{-1}=\mathfrak{D},
\end{equation}
is the Garside element.
\end{proposition}
\vspace{0.5em}
\begin{proposition}\label{prop: main}
\begin{enumerate}[parsep=5pt]
\item
Let $\gamma \subset \KK^\circ$ be a small simple loop around the origin based at the point $\ast \in \KK^\circ$ and lying in the complexification of a line in $V$ spanned by the origin and the point $\ast$. Consider $\gamma$ as an element of $\pi_1(\KK^\circ,*)$. The under the isomorphism \eqref{eq:nagiso}
\begin{equation}
\gamma=\Pi,
\end{equation}
up to conjugacy. \vspace{0.5em}
\item
Let $\sigma \subset (\KK')^\circ$ be a small simple loop around the origin based at the point $\ast'\in (\KK')^\circ$ and lying in the complexification of a line in $V$ spanned by the origin and the point $\ast'$. The push-forward map, together with the isomorphism \eqref{eq:nagiso} give an inclusion
\begin{equation}
\pi_1((\KK')^\circ,\ast')\hookrightarrow \pi_1(\KK^\circ,\ast)\cong \operatorname{Art}
\end{equation}
into the Artin group. Under this identification (and up to conjugacy),
\begin{equation}
\sigma=\Pi^h.
\end{equation}
\end{enumerate}
\end{proposition}

\subsubsection{Simple singularities} \label{subsection:simplesingularities}

\begin{definition} 
Two polynomials $f_1\in \mathbb C[x_1,\ldots,x_{m_1}]$ and $f_2\in \mathbb C[x_1,\ldots,x_{m_2}]$ are said to be \textbf{stably equivalent} if there exists an $m_3\ge \operatorname{max}(m_1,m_2)$ such that
$$
\frac{\mathbb C[[x_1,\ldots,x_{m_3}]]}{(f_1+x_{m_1+1}^2+\ldots+x_{m_3}^2)}\cong
\frac{\mathbb C[[x_1,\ldots,x_{m_3}]]}{(f_2+x_{m_2+1}^2+\ldots+x_{m_3}^2)}.
$$
Two hypersurface singularities are said to be stably equivalent if they are defined by stably equivalent polynomials. 
\end{definition} 
A basic observation is that the mini-versal deformation spaces (and the discriminants) for stably equivalent hypersurface singularities can be identified. This justifies the following classification:

\begin{definition}
The $m$-dimensional \textbf{ADE singularity} $T \in \left\{A_k,D_4,E_6,E_7,E_8\right\}$ is an isolated hypersurface singularity that can be defined by the vanishing of a single polynomial:  \vspace{0.5em}
\begin{equation}
\left\{ 
\begin{split}
&f_{A_k} := x_1^2 + \ldots + x_m^2 + x^{k+1}_{m+1} = 0; \\
&f_{D_4} := x_1^2 + \ldots + x_{m+1}(x_m^2 + x^{k-2}_{m+1}) = 0; \\
&f_{E_6} := x_1^2 + \ldots + x_m^3 + x^{4}_{m+1} = 0; \\
&f_{E_7} := x_1^2 + \ldots + x_m(x_m^2 + x^{3}_{m+1}) = 0; \\
&f_{E_8} := x_1^2 + \ldots + x_m^3 + x^{5}_{m+1} = 0.  
\end{split} \right. 
\end{equation} 
\end{definition}

Let $f$ be an $m$-dimensional ADE-singularity of type $T$. One can equip the middle-dimensional homology of $M_f$, 
\begin{equation}
H_m(M_f;\Z) \iso \bigoplus_{\mu(f)} \Z,
\end{equation}
with a bilinear form called the \textbf{intersection form} $\circ$ (for instance, formally, using cohomology with compact support). It is symmetric for $m$ even and alternating for $m$ odd. \\

One can also describe it more geometrically: consider a generic 1-parameter deformation of an ADE singularity of the form $f_w(\textbf{x}) = f(\textbf{x}) + w$ defined over a small disc $w \in B_r(0)$. Let $\gamma$ be a small loop, that corresponds to the boundary of $B_r(0)$, traversed once positively (this is the same $\gamma$ as in Proposition \ref{prop: main}.) We would like to understand $\rho_s([\gamma])$. 

\begin{definition}
Let $f$ be a singularity. A generic small perturbation to a Morse function $\tilde{f}$ is called a \textbf{Morsification} of $f$. It has a collection of non-degenerate critical points near $0$. We will always assume $\tilde{f}$ has distinct critical values. 
\end{definition}

Considered as a disc in the miniversal deformation, $f_w$ meets the discriminant locus in a non-transverse intersection point (which corresponds to the singularity at $0$), however, note that we can perturb the disc by morsifying $f \to \tilde{f}$ and considering instead the deformation $\tilde{f}_w(\textbf{x}) = \tilde{f}(\textbf{x}) + w$. Let $\tilde{\gamma}$ be the boundary loop of the perturb disc. Then if we choose $\tilde{f}$ to be sufficiently close to $f$, it is clear that $\rho_s([\gamma]) = \rho_s([\tilde{\gamma}])$ (because the two Milnor fibrations are in fact cobordant.) \\

Fix a Morsification $\tilde{f}$ and pick a regular value $\epsilon$ of $\tilde{f}$.  

\begin{lemma}[2.18 in \cite{MR3432159}]\label{lem:fibre_Morse}
Let $\delta$ and $\epsilon_\delta$ be as in Definition \ref{def:Milnorfibre}. Then provided the perturbation is sufficiently small, $\tilde{f}^{-1}(\epsilon_\delta)$ intersects $S_{\sqrt{\delta}}(0)$ transversely, $\tilde{f}^{-1}(\epsilon) \cap \overline{B}_{\sqrt{\delta}}(0)$ is a Liouville domain, and moreover, after attaching conical ends, this is exact symplectomorphic to the Milnor fibre of $f$. Call this space $F_{\epsilon_\delta}$.
\end{lemma}

\begin{definition} A \textbf{distinguished collection of vanishing paths} is an ordered family of vanishing paths $\gamma_i$ between $\epsilon$ and the singular values $x_s$, with $i=1, \ldots, \mu(f)$, such that:

\begin{enumerate}

\item  They only intersect at $\epsilon$.

\item Their starting directions $\R_+ \cdot \gamma_i'(0)$  are distinct.

\item They are linearly ordered by clockwise exiting angle at $\epsilon$.

\end{enumerate}
The resulting vanishing cycles give an ordered, so-called \textbf{distinguished basis} for the middle-dimension homology of the Milnor fibre.
\end{definition}

Every vanishing cycle is a half-dimensional Lagrangian sphere, and as in \cite[Proposition 1.5]{MR1978046}, we can ''double" each distinguished path to form a loop and the monodromy of $F_{\epsilon_\delta}$ around the loop is Hamiltonian isotopic to a generalized Dehn twist along the corresponding vanishing cycle. In particular, the action of a fixed sphere $L$ on homology is given by the \textbf{Picard-Lefschetz formula}
\begin{equation}
\tau_{[L]}([X]) = [X] + (-1)^{m(m+1)/2} ([X] \circ [L]) \cdot [L].  
\end{equation}

Thus, when $m \equiv 2 (mod \: 4)$, there is an isomorphism between $(H_n(M_f),\circ)$ and $R$ and using \ref{eq:nagiso} the monodromy representation $\rho_H$ of the miniversal deformation gives the standard action of the Weyl group on the root lattice and (See \cite[p.~129]{MR1660090}.)

\begin{definition}
The automorphism $\rho_H(\gamma) : H_m(M_f)  \to H_m(M_f) $ is called the \textbf{classical monodromy operator}.
\end{definition}

It follows from Proposition \ref{prop: main} that when $m \equiv 2 (mod \: 4)$ the classical monodromy operator acts as the Coxeter element. Now the general case follows via stabilization (See \cite[p.~63--64]{MR1660090}): That is, it is known that is we denote $M'_f$ the Milnor fiber of a stably equivalent singularity of dimension 2 and by $\rho'_f$ the corresponding homology representation, then there exists an isomorphism of middle dimensional homology 
\begin{equation}
\nu : H_2(M'_f;\Z) \to H_m(M_f;\Z)
\end{equation}
such that 
\begin{equation}
\rho_H(\gamma) = (-Id)^m \circ \nu \circ \rho'_f(\gamma) \circ \nu^{-1}.
\end{equation}
Thus 
\begin{equation}
\rho_H(\sigma) = \rho_H(\gamma^h)  = (-Id)^{h \cdot n}. 
\end{equation}
But the Coxeter number is even (except in the case that $T = A_{even}$), so
\begin{corollary} \label{cor:1}
When $T = A_{even}$ and $n$ is odd, $\rho_H(\sigma) = -Id$. Otherwise $\rho_H(\sigma) = Id$.
\end{corollary}

\begin{remark}
When $n$ is even, the result is obvious, because passing to the Weyl cover $\KK' \to \KK$ makes the entire geometric monodromy representation trivial. However, when $n$ is odd, there exists elements in the monodromy group of infinite order. Nevertheless, Corollary \ref{cor:1} shows that the classical monodromy operator of an ADE-singularity is always of finite order, regardless of parity.
\end{remark}

\subsubsection{The $A_{odd}$-Milnor fiber}
Now we specialize to the case of a singularity of type $(A_k)$ for $k$ odd and $m \geq 2$. \vspace{0.5em}
\begin{itemize}
\item
The Milnor (and Tjurina) algebra of the $m$-dimensional $A_{k}$-singularity is
\begin{equation}
\KK_{A_k} = \C[t_1,\ldots,t_k,t_{k+1}]/(t_{k+1}) = \C[t_1,\ldots,t_k]. \vspace{0.5em}
\end{equation}
\item 
If we define
\begin{equation}
h_{t_1,\ldots,t_k}(x_{m+1}) := x_{m+1}^{k+1} + t_1 x_{m+1}^{k} + \ldots + t_k
\end{equation}
Then the miniversal deformation space (see e.g. Theorem 7.9, Corollary 7.10, and Example 7.17 in \cite{eprint10}) is given the ring map
\begin{equation}
\C[t_1,\ldots,t_k] \to \frac{\C[t_1, \ldots , t_k][x_1,\ldots,x_{m+1}]}{( x_1^2 + \ldots + x_m^2 +h_{t_1,\ldots,t_k}(x_{m+1}) )}. \vspace{0.5em}
\end{equation}
\item 
Note that the discriminant locus of the deformation, denoted $\Delta_{A_k}$, is exactly the set of $(t_1,\ldots,t_k)$ where $h$ acquires a double root, i.e.,
\begin{equation}
\C[t_1,\ldots,t_{k+1}]/(t_{k+1} , D(t_1,\ldots,t_{k+1})). 
\end{equation}
The complement is canonically isomorphic to the configuration space of $(k+1)$-points
\begin{equation}
\KK_{A_k}^\circ \cong \mathrm{Conf}_{k+1}(\C). 
\end{equation}
\item 
As we can see from the Dynkin diagram, $Art(A_k) = B_{k+1}$ the classical \textbf{Braid group} on $(k+1)$-strands.
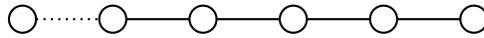
\begin{figure}[h!] 	\label{fig:dynkinAk}
\centering
  \begin{tikzpicture}[scale=.6]
    \draw (-1,0) node[anchor=east]{};
    \foreach \x in {0,...,5}
    \draw[xshift=\x cm,thick] (\x cm,0) circle (.3cm);
    \draw[dotted,thick] (0.3 cm,0) -- +(1.4 cm,0);
    \foreach \y in {1.15,...,4.15}
    \draw[xshift=\y cm,thick] (\y cm,0) -- +(1.4 cm,0);
	\end{tikzpicture}
	\caption{The Dynkin diagram for $A_k$}
\end{figure}
It is generated by $\left\{\sigma_1,\ldots,\sigma_n\right\}$ subject to two relations
\begin{equation} 
\begin{split}
\sigma_i \sigma_j \sigma_i  &= \sigma_j \sigma_i \sigma_j \: \: \mbox{ if } |i-j|=1, \\
\sigma_i \sigma_j &= \sigma_j \sigma_i \: \: \: \:\: \: \mbox{ if } |i-j|>1.  \vspace{0.5em}
\end{split}
\end{equation}
\item The Weyl group is isomorphic to the symmetric group on $(k+1)$-letters:
\begin{equation}
W(A_k) \cong \Sigma_{k+1}. 
\end{equation}
\item
The classical monodromy operator has order $h = k+1$. In fact since $k+1 = 2\ell$ and

\begin{equation}
f_{A_k}(\lambda^{\beta_0} X_0,\ldots,\lambda^{\beta_k}X_k) = \lambda^\beta f_{A_k}(X_0,\ldots,X_k) 
\end{equation}

for $(\beta_0,\ldots,\beta_k) = (\ell,\ldots,\ell,1)$, $\beta = k+1$. The singularity $V(f_{A_k})$ is weighted homogeneous with weights $(\frac{1}{2},\ldots,\frac{1}{2},\frac{1}{k+1})$. It follows that $\rho(\gamma^h) = \rho(\sigma)$ is the symplectic isotopy class of the boundary twist from \cite[Lemma 4.16]{MR1765826}, thus can be represented by a symplectomorphism that is supported in a tubular neighbourhood of the boundary of the Milnor fiber. \vspace{0.5em}
\end{itemize}


\section{Families of curves and cubic 3-folds} \label{sec:familiesofcurves}
This section is dedicated to proving Propositions \ref{prop:blowup1} and \ref{prop:blowup2}. \\

\textbf{Notation and conventions.} We work in the category of schemes of finite type over the field of complex numbers. We will say a scheme $X$ has a singularity of type $T$ at a point $x \in X$ if the completion of the local ring $\hat{\OO}_{X,x}$ is isomorphic to the standard complete local ring with singularity of type $T$. All our schemes would be \textbf{ADE-schemes}, which means that they are either smooth everywhere, or have a finite number of ADE-singularities. A \textbf{curve} will be a reduced, connected, complete scheme of pure dimension one (of finite type over $\C$) - but not necessarily irreducible. A \textbf{cubic threefold} is a hypersurface of degree $3$ in four-dimensional projective space. 

\subsection{The canonical model of a genus four curve} 
Let $C$ be a smooth, non-hyperelliptic curve of genus $g=4$. Denote by $\KK_C$ the canonical divisor. 
\begin{definition}
By (geometric) Riemann-Roch, the map 
\begin{equation}
\phi_{\mathcal{K}} : C \hookrightarrow C_{can} \subset \P H^0(C,\mathcal{K}_{C})^\vee \iso \P^{g-1}
\end{equation}
defined by linear system $|\mathcal{K}_C|$ embedds $C$ in $\P^3$ as a non-degenerate curve $C_{can}$ of degree $2g-2 = 6$. The embedded curve $C_{can} = \phi_{\mathcal{K}}(C)$
is called the \textbf{canonical model}.
\end{definition}
Consider the short exact sequence of sheaves
\begin{equation}
0 \to \II_{C_{can}} \to (\mathcal{O}_{\P^3})|_{C_{can}} \to \mathcal{O}_{C_{can}} := (\mathcal{O}_{\P^3})|_{C_{can}} / \mathcal{I}_{C_{can}} \to 0
\end{equation}
where $\II_{C_{can}}$ is the ideal sheaf and $\mathcal{O}_{C_{can}}$ is the structure sheaf of the $C_{can}$. For any $m \in \Z$, tensoring with the restriction of $\mathcal{O}_{\P^3}(m)$ to the curve is an exact functor, so we get
\begin{equation}
0 \to \II_{C_{can}}(m) \to (\mathcal{O}_{\P^3}(m))|_{C_{can}} \to \mathcal{O}_{C_{can}}(m) \to 0
\end{equation}
Pulling back by $\phi_{\mathcal{K}}$ and taking cohomology, we get a long exact sequenence
\begin{equation}
0 \to H^0(C,\phi_{\mathcal{K}}^* \II_{C_{can}}(m)) \to H^0(C,\phi_{\mathcal{K}}^* \mathcal{O}_{\P^3}(m)|_{C_{can}}) \to H^0(C,\phi_{\mathcal{K}}^* \mathcal{O}_{C_{can}}(m)) \to \ldots
\end{equation}
By definition, the pullback $\phi_{\mathcal{K}}^* \mathcal{O}_{C_{can}}(m)$ is $\mathcal{O}_{C}(m)$, so there is a map on vector spaces
\begin{equation} \label{eq:restrictionmap}
H^0(\P^3, \mathcal{O}_{\P^3}(m)) = Sym^m \C[X_0,\ldots,X_3] \stackrel{\phi_{\mathcal{K}}^*}{\longrightarrow} H^0(C,\mathcal{O}_{C}(m)). 
\end{equation}
\begin{lemma} \label{thm:complete23_1}
$C_{can}$ lies on a unique irreducible quadric surface $Q = V(q)$. The rank of $q$ is at least $3$. 
\end{lemma} 
\begin{proof}
Consider the restriction map \eqref{eq:restrictionmap} for $m=2$:
\begin{equation} 
H^0(\P^3, \mathcal{O}_{\P^3}(2)) \stackrel{\phi_{\mathcal{K}}^*}{\longrightarrow} H^0(C,\mathcal{O}_{C}(2d)). 
\end{equation}
The dimension of the domain is $\binom{3+2}{2} = 10$ (stars-and-bars), but the dimension of the target is $2d-g+1 = 12-4+1 = 9$ (Riemann-Roch). So there exists at least one quadric polynomial $q$ such that $C_{can} \subset Q := V(q)$. The canonical curve can not lie on a plane, so $Q$ must be non-degenerate. Now if there were two linearly independent quadrics such that $C_{can} \subset Q,Q'$, $C_{can}$ would also be contained in the intersection $Q \cap Q'$, which is a degree $4$ curve (B\'{e}zout). But $C_{can}$ has degree $6$, so that's impossible.
\end{proof}
\begin{lemma} \label{thm:complete23_2}
There is a 4-dimensional linear system $\Pi$ of cubic surfaces which contain $C_{can}$, which contains at least one smooth surface. 
\end{lemma}
\begin{proof}
Again, we consider the restriction map \eqref{eq:restrictionmap} for $m=3$:
\begin{equation} 
H^0(\P^3, \mathcal{O}_{\P^3}(3)) \stackrel{\phi_{\mathcal{K}}^*}{\longrightarrow} H^0(C,\mathcal{O}_{C}(3d)). 
\end{equation}
The dimension of the domain is $\binom{3+3}{3} = 20$, but the dimension of the target is $3d-g+1 = 18-4+1 = 15$. So the kernel is at least $5$ dimensional. $C_{can}$ does not lie on a plane, so the only reducible cubics that contain $C_{can}$ are hyperplane sections of $Q$, which form a $4$ dimensional family. Thus, we can choose an irreducible smooth cubic such that $C_{can} \subset S := V(s)$. But if there were two such linearly independent cubics $S$ and $S'$, then $C_{can}$ would be contained in a degree $9$ curve for the same reason as before. Thus the kernel is indeed $5$-dimensional. 
\end{proof}
\begin{lemma} \label{thm:complete23_3} 
$C_{can}$ is the scheme-theoretic complete intersection of $Q$ and $S$. Conversely, any smooth complete intersection of a quadric and a cubic is its own canonical model (in particular, it is not hyperelliptic). 
\end{lemma}
\begin{proof}
Since $C_{can}$ is a sextic, by B\'{e}zout, $C_{can} = Q \cap S$ is a complete intersection. Conversely, a $(2,3)$-complete intersection $C$ has degree $6$. Applying the adjunction formula shows that $\KK_C$ is the product $(2+3 - 4)H \cdot 2H \cdot 3H = 6H$, but this is exactly $\OO_C(1)$. 
\end{proof}
There are different notions of what it means for a variety $X \subset \P^n$ to be cut out by hypersurfaces. One is that $X$ is the set-theoretic intersection of these; one can also requre that $X$ is the scheme-theoretic complete intersection; and yet a third option is that the homogeneous ideal $\II_X$ of $X$ is cut out by the ideals of these hypersurfaces. Note that in general, these are not equivalent: scheme-theoretic complete intersection only implies that the homogeneous ideal of $X$ is generated in \textit{sufficiently high degrees} by the equations of the ideals of these hypersurfaces. However in our case, Max Noether's $AF+BG$ theorem implies that we can strenghten \ref{thm:complete23_3}:
\begin{corollary} \label{thm:complete23_4}
The ideal $\II_{C_{can}}$ is generated by $q$ and $s$. \noproof
\end{corollary} 
Since $q$ and $s$ have no common irreducible factors in this case, we get the Koszul resolution of the ideal sheaf $\II_{C_{can}}$:
\begin{equation}
0 \to \OO_{\P^3}(-5) \stackrel{(s,-q)}{\longrightarrow} \OO_{\P^3}(-2) \oplus \OO_{\P^3}(-3) \stackrel{\begin{pmatrix}q\\s\end{pmatrix}}{\longrightarrow} \II_{C_{can}} \to 0.
\end{equation}
Restricting to the curve, we see that 
\begin{equation}
(\II_{C_{can}} / \II_{C_{can}}^2)|_{C_{can}} \iso \left(\OO_{\P^3}(-2) \oplus \OO_{\P^3}(-3)\right)|_{C_{can}}, 
\end{equation}
so 
\begin{corollary}
The pullback of the normal bundle $\NN_{C_{can}/\P^3}$ to the curve is
\begin{equation} \label{eq:12_18}
\OO_{C}(12) \oplus \OO_{C}(18).
\end{equation}
\end{corollary}

\subsection{The ''Watchtower correspondence"} \label{subsec:watchtower}
It is classically known that for every canonically model $C_{can}$ the linear system $\Pi$ from Lemma \ref{thm:complete23_2} defines a rational map
\begin{equation}
\rho : \P^3 \to \Pi^\wedge \iso \P^4
\end{equation}
sending a point $x \in \P^3 \setminus C_{can}$ to the hyperplane $\left\{H \in \Pi \: \big| \: x \in H\right\}$ which consists of cubics which contain $C_{can} \cup \left\{x\right\}$. To understand the image, we choose a generic 
hyperplane $H \subset \P^3$ and consider the image of the restriction $\rho \big|_H$ in $\P^4$. We know that the the intersection $H \cap C_{can}$ consists of six points in general position, so the image of the map given by cubics passing via these points is must be a Del Pezzo surface of degree $3$, and this is a linear section of $X$. Thus $X$ is a cubic 3-fold. \\

To obtain a more precise description of $\rho$, we observe that a cubic hypersurface containing $C_{can}$ that contains also a point of $x \in Q \setminus C_{can}$ automatically contains the entire quadric (B\'{e}zout). Therefore the image of such a point is independent of the choice of $x$ and the entire open part $Q \setminus C_{can}$ is contracted; however, $\phi_\KK(C)$ itself is blown up to a $\P^1$-bundle, the projectivized normal bundle of $\phi_\KK(C)$. Thus, $X$ can be obtained by first blowing up $\P^3$ along $\phi_\KK(C)$ and then blowing down the proper transform $\tilde{Q}$ of $Q$. So $X$ has one singular point $p$, which is the image of $\tilde{Q}$. It is easy to see that if $Q$ is smooth then point $p$ is an $A_1$-singularity (and if singular then an $A_2$). Conversely, starting from a cubic 3-fold $X$ which is smooth everywhere except at a unique double point $p$, we can consider the projection from $p$ to a generic hyperplane
\begin{equation}
\pi_p : X \rightarrow \P^3.
\end{equation}
If we resolve the indeterminacity by blowing up $X$ at $p$, then $\pi_p$ is a dominant birational map and the image of the exceptional divisor is a quadric in $\P^3$; However, tangent lines to $X$ at $p$ that actually lie in $X$ blowdown onto a smooth genus 4 curve (the point $p$ is sometimes called the \textbf{watchtower}, and $C$ the \textbf{associated sextic} for the obvious reasons.) We denote
\begin{equation}
Y = Bl_p X = Bl_{\phi_\KK(C)} \P^3. 
\end{equation}

To summarize the discussion so far,
\begin{corollary} 
The projection from a singular point of a cubic threefold gives a natural correspondence between smooth non-hyperelliptic genus $4$ curves and cubic threefolds with a unique singularity of type $A_1$ or $A_2$. The singularity is $A_2$ if and only if the corresponding curve has a unique $g^1_3$ (i.e., if $Q$ is a quadric cone).
\end{corollary}

\begin{proposition}[Finkelberg] \label{prop:finkl}
Let $\ell$ be a line in $X$. If $p \in \ell$ then $\ell$ is a line $\overline{px}$ with $x \in C_{can}$; Conversely, every $x \in C_{can}$ defines a unique line in $X$ that passes via $p$. A line $\ell$ which does not contain $p$ is mapped onto a line $\overline{\ell} = \pi_p(\ell) \subset \P^3$ not contained in $Q$ which either connects two points of $C_{can}$ or is a tangent to a point $x \in C_{can}$. Conversely every such line $\overline{\ell}$ is the
image of a line in $X$.
\end{proposition}

There is a generalization of the above story which allows more interesting singularities. Our presentation follows the series of papers: \cite{MR3006172}, \cite{MR2561195},\cite{MR2895186}, and \cite{MR3263667} which deal with alternate birational compactifications of $\overline{M}_4$. 
\begin{remark}
Actually, the connection with the log minimal model program of Hasset-Keel runs deeper. We briefly explain the idea later in Section \ref{subsec:remarkonbirationalmodels}. 
\end{remark}
Given any cubic 3-fold $X \subset \P^4$ with a singular point $p \in X$ of multiplicity two, say $p = [0:0:1:0:0]$, an equation for $X$ can be written as
\begin{equation} \label{eq:definecubic}
X_2 \cdot q (X_0, X_1, X_3 , X_4) + s(X_0, X_1, X_3 , X_4)
\end{equation}
with $q$ and $s$ homogeneous of degrees $2$ and $3$ respectively. The ideal $(q, s)$ defines a scheme in $\P^3$ of type $(2,3)$. Conversely, given a $(2,3)$-complete intersection scheme $C$ together with a choice of generators $q$ and $s$ of the defining ideal, there is a cubic hypersurface $X$ with a singularity of multiplicity $2$ at the point $[0:0:1:0:0]$ defined by equation \eqref{eq:definecubic}. In a more coordinate invariant way, there is a commutative diagram:
\begin{equation}
\begin{tikzpicture}
\matrix (m) [matrix of math nodes, row sep=3em,
column sep=3em, text height=1.5ex, text depth=0.25ex]
{  & Q_p & Y  & E_p &  \\
\left\{p\right\} &  X &  &   \P^{3}& C_p  \\ };
\draw[right hook->] (m-1-2) --node[above]{\footnotesize{$i$}}  (m-1-3);
\draw[left hook->] (m-1-4) --node[above]{\footnotesize{$j$}}  (m-1-3);
\draw[dashed,->] (m-2-2) --node[above]{\footnotesize{$\pi_p$}}  (m-2-4);
\draw[->] (m-1-3) --node[above]{\footnotesize{$f$}} (m-2-2);
\draw[->] (m-1-3) --node[above]{\footnotesize{$g$}} (m-2-4);
\draw[right hook->] (m-2-1) --  (m-2-2);
\draw[left hook->] (m-2-5) --  (m-2-4);
\draw[->] (m-1-2) -- (m-2-1);
\draw[->] (m-1-4) -- (m-2-5);
\end{tikzpicture}
\label{eq:tikzpicture1}
\end{equation}
where $\pi_p$ is the projection, $Q_p$ is the projectivized tangent cone at $p$ (as before), $C_p$ is the scheme parametrizing the lines on $X$ passing through $p$, and $E_p$ is the exceptional divisor of the blowdown map $g$. We remark that as in the smooth case (Proposition \ref{prop:finkl}), the loci of lines lying on $X$ and passing through $p$ is exactly the cone over $C_p$, and $E_p$ is the proper transform of it. \\

We now recall the well known connection between the singularities of $C$ and the singularities of $X$ and $Y$. Namely, let $p \neq q \in X$ be another singular point, then since $p$ has multiplicity two, the line $\overline{pq}$ must be contained in $X$ (B\'{e}zout). Thus 
for every such singular point, we have $\pi_p(q) \in C$. The following are well known (see e.g. \cite{preprint}):
\begin{lemma}
Let us fix a point $x \in C_{p}$, and denote $\ell = \overline{px}$.
\begin{enumerate}
\item
If $C_{p}$ is smooth at $x$, then $X$ is smooth along $\ell$  except at $p$. 
\item
If $C_{p}$ has a singularity at $x$ and $Q$ is smooth at $x$, then $X$ has exactly two singular points $p, q$ on the line $\ell$. Moreover, $C_{p}$ has a singularity of type T at $x$ if and only if $X$ has a singularity of type T at $q$.
\item
If $C_{p}$ has a singularity at $x$, $Q$ is singular at $x$, and $S$ smooth at $x$, then the only singularity of $Y$ along the line $\ell$ is at $p$ itself. Moreover, if $C_{p}$ has a singularity of type T at $x$ if and only if $Y$ has a singularity at of type T at $x$ (where we have identified the exceptional divisor with the quadric in $\P^3$ using the projection).
\item
If $Q$ and $S$ are both singular at $x$, then $Y$ is singular along the line $\ell$.  
\end{enumerate}
\end{lemma}
As mentioned in \cite[Remark 1.2]{MR2895186}, if $x$ is a singular point of a complete intersection of $Q$ and $S$, then $x$ is a hypersurface singularity if and only if $Q$ and $S$ are not both singular at $x$; thus the comparison of types above is well defined using the stabilization of singularities. 

\begin{proposition} $X$ has isolated singularities if and only if $C_p$ is a curve with at worst hypersurface singularities. Assuming either of these equivalent conditions hold, let $p \in X$ be a singularity of corank $\leq 2$. With notations as above, we have:
\begin{enumerate}
\item
$C_p$ is a proper $(2, 3)$-complete intersection in $\P^3$;
\item
$C_p$ is reduced (but possibly reducible);
\item
The singularities of $C_p$ are in one-to-one correspondence (including the type!) with the singularities of $Y$. More precisely, if $p$ is a singularity of type $A_k$, then $Y = Bl_p X$ has a unique singular point along the exceptional divisor, which is of type $A_{k-2}$ (smooth for $k \leq 2$). Similarly, if $p$ is a $D_4$, then there are exactly three $A_1$'s along $E_p$.
\end{enumerate}
Furthermore, $Y$ is the blow-up of $\P^3$ along the reduced scheme $C_p$. In particular, $E_p$ is a ruled surface.
\end{proposition}
\begin{proof} 
This is classical. See e.g. \cite[p.~7]{preprint}. 
\end{proof}
A more detailed analysis leads to the following.
\begin{corollary} Let $X$ be a cubic threefold and $p$ an isolated singular point. Then the singularities of $X$ are at worst of type $T \in \left\{A_k\right\}_{k \in \N} \cup \left\{D_4\right\}$ if and only if the singularities of $C_p$ are at worst of type $T \in \left\{A_k\right\}_{k \in \N} \cup \left\{D_4\right\}$ and either $Q_p$ is irreducible, or $Q_p$ is the union of two distinct planes and $C_p$ meets the singular line of $Q_p$ in three distinct points. Moreover, under either of these equivalent conditions:
\begin{enumerate}
\item
The singularity at $p$ is $A_1$ if and only if $Q_p$ is a smooth quadric;
\item
The singularity at $p$ is $A_2$ if and only if $Q_p$ is a quadric cone and $C_p$ does not passes through the vertex of $Q_p$;
\item
The singularity at $p$ is $A_k$ for $k \geq 3$ if and only if $Q_p$ is a quadric cone, $C_p$ passes through the vertex $v$ of $Q_p$, and the singularity of $C_p$ at $v$ is $A_{k-2}$;
\item
The singularity at $p$ is $D_4$ if and only if $Q_p$ is the union of two distinct planes and $C_p$ meets the singular line of $Q_p$ in three distinct points.
\end{enumerate}
\end{corollary}
\begin{proof} 
See \cite[Corollary 3.2]{MR3263667}. 
\end{proof}
Let $C$ be the complete intersection of a quadric and a cubic. The only planar singularities of multiplicity two are the $A_k$-singularities. They are divided into two types: those that separate the curve and those that do not. Note that the type of singularities on such curves (and thus on cubic 3-folds) is severely restricted: since the arithmetic genus of $C$ is $4$, and the the local contribution of an $A_k$-singularity to the genus is $\ceil{\frac{k}{2}}$, it follows immediately that $C$ cannot admit an $A_k$-singularity with $k \geq 10$; and if $C$ possesses a separating singularity of type $A_{2k−1}$, then $C = C_1 \cup C_2$, where $C_1$ and $C_2$ are connected curves meeting in a single point with multiplicity $k$. In fact, a case-by-case analysis shows that
\begin{proposition}[2.1 in \cite{MR3263667}] \label{prop:2.1}
There exists a reduced $(2, 3)$-complete intersection possessing a non-separating singularity of type $A_k$ if and only if $k \leq 8$. Moreover, if $C$ is a $(2, 3)$-complete intersection with a separating $A_k$ singularity at a smooth point of the quadric on which it lies, then one of the following holds:
\begin{enumerate}
\item
$k = 5$, and $C$ is the union of a quintic and a line.
\item
$k = 7$, and $C$ is the union of a quartic and a conic.
\item
$k = 9$, and $C$ is the union of two twisted cubics.
\end{enumerate} \noproof
\end{proposition}

\subsection{Moduli of cubic 3-folds}
Let $\CC$ be the space of all nonzero cubic forms in homogeneous variables $X_0, \ldots , X_4$. For such a form $f$ let $F = V(f)$ be the cubic threefold in $\P^{4}$ it defines. We write $\CC_0$ for the set of $f \in \CC$ for which $F$ is smooth (as a scheme) and $\Delta := \CC \setminus \CC_0$ for the discriminant. We let 
\begin{equation} \label{eq:universalfamilyofcubics}
\FF = \left\{(f, [X_0: \ldots : X_4]) \in \CC \times \P^4 \: \big| \: f(X_0, \ldots , X_4) = 0\right\},
\end{equation}
denote the universal family of cubic threefolds and write $\pi$ for the projection $\FF \to \CC$. Note that the total space of $\FF$ is smooth (To see this, we differentiate the defining equation with respect to the coefficients of $f$ and note that the simultaneous vanishing of all partial derivatives with respect to these coefficients forces $X_0 = \ldots = X_4 = 0$.) We write $\FF_0$ for the topologically locally trivial fibration which is the restriction of \eqref{eq:universalfamilyofcubics} to $\CC_0 \hookrightarrow \CC$. \\

The group $G = SL_5(\C)$ acts on $\CC$ by substitution of variables. 
\begin{definition}
We say that two cubics $F = V(f)$ and $F' = V(f')$ are \textbf{projectively equivalent} if $f$ and $f'$ belong to the same $G$-orbit.
\end{definition}
We introduce the following family of cubic forms
\begin{equation}
f_{A,B} := A \cdot X_2^3 + X_0 X_3^2 + X_1^2 X_4 - X_0 X_2 X_4 + B \cdot X_1 X_2 X_3 
\end{equation}
where at least one of constants $A$ and $B$ is nonzero. Denote the associated cubic threefold as $F_{A,B}$. Let $k \neq 0$ be any scalar, by rescaling the variables
\begin{equation}
(X_0,X_1,X_2,X_3,X_4) \mapsto (k^{-2/3} \cdot X_0,X_1,k^{2/3} \cdot X_2,k^{1/3} \cdot X_3,X_4)
\end{equation}
we see that $F_{A,B}$ is projectively equivalent to $F_{k^2 A,kB}$. We denote 
\begin{equation}
\beta := 4A / B^2 \in \C \cup \left\{\infty\right\},
\end{equation}
and write $[F]_\beta$ for the projective equivalence class of $F_{A,B}$. This is justified because

\begin{lemma}
$F_{A,B}$ and $F_{A',B'}$ are projectively equivalent if and only if $\beta = \beta'$.
\end{lemma}

\begin{proof}
See \cite[Lemma 5.5]{MR1949641}.
\end{proof}

If $\beta \neq 0,1$ then $F_{A,B}$ has just two singularities, both of them of type $A_5$. If $\beta=0$ then $F_{A,B}$ has an $A_1$-singularity at $p = [0:0:1:0:0]$ as well as two $A_5$-singularities. If $\beta=1$ then $F_{A,B}$ is the secant variety of a rational normal curve of degree $4$, which is called the chordal cubic. It is singular along the rational normal curve. 

\begin{lemma}
Any cubic with two $A_5$ singularities is in the projective equivalence class $[F]_\beta$ for some $\beta \neq 1$. 
\end{lemma}

\begin{proof}
Follows from \cite[Theorem 5.7]{MR1949641}.
\end{proof}

Denote $q_{max}  := X_1 X_3 - X_0 X_4, s_{max} := X_0 X_3^2 + X_1^2 X_4$ and 
\begin{equation}
Q_{max} := V(q_{max}) \: , \: S_{max} := V(s_{max}) \subset \P^3. 
\end{equation}
\begin{corollary}
Up to projective equivalence, there is a unique cubic threefold $F_{max}$ with with two $A_5$-singularities and one $A_1$-singularity. It is defined by the cubic form
\begin{equation}
\begin{split}
f_{max} &:= X_2 \cdot q_0 + s_0 \\
&= X_0 X_3^2 + X_1^2 X_4 - X_0 X_2 X_4 +  X_1 X_2 X_3.
\end{split}
\end{equation}
Since $Q_{max} \cong \P^1 \times \P^1$ we can write
\begin{equation}
X_0 = Y_0 Y_2 \:, \: X_1 = Y_0 Y_3  \:, \: X_3 = Y_1 Y_2  \:, \: X_4 = Y_1 Y_3,
\end{equation}
where $Y_0,Y_1$ resp. $Y_2,Y_3$ are homogeneous coordinates on $\P^1$. Setting $p = [0:0:1:0:0]$, and projecting
\begin{equation}
\begin{split}
\pi_p &: \P^4 \to \P^3, \\
[X_0:X_1:X_2:X_3:X_4] &\mapsto [X_0:X_1:X_3:X_4].
\end{split}
\end{equation}
we see that the curve $C_{max} := V(q_{max},s_{max})$ is a complete intersection of bidegree $(3,3)$ given by 
\begin{equation}
Y_0 Y_1 (Y_1 Y_2^3 + Y_0 Y_3^3)
\end{equation}
in the bihomogeneous coordinates in $Q_{max}$. It consists of three rational curves: two lines $L_1$ and $L_2$ in the same ruling $(1, 0)$ and a smooth $(1, 3)$ curve, denoted $L_3$, that meets each line $L_i$ at a unique point $q_i$. See Figure \ref{fig:maximalA5} for illustration. 

\begin{figure}
\centering
			\fontsize{0.25cm}{1em}
			\def\svgwidth{3cm}
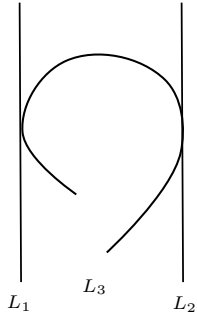
\caption{The curve $C_{max}$ sitting inside the quadric}
\label{fig:maximalA5}
\end{figure}

These intersection points are highly non-transverse (in fact, tangent with multiplicity 3); and the points $q_i$ are $A_5$-singularities for $C_{max}$ (as in Proposition \ref{prop:2.1}, see also the analysis in \cite[Theorem 3.5]{MR1949641}.)
\end{corollary}

\begin{definition} \label{def:definitionofmaximallydegeneate2A5}
Following \cite{MR3006172}, we call $F_{max}$ (resp. $C_{max}$) the \textbf{maximally degenerate} $2A_5$-cubic (resp. maximally degenerate $2A_5$-curve). 
\end{definition}

\subsection{Deformations and Monodromy} 

\begin{definition}
By a \textbf{meridian} around an irreducible divisor, we mean the boundary circle of a small disk transverse to the divisor at a generic point of it, traversed once positively.
\end{definition}

The degenerate cubic 3-fold $F_{max}$ has three singular points 
\begin{equation}
\left\{p,q_1,q_2\right\} = \left\{[0:0:1:0:0],[1:0:0:0:0],[0:0:0:0:1]\right\}
\end{equation}
of types $T_p = A_1$ and $T_{q_1} , T_{q_2} = A_5$, and no other singularities. 

\begin{lemma} \label{lem:local_analytic_nbhd}
The family of cubic 3-folds $\CC$ induces a simultaneous versal deformation of all singularities of $F_{max}$. 
\end{lemma}
\begin{proof}
By Theorem 1.1 in \cite{MR1735775}, it suffices to show that the sum of the Tjurina numbers of
$p,q_1,q_2$ is less than 16. But because the singularities of $F_{max}$ are quasihomogeneous,
their Tjurina numbers coincide with their Milnor numbers and $1+5+5<16$ satisfy the inequality. 
\end{proof}
\begin{remark}
Though we will need this result, we point out that Lemma \ref{lem:local_analytic_nbhd} is a particular case of a more general phenomena: the family of cubic 3-folds induces a simultaneous versal deformation of all singularities of any cubic 3-fold with only $A_k$ or $D_4$ singularities. See the proof of Lemma 1.5 in \cite{MR1949641}.
\end{remark}
As an immediate consequence, there exists an open neighborhood $U$ of $F$ in $\CC$ biholomorphic to 
\begin{equation} \label{eq:U}
U \iso \KK_{A_1} \times \KK_{A_5} \times \KK_{A_5} \times \C^{24}
\end{equation}
and under this identification, 
\begin{equation}
\Delta - U \iso (\KK_{A_1} - \Delta_{A_1}) \times (\KK_{A_5} - \Delta_{A_5}) \times (\KK_{A_5} - \Delta_{A_5}) \times \C^{24}. 
\end{equation}
In particular, a choice of meridians around each irreducible component of the discriminant determines an isomorphism
\begin{equation}
\pi_1(U - \Delta) \iso B_{2} \times B_6 \times B_6.
\end{equation}

\begin{lemma} \label{lem:local_analytic_nbhd2}
For a sufficiently small $\delta>0$, there exists a family of cubic 3-folds 
\begin{equation}
\begin{tikzpicture}
  \matrix (m) [matrix of math nodes,row sep=3em,column sep=4em,minimum width=2em]
  {
     \XX & B_\delta(0) \times \P^4 \\
     B_\delta(0)  & B_\delta(0)  \\};
  \path[->,font=\scriptsize,>=angle 90]
    (m-1-1) edge node [auto] {$s$} (m-2-1)
		(m-1-2) edge node [auto] {$s$} (m-2-2)
		(m-1-1) edge (m-1-2)
		(m-2-1) edge node [auto] {$=$} (m-2-2)
		;
\end{tikzpicture}
\end{equation}
such that:
\begin{enumerate}
\item
The central fiber $s=0$ is $X_0 = F_{max}$.
\item
For every $s \neq 0$, $X_s$ is smooth everywhere except at a single node at $p = [0:0:1:0:0]$. 
\end{enumerate}
\end{lemma}
\begin{proof}
The versal deformation space $\KK_{A_1} \iso \Spec(\C[t_1]) = \C$. We choose any generic disc of sufficiently small radius $\delta$ such that
\begin{equation}
B_\delta(0) \subset \left\{0\right\} \times \KK^\circ_{A_5} \times \KK^\circ_{A_5} \times \C^{24}
\end{equation}
We fix an open neighbourhood $U$ as in \eqref{eq:U} and take the inverse image under the identification as the base of our 1-parameter deformation. 
\end{proof}
Explicitly, we can choose a generic cubic form $s_1$ in homogeneous coordinates $X_0,X_1,X_3,X_4$ and consider the pencil
\begin{equation}
f_t := X_2 \cdot q_0 + [t \cdot s_1 + (1-t) \cdot s_0]. 
\end{equation}
The cubic 3-fold $F_t := V(f_t)$ has an $A_1$-singularity at $p$ for every $t$, and the fiber over $t=0$ is $F_{max}$. By Lemma 
\ref{lem:local_analytic_nbhd2}, it is clear that for a sufficiently small $\delta>0$ the family 
\begin{equation}
\CC = \left\{f_t\right\} \stackrel{t}{\rightarrow} B_\delta(0)
\end{equation}
is a simultaneous versal deformation of the $A_5$-singularities with a fixed node at $p$ as we wanted. 
\begin{definition}
Denote $\mathcal{P} : B_\delta(0) \to \P^4$ for the constant section $\mathcal{P}(s) = p$. Let 
\begin{equation}
\YY \subset B_\delta(0) \times Bl_p \P^4 
\end{equation}
denote the family of 3-folds obtained by blowing up the family $\XX$ along $\mathcal{P}$. 
\end{definition}
\begin{lemma}
The total space of $\YY$ is smooth. 
\end{lemma}
\begin{proof}
Away from the singular points in the central fiber $X_0$ that is obvious. Near $p$, the family $\XX$ is given by the vanishing of a single polynomial 
\begin{equation}
f(x_1,x_2,x_3,t) = (x_1^2 + x_2^2 + x_3^2) + (x_4^2-0) = x_1^2 + x_2^2 + x_3^2 + x_4^2
\end{equation}
where $(x_1,x_2,x_3,x_4,t) \in \C^4 \times \C$. Consider for example the coordinate patch where $x_1 \neq 0$. Then the blowup is described by the coordinate substitution 
\begin{equation}
(x_1,x_2,x_3,x_4,t) = (y_1,y_1 y_2 , y_1 y_3,y_1y_4,t)
\end{equation}
The family $\YY$ is given in these local coordinates by the vanishing of a polynomial
\begin{equation}
\tilde{f}(y_1,y_2,y_3,y_4,t) = y_1^2(1 + y_2^2 + y_3^2 + y_4^2 ) 
\end{equation}
and is clearly non-singular. Similarly for the $x_2 \neq 0$, $x_3 \neq 0$ etc. The versal deformation space of an $A_5$-singularity is given by the vanishing of 
\begin{equation}
g(x_1,x_2,x_3,x_4,t_1,\ldots,t_5) = (x_1^2+ x_2^2 + x_3^2) + (x_4^6 + t_1 x_4^5 + \ldots + t_5)  
\end{equation}
and is clearly non-singular (because $\frac{\partial g}{\partial t_5} = 1$.) Note that a generic complex line in $\KK_{A_5} \iso \C^5$ through the origin does not lie in the hyperplane $t_5=0$. Therefore the local model of the total space of the smoothing $\YY$ near $q_1$ and $q_2$ has no singular points. 
\end{proof}

Note that $Bl_p \P^4$ is a closed subvariety of $\P^4 \times \P^3$ defined by demanding that a pair of bihomogeneous coordinates $([X_0:\ldots:X_4],[Y_1 : \ldots Y_4])$ satisfy 
\begin{equation}
X_i Y_j - X_j Y_i =  0
\end{equation}
for all $1 \leq i,j \leq 4$. In turn, we identify $\P^4 \times \P^3$ with the Segre variety $\Sigma_{4,3} \subset \P^{19}$ which is defined as the image of the projective embedding
\begin{equation}
\begin{split}
\sigma &: \P^{4}\times \P^{3}\to \P^{19}, \\
([X_{0}:X_{1}:\cdots :X_{4}],[Y_{0}:Y_{1}:\cdots :Y_{3}]) &\mapsto [X_{0}Y_{0}:X_{0}Y_{1}:\cdots :X_{i}Y_{j}:\cdots :X_{4}Y_{3}],
\end{split} 
\end{equation}
(the $X_i Y_j$ are arranged in lexicographical order.) Finally, we embed the entire family $\YY \subset B_\delta(0) \times \P^{19} \subset \P^{1} \times \P^{19}$ into a large projective space $\P^{N}$ where $N = 2 \cdot 20 - 1 = 39$. We want to consider the restriction of the Fubini-Study form to $\YY$. Near each $A_5$-singularity of the central fiber, there are local analytic coordinates $z_j$ on $\P^N$ such that image of $\YY$ is defined by the equations
\begin{equation}
\left\{f_{A_5}(z_1,z_2,z_3,z_4) = t \: , \: z_j =0 \: \forall j \geq 5\right\}.
\end{equation}
By the same argument as in the proof of \cite[Lemma 1.7]{MR1978046}, there is an isotopy $\omega_t$ of K\"{a}hler forms on $\P^N$, starting with the Fubini-Study form $\omega_0 = \omega_{FS}$ and compactly supported near each singular point, which yields a form $\omega_1$ which in our analytic coordinates is exactly the standard form $\omega_{\C^N}$ in a small neighbourhood of the singularities. Most importantly, since the isotopy is through K\"{a}hler forms they are all non-degenerate on the image of $\YY$ (because $\YY$ is complex submanifold of $\P^N$) --so we define can restrict $\omega_1$ to $\YY$ and get a symplectic manifold which near $(0,q_1)$ and $(0,q_2)$ looks like the total space of a Milnor fibration.  

\begin{corollary} 
\label{cor:vanishingcycles2}
In a general fiber, there exist vanishing cycles $V'_1,\ldots,V'_5$ and $V''_1,\ldots,V''_5$ and the monodromy is the product of the total monodromy of two $A_5$-singularities.
\end{corollary}

\begin{remark}
We note that while the $\mu^2$ is defined over $\Gamma$, the parametrized quantum Massey product can actually be defined over the cone of curves. To utilize this property, it would be advantageous to choose minimal Morse functions (or at least ato understand the differentials explicitly) as we will do late on.
\end{remark}


\subsection{Remarks on birational models} \label{subsec:remarkonbirationalmodels}
The goal of this subsection is to make the construction seem a little less ad-hoc by placing it in a more general context. \\

Let $\MM_4$ be the moduli space of genus four curves. Let $\cMM_4$ denote the Deligne-Mumford moduli space of semistable curves. This is a proper smooth stack over $\mathrm{Spec}(\Z)$. The \textbf{boundary divisor} $\delta$ is defined to be the complement of $\MM_4$ in
$\cMM_4$; it consists of three irreducible components:
\begin{equation}
\cMM_4 \setminus \MM_4 = \delta_0 \cup \delta_1 \cup \delta_2.
\end{equation}
The generic point of $\delta_0$ corresponds to a curve with a single node whose normalization is connected. The general point of $\delta_i$, $i=1,2$ corresponds to a curve that has a single node whose normalization consists of two smooth curves with genus $i$ and $4 - i$ respectively. Let $\HH_4$ denote the hyperelliptic locus in $\MM_4$ and $\MM^{nh}_4$ the open substack parameterizing nonhyperelliptic curves. Let $Z$ be the \textbf{Petri divisor} in $\MM^{nh}_4$, i.e., the locus in $\MM^{nh}_4$ consists of non-hyperelliptic curves whose canonical embedding lies on a singular quadric surface in $\P^3$. We denote the closure of $Z$ in $\cMM_4$ as $\bar{Z}$. Note that $\bar{Z}$ contains the hyperelliptic locus. We denote the universal genus four curve as $\pi : \MM_{4,1} \to \MM_4$ and we set $\omega = \omega_{\MM_{4,1} / \MM_4}$ to be the relative dualizing sheaf.

\begin{proof}[Proof of Proposition \ref{prop:blowup1}]
Choose a generic point on $\delta_2$, and a generic 1-parameter rational curve in $\bar{\MM}_4$ that passes through it. Pulling back the universal family gives us a family of genus $4$ curves over $\P^1$, with only finitely many non-smooth fibers. A dimension count shows that it does not meet $\HH_4$ and only meets $\bar{Z}$ at finitely many points. Now we pass to the complex-analytic world via GAGA and take a small disc $\Delta$ centered at our original nodal curve. If we consider the canonical embedding via the ample line bundle $\omega_{\mathcal{C}/\Delta^*}$ and denote the bundle of projective spaces as $\mathcal{P} \iso \P^3 \times \Delta^*$, then a global version of the proof of Lemma \ref{thm:complete23_3} gives a quadric bundle $\overline{\mathcal{Q}}^* \subset \overline{\mathcal{P}}^*$. Since the entire family $\overline{\mathcal{C}}^*$ (including the singular fiber) does not meet the Petri divisor, we can trivialize 
\begin{equation}
\overline{\mathcal{Q}}^* \iso Q \times \Delta^* 
\end{equation}
where $Q = \P^1 \times \P^1 \subset \P^3$. Thus the resulting family $\overline{\mathcal{C}}^* \to \Delta^*$ has all the required properties. 
\end{proof}

\begin{remark}
Note that $\Delta^* \sim S^1$, so there are no non-trivial complex vector bundles. Since $\Delta^*$ is a \emph{Stein manifold} (like any non-compact Riemann surface), the Oka-Grauert-Gromov principle tells us the same should be true for holomorphic bundles. Thus the projective bundle 
\begin{equation}
\P_{\Delta^\ast}(\textbf{R}^0 (\pi_{\overline{\familyEmbeddedCurves}^\ast})_* \omega_{\overline{\familyEmbeddedCurves}^\ast/\Delta^\ast}) 
\end{equation}
which we called $\overline{\mathcal{P}}^*$ is actually holomorphically trivial as well. 
\end{remark}
There is a well-developed theory of canonical models for singular curves (see \cite{MR2481842} for a modern account), but the main obstacle to understanding the monodromy $\phi : C \to C$ comes from the fact that the canonical embedding cannot be extended to the central fiber of $\overline{\mathcal{C}} \to \Delta$. Thus we turn to studying alternate compactifications of the moduli of curves, hoping to replace $\delta_2$ with another divisor. The Hassett-Keel program outlines a principle that suggests we look at spaces of the form

\begin{equation}
\ModuliofStableCurves_g(\alpha):=\mathrm{Proj} \left( \bigoplus_n H^0(\ModuliofStableCurves_g,n(K_{\ModuliofStableCurves_g}+\alpha \delta) \right),
\end{equation}

where $\delta$ is the boundary divisor in $\overline{M}_g$, $K_{\ModuliofStableCurves_g}$ is the canonical divisor and $\alpha \in [0,1] \cap \Q$. One has $\ModuliofStableCurves_g(1)=\ModuliofStableCurves_g$ and $\ModuliofStableCurves_g(0)$ equal to the canonical model of $\ModuliofStableCurves_g$ for $g \gg 0$ (see \cite{MR2483934} for a survey with many references to the original papers.) The general prediction is that as $\alpha$ decreases, the various tail loci are replaced by loci of more complicated singularities. For $g=4$, much of this conjectural picture has already been proven (thanks to the works of: \cite{MR3229767,MR2561195,MR2895186, MR3263667,MR3006172}). In particular, we have a detailed description of $\ModuliofStableCurves_4(\frac{5}{9})$ where $\delta_2$ is replaced by the maximal $A_5$-curve. This model also affords a modular GIT-interpretation connected to the ball quotient model for cubic 3-folds (introduced in \cite{MR1949641}.) This is the background to the construction in the previous Sections.

\section{The toy model: Morse $A_\infty$-algebras} \label{sec:toymodel}

This section provides summary of Morse theory, both in order to fix notation, and because Morse theory serves as a ''toy model" for many of the paradigm that would play a pivotal rule in our later constructions. We will mostly follow the treatment in \cite{MR3084244} and \cite{MR2786590} (although we change the notations a bit). For more background see:  
\cite{MR0163331}, \cite{MR1001450},\cite{MR1239174},\cite{MR1362827},\cite{MR1707327},\cite{MR1929325} (to name just a few of the many classical references).

\subsection{Handlebodies and Morse functions}
Let $X$ be a manifold of (real) dimension $n$. 

\begin{definition}
An \textbf{$n$-dimensional $k$-handle} $h_k$ is a copy of $D^k \times D^{n-k}$ attached to the boundary of an $n$-manifold $X$ along $\partial D^k \times D^{n-k}$ by an embedding $\phi : \partial D^k \times D^{n-k} \to \partial X$.
\end{definition}

There is a canonical way to smooth corners, so we will interpret $X \cup_\phi h_k$ again as a smooth n-manifold. We will call $D^k \times \left\{0\right\}$ the \textbf{core} of the handle, $\left\{0\right\} \times D^{n-k}$ the \textbf{cocore}, $\phi$
the \textbf{attaching map}, $\partial D^k \times D^{n-k}$ (or its image) the \textbf{attaching region}, $\partial D^k \times \left\{0\right\}$ (or its image) the \textbf{attaching sphere} and $\left\{0\right\} \times D^{n-k}$ the \textbf{belt sphere}. 

Note that there is a deformation retraction of $X \cup_\phi h_k$ onto $X \cup (\text{core of }h_k)$ (where the attaching map is the restriction of $\phi$ to the core), so up to homotopy, attaching a $k$-handle is the same as attaching a $k$-cell. The number $k$ is called the index of the handle. 

\begin{definition}
Let $X$ be a compact $n$-manifold with boundary $\partial X$ decomposed as a disjoint union $\partial_- X \cup \partial_+ X$ of two compact submanifolds (either of which may be empty). If $X$ is oriented, orient $\partial_\pm X$ so that $\partial X = \partial_+ X \coprod \overline{\partial_- X}$ in the boundary orientation. We will call such $(X,\partial_-X)$ a compact pair. A \textbf{handle decomposition} of $X$ (relative to $\partial_- X$) is an identification of $X$ with a manifold obtained from $I \times \partial_- X$ by attaching handles, such that $\partial_- X$ corresponds to $\left\{0\right\} \times \partial_- X$ in the obvious way. A manifold $X$ with a given handle decomposition is called a relative handlebody built on $\partial_- X$, or if $\partial_- X = 0$ it is called a handlebody. 
\end{definition}

It is well known that: \vspace{0.5em}
\begin{itemize}
\item
Any handle decomposition of a compact pair $(X, \partial_- X)$ as above can be rearranged (by isotoping attaching maps) so that the handles are attached in order of increasing index. Handles of the same index can be attached in any order (or simultaneously). \vspace{0.5em}
\item
A $(k-1)$-handle $h_{k-1}$ and a k-handle $h_k$ ($1\leq k \leq n$) can be cancelled, provided that the attaching sphere of $h_k$ intersects the belt sphere of $h_{k-1}$ transversely at a single point. \vspace{0.5em}
\end{itemize}

\begin{definition}
A \textbf{Morse function} $f : X \to \R$ is a smooth function with nondegenerate critical points. That is, at each critical point $p \in Crit(f)$, the Hessian $D^2 f|_p : T_p X \times T_p X \to \R$ is a nondegenerate (symmetric) bilinear form. The dimension of the negative eigenspaces of $D^2 f|_p$ at $p$ is called the \textbf{Morse index} $|p| \in \N$. 
\end{definition}
\begin{morselemma} Given a critical point $p$ of a Morse function $f : X \to \R$, there exists an open neighbourhood $p \in \tilde{U}(p) \subseteq X$ and a local chart $\phi_p : B_\delta^{n-|p|} \times B_{\delta}^{|p|} \to \tilde{U}(p)$ that bring $f$ into the \textbf{normal form} 
\begin{equation}
\begin{split}
\phi^*_p f(x_1,\ldots,x_n) &= f(p) + \frac{1}{2}(x_1^2 + \ldots + x^2_{n-|p|}) - \frac{1}{2}(x^2_{n-|p|+1} + \ldots + x^2_n).
\end{split}  
\end{equation}
\end{morselemma}

Every smooth, compact manifold pair $(X,\partial_-X)$ as above admits a handle decomposition, he basic idea is that any 
smooth function $f: X \to [0,1]$ with $f^{-1}(0) = \partial_- X$ and $f^{-1}(1) = \partial_+X$ can be perturbed into a Morse function, and then the normal form Lemma shows that given a Morse function $f : X \to \R$, the sublevel sets of $f$ provide a decomposition of $X$ in terms of handle attachments. 

\begin{definition} Any handle decomposition on a compact pair $(X, \partial_- X)$ determines a dual handle decomposition on $(X, \overline{\partial_+ X})$ as follows. 

\begin{itemize}
\item
We glue a collar $I \times \overline{\partial_+ X}$ to $X$, and remove the collar $I \times \partial_- X$ on which the handlebody is built. 
\item
Each $k$-handle $D^k \times D^{n-k}$ can be interpreted as an $(n-k)$-handle glued to the part of $X$ \emph{above it} (reversing the roles of core and cocore). 
\end{itemize}
We will often say that we flipped the handle decomposition ''upside down". In terms of Morse theory, we are replacing a Morse function $f$ by $1-f$. 
\end{definition}

The connection between handlebodies and Morse theory would play an important rule in Section \ref{sec:compute4}, where we will use a variant called embedded Morse theory to construct a very specific of the 7-dimensional mapping torus. However, for the rest of the paper, we are mostly interested in studying Morse functions from a dynamical systems point of view. For that we need to fix an auxiliary choice of Riemannian metric $g : TX \otimes TX \to \R$ (we will also usually assume that $X$ is closed, or the function satisfies some convexity condition which ensures that gradient trajectories can not see the boundary). This enables us to extract topological information about $X$ by studying the flow lines of the gradient vector field. More precisely, let
\begin{equation}
\Psi : \R \times X \to X \: \: , \: \: (s,x) \mapsto \Psi_s(x)
\end{equation}
denote the group of diffeomorphism generated by negative gradient flow equation 
\begin{equation} \label{eq:morseflow}
\frac{d}{dt} \Psi_s(x) = -\nabla_g f|_{\Psi_s(x)} \: \: , \: \: \Psi_0(x)=x. 
\end{equation}
Then we can consider the \textbf{unstable} (descending) and \textbf{stable} (ascending) manifolds for each critical point $p \in Crit(f)$: 
\begin{equation}
W^u_p = \left\{x \in X \left|\right. \lim_{s \to -\infty} \Psi_s(x) = p\right\} \: \: , \: \: W^s_p = \left\{x \in X \left|\right. \lim_{s \to +\infty} \Psi_s(x) = p\right\}. 
\end{equation}
These are smooth manifolds of dimension $|p|$ and $n - |p|$, respectively; see \cite[Chapter 6]{MR1239174}. We denote $\overline{W^u_p}$ and $\overline{W^s_p}$ for the closure in $X$.  

\begin{definition}
The \textbf{degree} of $p$, denoted $deg(p)$, is $n-|p| = \dim(W^s_p)$.
\end{definition}

The idea is to use $(f,g)$ to construct a cochain complex that computes $H^*(X;\R)$. The complex $CM^*(f)$ is freely generated by the critical points, and is equipped with a differential $\mu^1$ that counts intersections of stable and unstable manifolds, as well as higher operations. But to do this it is better to focus our attention to special Morse functions that satisfy the following extra two conditions:

\begin{definition} \label{eq:euclideanmorsesmale}
The pair $(f,g)$ is called 
\begin{enumerate}[label=(\roman*)]
\item \textbf{Euclidean} if for each $p \in Crit(f)$ there exists an open neighbourhood $p \in \tilde{U}(p) \subseteq X$ and a local chart $\phi_p : B_\delta^{n-|p|} \times B_{\delta}^{|p|} \to \tilde{U}(p)$ such that both $f$ and $g$ are in normal form:
\begin{align*}
\phi^*_p f(x_1,\ldots,x_n) &= f(p) + \frac{1}{2}(x_1^2 + \ldots + x^2_{n-|p|}) - \frac{1}{2}(x^2_{n-|p|+1} + \ldots + x^2_n) ,\\
\phi^*_p g &= dx_1^{\otimes 2} + \ldots  dx_n^{\otimes 2}. 
\end{align*}
\item
\textbf{Morse-Smale} if for every pair of critical points $p,q \in Crit(f)$, the unstable and stable manifolds intersect each other transversally, $W_p^u \pitchfork W_q^s$.
\end{enumerate}
\end{definition}


This is hardly a restriction at all, because 
\begin{theorem}
Given any Morse function and metric $(f,g)$, there exist $L^2$-small perturbations of the metric on annuli around the critical points that yield Morse-Smale pairs. In particular, given a metric of normal form near the critical points, such a perturbation yields a Euclidean Morse-Smale pair.
\end{theorem}
\begin{theorem}
The gradient flow $\Psi_s$ of any Morse-Smale pair $(f,g)$ is topologically conjugate to the gradient flow $\Psi'_s$ of some Euclidean Morse-Smale pair $(f',g')$. That is, there exists a homeomorphism $h : X \to X$ such that $h \circ \Psi_s = \Psi'_s \circ h$. 
\end{theorem}

\subsection{Spaces of gradient trajectories} \label{subsec:toymodeltrajectoryspace}
It will be important for us to isolate the boundary component given by zero length trajectories from all other boundary components given by broken trajectories. Unfortunantly, there are broken trajectories with endpoints near a critical point arbitrarily close in the Hausdorff topology to the zero length trajectory at the critical point. In order to separate those boundary components \cite{MR3084244} uses the natural blowup construction of including the length of a trajectory in the Morse trajectory space, thus introducing a constant trajectory at the critical point for every non-negative length $L \in [0,+\infty)$, converging to a broken trajectory with domains $[0, \infty),(-\infty, 0]$ as $L \to \infty$ (a similar framework was used in \cite{MR2388043} to construct gluing maps for Seiberg-Witten Floer theory). 

\begin{definition} 
We introduce the different versions of \textbf{Morse trajectory spaces} for a general  Morse-Smale pair $(f,g)$ on $X$.
\begin{description}
\item[Infinite] 
For two critical points $p_- \neq p_+$, the space of unbroken Morse flow lines is the space of parametrized gradient flow lines $\gamma : \R \to X$ modulo time shift,
\begin{align}
\Morse(p_-,p_+) &:= \bigl\{ \gamma:\R\to X \,\big|\, \dot{\gamma}=-\nabla f(\gamma) ,
\lim_{s\to\pm\infty}\gamma(s)=p_\pm  \bigr\} / \R \\
&\simeq \bigl(W^u_{p_-}\cap W^s_{p_+}\bigr) / \R
\;\simeq\; W^u_{p_-}\cap W^s_{p_+} \cap f^{-1}(c) .
\end{align}
Denote the open interval in $\R$ whose endpoints are $f(p_-)$ and $f(p_+)$ as $I$. The space of unbroken Morse flow lines is canonically identified with the intersection of unstable and stable manifold modulo the $\R$-action given by the flows, or their intersection with a level set $f^{-1}(c)$ for any regular value $c \in I$. In either formulation, they carry canonical smooth structures, see e.g. \cite[Section 2.4.1]{MR1239174}. 

We will consider the constant trajectories at a critical point as part of a larger trajectory space below, hence it is our convention to set 
\begin{equation}
\Morse(p,p)=\phi
\end{equation}
for any $p \in Crit(f)$. 

\item[Half-infinite] For open subsets $U_-,U_+ \subset X$ and critical points $p_-,p_+ \in Crit(f)$ the spaces of half infinite flow lines
\begin{align}
\Morse(U_-,p_+) &:= \bigl\{ \gamma: [0,\infty) \to X \,\big|\, \dot{\gamma}=-\nabla f(\gamma), \gamma(0)\in U_-, \lim_{s\to\infty}\gamma(s)=p_+  \bigr\} \simeq W^s_{p_+}\cap U_-, \\
\Morse(p_-,U_+) &:= \bigl\{ \gamma: (-\infty,0] \to X \,\big|\, \dot{\gamma}=-\nabla f(\gamma), \lim_{s\to-\infty}\gamma(s)=p_-, \gamma(0)\in U_+  \bigr\} \simeq W^u_{p_-}\cap U_+
\end{align}
\item[Finite] Let  $U_-,U_+$ be open subsets of $X$. The space of finite unbroken flow lines, 
\begin{align}
\Morse(U_-,U_+) &:= \bigl\{ \gamma: [0,L] \to X \,\big|\, L\in [0,\infty), \dot{\gamma}=-\nabla f(\gamma), \gamma(0)\in U_-, \gamma(L)\in U_+  \bigr\} \\
&\simeq\; {\textstyle \bigcup_{L\in [0,\infty)}} U_- \cap \Psi_L^{-1}(U_+)
\;=\; \bigl( [0,\infty)\times  U_- \bigr) \cap \Psi^{-1}(U_+)
\end{align}
is identified with an open subset of $\Morse(X,X) \subseteq [0,\infty) \times X$ since the flow map $\Psi$ is continuous. Hence it naturally is a smooth manifold with boundary given by constant flow lines.
\end{description}
\end{definition}

From the smooth spaces of unbroken flow lines we obtain topological spaces of broken flow lines as follows: To unify notation we denote by $\UU_\pm\subset X$ a set that is either open $\UU_\pm=U_\pm$ or a set consisting of a single critical point $\UU_\pm=p_\pm$.
\begin{definition}
For two such subsets $\UU_\pm\subset X$ (of same or different type) we define the set of $k$-fold broken flow lines (also called the {\em $k$-stratum}) by
\begin{equation}
\cMorse(\UU_-,\UU_+)_k :=
\bigcup_{(p_1\ldots p_k)\in{\rm Critseq}(f,\UU_-,\UU_+)}
\cMorse(\UU_-,p_1)\times \cMorse(p_1,p_2) \ldots \times \cMorse(p_k,\UU_+),
\end{equation}
Here and throughout we use the notation of \textbf{critical point sequences} between $\UU_\pm$
\[
{\rm Critseq}(f,\UU_-,\UU_+) := \left\{ (p_1,\ldots, p_k) \left|
\begin{array}{l}
k\in\N_0, \;p_1,\ldots, p_k \in {\rm Crit}(f) , \\
 \cMorse(\UU_-,p_1), \cMorse(p_1,p_2) \ldots , \cMorse(p_k,\UU_+)\neq\emptyset
 \end{array}
 \right.\right\} .
\]
\end{definition}
To simplify notation we identify $\ul{p}\in{\rm Critseq}(f,\UU_-,\UU_+)$ with the tuple $\ul{p}=(\UU_-,p_1,\ldots, p_k,\UU_+)$, and denote $p_0:=\UU_-$, $p_{k+1}:=\UU_+$.
Critical point sequences form a finite set since they have to decrease in function value.
For $k=0$ we only have the empty critical point sequence and hence $\cMorse(\UU_-,\UU_+)_0 = \Morse(\UU_-,\UU_+)$.
Now,
\begin{definition}
The \textbf{Morse trajectory space} is the space of all \textbf{generalized trajectories},
\begin{align*}
\cMorse(\UU_-,\UU_+) &:= {\textstyle \bigcup_{k\in\N_0}}
\cMorse(\UU_-,\UU_+)_k .
\end{align*}
\end{definition}
In the following we denote broken flow lines by $\ul{\gamma}=(\gamma_0,\gamma_1\ldots, \gamma_k)\in \cMorse(\UU_-,\UU_+)_k$ and also write $\ul{\gamma}=\gamma_0 \in \cMorse(\UU_-,\UU_+)_0$ for the unbroken flow lines.
Note here that, by slight abuse of notation, we write $\gamma_i$ instead of $[\gamma_i]$ for the unparametrized flow lines in $\cMorse(p_i,p_{i+1})$. If $\UU_-$ resp. $\UU_+$ is a critical point, then $\gamma_0$ resp.\ $\gamma_k$ is an unparametrized flow line as well, otherwise it is defined on a half interval and hence parametrized.

With this notation we can make the following definitions.

\begin{definition} \label{def:evH}
We define the \textbf{evaluation at endpoints} maps
\begin{equation}\label{eq:evalintro}
\begin{split}
\ev_-&:\cMorse(X,p_+) \to X, \\
\ev_+&:\cMorse(p_-,X) \to X, \\
\ev_-\times \ev_+&:\cMorse(X,X) \to X\times X,
\end{split}
\end{equation}
by $\ev_-(\gamma_0,\ldots,\gamma_k)=\gamma_0(0)$ for any $k\in\N_0$,
by $\ev_+(\gamma_0,\ldots,\gamma_k)=\gamma_k(0)$ for $k\geq 1$, and by
$\ev_+(\gamma_0:[0,L]\to X)=\gamma_0(L)$ for a single trajectory $k=0$. 
\end{definition}
Let 
\begin{equation}
H\subset X 
\end{equation}
be a submanifold of codimension $1$ whose closure is transverse to $\nabla f$ (i.e., $\nabla f$ is nowhere tangent to $H$), and such that $\Psi_{\R_+}(H) \cap H = \emptyset$, where $\R_+=(0,\infty)$. Then $\Psi_{\R_-}(H),\Psi_{\R_+}(H)$ are open subsets in $X$. A particular example of this are level sets $H = f^{-1}(c)$, $c$ is a regular value for $f$. 
\begin{definition}
Given such an $H$, we can define the \textbf{evaluation maps at a local slices to the flow} 
\begin{equation}\label{eq:evalH}
\ev_H: \cMorse(\UU_-,\UU_+; \Psi_{\R_-}(H) , \Psi_{\R_+}(H) ) \to H ,
\quad \ul{\gamma} \mapsto \im\ul{\gamma}\cap H .
\end{equation}
for all trajectories that intersect $H$ but do not end there.
\end{definition}


Next, we define a metric on the Morse trajectory spaces
\begin{equation}
d_{\cMorse}(\ul{\gamma},\ul{\gamma}'):= d_\text{Hausdorff}(\overline{\im\ul{\gamma}},\overline{\im\ul{\gamma}'})
+  \big|\ell(\ul{\gamma})- \ell(\ul{\gamma}') \big|
\qquad\text{for}\; \ul{\gamma},\ul{\gamma}'\in\cMorse(\UU_-,\UU_+) ,
\end{equation}
by the Hausdorff distance and the renormalized length
\begin{equation}\label{length}
\ell : \cMorse(\UU_-,\UU_+) \to [0,1], \qquad
\ul{\gamma} \mapsto
\begin{cases}
\tfrac{L}{1+L}  &; \ul{\gamma} = \bigl(\gamma:[0,L]\to X\bigr) ,\\
1 &; \text{otherwise}.
\end{cases}
\end{equation}
Here the image of a generalized trajectory $\ul{\gamma}=(\gamma_0,\ldots,\gamma_k)$ is the union of the images in $X$ of all constituting flow lines (which is independent of the parametrization),
$$
\im\ul{\gamma}:= \im\gamma_0 \cup \ldots \cup \im\gamma_k \subset X .
$$
The closure $\overline{\im\ul{\gamma}}$ contains in addition the critical points $\lim_{s\to\infty}\gamma_{j-1}=\lim_{s\to-\infty}\gamma_{j}$ for $j=1\ldots k$ as well as $\lim_{s\to-\infty}\gamma_0$ resp.\ $\lim_{s\to\infty}\gamma_{k}$ in case $\UU_-$ resp.\ $\UU_+$ is a single critical point, and hence $\overline{\im\ul{\gamma}}$ is a compact subset of $X$.
We use closures since the Hausdorff distance
$$
d_{\text{Hausdorff}}(V,W) = \max\bigl\{  \adjustlimits \sup_{v \in V} \inf_{w \in W} d_X(v,w) ,  \adjustlimits \sup_{w \in W} \inf_{v \in V}d_X(w,v)  \bigr\},
$$
is a metric on the set of non-empty compact subsets of $X$.

\begin{lemma} \label{thm:cont eval}
The evaluation maps \eqref{eq:evalintro} and \eqref{eq:evalH} are continuous with respect to the Hausdorff distance. When restricted to the subsets of unbroken trajectories
$\Morse(p_-,p_+)$, $\Morse(X,p_+)$, $\Morse(p_-,X)$, resp.\ $\Morse(X,X)$,
the evaluation maps are smooth.
In fact, $\ev_H: \Morse(p_-,p_+)\supset {\rm dom}(\ev_H) \to H$, $\ev_-: \Morse(X,p_+) \to X$,  $\ev_+:\Morse(p_-,X)\to X$, and $\ev_-\times \ev_+ :\Morse(X,X)^* \to X \times X$ are embeddings, where $\Morse(X,X)^*$ denotes the nonconstant trajectories.
\end{lemma}
\begin{proof}
See page 14 in \cite{MR3084244}. 
\end{proof}

\begin{lemma}
The identifications of the spaces of unbroken flow lines as above 
\begin{align*}
\Morse(p_-,p_+) &\simeq W^u_{p_-}\cap W^s_{p_+} \cap f^{-1}(c), \\
\Morse(X,p_+) &\simeq W^s_{p_+}, \Morse(p_-,X) \simeq W^u_{p_-},  \\
\Morse(X,X) &\simeq [0,\infty)\times X,
\end{align*}
are homeomorphisms with respect to the metric $d_{\cMorse}$. 
\end{lemma}
\begin{proof}
This follows from the continuity of the flow in one direction and from the continuity of the evaluation maps in the other (as well as the continuity of the length conversion $L\mapsto\frac{L}{1+L}$ in the finite case). 
\end{proof}

\begin{remark} \label{rmk:metric}
Even though they induce the same topology, on $\Morse(p_-,p_+)$, $\Morse(X,p_+)$, and $\Morse(p_-,X)$ the Hausdorff distance $d_{\cMorse}$ is not equivalent to the distance on $W^u_{p_-}\cap W^s_{p_+} \cap f^{-1}(c) $ resp.\ $W^s_{p_+}$ resp.\ $W^u_{p_-}$. A counterexample for $\Morse(S^1,p_+)$ is a Morse function with one maximum and one minimum at $p_+$. Then consider Morse trajectories starting near the maximum. These initial points can be arbitrarily close, but if they lie on different sides of the maximum then the associated Morse trajectories have large Hausdorff distance. 
Similarly, on $\cMorse(X,X)$ the distance $d_{\cMorse}(\gamma:[0,L]\to X, \gamma':[0,L']\to X)$ is not equivalent to the distance $d_X(\gamma(0),\gamma'(0))+|L-L'|$ on $[0,\infty)\times X$ (but they still generate the same topology). 
\end{remark}

The Morse trajectory spaces for open sets $U_\pm \subset X$ are open subsets $\cMorse(U_-,p_+)= \ev_-^{-1}(U_-)$, $\cMorse(p_-,U_+)= \ev_+^{-1}(U_+)$, $\cMorse(U_-,U_+)= \ev_-^{-1}(U_-)\cap \ev_+^{-1}(U_+)$ of the Morse trajectory spaces for $U_\pm=X$. So from now on we can restrict our discussion to the Morse trajectory spaces $\cMorse(\UU_-,\UU_+)$ for $\UU_\pm= X$ or $\UU_\pm=p_\pm\in \crit(f)$.

The following is Theorem 2.3 in \cite{MR3084244}. 

\begin{theorem} \label{thm:corner}
Let $(f,g)$ be a Morse-Smale pair and let $\UU_-,\UU_+$ denote $X$ or a critical point ${\rm Crit}(f)$. Then: 
\begin{enumerate}
\item
The moduli space 
\begin{align*}
(\cMorse(\UU_-,\UU_+),d_{\cMorse})
\end{align*}
is a compact, separable metric space and can be equipped with the structure of a smooth manifold with corners. Its $k$-stratum is $\cMorse(\UU_-,\UU_+)_k$, with one additional $1$-stratum $\{0\}\times X$ given by the length $0$ trajectories in case $\UU_-=\UU_+=X$. 
\item
If the pair $(f,g)$ is Euclidean, the smooth structure on each $\cMorse(\UU_-,\UU_+)$ will be naturally given by the flow time and evaluation maps at ends \eqref{eq:evalintro} and regular level sets, as detailed in \cite[Section 4.3]{MR3084244}.
\end{enumerate}
\noproof
\end{theorem}

This ''folk theorem" can be deduced from much stronger constructions of global charts for Euclidean Morse-Smale pairs. See \cite[Theorem 2.6 and 2.7]{MR3084244} and the proof thereof.

\subsection{The Morse-Smale-Witten complex}
Assume that we have fixed a Morse-Smale pair $(f,g)$ on $X$. 

\begin{definition}
The orientation line at $p \in \crit(f)$ is the abelian group generated by the two orientations of $W_p^s$, with the relation that the two opposite orientations satisfy
\begin{equation}
[\Omega_p] + [-\Omega_p] = 0.
\end{equation}
We denote this group by $|\orn_p|$.
\end{definition}

Note that if $p_1,p_0$ are two critical points such that $\deg(p_1) = \deg(p_0) +1$, then $\Morse(p_1,p_0)$ is zero dimensional. Counting rigid morse trajectories with sign induces a canonical map (the precise definition of which is given in Appendix \ref{sec:signs}) on the orientation lines, denoted

\begin{equation}
\mu : |\orn_{p_1}| \to |\orn_{p_0}|.
\end{equation}

\begin{definition}
Given a Euclidean Morse-Smale pair $(f,g)$, we define the \textbf{Morse complex} as
\begin{equation}
CM^i(f,g) := \bigoplus_{deg(p)=i} |\orn_p|,
\end{equation}

The differential is defined to be the sum of $\mu$-terms multiplied by a global sign:
\begin{equation}
\mu^1(|\orn_{p_1}|) = \sum_{\deg(p_0) = \deg(p_1)+1} (-1)^n \mu(|\orn_{p_0}|).
\end{equation}
\end{definition}

\begin{lemma} $\mu^1 \circ \mu^1=0$.
\end{lemma}
\begin{proof} The signed count of boundary points in a one dimensional compact 1-manifold is zero. 
\end{proof}

Hence, the cohomology 
\begin{equation}
HM^k(f, g) = H^\bullet(CM^k(f,g), \mu^1) 
\end{equation}
is well-defined as the quotient of the module of \textbf{Morse-cycles}
\begin{equation}
ZM^k(f, g) = \left\{ a = \sum_i a_i \cdot |\orn_{p_i}| \: \bigg| \: \mu^1(a) = 0 \right\}.
\end{equation}
by the boundaries 
\begin{equation}
BM^k(f, g) = im(\mu^1).
\end{equation}

By virture of Conley's continuation principle, we have the following invariance result.
\begin{theorem} \label{thm:morsecontinuation}
Let $(f_0,g_0)$ and $(f_1,g_1)$ be two Morse-Smale pairs on $X$. Then there exists a
canonical homomorphism 
\begin{equation} \label{eq:conleycontinuation}
\Phi_{01} : H^\bullet(f_0,g_0) \to H^\bullet(f_1,g_1)
\end{equation}
such that if $(f_2,g_2)$ is another Morse-Smale pair, then 
\begin{equation}
\Phi_{12} \circ \Phi_{01} = \Phi_{02} \: , \: \Phi_{00} = id, 
\end{equation}
In particular, every such homomorphism is in fact an isomorphism. 
\end{theorem}
Let us recall the construction of $\Phi_{01}$ from \cite{MR1239174}. Given the Morse-Smale pairs $(f_0,g_0)$ and $(f_1,g_1)$, we choose an asymptotically constant homotopy over the real line, denoted $(f_s,g_s)_{s \in \R}$. This gives rise to the trajectory spaces which consist of parametrized flow lines $\gamma$ satisfying
\begin{equation}
\dot{\gamma}=-\nabla_{g_s} f_s(\gamma) 
\end{equation}
with suitable asymptotics 
\begin{equation}
\begin{split}
\lim_{t \to + \infty} \gamma(t) &= p_0,\\
\lim_{t \to - \infty} \gamma(t) &= p_1. 
\end{split}
\end{equation}
For a generic choice of the homotopy, these spaces are finite dimensional manifolds of dimension 
\begin{equation}
\deg(p_0) - \deg(p_1)
\end{equation}
and compact in dimension zero. As in the definition of the boundary operator, we
define
\begin{equation}
\Phi_{01} : CM^\bullet(f_0,g_0) \to CM^\bullet(f_1,g_1)
\end{equation}
by counting solutions with the apropriate signs. Moreover, it is shown in \cite{MR1239174} that on the level of cohomology, the homomorphism $\Phi_{01}$ is independent of the choice of homotopy $(f_s, g_s)$.

\subsection{Differentiable correspondences} 
One reason that $A_\infty$-algebras are interesting is because they arise naturally in geometry when one looks at the combinatorics of the codimension one boundary of many classical moduli spaces. The connection between geometry and algebra is mediated via the theory of operads.

To formalize this, Fukaya introduced the following definition in \cite{MR2656945}: consider a diagram of smooth maps  
\begin{equation} \label{eq:differentiablecorrespondence}
\begin{tikzpicture} 
  \matrix (m) [matrix of math nodes,row sep=3em,column sep=4em,minimum width=2em]
  {
     & \OO_{k+1}  &\\
     \smash{M^k} & \smash{\overline{\MM}_{k+1}}  & \smash{X}\\};
  \path[->]
    (m-2-2) edge node [above] {$\pi_2$} (m-2-1)
		(m-2-2) edge node [above] {$\pi_1$} (m-2-3)
		(m-2-2) edge node [auto] {$\pi_0$} (m-1-2); 
\end{tikzpicture}
\end{equation}
where $X$ is a closed oriented manifold and $\overline{\MM}_{k+1}$ is a compact oriented manifold with boundary and corners (thought of as a moduli space of some geometric objects with evaluation maps to the manifold and a structure map to $\OO_{k+1}$, which is some version of the topological $A_\infty$-operad.) We assume
\begin{equation}
\dim(\overline{\MM}_{k+1}) = \dim(X) + \dim(\OO_{k+1}) = \dim(X) + k - 2.
\end{equation}
Then one can define a multilinear linear operation on $C^\bullet(X)$ via the following push-pull construction. Let $f_i : P_i \to X$ be ''chains" of dimension $\dim (X) - d_i$. We consider the fiber product
\begin{equation} \label{eq:fukaya1}
(\overline{\MM}_{k+1} )\: \tensor[_{\pi_1 \times}]{X}{_{f_1 \times \ldots \times f_k}} \: (P_1 \times \ldots \times P_k)
\end{equation}
Assuming transversality can be achieved, \eqref{eq:fukaya1} is a smooth manifold with boundary and corners. $\pi_2 = ev_0 :\overline{\MM}_{k+1} \to X$ induces a smooth map $ev_0$ from \eqref{eq:fukaya1}. Let
\begin{equation} \label{eq:fukaya2}
\mu^k(P_1 \times \ldots \times P_k) = (ev_0)_* \left((\overline{\MM}_{k+1} )\: \tensor[_{\pi_1 \times}]{X}{_{f_1 \times \ldots \times f_k}} \: (P_1 \times \ldots \times P_k)\right)
\end{equation}
This is a chain of dimension
\begin{equation}
\dim(X)+k-2+ \sum(\dim(X)-d_i) - k \dim(X) = \dim(X) - \sum_i \left(d_i  -(k-2)\right)
\end{equation}
Using Poincar\'{e} duality, we identify chains of dimension $\dim(X) - d$ on $X$ with cochains of degree $d$. Then \eqref{eq:fukaya2} induces a map of degree $k - 2$ on cochains. This is (after degree shift) a map with required degree. \\

Provided some natural compatibility conditions are met (see \cite[Definition 4]{MR2656945} for the precise formulation), we say that the diagram is a $A_\infty$-differentiable correspondence or just an \textbf{$A_\infty$-correspondence} for short, and 

\begin{theorem}[4 in \cite{MR2656945}] \label{thm:fukayacorrespondence}
If there is an $A_\infty$-correspondence on $X$ then there exists a cochain complex $C^\bullet(X;\Z)$ whose homology group is $H^\bullet(X;\Z)$ and such that $C^\bullet(X;\Z)$ has a structure of $A_\infty$-algebra (see \ref{subsec:ainftyalgebradefinition} for the algebraic definition.) 
\end{theorem}

The precise choice of chain complex is actually quite important here. In general, chain level intersection theory is still rather mysterious, but for Morse cochains it is immediate: we simply replace $f$ with $-f$, which has the effect of reversing the direction of the gradient flow. Thus, the Poincar\'{e} dual of an unstable manifold is the stable manifold and vice verse. It will also be analytically simpler to construct the manifolds \eqref{eq:fukaya1} directly, where the ''chains" enter as asympototic conditions (and correspond to critical points in the Morse complex.) \\

Finally, we remark that in a similar fashion, there are (more complicated) ''correspondences between $A_\infty$-correspondences" that define $A_\infty$-homomorphism and $A_\infty$-homotopies. See Definition 8, 11 and Theorems 6, 10 respectively in \cite{MR2656945} for details. 

\subsection{Metric trees} \label{subsec:metrictree}
Building upon our initial success with the differential, we want to look at higher operations on the Morse complex defined by $A_\infty$-correspondences. The first step to find a relevent incarnation of the chain operad of the Tamari-Stasheff polytope (in the notation of the previous subsection, this is the same as chains on $\OO_{k+1}$.) For Morse theory, this is the space of metric ribbon trees. This sequence of finite CW-complexes was constructed by Boardman-Vogt (\cite{MR0236922}, \cite{MR0420609}) as a special case of a more general procedure called \textit{the W-construction}. This construction comes equipped with a natural cubical decomposition, which would be useful to us later on. \\

Let $d \geq 2$.
\begin{definition}
The $d$-th {\bf Stasheff moduli space}, $\Stashefftree_{d}$, is the moduli space of all metrics on ribbon trees with $d+1$ external vertices all of which have infinite type. 
\end{definition}

The topology of $\Stashefftree_d$ is such that the operation of contracting the length of an internal edge connecting two vertices is continuous. We identify a tree $T \in \Stashefftree_{d}$ with an edge $e = (v_-,v_+)$ of length zero with the tree obtained by collapsing $e$ and merging the two vertices together. In \cite{MR0158400}, Stasheff showed that $\Stashefftree_d$ is homeomorphic to $\R^{d-2}$ (see also Fukaya and Oh in \cite{MR1480992}). \\
\begin{definition} We denote the \textbf{universal family} over $\Stashefftree_d$ by 
\begin{equation}
\mathcal{U}_d \to \Stashefftree_d.
\end{equation}
\end{definition} 

\begin{definition}
There is natural compactification $\Stashefftree_d \subset \cStashefftree_d$, called the \textbf{Stasheff polyhedron}, which is simply obtained by adding \textbf{singular Stasheff trees}: stable singular ribbon trees with $d+1$ external vertices (again have infinite type). The induced partial compactification of the universal family is denoted
\begin{equation}
\overline{\mathcal{U}}_d \to \cStashefftree_d.
\end{equation}
\end{definition} 

The topology is such that a sequence of trees with fixed topology having an edge whose length becomes unbounded converges to a singular tree with a bivalent vertex precisely at that edge. The partial composition operations on trees are continuous and yield collections of maps
\begin{equation} \label{eq:compose_stasheff}
\circ_j :  \cStashefftree_{d_1} \times  \cStashefftree_{d_2} \to  \cStashefftree_{d}
\end{equation}
whose image exactly covers the boundary of $\cStashefftree_{d}$.\\

Let $(T,g_T)$ be a singular metric ribbon tree. Note that every flag $f = (v,e) \in \Flag(T)$ determine a unique orientation and length preserving map
\begin{equation}
e \to R
\end{equation}
taking $v$ to the origin and $e$ to the appropriate segment on either side of the origin. 

\begin{definition}
Let $T = (T,g_T)$ be a singular Stasheff tree. The \textbf{combinatorical type} $\Upsilon_T$ is the singular ribbon tree obtained by forgetting the metric, but keeping track of: \vspace{0.5em}
\begin{itemize}
\item
The set of edges of combinatorially zero length. \vspace{0.5em}
\item
The set of edges of combinatorially infinite length. \vspace{0.5em}
\end{itemize}
The set of edges which are not of combinatorially zero or infinite length are called edges of combinatorially finite length. We denote these three subsets of edges as 
\begin{equation}
\Edge^0(T) \:,\: \Edge^\infty(T) \:, \:\Edge^{finite}(T). 
\end{equation}
respectively. We define the combinatorial type of a \textbf{singular Stasheff forest} $T$ (which is just a disjoint union of trees) as the disjoint union of the combinatorical types of its connected components.
\end{definition}

For later, we note that the operations we described previously can be also be seen in terms of morphism on the combinatorical types.

\begin{definition}
A morphism $f : \Upsilon_{T'} \to \Upsilon_T$ of combinatorical types is a graph morphism which is surjective on the set of vertices $\Vert(\Upsilon_{T'}) \to \Vert(\Upsilon_{T})$ obtained by combining the following elementary procedures: \vspace{0.5em}
\begin{enumerate}
\item
\textbf{Cutting at a bivalent vertex.} If there exists a bivalent vertex $v$ and two adjacent edges $e^- = (v^-,v),e^+ = (v,v^+)$, necessarily of combinatorially infinite length, we may remove the vertex $v$, and introduce two new external vertices instead. The resulting graph is disconnected, and we denote the connected component that contains $e^-$ (respectively $e^+$) as $\Upsilon_{T^-}$ (resp. $\Upsilon_{T^+}$). The disjoint union of both of them is the forest $\Upsilon_{T}$. The inverse operation (combining two trees by feeding the output of one with the j-th input of the other) is called \textbf{grafting}. 
 \vspace{0.5em}
\item
\textbf{Making an edge length finite.} In the case we have an edge $e = (v^-,v^+)$ which is either in $\Edge^0(T)$ or $\Edge^\infty(T)$. The map $f$ is a bijection of vertices and edges and just changes the combinatorical type of the edge to $\Edge^{finite}(T)$. \vspace{0.5em}
\end{enumerate}
\end{definition}
Operation (2) can be used to give a partial order on combinatorical types and obtain a corresponding stratification of $\cStashefftree_d$. In fact, the space $\cStashefftree_d$ has \emph{two} nice stratification. The other one is called the \textbf{cubical cell decomposition}. The cells 
\begin{equation} \label{eq:cubical}
\Stashefftree[T \to T^\infty] 
\end{equation}

are defined in the following way: consider all the surjective maps of ribbon trees $T \to T^\infty$, where both $T$ and $T^\infty$ are ribbon trees with $d$ inputs, and such that the inverse image of every vertex is a connected tree. The stratum $\Stashefftree[T \to T^\infty]$ determined by such a map consists of metrics on such trees $T$ whose finite edges are those collapsed in $T^\infty$. In particular, the set of infinite edges in such a stratum is fixed, so the structure of a metric tree is determined by the lengths of the remaining edges and we have a canonical isomorphism

\begin{equation} 
\Stashefftree[T \to T^\infty] \cong (0,+\infty)^{\Edge(T) - \Edge(T^\infty)}, 
\end{equation}

which is just a product of copies of $(0,+\infty)$ indexed by edges $e \in \Edge(T)$ that get collapsed by $T \to T^\infty$. The closure of each cell is simply obtained by allowing edges to reach length $0$ or $\infty$, so:

\begin{equation} 
\cStashefftree[T \to T^\infty]  \cong [0,+\infty]^{\Edge(T) - \Edge(T^\infty)}. 
\end{equation}

and the cell structure on $\cStashefftree[T\to T^{\infty}] $ induced from its inclusion in $\cStashefftree_{d}$ is compatible with the natural cell structure on $[0,+\infty]^{\Edge(T) - \Edge(T^\infty)}$. Note that uncompactified moduli space $\Stashefftree_d$ itself can be described in this manner as $\Stashefftree[T\to T^{\infty}]$, with $T^\infty$ taken to be the corella (see \ref{fig:cubicdecomp}). \\

\begin{figure}
\subcaptionbox[Caption]{%
    $\Stashefftree_3$
    \label{subfig:sublabel2}%
}
[%
    0.4\textwidth 
]%
{%
		\fontsize{1cm}{1em}
		\def\svgwidth{4cm}
		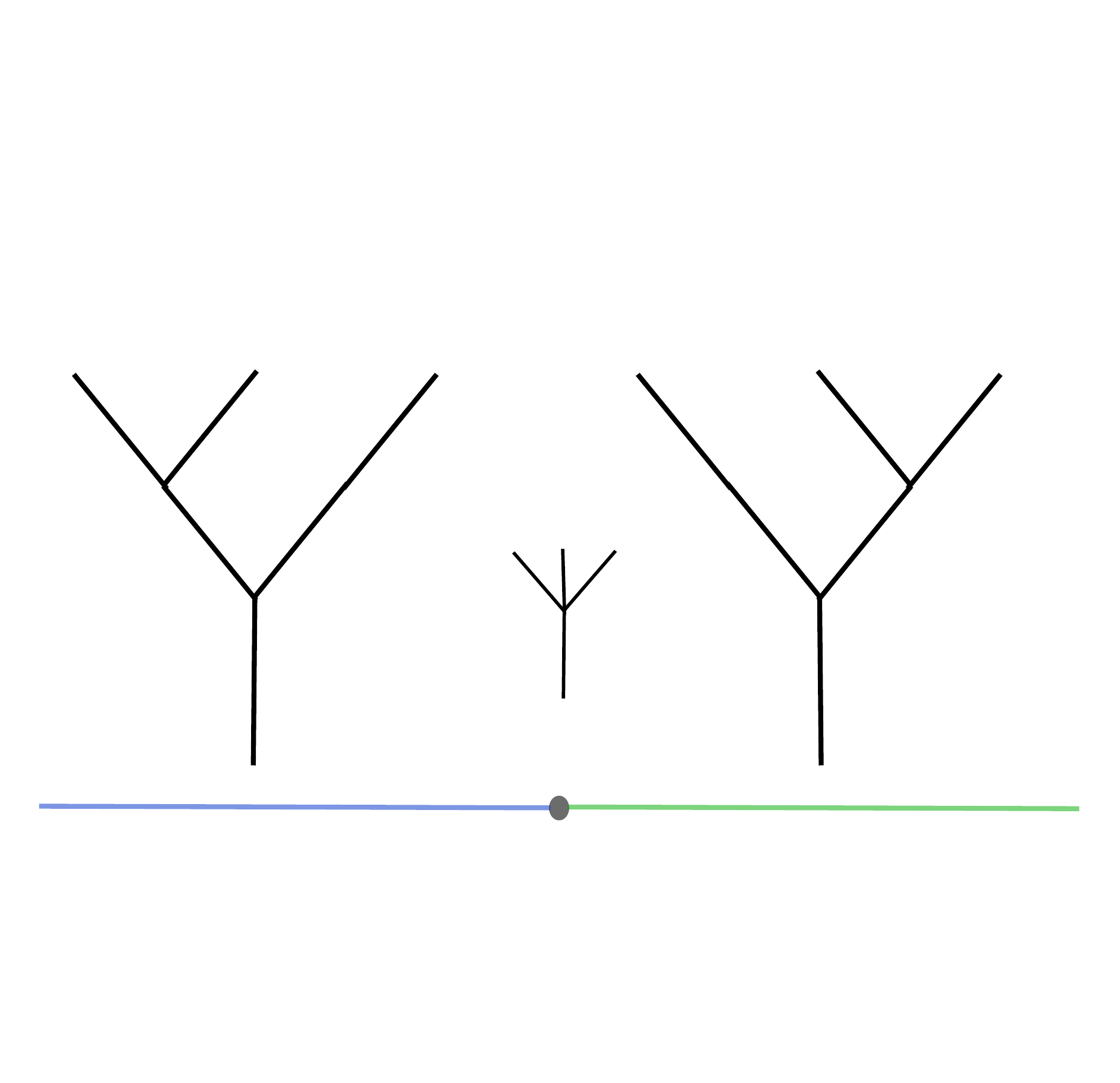
}%
\subcaptionbox[Caption]{%
    $\Stashefftree_4$
    \label{subfig:sublabel3}%
}
[%
    0.4\textwidth 
]%
{%
		\fontsize{1cm}{1em}
		\def\svgwidth{4cm}
		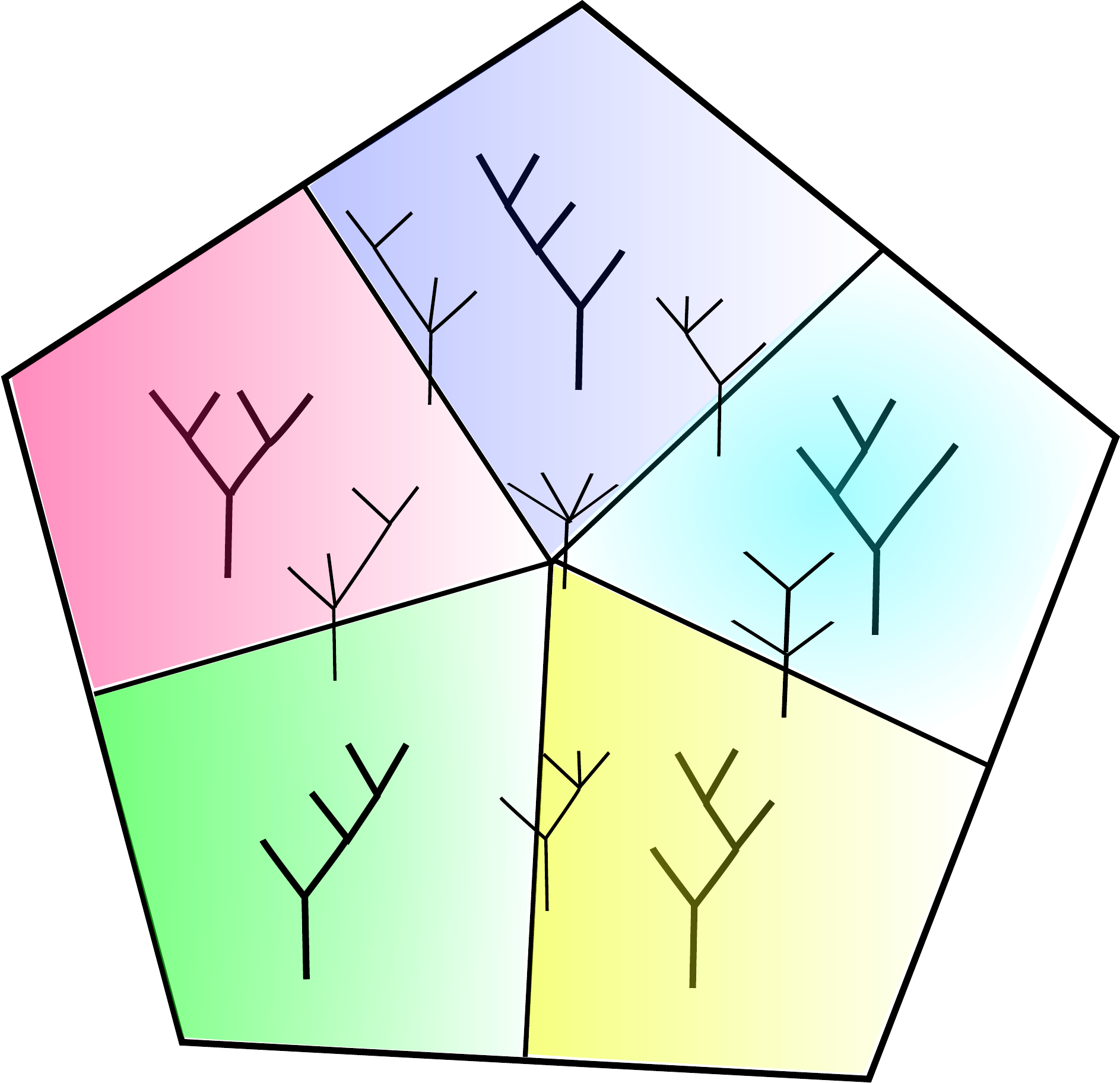
}%
\caption[Caption]{Cubical decomposition of the moduli of Stasheff trees.}
\label{fig:cubicdecomp}
\end{figure}


Let us once and for all fix a smooth structure on $[0,+\infty]$ using the \emph{gluing profile}
\begin{equation}
e^{-x} : [0,+\infty] \stackrel{\cong}{\longrightarrow} [0,1].
\end{equation}

This gives the closure of every cell $\cStashefftree[T \to T^\infty]$ the structure of a smooth manifold with corners. The result of this entire discussion can be summarized as
\begin{proposition}
The Stasheff polyhedron $\cStashefftree_d$ has the structure of a compact manifold with boundary and corners, whose interior is homeomorphic to $\R^{d-2}$ and whose boundary can be described by the fiber product
\[\partial \cStashefftree_d = \bigcup_{d_1+d_2 = d+1}  \: \bigcup_{1 \leq j \leq d_1} \cStashefftree_{d_1} \circ_j  \cStashefftree_{d_2} .\] 

Moreover, $\cStashefftree_d$ is naturally stratified by smooth manifolds $\Stashefftree[T^{\infty}]$, whose closure $\cStashefftree[T^{\infty}]$ are smooth manifold with corners homeomorphic to a closed cube. \noproof
\end{proposition}

\subsection{Perturbation scheme for gradient trees} \label{subsec:transversalityandperturbations}
Building on our initial success with the differential, we want to look at higher operations on the Morse complex defined by $A_\infty$-correspondences. \\

Let $(T,g_T)$ be a Stasheff tree. Every $e \in \Edge(T)$ is isometric to a segment in $\R$. Denote $t_e$ for the induced coordinate. 

\begin{definition}
A \textbf{gradient tree} is a continuous map $\psi : T \to X$ whose restriction to every edge is a smooth map $\psi_e : e \to X$ which solves the equation
\begin{equation} \label{eq:morseflow}
d\psi_e(\partial_{t_e}) = -\nabla_g f.
\end{equation}
\end{definition}

However, even as we try and define the cup product, we immediately stumble onto the first roadblock: given three critical points $\morseLabel=(p_0, p_1,p_2)$, the $p_0$ component of $\mu^2(p_2, p_1)$ should count triple intersections between the stable manifold of $p_0$, and the unstable manifolds of $p_1$ and $p_2$. Unfortunately, this space is necessarily empty unless $p_1 = p_2$! (in which case the intersection is not transverse ...) \\ 

This demonstrates the basic difficulty in constructing $A_\infty$-operations via Theorem \ref{thm:fukayacorrespondence}: we must take our original moduli spaces of Morse trees $\overline{\MM}_{k+1}$ and replace them a regularized space $\overline{\MM}'_{k+1}$, which is a manifold with boundary and corners (or at least a nice stratified space like a manifold with ends). Usually this involves writing down a Fredholm problem, constructing a universal moduli space and applying Sard-Smale. 
\begin{remark}
For pearl tree maps, this is the content of Sections \ref{sec:masseyproductquantum} and \ref{sec:parametrized}.
\end{remark}
In the Morse theory case, it turns out we can ''cheat" and use topological conjugacy to obtain a smooth structure on the trajectory spaces. In order to make fiber products transverse, we observe that the space above that defines the cup product admits a description as the moduli space of maps $T \to X$ from the unique trivalent tree $(T,g_T) \in \Stashefftree_2$ with three infinite external edges $e_0,e_1,e_2$, which converge at the respective ends to the critical points $p_0, p_1$ and $p_2$, and which solve on each edge the negative gradient flow equation \eqref{eq:morseflow}. So the most natural way to achieve transversality is to perturb this moduli space by modifying the differential equation obeyed on each edge as we approach the intersection point. 

\begin{definition} \label{def:perturbationdata}
A choice of \textbf{gradient flow perturbation data} on $(T,g_T)$ assigns to every $e \in \Edge(T)$ a smooth family of vector fields 
\begin{equation}
Y_e : e \to C^\infty(TX) 
\end{equation}
which vanish away from a bounded subset. 
\end{definition}

\begin{definition}
Given such a choice, a \textbf{perturbed gradient tree} is a continuous map $\psi : T \to X$ whose restriction to every edge is a smooth map $\psi_e : e \to X$ which solves the equation
\begin{equation} \label{eq:perturbmorseflow}
d\psi_e(\partial_{t_e}) = -\nabla_g f + Y_e.
\end{equation}
\end{definition}

If $\morseLabel=(p_0, p_1,\ldots,p_d)$ is a sequence of critical points of $f$, we define
\begin{equation}
\Stashefftree(\morseLabel)
\end{equation}
as the set of isomorphism classes of perturbed gradient trees such that the image of the k-th incoming leaf is $p_k$, and the image of the unique outgoing leaf is $p_0$. Similarly, the cubical cell $\Stashefftree[T \to T^\infty](\morseLabel)$ has an obvious meaning. It is easy to see that the grafting operation on Stasheff trees (see \ref{def:grafting}) extends to trees with Morse labels. Thus, this moduli space admits a regularization $\cStashefftree(\morseLabel)$ obtained by adding all singular perturbed gradient trees. 

\begin{lemma}
$\cStashefftree(\morseLabel)$ is a compact, metrizable space. 
$\noproof$
\end{lemma}
Note that this is particular implies that $\cStashefftree[T \to T^\infty](\morseLabel)$ is a compact, metrizable space as well. 

\begin{remark} In the original \cite{Fukaya19971}, Fukaya's used multiple Morse functions to define a structure of an $A_\infty$ pre-category (a similar approach was taken by Betz, Cohen and Norbury in a series of papers \cite{MR1270436}, \cite{MR3004282}, as well as in Nadler-Zaslow \cite{MR2449059}) However, to turn $(CM^\bullet,\mu^1)$ into a full-fledged $A_\infty$-algebra, we have to organize these perturbations into a coherent datum, in the spirit of Seidel's construction for Lagrangian Floer theory in \cite{MR2441780} and closely following the discussion in \cite{MR2529936} and \cite{MR2786590}, keeping our notation as similar as possible. 
\end{remark}
Let $d \geq 1$, and assume we have been given a choice of perturbation data for every element in $\cStashefftree_{d}$. We fix a ribbon tree $T$ with $d+1$ external edges and a flag $f = (v,e) \in \Flag(T)$. Note that every choice of metric $g_T$ determine a unique orientation and length preserving map
\begin{equation} \label{eq:te}
e \to \R
\end{equation}
taking $v$ to the origin and $e$ to the appropriate segment on either side of the origin. In particuler, given a cubical cell $\cStashefftree[T \to T^\infty]$, the flag $f$ determines a polyhedral subset in 
\begin{equation} \label{eq:poly}
[0,+\infty]^{\Edge(T) - \Edge(T^\infty)} \times \R
\end{equation}
whose fibre is the image of the segment $e$ under \eqref{eq:te}.
\begin{definition} \label{def:smoothfortrees}
We say that the choice of perturbation data is \textbf{smooth}, if given a flag $f = (e, v) \in \Flag(T)$, and a cubical cell $\cStashefftree[T \to T^\infty]$, the choice of perturbation data for the edge $e$ extends to a smooth map 
\begin{equation}
[0,+\infty]^{\Edge(T) - \Edge(T^\infty)} \times \R \to C^\infty(TX). 
\end{equation}
\end{definition}

A priori, each edge in a singular ribbon tree is equipped with at least two perturbation data: one comes from the singular tree, the other from the (smooth) tree wherein the edge lies. Compatibility with the gluing is the requirement that these two perturbation data agree. This is essential: in order to construct an $A_\infty$-correspondence, one of the conditions (the Maurer-Cartan axiom) is that the first boundary stratum (excluding the higher corner strata) is the fiber product of the interior of lower dimensional moduli spaces 
\begin{equation}
\MM'_{k-\ell+1} \times_{ev} \MM'_{\ell+1}. 
\end{equation}
If we choose the perturbations haphazardly, this clearly would not hold anymore --  so we must be careful to make all of our choices \textbf{coherent}, which usually mean pulling back a universal choice from an abstract moduli space (i.e., in the notations of equation \eqref{eq:differentiablecorrespondence}, the choices should be parametrized by $\OO_{k+1}$. In the Morse theory case this is the k-th moduli space of Stasheff trees, and in the case of pearl trees, this is the moduli space of pearl trees we construct in the begining of Section \ref{sec:masseyproductquantum}.)

\begin{definition} \label{def:universalchoiceofperturbationfortrees} A \textbf{universal consistent perturbation datum for trees} is a choice of a smooth family of perturbation data for every element of every Stasheff moduli space, which is compatible with the gluing maps and invariant under the automorphisms of each tree.
\end{definition}

Note that the condition of invariance implies that the perturbation datum vanishes whenever there is only one input (because the $\R$-translation symmetry forces a non-vanishing perturbation datum to have non-compact support, contradicting the boundedness requirement of Definition \ref{def:perturbationdata}.

\begin{lemma}[7.2 in \cite{MR2786590}] \label{lem:consistentchoices} Every smooth and consistent choice of perturbation data on the union of the Stasheff moduli spaces 
\begin{equation}
\coprod_{d'<d} \cStashefftree_{d'} 
\end{equation}
with the $k$-skeleton of $\cStashefftree_d$ extends smoothly and consistently to the $k + 1$ skeleton of $\cStashefftree_d$. Moreover, the restriction map from the space of all perturbation data to the space of perturbation data for a given singular Stasheff tree $(T, g_T)$ is surjective. \noproof
\end{lemma}
Every maximal cell in $\Stashefftree(\morseLabel)$ determined by a combinatorial type $\Upsilon =[T \to T^\infty]$ has a fiber product description. 
\begin{definition}
We define the edge space to be the product of Morse trajectory spaces over all edges of $\Upsilon$ of the moduli spaces of gradient trajectories:
\begin{equation}
\Edge_{(\Upsilon,\morseLabel)} = \bigtimes_{e \in \Edge(\Upsilon)} \MM_e(...,...)
\end{equation}
where 
\begin{equation}
\MM_e(...,...) = 
	\begin{cases}
		\bar{\MM}^{morse}(p_{v_-},X), & \text{if } v_- \text{ is a leaf} \\ \vspace{0.1em}
    \bar{\MM}^{morse}(X,p_{v_+}), &\text{if } v_+ \text{ is the root}, \\ \vspace{0.1em}
    \bar{\MM}^{morse}(X,X), &\text{if both vertices are internal }.
	\end{cases} \vspace{0.5em}
\end{equation}
\end{definition}
There is an evaluation map, denoted
\begin{equation}
ev_e : \MM_e(...,...) \to X \times X
\end{equation}
where 
\begin{equation}
ev_e = 
	\begin{cases}
		\left\{p_{v_-}\right\} \times ev, & \text{if } v_- \text{ is a leaf} , \\ \vspace{0.1em}
    ev \times \left\{p_{v_+}\right\}, &\text{if } v_+ \text{ is the root}, \\ \vspace{0.1em}
    ev_- \times ev_+, &\text{if both vertices are internal }.
	\end{cases} \vspace{0.5em}
	\end{equation}
The product of all evaluation maps is called the \textbf{Morse pseudo-cycle with boundary}.
\begin{definition}
We define the Flag space to be the product
\begin{equation}
\Flag_\Stashefftree = \bigtimes_{f \in \Flag(\Upsilon)} X_f
\end{equation}
over all flags $f = (v,e)$ of copies of the total space, i.e., $X_f = X$; The vertex space the be the product 
\begin{equation}
\Vert_\Stashefftree = \bigtimes_{v \in \Vert(\Upsilon)} X_v. 
\end{equation}
with $\Delta : \Vert_\Stashefftree \to \Flag_\Stashefftree$ being the obvious product of diagonal inclusion maps. 
\end{definition}
For every combinatorial type $\Upsilon$, asymptotics $\morseLabel$, the space $\Stashefftree_\Upsilon(\morseLabel)$ is given by the following fiber product diagram
\begin{equation}
\begin{tikzpicture}
    \matrix (m) [
            matrix of math nodes,row sep=3em, column sep=5.0em,
						text height=1.5ex, text depth=0.25ex
            ]   {
      & \Vert_{(\Upsilon)} \\
     \Edge_{(\Upsilon,\morseLabel)} & \Flag_{(\Upsilon,\morseLabel)} \\};
    \path[->]
    (m-2-1) edge node[auto] {$ev_T$} (m-2-2)
		(m-1-2) edge node[auto] {$\Delta$} (m-2-2);
\end{tikzpicture}   
\end{equation}
Since perturbing of the gradient flow equation on a bounded subset of an edge integrates to an essentially arbitrary
diffeomorphism (see Lemma 7.3 in \cite{MR2786590} and the discussion preceding Definition 2.6 there), it is clear that

\begin{lemma}[7.4 in \cite{MR2786590}]
All moduli spaces $\Stashefftree(\morseLabel)$ are manifolds of the expected dimension for a generic choice of universal perturbation data. \noproof
\end{lemma}

Let $\morseLabel = (p_d,\ldots,p_0)$ be an ordered tuple of critical points, $p_i \in Crit(f)$. We define their virtual dimension to be

\begin{equation}
\virdim(\morseLabel) := d-2 + \sum_{i=1}^d \deg(p_i) - \deg(p_0).
\end{equation}

\begin{corollary}
Given a generic choice of universal perturbation data, if the virtual dimension of $\morseLabel$ is zero, then $\Stashefftree(\morseLabel)$ has a finite number of points. If  $\virdim(\morseLabel)=1$, it is a one-manifold with ends.  
\end{corollary}

From now on, assume we have a fixed such a choice of perturbation data for trees. Note that if the virtual dimension of $\morseLabel$ is zero, the fiber product description of $\Stashefftree(\morseLabel)$ determines an isomorphism
\begin{equation}
\lambda(\Stashefftree(\morseLabel)) \otimes \lambda(X^{d+1}) \iso \lambda(X) \otimes \lambda(\Stashefftree_d) \otimes \lambda(W^s(p_0)) \otimes \lambda(W^u(\morseLabel^{in}))
\end{equation}
by \eqref{eq:fibeproductconvention} in Appendix \ref{sec:signs}, where $W^u(\morseLabel^{in})$ is the product of the descending manifolds of the critical points $\morseLabel^{in} = (p_d,\ldots,p_1)$. Whenever $\psi$ is a rigid tree in $\Stashefftree(\morseLabel)$, the above
isomorphism and Equation \eqref{eq:decom_tangent_space_crit_points} give a natural map
\begin{equation}
\mu_\psi : \oo_{p_d} \otimes \ldots \oo_{p_1} \to \oo_{p_0}.
\end{equation}

\begin{definition} 
We define the d-th Morse higher product 
\begin{equation}
\mu^d = \mu_0^d: CM^\bullet(f,g) \otimes \ldots CM^\bullet(f,g) \to CM^\bullet(f,g)[2-d]
\end{equation}
to be a sum over the induced maps $\mu_\psi$ on orientation bundles
\begin{equation}
[p_d] \otimes \ldots \otimes [p_1] \mapsto \sum_{p_0} \sum_{\psi} (-1)^{(n+1)\dim(p_0) + \dagger(\morseLabel)} \mu_\psi([p_d] \otimes \ldots [p_1])
\end{equation}
where the sign is given by:
\begin{equation}
\dagger(\morseLabel) = \sum_{k=1}^d k p_k
\end{equation}
\end{definition} 

\begin{proposition}[2.15 in \cite{MR2786590}]
The operations $\mu^d$ satisfy the axioms of an $A_\infty$-algebra. 
\end{proposition}

Note that uniqueness follows from existence: given two choices $(f_0,g_0)$ and $(f_1,g_1)$of Morse-Smale pairs on $X$, along with perturbation data for gradient trees $\vec{X}_0$ and $\vec{X}_1$ respectively, we equip the interval $[0,1]$ with a metric $g^{std} = dt^2$ coming from the standard inclusion in $\R$, and a Morse function $f^{std}$ that has two degree zero critical points at $t=0$ and $t=1$ and a unique degree one critical point at $t=1/2$. Consider a 1-parameter family $(f_t,g_t)$ interpolating between our two choices. Then we can construct a new Morse-Smale pair $(F,G)$ by
\begin{equation}
f_{[0,1]}(t,x) := f^{std}(t) + f_t(x) 
\end{equation}
and $g_{[0,1]} \big|_{ \left\{t\right\} \times X} = g_t$, $p_{2}^* f_{[0,1]} = dt^2$ where $p_2 : X \times [0,1] \to [0,1]$ is the projection onto the second factor. It is clear that we can extend our perturbation data for gradient trees to a new choice $\vec{X}$ in such a way that it coincides with the previous choices for $t=0$ and $t=1$. Then $(f_{[0,1]},g_{[0,1]},\vec{X}_{[0,1]})$ give rise to an $A_\infty$-structure on $CM^\bullet(f_{[0,1]},g_{[0,1]})$ which comes with quasi-isomorphism from the given $A_\infty$-structures defined by $(f_i,g_i,\vec{X}_i)$ for $i=1,2$ (by embedding them from the $t=i$ boundary.) This can be thought of as a chain-level version of the continuation method from Theorem \ref{thm:morsecontinuation}.

\begin{remark}
This is essentially the same as Fukaya's explicit $[0,1] \times C^\bullet$ model, see \cite[p.~8]{MR2555938} and \cite[section 4.2]{MR2553465}. 
\end{remark}

As a result of this construction, we see that

\begin{corollary} \label{cor:morseindependence}
The Morse $A_\infty$-algebra is independent of all choices up to an $A_\infty$-quasi-isomorphism. 
\end{corollary}
Note that the same interpolation trick also applies to prove that in the case of a mapping torus $X = M_\phi$, the Morse $A_\infty$-algebra depends only on the isotopy class of $\phi$. 
\begin{remark}
The method we used does not require a geometric model for an $A_\infty$-homomorphism. However, it is by no means the only way to prove Corollary \ref{cor:morseindependence}.  

One possible approach is to introduce quilted metric trees (or mangroves): these are planar metric trees with a subset of colored vertices such that every simple path from the root to a leaf meets precisely one colored vertex, which also satisfy a certain balancing condition (see the original \cite{MR0420609} and \cite{MR2660682} for a detailed summary.) We can define multilinear operations which count solutions to perturbed gradient equations based on families of quilted metric pearl. The moduli space of such trees is a model for the for the multiplihedron, so the result should satisfy the $A_\infty$-homomorphism equation. This is the approach to invariance taken in \cite{eprint2}. 
\end{remark}

\section{Computation (I): the Johnson homomorphism} \label{sec:compute3dim}
In this section, we will compute some of the higher operations $\mu^d$ of Morse $A_\infty$-algebra on the mapping torus of surface. \\
 
A few preliminary remarks: this computation happens completely in the surface (mapping torus), which we think of as an \emph{abstract curve}, i.e., we ignore the embedding in $\P^3$. To emphasize this, we choose to denote $\Sigma$ for the genus 4 surface. The purpose of this section is to serve as an illustration of the connection to $\tau_2$, as well as a basis for our computation of the main term in the quantum Massey product in Section \ref{sec:compute4}. The construction proceeds by cutting the surface to pieces and looking at a few standard models: the circle, the annulus, the genus 2 surface with one boundary component, and then gluing this local models together.

\subsection{The circle} 
Consider the circle $S = S^1$ is parametrized as
\begin{equation}
S = \left\{e^{s \pi \imath} \: \bigg| \: s \in [0,2] \right\}.
\end{equation}
We consider a minimal Morse function $f^S :S \to \R$ with a degree zero critical point at $s=0$, and a degree one critical point at $s=1$, and the standard metric $g^S = ds^2$. Denote them $1,\mathfrak{t}$, respectively. We also fix a choice of a universal perturbation datum for trees.

\subsection{The annulus} \label{subsec:annulus}
Consider the annulus $A$ of positive area $2\epsilon>0$. Fix a constant $0<\delta_\epsilon < \frac{\epsilon}{2}$. Let us give the annulus radial coordinate $r \in [-\epsilon, \epsilon]$ and angle coordinate $\theta$ and choose a smooth, monotone function $k : \R \to \R$ with even derivative such that \qquad \qquad
\begin{equation}
\begin{split}
k(r) = 0 & \iff r \leq - \delta_\epsilon \\ 
k(r) = 1 & \iff r \geq \delta_\epsilon.
\end{split}
\end{equation}
\begin{definition}
The map
\[ T : (r,\theta) \mapsto ( r , \theta + k(r) ) \]
is called a \textbf{Model Dehn twist} around the core curve. See Figure \ref{fig:dehntwisteffect} for an illustration. 
\end{definition}

\begin{figure}[htb] 
\begin{center}
\centering
			\fontsize{0.25cm}{1em}
			\def\svgwidth{4cm}
 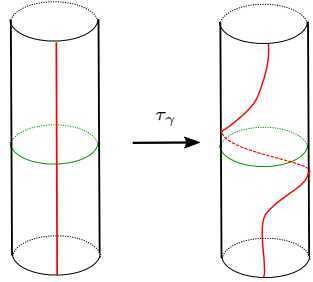
\end{center}
  \caption{The effect of the model Dehn twist on the annulus}
	\label{fig:dehntwisteffect}
\end{figure}

\begin{lemma}
Any two model Dehn twists are isotopic. \noproof
\end{lemma}

This justifies talking about ''\emph{the} Dehn twist", which is the mapping class $\tau = [T]$. 

\begin{remark}
We are being a bit cavalier here by blurring the distinction between the ordinary mapping class group and the symplectic mapping class group, especially considering the fact that surfaces are not simply connected -- thus the later should be equipped with the Hamiltonian topology. However one can check that this poses no additional difficulties and the standard results of mapping class group theory, the Torelli group etc still hold in this setting, see \cite[Appendix B]{MR2383898}. 
\end{remark}

Let $A_T$ be the mapping torus of the Dehn twist. Let $f^{A}(r)$ be any smooth, positive, strictly monotone function in the radial coordinate. Then,
\begin{lemma} 
The pair
\begin{equation} \label{eq:metricmodel}
(f^{A_T}(t,r),g^{A_T}) = (f^{S^1}(t) +  f^{A}(r), (d \theta + t \dot{k}(r) dr)^2 + dr^2 + dt^2)
\end{equation}
is Morse-Smale pair on $A_T$.
\end{lemma} 
\begin{proof}
The Morse function is well defined on the cylinder $S^1 \times A$, and manifestly invariant under the monodromy, so it descends to the mapping torus. As for the metric, note that if we pullback by the monodromy, 
\begin{equation}
([d \theta + \dot{k}(r)dr] + (t-1) \dot{k}(r)dr)^2 + dr^2 + dt^2 = d (\theta + t \dot{k}(r) dr)^2 + dr^2 + dt^2. 
\end{equation}
\end{proof}

\subsection{General remarks} \label{subsec:generalremarks}
Let $\Sigma$ be any closed surface of genus $g$. There is a standard presentation of the surface which comes from gluing the edges of a $4g$-gon in the way prescribed by the labeling 
\begin{equation}
(a_1 b_1 a_1^{-1} b_1^{-1})\cdot (a_2 b_2 a_2^{-1} b_2^{-1})\cdot \ldots \cdot (a_g b_g a_g^{-1} b_g^{-1}) = [a_1,b_1] \cdot [a_2,b_2] \cdot \ldots \cdot [a_g,b_g]
\end{equation}

of the sides starting from a vertex and labeling in the clockwise direction. We fix a point $m_+$ in its interior and denote the vertex $m_-$. Now we can find two $2g$-tuples of ''standard" pairwise disjoint arcs, $(\zeta_1,\ldots,\zeta_{2g})$ and $(\eta_1,\ldots , \eta_{2g})$ such that $\zeta_i$ intersects $\eta_j$ if only if $i=j$, in which case the intersection is transverse at one point. The curves $(\zeta_{2i-1},\zeta_{2i})$ coincide with the edges labelled $(a_i , b_i)$ and the curves $(\eta_{2i-1},\eta_{2i})$ connect the midpoints of $(\zeta_{2i-1},\zeta_{2i})$ radially to $m_+$, see Figure \ref{fig:standardmorsesmale} for an illustration in the genus $g=2$ case. 

\begin{figure}[h]
		\fontsize{0.3cm}{1em}
		\def\svgwidth{7cm}
		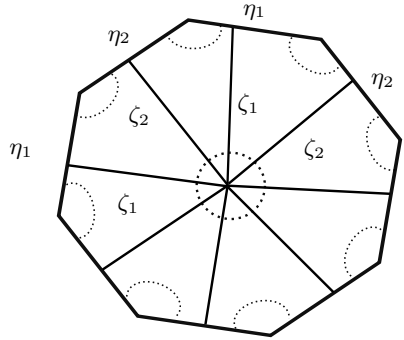
		\caption{Standard Morse-Smale diagram}
		\label{fig:standardmorsesmale}
\end{figure}

We define a standard Morse-Smale pair $(f^\Sigma,g^\Sigma)$ to be a minimal Morse function with $m_+$ being the unique critical point of degree $0$; $m_-$ the unique critical point of degree $2$; and $g$ pairs of critical points $\left\{a_i,b_i\right\}$ of degree 1 with separatrix (which is a 1-manifold diffeomorphic to the circle that can be obtained by smoothing the compactified ascending manifold) $\left\{\zeta_i,\eta_i\right\}$. 

\begin{remark}
Conversely, given a minimal Morse function and a metric, cutting along the separatrix of the critical points of index 1 gives a presentation of this form. 
\end{remark}

We begin with a few simple observations. \vspace{0.5em}
\begin{itemize}
\item
Let $2 \leq k<\infty$ be an integer. Lemma \ref{lem:consistentchoices} implies that if we manage to choose a perturbation datum that defines an $A_k$-structure (that means smooth, consistent and coherent) we can always extend the choice to an $A_\infty$-structure without modifying our previous choices. \vspace{0.5em}
\item
If we choose a small disc around $m_+$ and remove it, the resulting Morse-Smale pair on the surface with one boundary component $\Sigma_{g}^\partial$ is \textbf{boundary-convex}, so we can assume that all gradient trees that start from a critical point of degree $>0$ do not meet the regular level set $f^{-1}(c)$ where $c =m_+ - \epsilon$ for some $\epsilon>0$ small. Note $\Sigma_{g}^\partial$ retracts onto the union of the 1-skeleton. Moreover, it is possible to define a choice of perturbation datum for trees that defines an $A_k$-structure where each time-dependent vector field is taken to have support in any open subset which contains the union of the separatrix, small discs around the critical 1-points, and a small disc around $m_-$ (from now on, we refer to such small discs as \textbf{caps} and always assume that they are contained in a local Morse chart around the critical points.) \vspace{0.5em}
\item
Given any choice of smooth, coherent perturbation datum for trees on $\bigcup_{\ell<k} \Stashefftree_\ell$, the recursive choice of time-dependent vector fields needed to make to make it regular is a ''small choice" (in the sense that their $C^2$-norm can be taken to be arbitrary small). 
\end{itemize}
Let us introduce some convenient terminology. 
\begin{definition}
Given a critical point $p \in \crit(f)$ with a choice of cap $p \in D \subset \Sigma$, we define a \textbf{point-pushing map} $\Psi_1$ to be the time-1 map obtained by integrating a time-dependent vector field $X_t$, compactly supported in $D$, and such that (in the local Morse chart on the cap): 
\begin{equation}
\begin{split}
supp(X_t) &\Subset B_{\epsilon/2}(0), \\
\Psi_t(B_{\epsilon/2}(0)) &\subset B_{\epsilon}(0) \\
\Psi_1(0) &\neq 0. 
\end{split}
\end{equation}
for some fixed $\epsilon>0$ and $0 \leq t \leq 1$.
\end{definition}

\begin{definition} \label{def:pushoff}
Given a simple closed curve $\gamma \subset \Sigma$, we define a \textbf{small pushoff} $\Phi_1$ in the following way: we make a choice of tubular neighbourhood $T_\gamma \subset \Sigma$. Then $\Phi_1$ is the isotopy obtained by integrating a time-dependent vector field, 
\begin{equation} \label{eq:pushoff}
\begin{split}
X_{t,s,x} &:= \rho(t,s,x) \cdot \frac{\partial}{\partial x}, \\
\rho(t,s,x) &\equiv 0 \text{ if } |x|>\epsilon/2, \\
\rho(t,s,x) &\neq 0 \text{ if } x=0,
\end{split}
\end{equation}
Here $\epsilon$ is some fixed, sufficiently small positive number and we require that $\Phi_1$ takes the disc bundle of radius $\epsilon/2$ into the disc bundle of radius $\epsilon$ and $\Phi_1(\gamma) \cap \gamma = \phi$. The coordinates in equation \eqref{eq:pushoff} come from fixing an isomorphism between the cylinder and the normal bundle of the curve
\begin{equation}
(s,x) \in S^1 \times \R \iso \NN_\gamma \iso T_\gamma.
\end{equation}
\end{definition}

\begin{remark}
Later we will want to consider surfaces with boundary as well. We note that Definition \ref{def:pushoff} generalizes to simple arcs (mutatis mutandis) in the obvious way. 
\end{remark}

We start by considering the product, $\mu^2$. We need to specify a time-dependent vector field for each edge of the unique trivalent tree 
\begin{equation} 
\begin{split} \left\{ \right. \end{split} 
\begin{split} \scalebox{.25}{\begin{tikzpicture}
        \node[Feuille](0)at(0.,1.5){};
        \node[Feuille](1)at(1.,0.){};
        \node[Feuille](2)at(2.,1.5 ){};
				\node[Feuille](3)at(1.,-1.5){};
        \draw[Blue](0)--(1);
        \draw[Green](2)--(1);
				\draw[Black](1)--(3);
\end{tikzpicture}}  \end{split} 
\begin{split} \left. \right\} \end{split} 
\begin{split}\quad = \quad \end{split} \begin{split}  \Stashefftree_2 \end{split}.
\end{equation}
This can be done by making a generic choice of point-pushing vector field at each critical point\footnote{In fact, we can do slightly better: there is actually no need for point-pushing around the unique degree zero critical point.}; along with a small pushoff of each separatrix. See Figure \ref{fig:mu2colors} for an illustration in the case of the torus and the genus two surface (obtained by identifying opposite sides, as usual.)

\begin{figure}[h] 
\begin{subfigure}{\linewidth}
\begin{center}
 \def\svgwidth{0.5\columnwidth}
 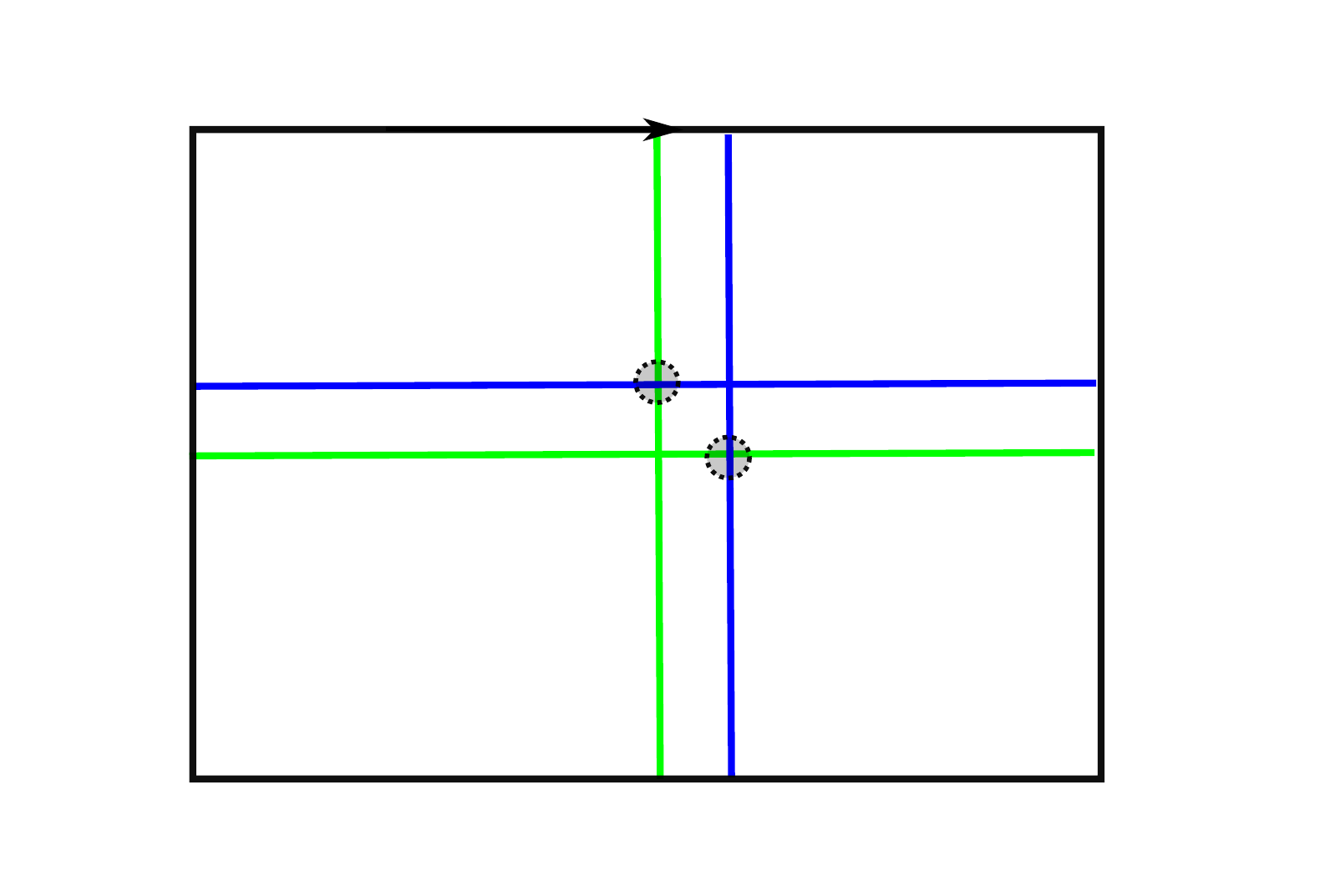
\end{center}
	
	\vspace{2em}
\end{subfigure} 
\begin{subfigure}{\linewidth}
\begin{center}
 \def\svgwidth{0.5\columnwidth}
 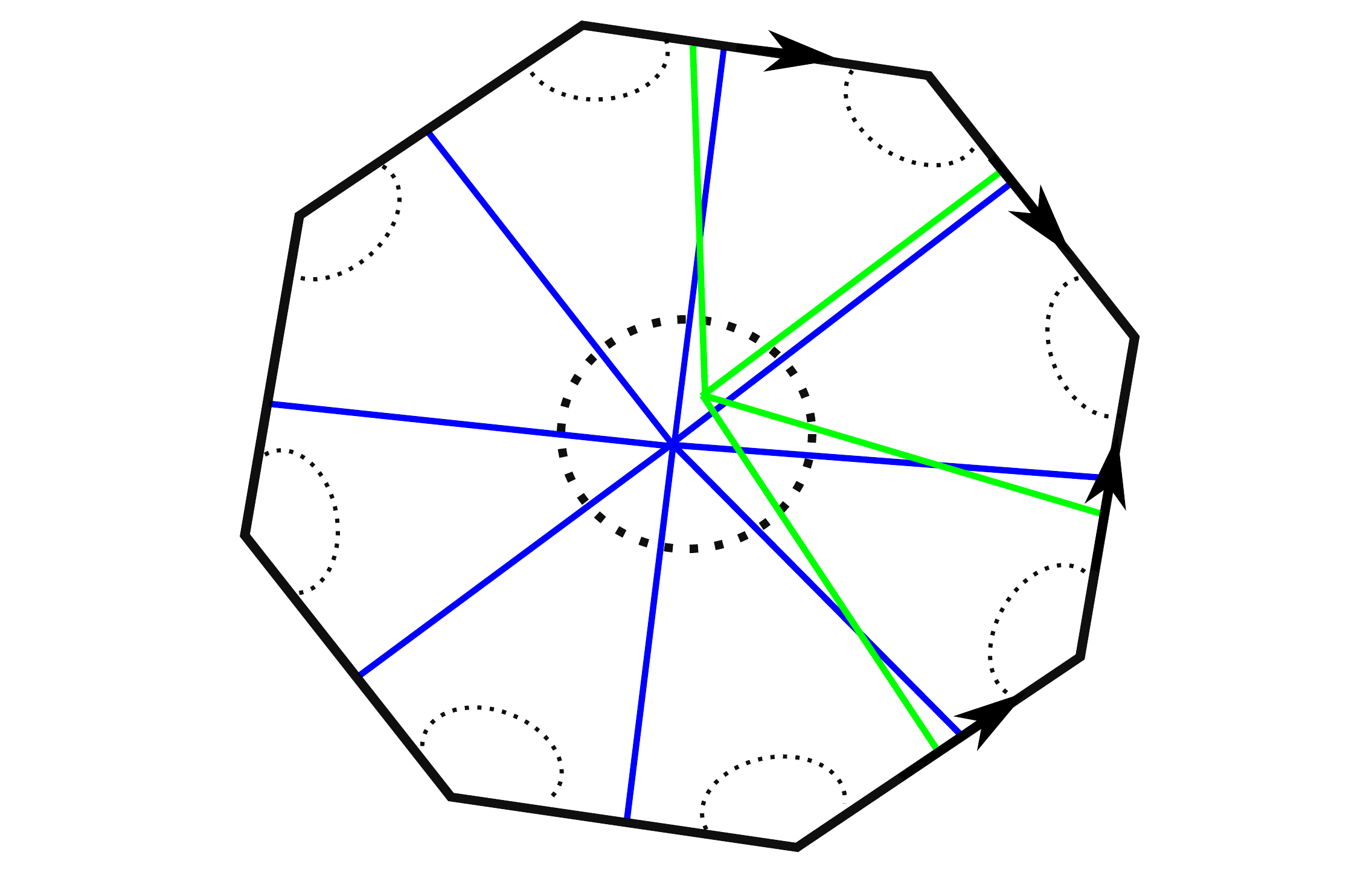
\end{center}
\end{subfigure}
  \caption{Computing $\mu^2(p_2,p_1)$, where $\deg(p_1)=\deg(p_2)=1$}
	\label{fig:mu2colors}
\end{figure}

Next, we consider the 1-parameter family necessary to define $\mu^3$. We choose some perturbation for the corolla. Note that there are no triple intersections (see Figure \ref{fig:mu3colors}). 

\begin{equation} 
\begin{split} \scalebox{.25}{\begin{tikzpicture}
        \node[Feuille](0)at(-1.,1.5){};
				\node[Feuille](1)at(1.,1.5){};
				\node[Feuille](2)at(3.,1.5 ){};
        \node[Feuille](3)at(1.,0 ){};
        \node[Feuille](4)at(1.,-1.5 ){};
        \draw[Blue](0)--(3);
				\draw[Violet](1)--(3);
				\draw[Green](2)--(3);
				\draw[Black](3)--(4);
\end{tikzpicture}}  \end{split} 
\begin{split}\quad \in \quad \end{split} \begin{split}  \Stashefftree_3 \end{split}.
\end{equation}

We interpolate between our choice of the corolla and the boundary values coming from the choices for $\partial \Stashefftree_3$ inherited from the broken Stasheff trees (i.e., one should imagine the purple lines converging to the blue ones on one side, and to the green ones on the other.) 

\begin{figure}[htb] 
\begin{center}
 \def\svgwidth{0.5\columnwidth}
 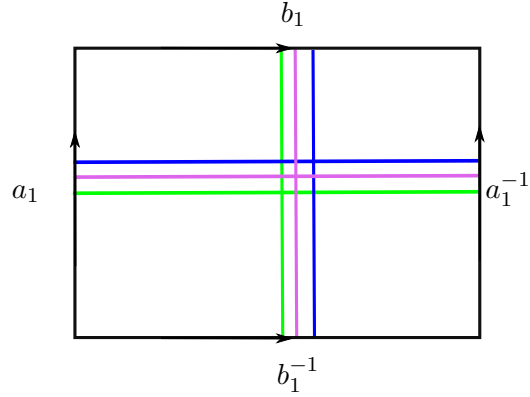
\end{center}
  \caption{The corolla for $\mu^3(p_3,p_2,p_1)$ where $\deg(p_i)=1$}
	\label{fig:mu3colors}
\end{figure}

Note that even though the cohomology algebra of any surface is intrinsically formal, that just means that \emph{the Massey products} vanish. It is still possible to have nontrivial $\mu^3$. \\

\subsection{Morse-Smale pair on the mapping torus}
Now we specialize to our case. From now on, $\Sigma$ is a genus $4$ surface with an action of a separating Dehn twist $\phi : \Sigma \to \Sigma$ of genus $2$. 

\begin{remark}
The genus restriction is of course extraneous, but due to our intended application (and to keep the discussion concrete and the pictures pretty) we leave it as it is. The reader may replace $g=4$ with $g \geq 4$ everywhere and make the obvious modifications.
\end{remark}

The first thing we wish to note is, 

\begin{lemma} \label{lem:collare} There exists a collar neighbourhood 
\begin{equation}
N^\gamma \subset \Sigma 
\end{equation}
with a symplectomorphism 
\begin{equation}
(A_{2\epsilon},d\theta \wedge dr) \to (N^\gamma,\omega) \: \: , \: \:  (r, \theta) \mapsto e(r, \theta)
\end{equation}
such that $\gamma$ is the core curve, and 
\begin{equation}
e^{-1} \circ \phi \circ e 
\end{equation}
is a model Dehn twist. \noproof
\end{lemma}

The situation is depicted in Figure \ref{fig:surfacesigma4} below.

\begin{figure}[htb] 
\begin{center}
 \def\svgwidth{0.7\columnwidth}
 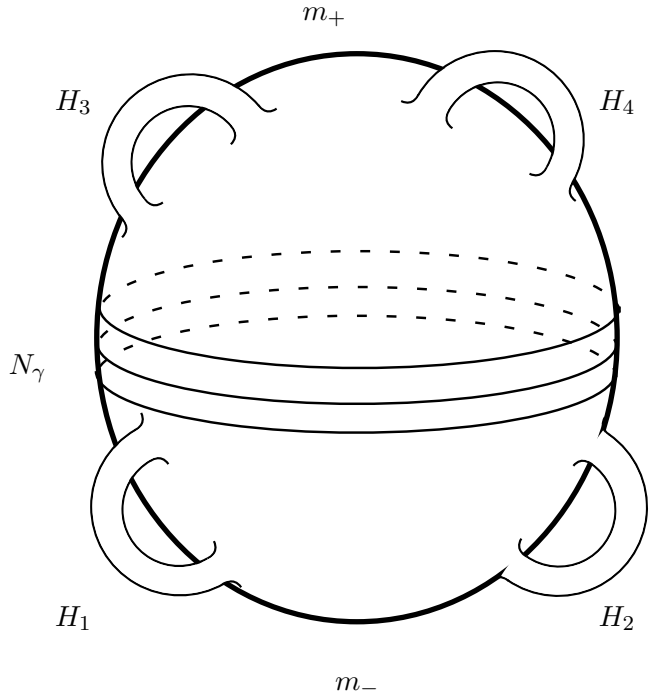
\end{center}
  \caption{The surface $\Sigma$}
	\label{fig:surfacesigma4}
\end{figure}

\begin{definition} We denote the two resulting collars 
\begin{equation}
Z_- := e( (-\epsilon,-\delta_\epsilon] \times S^1) \: \: , \: \: Z_+ := e( [\delta_\epsilon,\epsilon) \times S^1).
\end{equation}
Cutting along $\gamma$ separates $\Sigma$ into two subsurfaces. We denote the one which contains the collar $Z_\pm$ as $\Sigma_\pm$; they will be referred to as the upper subsurface and the bottom subsurface (because we visualize our gradient flow as coming from a height function.) 
\end{definition}

Without loss of generality, we normalize our minimal Morse-Smale pair $(f^\Sigma,g^\Sigma)$ and assume that it satisfies the following properties:
\begin{enumerate}
	\item[(a)]
	There is a unique minima $m_-$, a unique maxima $m_+$, and eight saddle points, which we denote 
	\begin{equation}
	\left\{a_1,b_1,a_2,b_2,a_3,b_3,a_4,b_4\right\}.
	\end{equation}
	The pair $\left\{a_i,b_i\right\}$ corresponds to the one-handle $H_i$, $i=1,\ldots,4$ in figure \ref{fig:surfacesigma4}. \vspace{0.5em}
	\item[(b)]
	All critical values are distinct. $f(m_-)=-2$, and $f(m_+)=2$. The critical values of the saddle points in the lower subsurface are between $-3/2$ and $-1$. The critical values of the saddle points in the upper subsurface are between $1/2$ and $1$. \vspace{0.5em}
	\item[(c)]
	Under the symplectomorphism $e : N_\gamma \to A_{2\epsilon}$ from \ref{lem:collare}, the function $f^\Sigma$ is $r$, the radial coordinate. Moreover, $N_\gamma = f^{-1}([-\epsilon,\epsilon])$.
\end{enumerate}

\begin{definition}
We define a Morse-Smale pair $(f,g)$ on the mapping torus $\Sigma_\phi$ in the following way: 
\begin{equation}
f = f^\Sigma(x) + f^{S^1}(t). 
\end{equation}
The metric $g$ is defined by setting it to be the sum 
\begin{equation}
g \big|_{(Z_\pm) \times S^1} = g\big|_{Z_\pm}^\Sigma + g^{S^1} 
\end{equation}
when restricted to the mapping torus of the upper and lower subsurfaces, and by the requirement that the restriction to the mapping torus of the tubular neighborhood is
\begin{equation}
g \big|_{(N_\gamma)_\phi} = g^{A_T}.
\end{equation}
\end{definition}
Now, we need to construct a coherent system of perturbations. We will consider the genus 4 surface $\Sigma$ from Figure \ref{fig:surfacesigma4} can be decomposed as the connected sum of a cap around the minima, a genus 2 surface with two boundary components, a tube around $\gamma$, another genus 2 surface with two boundary components, and a cap around the maxima. See Figure \ref{fig:genus4surfacedecmposed}.  

\begin{figure}[h] 
\begin{subfigure}{0.3\linewidth}
\begin{center}
 \def\svgwidth{0.5\columnwidth}
 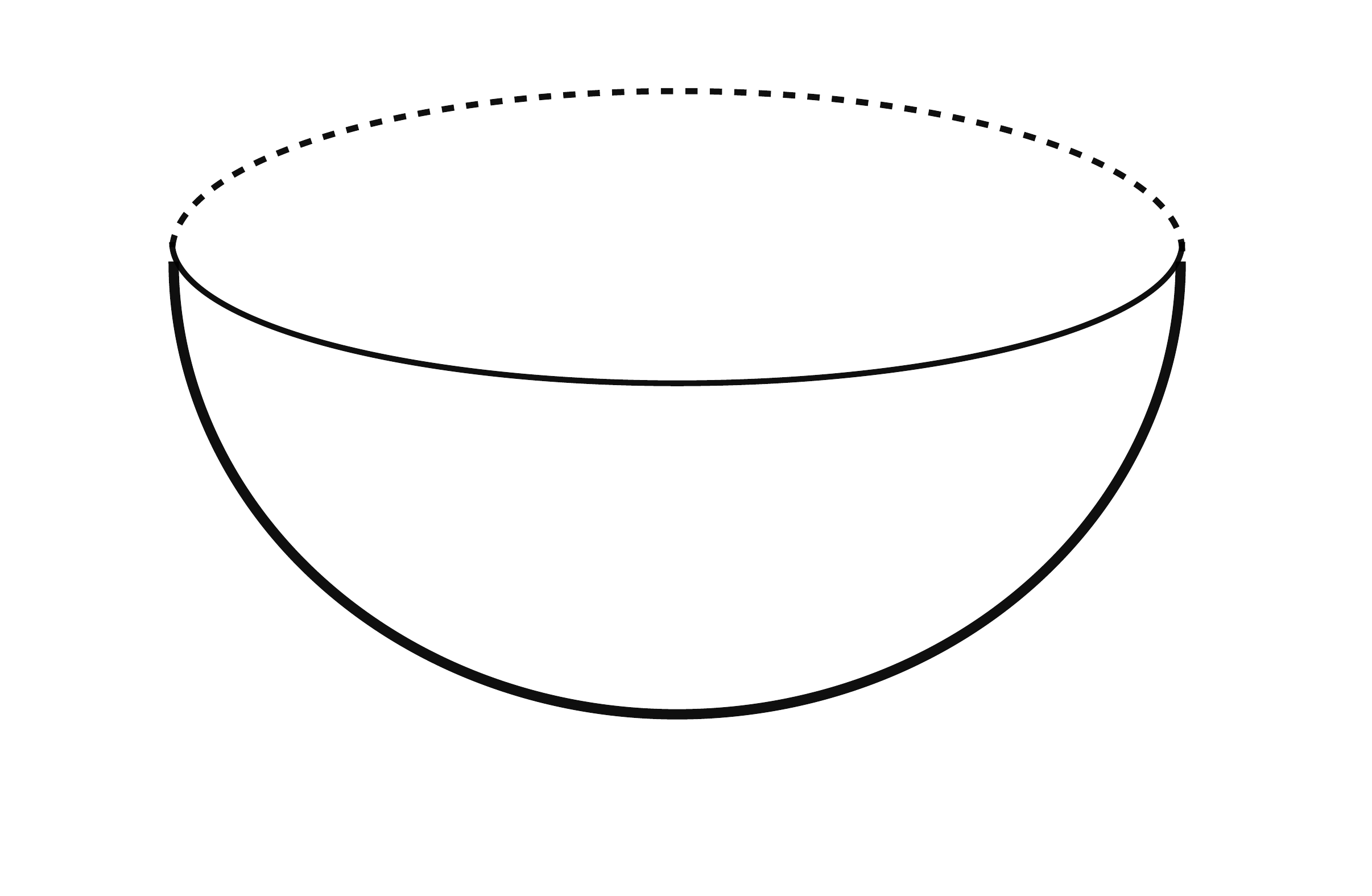
\end{center}
\vspace{0.5em}
\end{subfigure} \hfill
\begin{subfigure}{0.68\linewidth}
\begin{center}
 \def\svgwidth{0.5\columnwidth}
 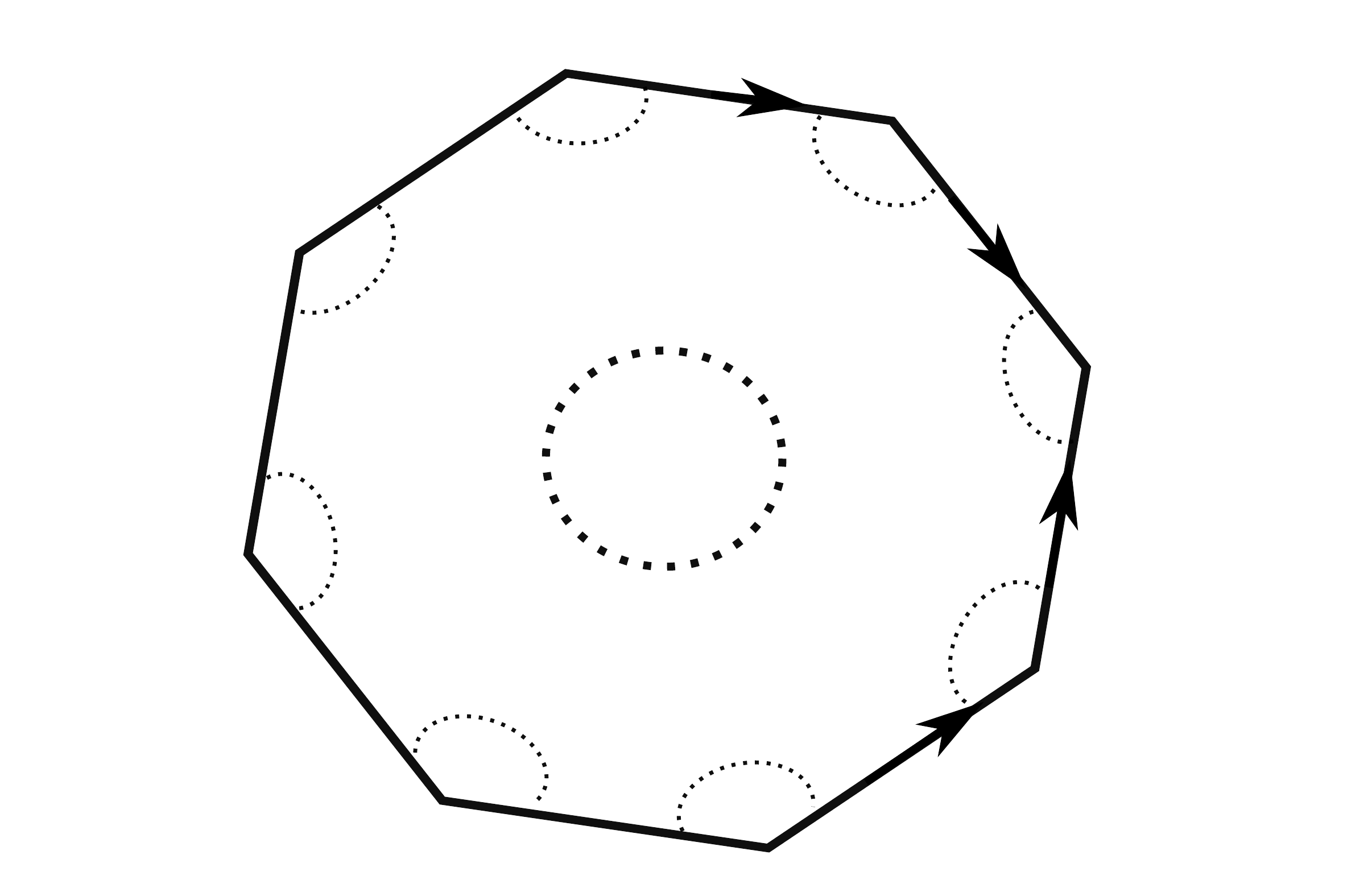
\end{center}
\vspace{0.5em}
\end{subfigure}
\begin{subfigure}{0.3\linewidth}
\vspace{0.5em}
\begin{center}
 \def\svgwidth{0.5\columnwidth}
 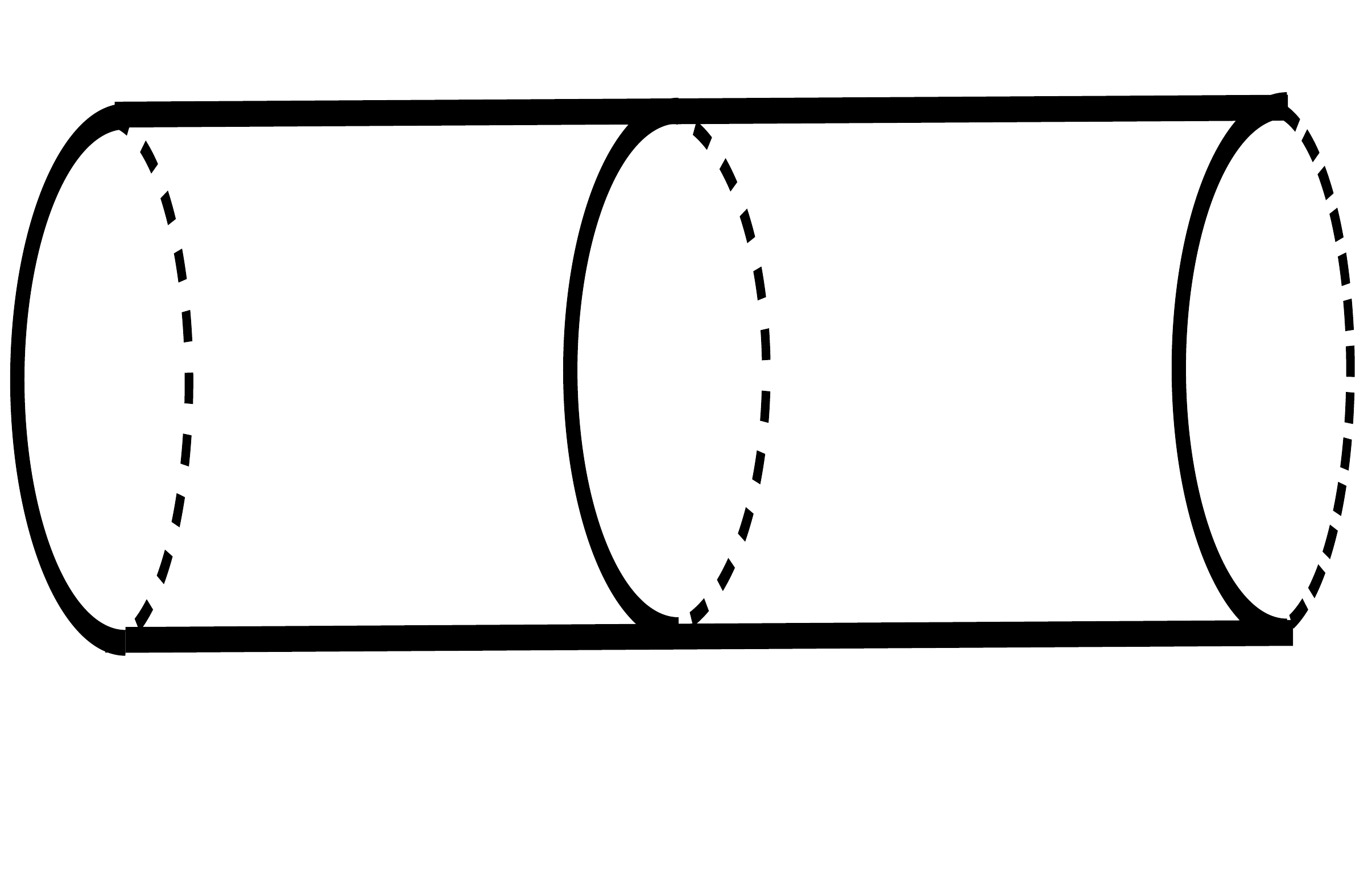
\end{center}
\end{subfigure} \hfill
\begin{subfigure}{0.68\linewidth}
\vspace{0.5em}
\begin{center}
 \def\svgwidth{0.5\columnwidth}
 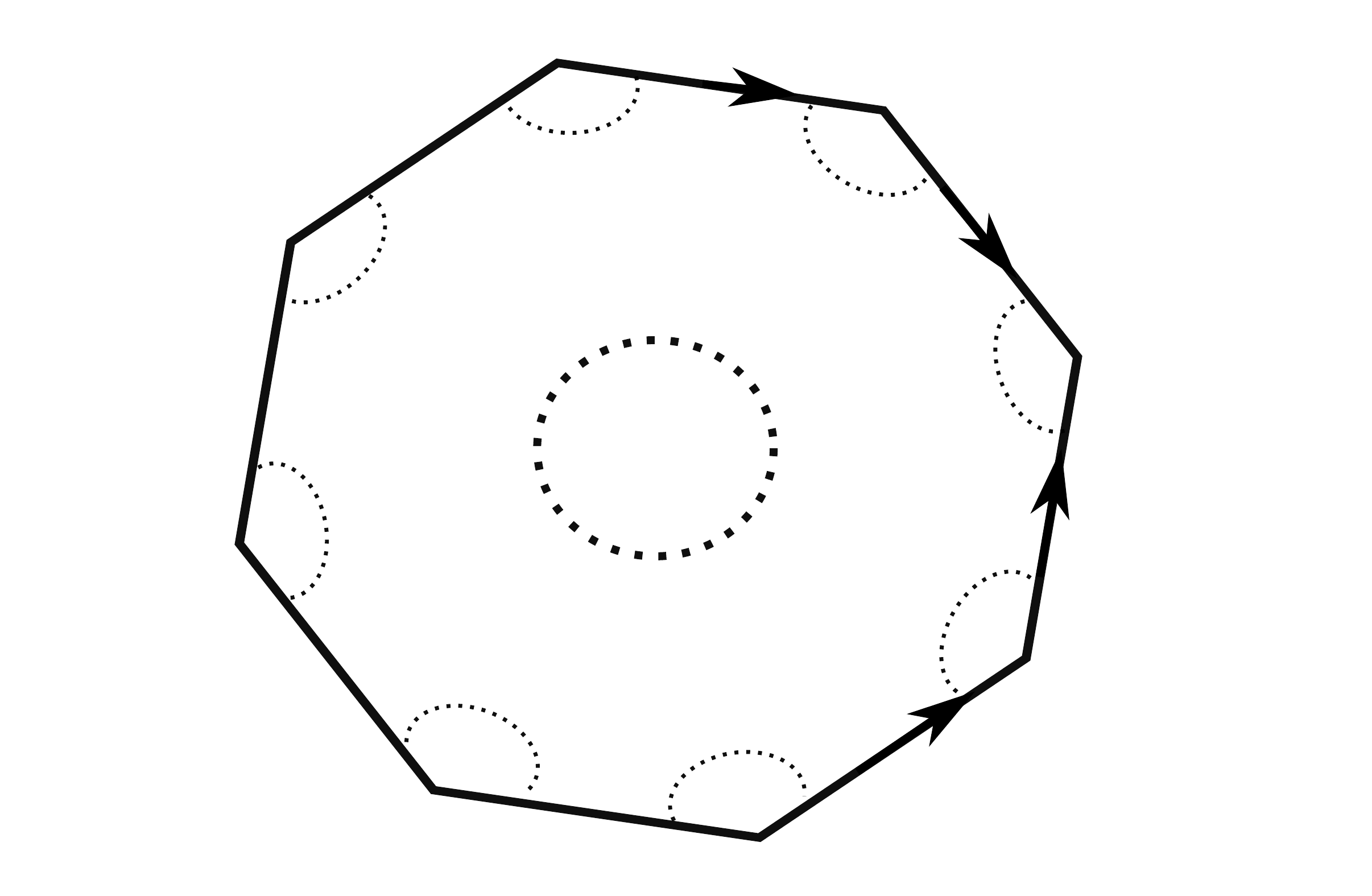
\end{center}
\end{subfigure}

\begin{subfigure}{\linewidth}
\begin{center}
 \def\svgwidth{0.5\columnwidth}
 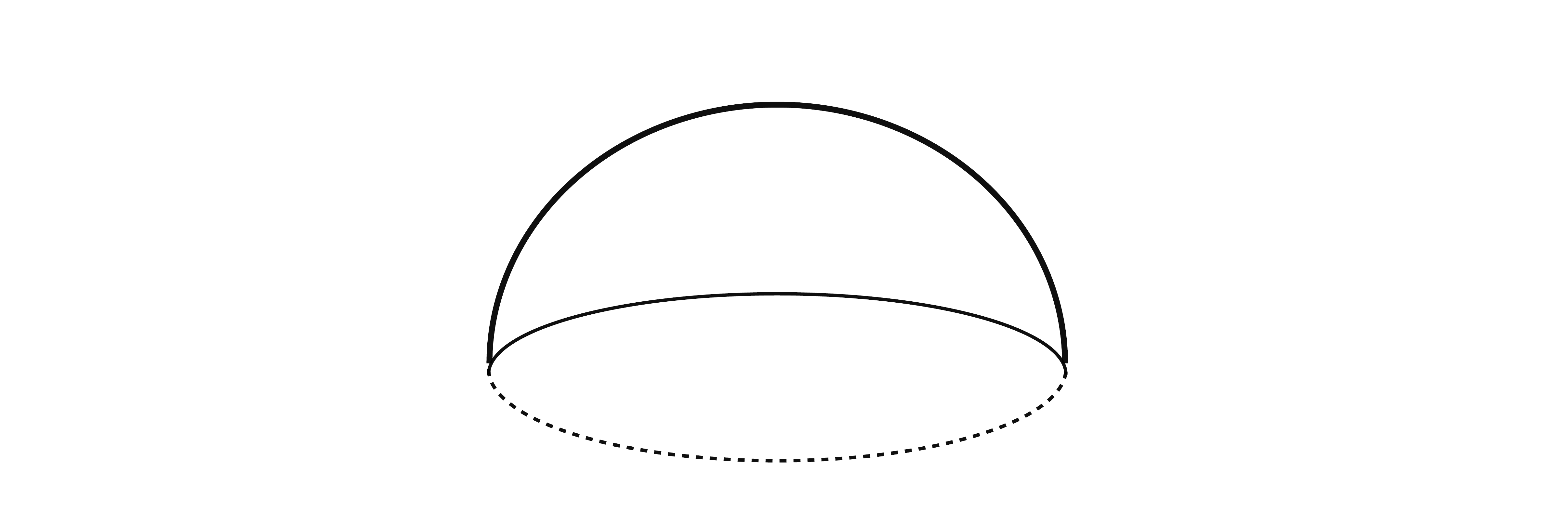
\end{center}
\end{subfigure}

  \caption{Computing $\mu^2(p_2,p_1)$, where $\deg(p_1)=\deg(p_2)=1$}
	\label{fig:genus4surfacedecmposed}
\end{figure}

Note that this induces a decomposition of $\Sigma_\phi$ as well. In fact, the only piece influenced by the monodromy is the tube, where the behavior is described by a standard local model (see Section \ref{subsec:annulus}.) However, this in turn would heavily influence any bounding cochain that connects critical points that lie above and below the tube. 

\subsection{A standard model for the bottom subsurface}
To analyze the new contributions to $\mu^3$ clearly, we want to modify the standard perturbation datum from Section \ref{subsec:generalremarks} a bit and pull all the intersections of $\left\{a_i,b_i\right\}$ with $i=1,2$ (which happen near $m_+$, as usual) into the lower subsurface. Let us look more closely at $\Sigma_-$ and the flow lines of critical points of index 1 as in Figure \ref{fig:genus4bottomsubsurface}. We fix $i \in \left\{1,2\right\}$ and choose the i-th colored rectangle $R_i$ (for a schematic picture see Figure \ref{fig:I1}.) 

\begin{figure}
\begin{center}
 \def\svgwidth{0.5\columnwidth}
 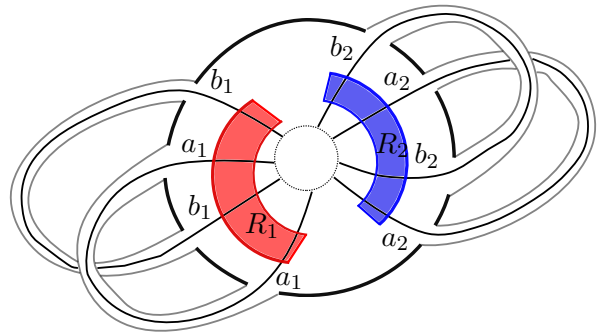
\end{center}

  \caption{The bottom subsurface}
	\label{fig:genus4bottomsubsurface}
\end{figure}

\begin{figure}[htb] 
\begin{center}
 \def\svgwidth{0.7\columnwidth}
 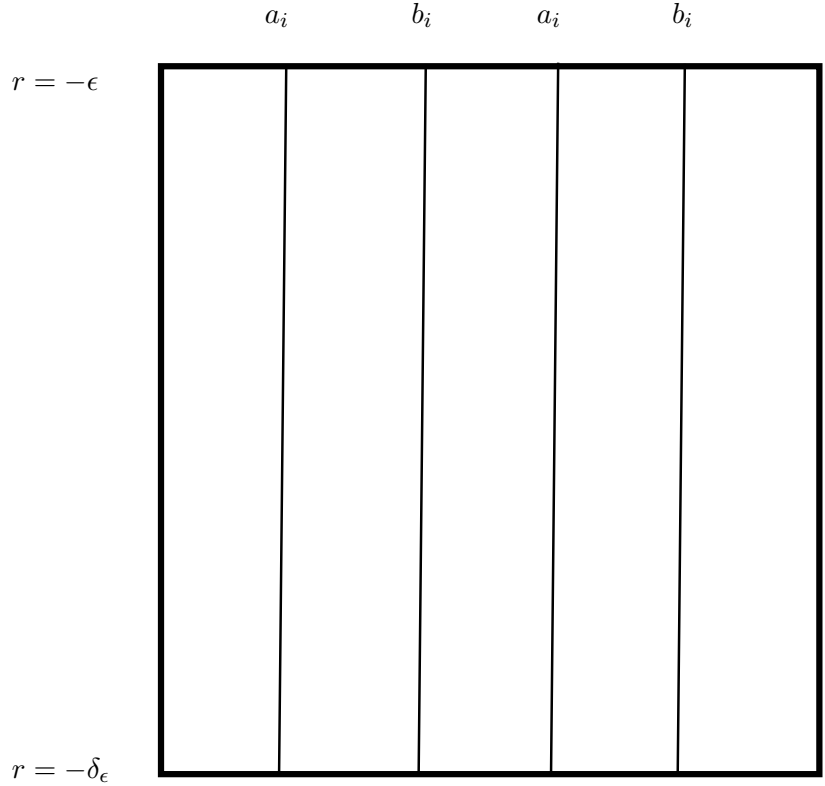
\end{center}
  \caption{The square $R$ (with ascending manifolds)}
	\label{fig:I1}
\end{figure}

Denote $R = R_i$. Inside $R$, there is a vector field $\frac{\partial}{\partial \theta}$ which is transversal to the gradient flow everywhere. We want the restrictions of our time dependent vector fields to the square $R$ to all be of the form 
\begin{equation}
X^e = \rho^e(t,r,\theta) \cdot \frac{\partial}{\partial \theta}, 
\end{equation}
for certain functions $\rho^e(t,r,\theta)$ which depend on the edge $e$. \\\\

We start by considering the product, $\mu^2$. We need to specify a time-dependent vector field for each edge of the unique trivalent tree 
\begin{equation} 
\begin{split} \left\{ \right. \end{split} 
\begin{split} \scalebox{.25}{\begin{tikzpicture}
        \node[Feuille](0)at(0.,1.5){};
        \node[Feuille](1)at(1.,0.){};
        \node[Feuille](2)at(2.,1.5 ){};
				\node[Feuille](3)at(1.,-1.5){};
        \draw[Blue](0)--(1);
        \draw[Green](2)--(1);
				\draw[Black](1)--(3);
\end{tikzpicture}}  \end{split} 
\begin{split} \left. \right\} \end{split} 
\begin{split}\quad = \quad \end{split} \begin{split}  \Stashefftree_2 \end{split}.
\end{equation}

We perturb the gradient trajectories on the upper-right (in green) by a small pushoff. We choose the scaling function $\rho(r,\theta)$ on the upper-left (i.e., the blue) edge so that the time-1 map would generate an isotopy sending the ascending manifolds marked in \ref{fig:I2} (in black) to the blue ones.
 
\begin{figure}[htb] 
\begin{center}
 \def\svgwidth{0.7\columnwidth}
 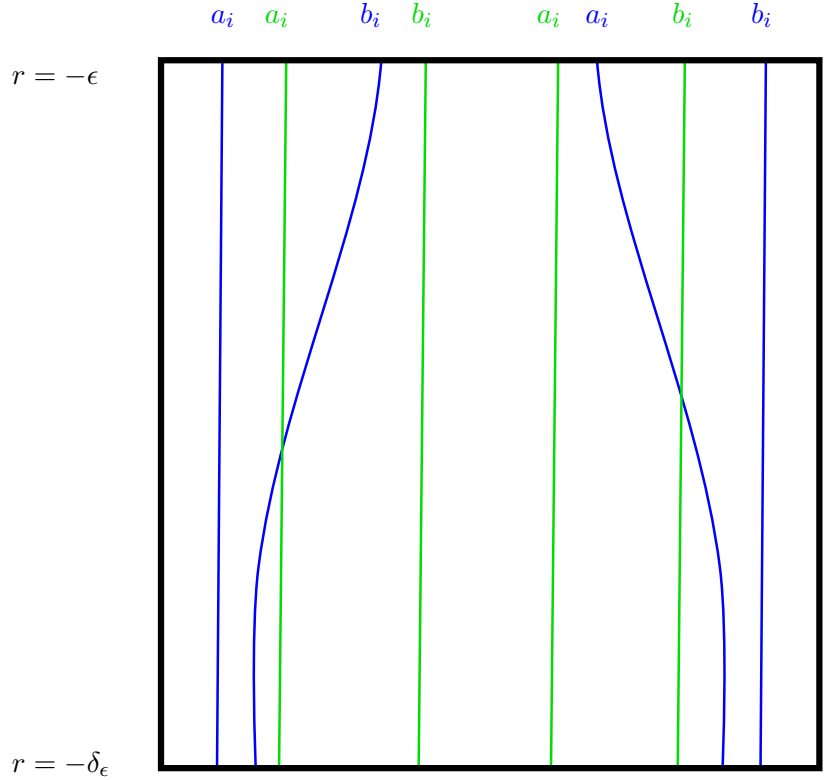
\end{center}
  \caption{The square $R$ (with perturbation data for the trivalent tree)}
	\label{fig:I2}
\end{figure}

Next, we consider the 1-parameter family necessary to define $\mu^3$. Starting from the corolla

\begin{equation} 
\begin{split} \scalebox{.25}{\begin{tikzpicture}
        \node[Feuille](0)at(-1.,1.5){};
				\node[Feuille](1)at(1.,1.5){};
				\node[Feuille](2)at(3.,1.5 ){};
        \node[Feuille](3)at(1.,0 ){};
        \node[Feuille](4)at(1.,-1.5 ){};
        \draw[Blue](0)--(3);
				\draw[Violet](1)--(3);
				\draw[Green](2)--(3);
				\draw[Black](3)--(4);
\end{tikzpicture}}  \end{split} 
\begin{split}\quad \in \quad \end{split} \begin{split}  \Stashefftree_3 \end{split}.
\end{equation}

We choose the scaling functions of the appropriate edges to give the isotopy in Figure \ref{fig:I3}. This extends to the left interval as well

\begin{equation} 
\begin{split} \left\{ \right. \end{split} 
\begin{split} \scalebox{.25}{\begin{tikzpicture}

				\node[Boite](x)at(1.,-3){  $0$ };
				\node[Boite](y)at(-11.,-3){$ \; -\infty \;$};
				
        \node[Feuille](0)at(-1 -6.,1.5){};
				\node[Feuille](1)at(1. -6,1.5){};
				\node[Feuille](2)at(3. -6,1.5 ){};
        \node[Feuille](3)at(1. -6,0 ){};
        \node[Feuille](4)at(1. -6,-1.5 ){};
				\node[Feuille](5)at(0.2 -6, 0.75 ){};
				
				\draw[Black](x)--(y);
        \draw[Blue](0)--(5);
				\draw[Violet](1)--(5);
				\draw[Blue](5)--(3);
				\draw[Green](2)--(3);
				\draw[Black](3)--(4);
\end{tikzpicture}}  \end{split} 
\begin{split} \left. \right\} \end{split} 
\begin{split}\quad \subset \quad \end{split} \begin{split}  \Stashefftree_3 \end{split}.
\end{equation}

We can extend the perturbation we choose over $[0,+\infty)$ as well, but in this case we must choose an interpolating family between the different $\rho^e(t,r,\theta)$ assigned to different edges. See Figures \ref{fig:I4} (purple $\mapsto $ green) and \ref{fig:I5} (purple $\mapsto $ blue) for an illustration. 

\begin{figure}[htb] 
\begin{center}
 \def\svgwidth{0.7\columnwidth}
 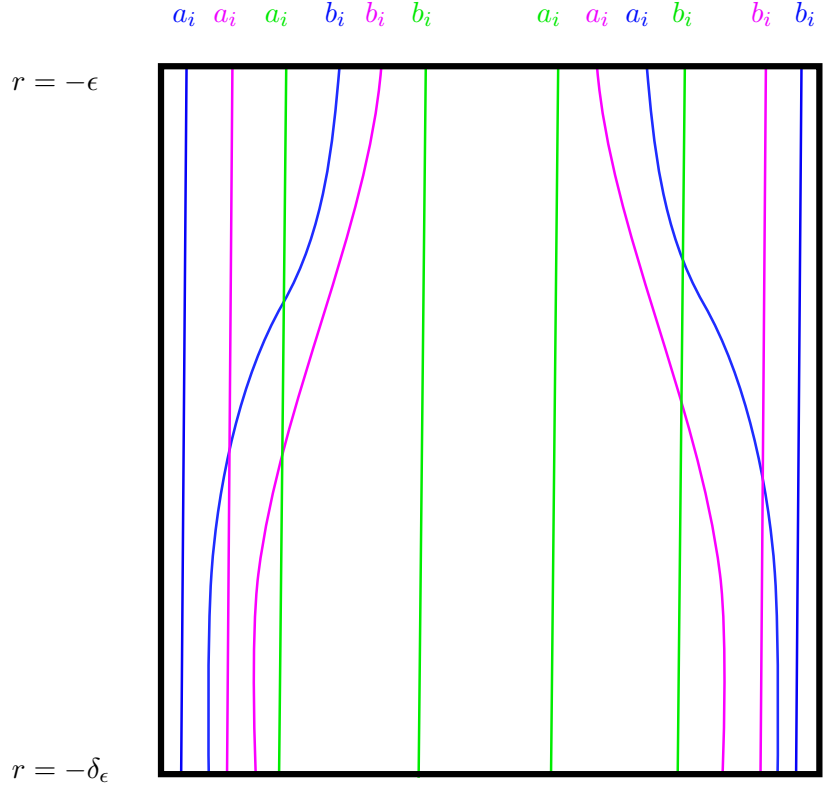
\end{center}
  \caption{The square $R$ (with perturbation data for the corolla)}
	\label{fig:I3}
\end{figure}

\begin{figure}[htb] 
\begin{center}
 \def\svgwidth{0.7\columnwidth}
 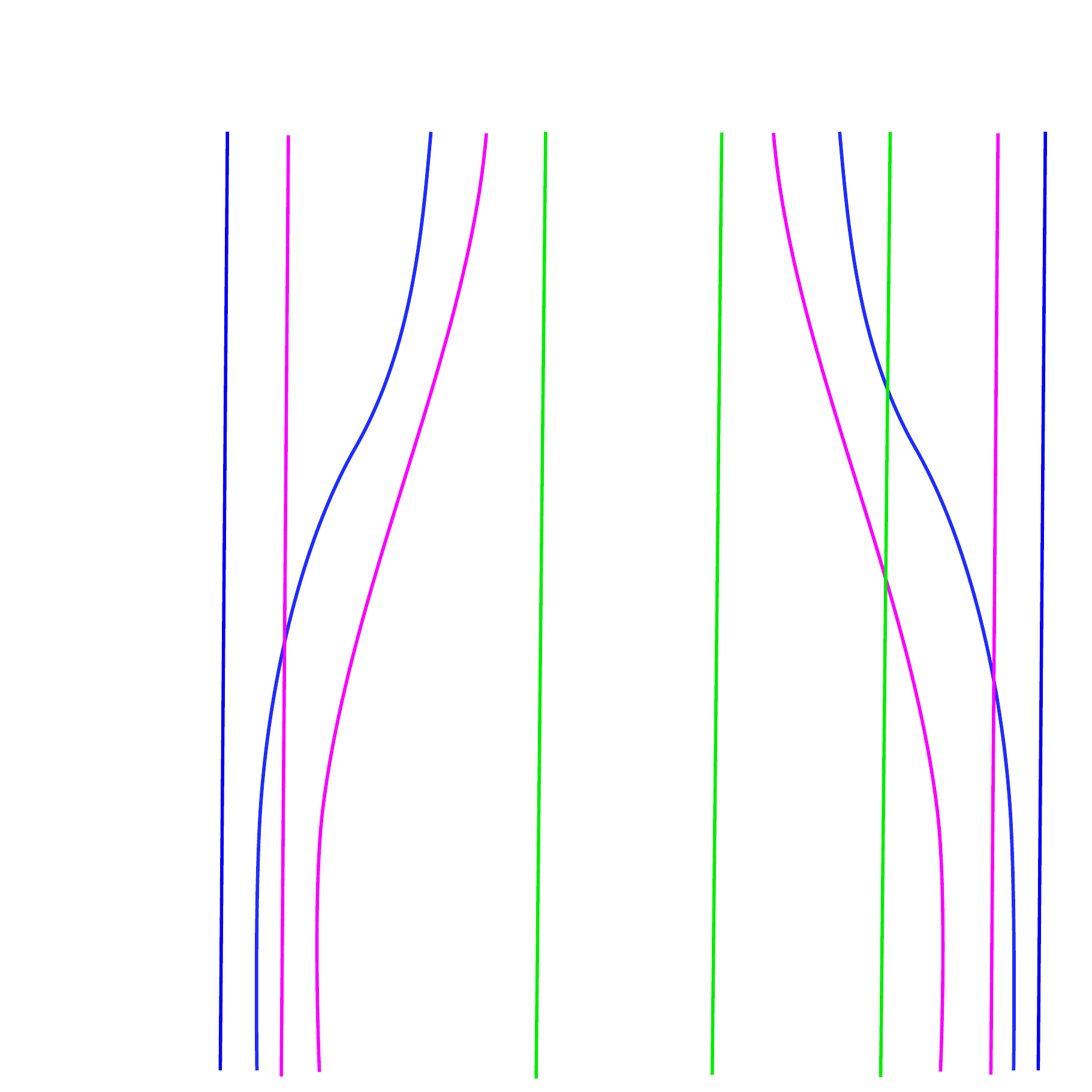
\end{center}
  \caption{The square $R$ (A family of perturbations over the right interval)}
	\label{fig:I4}
\end{figure}

\begin{figure}[htb] 
\begin{center}
 \def\svgwidth{0.7\columnwidth}
 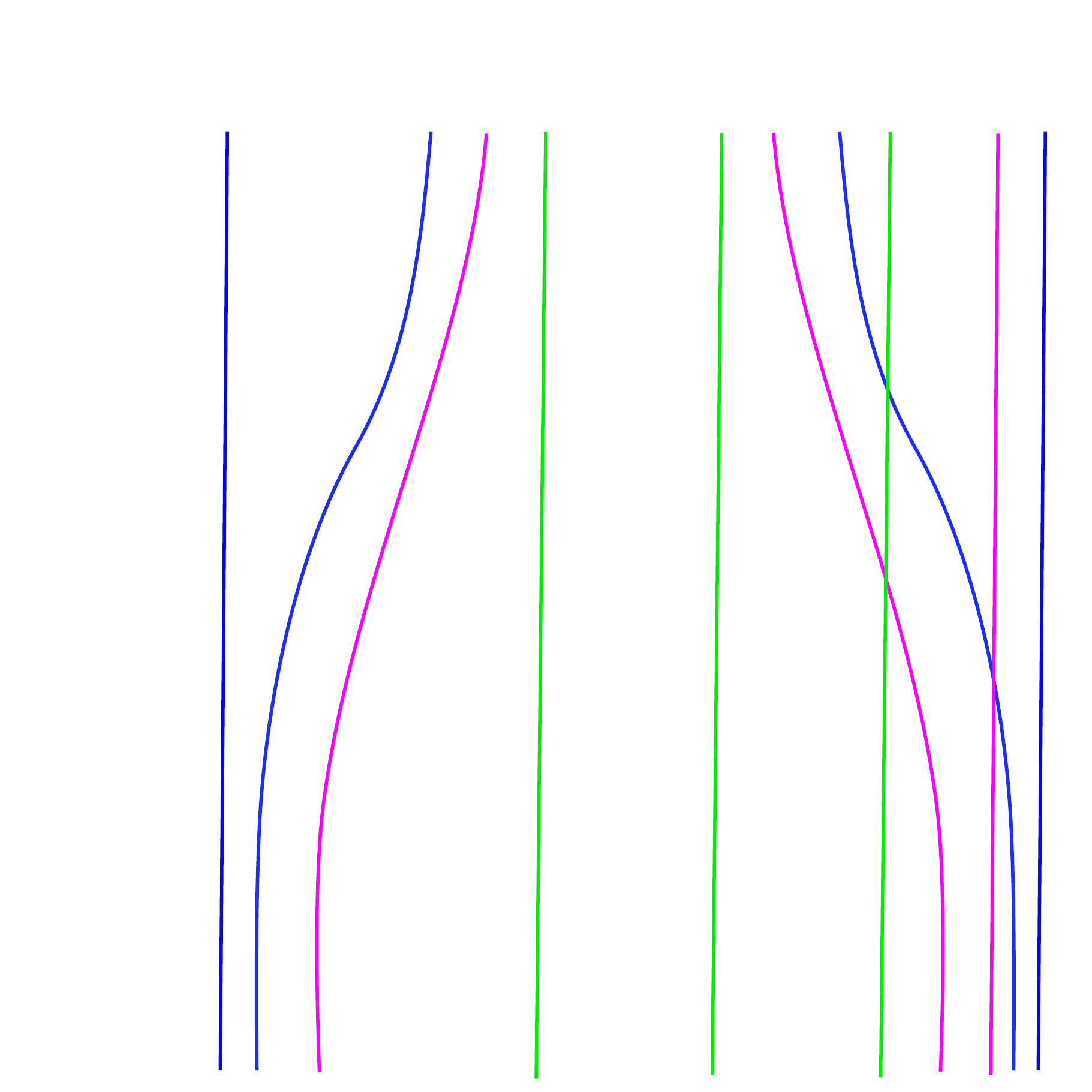
\end{center}
  \caption{The square $R$ (A family of perturbations over the right interval)}
	\label{fig:I5}
\end{figure}

The vector field $\frac{\partial}{\partial \theta}$ can be extended to the entire subsurface. We recall the following standard definitions:

\begin{definition}
Consider the partition of $\Sigma_-$ by the connected components of level-sets of $f$. The corresponding quotient space, denoted $\Gamma(f)$, has the structure of a one-dimensional CW-complex and is called the \textbf{Kronrod-Reeb graph} of $f$. Notice that the vertices of the graph are of the following three types. 
\begin{description}
\item[$\partial$-vertex] Connected components of the boundary.  
\item[$e$-vertex] Local extremes of $f$.
\item[$c$-vertex] Critical components of level-sets of $f$ (a connected component $\omega$ of a level-set $f^{-1}(c)$ is called \textbf{critical} if it contains a critical point of $f$; otherwise, $\omega$ is \textbf{regular}.) 
\end{description}
\end{definition}


\begin{definition}
We introduce the structure of a one-dimensional foliation with singularities $\Delta_f$ in the following way: a subset $\omega$ is a leaf of this foliation if and only if $\omega$ is either a critical point of f or a path-component of a set $f^{-1}(c) \setminus \crit(f)$ for some $c \in \R$.
\end{definition}

Let $F$ be a vector field tangent to the boundary. We denote by $\Phi_s(x)$ the flow generated by $F$, and $\mathrm{Fix}(\Phi)$ the set of fixed points of the flow. 

\begin{definition}
We say that $F$ is \textbf{locally linear} if for each $z \in \mathrm{Fix}(\Phi)$ there are local coordinates $(x,y)$ in which $z = (0,0)$ and $F$ is a linear vector field, i.e. $F(x,y) = V \cdot (x,y)^t$, where $V$ is a constant 2-by-2 matrix. 
\end{definition}

\begin{lemma}
There exists an locally linear vector field $F$ on the surface whose trajectories are
precisely the leaves of $\Delta_f$. In the standard coordinates in every $R_i$, 
\begin{equation}
F = \frac{\partial}{\partial \theta}.
\end{equation}
\end{lemma}
\begin{proof}
Consider $G$, the Hamiltonian vector field of $f$. Then the trajectories of $G$ are precisely the leaves of $\Delta_f$. Moreover, since $f$ is constant on the connected components of the boundary, we conclude that $G$ is tangent to the boundary. \\
Let $z$ be a critical point of f and $(x_1, x_2)$ be a local Morse chart centered at $z  = (0,0)$ such that 
\begin{equation}
f(x_1, x_2) = f(0,0) + \frac{1}{2} x_1^2 - \frac{1}{2} x_2^2. 
\end{equation}
Then we define a vector field near $z$ by setting: 
\begin{equation}
F^z(x, y) = (x_2, -x_1).
\end{equation}
Evidently, $F^z$ is linear and its trajectories are subsets of leaves of the foliation. Moreover, $F^z$ is collinear with $G$. Now using the partition of unity we glue $G$ with all the local vector fields $\left\{F^z\right\}_{z \in \crit(f)}$, so that the resulting vector field $F$ will be locally linear.
\end{proof}

Using the vector field $F$ we extend the vector fields we have described for $R_1,R_2$ to the entire collar $Z_-$ and then the entire subsurface $\Sigma_-$ so that we will have no intersections involving critical points of degree $1$ in $\Sigma_- \setminus \left(R_1 \cup R_2\right)$ at all. The only additional thing that we need to pay attention to is a coherent choice of a compactly supported, time-dependent vector field in the small cap around the minima $m_-$. These choices can be done coherently (see the discussion in subsection \ref{subsec:transversalityandperturbations}), and extended to the higher Stasheff polytopes to give an $A_\infty$-algebra. Since the mapping torus $(\Sigma_-)_\phi$ is simply the product $\Sigma_- \times S^1$, we construct our perturbations as a sum, acting independently in the base and fiber directions. In the next section, we construct an extension of this perturbation datum to the entire mapping torus and obtain an $A_\infty$-algebra, denoted

\begin{equation}
\EuScript{C} = (CM^\bullet(\Sigma_\phi;\Q),\mu_\EuScript{C}^1,\mu_\EuScript{C}^2,\mu_\EuScript{C}^3,\ldots).
\end{equation}
We conclude the section by noting the following fact. \\

\textbf{Observation.} \emph{The perturbations we have constructed already ''spend" all the necessary intersections in $R_1$ and $R_2$}. \\

To make this a little more precise, consider the spaces of Morse gradient trees involved in computing any higher product $\mu^d(\ldots,\cdot,\ldots)$ whose inputs are all in $\left\{a_1,b_1,a_2,b_2\right\}$. In the generic situation, the image of the Morse pseudocycle of the saddle points is homeomorphic to $S^1$, and we can smooth it out and treat it like a closed, differentiable 1-manifold (see, e.g., the discussion in the beginning of \cite{MR2850125} regarding isotopies and smooth vs. topological manifolds in the two-dimensional case). The intersections of the compactified ascending manifold with the subsurface form arcs $\overline{W^u}(a_i) \cap \Sigma_-$, $\overline{W^u}(b_j) \cap \Sigma_-$ ($i=1,2$). Then a fter we turn on the perturbations the arcs \textbf{no longer bound half-bigons} (\cite[Section 1.2.7]{MR2850125}.) \\\\
As another way to think of this, we note that the product of the evaluation maps at the regular level set $H = f^{-1}(-\epsilon)$ from definition \ref{def:evH} gives $d$ pairs of points on 
\begin{equation}
S^1 \times S^1 \setminus \left\{(x,x) \: \big| \: x \in S^1\right\} 
\end{equation}
for every gradient tree. The perturbations act on this configuration and ''untangle" it, i.e., if we assign to every pair of points the short arc on the circle between them. Then after perturbation arcs that belong to the two different pairs must either contain one another or be disjoint. See Figures \ref{fig:markedcircle2a}--\ref{fig:markedcircle2b} and \ref{fig:markedcircle3a}--\ref{fig:markedcircle3b} for some examples in the case $d=3$. 

\begin{figure}[htb]
  \begin{subfigure}[b]{.45\linewidth}
		\centering
				\fontsize{0.25cm}{1em}
			\def\svgwidth{4cm}
			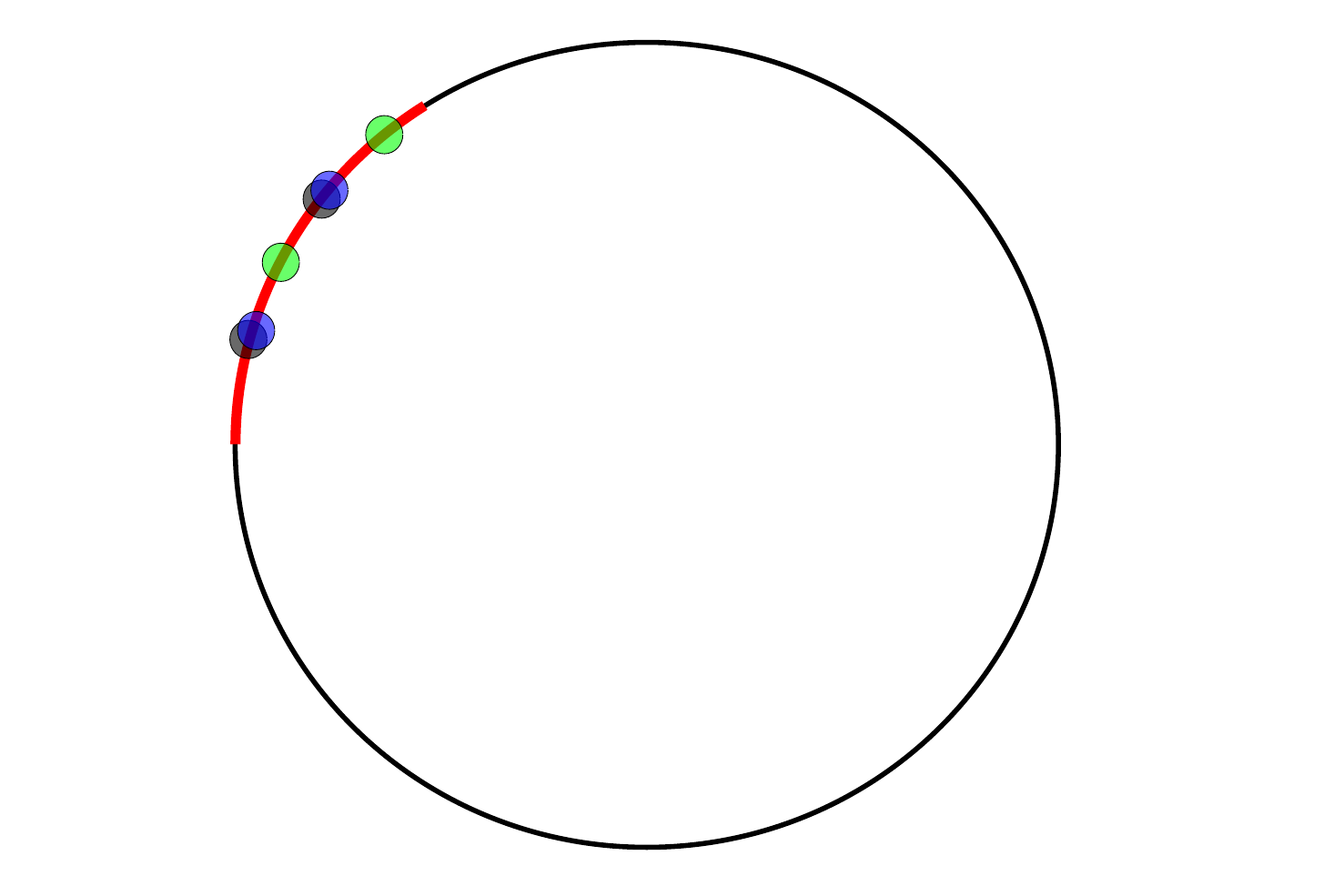
			\caption{Before perturbations (black and blue coincide)}
			\label{fig:markedcircle2a}
  \end{subfigure}\hfill
  \begin{subfigure}[b]{.45\linewidth}
	\centering
					\fontsize{0.25cm}{1em}
			\def\svgwidth{4cm}
			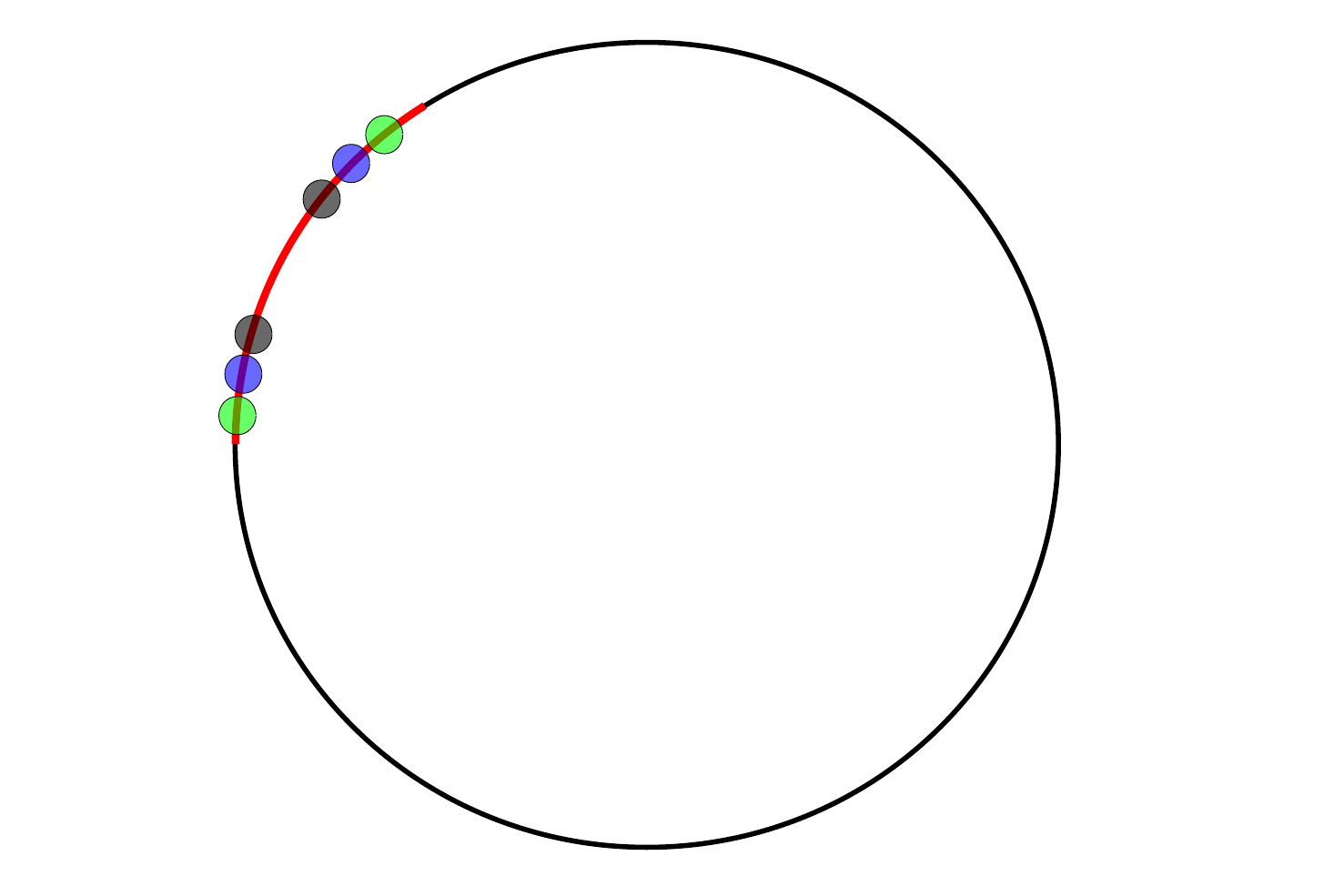
			\caption{After perturbations}
			\label{fig:markedcircle2b}
			\vspace*{4mm}
  \end{subfigure} 
	\caption{A gradient tree with asymptotics $(a_1,a_1,b_1)$}
\end{figure}	

\begin{figure}[htb]
  \begin{subfigure}[b]{.45\linewidth}
		\centering
				\fontsize{0.25cm}{1em}
			\def\svgwidth{4cm}
			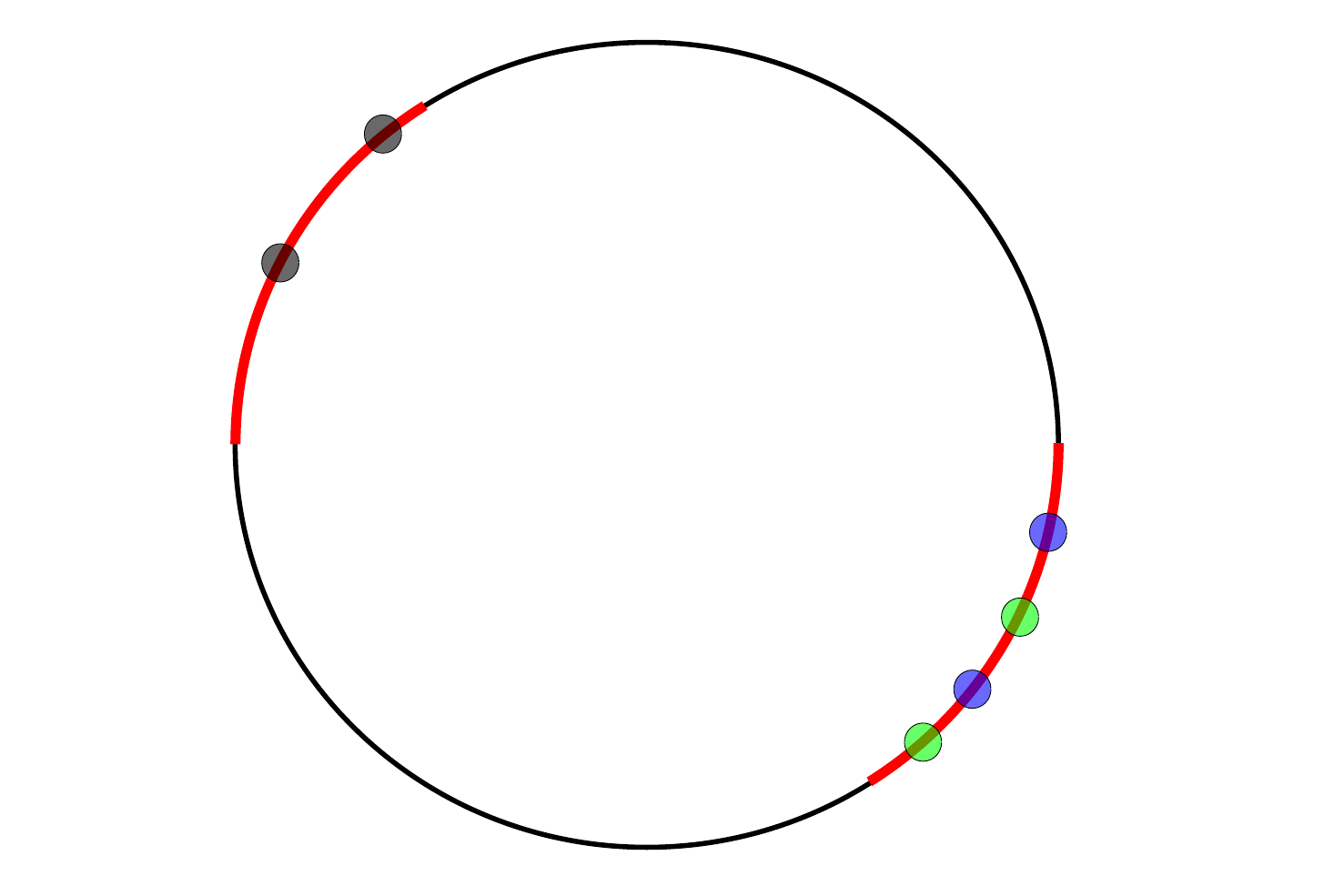
			\caption{Before perturbations (black and blue coincide)}
			\label{fig:markedcircle3a}
  \end{subfigure}\hfill
  \begin{subfigure}[b]{.45\linewidth}
	\centering
					\fontsize{0.25cm}{1em}
			\def\svgwidth{4cm}
			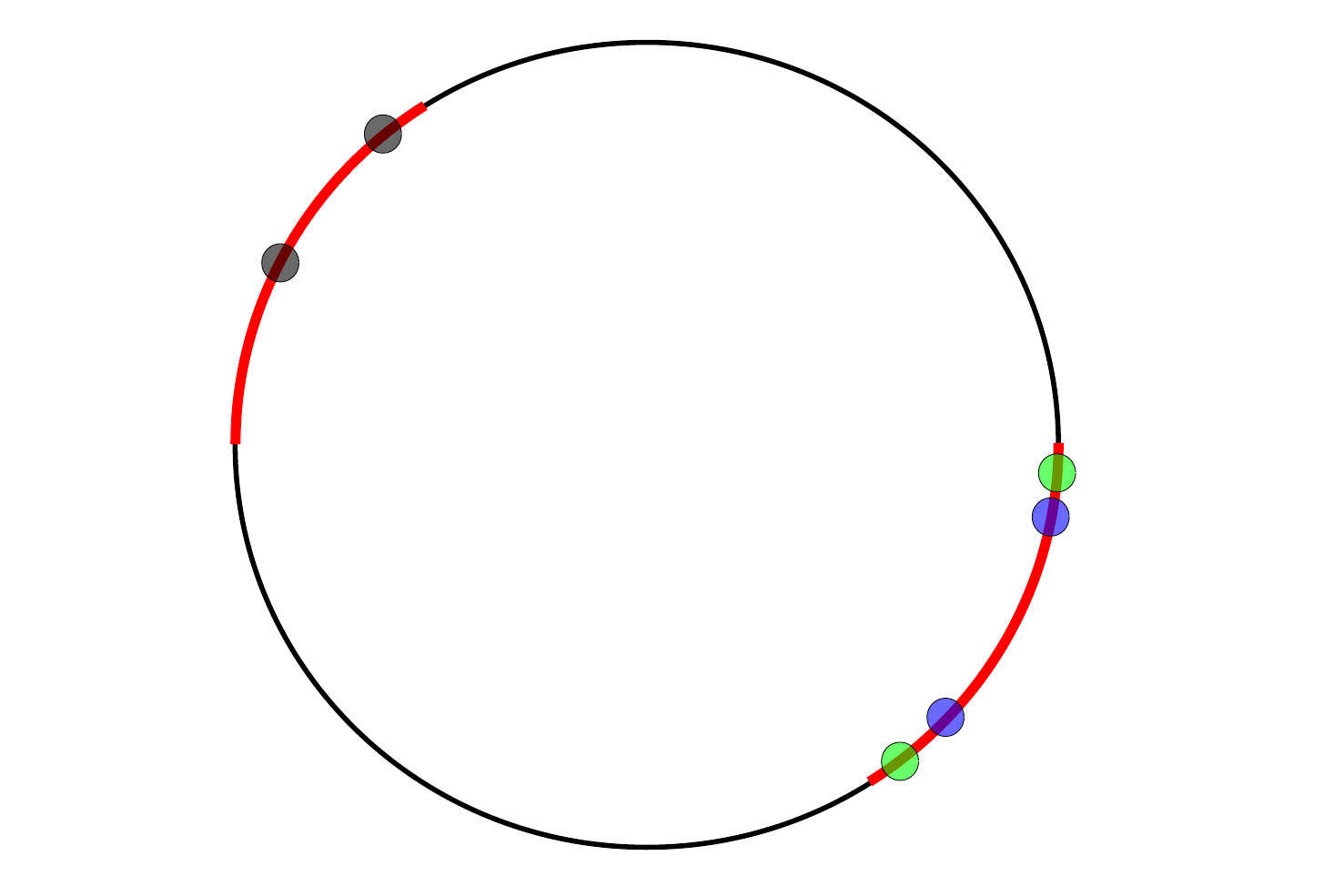
			\caption{After perturbations}
			\label{fig:markedcircle3b}
			\vspace*{4mm}
  \end{subfigure} 
	\caption{A gradient tree with asymptotics $(b_1,b_2,a_2)$}
\end{figure}	


\subsection{A standard model for the upper subsurface}

The upper subsurface is a genus $2$ surface $(\Sigma_+,\omega)$ with one boundary component. We fix a cap $D_{2 \epsilon} \hookrightarrow \Sigma_+$ around the maxima (see figure \ref{fig:surfacesigma2plus} below). Denote

\begin{equation}
\begin{split}
D_{\epsilon} &= \left\{(r,\theta) \in D_{2 \epsilon} \: \big| \: |r| \leq \epsilon \right\}, \\
S_{2 \epsilon} &= \partial D_{2 \epsilon} \: \: , \: \: S_{\epsilon} = \partial D_{\epsilon}.
\end{split}
\end{equation}

\begin{figure}[htb] 
\begin{center}
 \def\svgwidth{0.8\columnwidth}
 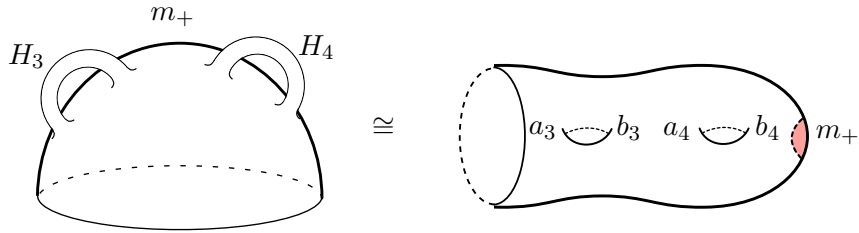
\end{center}
  \caption{The subsurface $\Sigma_+$ with the cap $D_{2 \epsilon}$ (in red)}
	\label{fig:surfacesigma2plus}
\end{figure}

Outside the smaller cap $D_{\epsilon}$, the story is very similar to what we've discussed in the previous section: the place of $Z_-$ is taken by the annulus $D_{2 \epsilon} \setminus D_{\epsilon}$. We construct perturbations in small tubes around $\overline{W^u}(a_i) \cap (\Sigma_+ \setminus D_{\epsilon})$ and $\overline{W^u}(b_j) \cap (\Sigma_+ \setminus D_{\epsilon})$ and extend them to the rest of the $\Sigma_+ \setminus D_{\epsilon}$. In $D$ itself however, we extend by a generic time-dependent vector field (which integrates to a point-pushing isotopy). This could have been a potential source of trouble, but since the points on $S_{\epsilon}$ are not ''entangled", they can not bound a half-bigon, so the algebraic intersection number between them is zero. This is also easy to see by direct verification, as is demonstrated in figure \ref{fig:cap1} for $d=2$. More generally, we note that the same true for every critical point in the surface (in particular, the saddle points), so we can modify our existing perturbation datum by adding a small, generic, compactly supported coherent system of time-dependent vector fields around each critical point. This works well because we have also chosen our initial shifts (i.e., consider the \emph{upper boundary} of figure \ref{fig:I5}) so that the various configurations of pairs of points formed by perturbed trajectories would not be entangled. This would be important later on. 

\begin{figure}[htb]
\begin{minipage}[b]{.4\linewidth}
\centering
			\fontsize{0.4cm}{1em}
			\def\svgwidth{5cm}
			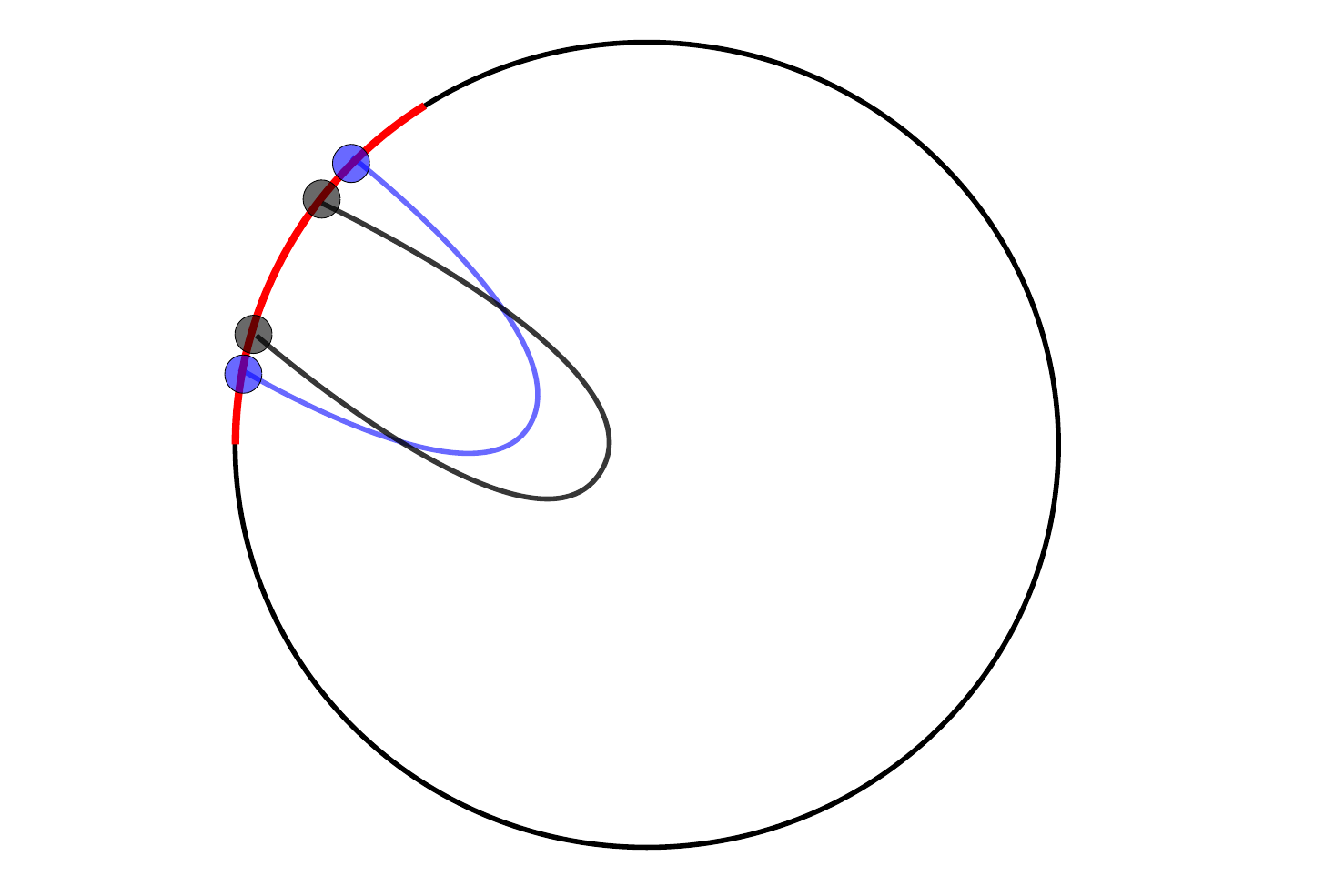
\subcaption{}
\end{minipage}%
\begin{minipage}[b]{.4\linewidth}
\centering
			\fontsize{0.4cm}{1em}
			\def\svgwidth{5cm}
			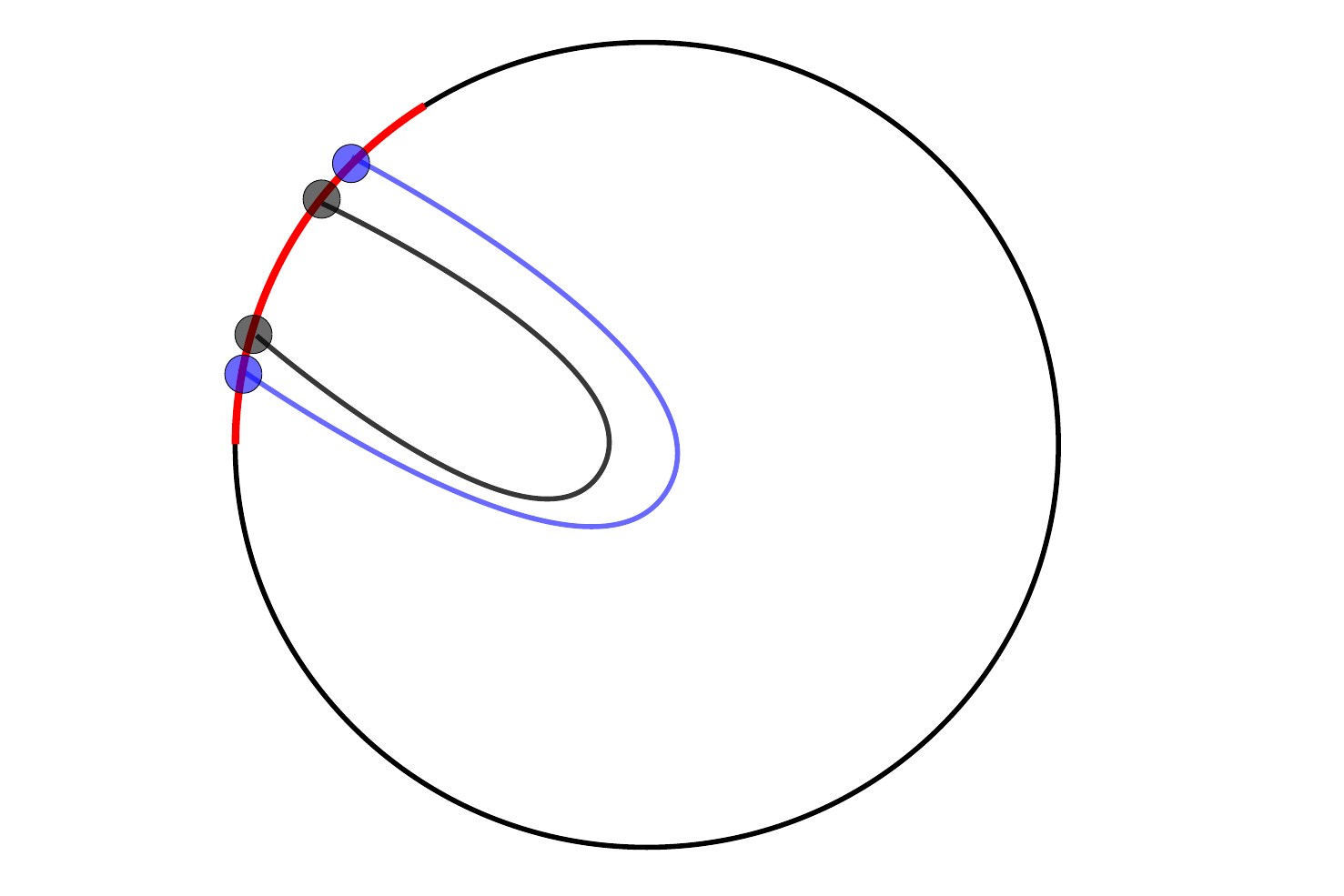
\subcaption{}
\end{minipage}
\begin{minipage}[b]{.4\linewidth}
\centering
			\fontsize{0.25cm}{1em}
			\def\svgwidth{5cm}
			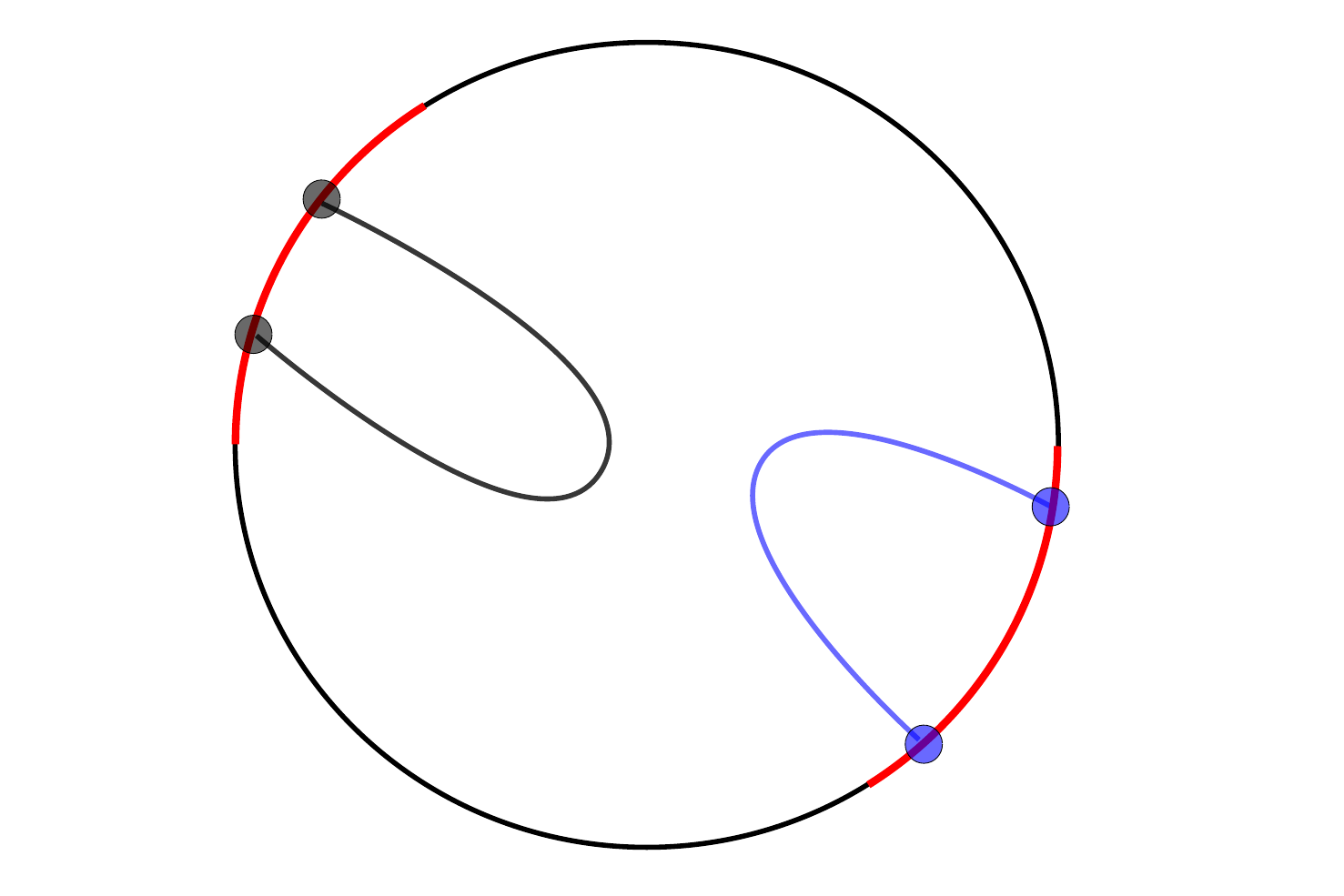
\subcaption{}
\end{minipage}
\begin{minipage}[b]{.4\linewidth}
\centering
			\fontsize{0.25cm}{1em}
			\def\svgwidth{5cm}
			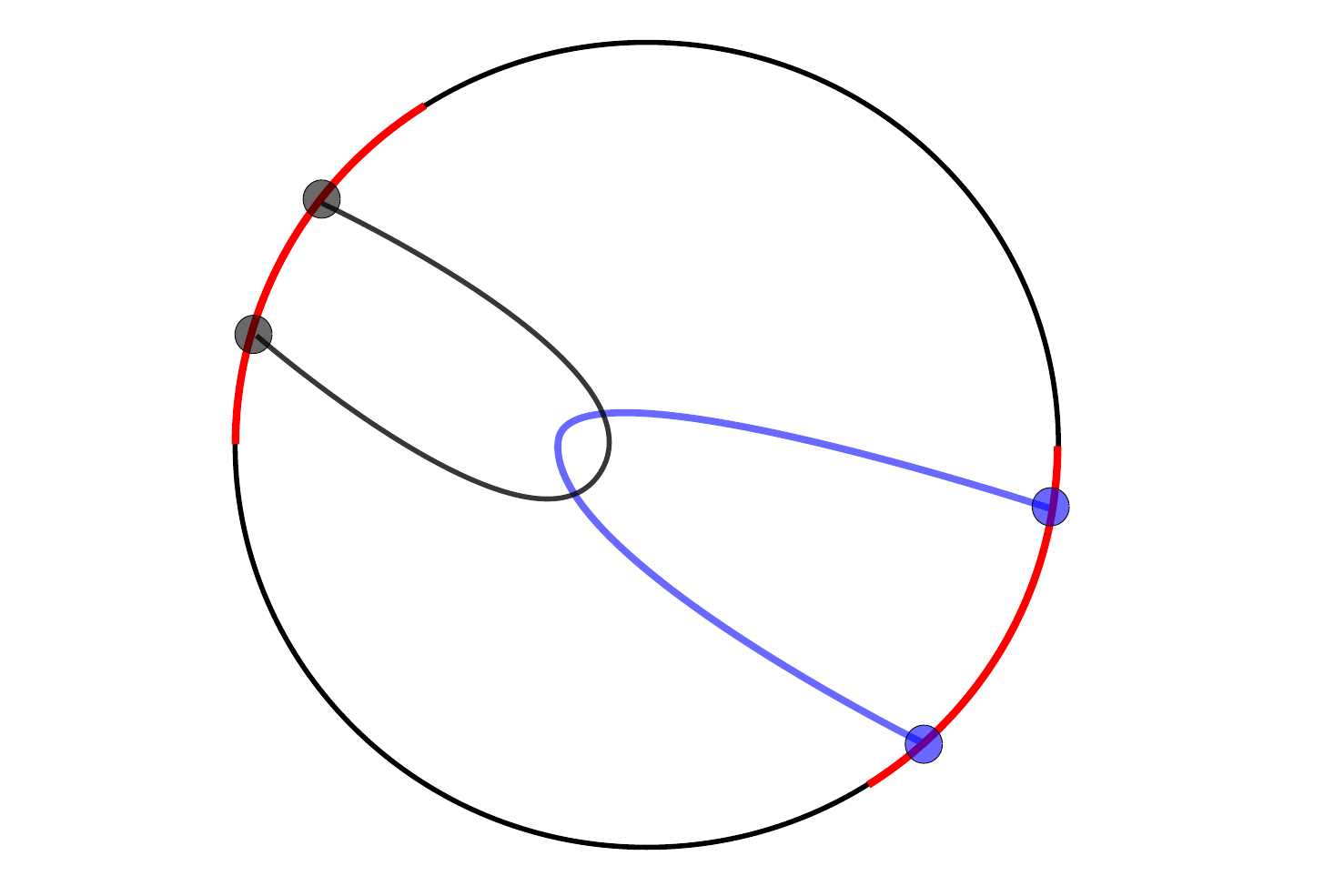
\subcaption{}
\end{minipage}
\caption{There are no contributions to $\mu^2$ from intersections occuring in the cap}\label{fig:cap1}
\end{figure} 

\subsection{Differentials and products in the mapping torus} \label{subsec:differentialsandproducts}

Much like $\Sigma_-$, the mapping torus of the upper subsurface $(\Sigma_+)_\phi$ is just the product $\Sigma_+ \times S^1$. However, if we want to extend the previous discussion, our perturbations would have to accomodate the gradient trajectories coming from $(\Sigma_-)_\phi$. Solving the ODE \eqref{eq:morseflow} with respect to the metric \eqref{eq:metricmodel}, we see that as the base parameter $t \in S^1$ rotates, Morse flow lines passing through the ''neck" $(N_\gamma)_\phi$ must undergo a full rotation as well: 

\begin{figure}[htb]
\begin{minipage}[b]{.4\linewidth}
\centering
			\fontsize{0.4cm}{1em}
			\def\svgwidth{5cm}
			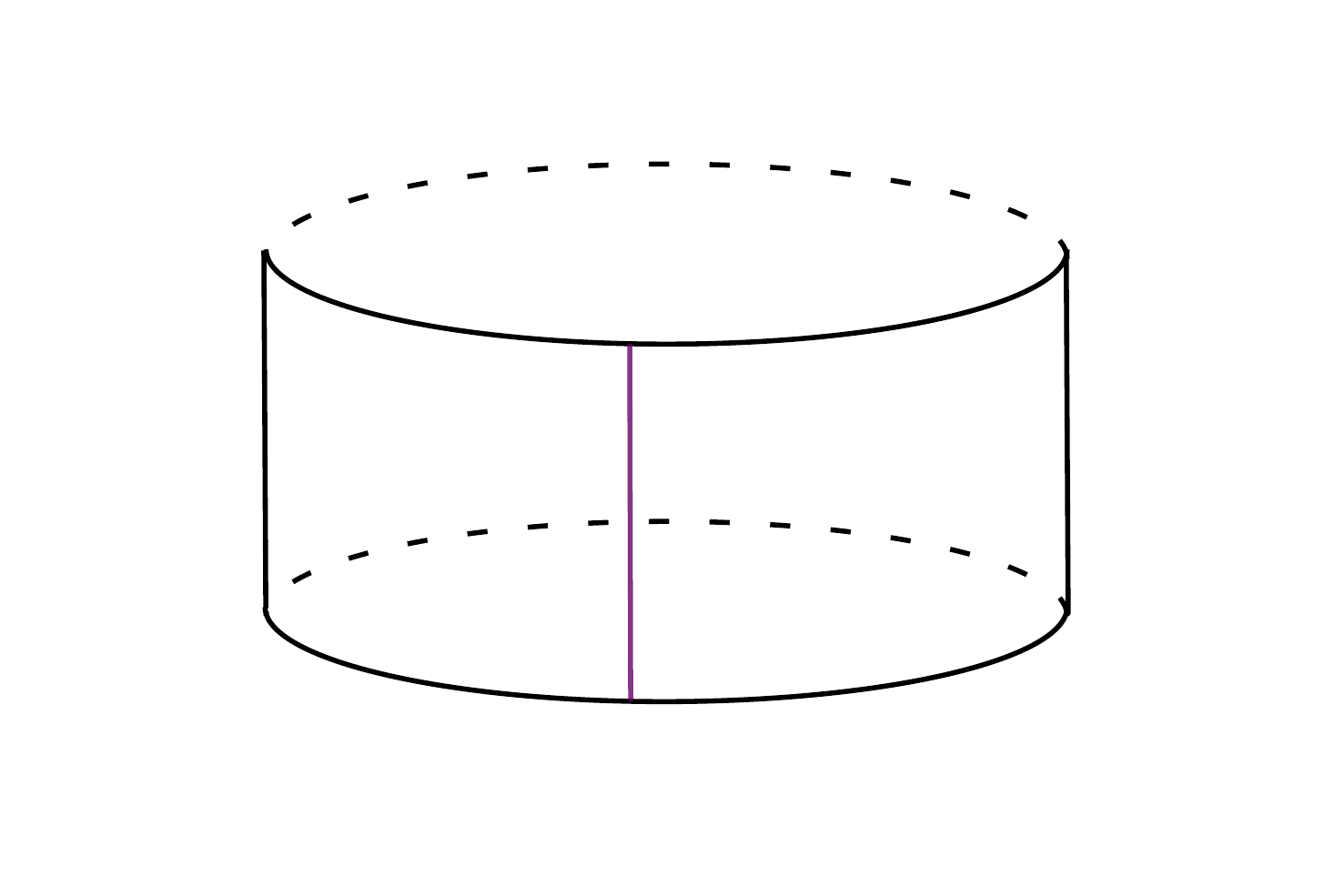
\subcaption{$t=0$ (no rotation)}
\end{minipage}%
\begin{minipage}[b]{.4\linewidth}
\centering
			\fontsize{0.4cm}{1em}
			\def\svgwidth{5cm}
			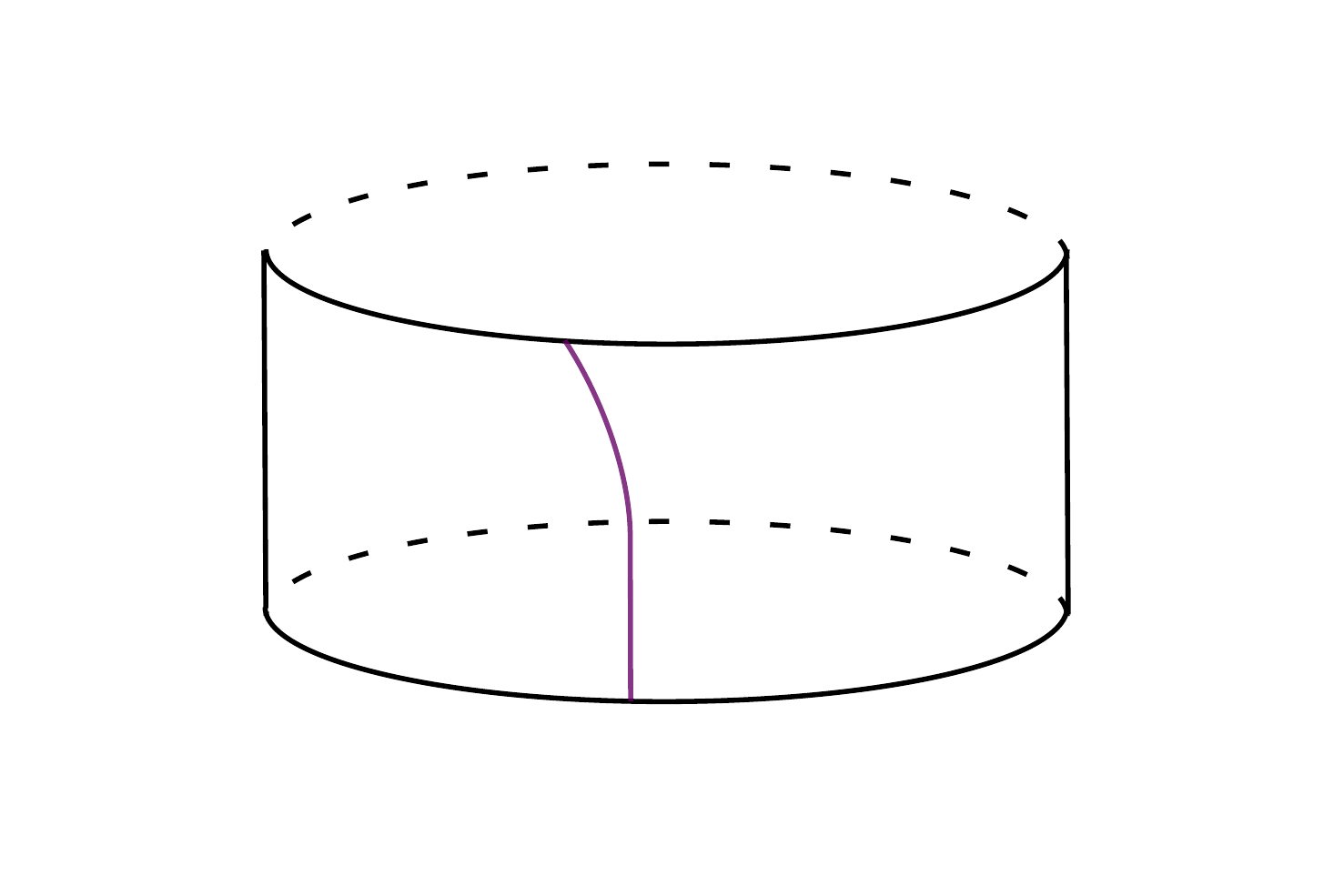
\subcaption{$t>0$ is small}
\end{minipage}
\begin{minipage}[b]{.4\linewidth}
\centering
			\fontsize{0.25cm}{1em}
			\def\svgwidth{5cm}
			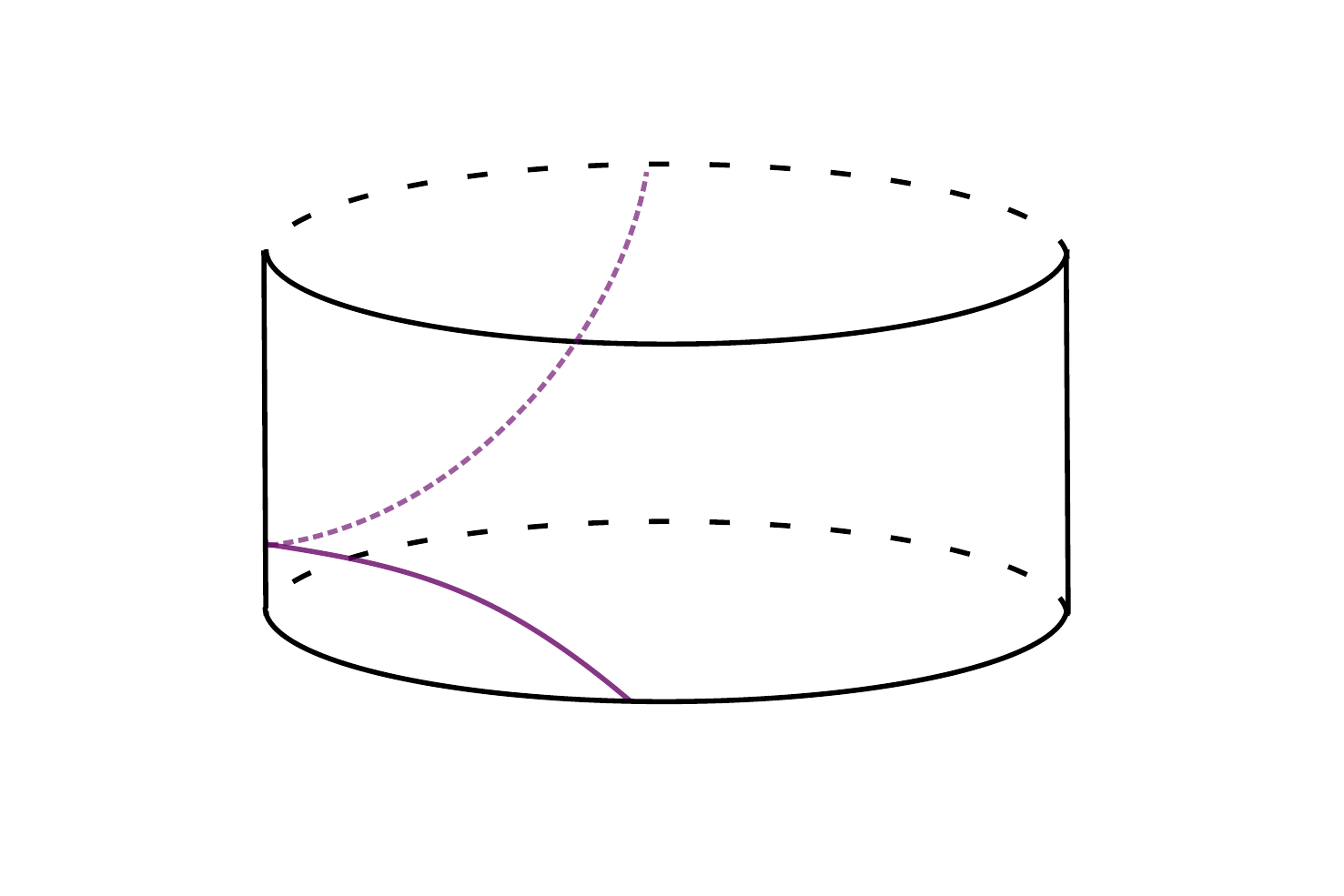
\subcaption{$t \approx 1/2$}
\end{minipage}
\begin{minipage}[b]{.4\linewidth}
\centering
			\fontsize{0.25cm}{1em}
			\def\svgwidth{5cm}
			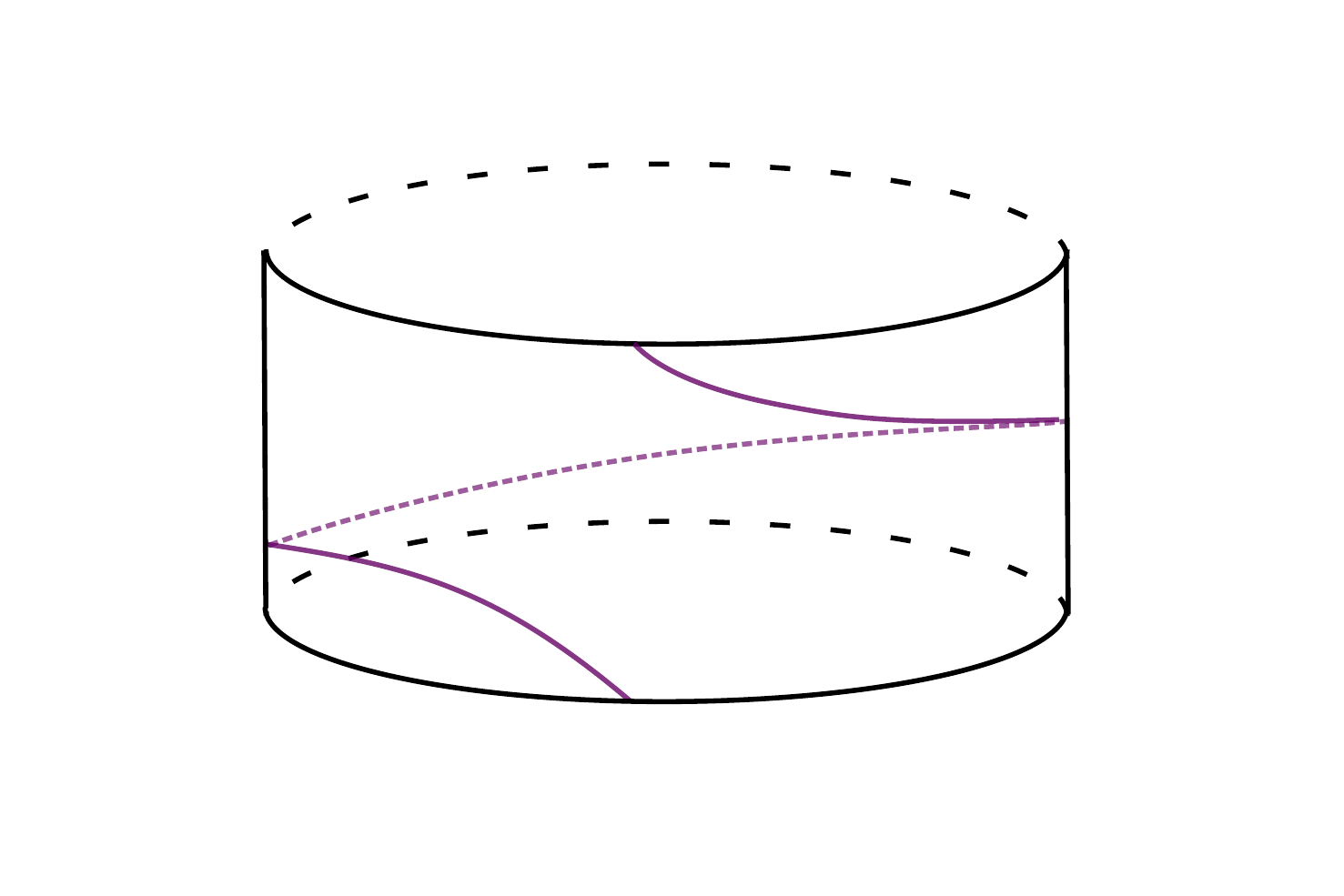
\subcaption{$t=1$ (a full Dehn twist)}
\end{minipage}
\caption{The effect of $(N_\gamma)_\phi$ on gradient trajectories (in purple)}\label{fig:rotation1}
\end{figure} 

\begin{remark} If we glue the top and bottom boundary tori of $(N_\gamma)_\phi$, we will get a nilmanifold. It is well known that such manifolds (except the tori) are never formal. From that perspective, the existence of non-trivial Massey products is to be expected.
\end{remark}

We would like to construct a perturbation datum on the mapping tori by taking an $S^1$-parametrized family of perturbation datums on the fiber, and take the sum with the perturbation datum of the base. But there is a complication: in a 1-parameter family, the rotation in $(N_\gamma)_\phi$ causes the ascending flow lines to necessarily intersect the descending manifolds of the saddle points in $\Sigma_+$ (see figure \ref{fig:rotation2}).
\begin{figure}[htb] 
\begin{center}
 \def\svgwidth{0.7\columnwidth}
 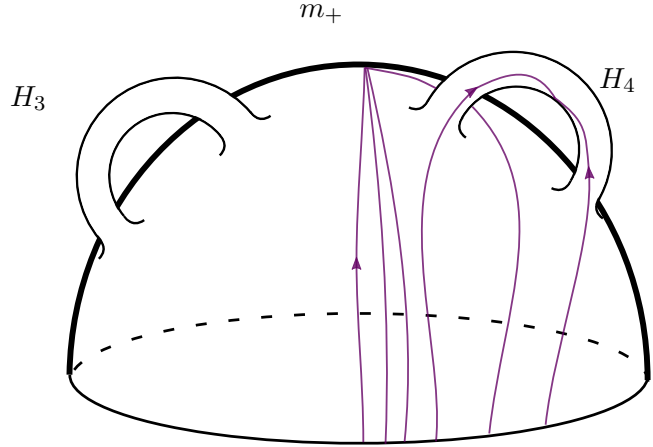
\end{center}
  \caption{The rotated gradient flow lines (in purple) give rise to new differentials}
	\label{fig:rotation2}
\end{figure}

However, a more careful examination of the local model reveals that we meet every critical point in the upper subsurface \emph{twice} with \emph{opposite signs}. Thus, the new differentials cancel out.  \\

Note that since we have shown that the Morse function on the mapping torus is minimal, $\mu^2$ coincides with the product on cohomology. It follows from Lemma \ref{lem:tau1vanish} that there is an isomorphism of graded associative algebras

\begin{equation}
(CM^\bullet(\Sigma_\phi;\Z),\tilde{\mu}^2) \iso (CM^\bullet(\Sigma;\Z),\mu^2)[\mathfrak{t}]/(\mathfrak{t}^2)
\end{equation}

where $\mathfrak{t}$ is a formal parameter of degree one.

\subsection{Higher multiplication} \label{subsec:highermultiplicationinmorsetheory}
Given a pair a standard pair of critical 1-points coming from $\Sigma$, we constructed the perturbation datum for $\mu^3$ to ensure that the intersection between the ascending manifolds\footnote{The intersection corresponds to the unique vertex in a Stasheff tree $T \in \Stashefftree_3$ connected to two leaves.} can only happen in $(\Sigma_-)_\phi \iso \Sigma_- \times S^1$, or in the $D \times S^1$, where $D$ is the cap of the maxima. Since we are only interested in the terms of the form $\mathfrak{t} \cdot z_0$ in $\mu^3(z_3,z_2,z_1)$ where $\deg(z_i)=1$, we note that only the former type of intersections can contribute (the critical 1-points lie below the cap.) Note that the 1-manifold in which two perturbed gradient lines meet in $(\Sigma_-)_\phi$ is either: empty, a (generic) section 
\begin{equation}
\left\{p\right\} \times S^1 \subset \Sigma_- \times S^1 
\end{equation}
where $p \in \Sigma_-$ is a generic point, or two generic sections with opposite orientations. If we have any critical point $z$, in the upper-subsurface, the intersection with the descending manifold of 
\begin{equation}
\mathfrak{t} \cdot \overline{z} := \begin{cases}
\mathfrak{t} \cdot b_i &\mbox{if } z = a_i \\ 
\mathfrak{t} \cdot a_j & \mbox{if } z = b_j 
\end{cases}
\end{equation}
gives is a 1-manifold 
\begin{equation}
\left\{q\right\} \times S^1 \subset \Sigma_+ \times S^1 
\end{equation}
where $p \in \Sigma_+$ is a generic point, and vanishes for any other descending manifold of the form $\mathfrak{t} \cdot z'$ with $\deg(z')=1$ and $z' \neq \overline{z}$. \\

The proof of Proposition \ref{prop:mainterm2} now follows immediately once we make the following observation: as they rotate in the neck, perturbed gradient line emanating from a generic section in the bottom sub-surface \emph{wraps around the entire upper-subsurface once} and must meet a generic section of the upper-subsurface at a unique point. See Figure \ref{fig:newmu3}. 

\begin{figure}
\begin{center}
 \def\svgwidth{0.6\columnwidth}
 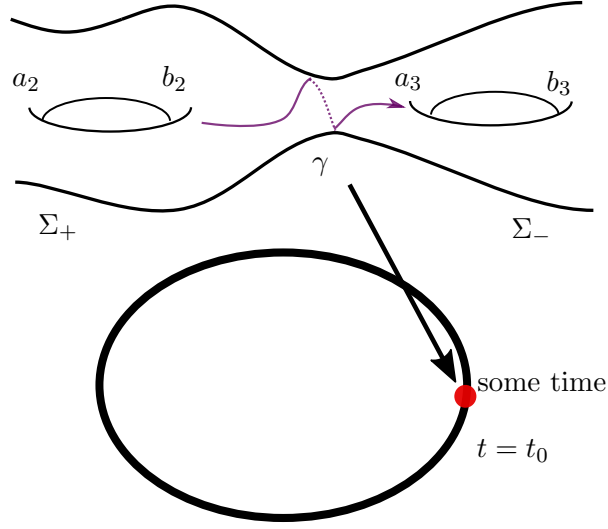
\end{center}
  \caption{At some time $t_0 \in S^1$, a finite length Morse trajectory connects the two sections.}
\label{fig:newmu3}
\end{figure}

\begin{proof}[Proof of \ref{claim:nontrivialmatrixmasseyproduct}]
The vanishing condition holds
\begin{equation}
\begin{pmatrix}
 a_1 &  a_1 \\
0 & 0
\end{pmatrix} \cdot \begin{pmatrix}
 b_1 & 0 \\
- b_2 & 0
\end{pmatrix} = \begin{pmatrix}
( m_-) - ( m_-) & 0 \\
0 & 0
\end{pmatrix} 
\end{equation}
\begin{equation}
\begin{pmatrix}
 b_1 & 0 \\
- b_2 & 0
\end{pmatrix} \cdot \begin{pmatrix}
 b_1 &  b_3 \\
0 & 0
\end{pmatrix} = \begin{pmatrix}
0 & 0 \\
0 & 0
\end{pmatrix}  
\end{equation}
so the Massey product is well-defined. Second, because the Morse function $\tilde{f}$ is minimal, the Morse $A_\infty$-algebra is minimal as well (i.e., $\tilde{\mu}^1=0$) and there are no bounding cochains. So we only need to compute the main term. In fact, note that it is enough to consider only the $\mathfrak{t}$-part of the Massey product and show that it is nontrivial. Let $\Theta$ be the $\mathfrak{t}$-part of

\begin{equation}
\mu^3(\begin{pmatrix}
a_1  & a_1 \\
0 & 0
\end{pmatrix}
,
\begin{pmatrix}
 b_1   & 0 \\
- b_2   & 0
\end{pmatrix}
,
\begin{pmatrix}
 b_1  &   b_3  \\
0 & 0
\end{pmatrix})
\end{equation}
which is a 2-by-2 matrix whose entries are cohomology classes of degree $2 = 1+1+1-1$ (thus of the form $\mathfrak{t} \cdot z$ with $\deg(z)=1$.) Unwinding the definition, 
\begin{equation}
\begin{split}
&\Theta_{i_3 i_0} = \sum^{2}_{i_1,i_2=1} \mu^3 \left( \begin{pmatrix}
 a_1  & a_2 \\
0 & 0
\end{pmatrix}_{i_3 i_{2}} ,\begin{pmatrix}
b_1   & 0 \\
- b_2   & 0
\end{pmatrix}
_{i_{2} i_{1}}, \begin{pmatrix}
b_1   & b_3   \\
0 & 0
\end{pmatrix}_{i_1 i_0} \right)
\end{split}
\end{equation}
for $1 \leq i_3,i_0 \leq 2$. We can immediatly disregard all sequences $(i_3,i_2,i_1,i_0)$ where $i_3=2$ or $i_1=2$ because the result would be zero. Thus, 

\begin{equation}
\Theta_{21} = \Theta_{22}=0. 
\end{equation}

We can compute the $\mathfrak{t}$-part of the following products using Proposition \ref{prop:mainterm2}: 
 
\begin{equation}
\begin{split}
&\mu^3(\begin{pmatrix}
 a_1 &  a_2 \\
0 & 0
\end{pmatrix}
,
\begin{pmatrix}
 b_1 & 0 \\
- b_2 & 0
\end{pmatrix}
,
\begin{pmatrix}
 b_1 &  b_3 \\
0 & 0
\end{pmatrix})_{1 1} \\ &= \mu^3( a_1,  b_1 , b_1)  - \mu^3( a_1,  b_2 ,  b_1) = 0.
\end{split}
\end{equation}

\begin{equation}
\begin{split}
&\mu^3(\begin{pmatrix}
 a_1 &  a_2 \\
0 & 0
\end{pmatrix}
,
\begin{pmatrix}
 b_1 & 0 \\
- b_2 & 0
\end{pmatrix}
,
\begin{pmatrix}
 b_1 &  b_3 \\
0 & 0
\end{pmatrix})_{1 2} \\
&= \mu^3( a_1, b_1 , b_3)  - \mu^3( a_1,  b_2 , b_3) = \mathfrak{t}b_3.
\end{split}
\end{equation}

It remains to show that $\Theta$ is a nontrivial coset. First, we compute the relevent degree two piece ambiguity ideal $\II$. By definition, it is generated by 

\begin{equation} \label{eq:ambiguity39}
\begin{pmatrix}
\mathfrak{t} a_1 & \mathfrak{t} a_2 \\
0 & 0
\end{pmatrix} \cdot \begin{pmatrix}
m_1 & m_2 \\
m_3 & m_4
\end{pmatrix} = \begin{pmatrix}
\mathfrak{t} (m_1 a_1 + m_3 a_2) & \mathfrak{t} (m_2 a_1 + m_4 a_2) \\
0 & 0
\end{pmatrix}
\end{equation}

\begin{equation} \label{eq:ambiguity40}
\begin{pmatrix}
n_1 & n_2 \\
n_3 & n_4
\end{pmatrix} \cdot
\begin{pmatrix}
\mathfrak{t} b_1 & \mathfrak{t} b_3 \\
0 & 0
\end{pmatrix} = \begin{pmatrix}
\mathfrak{t} n_1 b_1  & \mathfrak{t} n_1 b_3 \\
\mathfrak{t} n_3 b_1 & \mathfrak{t} n_3 b_3
\end{pmatrix}
\end{equation}

where the $m_i$ and $n_j$ are arbitrary integer numbers. We assume that the Claim is false (i.e., $\Theta \in \II$) and arrive at a contradiction: Let us denote 
\begin{equation}
\Theta = \begin{pmatrix}
x & y \\
z & w
\end{pmatrix} \in \II
\end{equation}
with entries $x,y,z,w \in \mathfrak{t} \cdot H^1(\Sigma_\phi)$. Then by definition $\Theta$ can be written as a sum of matrices of the form \eqref{eq:ambiguity39}--\eqref{eq:ambiguity40}. So on the one hand, we must have $n_1=1$ (because $\Theta_{12} = \mathfrak{t}b_3$), but on the other hand the only solution to
\begin{equation}
\mathfrak{t} (m_1 a_1 + m_3 a_2 + n_1 b_1) = 0 = \Theta_{11}
\end{equation}

is $n_1=m_1=m_3=0$. 
\end{proof}

\section{Computation (II): Quantum Cohomology of $Y = Bl_C \P^3$} \label{sec:cohomologylevelcomputations}
In this section, we concentrate all cohomology level computations like the ordinary cohomology of $Y$ and the Gromov-Witten invariants of classes relevent to the computation of the ambiguity ideal. \\

For the convenience of the reader, we provide a summary of the principal actors and notation of Section \ref{sec:familiesofcurves}:
\begin{itemize}
\item
The base $S = S^1$ is parametrized in the following way:
\begin{equation}
S = \left\{e^{s \pi \imath} \: \bigg| \: s \in [0,2] \right\} \vspace{0.5em}
\end{equation}
\item
Our ambient space is the trivial family 
\begin{equation}
(S^1 \times \CP{3}, \Omega)
\end{equation}
equipped with a vertical 2-form which restricts to the standard K\"{a}hler form $\omega_{FS}$ on each fiber. \vspace{0.5em}
\item
$\KK \subset S^1 \times \CP{3}$ is an $S$-family of canonically embedded, genus $4$ smooth algebraic curves. We denote the abstract $S$-family as $\CC$. \vspace{0.5em}
\item
The holomorphic fiberwise normal bundles 
\begin{equation}
\NN_{K_s/\CP{3}} = \OO\big|_{K_s}(2) \oplus \OO\big|_{K_s}(3)
\end{equation}
to each space curve piece together to form a smooth, (real) rank $4$ bundle, which we denote as $\NN \to K$. \vspace{0.5em}
\item
$\KK \subset \QQ$ is the $S$-family of quadric hypersurfaces which contain the genus $4$ curves. It comes with a fixed trivialization
\begin{equation} 
\QQ \iso S^1 \times \CP{1} \times \CP{1} \vspace{0.5em}
\end{equation} 
\item
$\YY$ is an $S$-family of 3-folds with fiber 
\begin{equation}
Y_s = Bl_{K_s} \P^3, 
\end{equation}
the complex manifold obtained as the algebro-geometric blowup of projective space at the complete intersection of a quadric and a cubic for every $s \in S$. $\XX$ is an $S$-family of cubic 3-folds, and each $Y_s$ has an alternate description as the 1-point blowup of $X_s$. \vspace{0.5em}
\item
There is a smooth blowdown map 
\begin{equation}
\pi : \YY \to S^1 \times P^3
\end{equation}
which is holomorphic in each fiber. \vspace{0.5em}
\item
The $S$-family of exceptional divisors is denoted
\begin{equation}
\EE \subset \YY \vspace{0.5em}
\end{equation}
\item
The fiberwise normal bundles $\NN_{E_s/\P^3}$ form a smooth bundle $\NN \to \EE$. \vspace{0.5em}
\end{itemize}

\subsection{A standard model for $Tot_{\OO(-1)}$\label{subsec:Otoy}} We open with a preparatory computation. Recall,
\begin{definition}
A \textbf{Hirzebruch surface} is a holomorphic $\CP{1}$-bundle over $\CP{1}$. 
\end{definition}
These are all complex surfaces by the remark above. In fact, any Hirzebruch surface arises as the projectivization of rank two holomorphic bundle over $\CP{1}$. The latter are all split, so any such surface is a projectivization of the form
\begin{equation}
\mathbb{F}_k := \P(\OO_{\CP{1}} \oplus \OO_{\CP{1}}(k)) 
\end{equation}
for some $k \in \N_{\geq 0}$ which is determined uniquely as minus the minimal self-intersection number of the section. If $k = 0$ then $\mathbb{F}_0 = \CP{1} \times \CP{1}$ and there is a whole $\CP{1}$ worth of such sections; otherwise it is holomorphically rigid and it is in fact unique. Differentiably, $\mathbb{F}_k$ is the unique non-trivial $S^2$-bundle over $S^2$ for $k$ odd and the trivial bundle $S^2 \times S^2$ when $k$ is even. \\

Following Seidel's \cite[section (5c)]{eprint8} and keeping our notation as close as possible, we denote $\bar{Y}^{\mathit{toy}}$ for $\mathbb{F}_1 = Bl_{pt} \CP{2}$, considered as a toric symplectic (and K\"{a}hler) manifold. It is characterized symplectically by its moment polytope,
\begin{equation}
\left\{(t_1,t_2) \in \R^2 \: \big| \: t_1 , t_2 \geq 0 \: , \delta \leq t_1+t_2 \leq \epsilon \: \right\}. 
\end{equation}
for some constants $0 < \delta < \epsilon$. The preimage of ${t_1 + t_2 = \epsilon}$ under the moment map is the line at infinity in $\CP{2}$, denoted by $L^{\mathit{toy}} \iso \CP{1}$. The preimage of ${t_1 + t_2 =\delta}$ is the exceptional divisor $E^{\mathit{toy}}$. In particular,
\begin{equation}
c_1(\bar{Y}^{\mathit{toy}}) = - PD([E^{\mathit{toy}}]) + 3 PD([L^{\mathit{toy}}]). 
\end{equation}
Write $Y^{\mathit{toy}} =  \bar{Y}^{\mathit{toy}} \backslash L^{\mathit{toy}}$ for the complement of the hyperplane at infinity. This is a disc bundle over the sphere and a non-compact symplectic manifold, but because the divisor we remove has positive normal bundle, doing pseudo-holomorphic curve theory in its complement is unproblematic. \\

From the standard handlebody description of $D^2$-bundle over $S^2$ (see e.g., \cite[p. 103]{MR1707327}) it follows easily that there is a perfect Morse function, 
\begin{equation}
f^{toy} : Y^{\mathit{toy}} \to \R 
\end{equation}
that is quadratic at infinity with respect to some Hermitian metric and with two critical points $1,\mathfrak{b} \in E \iso \CP{1}$ of degree zero and two, correspondingly; as well as a suitable perturbation datum. On the cohomology level, we denote by $F$ the class of a fiber in $E^{toy}$, and by $u = −PD([E^{\mathit{toy}}]) \in H^2_{cpt}(Y^{\mathit{toy}})$ the negative Poincare dual of the exceptional divisor. With this sign convention,
\begin{align}
& u(F) = 1, \label{eq:negative-pd} \\
& c_1(Y^{\mathit{toy}}) = u, \label{eq:c1-toy} \\
& [\omega_{Y^{\mathit{toy}}}] = 2\pi\delta \, u, \label{eq:omega-toy} \\
& u^2 = -\mathit{PD}([\mathit{point}]), \label{eq:toy-gitler}
\end{align}
where in \eqref{eq:c1-toy} and \eqref{eq:omega-toy} we have mapped $u$ to $H^2(Y^{\mathit{toy}})$. It follows that $Y^{\mathit{toy}}$ is monotone:
\begin{equation} \label{eq:monotonicity}
c_1(Y^{\mathit{toy}}) = \textstyle\frac{1}{2\pi\delta} [\omega_{Y^{\mathit{toy}}}] \in H^2(Y^{\mathit{toy}};\R).
\end{equation}

$Y^{\mathit{toy}}$ is a line bundle over $E^{\mathit{toy}}$, and the fibres are Poincar\'e dual to $u \in H^{2}(Y^{\mathit{toy}})$, as one can see from \eqref{eq:toy-gitler}. The contribution of lines lying inside the exceptional divisor is
\begin{equation}
\begin{aligned}
\textstyle \int_{Y^{\mathit{toy}}} (u \ast_F u) \, u & = \langle u, u, u \rangle_{F}^{Y^{\mathit{toy}}} \\ & = \langle u, u \rangle_{F}^{Y^{\mathit{toy}}} = 1.
\end{aligned}
\end{equation}
Comparison with \eqref{eq:toy-gitler} yields
\begin{equation} \label{eq:toy-quantum-gitler}
u \ast_F u = - u.
\end{equation}
In principle, the multiples $mF$, $m>1$, could also contribute to the quantum product. But the virtual dimension of the associated moduli space of three-pointed holomorphic spheres is $4 + 2 m>6$, while the image of the evaluation map is contained in $\CP{1} \times \CP{1} \times \CP{1}$. Hence the virtual fundamental cycle maps to zero under evaluation.

\subsection{The ordinary cohomology ring}
We fix a base point $* \in S^1$ be fixed. In this subsection we compute the cohomology ring of a fiber $Y = Y_*$ of the 7-dim mapping tori. As we recall from \eqref{eq:tikzpicture1}, we can approach such a Fano 3-fold from two different perspectives: 
\begin{equation}
\begin{tikzpicture}
\matrix (m) [matrix of math nodes, row sep=3em,
column sep=3em, text height=1.5ex, text depth=0.25ex]
{  & Q' & Y  & E &  \\
\left\{p\right\} &  X &  &   \P^{3}& C \\ };
\draw[right hook->] (m-1-2) --node[above]{\footnotesize{$i$}}  (m-1-3);
\draw[left hook->] (m-1-4) --node[above]{\footnotesize{$j$}}  (m-1-3);
\draw[dashed,->] (m-2-2) --node[above]{\footnotesize{$\pi$}}  (m-2-4);
\draw[->] (m-1-3) --node[above]{\footnotesize{$\sigma$}} (m-2-2);
\draw[->] (m-1-3) --node[above]{\footnotesize{$b$}} (m-2-4);
\draw[right hook->] (m-2-1) --  (m-2-2);
\draw[left hook->] (m-2-5) --  (m-2-4);
\draw[->] (m-1-2) -- (m-2-1);
\draw[->] (m-1-4) -- (m-2-5);
\end{tikzpicture}
\end{equation}
where in the diagram above: $X$ is a cubic 3-fold, $p$ is a fixed ODP, $\pi$ is the projection, $Q'$ is the projectivized tangent cone at $p$ (as before), $C$ is the genus $4$ smooth curve which parametrized the lines on $X$ passing through $p$, and $E$ is the exceptional divisor of the blowdown map $b$. We introduce the following additional cycles: 
\begin{itemize}
\item
$L$ is the preimage of a generic line under $b : Y \to \P^3$.
\item
$H$ is the preimage of a generic hyperplane under $b : Y \to \P^3$.
\item
$\breve{H}$ is the preimage under $\sigma : Y \to X$ of the of a generic line in $X \subset \P^4$. 
\item
$F$ is a fixed exceptional fiber.
\item
$\left\{A_i,B_i\right\}$ where $i=1,\ldots,4$ are fixed cycles defining a standard symplectic basis of $H^1(C,\Z) \iso \Z^{\oplus 8}$ equipped with the intersection form. See Figure \ref{fig:stdbasis} for an illustration. 
\begin{figure}[htb]
\centering
			\fontsize{0.25cm}{1em}
			\def\svgwidth{8cm}
 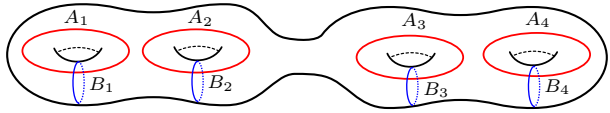
\caption{The standard basis}
\label{fig:stdbasis}
\end{figure}
\item
$\left\{V'_i\right\}$ and $\left\{V''_i\right\}$ for $i=1,\ldots,5$ are the vanishing cycles coming from the degeneration of $X$. We take them to be contained in small disjoint copies of the $A_5$-Milnor fiber. 
\end{itemize}
\textbf{Notation.} Note that the map $b \circ i: Q' \to \P^3$ identifies $Q'$ with the unique quadric $Q$ containing $C$. To keep notation compact, we will usually not differentiate between them (the distinction should be clear from context). We also use the same symbols cycles and their homology classes (this should be clear from context) and denote the Poincare dual cohomology classes by the corresponding small letter e.g. $PD(A) = a$ etc. Finally, as in the local model, we denote $u = -PD(F) = -f$. We treat this cohomology class as a formal variable of degree 2. 
\begin{lemma} \label{lemma:relationsintheblowup}
The following relations hold: 
\begin{equation}
Q = 2H - E \: , \: \breve{H} = 3H - E \: , \: H = \breve{H} - Q \: , \: E = 2H - 3Q. 
\end{equation}
\end{lemma}
\begin{proof}
Let $X_0, \ldots , X_5$ be coordinates in $\C^5$ such that $p = (0 , 0 , 1 , 0 , 0 )$. Then $X_0: \ldots : X_4$ are homogeneous coordinates in $\P^4$ and the rational map $\pi = b \circ \sigma^{-1}$ is just the projection 
\begin{align*}
[X_0: X_1 : X_2 : X_3 : X_4] \mapsto [X_0: X_1  : X_3 : X_4].
\end{align*}
As in \eqref{eq:definecubic}, the equation for cubic 3-fold $X$ can be written as
\begin{align*}
X_2 \cdot q (X_0, X_1, X_3 , X_4) + s(X_0, X_1, X_3 , X_4)
\end{align*}
where $q$ and $s$ are homogeneous forms of degree 2 and 3 respectively. Recall that $\sigma$ is the blowup of the genus 4 curve 
\begin{align*}
C = \left\{q (X_0, X_1, X_3 , X_4) =  s(X_0, X_1, X_3 , X_4) = 0\right\},
\end{align*}
so the inverse map $\pi^{-1} = \sigma \circ b^{-1}$ is given in these coordinates by the formula 
\begin{equation}
\begin{split}
[X_0: X_1  : X_3 : X_4] \mapsto [ &X_0 \cdot q(X_0, X_1 , X_3 , X_4) :  X_1 \cdot q(X_0, X_1  , X_3 , X_4) :  -s(X_0, X_1 , X_2 , X_3 , X_4)  \\
&: X_3 \cdot q(X_0, X_1  , X_3 , X_4) :  X_4 \cdot q(X_0, X_1  , X_3 , X_4) ]
\end{split}
\end{equation}
It follows that the proper preimage of the quadric $Q = V(q) \subset \P^3$ is contracted by $\sigma$, whence $Q = 2H - E$. Similarly, the proper preimage of the cubic $S = V(s) \subset \P^3$ is obtained by setting $X_2=0$, so $\breve{H} = 3H - E$. The last two relations are an easy consequence of the first.
\end{proof}
\begin{lemma}
The first Chern class is
\begin{equation}
c_1(Y) = 4h - e = 2\breve{h} - q \in H^2(Y)
\end{equation}
\end{lemma}
\begin{proof}
Follows from the adjunction formula. 
\end{proof}
The pushforward and the pullback
\begin{equation} \label{eq:lerayhirsch}
H^\bullet(E) \rightarrow H^{\bullet+2}(Y) \: , \: H^\bullet(\P^3) \rightarrow H^\bullet(Y)
\end{equation}
respectively, are both injective, their images are disjoint, and together they span $H^\bullet(Y)$. Using Leray-Hirsch we then write
\begin{equation} \label{eq:blowup-cohomology}
H^\bullet(Y) \iso H^\bullet(\P^3) \oplus H^\bullet(E)[-2] \iso H^\bullet(\P^3) \oplus u H^\bullet(C).
\end{equation}
\begin{lemma}
We have the following presentation for the middle dimensional homology:
\begin{equation}
H_3(Y) = \bigslant{\Z\left\langle V'_1,\ldots,V'_5,V''_1,\ldots,V''_5\right\rangle}{\left\langle V'_1 + V'_3 + V'_5 , V''_1 + V''_3 + V''_5  \right\rangle}. 
\end{equation}
\end{lemma}
\begin{proof}
By the blowup description, there is no torsion in homology and $H^3(Y) \iso u H^1(C)$ has rank $8$. On the other hand, the vanishing homology of each of the $A_5$-singularity is a copy of the $A_5$ root lattice.
Comparing ranks, we see that must have at least two relations. But given any homology relation 
\begin{equation}
\sum_i \lambda_i W_i = 0 \: , \: W_i \in H_3(Y) 
\end{equation}
we can intersect it with another $H_3$-class $W$ to obtain 
\begin{equation}
\sum_i \lambda_i (W_i \cdot W) = 0. 
\end{equation}
Since the two balls $B'$ and $B''$ are disjoint, it is clear that vanishing cycles coming from different singularities are orthogonal under the intersection pairing. By examining the nullspace of each individual $A_5$-intersection matrix, 
\begin{equation} \label{eq:A5sing}
\begin{pmatrix}
0 & -1 & 0 & 0 & 0\\
1 & 0 & -1 & 0 & 0\\
0 & 1 & 0 & -1 & 0\\
0 & 0 & 1 & 0 & -1\\
0 & 0 & 0 & 1 & 0
\end{pmatrix}
\end{equation}
we see that the only possible relations are the two written above.
\end{proof}
The following relations above are known to hold in the Chow ring.
\begin{lemma} \label{lem:IskovskikhProkhorov}
Let $C \subset \P^3$ be a smooth curve of degree $d$ and genus $g$. Then 
\begin{equation}
\begin{split}
E \cdot E &= (4d +2g-2)F -d L, \\
E \cdot H &= d F, \\
E \cdot F &= −pt \\
\end{split}
\end{equation} 
\end{lemma}
\begin{proof}
See \cite[Section \textsection 6.7]{MR1644323}. 
\end{proof}
Note that comparing the cohomology decomposition \eqref{eq:blowup-cohomology} with the description in Fulton, we see that the natural maps $A_{\bullet}(Y) \to H^{2 \bullet}(Y)$ are isomorphisms, and the same relations hold for the Poincar\'{e} duals: $h$, $u$, $\ell$ and $f$. 
\begin{lemma}
We have the following presentation for the cohomology of $Y$:
\begin{itemize}
\item
$H^0 = \Z\left\langle y\right\rangle$, where the class $y = PD(Y)$.
\item
$H^1=0$.
\item
$H^2 = \Z\left\langle h,u \right\rangle$, with $h = PD(H)$ and $u  = -PD(E)$.
\item
$H^3 = \Z\left\langle u a_i,u b_i\right\rangle$ where $i=1,\ldots,4$. Each class is Poincare dual to the total transform of one of the standard cycles in $C$ under $\sigma$. Alternatively, 
\begin{equation}
\begin{split}
&u \cdot a_1 = PD(V'_1) \: , \: u \cdot b_1 = PD(V'_2), \\
&u \cdot a_2 = PD(V'_4) \: , \:  u \cdot b_2 = PD(V'_5), \\
&u \cdot a_3 = PD(V''_1) \: , \:  u \cdot b_3 = PD(V''_2), \\
&u \cdot a_4 = PD(V''_4) \: , \:  u \cdot b_4 = PD(V''_5).
\end{split}
\end{equation}
\item
$H^4 = \Z \left\langle \ell,f \right\rangle$, with $\ell = PD(L)$ and $f  = PD(F)$.
\item
$H^5=0$. 
\item
$H^6 = \Z\left\langle pt\right\rangle$, the Poincare dual of the point. 
\end{itemize}
The ring structure is 
\begin{table}[h!]
  \centering
  \resizebox{0.6\textwidth}{!}{ 
	\centering
	\begin{tabular}{| >{$}c<{$} || >{$}c<{$} | >{$}c<{$} | >{$}c<{$} | >{$}c<{$} | >{$}c<{$} | >{$}c<{$} | >{$}c<{$} | >{$}c<{$} |}
	\hline
		\cup & y & h & u & u \cdot a_i & u \cdot b_j  & \ell & f & pt\\ \hhline{|=||=|=|=|=|=|=|=|=|}
		y & y & h & u & u \cdot a_i & u \cdot b_j  & \ell & f & pt\\ \hline
		h & h  & \ell & -6f & 0 & 0 & pt & 0 & 0\\ \hline
		u & u  & -6f & 30f - 6 \ell & 0 & 0 & 0 & pt & 0 \\ \hline
		u \cdot a_i & u \cdot a_i & 0 & 0 & 0 & -\delta_{ij} \cdot pt & 0 & 0 & 0\\ \hline
		u \cdot b_j & u \cdot b_j  & 0 & 0 & \delta_{ij} \cdot pt & 0 & 0 & 0 & 0\\ \hline
		\ell & \ell & pt & 0 & 0 & 0 & 0 & 0 & 0\\ \hline
		f & f & 0 & pt & 0 & 0 & 0 & 0 & 0\\ \hline
		pt & pt & 0 & 0 & 0 & 0 & 0 & 0 & 0\\ \hline
	\end{tabular}}
\end{table}
\end{lemma}
\begin{proof}
Since $y$ is the unit in cohomology, the first line and column in the multiplication table are immediate. Since $H^5 = 0$, $h \cdot (u \cdot a_i) = h \cdot (u \cdot b_j)$ must vanish. Similarly, $u \cdot (u \cdot a_i) = u \cdot (u  \cdot b_j) = 0$. It is clear that $h^2 = \ell$,  $h \cdot \ell = h^3 = pt$, and $h \cdot u = -h \cdot e = -6f$. A generic line in projective space does not meet the curve $C$, so $u \cdot \ell = 0$. Using the local model gives $u \cdot f = pt$, as well as $(u \cdot a_i) (u \cdot b_j) = -pt$. The relations in Lemma \ref{lem:IskovskikhProkhorov} allow us to determine $u \cdot u$. Now the identification of the $H^3$-classes with the Poincare dual of the vanishing cycles follows from \eqref{eq:A5sing} and the non-degeneracy of the cup-product.
\end{proof}

\begin{lemma}
The second Chern class is 
\begin{equation}
c_2(Y) = 12 h^2 +24 f
\end{equation}
\end{lemma}
\begin{proof}
Since the curve $C$ has degree $6$, it is clear that $c = PD(C) = 6h^2$. The total Chern class of projective space is given by 
\begin{equation}
c(\P^3) = (1+h)^4 = 1 + 4h + \binom{4}{2} h^2 + \ldots
\end{equation}
we see that $c_1(\P^3) = 4h , c_2(\P^3) = 6h^2$. The result now follows from the previous Lemma and \cite[p. 610]{MR507725}. 
\end{proof}

\subsection{Cone of curves} \label{sec:coneofcurvescomputation}
Let $(Y,J)$ be a complex manifold. Recall that a complex valued (1,1)-form $\omega$ is 
real if $\overline{\omega} = \omega$. In that case, if we define $g(v,w) : =\omega(v,J w)$, then $g$ is a smooth, symmetric bilinear form on $Y$. We say that $\omega$ is \textbf{positive} if $g$ is a Riemannian metric.

The triple of data $(J,\omega,g)$, where $J$ is an almost complex structure, $\omega$ is a real positive $(1,1)$-form, and $g$ is the associated Riemannian metric as defined above together define the structure of an \textbf{Hermitian manifold} on $Y$, with $h = g + i \omega$ as the Hermitian metric.
\begin{definition}
We say that $Y$ is K\"{a}hler if $d \omega = 0$. 
\end{definition}
In that case, the closed form $\omega$ is called the K\"{a}hler form. Alternatively, 
\begin{definition}
A K\"{a}hler manifold as a symplectic manifold $(Y,\omega)$ equipped with a choice of a compatible integrable almost complex structure $J$.
\end{definition}
The two definitions are obviously equivalent: given a Hermitian manifold $(Y,h,J)$ with a closed $h$, the form 
\begin{equation}
\omega = \frac{i}{2}(h - \overline{h})
\end{equation}
is symplectic and compatible with $J$. Conversely, any symplectic form $\omega$ compatible with an almost complex structure $J$ must be a $(1,1)$-form, written locally in coordinates as
\begin{equation}
\omega = \frac{i}{2} \sum_{ j,k } h_{j \bar{k}}(z) dz_j \wedge d\bar{z_k}
\end{equation}
where $h_{j \bar{k}}$ are smooth functions. Since $\omega$ is real-valued, closed and non-degenerate, it easily follows that $h_{j \bar{k}}$ is a Hermitian metric. 

\begin{definition}
We define the \textbf{K\"{a}hler cone} as the set 
\begin{equation}
K \subseteq H^{1,1}(Y;\R)
\end{equation}
of cohomology classes of K\"{a}hler forms. This is an open convex cone. 
\end{definition}
It is interesting to consider the algebraic analogues $K$, which consist of suitable integral divisor classes

\begin{definition}
An $\R$-divisor on $Y$ is a finite sum $\sum_i a_i D_i$ with $a_i \in \R$ and each $D_i$ an irreducible divisor (codimension one subvariety) in $Y$. The \textbf{Neron-Severi space} 
\begin{equation}
N^1(Y) \subset H^2(Y;\R)
\end{equation}
is defined as the image in cohomology of divisors modulo the relation of \textbf{numerical equivalence}: $D_1 \simeq D_2$ if the intersection numbers $D_1 \cdot C = D_2 \cdot C$ are equal for all irreducible curves $C$ in $Y$.
\end{definition}

A very important fact is that the ''Neron-Severi part" of the K\"{a}hler cone
\begin{equation}
K_{NS} = K \cap N^1(Y)
\end{equation}
can be characterized in simple algebraic terms.  

\begin{definition}
A line bundle $L$ on a projective variety $Y$ is very ample if it has enough global sections to give an embedding into projective space. $L$ is \textbf{ample} if some positive multiple $L^{\otimes d}$ is very ample. The cone generated by all ample divisors in $N^1(Y)$ is called the \textbf{ample cone}. 
\end{definition}

It is well known that $K_{NS}$ is an open cone and coincides with the ample cone.

\begin{definition}
The dual vector space $N_1(Y) \subset H_2(Y;\R)$ is the space of real 1-cycles $\sum_i a_i C_i$ modulo numerical equivalence. This is the space spanned by all algebraic curves. A 1-cycle is \textbf{effective} if all $a_i \geq 0$. The cone in $N_1(Y)$ spanned by the classes of all effective one-cycles
\begin{equation}
NE(X) := \left\{\sum_i a_i [C_i] \bigg| C_i \subset X \text{ is an irreducible algebraic curve }, a_i \geq 0\right\} \subseteq N_1(X)
\end{equation}
is called the \textbf{cone of curves}. 
\end{definition}

\begin{proposition}[\cite{MR2095471}, p. 52]
The closure $\overline{NE}(Y) \subseteq N_1(Y)$ in the usual $\R$-topology is a closed cone, called the \textbf{closed cone of curves}. 
\end{proposition}

There is a numerical characterization of ampleness which showes that if we know the closed cone of curves, we understand the ample divisors on $Y$ as well. 

\begin{proposition}[\cite{MR2095471}, p. 53] A divisor $D$ is ample if and only if $D \cdot C>0$ for all $0 \neq C \in \overline{NE}(Y)$. 
\end{proposition}

\begin{lemma} \label{lem:coneofcurvesdescription}
$\overline{NE}(Y) \subseteq N_1(Y)$ is generated by $F$ and $R = L- 3 F$. 
\end{lemma}
\begin{proof}
We immediately see that in our case $\overline{NE}(Y) \subseteq N_1(Y)$ is generated by $F$ and $R$, where $R$ should correspond to the strict transforms of space curves $C' \subset \P^{3}$ of degree $m$ which intersect $C$ in $n$ points, with the ratio $n/m$ maximal (it is not apriori obvious that $R$ should correspond to an actual effective curve). So let $C'$ be an irreducible degree $m$ which intersect $C$ in $n$ points. If $C'$ is not contained in the quadric $Q$, then $n \leq C' \cdot Q = 2m$. If $C'$ is contained in the quadric, then it has some bidegree $(a,b)$ where $a+b = m$ and it meets the curve $C$ at $3a+3b = 3m$ points. Thus the slope of $R$ is bounded by $3$, and is achieved by any generic hyperplane section $H \cap Q$. 
\end{proof}  
In particular, $m \cdot F$ lies on the boundary of cone of curves for every $m \geq 1$. Thus, the only solution to $A + B = m \cdot F$, where $A$ and $B$ are algebraic curves is $A = m_A \cdot F$, $B = m_B \cdot F$ with $m_A + m_B = m$ and $m_A,m_B \geq 1$.

\begin{corollary}
The ample cone is generated by $h$ and $u+3h$. 
\end{corollary}
Note that this implies that $Y$ is Fano because the class of the anti-canonical divisor is $c_1(Y) = u+4h$. 

\begin{remark}
In \cite{Blanc-Lamy}, Blanc-Lamy classified all smooth curves in $\P^3$ whose blowup is (weak) Fano. Our genus-degree pair $(g,d)=(4,6)$ is the last in the set they called $\AA_1$. 
\end{remark}

\begin{figure}[htb]
\begin{minipage}[b]{.5\linewidth}
\centering
			\fontsize{0.25cm}{1em}
			\def\svgwidth{5cm}
 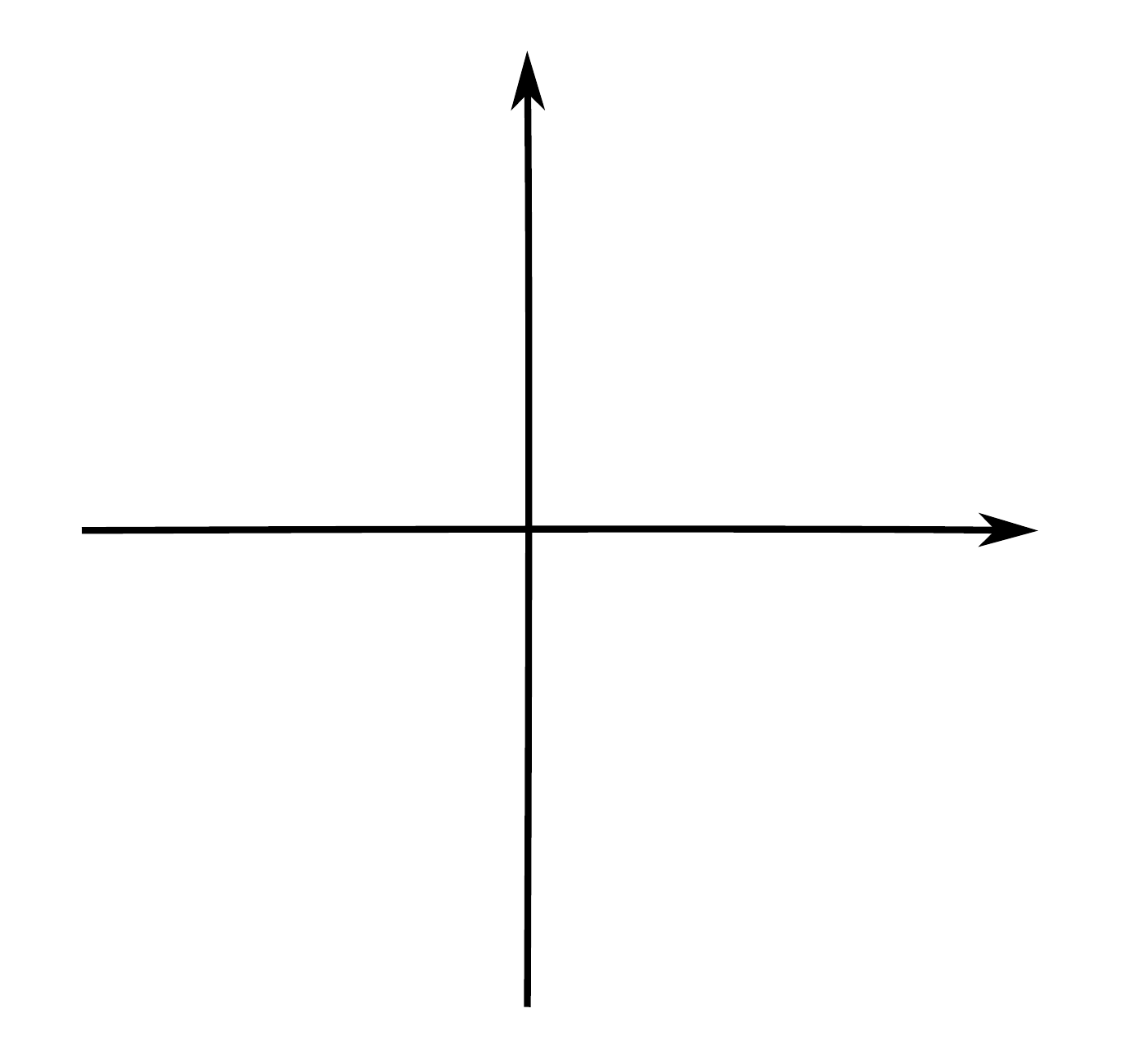
\subcaption{The cone of curves}
\end{minipage}%
\begin{minipage}[b]{.5\linewidth}
\centering
			\fontsize{0.25cm}{1em}
			\def\svgwidth{5cm}
 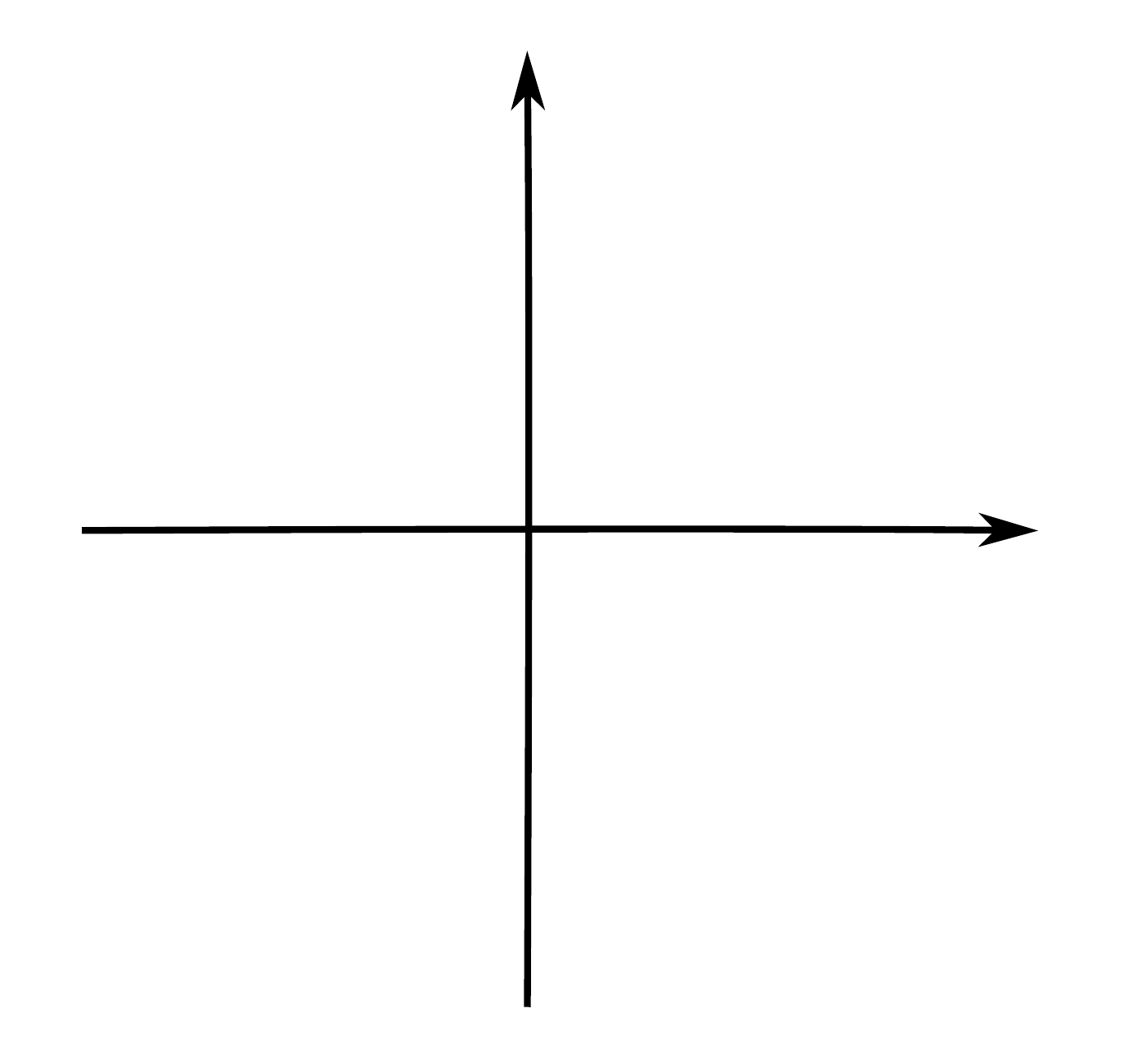
\subcaption{The ample cone}
\end{minipage}
\label{fig:coneofcurves}
			\end{figure} 

\subsection{Exceptional fibers} \label{subsec:F}
In this subsection, we consider holomorphic curves $u : S^2 \to Y$ of homology class 
\begin{equation}
A = m \cdot F \in H_2(Y;\Z) \: , \: m \geq 1.
\end{equation}
Let $\textbf{J}_{alg}$ be the integrable almost complex structure. When $m=1$, Regularity of $\MM_{0,3}(A,\textbf{J}_{alg})$ is ensured by \cite[Lemma 3.3.3]{MR2954391}. Indeed, the moduli space
consists of the fibers of the projection $C \times \P^1 \to C$ and so it smooth and of the expected dimension
\begin{equation}
6 + 2c_1(F) + 6-6 = 8.
\end{equation}
\begin{lemma}
The Gromov-Witten invariant $GW_{A,3}$ vanishes for $m>1$.
\end{lemma}
\begin{proof}
Assume that is not the case. Then the moduli space $\MM_{0,3}(A,\textbf{J})$ is non-empty for every $\omega$-tame almost complex structure $\textbf{J}$ (if not, $\textbf{J}$ would be regular for trivial reasons, which would force the GW-invariant to vanish.) In particular, this is true for $\textbf{J}_{alg}$, so there exists a stable genus zero pseudo-holomorphic curve in $Y$, whose irreducible components represent classes
\begin{equation}
A'_1 ,\ldots, A'_r, A''_1,\ldots,A''_s
\end{equation}
The notation is such that the components representing $A′_i$ are those which lie inside the exceptional divisor (and therefore inside a single fibre of $E \to C$.) Since $h(A) = m \cdot h(F) = 0$, and $[\omega_{\P^3}](A′_i) = h(A'_i)  = 0$,
\begin{equation} \label{eq:sumofterms}
\sum^{s}_{i=1} h( A''_i) =0. 
\end{equation}
The remaining components are not contained in $E$, hence satisfy 
\begin{equation}
\begin{split}
A''_i \cdot E_s &= -u(A''_i)  >0  , \\
[\omega_{Y}](A''_i) &= u(A''_i) +4h(A''_i) >0.
\end{split}
\end{equation}
If $s>0$ we would get a contradiction with equation \eqref{eq:sumofterms}. Therefore, $s=0$ and all spheres are multiple covers of exceptional fibers. This reduces the computation to the local model, but we have already determined that the invariant vanishes in this case. 
\end{proof}
To compute $GW_{F,1}$, we consider cohomology classes in $H^4(Y)$: 
\vspace{0.5em}
\begin{itemize}
\item
$GW_{F,1}(\ell)=0$. A generic line of this homology class does not meet the exceptional divisor.  \vspace{0.5em}
\item
$GW_{F,1}(f) = -1$. The intersection number of the exceptional divisor and an exceptional fiber. \vspace{0.5em}
\end{itemize}
To compute $GW_{F,2}$, we consider pairs of cohomology classes $(x_2,x_1)$ such that $\deg(x_2) +  \deg(x_1)$. By the symmetry axiom, we can restrict to the case where $\deg(x_2) \geq \deg(x_1) \geq 2$. If $\deg(x_1) = 2$, we can reduce to the previous case by the divisor axiom. So $\deg(x_2) = \deg(x_1) = 3$ and the only possible cases are  \vspace{0.5em}
\begin{itemize}
\item
$GW_{F,2}(u \cdot a_i,u \cdot a_j)=0$. Since the projections to $\P^{3}$ do not intersect. \vspace{0.5em}
\item
$GW_{F,2}(u \cdot a_i,u \cdot b_j)= \delta_{ij} $. If $i \neq j$, the invariant is zero for the same reason as before; otherwise the image of the 2-point Gromov-Witten pseudo-cycle is
\begin{equation}
\left\{(x,y) \in Y \times Y \: \big| \: b(x) = b(y) \in C \right\}\subset Y \times Y. 
\end{equation}
Choose standard chain representatives $A_i$ and $B_i$ which intersect transversely at a single point $p \in C$, and consider 
\begin{equation}
\left\{(x,y) \in Y \times Y \: \big| \: b(x) \in A_i \: , \:  b(y) \in B_i \right\}\subset Y \times Y. 
\end{equation}
We wish to compute the intersection number these two cycles. Note that since the intersection is transverse in the base direction, we can restrict to the fiber over $p$. 
This means we need to compute the self-intersection number of the diagonal in the toy model
\begin{equation}
\left\{(x,x)  \: \big| \: x \in E^{toy} \right\} \subset Y^{toy} \times Y^{toy}. 
\end{equation}
By definition, this is  
\begin{equation}
e(E^{toy})^2 = c_1(E^{toy})^2 = (-1)^2=1.
\end{equation}
\item
$GW_{F,2}(u \cdot b_i,u \cdot a_j)=-\delta_{ij}$. Follows from the symmetry axiom.  \vspace{0.5em}
\item
$GW_{F,2}(u \cdot b_i,u \cdot b_j)=0$. Since the projections to $\P^{3}$ do not intersect. \vspace{0.5em}
\end{itemize}
To compute $GW_{F,3}$, we consider triples of cohomology classes $(x_3,x_2,x_1)$ such that $\deg(x_3) + \deg(x_2) +  \deg(x_1) = 8$ and $\deg(x_3) \geq \deg(x_2) \geq \deg(x_1) \geq 2$. This forces $\deg(x_1)=2$ which reduces to the previous case. 

\begin{lemma} \label{lem:kpointexceptioalclass}
Let $k>3$. Up to permutation and sign, the only nontrivial $k$-point Gromov-Witten invariants in the class $F$ are
\begin{equation}
GW_{F,k}(f,u,\ldots,u) \: , \: GW_{F,k}(u,u,\ldots,u) \: , \: GW_{F,k}(u \cdot a_i,u \cdot b_i,u,\ldots,u)
\end{equation}
\end{lemma}
\begin{proof}
The lemma follows from the following observations: \vspace{0.5em}
\begin{enumerate}
\item
We can use the divisor axiom to bring $GW_{F,k}(\underline{x})$ to a reduced form: either $k=3$ or there are no more degree 2 classes. \vspace{0.5em}
\item
If $h$ appears anywhere, the invariant vanishes because $h(F) = 0$. \vspace{0.5em}
\item
Any Gromov-Witten invariant where one of the insertions is $\ell$ or $pt$ must vanish (because they can be represented by cycles whose support is disjoint from the exceptional divisor). \vspace{0.5em}
\item
Any Gromov-Witten invariant where two of the insertions are $f$ must vanish (because they can be represented by cycles whose support is taken to lie on different fibers). \vspace{0.5em}
\item
The same is true for the case where there are more then three insertions of degree $3$, because their image under blowdown is represented by three 1-cycles, say $D_1, D_2$ and $D_3$ (taken to be in generic position), then the triple product 
\begin{equation}
D_1 \times D_2 \times D_3 \subset C \times C \times C 
\end{equation}
does not intersect the triple diagonal 
\begin{equation}
\Delta_3 := \left\{(x,x,x) \: \big| \: x \in C \right\} C \times C \times C. 
\end{equation}
\end{enumerate}
\end{proof}
\subsection{Lines and recursion} \label{subsec:lines_and_recursion}
Let $L \in H_2(Y;\Z)$ be the class of a line. Each rational curve in the homology class $L$ can be blown down to a line that does not meet $C$. The Hilbert scheme of lines in $\P^3$ is $G(2,4)$ of complex dimension $2(4-2)=4$. It is clear that in the Grassmannian, lines that meet the curve have codimension $\geq 2$ and the moduli space $\MM_{0,3}(L,\textbf{J}_{alg})$ is smooth of the expected dimension 
\begin{equation}
6 + 2c_1(L-3F) -6 = 8
\end{equation}
so the Gromov-Witten invariants are enumerative. Essentially by definition, it is clear that any GW-invariant $GW_{L,k}(\ldots)$ that includes an \textbf{exceptional class}: $u$, $u \cdot a_i$, $u \cdot b_j$ or $f$ must be automatically zero. Other classical computations are: 
\begin{equation}
\begin{split}
GW_{4,L}(h,h,h,h) &= 2 \\
GW_{3,L}(h,h,pt) &= 1 \\
GW_{2,L}(pt,pt) &= 1 
\end{split}
\end{equation}
which correspond to: the number of lines that meet four hyperplanes, and the fact that a point and two hyperplanes (or two points) determines a unique line in $\P^3$. All follow from basic Schubert calculus. \\

To define the quantum ring structure, it is convenient to choose an integer basis ($k=1,\ldots,4$):
\begin{equation}
\begin{split}
T_0 &= y \\
T_1 &= h \\
T_2 &= u \\
T_{2k+1} &= a_k  \\
T_{2k+2} &= b_k \\
T_{11} &= f \\
T_{12} &= \ell \\
T_{13} &= pt 
\end{split}
\end{equation}
of (the $\Z$-module) $H^\bullet(Y)$. Define the integer matrix $g_{ij} := \int_Y T_i \cup T_j$ and let $g^{ij}$ be the inverse. Explicitly,
\begin{equation}
g_{ij} = g^{ij} = \left(
\begin{array}{*{14}c}
0&0&0&0&0&0&0&0&0&0&0&0&0&1 \\
0&0&0&0&0&0&0&0&0&0&0&0&1&0 \\
0&0&0&0&0&0&0&0&0&0&0&1&0&0 \\
0&0&0&0&1&0&0&0&0&0&0&0&0&0 \\
0&0&0&1&0&0&0&0&0&0&0&0&0&0 \\
0&0&0&0&0&0&1&0&0&0&0&0&0&0 \\
0&0&0&0&0&1&0&0&0&0&0&0&0&0 \\
0&0&0&0&0&0&0&0&1&0&0&0&0&0 \\
0&0&0&0&0&0&0&1&0&0&0&0&0&0 \\
0&0&0&0&0&0&0&0&0&0&1&0&0&0 \\
0&0&0&0&0&0&0&0&0&1&0&0&0&0 \\
0&0&1&0&0&0&0&0&0&0&0&0&0&0 \\
0&1&0&0&0&0&0&0&0&0&0&0&0&0 \\
1&0&0&0&0&0&0&0&0&0&0&0&0&0 \\
\end{array}\right)
\end{equation}
\begin{definition}
Given a homology class $A \in H^2(Y)$, four classes $\mu_1,\ldots,\mu_4$ and a collection of classes $\underline{x} = (x_k,\ldots,x_1)$ such that 
\begin{equation}
2\sum \codim(\underline{x}) + 2\sum_{i=1}^d \codim(\mu_i) = 2c_1(A) + 2n + 2(k+4)-6
\end{equation}
We have the following relation (which is really the WDVV equation in disguise)
\begin{equation} \label{eq:WDVVrelation}
\begin{split}
0 = \sum_{A_1,A_2} \sum_{\underline{y},\underline{z}} \sum_{i,j} &\bigg( g^{ij} \cdot GW_{A_1}(\underline{y},\mu_1,\mu_2,T_i) \cdot GW_{A_2}(\underline{z},\mu_3,\mu_4,T_j) \\
&-g^{ij} \cdot GW_{A_1}(\underline{y},\mu_1,\mu_3,T_i) \cdot GW_{A_2}(\underline{z},\mu_2,\mu_4,T_j) \bigg)
\end{split}
\end{equation}
where the sum is taken over: \vspace{0.5em}
\begin{enumerate}
\item
All effective classes $A_1,A_2 \in H^2(Y)$ with $A_1 + A_2 = A$, \vspace{0.5em}
\item
All $\underline{y}= (x_{i_1},\ldots,x_{i_{n_1}})$ and $\underline{z} = (x_{j_1},\ldots,x_{j_{n_2}})$
such that 
\begin{equation}
i_1 < \ldots < i_{n_1} , j_1 < \ldots < j_{n_2}, 
\end{equation}
the two sets of indices are disjoint, and 
\begin{equation}
\left\{i_1,\ldots,i_{n_1} \right\} \cupdot \left\{j_1,\ldots,j_{n_2}\right\} = \left\{1,\ldots,n\right\}
\end{equation}
(i.e., ''the classes of $\underline{x}$ get distributed in all possible ways onto the two factors"). \vspace{0.5em}
\item
All $0 \leq i,j \leq q$.
\end{enumerate}
Equation \eqref{eq:WDVVrelation} would be denoted as $\EE_A(\underline{x}; \mu_1,\mu_2 \: \big| \: \mu_3,\mu_4)$. 
\end{definition}
The following Lemma is a small modification of the one that appeared in \cite[Example 5.1]{MR1832328}.
\begin{lemma} \label{lem:gathmann}
Let $A =  L+ a F$. Consider a GW-invariant $GW_{A,k}(x_k,\ldots,x_1)$, and let $b$ be the number of fiber classes in $\underline{x} = (x_k,\ldots,x_1)$ and assume that $\deg(x_i) \neq 2$ for all $1 \leq i \leq k$. Then 
\begin{equation} \label{eq:145}
\begin{split}
(a+b) \cdot GW_{A,k}(\underline{x}) + \sum_{\underline{z}} (-1)^\dagger \cdot GW_{A,k-1}(\underline{y},f) &=6 \cdot (7+2a) \cdot GW_{A+  F,k+1}(\underline{x},f) \\
&+ ((a+1)^2-6) \cdot GW_{A+ F,k+1}(\underline{x},\ell). 
\end{split}
\end{equation} 
where the sum is taken over all pairs of $\left( x_{j_1} = u \cdot a_i, x_{j_2} = u \cdot b_i\right)$ and $\left( x_{j_1} = u \cdot b_i, x_{j_2} = u \cdot a_i\right)$ in $\underline{x}$. In the first case the sign is positive and in the latter case negative. 
\end{lemma}
\begin{proof}
We separate the terms with $A_1 = 0$ and $A_2=0$ and simplify the relation $\EE_{A+F}(\underline{x} ; h, h \big| u , u )$:
\begin{align} 
0 &= GW_{A+ F,k+3}(\underline{x},h,h,u \cdot u) + GW_{A+ F,k+3}(\underline{x},u,u,h \cdot h) \label{eq:146} \\ 
&- GW_{A+ F,k+3}(\underline{x},h,u,h \cdot u) - GW_{A+ F,k+3}(\underline{x},h,u,h \cdot u) \label{eq:147} \\  
&+ \sum_{A_1,A_2 \neq 0} \sum_{\underline{y},\underline{z}} \sum_{i,j} g^{ij} \cdot GW_{A_1}(\underline{y},h,h,T_i) \cdot GW_{A_2}(\underline{z},u,u,T_j) \label{eq:148} \\ 
&- \sum_{A_1,A_2 \neq 0} \sum_{\underline{y},\underline{z}} \sum_{i,j} g^{ij} \cdot GW_{A_1}(\underline{y},h,u,T_i) \cdot GW_{A_2}(\underline{z},h,u,T_j). \label{eq:149}
\end{align}
Note that $u(A+F) = (a+1)$ and $h(A+F) = 1$, so by applying the divisor axiom and Lemma \ref{lem:IskovskikhProkhorov} to lines \eqref{eq:146} and \eqref{eq:147}, we obtain the righthand side of \eqref{eq:145}. By Lemma \ref{lem:coneofcurvesdescription}, the only possible way to split $A+F = L+(a+1)F$ in $\overline{NE}(Y)$ is into $L+(a+1-c)F$ and $cF$. But any Gromov-Witten invariant in the class $cF$, $c>1$ vanishes identically. So the only terms that contribute are $A = L+aF$ and $F$. Since $h(F)=0$ we apply the divisor axiom to the terms in the sum and find that line \eqref{eq:149} cancels and the only terms in \eqref{eq:148} that can survive is 
\begin{equation}
\sum_{\underline{y},\underline{z}} \sum_{i,j}  g^{ij} \cdot GW_{A}(\underline{y},h,h,T_i) \cdot GW_{F}(\underline{z},u,u,T_j)
\end{equation}
By Lemma \ref{lem:3pointexceptioalclass} and Lemma \ref{lem:kpointexceptioalclass}, the only nontrivial contributions come from 
\begin{enumerate}
\item
$GW_{F}(u,u,T_{11}) = GW_{F}(f)$ (with $\underline{z}$ empty). 
\item
$GW_{F}(f,u,u,T_2) = GW_{F}(f,u)$. 
\item
$GW_{F}(\underline{z},u,u,T_{2}) = GW_{F}(u \cdot a_i,u \cdot b_i, u,u,u)$. 
\item
$GW_{F}(\underline{z},u,u,T_{2}) = GW_{F}(u \cdot b_i,u \cdot a_i, u,u,u)$. 
\end{enumerate}
From (1) we get $a \cdot GW_{A,k}(\underline{x})$, and from each (2) we get $1 \cdot GW_{A,k}(\underline{x})$ (and there are $b$ fiber classes). Finally, the contribution of each (3) term is 
\begin{equation}
GW_{A,k}(\underline{y},h,h,T_{11}) = GW_{A,k}(\underline{y},h,f)
\end{equation}
where $\underline{y} = (,\ldots,\widehat{u \cdot a_i},\ldots,\widehat{u \cdot b_i},\ldots)$. Similarly for (4) with negative sign. 
\end{proof}
We can apply the Lemma to compute some Gromov-Witten invariants for the homology class $A = L-F$. We immediately see that
\begin{equation}
GW_{L-F}(\ell,\ell,\ell) = 12 \: , \: GW_{L-F}(\ell,pt) = 6 
\end{equation}
and by definition any invariant of the form $GW_{L-F}(\ldots)$ with two insertions $x_i = x_j=f$ must vanish. Applying $\EE_{L}(\ell,\ell ; h, h \big| u,f )$ and $\EE_{L}(pt ; h, h \big| u,f )$, we see that:

\begin{equation} \label{eq:gwlf}
 GW_{L-F}(f,\ell,\ell) = 1 \: , \: GW_{L-F}(f,pt) = 1
\end{equation}
As for the odd classes, 
\begin{lemma} Let $w_2,w_1 \in H^1(C)$. Denote $x_2 = u \cdot w_2$ and $x_1 = u \cdot w_1$. Then
\begin{equation} \label{eq:gwlf2}
GW_{L-F}(\ell,x_2,x_1,\ell) = \deg(w_2 \cdot w_1). 
\end{equation} \label{lem:lfoddclass}
\end{lemma}
\begin{proof}
and by Lemma \ref{lem:gathmann} with $a=-1$ and $\underline{x} = (\ell,\ell,x_2,x_1)$
\begin{equation} \label{eq:gwlf2}
\begin{split}
- GW_{L-F,4}(\ell,\ell,x_2,x_1) + \sum_{\underline{z}} (-1)^\dagger \cdot GW_{L-F,3}(\ell,\ell,f) &=6 \cdot (7-2) \cdot GW_{L,5}(\ell,\ell,x_2,x_1,f) \\
&+ (0-6) \cdot GW_{L,5}(\ell,\ell,x_2,x_1,\ell). 
\end{split}
\end{equation}
The entire left hand side is automatically zero (because $x_i$ are exceptional). Now simplify using \eqref{eq:gwlf}.
\end{proof}

\subsection{Secant lines} \label{subsec:T}
Let $T = L - 2F \in H_2(Y;\Z)$. By the description of lines in Section \ref{subsec:watchtower}, we see that the rational curves in $\MM_{0,3}(T,\textbf{J}_{alg})$ parametrize the strict transform of lines in $\P^3$ which are either secant or simply tangent to $C$. Thus, the moduli space is smooth, and of the expected dimension 
\begin{equation}
6 + 2c_1(L-2F) + 6-6 = 10 = 4 + 2 + 2 + 2.
\end{equation}
There is a unique degree $6$ class ($pt$) and 
\begin{equation}
GW_{T,1}(pt) = \frac{d(d-3)}{2} + 1-g = 3 \cdot 3 + 1-4 = 6 
\end{equation}
by Lemma \ref{lem:gathmann}. To compute $GW_{T,2}$, without loss of generality, we consider cohomology classes $(x_2,x_1)$ with $\deg(x_2) \geq \deg(x_1) > 2$ and $\deg(x_2) + \deg(x_1) = 8$. Since there are no 5-classes, the only possible combination is $8=4+4$. Repeated application of Lemma \ref{lem:gathmann}, and using the values previously computed in \eqref{eq:gwlf}, yield
\begin{itemize}
\item
$GW_{T,2}(\ell, \ell) = d(d-2)+1-g = 21$. \vspace{0.5em}
\item
$GW_{T,2}(\ell, f) = d-1 = 5$. \vspace{0.5em}
\item
$GW_{T,2}(f, \ell) = 5$. By symmetry. \vspace{0.5em}
\item
$GW_{T,2}(f, f) = 1$. Using the relation $\EE_{L-F}(f;h,h \big| u,f)$. \vspace{0.5em}
\end{itemize}
\begin{remark}
Of course, these numbers can also be obtained by more classical methods like Schubert cycle computations, e.g. given a smooth nondegenerate curve of degree $d$ and genus $g$, it is well known (see for example \cite[Section 3.4.3]{3264}) that the class of the \textbf{locus of chords} to $C$ is 
\begin{equation}
(\binom{d-1}{2} -g) \cdot \sigma_2 + \binom{d}{2} \cdot \sigma_{1,1} \in A^2(G(1,3))
\end{equation}
so the number of chords to $C$ that meet two given lines in $\P^3$ is the degree of  
\begin{equation}
(\binom{d-1}{2} -g) \cdot \sigma_2 \sigma_1^2 + \binom{d}{2} \cdot \sigma_{1,1} \sigma_1^2 \in A^4(G(1,3)) \iso \Z
\end{equation}
which is $(\binom{5}{2}-4) + \binom{6}{2} = 10-4 + 15 = 21$. Similarly for the others. 
\end{remark}
It remains to consider $GW_{T,3}(x_3,x_2,x_1)$, using symmetry and divisor we only need to consider cohomology classes whose degrees satisfy $\deg(x_3) \geq \deg(x_2) \geq \deg(x_1) > 2$ and $\sum_i \deg(x_i) = 10$. The only possible decomposition is $4+3+3 = 10$, 
\begin{lemma} \label{lem:computeforTclass}
Let $z_2,z_1 \in H^1(C)$. Denote $x_2 = u \cdot z_2$ and $x_1 = u \cdot z_1$. Then the following relations hold
\begin{align} \label{eq:1103}
GW_{T,3}(f, x_2 , x_1 )  &=  \deg(z_2 \cdot z_1) \\ \label{eq:1104}
GW_{T,3}(\ell, x_2 , x_1 ) &=  5\deg(z_2 \cdot z_1) 
\end{align} 
\end{lemma}
\begin{proof}
By Lemma \ref{lem:gathmann} with $a=-2$ and $b=1$,
\begin{equation} 
\begin{split}
(a+b) \cdot GW_{T,3}(f, x_2 , x_1) + \sum_{\underline{z}} (-1)^\dagger \cdot GW_{T,2}(f,f) &= 6 \cdot 3 \cdot GW_{L-F,3}(f, x_2 , x_1,f) \\
&+ ((-1)^2 - 6) GW_{L-F,3}(f, x_2 , x_1,\ell)
\end{split}
\end{equation} 
The right hand side vanishes because the blowdown of a rational curve in class $L-F$ must can not meet $C$ in more then one point. Simplifying the left hand side gives \eqref{eq:1103}. Applying Lemma \ref{lem:gathmann} again 
\begin{equation} 
\begin{split}
(a+b) \cdot GW_{T,3}(\ell, x_2 , x_1) + \sum_{\underline{z}} (-1)^\dagger \cdot GW_{T,2}(\ell,f) &= 6 \cdot 3 \cdot GW_{L-F,3}(\ell, x_2 , x_1,f) \\
&+ ((-1)^2 - 6) GW_{L-F}(\ell, x_2 , x_1,\ell)
\end{split}
\end{equation} 
The $GW_{L-F,3}(\ell, x_2 , x_1,f)$ term is zero for the same reason as before. By Lemma \ref{lem:lfoddclass}, 
\begin{equation} 
\begin{split}
- 2 \cdot GW_{T,3}(\ell, x_2 , x_1) + 5 \deg(z_2 \cdot z_1) =  -5 \deg(z_2 \cdot z_1)
\end{split}
\end{equation} 
\end{proof}
\begin{remark}
Note that when $\deg(z_2 \cdot z_1) \neq 0$, these five lines are exactly the five lines that are incident to a fixed line and the exceptional fiber over the intersection point of the 1-cycles in the curve. 
\end{remark}

\begin{proof}[Proof of Lemma \ref{lem:toproveinsubsecT}]
For the Gromov-Witten invariants, we apply the divisor axiom, the Lemma \ref{lem:computeforTclass} above and use the fact that $h ( T ) = h ( L - 2 F ) = 1$ and $u ( T ) = - 2 u ( F ) = - 2$. All products $h \star_T h, u \star_T h, u \star_T u$ have degree $2 + 2 - 4 = 0$, so they are a multiple of $y$. The coefficients are computed via Poincare duality as before. Similarly $pt \star_T h = $ and $pt \star_T u$ have degree $4$, and 
\begin{equation}
\begin{split}
\int_Y (pt \star_F h) \cup h &= GW_{F,3}(pt,h,h) = 6,  \\
\int_Y (pt \star_F h) \cup u &= GW_{F,3}(pt,u,h) = -12. \\
\int_Y (pt \star_F u) \cup h &= GW_{F,3}(pt,h,u) = -12,  \\
\int_Y (pt \star_F u) \cup u &= GW_{F,3}(pt,u,u) = 24. \\
\end{split}
\end{equation}
Similarly for all the others.
\end{proof}
\subsection{Ruling lines} \label{subsec:R}
Let $R = L - 3F \in H_2(Y;\Z)$. Appealing to Section \ref{subsec:watchtower} once more, the rational curves in $\MM_{0,3}(R,\textbf{J}_{alg})$ parametrize the strict transform of two families, each comprising an $\P^1$-worth of ruling lines in the strict transform $Q'$ of the quadric $Q \subset \P^3$. This moduli space is smooth, and of the expected dimension 
\begin{equation}
6 + 2c_1(L-3F) -6 +2= 4 = 2 +2.
\end{equation}
To compute $GW_{R,1}$, we note that the image of the 1-point evaluation pseudocycle restricted to each of the $\P^1$-families is exactly $Q' \subset Y$. Thus by Lemma \ref{lemma:relationsintheblowup}
\begin{equation}
\begin{split}
GW_{R,1}(\ell)&=2(2H-E) \cdot L = 4, \\
GW_{R,1}(f)&=2(2H-E)\cdot F = (4H - 2E) \cdot F = 2. 
\end{split}
\end{equation} 
\begin{remark}
We remark that these numbers (that we computed explicitly), could also be deduced from the classical predictions for the virtual number of 3-secants (due to Cayley). A modren account appears in \cite{MR1068965}, and the formulas for 1-point GW invariants that were computed in Example 9.1 \cite{MR1832328}.
\end{remark}
To compute $GW_{R,2}$, applying the symmetry and divisor axioms as before, we only need to consider pairs of cohomology classes of degree $3$. \vspace{0.5em} 
\begin{itemize}
\item
$GW_{R,2}(u \cdot a_i, u \cdot a_j)$.
\item
$GW_{R,2}(u \cdot a_i, u \cdot b_j)$.
\item
$GW_{R,2}(u \cdot b_i, u \cdot a_j)$.
\item
$GW_{R,2}(u \cdot b_i, u \cdot b_j)$.
\end{itemize}
These can be easily computed using recursion. 
\begin{lemma} \label{lem:Rcompute3classes}
Let $w_2,w_1 \in H^1(C)$. Denote $x_2 = u \cdot w_2$ and $x_1 = u \cdot w_1$. Then 
\begin{equation} 
GW_{R,2}(x_2 , x_1 )  =  2 \deg(w_2 \cdot w_1).
\end{equation} 
\end{lemma}
\begin{proof}
By Lemma \ref{lem:gathmann} with $a=-3$ and the GW invariants computed above,
\begin{equation} 
-3 \cdot GW_{R,3}(x_2 , x_1) + 2 \deg(w_2 \cdot w_1) = 6 \cdot \deg(w_2 \cdot w_1) -2 \cdot 5 \cdot \deg(w_2 \cdot w_1)
\end{equation} 
\end{proof}
To compute $GW_{R,3}$, we note that because of degree considerations, we can always use the symmetry and divisor axiom to reduce the computation to the previous cases.

\section{Computation (III): determining the ambiguity} \label{sec:ambiguity}
Our goal in this section is to prove Proposition \ref{prop:quantumproduct}, as well as establish some results that would help with the computation of the main term in the quantum Massey product. \\

Recall that the proof of Proposition \ref{prop:4.2} is based on the fact that for suitable choices of $J$ and $\textbf{J}$, the unparametrized moduli space $\MM(A,J)$ can be identified with a fibre of the projection $\MM(A,\textbf{J}) \to S^1$. In other words: we observe that $1 \times 1 \times \mathfrak{t}$ is Poincar\'{e} dual to the submanifold 
\begin{equation}
C_3 := \YY \times \YY \times Y_{t_0} \subset \YY^{\times 3}. 
\end{equation}
So if we let $\circ$ denote the intersection product on 
\begin{equation}
H^\bullet(\YY^{\times 3}) \iso H^\bullet(\YY) \otimes H^\bullet(\YY) \otimes H^\bullet(\YY) , 
\end{equation}
(where the identification was by a K\"{u}nneth isomorphism), then 
\begin{equation} \label{eq:magic3}
(ev_3)_*[\MM(A,\textbf{J})] \circ [C_3] = (i \times i \times i)_* (ev_3)_*[\MM(A,J)]. 
\end{equation}

Note also that the support of the Poincar\'{e} dual of $1 \times \mathfrak{t} \times \mathfrak{t}$ can be taken to be 
\begin{equation}
\YY \times Y_{t_1} \times Y_{t_0} \subset \YY^{\times 3}
\end{equation}
with $t_0 \neq t_1$, thus disjoint from the image of the 3-point Gromov-Witten pseudocycle. 

\begin{remark} \label{rem:magical}
Of course, similar considerations apply to 2-point Gromov-Witten pseudocycle: we can define a submanifold
\begin{equation} \label{eq:magic2}
C_2 := \YY  \times Y_{t_0} \subset \YY^{\times 2},
\end{equation}
Poincar\'{e} dual to $1 \times \mathfrak{t}$; and 
\begin{equation}
(ev_2)_*[\MM(A,\textbf{J})] \circ [C_2] = (i \times i)_* (ev_2)_*[\MM(A,J)]. 
\end{equation} 
That would be useful later, when we want to use the divisor axiom. 
\end{remark}
Equation \eqref{eq:magic3} immediately implies that 

\begin{equation}
(\mathfrak{t} z_2) \tilde{\star} z_1 = \mathfrak{t}(z_2 \star z_1) = z_2 \star (z_1 \mathfrak{t}),  
\end{equation}
for any two cohomology classes $z_2,z_1 \in H^\bullet(Y)$; Similarly, it is clear that 

\begin{equation}
(\mathfrak{t} z_2) \tilde{\star} (\mathfrak{t} z_1) = 0. 
\end{equation}

Moreover, if we denote 
\begin{equation} \label{eq:newcontribution}
z_2 \tilde{\star} z_1 = w' + \mathfrak{t}w''
\end{equation}
where $z_2,z_1,w',w'' \in H^\bullet(Y)$ then $w' = z_2 \star z_1$ by \ref{prop:4.2}, where the product is taken in the fiber $H^\bullet(Y)$. So the only possibly interesting term is $w''$. We will call such contribution a \textbf{parametric correction term}. 

\begin{remark}
Proposition \ref{prop:quantumproduct} is equivalent to the statement that all parametric correction terms vanish. 
\end{remark}

A basic observation is that in order for $\widetilde{GW}(z_3,z_2,z_1)$ to have a nontrivial correction term, we must have

\begin{equation}
\deg(z_3) + \deg(z_2) + \deg(z_1) = 7 + 2c_1(A) +6 - 6 = 7 + 2c_1(A) \: , \: A \neq 0
\end{equation}

which is an odd number. Thus, the number of $H^3(Y)$-classes must be odd. Without loss of generality, we can assume 
\begin{equation}
\deg(z_3) \geq \deg(z_2)  \geq \deg(z_1) \geq 2. 
\end{equation}

\begin{observation} \label{thm:observation}
The following are the only cases we need to check: \vspace{0.5em}
\begin{enumerate}
\item
$\deg(z_1)=2$, which means $\deg(z_2)=3$ and $\deg(z_1)=2 + 2c_1(A)$. \vspace{0.5em}
\item
$\deg(z_i)=3$ for all $i$ (and then $c_1(A)=1$); \vspace{0.5em}
\item
$\deg(z_1)=3$, but $\deg(z_3) \geq \deg(z_2) \geq 4$. In that case $\deg(z_3) + \deg(z_2) = 4+2c_1(A)$ forces $c_1(A)=2$ and both degrees are four. \vspace{0.5em}
\end{enumerate}
\end{observation}
\begin{remark}
Note that  $\deg(z_1)=\deg(z_2) = 2$ is impossible, because that would imply $\deg(z_3) = 3$, and so $c_1(A)=0$ which contradicts the energy axiom because of monotonicity. 
\end{remark}

In the following subsection we analyze each of these cases for all classes $A \in H_2^S(M)$ with $c_1(A) \leq 2$ or $A=2F$. 

\subsection{Proof of Lemma \ref{lem:decomposecohomology}} \label{subsec:proofoflemma}
We will think of $Y = Bl_p X$ as the 1-point blow up of a smoothing of a degenerate Fano 3-fold with two $A_5$-configurations of vanishing cycles. Let $T_i$, $i=1,2$ be a tubular neighbourhood around each $A_5$-configuration of vanishing cycles, and let $T = T_1 \cup T_2$. 

\begin{lemma} \label{lem:choice1}
Every homology class in $H_d(Y)$ of degree $3 \neq d<6$ can be represented by a manifold $C \subset Y$ which is disjoint from $T$.
\end{lemma}
\begin{proof}
This is obvious for any class of degree $< 3$ due to transversality. It remains to understand $H_4(Y)$. This is a consequence of Poincar\'{e} duality: For every $c \in H^2(Y;\Z)$ there exists a line bundle $\LL$ such that $c = c_1(\LL)$. But if we consider the restriction $\LL|_V$ for any vanishing cycle $V \iso S^3$, then $H^2(S^3) = 0$ and we can find a non-vanishing section and extend it to all of $Y$. 
\end{proof}

\begin{lemma} \label{lem:choice3}
Every homology class in $H_d(Y)$ of degree $d<6$ can be represented by a manifold $C \subset Y$ that has the following property: there exists a tubular neighbourhood $C \subset T_C \subset Y$ such that 
\begin{equation}
\psi_Y \big|_{T_C} \iso_{symp} id.
\end{equation}
\end{lemma}
\begin{proof}
For homology classes of degree $\neq 3$, this is immediate since the monodromy is compactly supported near the vanishing cycles (say near $T$). But in fact, we can do better -- the monodromy is the product of two $\beta=6$ powers of the total monodromy of an $A_5$-singularities (which are a special case of weighted homogeneous singularities) which is symplectically isotopic to the right-handed fibered Dehn twist that comes from the circle action on the contact type boundary (See: \cite[Lemma 4.16]{MR1765826}, and the recent \cite{eprintacuavdek}.) In particular, it is compactly supported near $\partial T$ and can be made disjoint on the chain-level from the vanishing cycles themselves.
\end{proof}

\begin{corollary} \label{cor:choice2}
Let $z$ be any class in $H^{\bullet}(Y)$. Taken as a cohomology class in the mapping torus via \eqref{eq:additivedecomposition}, it can be represented by as $PD([\CC])$ where the cycle $\CC = C \times S^1$ is of the form 
\begin{equation}
C \times [0,1] \subset Y \times [0,1] \to \YY, 
\end{equation}
with $C \subset Y$ a manifold as in Lemma \ref{lem:choice3}.
\end{corollary}
Of course, any cohomology class of the form $\mathfrak{t}z$ can be represented as the Poincar\'{e} dual of a cycle $C \times \left\{t\right\}$ for some $ t \in S^1$. Arguing via intersection theory, the facts that isomorphism \eqref{eq:additivedecomposition} preserves the cup-product is immediate.

\subsection{Topology of $\mathcal{E}$} \label{subsec:topologyofE}
To understand the Gromov-Witten pseudocycle for fiber classes and later the Morse theory on the mapping torus, we need to understand the exceptional divisor first. In the case of a single fiber, it is obvious that $E = C \times \P^1$ (the exceptional divisor is a ruled surface.) The purpose of this subsection is to prove the following result

\begin{lemma} \label{lem:exceptionalfibration}
The same holds for the entire 5-dimensional exceptional fibration, that is, $\EE$ is a trivial rank $1$ projective bundle over $\CC$. 
\end{lemma}

Thus, it is clear from this (and the discussion in subsection \ref{subsec:F}) that we will obtain no parametric correction terms to $\tilde{\star}$ from $A = F$ and $A = 2F$. 

\subsection{Projective bundles}
\begin{definition}
Let $Z$ be a paracompact topological space, and $n \geq 1$. We define a rank $n$ \textbf{projective bundles} over $Z$, to be a fiber bundle with fiber $\CP{n}$ and structure group 
\begin{equation}
Aut(\CP{n}) = PGL_{n+1}(\C). 
\end{equation}
\end{definition}
\begin{example}
There is one obvious class of projective bundles on $Z$, namely those we obtain by projectivizing a complex vector bundle. Recall that if $V$ is a complex vector bundle over $Z$, then taking the fibre-wise projective space yields the associated projective bundle:
\begin{equation}
\P(V):=\bigslant{(V \setminus \left\{0\right\})}{\sim} \: \: , \: \: \text{where }v \sim \lambda v \text{ for every }\lambda \in \C - \left\{0\right\}
\end{equation}
The fibres of $\P(V)$ are complex projective spaces $\P(V_z)$ and if $V$ is a holomorphic vector bundle over a complex manifold then the total space $\P(Z)$ is a complex manifold. 
\end{example}
If a projective bundle is trivialized on a open cover $\left\{U_\alpha\right\}_\alpha$ of $Z$, then to specify the bundle all we need to do is specify transition functions on all the $U_\alpha \cap U_\beta$ between the two trivializations, subject to the constraint that they are compatible with the triple overlaps. But this gives a 1-cocycle, with values in the structure group $PGL_{n+1}(\C)$. As usual, changing the 1-cocycle by a coboundary yields an isomorphic bundle. Therefore $H^1(Z;PGL_{n+1}(\C))$ classifies $\CP{n}$-bundles on $Z$. Since $H^1(Z;GL_{n+1}(\C))$ classifies rank $n + 1$ vector bundles, again via transition functions, the natural map on cocycles 
\begin{equation}
H^1(Z;GL_{n+1}(\C)) \to H^1(Z;PGL_{n+1}(\C))
\end{equation}
corresponds to projectivization.
\begin{lemma}
$H^2(Z;\OO^*_Z) \iso H^3(Z;\Z)$.
\end{lemma}
\begin{proof}
We consider the exponential exact sequence
\begin{align*}
0 \to \Z \to \OO_Z \to \OO^*_Z \to 0.
\end{align*}
Since the structure sheaf $\OO_Z$ is acyclic, the long exact sequence in cohomology gives 
\begin{align*}
0 \to H^2(Z;\OO^*_Z) \iso H^3(Z;\Z) \to 0.
\end{align*}
\end{proof}
\begin{lemma}
Assume that $H^3(Z;\Z)$ has no $n$-torsion. Then every $PGL_n$-bundle lifts to a $GL_n$-bundle, and any two vector bundles give the same projective bundles if and only if they obtained from each other by tensoring ("twisting") with an invertible line bundle.
\end{lemma}
\begin{proof}
Consider the pair of short exact sequencess

\[
\begin{tikzpicture}[description/.style={fill=white,inner sep=2pt}]
\matrix (m) [matrix of math nodes, row sep=1.5em,
column sep=2.5em, text height=1.5ex, text depth=0.25ex]
{ 1 & \OO^* & GL_n(\C)_Z & PGL_n(\C)_Z & 1   \\
  1 & \mu_n & SL_n(\C)_Z & PGL_n(\C)_Z & 1     \\};
\path[->] 
(m-1-1) edge  (m-1-2)
(m-1-2) edge  (m-1-3)
(m-1-3) edge  (m-1-4)
(m-1-4) edge  (m-1-5)
(m-2-1) edge  (m-2-2)
(m-2-2) edge  (m-2-3)
(m-2-3) edge  (m-2-4)
(m-2-4) edge  (m-2-5)
(m-2-2) edge  (m-1-2)
(m-2-3) edge  (m-1-3)
(m-2-4) edge  (m-1-4);
\end{tikzpicture}
\]
and the corresponding long exact sequences in non-Abelian cohomology
\[
\begin{tikzpicture}[description/.style={fill=white,inner sep=2pt}]
\matrix (m) [matrix of math nodes, row sep=1.5em,
column sep=2.5em, text height=1.5ex, text depth=0.25ex]
{ H^1(Z,\OO_Z^*) & H^1(Z;GL_n(\C)) & H^1(Z;PGL_n(\C)) & H^2(Z,\OO_Z^*)   \\
  H^1(Z,\mu_n) & H^1(Z;SL_n(\C)) & H^1(Z;PGL_n(\C)) & H^2(Z,\mu_n)     \\};
\path[->] 
(m-1-1) edge  (m-1-2)
(m-1-2) edge  (m-1-3)
(m-1-3) edge  (m-1-4)
(m-2-1) edge  (m-2-2)
(m-2-2) edge  (m-2-3)
(m-2-3) edge  (m-2-4)
(m-2-1) edge  (m-1-1)
(m-2-2) edge  (m-1-2)
(m-2-3) edge  (m-1-3)
(m-2-4) edge  (m-1-4);
\end{tikzpicture}
\]
Since $H^2(Z; \mu_n)$ is n-torsion and the middle map 
\begin{equation}
H^1(Z;PGL_n(\C)) \to H^1(Z;PGL_n(\C))
\end{equation}
is an isomorphism, the image of $H^1(Z;GL_n(\C))$ in $H^2(Z,\OO_Z^*) \iso H^3(Z;\Z)$ is n-torsion, hence zero by assumption. 
\end{proof}

\subsection{Proof of Lemma \ref{lem:exceptionalfibration}}
Consider the pullback of the normal bundle $\NN \to \mathcal{C}$. According to the previous subsection, it is enough to show that there exists a complex line bundle $\LL \to \mathcal{C}$ and a global splitting 
\begin{equation} 
\NN \iso \LL \otimes (\OO \oplus \OO) = \LL \oplus \LL.
\end{equation} 
We note that over each fixed $s$-slice ($s \in S^1$), we have a splitting as a direct sum of holomorphic bundles as in \eqref{eq:12_18}. Now fix a base point $* \in S$. Since the first chern class $c_1(\NN_*) = 30$ is even, we can choose a line bundle $\LL_* \to C_*$ and an isomorphism of complex vector bundles
\begin{equation} 
\NN_* \iso \LL_* \otimes (\OO \oplus \OO) = \LL_* \oplus \LL_* \: , \: c_1(\LL_*)=15.
\end{equation} 
This is useful in order to get a local toric picture as a $\C^2$-bundle. However, we have a topological problem: as we try to extend our initial isomorphism and go in a loop around the base, the monodromy acts on our fixed frame by a gauge transformation which we denote as
\begin{equation}
\Phi \in C^\infty(C_*,U(2)). 
\end{equation}
There is a potentially problematic part coming from the phase $U(2) \stackrel{det}{\rightarrow} S^1$. Fortunately, the classification of complex vector bundles over a Riemann surface extends to 3-manifolds.
\begin{lemma}[\cite{MR0123331}] \label{thm:u2bundles}
Let $X$ be a CW-complex, $\dim(X) \leq 3$. 
\begin{enumerate}
\item
Two $U(2)$-bundles $E_1$ and $E_2$ on $X$ are isomorphic if and only if $c_1(E_1) = c_1(E_2)$. 
\item
A $U(2)$-bundle can be reduced to an $SU(2)$-bundle if and only if $c_1(E) = 0$. 
\end{enumerate} 
In particular, any $SU(2)$-bundle is trivial.
\noproof
\end{lemma}
So it is enough to prove that $c_1(\NN)$ is divisible by two in the cohomology ring $H^\bullet(\mathcal{C};\Z)$. Equivalently, we must prove that the image of first Chern class under mod two reduction 
\begin{equation}
\begin{split}
H^\bullet(\mathcal{C};\Z) &\to H^\bullet(\mathcal{C};\Z_2) \\
c_1(\NN) &\to w_2(\NN)
\end{split}
\end{equation}
vanishes. 
\begin{definition}
We define a \textbf{Spin structure} on a manifold bundle $Z \to S^1$ to be a Spin structure on the vertical tangent bundle $T^v Z$. 
\end{definition}
The following theorem is well-known (see e.g., \cite[p.~78--85]{MR1031992}.)
\begin{theorem}
Let $Z$ be a closed oriented manifold. An oriented vector bundle $E \to Z$ has a \textbf{Spin structure} if and only if $w_2(E) = 0$. Furthermore, if $E'' = E \oplus E'$ as oriented vector bundles, then the choice of spin structures on any of the three bundles determines a spin structure on the third. \noproof
\end{theorem}
Thus, the normal sequence coupled with the fact that the ambient space $S^1 \times \P^3$ is Spin readily imply that $\NN$ is Spin if and only if the surface bundle $\CC \to S^1$ is Spin as well.
\begin{theorem} \label{th:steifel}
Let $Z$ be a smooth, connected oriented 3-manifold. Then $Z$ is parallelizable. 
\end{theorem}
\begin{proof}
The Wu relation $w_1^2 + w_2 = 0$ implies that if $Z$ is oriented it is spin as well. But a connected 3-manifold is parallelizable if and only if the first two Stiefel-Whitney classes $w_1,w_2$ vanish. 
\end{proof}
To complete the proof of the Lemma we only need to recall that,
\begin{proposition}[\cite{MR0286136}] Any surface bundle over $S^1$ admits a spin structure.
\end{proposition}
\begin{proof}
The total space of any surface bundle $p : Z \to S^1$ is an oriented 3-manifold, thus parallelizable by \ref{th:steifel}. Now we see that 
\begin{align*}
TZ \iso Z \times \R^3 \iso T^v Z \oplus p^* T \!S^1 \iso T^v Z \oplus (Z \times \R). 
\end{align*}
Thus all the Stiefel-Whitney classes of $T^v Z$ must vanish. In particular, the surface bundle $p$ is spin.
\end{proof}

\subsection{Secant lines} 
Case (1) in \ref{thm:observation} is checking $\widetilde{GW}_{T,2}(z_2,z_1)$ with 

\begin{equation}
\begin{split}
z_2 &\in \left\{ pt\right\}, \\
z_1 &\in \left\{ \mathfrak{u} \cdot a_i , \mathfrak{u} \cdot b_j\right\}.
\end{split}
\end{equation}
It is clear that no such secant line exist (e.g. by considering the Fano 3-fold $X$ instead, and using a degeneration argument for the Fano surface). Case (2) is impossible due to degree reasons. We consider case (3) which means we need to compute $\widetilde{GW}_{T,2}(z_2,z_1)$ where $z_2$ and $z_1$ are either the line or the fiber classes. The same degeneration argument to the singular Fano 3-fold shows that no parametric correction terms exist. 

\subsection{Ruling lines} 
Case (1) in \ref{thm:observation} reduces to checking $\widetilde{GW}_{R,2}(z_2,z_1)$ with 

\begin{equation}
\begin{split}
z_2 &\in \left\{ \ell, f\right\}, \\
z_1 &\in \left\{ \mathfrak{u} \cdot a_i , \mathfrak{u} \cdot b_j\right\}.
\end{split}
\end{equation}

However in all those cases $\widetilde{GW}_{R,2}(z_2,z_1) =0$ (in particular, there are no $\mathfrak{t} \cdot w''$ terms): choose a chain-level representative for $z_2$ using Corollary \ref{cor:choice2}. As in subsection \ref{subsec:R}, note that there are only finitely many ruling lines in each fiber $Y_t = Bl_{C_t} \P^3$ that meet a given degree $4$ class (and that we understand them explicitly); it is clear that non of them intersects the vanishing cycles. \\

Case (3) is impossible because $c_1(R) \neq 2$. It remains to deal with case (2). 

\begin{lemma}
There are no parametric correction terms to $\widetilde{GW}_{R,3}(z_3,z_2,z_1)$ when $\deg(z_i)=3$.  
\end{lemma}

The idea for the proof is to consider the image of the 3-point parametrized Gromov-Witten pseudocycle under blowdown 

\begin{equation} \label{eq:quadricbundle}
\MM(A,\textbf{J}) \stackrel{ev_3}{\longrightarrow} \YY \times_\pi \YY \times_\pi \YY \stackrel{b \times b \times b}{\longrightarrow} (\P^3 \times \P^3 \times \P^3) \times S^1
\end{equation}
where in the right hand side we have implicitly a fixed choice of trivialization of the bundle of projective spaces. This is the same as the evaluation map from
\begin{equation} \label{eq:quadricbundle}
\begin{split}
\{  (t,p_3,p_2,p_1) \: \big| \: &\text{ the points $p_i$ are distinct smooth}  \\
&\text{ points on $\QQ_t$ } \} / Aut(\P^1 \times \P^1) 
\end{split}
\end{equation}
where $\mathcal{Q} \iso Q \times S^1$ is the quadric bundle (see the discussion in subsection \ref{subsec:remarkonbirationalmodels}.) Instead of computing the pushforward of the virtual fundamental class and how it behaves with respect to cohomology classes in the triple product, we can project to a single copy of $\YY$ and dualize. The domain of the induced map on cohomology by 1-point evaluation is
\begin{equation}
H^2(\YY) \iso H^2(Y) \oplus (H^1(Y) \otimes H^1(S^1)) \to H^\bullet(\MM(A,\textbf{J})) 
\end{equation}
where the direct sum on the right-hand side is the K\"{u}nneth decomposition of the product. Since the 1-parameter family of embedded curves $C_t$ actually come from \emph{smoothing} a singular curve $C_0$, we can form our mapping torus by taking the boundary circle of a punctured disc $0 < |t| \leq \delta$ for a sufficiently small $\delta>0$. It is clear that $H^1(Y) \otimes H^1(S^1)$ goes to zero, because all the 1-cycles are supported in a ball of arbitrarly small radius around the $A_5$-singularities and can not meet any of the ruling lines that generate $H_2(\mathcal{Q}) \iso H_2(Q)$.

\section{Parametrized Gromov-Witten theory}\label{sec:nodal}

Let $\mathcal{P} \to \mathcal{Q}$ be a smooth family of singular pearl trees, of a fixed combinatorial type $\Upsilon$. Then we can apply stablization, and get the following commutative diagram

\begin{equation}
\begin{tikzpicture}
\matrix(m)[matrix of math nodes,
row sep=3em, column sep=2.5em,
text height=1.5ex, text depth=0.25ex]
{\mathcal{P} & \mathcal{P}'\\
\mathcal{Q} & \mathcal{Q}' \\};
\path[->,font=\scriptsize]
(m-1-1) edge (m-1-2)
(m-1-1) edge (m-2-1)
(m-1-2) edge (m-2-2)
(m-2-1) edge node[auto] {st} (m-2-2);
\end{tikzpicture}
\end{equation}

where $\mathcal{P}' \to \mathcal{Q}'$ is a smooth family of singular pearl trees, of a fixed \emph{stable} combinatorial type $\Upsilon' = st(\Upsilon)$. Let $\mathcal{C} \to \mathcal{Q}$ (resp. $\mathcal{C}' \to \mathcal{Q}'$) denote the two dimensional sphere part. Then, $\mathcal{C}$ inherits an almost complex structure by pulling back the almost complex structure on $\mathcal{C}'$ that comes from restricting the universal choice for pearls to $\mathcal{Q}' \subseteq \mathcal{Q}_{\Upsilon'}$. Note that $J$ is constant on every sphere component that gets collapsed by the stabilization map. For every stable map, there is an induced parametrized del-bar operator 
\begin{equation} \label{eq:familydelbar}
\frac{1}{2}
\end{equation}
Note that the fixed combinatorial type of $\mathcal{C}$ can be represented as a \textbf{rooted forest} $(T,E,\Lambda)$, i.e., if we denote the number of connected components of $\mathcal{C}$ by $\ell$, then the combinatorial type is a disjoint union of a collection of rooted, $d_\mu$-labelled trees
\begin{equation}
(T_\mu,E_\mu,\Lambda_\mu) \: , \: i \in \left\{1,\ldots,\ell\right\}. 
\end{equation}
\begin{definition}
We denote
\begin{equation}
Z(T_\mu) \subset (S^2)^{|E_\mu|} \times (S^2)^{d_\mu}
\end{equation}
to be the set of all tuples $\z = (\left\{z_{\alpha \beta}\right\}_{\alpha E \beta}, \left\{z_i\right\}_{i=1}^{d_\mu})$ such that the points $z_{\alpha \beta} \in S^2$ for $\alpha E \beta$ and $z_i \in S^2$ for $\alpha_i = \alpha$ are pairwise distinct for every $\alpha \in T_\mu$; and under the cross-ratio embedding $\rho_{ijkl} \in \mathcal{W} \subset (S^2)^N$ for every $1 \leq i,j,k,l \leq d_\mu$. 
\end{definition}
We observe that the curve $\mathcal{C}_\mu \to \mathcal{Q}$ is equivalent to a smooth map $\z_\mu : \mathcal{Q} \to Z(T_\mu)$. Let $\A$ be a collection of spherical homology classes $A_\alpha \in H_2(M;\Z)$, indexed by vertices of $T$. 
\begin{definition}
We consider the moduli space 
\begin{equation}
\tilde{\MM}(\A)
\end{equation}
of all tuples $(\b = \left\{b_\mu\right\},\u = \left\{u_\mu\right\})$ where: \vspace{0.5em}
\begin{itemize}
\item
$b_\mu \in B$. \vspace{0.5em}
\item
$u_\mu : \z_\mu \to E_b$ is a smooth family of nodal maps whose restriction to every fiber $q \in \mathcal{Q}$ is $J_{b,q,z}$-holomorphic. \vspace{0.5em}
\item
The restriction of $\u$ to every sphere $S_\alpha$ has homology class $A_\alpha$. \vspace{0.5em}
\end{itemize}
There is a reparametrization group $G$ which acts on $\tilde{\MM}(\A)$ by
\end{definition}

This moduli space can be described as the inverse image of the evaluation map in the following way: The parametrized de-singularization of each individual connected component is denoted as $\mathcal{C}_\mu^\nu \to \mathcal{Q}$. We write $\mathcal{C}^\nu \to \mathcal{Q}$ for the disjoint union of $\mathcal{C}_\mu^\nu$, and note that it is a de-singularization of $\mathcal{C}$. The fiber of $\mathcal{C}^\nu$ consists of disjoint spheres $S_\alpha \iso S^2$, indexed by the vertices of $T$. 
\begin{definition}
There is a moduli space of \textbf{simple tuples} 
\begin{equation}
\tilde{\mathcal{ST}}(\A)
\end{equation}
is defined to be all tuples $(\b = \left\{b_\alpha\right\},\u = \left\{u_\alpha \right\})$ where: \vspace{0.5em}
\begin{itemize}
\item
$b_\alpha \in B$. \vspace{0.5em}
\item
$u_\alpha : S_\alpha \times \mathcal{Q} \to E_{b_\alpha}$ is a smooth family of maps whose restriction to every fiber $q \in \mathcal{Q}$ is $J_{b,q,z}$-holomorphic. \vspace{0.5em}
\item
Each map $u_\alpha$ is somewhere injective, and if $\alpha \neq \alpha'$ and $b_{\alpha} = b_{\alpha'}$ the images of $u_\alpha$ and $u_{\alpha'}$ are distinct. 
\end{itemize}
\end{definition}

\begin{definition}
We define the following properties for an almost complex structure $J$ on the family $\mathcal{P}' \to \mathcal{Q}'$: \vspace{0.5em}
\begin{description}
\item[(MF)] 
\item[(H)] $D_{b,q,u}$ is onto for every solution $u_\alpha : S_\alpha \times \mathcal{Q} \to E_{b_\alpha}$ of \eqref{eq:familydelbar}, where $\alpha$ is a root vertex. \vspace{0.5em}
\item[(V)] $D_{b,q,z,v}$ is onto for each $z \in S^2$ and each simple $J$-holomorphic sphere $v : S^2 \to M$. \vspace{0.5em}
\item[(E)] The evaluation map \eqref{eq:evaluationmap} is transverse to $\tilde{\Delta} \times \left\{\underline{w}\right\}$ for every rooted k-labelled tree $T$ and every collection of homology classes $\underline{A}$. \vspace{0.5em}
\end{description}
The set of families of almost complex structures that satisfy one of them would be denoted as $\complexJ_{(\ldots)}(S^2,M,\omega;\underline{w})$. 
\end{definition}

\begin{theorem}
For every $J$, the evaluation map (6.7.2) is a pseudocycle of dimension $\dim(M) + 2c_1(A) + 2k-2|I|$. 
\end{theorem}

\textbf{Notation.} Fix a real number $p > 2$ and integers $k \geq 1$, and $\ell$ sufficiently large. Let $W^{k,p}(S^2,M)$ denote the completion of $C^\infty(S^2,M)$ with respect to a distance function based on the Sobolev $W^{k,p}$-norm. This norm is defined as the sum of the $L^p$-norms of all derivatives up to order $k$. By the Sobolev embedding theorem maps $u \in W^{k,p}(S^2,M)$ are continuous. The corresponding metric on $C^\infty(S^2,M)$ can be defined, for example, by embedding $M$ into some high dimensional Euclidean space $\R^N$ and then using the Sobolev norm on the ambient space $W^{k,p}(S^2,\R^N)$. Since $S^2$ and $M$ are compact any two distance functions are compatible, and so as a topological vector space $W^{k,p}(S^2,M)$ is independent of choices. Similarly, we denote by $W_{S^2}^{k,p}(u^* TM)$ the completion of the space $\Omega^0(S^2,u^*TM)$ of smooth sections of the pullback bundle $u^* TM \to M$ with respect to the Sobolev $W^{k,p}$-norm; and $W_{S^2}^{k,p}(u^* TM \otimes_{J} \wedge^{\!{0,1}} S^2)$ is the completion of the space $\Omega^{0,1}(S^2,u^*TM)$ of $J$-antilinear forms with values in $u^* TM$. 

\subsection{Step 1 - choosing good Floer data} 
As in section \cite[Section 3.2]{MR2954391}, there is a Banach bundle $\mathcal{E} \to \mathcal{B} \times \complexJ(M,\omega)^\ell$, where
\begin{equation}
\mathcal{B} =  W^{k,p}(S^2,M)  \: , \: \mathcal{E}_{J,u} = W_{S^2}^{k-1,p}(u^* TM \otimes_{J} \wedge^{\!{0,1}} S^2).
\end{equation}

\begin{proposition} The \textbf{universal delbar operator} 
\begin{equation}
\mathcal{F} : \mathcal{B} \times  \complexJ(M,\omega)^\ell \to \mathcal{E}  \: , \: \mathcal{F} (u, J) \to \delbar_J(u)
\end{equation}
defines a $C^{\ell-k}$-smooth section, and the linearization 
\begin{equation}
\begin{split}
D\mathcal{F}_{(u,J)}(\zeta,Y) &: W_{S^2}^{k,p}(u^* TM) \times C^\ell(M,End(TM,J,\omega)) \\
&\to W_{S^2}^{k-1,p}(u^* TM \otimes_{J} \wedge^{\!{0,1}} S^2) 
\end{split}
\end{equation}
is given by the formula
\begin{equation} \label{eq:J}
D\mathcal{F}_{(u,J)}(\zeta,Y) := D_u \zeta + \frac{1}{2} Y(u) \circ du \circ j_{\CP{1}}. 
\end{equation}
\end{proposition}
Now fix pairwise disjoint points $z_1,\ldots,z_k \in S^2$. Evaluation at these points defines a smooth map
\begin{equation}
ev^k : \mathcal{B} \to M^{\times k}
\end{equation}

\begin{proposition}
Let $k \in \N$, $u \in \mathcal{B}$ and $J \in \complexJ^\ell_{S^2,\tau}(M,\omega)$ satisfy $\delbar_J u = 0$. If $u$ is nonconstant, then the linearization of
\begin{equation}
(\mathcal{F},ev^k) : \mathcal{B} \times \complexJ^\ell(M,\omega) \to \mathcal{E}
\end{equation}
at $(u,J)$ is surjective. More precisely, for any nonempty open subset $U \subset S^2 \setminus \left\{z_1,\ldots,z_k \right\}$, the restriction of the linearization to the subspace
\begin{equation}
T_u  \mathcal{B} \oplus T_J \complexJ^\ell_{U,\tau}(M,\omega) \subset T_u  \mathcal{B} \oplus T_J \complexJ_{S^2,\tau}^\ell(M,\omega)
\end{equation}
of sections with support in $U$ is surjective.
\end{proposition}

\begin{proposition}
If $u \in \mathcal{B}$ is constant and $J \in \complexJ^\ell_{S^2,\tau}(M,\omega)$, then the linearization of
\begin{equation}
(\mathcal{F},ev^1) : \mathcal{B} \times \complexJ(M,\omega)^\ell \to \mathcal{E}
\end{equation}
at $u$ is surjective.
\end{proposition}

\begin{definition} Let $\MM^*(S^2;J)$ denote the space of all simple $J$-holomorphic $u : S^2 \to M$. The \textbf{universal moduli space of simple curves} is
\begin{equation}
\MM^*(S^2;\complexJ(M,\omega)^\ell):= \left\{ (u,J) \: \big| \: u \in \MM^*(S^2;J) , J \in \complexJ(M,\omega)^\ell \right\}
\end{equation}
For any fixed $A \in H_2(M;\Z)$, there is a connected component $\MM^*(A,S^2; \complexJ(M,\omega)^\ell)$ of simple curves in class $[u] = A$. 
\end{definition}

\begin{proposition}
The universal moduli space is a separable $C^{\ell-k}$-Banach submanifold of $\mathcal{B}$. 
\end{proposition}
\begin{proof}
The map 
\begin{equation}
\mathcal{F} : \mathcal{B}  \to \mathcal{E}  \: , \: \mathcal{F} (u, J) \to \delbar_J(u)
\end{equation}
defines a $C^{\ell-k}$-section of the bundle, and the linearization 
\begin{equation}
D\mathcal{F}_{(u,J)}(\zeta,Y) : W_{S^2}^{k,p}(u^* TM) \times C^\ell(M,End(TM,J,\omega)) \to W_{S^2}^{k-1,p}(u^* TM \otimes_{J} \Omega^{0,1} S^2) 
\end{equation}
is given by the formula
\begin{equation} \label{eq:J}
D\mathcal{F}_{(u,J)}(\zeta,Y) := D_u \zeta + \frac{1}{2} Y(u) \circ du \circ j. 
\end{equation}
The proof of \cite[Proposition 3.2.1]{MR2954391} shows that the linearization is surjective whenever $u \in \MM^*(A,S^2; \complexJ(M,\omega)^\ell)$, and therefore each connected component $\MM^*(A,S^2; \complexJ(M,\omega)^\ell)$ is a separable $C^{\ell-k}$-Banach submanifold. Since the cohomology of $M$ is finitely generated, the universal moduli space is 2nd countable. 
\end{proof}

\begin{definition}
Let $A \in H_2(M;\Z)$. An almost complex structure $J$ on $M$ is called regular for $A$ if $D_u$ is onto for every $u \in \MM^*(A,S^2; \complexJ(M,\omega)^\ell)$. We say that $J$ is \textbf{regular} if it regular for every $A$. 
\end{definition}

\begin{proposition}
The set of regular almost complex structures, denoted
\begin{equation}
\complexJ_{reg}(M,\omega) \subseteq \complexJ(M,\omega)
\end{equation}
is comeager. If $J \in \complexJ_{reg}(M,\omega)$ then for any $A$, the space $\MM^*(A,S^2; \complexJ(M,\omega))$ is a smooth manifold of dimension 
\begin{equation}
2n+2c_1(A)
\end{equation}
with a natural orientation.
\end{proposition}
\begin{proof}
The same as \cite[Theorem 3.1.5]{MR2954391}.
\end{proof}

\begin{definition}
We choose our Floer data to be an almost complex structure $J \in \complexJ(\pi,\Omega)$ such that for every $b \in B$, the restriction $J_b \in \complexJ_{reg}(E_b,\omega_b)$.
\end{definition}

\subsection{Step 2 - controlling the bubbles}
We fix a homology class $A \in H^2(M; \Z)$. Let $J = \left\{J_z\right\}_{z \in \Sigma}$ be a smooth family of domain-dependent $\Omega$-tame almost complex structures on $E$ which coincide with the Floer data outside $\Sigma^{cpt}$. We associate to it an $\Sigma$-parametrized family 
\begin{equation}
g = \left\{g_z\right\}_{z \in \Sigma} 
\end{equation}
of fiberwise metrics as in \eqref{eq:fiberwise_metrics}, and denote their Levi-Civita connections by 
\begin{equation}
\nabla = \left\{\nabla^z\right\}_{z \in \Sigma}.
\end{equation} 
Finally, we let
\begin{equation}
\tilde{\nabla} = \left\{\tilde{\nabla}^z\right\}_{z \in \Sigma} 
\end{equation}
denote the corresponding family of Hermitian connections defined by the formula
\begin{equation}
\tilde{\nabla}^z_X Y:= \nabla^z_X Y - \frac{1}{2} J_z (\nabla_X J_z) Y \: , \:  X , Y \in T^v E.
\end{equation}

We consider the moduli space of triples

\begin{equation}
\MM^{Vert,*}(A,E,\Omega;\left\{J_z\right\}) = \bigcup_{b \in B} \bigcup_{z \in \Sigma} \left\{(b,z)\right\} \times \MM^*(s_A(b),E_b,\omega_b;J_z)
\end{equation}

which represent simple vertical $J_z$-holomorphic spheres in $E_b$ representing the class $A$. 

\begin{lemma} The corresponding linearized operator 

\begin{equation}
D_{b,z,v} : T_{(b,z)} (B \times \Sigma) \oplus \Omega^0(\Sigma,v^* TM) \to \Omega^{0,1}(\Sigma,v^* TM)
\end{equation}

is given by 

\begin{equation} \label{eq:verticalJ}
D_{b,z,v}(\zeta,\xi) = D_{v,J_{(b,z)}}(\xi) - \frac{1}{2} J_{(b,z)}(v) \dot{J}_\zeta(v) \partial_{J_{(b,z)}}(v)
\end{equation}
where the linear map 
\begin{equation}
T_{(b,z)} (B \times \Sigma) \to \Omega^0(M,End(TM)) : \zeta \mapsto \dot{J}_\zeta
\end{equation}
denotes the differential of the map $(b,z) \mapsto J_{(b,z)}$.
\end{lemma}
\begin{proof}
To justify this formula let us first fix a small $\epsilon>0$ and two paths 
\begin{equation}
\begin{split}
x &: (-\epsilon,\epsilon) \to M \: , \:  \lambda \mapsto x(\lambda), \\
z &: (-\epsilon,\epsilon) \to B \times \Sigma \: , \: \lambda \mapsto z(\lambda).
\end{split}
\end{equation}
Then the path $\lambda \mapsto z(\lambda)$ defines a complex structure on the pullback tangent bundle $x^* TM$ given by $J_{z(\lambda)} x(\lambda)$. The formula 

\end{proof}
\begin{remark}
Note that when $J_{(b,z)}$ is constant around $(b_0,z_0)$, equation \eqref{eq:verticalJ} specializes to \eqref{eq:J}.
\end{remark}

We denote the space of regular almost complex structures for vertical $J$-holomorphic spheres in the class $A$ as
\begin{equation}
\complexJ^{Vert}_{reg}(\pi,\Omega; A) := \left\{ \left\{J_z\right\} \in \complexJ_{comp}(\pi,\Omega) \: \big| \: z \in \Sigma, v \in \MM^*(s_A(b),E_b,\omega_b;J_z) \: \Rightarrow \: D_{(b,z,v)} \text{ is onto }\right\}. 
\end{equation}

\begin{remark} The important thing is to note that this space is not empty, because of the choice of Floer data in step (1). 
\end{remark}

Let $A \in H_2(M;\Z)$ be a spherical homology class. The next proposition spells out the relevant transversality result for vertical $(J_z)$-holomorphic spheres. 

\begin{proposition}
The set $\complexJ^{Vert}_{reg}(\pi,\Omega; A)$ is comeager in $\complexJ_{comp}(\pi,\Omega)$. Moreover, if 
$\left\{J_z\right\} \in \complexJ^{Vert}_{reg}(\pi,\Omega; A)$ then $\MM^{Vert,*}(A,E,\Omega;\left\{J_z\right\})$ is a smooth oriented manifold of the expected dimension 
\begin{equation}
\dim(M) + 2 + 2c_1(A) + \dim(B). 
\end{equation} 
\end{proposition}
\begin{proof}
This is just a parametrized version of Theorems 3.1.5 and 3.1.7 in \cite{MR2954391}.
\end{proof}

\subsection{Step 3 - graph pseudocycle}
We fix a homology class $A \in H^2(M; \Z)$. We let $J = \left\{J_z\right\}_{z \in S^2}$ be a smooth family of domain-dependent $\Omega$-tame almost complex structures on $E$. We associate to it an $S^2$-parametrized family 
\begin{equation}
g = \left\{g_z\right\}_{z \in S^2} 
\end{equation}
of fiberwise metrics as in \eqref{eq:fiberwise_metrics}, and denote their Levi-Civita connections by 
\begin{equation}
\nabla = \left\{\nabla^z\right\}_{z \in S^2}.
\end{equation} 
Finally, we let
\begin{equation}
\tilde{\nabla} = \left\{\tilde{\nabla}^z\right\}_{z \in S^2} 
\end{equation}
denote the corresponding family of Hermitian connections defined by the formula
\begin{equation}
\tilde{\nabla}^z_X Y:= \nabla^z_X Y - \frac{1}{2} J_z (\nabla_X J_z) Y \: , \:  X , Y \in T^v E.
\end{equation}
\begin{definition} We shall denote the moduli space of vertical J-holomorphic maps $(b,u: \Sigma \to E_b)$ that represent the class $A$ by 
\begin{equation}
\MM(A,E,\left\{J_z\right\}) = \bigcup_{b \in B} \left\{b\right\} \times \MM(A,E_b,\left\{J_z\right\})
\end{equation}
and the subset of simple elements by 
\begin{equation}
\MM^*(A,E, \left\{J_z\right\}) = \bigcup_{b \in B} \left\{b\right\} \times \MM^*(A,E_b,\left\{J_z\right\}).
\end{equation}
\end{definition}

\begin{definition}
For any integer $d \geq 1$, let 
\begin{equation}
\MM(A,E,\left\{J_z\right\})_{0,d}
\end{equation}
be the space of all tuples $(b,u,\underline{z})$ consisting of $b \in B$, a vertical J-holomorphic maps $u : \Sigma \to E_b$ and $d$ pairwise disjoint points $\underline{z} \in \Sigma$. There is an evaluation map 
\begin{equation} \label{eq:evaluationofgraphs}
\begin{split}
ev &: \MM^*(A,E,\left\{J_z\right\})_{0,d} \to E \times_\pi E \times_\pi \ldots \times_\pi E, \\
ev(b,u,\underline{z}) &= (u(z_1),\ldots,u(z_{d})).
\end{split}
\end{equation}
\end{definition}

\begin{proposition}
There is a comeager subset
\begin{equation}
J = \left(J_z\right)_{z \in \Sigma} \in \complexJ_{ref}(\Sigma,E;\Omega) \subseteq \complexJ(\Sigma,E;\Omega)
\end{equation}
such that the evaluation map \eqref{eq:evaluationofgraphs} is a pseudocycle of dimension 
\begin{equation}
2n+2c_1(A)+2d + \dim(B)
\end{equation}
\end{proposition}

\subsection{Fredholm setup}
Fix a real number $p > 2$ and an integer $k \geq 1$. We restrict to a trivializing neighbourhood $U \subseteq B$ such that 
\begin{equation}
(E,\Omega) \iso (B \times M  ,\pi_2^* \omega) 
\end{equation}
as a locally Hamiltonian fibration ($\pi_2$ is the projection $B \times M \to M$ onto the second factor). There is an induced trivilization of the symplectic bundles
\begin{equation}
T^v E \iso \pi_2^* TM.
\end{equation}
Using the trivialization, we consider $J$ as a family of almost complex structures (resp. metrics) 
\begin{equation}
\left\{J_{(b,z)}\right\}_{B \times \Sigma} \: , \: \left\{g_{(b,z)}\right\}_{B \times \Sigma}
\end{equation}
on $TM$. When we fix $b_0 \in B$, we write
\begin{equation}
J^{b_0} = \left\{J_{(b_0,z)}\right\}_{z \in S^2} \: , \: g^{b_0} = \left\{g_{(b_0,z)}\right\}_{z \in S^2}.
\end{equation}
We denote
\begin{equation} \label{eq:modulispace}
\MM|_U(A,E,\left\{J_z\right\}) = \bigcup_{b \in U} \left\{b\right\} \times \MM(A,M,\left\{J_{(b,z)}\right\})
\end{equation}
and think of any vertical holomorphic map $(b,u) \in \MM|_U(A,E,\left\{J_{(b,z)}\right\})$ as a pair of $b \in B$ and $ u : \Sigma \to M$.  \\

Let $W^{k,p}(\Sigma,M)$ denote the completion of $C^\infty(\Sigma,M)$ with respect to a distance function based on the Sobolev $W^{k,p}$-norm. This norm is defined as the sum of the $L^p$-norms of all derivatives up to order $k$. 
By the Sobolev embedding theorem maps $u \in W^{k,p}(\Sigma,M)$ are continuous. The corresponding metric on 
$C^\infty(\Sigma,M)$ can be defined, for example, by embedding $M$ into some high dimensional Euclidean space $\R^N$ and then using the Sobolev norm on the ambient space $W^{k,p}(\Sigma,\R^N)$. Since $\Sigma$ and $M$ are compact any two distance functions are compatible, and so as a topological vector space $W^{k,p}(\Sigma,M)$ is independent of choices. Similarly, we denote by $W_{\Sigma}^{k,p}(u^* TM)$ the completion of the space $\Omega^0(\Sigma,u^*TM)$ of smooth sections of the pullback bundle $u^* TM \to M$ with respect to the Sobolev $W^{k,p}$-norm; and $W_{\Sigma}^{k,p}(u^* TM \otimes_{J^b} \Omega^{0,1} \Sigma)$ is the completion of the space $\Omega^{0,1}(\Sigma,u^*TM)$ of $J^b$-antilinear forms with values in $u^* TM$. 

\begin{definition}
The connected component of $W^{k,p}(\Sigma,u^* TM)$ of maps of homology class $A$ is denoted
\begin{equation}
\mathfrak{B}^{k,p} = \left\{u \in W^{k,p}(\Sigma,M) \: \big| \: [u] = A \right\}.
\end{equation} 
\end{definition}

\begin{definition} Let
\begin{equation}
exp^{b_0}_{z}: T_z M \to M
\end{equation}
the exponential map with respect to the domain-dependent family of metrics $g^{b_0} = (g^{b_0}_z)_{z \in \Sigma}$. 
\end{definition}
Given a vertical curve $(b,u) \in B \times W^{k,p}(\Sigma,M)$ and a vector field $\zeta \in W_{\Sigma}^{k,p}(u^* TM)$, let 
\begin{equation}
(b,u_\zeta) \in  B \times W^{k,p}(\Sigma,M) 
\end{equation}
denote the curve
\begin{equation}
u_\zeta(z) := exp^b_{u(z)}(\zeta).
\end{equation}

\begin{lemma}
$W^{k,p}(\Sigma,M)$ is a smooth separable Banach manifold, and the tangent space at $u \in C^\infty(\Sigma,M)$ is 
\begin{equation}
T_u W^{k,p}(\Sigma,M) = W^{k,p}(\Sigma,u^* TM).
\end{equation}
Around any $u \in W^{k,p}(\Sigma,M)$, and any $b_0 \in B$, we can define a local coordinate charts by geodesic exponentiation, denoted 
\begin{equation}
\begin{split}
W_{\Sigma}^{k,p}(u^* TM) &\to W^{k,p}(\Sigma,M), \\
\zeta &\mapsto u_\zeta.
\end{split}
\end{equation} \noproof
\end{lemma}

\begin{definition}
Let
\begin{equation}
\tilde{\Pi}^{b}_{u \to u_\zeta}: u^* T M \to u_\zeta^* TM
\end{equation}
be complex bundle isomorphism obtained by parallel transport along the image of the geodesics
\begin{equation}
s \mapsto exp^b_{u(z)}(s \zeta(z)).
\end{equation}
using the $\Sigma$-family of hermitian connections $\tilde{\nabla}^{(b,\cdot)}$.
\end{definition}

We fix $(b,u) \in B \times \mathfrak{B}^{k,p}$. Given any $\zeta \in T_u W^{k,p}(\Sigma,M)$, we obtain an identification
\begin{equation} \label{eq:connectiondefinesbundle}
\begin{split}
W_\Sigma^{k,p}(u^* TM \otimes_\C \Omega^{0,1} \Sigma) &\to W_\Sigma^{k,p}(u_\zeta^* TM \otimes_\C \Omega^{0,1} \Sigma) \\
\alpha &\mapsto \tilde{\Pi}^{b}_{u \to u_\zeta}  \alpha.
\end{split}
\end{equation}

\begin{lemma}
The isomorphisms \eqref{eq:connectiondefinesbundle} define a Banach bundle
\begin{equation}
\mathfrak{E}^{k-1,p} \to B \times \mathfrak{B}^{k,p} 
\end{equation}
whose fibre at $(b,u)$ is
\begin{equation}
W_\Sigma^{k,p}(u^* TM \otimes_{J^b} \Omega^{0,1} \Sigma).
\end{equation} \noproof
\end{lemma}
This vector bundle admits a section 
\begin{equation}
\mathcal{S} : B \times \mathfrak{B}^{k,p} \to \mathfrak{E}^{k,p}, 
\end{equation}
given by
\begin{equation}
\mathcal{S}(b,u) = \delbar_{J^b}(u). 
\end{equation}
We denote 
\begin{equation}
\mathcal{S}^b := \mathcal{S}(b,\cdot) : \left\{b\right\} \times \mathfrak{B}^{k,p} \to \mathfrak{E}^{k,p}. 
\end{equation}
Note that for every $z \in \Sigma$, the image of $\mathcal{S}^b$ is the $(0,1)$-form with values in the pullback bundle given by the formula
\begin{equation}
\delbar_{J^b}(u)|_z = \frac{1}{2}(du(z) + J^b_z(u(z)) \circ du(z) \circ j_\Sigma(z).
\end{equation}
The zeros of $\mathcal{S}^b$ are $J^b$-holomorphic maps and thus the moduli space \eqref{eq:modulispace} is the intersection of the image of $\mathcal{S}$ with the zero section.

\begin{definition}
We define
\begin{equation}
\mathcal{F}_{(b,u)} : B \times W_\Sigma^{k,p}(u^* TM) \to W_\Sigma^{k,p}(u^* TM \otimes_{J} \Omega^{0,1} \Sigma)
\end{equation}
by setting 
\begin{equation}
\mathcal{F}^b_u = (\Pi^b_{u \to u_\zeta})^{-1} \delbar_{J^b}(u_\zeta). 
\end{equation}
This map is precisely the vertical part of the section $\mathcal{S}^b$ with respect to the trivialization determined by $\tilde{\nabla}$. We define the \textbf{linearization} to be the vertical differential
\begin{equation}
D_{(b,u)} := d\mathcal{F}_{(b,u)}(0) : T_b B \times W^{k,p}_\Sigma(u^* TM) \to W^{k,p}_\Sigma(u^* TM \otimes_{J} \Omega^{0,1} \Sigma ).
\end{equation}
\end{definition}

\begin{lemma}
For any smooth map $u : \Sigma \to M$, 
\begin{equation}
(D_{(b,u)} \zeta)(z) = \frac{1}{2}(\nabla \zeta + J_z(u) \nabla \zeta \circ j_\Sigma) - \frac{1}{2} J_z(u) (\nabla_\zeta J_z) \partial_J(u). 
\end{equation}
\end{lemma}
\begin{proof}
Consider the path $\R \to C^\infty(\Sigma,M)$ given by $s \mapsto exp_u(s \zeta)$. Then, by definition, 

\end{proof}

For the moduli space to be a manifold of the expected dimension, the image of $d\mathcal{S}$ at every $(b,u)$ is complementary to the tangent space $T_b B \oplus T_u \mathfrak{B}^{k,p}$ of the zero section. But for any $b \in B$ and $u$ which is $J^b$-holomorphic, $d\mathcal{S}$ is given by
\begin{equation}
d\mathcal{S}|_{(b,u)} : T_b B \oplus T_u \mathfrak{B}^{k,p} \to T_b B \oplus T_u  \mathfrak{B}^{k,p} \oplus \mathfrak{E}_{(b,u)}^{k,p}.
\end{equation}
If we consider now the projection onto the third component, the above transversality translates into the fact that
\begin{equation}
\pi_3 \circ d\mathcal{S}|_{(b,u)} : T_b B \oplus T_u \mathfrak{B}^{k,p} \to W_\Sigma^{k,p}(u^* TM \otimes_{J^b} \Omega^{0,1} \Sigma)
\end{equation}
is onto. It would be convenient to introduce the notation 
\begin{equation}
D \mathcal{S}|_{(b,u)} := \pi_3 \circ d\mathcal{S}|_{(b,u)}.
\end{equation}
We then have:
\begin{equation}
\end{equation}

\subsection{Gromov bordification} \label{subsec:gromovbordification}
Fix $d \geq 1$ and $\morseLabel = (p_d,\ldots,p_1,p_0) \subset \crit(f)$ a tuple of $d+1$ critical points. Assume we have made a consistent universal choice of perturbation data. 
In particular, there is now a choice of perturbation data on every singular pearl tree and we can define a pseduo-holomorphic pearl tree map $\textbf{u} = (\underline{b},u)$ based on $P$ in exactly the same way as we did for pearl trees in Definition \ref{def:pseudoholomorphicpearltreemap}. Note that every pseduo-holomorphic pearl tree map induces a homology decomposition 
\begin{equation} \label{eq:homologydecomposition1}
\underline{A} = \left\{A_v\right\}_{v \in \Vert^{int}(\Upsilon)} 
\end{equation}
where the pearl tree $P$ is of type $\Upsilon$; and $A_v = [u_v] \in H_2(M)$ the homology class of the restriction of $u$ to $C_v$. As usual, such a pearl tree map is stable if the underlying tree is stable and the restriction of $u$ to the sphere component is a stable map. The total homology of $u$ is the sum $\sum_v [u_v]$. 

\begin{definition}
As a set, we define the \textbf{bordification of the moduli space of pseduo-holomorphic pearl trees maps} with asymptotics $\morseLabel$ and homology $A$ to be
\begin{equation}
\begin{split}
\bar{\PP}_d(\morseLabel,A) = \left\{\right. [(P,\underline{b},u)]  \:  \: \big| \:  \: &(\underline{b},u) \text{ is based on }P, \text{ stable, has}  \\
&\text{homology class }A\text{ and asymptotics }\morseLabel  \: \left.\right\}.
\end{split}
\end{equation}
\end{definition}

We give it the obvious topology induced from the manifold with boundary and corners structure on Morse trajectories and the $C^\infty$-Gromov topology of pseduo-holomorphic maps. One can check that it is locally metrizable, 2nd countable, and Hausdorff (see \cite{MR3517817} for a detailed discussion of this in the setting of Lagrangian Floer homology.) Gromov compactness for nodal curves (the relevant version is the one for domain-dependent almost complex structures, see e.g. \cite[Theorem 5.2]{MR2399678}) combined with compactness of moduli spaces of gradient trees readily implies that $\bar{\PP}_d(\morseLabel,A)$ is compact as well. \\

\textbf{Notation.} The moduli space $\bar{\PP}_d(\morseLabel)$ has a stratification by singular pearl type $\Upsilon$, and we denote $\PP_\Upsilon(\morseLabel,A)$ for the subset of isomorphism classes of pseduo-holomorphic pearl tree map whose domain has combinatorial type $\Upsilon$. Similarly, given a homology decomposition as in equation \eqref{eq:homologydecomposition1}, we denote $\PP_\Upsilon(\morseLabel,\underline{A})$ for all isomorphism classes of pearl tree maps with type $\Upsilon$ and $[u] = \underline{A}$.

\subsection{Transversality} \label{subsec:transversality}
Fix $d_0 \geq 1$, a universal choice of perturbation datam $\morseLabel$ and $\underline{A}$ as above. Let $\Upsilon = (T,\underline{C})$ denote a stable pearl type with $d_0$ leaves of expected codimension $\virdim(\Upsilon) \leq 1$. Suppose our choice of coherent choice of perturbation data is already \textbf{regular} (which means it satisfies ($V$),($E$),($EV$),($T_+$) and $(B)$ parametric regularity, see Section \ref{sec:parametrized}) for all stable pearl trees with $\Upsilon' < \Upsilon$. We want to put a smooth structure on moduli spaces of pearl tree maps. To do so, we consider the pearl $\mathcal{C} \to \Pearltree$ and the tree parts $\mathcal{T} \to \Pearltree$ separately and describe our moduli space as the fiber product of moduli spaces of gradient trees and vertical maps. \\

Formally, an infinitesimal deformation of $Y^\Upsilon$ is given by a pair
\begin{equation}
(\delta Y)^\Upsilon = ((\delta X)^\Upsilon,(\delta J)^\Upsilon)
\end{equation}

where $\delta X$ is a vector field, and $\delta J^\Upsilon$ is a smooth choice of tangent vectors to the moduli space of almost complex structures given by equation \ref{eq:variationofalmostcomplexstructure}. Of course, $\delta X$ is required to be compactly supported and $\delta J^\Upsilon$ is required to vanish on the cylindrical ends. Such an infinitesimal deformation can be then ''exponentiated"
to an actual one,

\begin{equation} \label{eq:deformedalmsotcomplexstructre}
(\hat{Y})^\Upsilon := (X^\Upsilon + (\delta X)^\Upsilon,J^\Upsilon exp(-J^\Upsilon \cdot (\delta J)^\Upsilon)
\end{equation}

To retain consistency with the previous choices of perturbation data, we restrict the space of allowable deformations by introducing the following definition, following \cite[Section (9k)]{MR2441780} and \cite{eprint2}. 

\begin{definition}
A choice of compact part is an open subset $\PP_d^{cpt}$ of the universal pearl tree $\PP_d$ such that the intersection with every fiber is nonempty and disjoint from the cylindrical ends, and above every thick-thin chart as in Definition \ref{def:thickthinchartcoordinate}, the restriction to each fiber lies in the thick part of the decomposition. 
\end{definition}

Fixing such a choice, we require that our deformations vanish outside the compact part, and which are decay asymptotically with respect to the
partial compactification $\overline{\PP}_d$. The last-mentioned condition means that every $(\delta X ,\delta J)$ would extends smoothly to a pair defined on $\PP_d$, which vanishes (up to infinite order) on the boundary. Note that every such a choice, the new perturbation datum \eqref{eq:deformedalmsotcomplexstructre} maintains consistency with the previous choices; and its asymptotic behavior near the boundary is the same as that of the original $Y^\Upsilon$. \\

Recall that the combinatorial type of $\mathcal{C} \to \Pearltree$ is a forest,  
\begin{equation}
\Upsilon_{Sph}  = \bigcup_{v \in \Vert^{int}(T)} \Upsilon_{v}
\end{equation}
where each tree $\Upsilon_{v}$ is the combinatorial type of a $|v|$-marked pearl $\mathcal{C}_v$. 

\begin{definition}
Denote
\begin{equation} \label{eq:moduliofpearlmaps}
\MM^{pearl}_{\mathcal{C} \to \mathcal{Q}}(\underline{A}) = \bigtimes_{v \in \Vert^{int}(\Upsilon_{v})} \MM^{pearl}_{\mathcal{C}_v \to \Pearltree}(A_v).
\end{equation}

where $\MM^{pearl}_{\mathcal{C}_v \to \Pearltree}(A_v)$ is the moduli space of pairs $(q,\underline{u})$ where $q \in \Pearltree$, and $\underline{u} = (b,u)$ is a vertical pearl map $u : S^2 \to E_b$ that is pseudo-holomorphic with respect to the restriction of the almost complex structure $Y^\Upsilon_q$ to $\mathcal{C}_v$.
\end{definition}

\begin{remark}
In particular, it is certainly possible for a tuple 
\begin{equation}
(\underline{b},\underline{\textbf{u}}) \in \left(\bigtimes_{v \in \Vert^{int}(\Upsilon_{v})} \MM^{pearl}_{\mathcal{C}_v \to \Pearltree}(A_v)\right) 
\end{equation}
to have $b_i = b_j$ while $i \neq j$. We will remark on this again when we discuss transversality. We also would like to emphasize that the family of pearls-with-bubbles is still parametrized by the entire moduli space $\Pearltree$, i.e., the dependence on the various edge lengths of $\mathcal{T} \to \Pearltree$ is maintained. 
\end{remark}

As a first step, we want to modify the given perturbation datum slightly, in such a way that the almost complex structures induced by $\textbf{J}^\Upsilon$ on every family of holomorphic curves $\mathcal{C}_v \to Q_\Upsilon$ satisfies (H),(V) and (EV). 
\begin{remark}
Of course, we can only hope to achieve that because our initial choice of Floer data was good.
\end{remark}

\begin{theorem}
For a generic choice of perturbation $(\delta Y)^\Upsilon$, each one of these moduli spaces is smooth, and of the expected dimension. 
\end{theorem}
\begin{proof}
This is done in the same way as \cite[p.~128--129]{MR2441780}: we form a universal moduli space, and consider the tangent space to $J^\Upsilon$. The same analysis that we used in Section \ref{sec:parametrized} to show that $\JJ^{(V)+(E)+(EV)} \subseteq \JJ$ is comeager (see the proof of Theorem \ref{thm:pseudocycle}) applies here word for word, the only difference is that everything depends on an additional parameter $q \in \QQ_\Upsilon$ (the family parameter.) 
\end{proof}
Each moduli space of vertical $|v|$-marked spheres also has an evaluation map, denoted
\begin{equation}
ev_v : \MM^{pearl}_{\mathcal{C}_v \to \Pearltree}(A_v) \to E \times_\pi \ldots \times_\pi E \hookrightarrow E^{\otimes |v|}
\end{equation} 

Define
\begin{equation}
\begin{split}
ev^+_C &:= \bigtimes_{A_v \neq 0} ev_v ,\\
ev^0_C &:= \bigtimes_{A_v = 0} ev_v, \\
ev_C &:=  ev^+_C \times ev^0_C
\end{split}
\end{equation}



\begin{theorem}
For a generic choice of perturbation data $Y^\Upsilon$, the evaluation maps $ev^+_C$, $ev^0_C$, and $ev_C$ are pseudo-cycles. 
\end{theorem}
\begin{proof}
Same. 
\end{proof}

However, this alone does not guarantee that sphere evaluation map will be a pseduochain! Here we encounter a multiple cover-type problem mentioned in the introduction to Section \ref{sec:parametrized}: for example, what happens when the edge $e$ connecting two internal vertices $v$ and $w$ is in $\Edge^{int}_{0}(\Upsilon)$? In that case the two families $\mathcal{C}_v$ and $\mathcal{C}_w$, associated to two different connected components of the sphere part of $\Upsilon$, are required to map to the same fiber. The compactification now includes many bubble tree configurations whose components that may intersect or cover each other, and we must take that into consideration. A typical strata to keep in mind for the rest of the subsection appears in Figure \ref{fig:problemwithcollision} below. 

\begin{figure}[htb]
\begin{minipage}[b]{.45\linewidth}
\centering
			\fontsize{0.25cm}{1em}
			\def\svgwidth{6cm}
 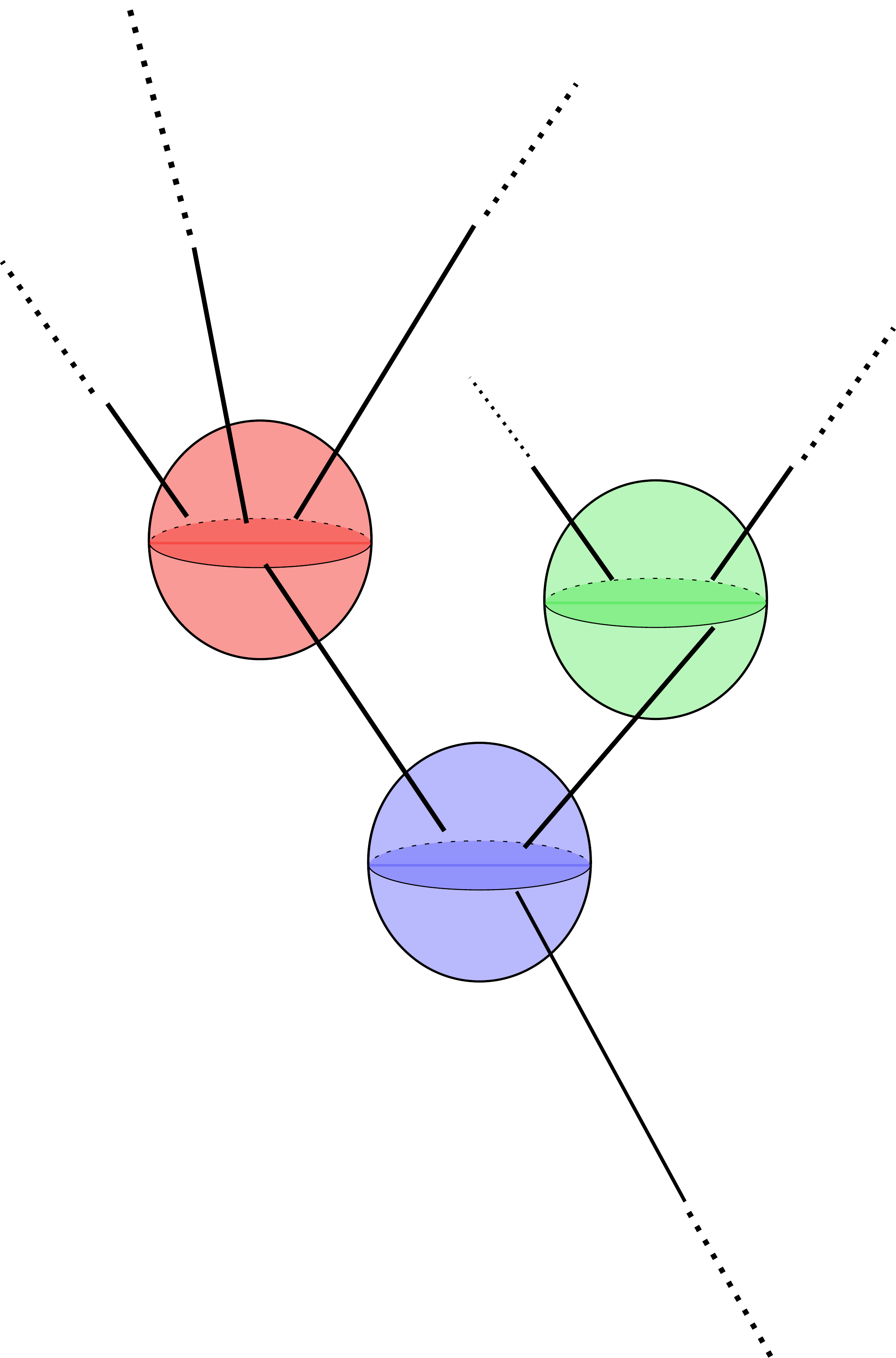
\subcaption{There might be a null cluster in $\Upsilon$ \vspace{2.4em}}
\end{minipage}%
\vspace{0.1em}
\begin{minipage}[b]{.45\linewidth}
\centering
			\fontsize{0.25cm}{1em}
			\def\svgwidth{6cm}
 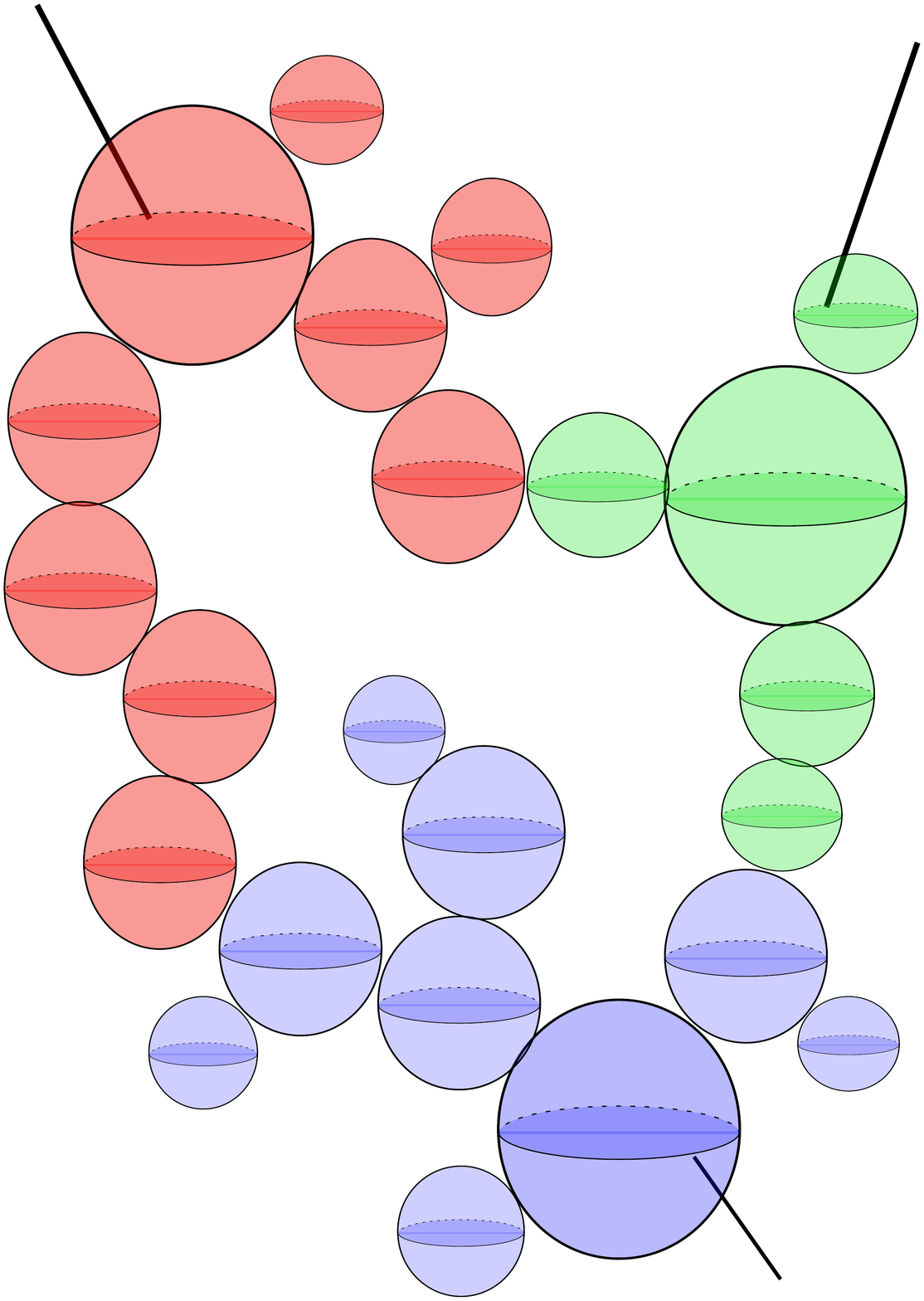
\subcaption{The compactification can include configurations where different bubbles collide, or map to the same image}
\end{minipage}
\caption{The pearl tree type $\Upsilon$ can have edges with combinatorial length zero}
\label{fig:problemwithcollision}
\end{figure} 

Thus, for the purpose of establishing transversality, it will be useful to consider an alternative perspective. 

\begin{definition}
The set of \textbf{null clusters} of $\Upsilon$ is defined as the quotient 
\begin{equation}
Null(\Upsilon) := \Vert^{int}(T)/\sim_{null} 
\end{equation}
where $v \sim_{null} w$ if they are connected by an edge $e \in \Edge_0^{int}(T)$. Note that every null cluster represents a broken pearl. We change our notations a bit and consider the sphere part $\mathcal{C}$ as a disconnected nodal surface where each connected component, $\mathcal{C}_{[v]}$, corresponds to a broken pearl. 
\end{definition}

We define a new combinatorial type (still a forest)

\begin{equation}
\Upsilon_{NSph}  = \bigcup_{[v] \in Null(\Upsilon)} \Upsilon_{[v]}
\end{equation}

where each tree $\Upsilon_{[v]}$ is now the combinatorial type of a broken pearl $\mathcal{C}_{[v]}$.

\begin{definition}
Denote $\MM^{pearl}_{\mathcal{C}_{[v]} \to \Pearltree}(A_v)$ for the moduli space of pairs $(q,\underline{u})$ where $q \in \Pearltree$, and $\underline{u} = (b,u)$ is a vertical nodal pearl map with a nodal domain $\mathcal{C}_{[v]}$ of a fixed combinatorial type $\Upsilon_{[v]}$, and pseudo-holomorphic with respect to the restriction of the almost complex structure $Y^\Upsilon_q$. Let 
\begin{equation}
\Vert_{(\Upsilon,\underline{A})}  := \bigtimes_{[v] \in Null(\Upsilon)} \MM^{pearl}_{\mathcal{C}_{[v]} \to \Pearltree}(A_v)
\end{equation}
be the product of all such moduli spaces.
\end{definition}

Each such moduli space of vertical marked broken spheres also has an evaluation map, denoted
\begin{equation}
ev_{[v]} : \MM^{pearl}_{\mathcal{C}_{[v] \to \Pearltree}}(A_[v]) \to E \times_\pi \ldots \times_\pi E \hookrightarrow E^{\otimes \ldots}
\end{equation} 
where we just ignore all the pairs that marked points $(v^-,v^+)$ that correspond to edges $e \in \Edge_0^int(\Upsilon)$. The product of all these evaluation maps is denoted $ev^N_C : \Vert_{(\Upsilon,\underline{A})} \to E^{|\Flag_{\neq 0}(T)|}$, where $f = (v,e) \in \Flag_{\neq 0}(T)$ if and only if $e \notin \Edge_0^{int}(T)$. Similarly, 
\begin{definition}
We define the moduli space of all gradient trees based on $\mathcal{T} \to \Pearltree$, denoted 
\begin{equation} \label{eq:moduliofpearlmaps}
\Edge_{(\Upsilon,\morseLabel)}
\end{equation}
to be the product over all edges of $T$ of the moduli spaces of gradient trajectories of non-zero length:
\begin{equation}
\Edge_{(\Upsilon,\morseLabel)} = \bigtimes_{e \notin \Edge_{0}(\Upsilon)} \MM_e(...,...)
\end{equation}
where 
\begin{equation}
\MM_e(...,...) = 
	\begin{cases}
		\bar{\MM}^{morse}(p_{v_-},E), & \text{if } v_- \in \Vert^{ext}(T), \\ \vspace{0.1em}
    \bar{\MM}^{morse}(E,p_{v_+}), &\text{if } v_+ \in \Vert^{ext}(T), \\ \vspace{0.1em}
    \bar{\MM}^{morse}(E,E), &\text{if both vertices are internal }.
	\end{cases} \vspace{0.5em}
\end{equation}
\end{definition}
There is an evaluation map, denoted
\begin{equation}
ev_e : \MM_e(...,...) \to E \times E
\end{equation}
where 
\begin{equation}
ev_e = 
	\begin{cases}
		\left\{p_{v_-}\right\} \times ev, & \text{if } v_- \in \Vert^{ext}(T), \\ \vspace{0.1em}
    ev \times \left\{p_{v_+}\right\}, &\text{if } v_+ \in \Vert^{ext}(T), \\ \vspace{0.1em}
    ev_- \times ev_+, &\text{if both vertices are internal }.
	\end{cases} \vspace{0.5em}
\end{equation}

\begin{definition} 
We define the Flag space to be the product
\begin{equation}
\Flag_\Pearltree = \bigtimes_{f \in \Flag_{\neq 0}(T)} E_f
\end{equation}
over all flags $f = (v,e)$ of copies of the total space of $\pi$, i.e., $E_f = E$; For every pearl type $\Upsilon$, asymptotics $\morseLabel$ and homology decomposition $\underline{A}$ as above, the space $\PP_\Upsilon(\morseLabel,A)$ is given by the following fiber product 
\begin{equation}
\begin{tikzpicture}
    \matrix (m) [
            matrix of math nodes,row sep=3em, column sep=5.0em,
						text height=1.5ex, text depth=0.25ex
            ]   {
      & \Vert_{(\Upsilon,\underline{A})} \\
     \Edge_{(\Upsilon,\morseLabel)} & \Flag_{(\Upsilon,\morseLabel,\underline{A})} \\};
    \path[->]
    (m-2-1) edge node[auto] {$ev^N_C$} (m-2-2)
		(m-1-2) edge node[auto] {$ev^+_T$} (m-2-2);
\end{tikzpicture}  \label{eq:flag_correspondence}
\end{equation}
\end{definition}

We will refer to diagram \eqref{eq:flag_correspondence} as the Flag correspondence determined by $(\Upsilon,\morseLabel,\underline{A})$. 

\begin{definition}
We denote the continuous extension to the bordification as
\begin{equation}
\overline{ev}^+_T : \overline{\Edge}_{(\Upsilon,\morseLabel)} \to \Flag_{(\Upsilon,\morseLabel,\underline{A})} \: , \: \overline{ev}^N_C : \overline{\Vert}_{(\Upsilon,\underline{A})} \to \Flag_{(\Upsilon,\morseLabel,\underline{A})}.
\end{equation}
\end{definition}
\begin{theorem} \label{thm:evaluationN}
For a generic choice of deformation $(\delta Y)^\Upsilon$, the evaluation maps $ev^N_C$ are pseudo-cycles. 
\end{theorem}
\begin{proof}
Follows from a parametrized version of the proof that $\JJ^{(V)+(E)+(EV)+(B)} \subseteq \JJ^{(V)+(E)+(EV)}$ is comeager in the same way as before. 
\end{proof}

\begin{remark}
Here, the fact that our perturbation data depends on the tuple of base projections of all sphere components $\underline{b}$ simultaneously is crucial. 
\end{remark}

We can further reduce the number of our marked points by introducing the \textbf{reduced evaluation pseudocycle} $ev^R_C$ (which means we only evaluate at points that belong to the reduced index set and modify the target accordingly, see Definition \ref{def:reducedindexset}). 

\begin{theorem} \label{thm:reducedevaluationsubmersion}
For a generic choice of deformation $(\delta Y)^\Upsilon$, $ev^R_C$ is a pseudocycle and a submersion onto its image. 
\end{theorem}
\begin{proof}
See Section \ref{sec:parametrized} and \cite[Proposition 5.7]{MR2399678}. 
\end{proof}

\begin{remark}
We need to exercise a bit of caution: strictly speaking, Proposition 5.7 as stated is incorrect, because no such universal moduli space exists -- because of the usual problems with Banach manifolds and varying marked points, see \cite[Section 3]{trivialisotropy} for an overview. However, a minor modification of the same argument, using the same inductive framework as we use in Section \ref{sec:parametrized} holds. See \cite[Theorem 6.3.1]{MR2954391} for an argument of the same flavour. 
\end{remark}

Note that this implies that we can safely neglect any unstable pearl type $\Upsilon' \leq \Upsilon$ because we are only interested in strata of codimension $\leq 1$.

\begin{lemma} \label{lem:treeswithpositivelength}
For a generic choice of perturbation data $Y^\Upsilon$, the evaluation maps $ev^+_{T}$ is a submersion onto its image. 
\end{lemma}
\begin{proof}
Follows from Lemma \ref{lem:consistentchoices}. 
\end{proof}

\begin{remark} Note that $ev^N_C$ can be obtained from $ev^R_C$ by composing with suitable diagonal embedding of the target (i.e., for every set $\Lambda_{ghost}$ of marked points on a ghost component, we compose with $id \times \ldots \times E \times \ldots id \hookrightarrow id \times \ldots (E \times_\pi \ldots \times_\pi E) \times \ldots id$ where in the right hand side we added $|\Lambda_{ghost}|-1$ copies of $E$ and rearrange accordingly.) 
\end{remark}

\subsection{Gluing}
\begin{proposition} \label{prop:gluing}
Suppose that $\Upsilon$ is a combinatorial stable pearl tree type of expected dimension less then one. If $\Upsilon'$ was obtained from $\Upsilon$ by type (I),(II) or (III) degeneration. Then the stratum $\PP_{\Upsilon'}(\morseLabel)$ has a tubular neighborhood in $\PP_\Upsilon(\morseLabel)$.
\end{proposition}
\begin{proof}
The type (II) and (III) cases follows from the existence of the gluing maps used to construct the structure of a manifold with boundary and corners in Morse theory (Theorem \ref{thm:corner}.) It remains to deal with type (I): in this setting, there is a pair of pearls, say $C_v$ and $C_w$, connected by a combinatorial length zero edge $e \in \Edge(\Upsilon')$. We would like to reduce the problem to the same situation as in \cite[Chapter 10]{MR2954391}. \\

Consider the parametrized version 
\begin{equation} \label{eq:parametrizedevaluationofapair}
ev_e : \MM^{pearl}_{\mathcal{C}_v \to \Pearltree}(A_v)  \times \MM^{pearl}_{\mathcal{C}_w \to \Pearltree}(A_w) \to E \times_\pi E
\end{equation}
of the pseudocycle of a pair from equation \ref{eq:pseudocycleofapair}. By assumption, it is transverse to the (parametrized) diagonal in $E \times_\pi E$. Since the problem is local in nature, we can assume without loss of generality that $E \iso M \times D^{m}$ as in Lemma \ref{lem:trivialization} (where $m = \dim(B)$.) This means we can view $\MM^{pearl}_{\mathcal{C}_v \to \Pearltree}$ as moduli spaces that consist of a pair $(b,u)$ where $b \in B$ and $u$ is a holomorphic map to a fixed target space $(M,\omega)$. Similarly, the parametrized diagonal is $B \times \Delta$ where $\Delta \subset M \times M$ is the usual diagonal, and \eqref{eq:parametrizedevaluationofapair} can be rewritten as  
\begin{equation} 
ev_e : \MM^{pearl}_{\mathcal{C}_v \to \Pearltree}(A_v)  \times \MM^{pearl}_{\mathcal{C}_w \to \Pearltree}(A_w) \to B \times (M \times M).
\end{equation}
This is transverse by assumption of (B)-regularity. Now we just apply a parametrized version the usual Machinery of Chapter 10 in \cite{MR2954391}. 
\end{proof}

\subsection{Compactness}
From now on, we fix a universal, consistent and regular choice of perturbation data.

\begin{proposition} \label{prop:compactness}
For any stable pearl type $\Upsilon$ of expected dimension at most one, the moduli space $\overline{\MM}^{pearl-tree}_{\Upsilon}(A,\morseLabel)$ is compact and the closure of $\MM^{pearl-tree}_{\Upsilon}(A,\morseLabel)$ contains only configurations with domain degenerations of Type (I), (II) and (III).
\end{proposition} 

\begin{proof}
Let $\Upsilon = (T,\underline{C})$. Because of the existence of local distance functions, it suffices to check sequential compactness. Let $\textbf{u}^\nu : |P^\nu| \to E$ be a sequence of stable pearl tree maps disks of type $\Upsilon$ with asymptotics $\morseLabel$ and homology $A$. The result follows immediately once we make three observations: Denote $\textbf{v}^\nu := (\underline{b}^\nu,\underline{u}^\nu)$ for the sphere part. Let's restrict our attention to a specific connected component, say $(b_0^\nu,v_0^\nu)$. Then when considered as a vertical map, Gromov compactness imply that $b_0^\nu \to b_0^\infty$ and $v_0^\nu \to v_0^\infty$. Thus there is a convergence of the domains, meaning $[C_0^\nu] \to [C^\infty]$ as equivalence class in DM-space. So $v_0^\infty : \hat{C}^\infty \to E_{b_0^\infty}$ is a stable vertical pearl map. In fact, it is \emph{domain stable}: because of the way we constructed our perturbation datum, $\hat{C}^\infty$ contains no unstable components -- bubbling happens in codimension two, and. Thus there exists a unique limit $\textbf{v}^\infty$ of $\textbf{v}^\nu$ which is always pearl-map. The second is that the underlying trees $T^\nu$ must converge to some $T^\infty$. The third is that if we restrict our attention to a specific component in the tree part, say $\gamma^\nu_0$, the sequence must have a limit in the Morse trajectory space, which is a (broken) Morse trajectory. 
\end{proof}
Finally, 
\begin{proof}[Proof of Theorem \ref{thm:mainfacts}]
Taken together, Theorem \ref{thm:reducedevaluationsubmersion} and Lemma \ref{lem:treeswithpositivelength} prove the first part of Theorem \ref{thm:mainfacts}. The second part follows from Proposition \ref{prop:compactness} and \ref{prop:gluing}.
\end{proof}

\subsection{Orientations}

\begin{lemma} \label{lem:orient2}
There exists orientations on $\PP_\Upsilon(A)$ which are compatible with the morphisms of Section \ref{subsec:cubicaldecomposition} in the following sense: \vspace{0.5em}
\begin{itemize}
\item
If $\Upsilon$ is obtained from $\Upsilon'$ by grafting, then the isomorphism 
\begin{equation}
\PP_\Upsilon(A) \to \PP_{\Upsilon'}(A) 
\end{equation}
is orientation preserving. \vspace{0.5em}
\item
If $\Upsilon$ is obtained from $\Upsilon'$ by a Type (I)-(III) degeneration, the inclusion 
\begin{equation}
\PP_\Upsilon(A) \to \PP_{\Upsilon'}(A) 
\end{equation}
coming from Proposition \ref{prop:gluing} induces an orientation using the outward unit normal convention which is given by a universal sign depending only on $\Upsilon$ and $\Upsilon'$.  \vspace{0.5em}
\end{itemize}
\end{lemma} 
\begin{proof}
This is a special case of the existence of systems of orientations for Lagrangian Floer theory in Fukaya-Oh-Ohta-Ono \cite{MR2553465}, made simpler by the fact that spaces of holomorphic curves with closed domains are canonically oriented by the almost complex structure -- without the need for an end datum. See also \cite{eprint2}. 
\end{proof}

In particular, it follows from Lemma \ref{lem:orient2} that contributions from opposite boundary points of the one-dimensional connected components cancel.

\begin{remark}
in fact, note that \eqref{eq:iso_det_bunldes_pearl3} is actually independent of $A$. 
\end{remark}

\begin{lemma} \label{lem:orient1}
Let $[\textbf{u}]$ be an equivalence class of stable pearl tree maps of combinatorial type $\Upsilon$. Then there is a canonical identification
\begin{equation} \label{eq:iso_det_bunldes_pearl3}
\lambda(\PP_{\Upsilon}(\morseLabel;A)) \iso \lambda(\PP_\Upsilon) \otimes \oo_{p_0} \otimes \oo^\vee_{\morseLabel^{in}}
\end{equation}
where $\morseLabel^{in} = (p_d,\ldots,p_1)$, $\oo^\vee$ is the inverse orientation, and $\oo_{(v^s,\ldots,v^1)} = \bigotimes_{i=1}^s \oo_{v^i}$.
\end{lemma}
\begin{proof}
This is standard: since we have already eliminated all bubbles and established regularity, we can restrict to a miniversal deformation space and work with a fixed domain $P$ and work locally. Then we have an induced map on the sphere part 
\begin{equation}
\PP_\Upsilon^{univ} \to \JJ(\underline{C}) \: , \: m \mapsto j(m)
\end{equation}
As in \cite{eprint2} and \cite{MR2038115}, we fix $k \geq 1, p > 2$ and consider the standard functional analysis framework: we denote $Map^{k,p}(P,E)$ for the space of all $W^{k,p}$-smooth maps $P \to E$ which map each sphere component completely to a single fiber, Let $l>> k$ be an integer and consider  
\begin{equation}
\BB^{k,p,l} := \PP_\Upsilon \times Map^{k,p}(P,E) \times \mathfrak{Y}_\upsilon^l(E)
\end{equation}
where $\mathfrak{Y}_\upsilon^l(E)$ is the space of all possible choices of deformations for the perturbation data of $\Upsilon$ (given our background universal choice.) There is a Banach bundle $\EE$ over it given by the direct sum of $W^{k-1,p}$-smooth sections of $(0,1)$-forms on the sphere and tree part, respectively. We consider a section $\BB \to \EE$ 
\begin{equation}
\delbar_\Upsilon^{univ}  = (\delbar_\Upsilon,(\frac{d}{ds} - \nabla))
\end{equation}
given by combining the Cauchy-Riemann and the shifted gradient operators and define the \textbf{local universal moduli space} as the inverse image $\delbar^{univ}(0) \subset \BB^{k,p,l}$. Note that we have bypassed the ''loss of derivatives" issue because we apriori know that the transition maps between different slices are $C^\infty$-smooth. The linearization $D \delbar_\Upsilon^{univ} \big|_{\textbf{Y}^\Upsilon}$ of the section restricted to the perturbation $\textbf{Y}^\Upsilon$ is an isomorphism by assumption; thus determines the isomorphism \eqref{eq:iso_det_bunldes_pearl3}. See a similar discussions in e.g., \cite[Proposition 11.13]{MR2441780}. 
\end{proof}

However, for reasons that would become apparent, we prefer to twist the orientation on the d-th moduli space of pearls. If we denote $\EE_{\morseLabel}$ for the orientation coming from Lemma \ref{lem:orient1}, we actually define the orientation to be given by the formula
\begin{equation} \label{eq:iso_det_bunldes_pearl4}
\lambda(\PP_\Upsilon(\morseLabel;A)) \iso \lambda(E^d) \otimes \EE_{\morseLabel}.
\end{equation}

This essentially corresponds to the fiber product description of \eqref{eq:flag_correspondence}. As before, 

\begin{definition} \label{def:inducedmaponorientationlines}
Given a rigid holomorphic pearl tree $u$, the isomorphism \eqref{eq:iso_det_bunldes_pearl4} and Equation \eqref{eq:decom_tangent_space_crit_points} give a natural map
\begin{equation}  \label{eq:iso_det_bunldes_stasheff}
\mu_u : |\oo_{p_d}| \otimes \cdots \otimes |\oo_{p_1}| \mapsto |\oo_{p_0}|. 
\end{equation}
\end{definition}

\subsection{The $A_\infty$-axiom} \label{subsec:ainftypearl}
The purpose of this section is to prove Proposition \ref{prop:ainftypearl}. As usual, the proof is by examining the combinatorics of the ends of the moduli space $\PP_{d}(A,\morseLabel)$ when $\virdim(A, \morseLabel) = 1$. Since it is clear that strata with broken pearls are internal (and cancel in pairs of Type (I)/Type (II) degenerations), there are only finitely many of them, and we have ruled out bubbling, the only thing we need to verify are the signs of Type (III)/grafting boundary terms. \\

We wish to find the sign difference between the sign difference between the product orientation (given by \eqref{eq:iso_det_bunldes_pearl4}) and the boundary orientation. The only possible sources of difference are Koszul signs coming from commuting various terms in \eqref{eq:iso_det_bunldes_pearl4} past each other, and the difference between the coherent and product orientation on the abstract moduli spaces. Now note that the signs in Lemma \ref{lem:boundary_orientation_stasheff} and \ref{lem:boundary_orientation_pearl} \emph{are identical}, and the terms in \eqref{eq:iso_det_bunldes_pearl4} are a minor modification of those in used in the proof of Proposition 4.22 in \cite{MR2529936} (which is why we choose our convention to be \eqref{eq:iso_det_bunldes_pearl4} and not the more natural equation \eqref{eq:iso_det_bunldes_pearl3}). The same argument shows that they differ by the global sign and establishes the $A_\infty$-equation. 

\subsection{Independence from choices} \label{sec:independence}
We have made many choices in the construction of the pearl complex and the associated system of perturbations. However, they are all inconsequential (up to an $A_\infty$-quasi-isomorphism). 


This can be shown as follows: consider a symplectic fibration

\begin{equation} \label{eq:01model}
(E_{[0,1]},B_{[0,1]},\pi_{[0,1]}) := ([0,1]  \times E , [0,1] \times B  , id \times \pi)
\end{equation}

with fiber $(M,\omega)$. Assume we have a given $\pi_{[0,1]}$ a closed 2-form $\Omega \in Z^2(E_{[0,1]})$ which makes it into a monotone LHF. The restriction to $\left\{0\right\} \times E$ and $\left\{1\right\} \times E$ are LHF's and $\pi_{[0,1]}$ is an isotopy between them in the sense of Definition \ref{def:isotopyofLHFs}. We denote them 
\begin{equation}
(E_{0},\pi_{0},\Omega_0) \: , \: (E_{1},\pi_{1},\Omega_1)
\end{equation}
respectively. Given any two choices of base and perturbation data $(f_0,g_0,J_0,\textbf{Y}_0)$ and $(f_1,g_1,J_1,\textbf{Y}_1)$, there are $A_\infty$-algebra structures on the parametrized pearl complexes of $\pi_0$ and $\pi_1$. We define an $A_\infty$-quasi-isomorphism in the following way: \vspace{0.5em}

\begin{itemize}
\item
We male a choice of a new base datum for $\pi_{[0,1]}$: the Morse-Smale pair on $\pi_{[0,1]}$ is defined in the same way as we have done in the Morse $A_\infty$-case (See the discussion preceding Corollary \ref{cor:morseindependence}.) The restriction of the almost complex structure
\begin{equation}
J_{[0,1]} \in \complexJ(\pi_{[0,1]},\Omega_{[0,1]}) 
\end{equation}
is required to coincide with $J_i$ on $E_i$ for $i=1,2$. \vspace{0.5em}
\item
We make a universal choice of regular perturbation data for pearls $\textbf{Y}_{[0,1]}$ that interpolates between our two previous universal choices $\pi_0$ and $\pi_1$ (this is possible because all choices involved are contractible, see \cite[p.~144--145]{MR2441780}, and because regularity is generic.)  \vspace{0.5em}
\item
The inclusion of $\pi_0$ and $\pi_1$ into $\pi_{[0,1]}$ is an $A_\infty$-homomorphism. In fact, it is an $A_\infty$-quasi-isomorphism because the differential on each half of the base is just the continuation map from Morse theory. Thus, $A_\infty$-algebra of $\pi_{[0,1]}$ is a $[0,1] \times C^\bullet$ model as in \cite[Section 4.2]{MR2553465} and therefore defines an $A_\infty$-quasi-isomorphism. 
\end{itemize}

Note that this also proves Proposition \ref{prop:existencepearlMphi}(4). 

\begin{remark}
We had to keep the choice of cylindrical ends fixed for technical reasons, but of course, this is not a real restriction. One can easily show independence of that choice using one of the other methods for prove well-definedness in Floer theory. For example, we can apply a formal ''double category" trick as in \cite[Section 10(a)]{MR2441780} (relying on the independence of parametrized quantum cohomology from the choice of almost complex structure as well as the proof of uniqueness for the Conley continuation map.) 
\end{remark}

\section{The pearl complex (coherence)} \label{sec:masseyproductquantum}
This section and the next (\ref{sec:parametrized} and \ref{sec:masseyproductquantum}) discusses those parts of the theory of $J$-holomorphic curve theory which we will need to construct Quantum Massey products. \\

Our particular setup is essentially a chimera, which we form by merging of two well-known constructions in the literature: the ''\emph{Morse-Bott}" or ''\emph{pearly}" approach to Lagrangian Floer theory and the parametric pseudo-cycle approach to family Gromov-Witten invariants. As a result, our moduli spaces posses certain idiosyncrasies which we would like to point out: \vspace{0.5em}

\begin{itemize}
\item
The first type (which are of a rather technical nature), are changes that come from adapting the pearl complex to the fibration setting. Most of them are obvious, but we included a discussion indicating how the Fredholm setup and Gluing analysis are effected. \vspace{0.5em}
\item
The second type of changes is that we eschew generality and the use of virtual techniques like Kuranishi structures or polyfolds in favor of the more classical geometric regularization method. As a result, we obtain actual \emph{pseudo-chains} and not just homology classes. This would be important in the computation (which take place in Section \ref{sec:compute4}.) \vspace{0.5em}
\item
The third type (which is more serious) is the a \emph{multiple cover problem}: when the lengths of edges in a pearl tree go to zero, we might encounter a situation in which two pearls, or bubbles thereof, ''collide". See illustration in Figure \ref{fig:problemwithcollision} below. The usual remedy on the monotone case to multiple covers is to factor them via a simple map, unfortunately, this does little good here because the incidence conditions on each sphere are entirely different. We will need to take some precautions ''beforehand" (which means when the length parameters are suitably small) to avoid getting into this situation to begin with. This would be made precise when we discuss transversality (see the Definition of (B)-regularity in Section \ref{sec:parametrized}.) 
\end{itemize}

\begin{figure}[htb]
\begin{minipage}[b]{.48\linewidth}
\centering
			\fontsize{0.25cm}{1em}
			\def\svgwidth{8cm}
 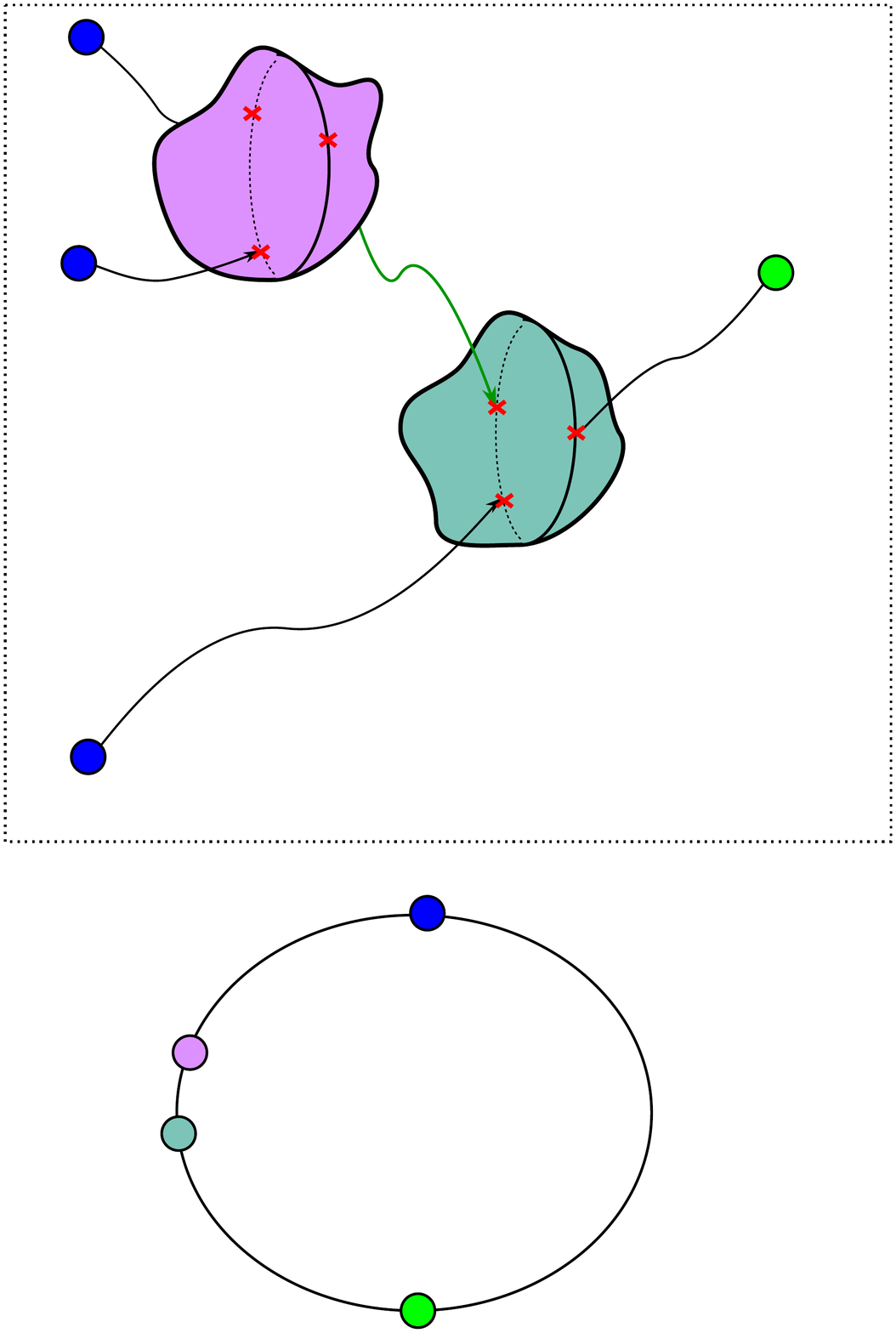
\subcaption{When the length of edges is bounded, \\
we define our moduli space by a fiber product}
\end{minipage}%
\vspace{0.1em}
\begin{minipage}[b]{.48\linewidth}
\centering
			\fontsize{0.25cm}{1em}
			\def\svgwidth{8cm}
 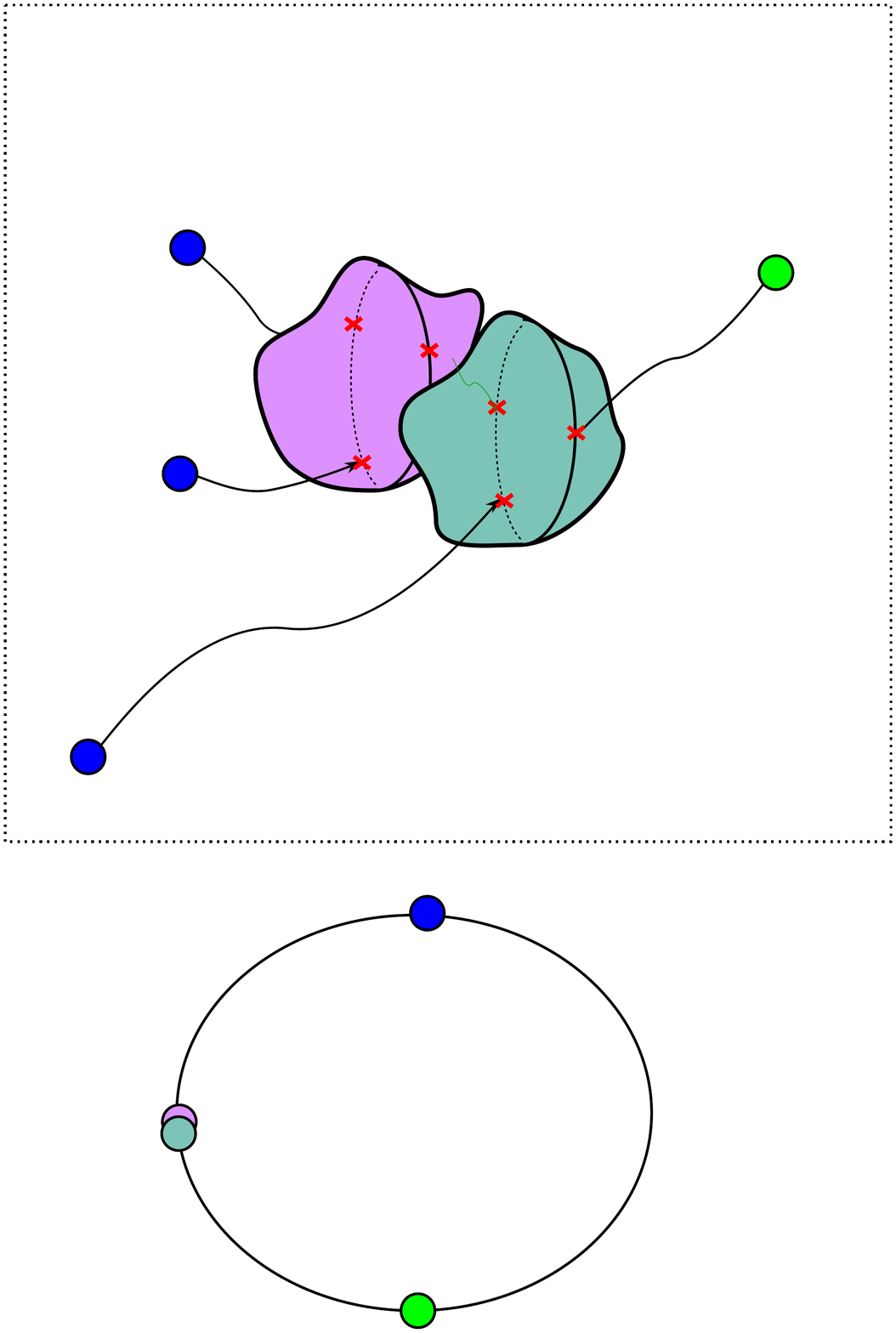
\subcaption{As it approaches zero, we can no longer \\
consider each pearl individually}
\end{minipage}
\caption{Collision problem for pearl trees defining $\mu^3_A$ over the base $B = S^1$}
\label{fig:problemwithcollision}
\end{figure} 
			
A word about navigation: we have divided the analytic framework into two sections: In this section, we disucss the steps that we need to take to define an $A_\infty$-structure. The next part includes the general setup and facts related to holomorphic curves and the graph pseudocycle (and also contains as a special case the basis for defining parametrized quantum cohomology in our setting, which does not require any successive choices of perturbations.) All our previous notations and assumptions from Section \ref{subsec:locallyhamiltonianfibrations} carry over to this Section as well.




\subsection{Consistent choice of cylindrical ends} \label{subsec:consistent_choice_of_cylindircal_ends}
Let $S$ be a smooth, pointed, oriented surface. We consider a fibre bundle 
\begin{equation}
\pi : \mathcal{S} \to \mathcal{R}
\end{equation}
whose fibre is $S$, and whose structure group is the group of oriented diffeomorphisms of $\hat{S}$ which are the identity on the marked points. Then each fibre $\mathcal{S}_r$ then has a natural compactification $\hat{\mathcal{S}}_r$ and these fit together canonically into a smooth proper fibre bundle
\begin{equation}
\hat{\pi} : \hat{\mathcal{S}} \to \hat{\mathcal{R}}
\end{equation}
Moreover, each marked point of $S$ gives rise to a section 
\begin{equation}
\zeta : \mathcal{R} \to \hat{S}, 
\end{equation}
such that
\begin{equation}
S = \hat{S} \setminus \bigcup_\zeta \zeta(\mathcal{R}).
\end{equation}
We will often identify the points of $S$ with the corresponding sections. Now suppose that $\mathcal{S}$ comes equipped with an almost complex structure $I_{\mathcal{S}}$ on the fibrewise tangent bundle, which extends to $\hat{\mathcal{S}}$, hence makes each fibre $\mathcal{S}_r$ into a pointed Riemann surface. In that situation, we call $\mathcal{S}$ a \textbf{family of pointed surfaces}.
\begin{definition}
A set of strip and cylinder data for $\mathcal{S} \to \mathcal{R}$ consists of proper embeddings 
\begin{equation}
\epsilon_\zeta : Z^\pm \times \mathcal{R} \to \mathcal{S} \: \text{ or } \: \epsilon_\zeta : A^\pm \times \mathcal{R} \to \mathcal{S}
\end{equation}
fibered over $\mathcal{R}$, one for each $\zeta$, which restrict to strip-like or cylindrical ends on each fibre. These will automatically extend to smooth, fibrewise holomorphic embeddings $\hat{\epsilon}_\zeta$, whose restriction to $\pm \infty$ is the corresponding section $\zeta$.
\end{definition}
Let $d \geq 2$. The moduli space of $(d + 1)$-marked disks, 
\begin{equation}
\mathcal{R}_{d}, 
\end{equation}
is the space of isomorphism classes of marked disks $(D; z_0, \ldots , z_d)$ with respect to the $PSL_2(\R)$-action. Since $(d+1)$-marked discs have no nontrivial automorphisms (preserving the distinguished incoming point at infinity), there is a universal family, denoted by
\begin{equation}
\mathcal{D}_{d} \to \mathcal{R}_{d}. 
\end{equation}
The base, which is called \textbf{the moduli space of stable discs}, is a smooth manifold diffeomorphic to $\R^{d-2}$. As a set, the Deligne–Mumford compactification is given by the disjoint union over the collection of all d-leafed ribbon trees of products of lower-dimensional moduli spaces
\begin{equation} \label{eq:moduliofbrokendiscs}
\overline{\mathcal{R}}_{d} = \coprod_T \mathcal{R}_T \: \text{, where }\: \mathcal{R}_T = \prod_{v \in V} \mathcal{R}_{|v|};
\end{equation}
Similarly, by allowing marked points to come together and bubble off a disk we obtain a partial compactification of the universal family
\begin{equation}
\overline{\mathcal{D}}_{d} \to \cStasheffdisc_d. 
\end{equation}
\begin{lemma}
There is a smooth structure on $\cStasheffdisc_d$ which makes it into a compact, $(d-2)$-dimensional manifold with boundary and corners. 
\end{lemma}
More precisely, recall that there is a gluing operation of pointed disks, which takes as input: 
\begin{itemize}
\item
A d-leafed ribbon tree $T$, \vspace{0.2em}
\item
For each vertex, a $|v|$-pointed disk $D_v$ with strip-like ends, \vspace{0.2em}
\item
For each interior edge $e$, a gluing parameter $\rho_e \in (-1,0)$,  \vspace{0.2em}
\end{itemize}
and produces a new $(d+1)$-pointed disc $D_\rho$ with strip-like ends by:
\begin{enumerate}
\item
\textbf{Cutting finite strips from each disc}. For every interior edge, denote $l_e := e\log(-\rho_e)/\pi$. For every vertex $v$ of $T$, \vspace{0.2em}
\begin{equation} \label{eq:gluesurfaces}
D'_v := D_v \backslash \coprod_{f = (v,e) \in \Flag^{int}(T)} \epsilon_f(Z^\pm_{l_e})
\end{equation}
where $\epsilon_f$ is the strip-like end of $D_v$ which corresponds to the flag $f = (v,e)$. \vspace{0.2em}
\item
\textbf{Identifying the stubs}. Set $D_\rho := \coprod_{v} D'_v / \sim$, where
\begin{align*}
\epsilon_{f^+(e)} (s,t) \sim \epsilon_{f^-(e)} (s+l_e,t) 
\end{align*}
for each interior edge $e$ and $(s,t) \in [0,-l_e] \times [0,1]$ and $f^\pm(e)$ are the appropriate positive/negative flags. 
\end{enumerate}
We can also allow the some gluing parameters to be zero (simply by omitting the degenerate edges in both steps of the gluing process). The strip-like ends of $D_\rho$ are inherited from those associated to the exterior flags of $T$. See \cite[Subsection (9e)]{MR2441780} for an in-depth discussion. 
\begin{definition}
An important byproduct of the gluing construction is the \textbf{thick-thin decomposition} of $D_\rho$. The thin part is defined to be the union of
the strip-like ends and the finite strips of length $l_e$. The thick part is the complement.
\end{definition}
A straightforward generalization of the gluing construction is to start with a ribbon tree and \emph{families} $\mathcal{D}_{|v|} \to \mathcal{R}_{|v|}$ of $|v|$-pointed discs with strip-like ends, for each vertex $v$. Gluing then produces a new family
\begin{equation}
\mathcal{D} \to \mathcal{R} := (-1,0]^{\Edge^{int}(T)} \times \prod_v \mathcal{R}_{|v|}
\end{equation}
As before this has a partial compactification $\overline{\mathcal{D}}_{|v|} \to \overline{\mathcal{R}}_{|v|}$ where the gluing parameters are allowed to become zero.
\begin{definition}
A \textbf{universal choice of strip-like ends for disks} to be a choice, for every $d \geq  2$, of a set of strip-like ends $\left\{\epsilon_k^{d+1}\right\}$ for $\mathcal{D}_{d+1} \to \mathcal{R}_{d+1}$.
\end{definition}
Given such a choice, we can put a topology and smooth structure on $\overline{\mathcal{R}}_{d}$ by gluing universal families. That is, given a d-leafed stable tree $T$, we associate to each vertex the universal families $\mathcal{D}_{|v|} \to \mathcal{R}_{|v|}$. Gluing these together yields
a family of $(d+1)$-pointed discs
\begin{equation} \label{eq:9dot12}
\mathcal{D} \to \mathcal{R} := (-1,0)^{\Edge^{int}(T)} \times \mathcal{R}_T
\end{equation}
with a smooth classifying map 
\begin{equation}
\gamma^T : \mathcal{R} \to \mathcal{R}_{d+1}. 
\end{equation}
Allowing gluing parameters to become zero gives a degenerate glued surface, which is a disjoint union of pointed disks governed by a the collapsed tree $\bar{T}$ (which is obtained by contracting all $\rho_e=0$ edges). Clearly, this determines a point in the $\bar{T}$-stratum of \eqref{eq:moduliofbrokendiscs}, and thus a canonical extension of the gluing maps
\begin{equation} \label{eq:9dot14}
\bar{\gamma}^T : \cStasheffdisc :=  (-1,0]^{\Edge^{int}(T)} \times \mathcal{R}_T  \to \cStasheffdisc_{d+1}.
\end{equation}
The topology and smooth structure on $\overline{\mathcal{R}}_{d}$ are characterized by requiring the (associative!) gluing maps $\bar{\gamma}^T$ are continuous and smooth (and they are independent of the universal choice used to define them). Moreover, define $\mathcal{D}_T \to \mathcal{R}_T$ to be the disjoint union of the pullbacks of the universal fibrations $\mathcal{D}_{|v|} \to \mathcal{R}_|v|$, $v \in \Vert(T)$; the fibres of $\mathcal{D}_{T}$ will be disjoint unions of $|v|$-pointed discs. The universal family $\mathcal{D}_{d+1} \to \mathcal{R}_{d+1}$ admits a partial compactification, which set-theoretically is
\begin{equation} \label{eq:9dot16}
\overline{\mathcal{D}}_{d+1} = \coprod_T \mathcal{D}_{T} \to \cStasheffdisc_{d+1}.
\end{equation}
This can be equipped with a topology and smooth structure, by using gluing maps on the total spaces of families. By definition, the gluing map $\gamma^T$ associated to \eqref{eq:9dot12} is
covered by a fibrewise isomorphism
\begin{equation}
\Gamma^T : \mathcal{D} \to \mathcal{D}_{d+1}.
\end{equation}
As we observed before, $\mathcal{D}$ has a natural partial compactification $\overline{\mathcal{D}} \to \cStasheffdisc$. By comparing this with \eqref{eq:9dot16}, one sees that $\Gamma^T$ extends to $\overline{\Gamma}^T : \overline{\mathcal{D}} \to \overline{\mathcal{D}}_{d+1}$. \\

Note that the family \eqref{eq:9dot12} carries two sets of strip-like ends (which a priori might be different): The first set is obtained from the gluing construction; and the second set is pulled back from the universal choice via the classifying map.
\begin{definition} \label{def:gluingparametersaresmall}
We call our universal choice \textbf{consistent} if, for every $T$, there is an open subset $\bar{U} = \bar{U}^T \subset \cStasheffdisc$ containing $\left\{0\right\} \times \prod_v \Stasheffdisc_{|v|}$, such that the two abovementioned sets of strip-like ends agree over $U = \bar{U} \cap \Stasheffdisc$. If they exist, we will refer to such $\bar{U}$ as \textbf{subsets where the gluing parameters are sufficiently small}.
\end{definition}
The space of possible choices of strip-like ends is contractible, and the consistency requirement is inductive, so 
\begin{lemma}
There exists consistent universal choices of strip-like ends.
\end{lemma}
\begin{proof}
See Lemma 9.3 in \cite{MR2441780}.
\end{proof}
Let us, once and for all, fix such a choice. \\

There is an anti-holomorphic involution on $\MM_{0,d+1}$ given by
\begin{equation} \label{eq:antiholo}
\overline{[(\CP{1};z_0,\ldots,z_d)]} = [(\CP{1};\overline{z_0},\ldots,\overline{z_d})]. 
\end{equation}
which canonically extends to $\cMM_{0,d+1}$. Let 
\begin{equation} \label{eq:antiholo}
\MM^\R_{0,d+1} \subset \MM_{0,d+1} \: , \: \cMM^\R_{0,d+1} \subset \cMM_{0,d+1}
\end{equation}
denote the fixed points. 
\begin{definitionlemma}
$\MM^\R_{0,d+1}$ has $d!$-connected components. We call this component the \textbf{d-th moduli of pearls} and denote it as 
\begin{equation}
\MM_{0,d+1}(\R).
\end{equation} 
The closure of $\MM_{0,d+1}(\R)$ in $\cMM_{0,d+1}$ (which is contained in $\cMM^\R_{0,d+1}$) is called the \textbf{compactified moduli of pearls} and denoted $\cMM_{0,d+1}(\R)$.  
\end{definitionlemma}
\begin{proof}
See Lemma 10.1 in \cite{MR1480992} and the discussion there. 
\end{proof}
The doubling trick extend to families as well. More precisely, if we fix a point of $\cStasheffdisc_T$, then any family of strip-like ends on the disc doubles to give local cylindrical ends on the corresponding family of pearls near the punctures. We have a gluing construction as before as well as extended gluing maps 
\begin{equation} \label{eq:9dot15}
\bar{\gamma}^T : \cStasheffdisc :=  (-1,0]^{\Edge^{int}(T)} \times \MM_{0,T}(\R) \to \cMM_{0,d+1}(\R)
\end{equation}
\begin{lemma}
The topology and smooth structure they define on $\cMM_{0,d+1}(\R)$ is independent of the choice of cylindrical ends, identical to the one obtained by restriction from $\cMM_{0,d+1}$, and with this structure $\cMM_{0,d+1}(\R)$ is a manifold with boundary and corners diffeomorphic to $\cStasheffdisc_{d+1}$. Moreover, the topology and smooth structure of the universal curve $\cUU_{0,d+1}(\R) \to \cMM_{0,d+1}(\R)$ from the embedding in the restriction of the universal curve $\cUU_{0,d+1}|_{\cMM_{0,d+1}(\R)} \to \cMM_{0,d+1}(\R)$ are the same as those coming from the gluing construction.
\end{lemma}
\begin{proof}
The maps \eqref{eq:9dot15} are obtained by restricting the domain of the complexified gluing maps from Lemma 9.2 in \cite{MR2441780}. 
\end{proof}
\begin{corollary}
There exists a consistent (smooth) universal choice of cylindrical ends for pearls.
\end{corollary}

\begin{remark}
Contrast this with the fact that there is no consistent universal choice of cylindrical ends for the whole $\cMM_{0,d+1}$, as we can see from the existence of $\psi$-classes. 
\end{remark}

\subsection{Pearls-with-bubbles}
The first thing to note is that (unfortunately), the Gromov bordification of the moduli space of pearl trees must include some bubble trees as well. We will eventually prove that they only appear in lower dimensional strata (and therefore can be ignored,) but for now let us make the following definitions. 

\begin{definition} A compact, complex \textbf{nodal curve} (or Riemann surface) $\Sigma$ is obtained from a collection of smooth, compact, complex curves by identifying a collection of distinct \textbf{double points}. A point $z \in \Sigma$ is \textbf{smooth} if it is not equal to any nodal point. We will only consider curves which have genus zero. 
\end{definition}
A node $z \in \Sigma$ is called \textbf{separating} if $\Sigma \setminus \left\{z\right\}$ has more connected components than $\Sigma$, otherwise it is non-separating. Let $\Sigma^\nu$ denote the normalization of $\Sigma$. 
\begin{definition} \label{def:markednodal}
A $m$-marked \textbf{nodal curve} is a pair
\begin{equation}
(\Sigma,\underline{z}) 
\end{equation}
which consist of a nodal curve $\Sigma$ together with a collection $\underline{z} = (z_1,\ldots, z_m)$ of distinct, smooth points. A \textbf{special point} is a marking or node. A \textbf{morphism} of marked nodal curves is a holomorphic map of the normalizations which descends to a continuous map of the curves, and induces an order-preserving bijection of the marked points. An isomorphism is a bijective morphism. A $m$-marked nodal curve is \textbf{stable} if it has finite automorphism group. Equivalently, each component contains at least three special points.  
\end{definition}

Each connected $m$-marked curve $\Sigma$ has a \textbf{combinatorial type} which is defined to be the tree $T = (V,E,\Lambda)$ obtained by replacing each sphere with a vertex, each nodal point with an interior edge connecting the two vertices that correspond to the irreducible components connected by the node, and each marking with a label (and an exterior vertex). 

Conversely, given \vspace{0.5em}
\begin{itemize}
\item
A $m$-labelled tree $T = (V,E,\Lambda)$, as well as, \vspace{0.5em}
\item
A tuple 
\begin{equation}
   \z = (\{z_{\alpha\beta}\}_{\alpha E\beta},\{z_i\}_{1\leq i\leq m})
\end{equation}
of points, such that for any vertex $\alpha \in V$, the \textbf{special points}
\begin{equation}
    SP_\alpha:=\{z_{\alpha\beta}\mid \alpha E\beta\}\cup \{z_i\mid
   \alpha_i=\alpha\}
\end{equation}
are pairwise distinct, \vspace{0.5em}
\end{itemize}
we can canonically define a $m$-marked nodal curve $\Sigma_\z$ with combinatorical type $T$:
\begin{enumerate}
\item
First, we form a disjoint collection of standard spheres
\begin{equation}
\coprod_{\alpha \in V} S_\alpha. \vspace{0.2em}
\end{equation}
\item
$\Sigma_\z$ is defined to be the quotient space obtained by gluing together the disjoint spheres from (1) along each pair $z_{\alpha \beta}$ for $\alpha E \beta$.  \vspace{0.2em}
\item
Finally, we assign the marked points $z_i \in S_{\alpha_i}$ according to the labels.  \vspace{0.2em}
\end{enumerate}
Moreover, given two such curves $\Sigma_{\z}$ and $\Sigma_{\z'}$ modelled over trees $T$ and $T'$ respectively, we define a \textbf{morphism} to be a tuple
\begin{equation}
   \phi=(\tau,\{\phi_\alpha\}_{\alpha \in V}),
\end{equation}
where $\tau :T\to T'$ is a tree homomorphism and 
\begin{equation}
\phi_\alpha : S^2\cong S_\alpha \to S_{\tau(\alpha)}\cong S^2 
\end{equation}
are (possibly constant) holomorphic maps which respect the edges and the markers in the obvious way. $\phi$ is called an isomorphism if $\tau$ is a tree isomorphism and each $\phi_\alpha$ is biholomorphic. Every morphism (resp. isomorphism) in the sense of definition \ref{def:markednodal} comes from such a construction. In other words, the assignment that sends a nodal curve $\Sigma = \Sigma_\z$ to its combinatorial type $T$ extends to a functor from the category of marked nodal curves to the category of graphs, and there is a canonical homomorphism $Aut(\Sigma) \to Aut(T)$, whose kernel is the product of the automorphism groups of the irreducible components of $\Sigma$. Observe that the curve $\Sigma_{\z}$ is stable if and only if the tree $T$ is stable, and stabilization of trees induces a canonical stabilization of nodal curves $\Sigma \mapsto \st(\Sigma)$. Isomorphisms from $\z$ to itself are called automorphisms. If $\z$ is stable its only automorphism is the identity (see the discussion after Definition D.3.4 in \cite{MR2954391}.)
\begin{lemma}[2.5 in \cite{MR2217687}]\label{lem:stab}
Let $\z$ be a nodal curve modelled over the tree $T$ and $\st(\z)$ its
stabilization. Then the stabilization map induces a morphism
\begin{equation}
   \st=(\tau,\{\phi_\alpha\}_{\alpha \in V}):\Sigma_\z \mapsto \Sigma_{\st(\z)}
\end{equation}
with the following property: There exists a collection of subtrees $T' \subset T$ such that $\phi_v$ is constant for $v \in V'$ and biholomorphic otherwise. Moreover, each sphere $S_\alpha$ with $v \notin V'$ carries at least $3$ special points. \noproof
\end{lemma}  

\begin{definition}
A $(d+1)$-marked \textbf{pearl-with-bubbles} is a genus zero, punctured, nodal Riemann surface whose stabilization $\st(C)$ is a $(d+1)$-marked pearl. 
\end{definition}
The combinatorial type of $C$ is a $(d+1)$-labelled tree $T$ with a root $v_{main} \in \Vert^{int}(T)$ (this is the subtree $T'$ from Lemma \ref{lem:stab} that does not get collapsed, and corresponds to the pearl). The remaining internal vertices are called bubble vertices. $C$ inherit a choice of sign for the marked points from $\st(C)$, as well cylindrical ends, if such a choice has been made. \\

This has an obvious generalization 
\begin{definition}
By a \textbf{smooth family of curves of fixed type $T$} over a manifold $Q$ we mean a bundle $\Sigma_Q \to Q$ such that each fiber is a curve of type $T$: $\Sigma_q = \Sigma_{\z_q}$ and $\z_q$ is smoothly varying with respect to the base parameter. In other words, the homeomorphism type of the fiber is fixed, say $\Sigma$, but the complex structure $J$ depends smoothly on $q$. 
\end{definition}
A smooth family of pearls-with-bubbles is a family of pearls which is a smooth when considered as a family of curves of a fixed type.

\subsection{Stable maps} \label{subsec:holomorphicmaps2}
Let $(M,\omega)$ be a symplectic manifold. Let $\Sigma$ be a possibly nodal Riemann surace. Denote $T = (V,E,\Lambda)$ for the combinatorial type. A continuous map $u : \Sigma \to M$ is a collection $\left\{u_\alpha\right\}_{\alpha \in V}$ of continuous maps $u_\alpha : S_\alpha \to M$ such that $u_\alpha(z_{\alpha \beta}) = u_\beta(z_{\beta \alpha})$ whenever $\alpha E \beta$. We define the space of domain-depenedent almost complex structures on $\Sigma$ to be 
\begin{equation}
\JJ_T : = \prod_{\alpha \in T} \JJ_{S_\alpha}(M,\omega).
\end{equation}

\begin{definition}
Let $\textbf{J} \in \JJ_T$. A continuous map $u : \Sigma \to M$ is $\textbf{J}$-holomorphic if the restriction to every sphere $u_\alpha : S_\alpha \to M$ is $\textbf{J}_\alpha$-holomorphic. 
\end{definition}

All the adjectives from the Section \ref{subsec:holomorphicmaps1} such as: simple, multi-covered etc can be applied to nodal domains \emph{mutatis mutandi}, by simply replacing the curve $\Sigma$ with its normalization $\Sigma^\nu$. In particular, every pseudo-holomorphic map $u$ defines a homology decomposition, $[u] = \left\{[u_\alpha]\right\}$ indexed by the vertices of the tree.

\begin{definition}
A \textbf{weighted $m$-labelled tree} 
\begin{equation}
(T,\{A_\alpha\}) 
\end{equation}
is a $m$-labelled tree $T$ with a homology decomposition $\underline{A} = \left\{A_\alpha\right\}$, where $A_\alpha \in H_2(M;\Z)$ indexed by the vertices $\alpha \in V$. 
\end{definition}

We call $(T,\{A_\alpha\})$ \textbf{weighted stable} if each component $\alpha\in T$ with $A_\alpha=0$ carries at least 3 special points. Note that stability implies weighted stability but not vice versa. 

\begin{definition}
A pseudo-holomorphic map $u : \Sigma \to M$ is stable if the weighted tree $(T,[u])$ is stable.
\end{definition}

\begin{definition}
We call a component $\alpha$ with $A_\alpha=0$ a \textbf{ghost component}, and a maximal subtree consisting of ghost components a
\textbf{ghost tree}. 
\end{definition}
Note that a pseudo-holomorphic map $u : \Sigma \to M$ is stable if and only if every component with zero energy has at least three special points. 

\begin{definition} \label{def:reducedindexset}
The \textbf{reduced index set} of $(T,\{A_\alpha\})$ is the subset $R\subset\{1,\dots,m\}$ containing all marked points on
components with $A_\alpha\neq 0$, and the maximal marked point on each ghost tree. 
\end{definition}

\begin{remark}
As remarked in \ref{subsec:holomorphicmaps2}, all the discussion goes through for vertical stable maps $(b,u)$ with $b \in B$ and $u : \Sigma \to E_b$ without change. 
\end{remark}

\subsection{Singular pearl trees}
\begin{definition}
A \textbf{singular pearl tree} $P = (T,\underline{C})$ with $d+1$ external edges consists of \vspace{0.5em}
\begin{itemize}
\item
A singular metric tree 
\begin{equation}
T = (T,g_T) 
\end{equation}
with $d+1$ external edges of infinite type. \vspace{0.5em}
\item
A collection of $|v|$-marked pearls-with-bubbles
\begin{equation}
\underline{C} = \left\{(C_v,\underline{\smash{z}}_v)\right\}_{v \in \Vert^{finite}(T)},
\end{equation}
indexed the set of all vertices of finite type. \vspace{0.5em}
\end{itemize}
We require that the marked points on $C_v$ are ordered in such a way that they are in order-preserving bijection with the flags adjacent to $v$. 
\begin{remark}
Note that a pearl tree is a special case of a singular pearl tree where the metric only takes values in $(0,\infty)$ and each $C_v$ is a marked pearl. 
\end{remark}
\end{definition}
We generalize the notion of a pearl-with-bubbles a bit to allow for the possibility of more then one main vertex (this correspond to broken pearl-with-bubbles, whose stablization is the double of a broken nodal disc) and define
\begin{definition}
An \textbf{isomorphism} of singular pearl trees is an equivalence of the underlying metric tree together with an isomorphism of collections of nodal curves. 
\end{definition}
In particular, given an edge of length zero in $T$, the operation of removing that edge and merging the two pearl-with-bubbles together to a nodal pearl-with-bubbles with two main vertices is an isomorphism of singular pearl trees. We will favor the latter description in such a case.
\begin{definition}
The \textbf{topological realization} $|P|$ of a singular pearl tree $P$ is obtained by removing the vertices from the tree and gluing in the nodal curves by attaching the marking to the edges of tree, according to the bijection between the marked points and the flags.
\end{definition}
\begin{definition}
A singular pearl tree $P = (T,\underline{C})$ is \textbf{stable} if the tree $T$ is stable and every nodal curve is stable. 
\end{definition}
In particular, a stable singular pearl tree has no bubble vertices and every pearl tree is necessarily stable.
\begin{lemma}
A stable singular pearl tree has no automorphism except the identity. 
\end{lemma}
\begin{proof}
The singular Stasheff tree $T$ has no automorphism but the identity, and each marked pearl-with-bubble must have valency $|v| \geq 3$, thus no automorphism as well. 
\end{proof}
\begin{definition} \label{def:combinatorialtypeforpearls}
The \textbf{combinatorial type} 
\begin{equation}
\Upsilon_P 
\end{equation}
of any singular pearl tree $P$ is the graph obtained by gluing together the combinatorial types of the nodal curves along the edges corresponding to the edges of $T$ with some additional decorations. More precisely, \vspace{0.5em}
\begin{itemize}
\item
We record the partition of the set of internal vertices
\begin{equation}
\Vert^{int}(\Upsilon_P) = \Vert^{main}(\Upsilon_P) \sqcup \Vert^{bubble}(\Upsilon_P),
\end{equation}
coming from the combinatorial type of the pearls-with-bubbles. We declare vertices of each pearl-tree-with-bubbles to be of finite type, and internal vertices of $T$ of infinite type to be main vertices. \vspace{0.5em}
\item
We retain the information of the subset
\begin{equation}
\Edge^{int}_{0}(\Upsilon_P) 
\end{equation}
of interior edges with zero length. The length function on edges of $T$ extends to the edges of $\Upsilon_P$ coming from the pearl-with-bubbles, by setting the lengths of those edges to be zero. Note that stability (and being non-singular) depend only on the combinatorial type.
\end{itemize}
\end{definition}
We denote 
\begin{equation}
\Edge^{int}_{(0,\infty)}(\Upsilon_P) 
\end{equation}
for the set of edges of finite, positive length and 
\begin{equation}
\Edge^{int}_{\infty}(\Upsilon_P) 
\end{equation}
for the set of edges of infinite length. They are both uniquely determined by the infinite type vertices and $\Edge^{int}_{0}(\Upsilon_P)$. Like in Definition \ref{def:combinatorialtypeforpearls}, we note that from a decorated graph $\Upsilon$ one can define a singular pearl tree $P$ based on it by specifying the special points indicating how to attach the spheres to each other. 
\begin{definition}
Every singular pearl tree has a decomposition into a one-dimensional part (which consists of edges of positive length) called the \textbf{tree part}, and a two-dimensional part, called the \textbf{sphere part}, and which is the disjoint union of nodal curves of genus zero.
\end{definition}
We will often refer to various subsets of vertices and edges of a singular pearl tree $P$, by which we mean the vertices and edges of the combinatorial type $\Upsilon_P$. This should cause no confusion.  

\subsection{Moduli of pearl trees} \label{subsec:moduliofpearls}
Let $d \geq 2$. For each stable pearl type $\Upsilon$ with $d$ leaves we denote by 
\begin{equation}
\QQ_\Upsilon 
\end{equation}
the set of isomorphism classes of pearl trees $P$ such that $\Upsilon_P = \Upsilon$. There is a bijection between $\QQ_\Upsilon$ and the product of a cubical cell in the moduli space of Stasheff trees and copies of the moduli of pearls, indexed by the vertices:
\begin{equation}
\Pearltree_\Upsilon \iso (0,1)^{\Edge_{(0,+\infty)}^{int}}(\Upsilon) \times \bigtimes_{v \in V} \MM_{0,|v|}(\R) \iso \R^{d-2-|\Edge_{0}^{int}(\Upsilon)|-|\Edge_{\infty}^{int}(\Upsilon)|}.  
\end{equation}
which we use to put a topology and smooth structure on $\QQ_\Upsilon$. 
\begin{definition}
The \textbf{d-th moduli of pearl trees} is the smooth manifold 
\begin{equation}
\QQ_d = \coprod_\Upsilon \QQ_\Upsilon
\end{equation}
where the union is taken over all pearl types with $d$ leaves. We denote $\PP_d$ for the universal curve over $\QQ_d$, and $\TT_d,\mathcal{C}_d$ for the tree and sphere part, respectively.
\end{definition}
As a set, the closure $\cPearltree_\Upsilon$ is the set obtained by adding isomorphism classes of stable singular pearl trees where we allow edges to become zero or infinite length (and pearls to break). We define the codimension of a stable singular pearl type $\Upsilon$ to be
\begin{equation}
\codim(\Upsilon) := |\Edge_{0}^{int}(\Upsilon)|+|\Edge_{\infty}^{int}(\Upsilon)|.
\end{equation}
$\Upsilon$ is a pearl type if and only if the $\codim(\Upsilon)=0$.
\begin{definition} As a set, we define the \textbf{compactified d-th moduli of pearl trees} as the disjoint union
\begin{equation}
\bar{\QQ}_d = \coprod_\Upsilon \QQ_\Upsilon
\end{equation}
over all stable combinatorical types with d leaves. 
\end{definition}
Let $\Upsilon$ be a stable pearl type. The moduli spaces $\QQ_\Upsilon$ admit a universal curve 
\begin{equation}
\PP_\Upsilon \to \QQ_\Upsilon 
\end{equation}
which consists of isomorphism classes of pairs $[(P, z)]$ where $P$ is a singular pearl tree of type $\Upsilon$ and $z$ is a point in $P$, on a nodal curve or on an edge. The forgetful map $[(P, z)] \mapsto [P]$ is the universal projection. Because of the stability condition, there is a natural bijection
\begin{equation}
\PP_\Upsilon = \bigcup_{[P] \in \QQ_\Upsilon} P.
\end{equation}
We can partition it into the tree part (the locus where $z$ lies on an edge of $P$)
\begin{equation}
\TT_\Upsilon \to \QQ_\Upsilon,
\end{equation}
and a sphere part (the locus where $z$ lies on an nodal curve)
\begin{equation}
\mathcal{C}_\Upsilon \to \QQ_\Upsilon. 
\end{equation}
The intersection of the tree part and the sphere part is the set of marked points. As usual, there is a corresponding partial compactification (with respect to the topology defined in the next subsection) of the universal family, 
\begin{equation}
\bar{\PP}_d  \to \bar{\QQ}_d 
\end{equation}
by allowing the edges to break and the spheres to break. 
\begin{remark}
Note that unlike $\cStashefftree_d$ or $\cStasheffdisc_d$, when $d \geq 3$, the moduli of pearls will have multiple cells in the top dimension (because $\PP_d$ is disconnected). 
\end{remark}
For every fixed combinatorial type, each irreducible component of $\mathcal{C}_\Upsilon \to \QQ_\Upsilon$ is (a particular instance of) a family of pearls, thus the sphere part inherits a choice of cylindrical ends from the consistent universal choice we made before via the product of classifying maps
\begin{equation}
\QQ_\Upsilon \to \prod_{v \in \Vert^{int}(\Upsilon)} \MM_{0,|v|}(\R).
\end{equation}
\textbf{Notation.} It us convenient to consider disconnected combinatorial types $\Upsilon$ which are the disjoint union of stable singular pearls types. In that case we denote,
\begin{equation} \label{eq:pearltreebase}
\Upsilon = \sqcup_k \Upsilon_k \: \Rightarrow \: \cPearltree_{\Upsilon} = \bigtimes_k \cPearltree_{\Upsilon_k},
\end{equation}
and the universal pearl tree $\overline{\mathcal{P}_{\Upsilon}}$ is taken to be the disjoint union of the pullbacks of the universal pearl trees $\overline{\mathcal{P}_{\Upsilon_k}}$: If $\pi_k$ is the projections onto the k-th factor in equation \eqref{eq:pearltreebase} above then
\begin{equation} \label{eq:pearltreefamily}
\overline{\mathcal{P}_{\Upsilon}}= \bigsqcup_k \pi_k^* \overline{\mathcal{P}_{\Upsilon_k}}.
\end{equation}
\subsection{Cubical decomposition} \label{subsec:cubicaldecomposition}
While the compactified moduli of pearls is not naturally a manifold with boundary and corners\footnote{Of course, it can be made into such, see \cite{MR3153325}.} , it does have a cubical decomposition. Following \cite{eprint1} and \cite{eprint2}, we formalize this relationship between different $Q_\Upsilon$ by introducing the following operations on stable singular combinatorial types. 
\begin{definition} \label{def:graftingmorphism}
We say that $\Upsilon$ was obtained from $\Upsilon'$ by \textbf{grafting} (see Figure \ref{fig:grafting}), if there is a morphism of graphs
\begin{equation}
\tau : \Upsilon' \to \Upsilon
\end{equation}
which is a bijection on vertices and edges except it takes the root of a connected component $\Upsilon_-$ of $\Upsilon'$ and identifies it with k-th leaf of another connected component $\Upsilon_+$ of $\Upsilon'$. 
\end{definition} 
\begin{remark}
Note that the resulting $\Upsilon$ is never a pearl type, since it must contain at least one bivalent vertex. Grafting allows us to identify the ''actual boundary" (in contrast to the ''fake boundary" consisting of singular pearl trees with zero edge length) of the moduli space of pearl types with the product of lower dimensionals strata.
\end{remark} 
\begin{figure} 
\centering
  \begin{subfigure}[b]{.28\linewidth}
    \centering
		\fontsize{0.3cm}{1em}
		\def\svgwidth{5cm}
		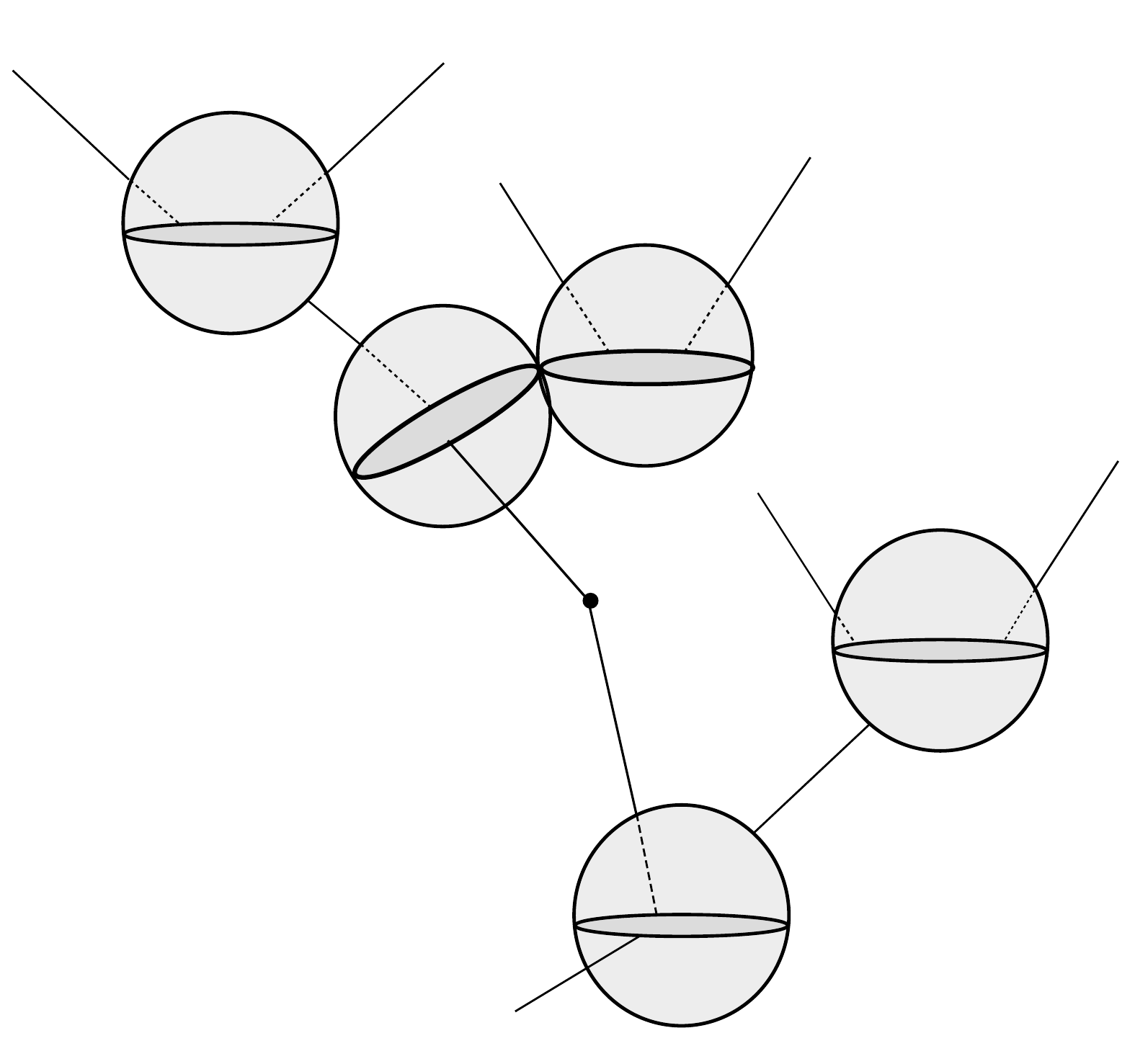
		\caption*{The combinatorial type $\Upsilon'$}\label{fig:typeupsilonprime}
  \end{subfigure} \hspace{0.15\linewidth}
  \begin{subfigure}[b]{.28\linewidth}
    \centering
		\fontsize{0.3cm}{1em}
		\def\svgwidth{5cm}
		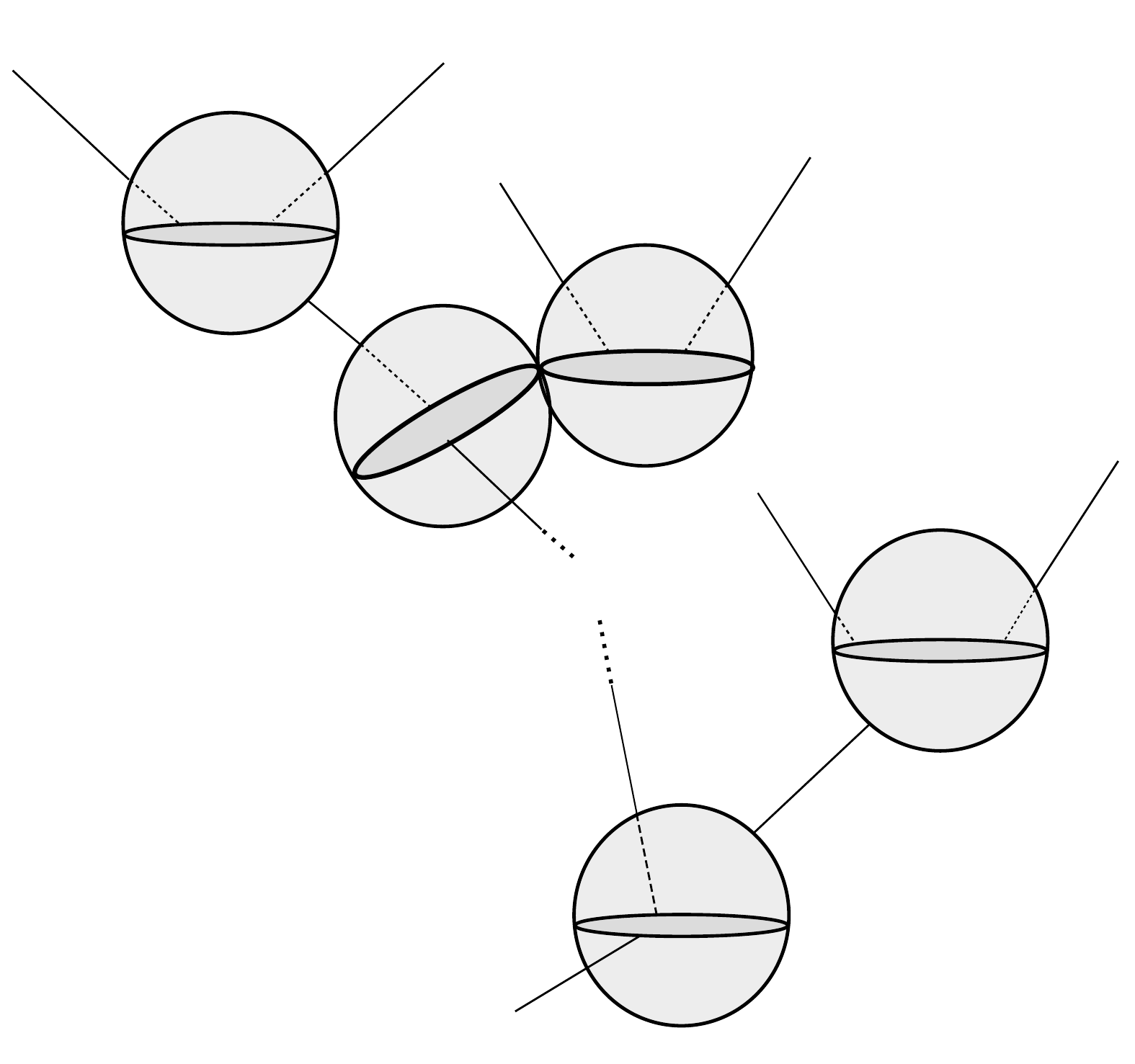
		\caption*{The combinatorial type $\Upsilon$}\label{fig:typeupsilon}
  \end{subfigure}
	\caption{Grafting}
	\label{fig:grafting}	
\end{figure} 
This operation extends to a diffeomorphism between the respective moduli spaces 
\begin{equation} \label{eq:graftingpearls}
\overline{\mathcal{Q}}_{\Upsilon'} \to \overline{\mathcal{Q}}_{\Upsilon} \: , \: [P'] \mapsto [P]
\end{equation}
and induces a map between the universal pearl trees by choosing a representative $P'$ for each isomorphism class $[P']$, identifying the root of one connected component with the k-th leaf of the other to obtain a new singular pearl tree $P$, and then taking the isomorphism class.
\begin{definition} \label{def:morphismofcombinatorialtypes} A morphism of combinatorial types 
\begin{equation}
\tau : \Upsilon' \to \Upsilon
\end{equation}
is a graph morphism which is surjective on the set of vertices $\Vert(\Upsilon') \to \Vert(\Upsilon)$, and was obtained by combining the following elementary morphisms (see Figure \ref{fig:combinatorialtype} for an illustration) \vspace{0.3em}
\begin{description}
\item[Type I]
We say that $\tau$ \textbf{collapsing an edge} if the map on vertices is a bijection except for a single vertex that has two pre-images
\begin{equation}
\Vert(\Upsilon') \setminus \left\{v_-,v_+\right\} \iso \Vert(\Upsilon) \setminus \left\{v\right\} \: , \: \tau^{-1}(v) = \left\{v_-,v_+\right\}
\end{equation}
The vertices $v_{-}, v_+$ are the end points of a edge $e = (v_-,v_+) \in \Edge^{int}_0(\Upsilon)$, and the map on edges satisfies
\begin{equation}
\Edge(\Upsilon') \iso \Edge(\Upsilon) \setminus \left\{e\right\}. \vspace{0.5em}
\end{equation} 
\item[Type II]
We say that $\tau$ \textbf{makes an edge non-zero} if it is a bijection on vertices and edges, including combinatorical type, except for a single edge
\begin{equation}
e \in Edge^{int}_0(\Upsilon') 
\end{equation}
which becomes finite, i.e., 
\begin{equation}
f(e) \in Edge^{int}_{(0,\infty)}(\Upsilon). \vspace{0.5em}
\end{equation} 
\item[Type III]
We say that $\tau$ \textbf{makes an edge finite} if it is a bijection on vertices and edges, including combinatorical type, except for a single edge
\begin{equation}
e \in Edge^{int}_{\infty}(\Upsilon') 
\end{equation}
which becomes finite, i.e., 
\begin{equation}
f(e) \in Edge_{(0,\infty)}^{int}(\Upsilon). \vspace{0.5em}
\end{equation} 
\end{description}
\end{definition}
		\begin{figure} 
		  \begin{subfigure}[b]{.28\linewidth}
    \centering
		\fontsize{0.3cm}{1em}
		\def\svgwidth{5cm}
		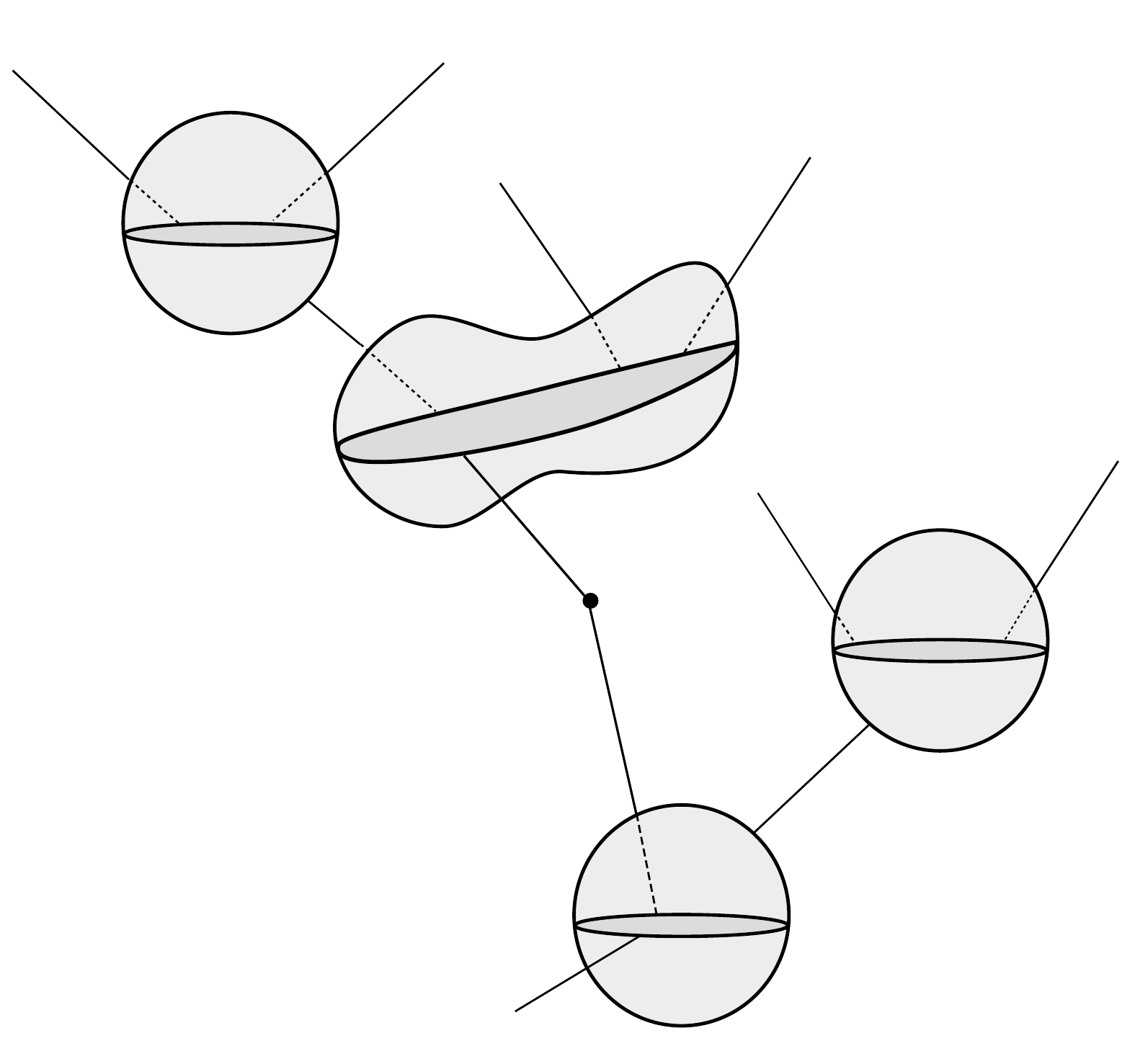
		\caption*{(I) Collapsing an edge}\label{fig:typeII}
\end{subfigure}\hfill
			  \begin{subfigure}[b]{.28\linewidth}
    \centering
		\fontsize{0.3cm}{1em}
		\def\svgwidth{5cm}
		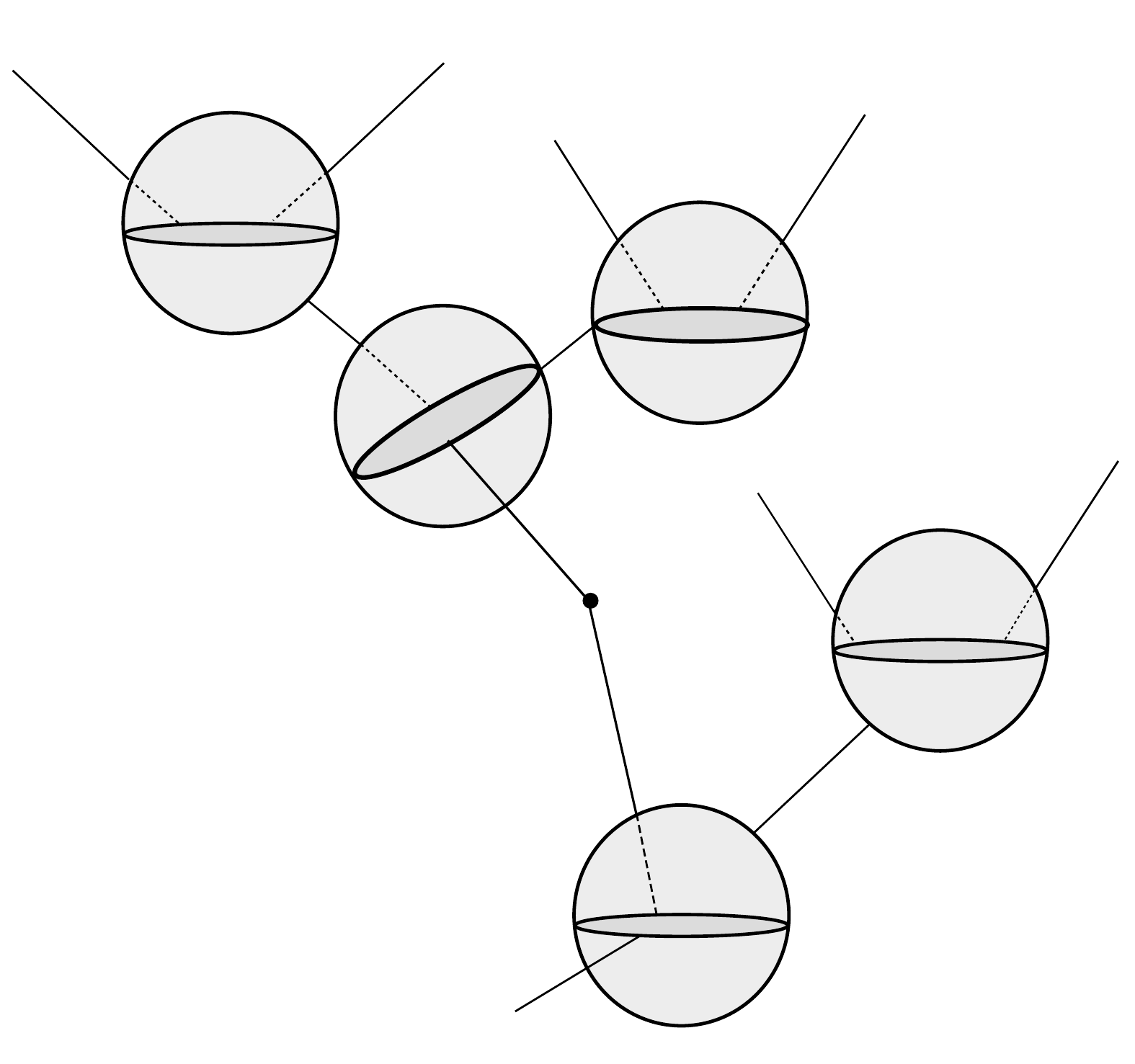
		\caption*{(II) Makes edge nonzero}\label{fig:typeIII}
		\end{subfigure}\hfill
		\begin{subfigure}[b]{.28\linewidth}
    \centering
		\fontsize{0.3cm}{1em}
		\def\svgwidth{5cm}
		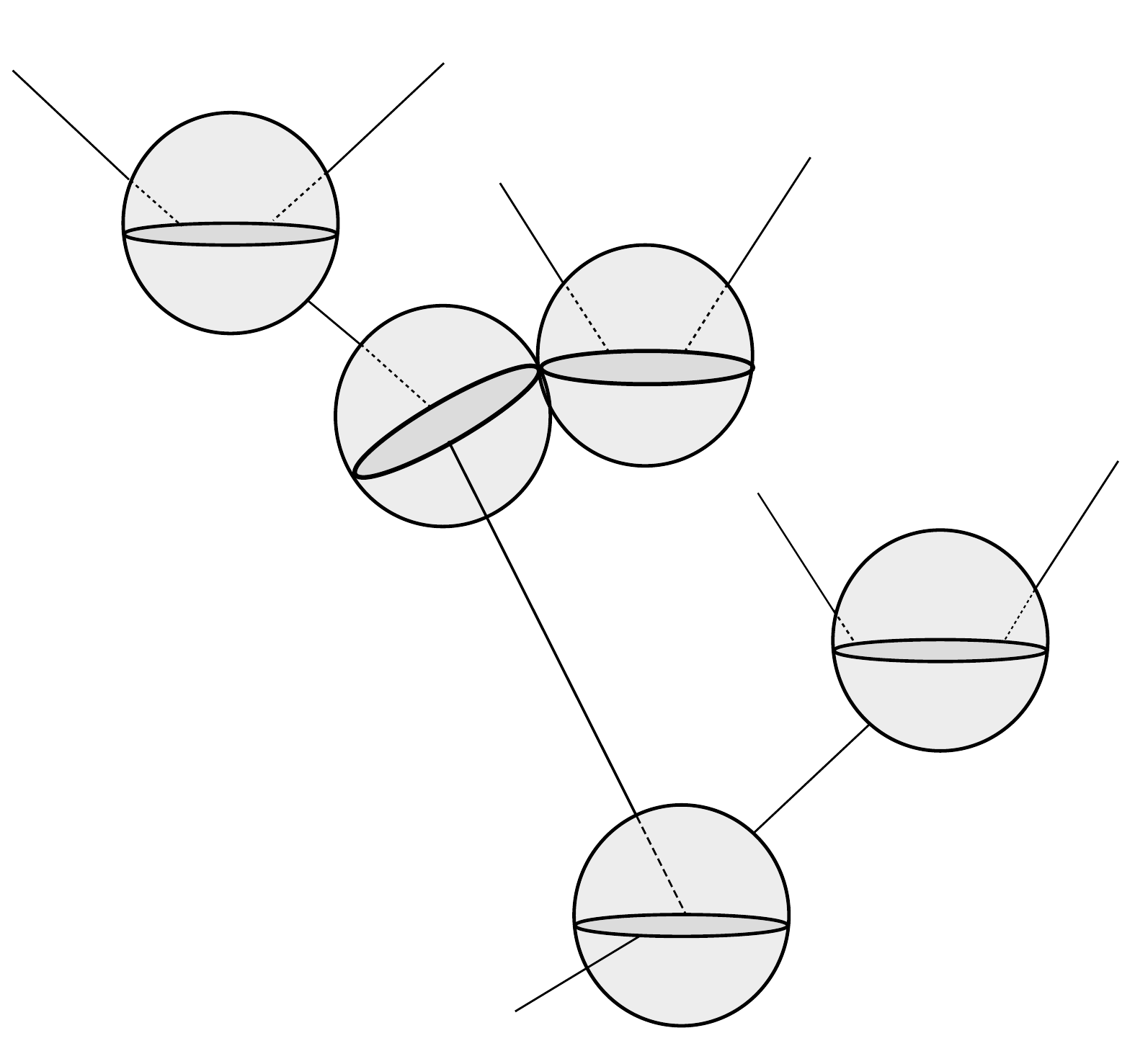
		\caption*{(III) Makes edge finite}\label{fig:typeIV}
		\end{subfigure}
	\caption{Elementary operations $\Upsilon' \to \Upsilon$ by type}
	\label{fig:combinatorialtype}
\end{figure}	
\begin{remark}
Definition \ref{def:morphismofcombinatorialtypes} induces an order relation on combinatorial types: we denote $\Upsilon' \leq \Upsilon$ if there exists a morphism $\tau : \Upsilon' \to \Upsilon$ of combinatorial types.
\end{remark}
Given a combinatorial type $\Upsilon'$, we choose a set decomposition of $\Edge^{int}_{0}(\Upsilon')$ into $\Edge^{int}_I(\Upsilon') \sqcup \Edge^{int}_{II}(\Upsilon')$, and assign gluing parameters to the interior edges in the following way: if $e \in \Edge^{int}_{(0,\infty)}$ then $\rho_e \in (0,1)$, if $e \in \Edge^{int}_I$ then $\rho_e \in (-1,0]$, if $e \in \Edge^{int}_{II}$ then $\rho_e \in [0,1)$, and if $e \in \Edge^{int}_{\infty}$ then $\rho_e \in (0,1]$. Assign the universal family of pearls $\UU_{0,|v|}(\R) \to \MM_{0,|v|}(\R)$ to every vertex. There is a straightforward generalization of the gluing construction given for disks (or pearls) where we glue all subtrees with negative gluing parameters, ignore the vertices with gluing parameter zero and define the length of the remaining edges with positive gluing parameter as $g_T(e) = -\log(-\rho_e)$. Then combining the extended gluing map for pearls from equation \eqref{eq:9dot14} with the compactified cubical decomposition for Stasheff trees gives an \textbf{extended gluing map for pearl trees}
\begin{equation} \label{eq:extendedgluingmapsforpearls}
\gamma^{\Upsilon'} : \Pearltree_{\Upsilon'} \times (-1,0]^{\Edge^{int}_I(\Upsilon)} \times [0,1)^{\Edge^{int}_{II}(\Upsilon)} \times (0,1]^{\Edge^{int}_{\infty}(\Upsilon)} \to \cPearltree_{d} 
\end{equation}
Thus, each morphism of topological types $\Upsilon' \to \Upsilon$ gives a \textbf{corner chart} 
\begin{equation} \label{eq:elementarypearls}
\Pearltree_{\Upsilon'} \to \cPearltree_{\Upsilon}.
\end{equation}
it is easy to see that the corner charts are compatible (follows from the associtivity of the gluing process for pearls) and each of the maps involved lifts to a smooth map of universal pearl trees 
\begin{equation}
\begin{tikzpicture}[every node/.style={midway}]
  \matrix[column sep={4em,between origins}, row sep={2em}] at (0,0) {
    \node(R) {$\mathcal{P}_{\Upsilon'}$}  ; & \node(S) {$\overline{\mathcal{P}_{\Upsilon}}$}; \\
    \node(R/I) {$\mathcal{Q}_{\Upsilon'}$}; & \node (T) {$\overline{\mathcal{Q}}_{\Upsilon}$};\\
  };
  \draw[<-] (R/I) -- (R) node[anchor=east]  {$$};
  \draw[->] (R) -- (S) node[anchor=south] {$$};
  \draw[->] (S) -- (T) node[anchor=west] {$$};
  \draw[->] (R/I) -- (T) node[anchor=north] {$$};
\end{tikzpicture}
\end{equation}
To summarize,
\begin{claim} For any combinatorial type $\Upsilon$, the closure is a compact manifold with boundary and corners, homeomorphic to a closed cube,
\begin{equation}
\cPearltree_{\Upsilon} \iso [0,1]^{d-2-\codim(\Upsilon)},
\end{equation}
whose codimension one strata are labelled by elementary morphism $\Upsilon' \to \Upsilon$. 
\end{claim}
The topology on $\overline{\mathcal{Q}}_d$ is the quotient topology, determined by identifying the boundary strata of different top dimension cubes along the images of elementary operations. It is easy to see that $\overline{\mathcal{Q}}_d$ is a compact, Hausdorff, 2nd countable, and regular (thus metrizable by Urysohn's metrization theorem). Moreover, 
\begin{definition} \label{def:thickthinchartcoordinate}
From the corner charts described above every $P \in \mathcal{Q}_d$ inherits a \textbf{thick-thin decomposition}, where the thin parts of $P$ are by definition the finite cylinders in each pearl that are inherited from the gluing parameters if $P$ lies in a type (I) chart (or a composition of such charts). If $P$ does not lie in such a chart, then $P$ has no thin parts. The thick part is the complement of the thin part.
\end{definition}

\subsection{Smooth choice of perturbation data}
Let $\Upsilon$ be a combinatorial type. Assume we have fixed Floer data. 
\begin{definition}
A \textbf{smooth choice of perturbation data} 
\begin{equation}
Y^\Upsilon = (X^\Upsilon,J^\Upsilon)
\end{equation}
for $\Pearltree_\Upsilon$ is: \vspace{0.5em}
\begin{itemize}
\item[(a)]
A smooth family of compatible, almost complex structures
\begin{equation}
J \in \mathcal{J}(\pi,\Omega)_{\overline{\mathcal{C}}_\Upsilon}. \vspace{0.5em}
\end{equation}
compatible with the Floer data near the cylindrical ends. 
\item[(b)]
A choice of perturbation data for each element of $\overline{\mathcal{T}}_\Upsilon$ which is smooth and compatible with automorphism in the sense of perturbation data for trees (see Definition \ref{def:smoothfortrees}.) \vspace{0.5em}
\end{itemize}
\end{definition}

For every $q \in \Pearltree$, the restriction of $\textbf{Y}^\Upsilon = (\textbf{X}^\Upsilon,\textbf{J}^\Upsilon)$ induces a perturbation datum on the pearl tree $\mathcal{P}_q$. 

\begin{definition}
Let $d \geq 1$. A smooth choice of perturbation data for $\PP_{d+1} \to \QQ_{d+1}$ is a continuous choice of perturbation data for each $d$-leafed pearl tree such that for every stable pearl type $\Upsilon$ with $d$-leaves, the restriction to $\PP_{\Upsilon} \to \QQ_{\Upsilon}$ is a smooth perturbation datum. 
\end{definition}

When we discuss perturbation data for families, we will always assume that they are smooth, unless states explicitly to the contrary. 

\section{The pearl complex (analytic details)}  \label{sec:parametrized}
The purpose of this section is to understand the Functional analysis needed to construct parametrized quantum cohomology and prove the Theorems in Section \ref{subsec:transversality}. 

\subsection{Parametrized pearl pseudo-cycle}
Let $\LHF$ be a Locally Hamiltonian fibration with monotone fiber $(M,\omega)$ of dimension $2n$. Denote $V = T^v E$. Consider a (possibly disconnected) Riemann surface $\Sigma$, and a tree $T = (V,E)$. All our previous notations and assumptions from Section \ref{subsec:locallyhamiltonianfibrations} carry over to this Section as well. For simplicity we first assume: that there is a fixed (domain-independent) reference almost complex structure $J^{base}$ (which still depends on $b \in B$, of course) in $\complexJ(\pi,\Omega)$ and an open, nonempty subset $\Sigma^{cpt} \subset \Sigma$. We denote by $\JJ$ the set of all (smooth) domain-dependent almost complex structure $\textbf{J} \in \complexJ_{\Sigma}(\pi,\Omega)$ that $\textbf{J}$ is of the form 
\begin{equation} \label{eq:perturbationsofJ}
\textbf{J}_{(b,z)} = J^{base} + \textbf{J}_{(b,z)}', 
\end{equation}
where $\textbf{J}_{(b,z)}'$ vanishes outside $\Sigma^{cpt} \times B$. 
\begin{remark}
We remark on the modifications needed to introduce the dependence on the extra parameters. The moduli parameter poses no additional complications; so does the tuple $\underline{b}$ which records the simultaneous positions of the different component. However, this arguments only suffice to deal with the pseudo-cylce coming from a single null-cluster (to borrow the language of Section \ref{subsec:transversality}) at a time. If we want to deal with all of them at once, we have to consider the case when different components 
collide. This requires some additional discussion. See Section \ref{subsec:colldetect}. 
\end{remark}
Let $A \in H_2(M)$ and $\underline{A}$ be a homology decomposition, indexed by the connected components $\Sigma_\alpha$ of $\Sigma$. We fix a collection of points $\underline{w}$ on $\Sigma$, pairwise distinct on each component. Consider the moduli space 
\begin{equation}
\MM_{0,k}^*(\underline{A};\underline{w};\textbf{J})
\end{equation}
of all tuples $(\underline{b},\underline{u},\underline{z})$ consisting of:
\begin{itemize}
\item
For every component $\Sigma_\alpha$, A solution $u_\alpha : \Sigma_\alpha \to E_{b_\alpha}$ of the Cauchy-Riemann equation with respect to $\textbf{J}_\alpha$, 
\item
A tuple $\underline{z} = (z_1,\ldots,z_k)$ of points on $\Sigma$,
\item
Collections of pairs of marked points $\left\{z_{\alpha \beta} , z_{\beta \alpha}\right\}$
\end{itemize}
The map $u$ is required to have homology decomposition $[u] = \underline{A}$. When considered together, we require that $[(\Sigma,\underline{z})] \in \WW$. 
\begin{remark}
The $u$ is not necessarily continuous, and only defines a map from the normalization $\Sigma^\nu$.  
\end{remark}
This moduli space carries an evaluation map 
\begin{equation} \label{eq:evaluationmapofW}
ev: \MM_{0,k}^*(\underline{A};\WW,\textbf{J})\: , \: ev(u,\underline{z}) := (u(z_1),\ldots,u(z_k)).
\end{equation}

As usual, there is a Gromov bordification $\overline{\MM}_{0,k}(\underline{A};\WW,\textbf{J})$ and a continuous extension of the evaluation map which we denote as $\overline{ev}$. 

\begin{theorem} \label{thm:pseudocycle}
Let $\JJ = \complexJ_{\Sigma}(\pi,\Omega)$. There exists a comeager set $\JJ^{(V)+(H)+(EV)} \subset \JJ$ such that for every $\textbf{J} \in \JJ^{(V)+(H)+(EV)}$, the moduli space $\MM_{0,k}^*(\underline{A};\WW,\textbf{J})$ is smooth of the expected dimension $\dim(E) + 2c_1(A) + k-3$. The evaluation map \eqref{eq:evaluationmapofW} is also smooth, and the restriction to of $\overline{ev}$ to any strata which contains a bubble factors via a manifold which has codimension at least two.
\end{theorem}

The Theorem is proved in following subsections.

\subsection{The graph construction} \label{subsec:graphs}
Let $\LHF$ be locally Hamiltonian fibration with fiber $(M,\omega)$. It is easy to transfer between the setting of a domain-dependent $\textbf{J}$ and a domain-independent $J$ on a larger fibration.

\begin{definition}
Denote $\tilde{E} = E \times \Sigma$, and $\tilde{\Omega}$ the non-degenerate 2-form defined as $\pi_1^* \Omega + \pi_2^* \omega_\Sigma$ where $\pi_1,\pi_2$ are projections onto the first and second factor. Let $\tilde{\pi} : \tilde{E} \to b$ be the Locally Hamiltonian fibration $\tilde{\pi}(x,z) = \pi(x)$. Each family $\textbf{J}$ gives rise to a unique almost complex structure $\tilde{\textbf{J}}$ on $\tilde{\pi}$ via the formula 
\begin{equation}
\tilde{\textbf{J}}(z, x) := j_z \oplus J_{z,x} \in Aut(T_z \Sigma \oplus T_{x} E). 
\end{equation}
\end{definition}
Evidently a smooth function $(b,u)$ is a solution of the parametrized del-bar equation \eqref{eq:verticalJholo} if and only if its graph $(b,\tilde{u})$, given by 
\begin{equation}
\tilde{u} = (z,u(z)) 
\end{equation}
is a $\tilde{J}$-holomorphic curve. 

\begin{definition}
Given any homology class $A \in H_2(M;\Z)$, we define $\tilde{A} \in H_2(\tilde{M};\Z)$ by

\begin{equation}
\tilde{A} := \iota_* A + [\Sigma \times \left\{pt\right\}]
\end{equation}
where $\iota : M \to \tilde{M}$ denotes the inclusion into the fiber (same for $E \to \tilde{E}$.)
\end{definition}

\subsection{The standard framework: variations and theme} \label{subsec:standardframework}
We concentrate here some general remarks about regularization of moduli spaces of holomorphic curves so that we can refer to them later on in the paper. \\

Let $(M,\omega)$ be a compact symplectic manifold. Let $\Sigma$ be a genus zero Riemann surface, possibly disconnected. Fix a real number $p > 2$ and integers $k \geq 1$, and $\ell$ sufficiently large. Let $W^{k,p}(\Sigma,M)$ denote the completion of $C^\infty(\Sigma,M)$ with respect to a distance function based on the Sobolev $W^{k,p}$-norm. This norm is defined as the sum of the $L^p$-norms of all derivatives up to order $k$. By the Sobolev embedding theorem maps $u \in W^{k,p}(\Sigma,M)$ are continuous. The corresponding metric on $C^\infty(\Sigma,M)$ can be defined, for example, by embedding $M$ into some high dimensional Euclidean space $\R^N$ and then using the Sobolev norm on the ambient space $W^{k,p}(\Sigma,\R^N)$. Since $\Sigma$ and $M$ are compact any two distance functions are compatible, and so as a topological vector space $W^{k,p}(\Sigma,M)$ is independent of choices. Until now all the almost complex structures we considered were smooth, however for regularization, it is better to consider moduli spaces of almost complex structures that are only $C^\ell$-smooth (because they are Banach manifolds). We will use a superscript $\ell$ to denote the smoothness of the almost complex structures e.g. $J \in \complexJ^\ell(\pi, \Omega)$ is a $C^\ell$-smooth section of $\complexJ(T^v E,\Omega)$ etc. \\

Let $\underline{A}$ be a homology decomposition, indexed by the set of connected components $\left\{\Sigma_v\right\}$ of $\Sigma$.

\begin{definition}
We define $\BB = \BB^{k,p}(\underline{A})$ be the connection component in $W^{k,p}(\Sigma,M)$ of maps such that $[u] = \underline{A}$, i.e., whose restriction to each connected component of the domain has homology $A_v$. 
\end{definition}
Similarly, we denote by $W_{\Sigma}^{k,p}(u^* TM)$ the completion of the space $\Omega^0(\Sigma,u^*TM)$ of smooth sections of the pullback bundle $u^* TM \to M$ with respect to the Sobolev $W^{k,p}$-norm. 

\begin{definition}
Fix some family of domain-dependent family of metrics $\textbf{g} = (g_z)_{z \in \Sigma}$. Let
\begin{equation}
exp_{z}: T_z M \to M
\end{equation}
denote the exponential map with respect to $(g_z)$. Given a curve $u \in W^{k,p}(\Sigma,M)$ and a vector field $\zeta \in W_{\Sigma}^{k,p}(u^* TM)$, let 
\begin{equation}
u_\zeta \in  W^{k,p}(\Sigma,M) 
\end{equation}
denote the curve
\begin{equation}
u_\zeta(z) := exp_{u(z)}(\zeta).
\end{equation} 
\end{definition}
\begin{lemma}
$\BB$ is a smooth separable Banach manifold, and the tangent space at $u \in C^\infty(\Sigma,M)$ is 
\begin{equation}
T_u W^{k,p}(\Sigma,M) = W^{k,p}(\Sigma,u^* TM).
\end{equation}
Around any $u \in \BB$, we can define a local coordinate charts by geodesic exponentiation, denoted 
\begin{equation}
\begin{split}
W_{S^2}^{k,p}(u^* TM) &\to \BB, \\
\zeta &\mapsto u_\zeta.
\end{split}
\end{equation} 
\end{lemma}
\begin{proof}
See \cite{MR0226681}.
\end{proof}
For a specific choice of $\textbf{J} = (J_z) \in \JJ^\ell_{S^2,\tau}(M,\omega)$, let $\textbf{g} = (g_z)$ be the corresponding family of metrics, where 
\begin{equation}
g_z = \frac{1}{2}(\omega(\cdot,J_z\cdot) + \omega(J_z\cdot,\cdot)).  
\end{equation}
We denote their Levi-Civita connections as
\begin{equation}
\nabla = ( {\nabla}_z )_{z \in \Sigma}.
\end{equation}
Finally, we let
\begin{equation}
\tilde{\nabla} = \left\{ \tilde{\nabla}_z \right\}_{z \in \Sigma} 
\end{equation}
denote the corresponding family of Hermitian connections defined by the formula
\begin{equation}
(\tilde{\nabla}_z)_X Y:= (\nabla_z)_X Y - \frac{1}{2} J_z ((\nabla_z)_X J_z) Y \: , \:  X , Y \in T^v E. 
\end{equation}
\begin{definition}
We define complex bundle isomorphism
\begin{equation}
\tilde{\Pi}_{u \to u_\zeta}: u^* T M \to u_\zeta^* TM
\end{equation}
obtained by parallel transport along the image of the geodesics
\begin{equation}
s \mapsto exp_{u(z)}(s \zeta(z)).
\end{equation}
using the $\Sigma$-family of hermitian connections $\tilde{\nabla}^{(\cdot)}$.
\end{definition}

We fix $u \in \BB$. Given any $\zeta \in T_u W^{k,p}(\Sigma,M)$, we obtain an identification
\begin{equation} \label{eq:connectiondefinesbundle}
\begin{split}
W_\Sigma^{k,p}(u^* TM \otimes_\C \Omega^{0,1} \Sigma) &\to W_\Sigma^{k,p}(u_\zeta^* TM \otimes_\C \Omega^{0,1} \Sigma) \\
\alpha &\mapsto \tilde{\Pi}_{u \to u_\zeta}  \alpha.
\end{split}
\end{equation}

\begin{lemma}
The isomorphisms \eqref{eq:connectiondefinesbundle} define a Banach bundle 
\begin{equation}
\mathfrak{E} := \mathfrak{E}^{k-1,p} \to \BB
\end{equation}
whose fibre at $u$ is the completion of the space $\Omega^{0,1}(\Sigma,u^*TM)$ of $\textbf{J}$-antilinear forms with values in $u^* TM$, i.e., 
\begin{equation}
W_\Sigma^{k,p}(u^* TM \otimes_{\textbf{J}} \Omega^{0,1} \Sigma).
\end{equation} \noproof
\end{lemma}
The del-bar operator $\delbar_{\textbf{J}}$ gives a section of this Banach bundle, denoted
\begin{equation}
\mathcal{S} = \mathcal{S}_{\textbf{J}}: \mathfrak{B} \to \mathfrak{E}.
\end{equation}

Our goal is to put a smooth structure on the zero-set $\mathcal{S}^{-1}(0)$. 

\begin{definition}
We define
\begin{equation}
\mathcal{F}_{u} :  W_\Sigma^{k,p}(u^* TM) \to W_\Sigma^{k,p}(u^* TM \otimes_{\textbf{J}} \Omega^{0,1} \Sigma)
\end{equation}
by setting 
\begin{equation}
\mathcal{F}_u = (\Pi_{u \to u_\zeta})^{-1} \delbar_{\textbf{J}}(u_\zeta). 
\end{equation}
This map is precisely the vertical part of the section $\mathcal{S}$ with respect to the trivialization determined by $\tilde{\nabla}$. We define the \textbf{linearization} to be the vertical differential
\begin{equation}
D_{u} := d\mathcal{F}_{u}(0) : W^{k,p}_\Sigma(u^* TM) \to W^{k,p}_\Sigma(u^* TM \otimes_{J} \Omega^{0,1} \Sigma ).
\end{equation}
When $u \in \SS^{-1}(0)$, then this has a simpler description, namely,
\begin{equation} \label{eq:linearization}
D_{u} : T_u \BB \to T_{(u,0)} \EE = T_u \BB  \oplus \EE_u \stackrel{\pi_2}{\longrightarrow} \EE_u. 
\end{equation}
\end{definition}
\begin{definition}
Our choice of almost complex structure $\textbf{J}$ is considered \textbf{regular} if $\mathcal{S}$ is transverse to the zero-section, or equivalently, if $D_u$, is onto for all $u \in \BB$. 
\end{definition}
Once we choose such a $\textbf{J}$, we can use the usual elliptic regularity arguments combined with a Taubes trick (see e.g., \cite[Proposition 3.1.10]{MR2954391} and \cite[p.~54--55]{MR2954391} for a reference) to deduce that every element in $\mathcal{S}^{-1}(0)$ is actually $C^\infty$-smooth. \\

Later in the paper (see the discussion in subsection \ref{subsec:transversality}), we will often in a \textbf{parametrized setting}: that means that the base of our Banach bundle $\BB$ is replaced by a product $Q \times \BB  $ where $Q$ is some manifold of extra parameters and $\textbf{J} = (J_{q,z})$ now depends both on the point in $z \in S^2$ and a parameter $q \in Q$. We will use the notation $\textbf{J}^z = (J_{q,z})_{q \in Q}$ or $\textbf{J}^q = (J_{q,z})_{z \in S^2}$ when we want to fix one of the parameters. Similarly for the metric, the connection etc. All the previous discussion carries through (we just define everything by restricting to a $q = const$ slice.) In particular, there is a parametrized del-bar operator $\delbar_{\textbf{J}}$ and we can consider a parametrized section $\mathcal{S} :Q \times \mathfrak{B} \to \mathfrak{E}$ as before and a parametrized parallel transport operator as before. The analogue of equation \eqref{eq:linearization} is now
\begin{equation} \label{eq:parametrizedlinearization}
D_{q,u} : T_q Q \oplus T_u \BB  \to T_{(q,u,0)} \EE = T_q Q \oplus T_u \BB  \oplus \EE_{q,u} \stackrel{\pi_2}{\longrightarrow} \EE_{q,u}. 
\end{equation}
when $u \in \mathcal{S}^{-1}(q, 0)$. See \cite{MR2218350} and \cite{MR2680275}. 
\begin{definition}
Our choice of parametrized almost complex structure $\textbf{J}$ is considered \textbf{regular} if $\mathcal{S}$ is transverse to the zero-section, or equivalently, if $D_{q,u}$, is onto for all $u \in \BB$. Yet another formulation is that the cokernel of $D^q_{u}$ (the linearization in a fixed $q$-slice) is covered by the variation in $T_q Q$.
\end{definition}
Now let $\LHF$ be a locally Hamiltonian fibration. To obtain a Fredholm description for moduli spaces of vertical maps $(b,u)$, we simply reduce to the parametrized setting by employing Lemma \ref{lem:trivialization} and performing all the functional analysis in a local trivialization $\tilde{\chi}$. After we passed to $C^\infty$-moduli spaces, we proceed to use elliptic regularity to ensure that the transition maps between different trivializations are smooth. We often want to distinguish between the base parameter and the other parameters (which come from some moduli space of domains), so we denote $J_{b,q,z}$ and $D_{b,q,u}$ etc.

\subsection{Horizontal sections}
Let $A \in H_2(M;\Z)$. We shall denote the moduli space of $\textbf{J}$-holomorphic curves $(b,u)$ that 
represent the class $A$ by 
\begin{equation} 
\MM(A,\Sigma,E;\textbf{J}) = \left\{(b,u) \: \big| \: b \in B \: , \: u \text{ satisfies } \eqref{eq:verticalJholo}, \text{ and }[u]=A \right\}.
\end{equation}

To understand the moduli space, we must look at the linearized operator 
\begin{equation} 
D_{(b,u)} : \Omega^0(\Sigma,u^* V_b) \to \Omega^{0,1}(\Sigma,u^* V_b)
\end{equation}

\begin{definition}
We say that $\textbf{J}$ is \textbf{(H)-parametric regular} if $D_{(b,u)}$ is onto for every homology class, every $(b,u) \in \MM(A,\Sigma,E;\textbf{J})$. We denote the subspace of all such almost complex structures as $\JJ^{(H)} \subset \JJ$. 
\end{definition}

\begin{remark}
One possible interpretation of the solutions in $ MM(A,\Sigma,E;\textbf{J})$ is through the graph construction: we think of $\tilde{u} : \Sigma \to \tilde{E}$ as a holomorphic section of the fibration $\tilde{E} \to B$ with respect to $\tilde{J}$. There is an operator $D_{(b,\tilde{u})}$ and a related Fredholm problem with $ind(D_{(b,\tilde{u})}) = ind(D_{(b,u)}) + 6 $, and the higher dimensional kernel of $D_{(b,\tilde{u})}$ accounts precisely for the tangent space of the orbit of $u$ under the reparametrization group $G = PSL_2(\C)$. 
See \cite[Remark 6.7.5]{MR2954391} for details. 
\end{remark}

\begin{proposition} \label{prop:Hregulariscomeager}
The set $\JJ^{(H)}$ is comeager. Moreover, if $\textbf{J} \in \JJ^{(H)}$ then $\MM(A,\Sigma,E;\textbf{J})$ is a smooth manifold of dimension $\dim(E) + 2c_1(A)$.
\end{proposition}

The proof of this proposition follows automatically from the considerations in subsection \ref{subsec:standardframework} once we show:
\begin{itemize}
\item
That we can construct a universal moduli space which is a Banach manifold,
\item
The projection onto the space of allowable perturbations is a submersion. 
\end{itemize}
The proof of this two facts would take the rest the subsection. Notations ($p,k,\ell,\BB,\EE$ etc) are the same as in \ref{subsec:standardframework} or the obvious modifications.  

\begin{definition} We define the \textbf{(H)-universal moduli space}
\begin{equation} \label{eq:universalmodulispace}
\MM(A,\Sigma,E;\complexJ^\ell) = \left\{(b,u,J) \: \big| \: b \in B \: , \: \textbf{J} \in \complexJ^\ell \text{ as in } \eqref{eq:perturbationsofJ}, \: u \in \MM(A,\Sigma,M;\textbf{J}) \right\}.
\end{equation}
\end{definition}

\begin{proposition}
The moduli space \eqref{eq:universalmodulispace} is a $C^{\ell-k}$-smooth seperable Banach submanifold of $B \times W^{k,p}(\Sigma, M) \times \complexJ^\ell$. 
\end{proposition}
\begin{proof}
When $A=0$ this follows from \cite[Lemma 6.7.6]{MR2954391}. When $A \neq 0$, the proof is a small modification of \cite[Proposition 6.7.7]{MR2954391}: there is a $C^{\ell-k}$-smooth Banach bundle
\begin{equation}
\mathfrak{E}^{k-1,p} \to B \times \mathfrak{B}^{k,p} \times \complexJ^\ell
\end{equation}
whose fiber over $(b,u, J)$ is the space 
\begin{equation}
E^{k-1,p}_{(b,u, J)} = W^{k-1,p}(\Sigma,u^* TM \otimes_{J^b} \Omega^{0,1} \Sigma)
\end{equation}
of $J^b$-antilinear one-forms on $\Sigma$ of class $W^{k-1,p}$ with values in the bundle $u^* TM$. The delbar operator gives a section (also $C^{\ell-k}$-smooth)
\begin{equation}
\begin{split}
&\mathcal{S} : B \times \mathfrak{B}^{k,p} \times \complexJ^\ell \to \mathfrak{E}^{k-1,p}, \\
&\mathcal{S}(b,u,J) = \delbar_{J^b}(u).
\end{split}
\end{equation}
We need to prove the moduli space is cut out transversely. As we remarked above, it is enough to consider the linearization 
\begin{equation}
\begin{split}
& D\mathcal{S}_{(b,u,J)} : T_b B\times W^{k,p}(\Sigma,u^* TM) \times T_J \complexJ^\ell, \\
& (\zeta,\left\{Y_{(b,z)}\right\}) \mapsto D_u \zeta + \frac{1}{2} Y_{(b,z)}(u) du \circ j_\Sigma,
\end{split}
\end{equation}
and show that it is surjective when $\mathcal{S}(b,u,J)=0$. Since $D_u$ is Fredholm the operator $D\mathcal{S}_{(b,u,J)}$ has a closed image and it suffices to prove that its image is dense whenever $\delbar = 0$. We prove this first in the case $k = 1$. If the image is not dense then, by the Hahn-Banach theorem, there exists $q > 1$ such that $\frac{1}{p} + \frac{1}{q} = 1$ and a nonzero section 
\begin{equation}
\eta \in L^q(B \times \Sigma,T^*\Sigma \otimes_{J^b} u^* TM)
\end{equation}
which annihilates the image of this operator. This means that 
\begin{equation} \label{eq:3.2.1}
\int_{\Sigma} \langle \eta, D_u \zeta \rangle dvol_{\Sigma } =  0.
\end{equation}
for every $u \in W^{1,p}(\Sigma,u^*TM)$ and,
\begin{equation}  \label{eq:3.2.2}
\int_{\Sigma} \langle \eta, Y_{(b,z)} du \circ j_\Sigma \rangle dvol_{\Sigma } =  0,
\end{equation}
for every $\left\{Y_{(b,z)}\right\} \in T_J \complexJ^\ell$. Then by \eqref{eq:3.2.1} and \cite[Proposition 3.1.10]{MR2954391} $\eta \in W^{1,p}$ and
\begin{equation} \label{eq:Dueta}
D_u^* \eta = 0.
\end{equation}
Since the homology class in non-zero the function $u$ cannot be constant. The critical point analysis of \cite[Lemma 2.4.1]{MR2954391} carries over and shows that $du$ can only vanish at finitely many points. From \eqref{eq:Dueta}, 
\begin{equation} 
\Delta \eta + (\text{ lower order terms })= D_u D_u^* \eta = 0.
\end{equation} 
As in \cite[Remark 3.2.3]{MR2954391}, we note that by Aronszajn's theorem (see \cite[Theorem 2.3.4]{MR2954391}) if $\eta$ vanishes on an open subset, it vanishes everywhere. Now, we show that $\eta(z) = 0$ for every $z \in \Sigma^{cpt}$ such that $du(z) \neq 0$. Suppose otherwise that $du(z_0) \neq 0$ and $\eta(z_0) \neq 0$ for some $z_0 \in \Sigma^{cpt}$. Then, by \cite[Lemma 3.2.2]{MR2954391}, there exists an infinitesimal almost complex structure 
\begin{equation}
\left\{Y_{(b,z)}\right\} \in T_J \complexJ 
\end{equation}
such that the map 
\begin{equation}
z \mapsto \langle \eta(z) , Y_z(u(z)) du(z) \circ j_\Sigma \rangle
\end{equation}
is positive at $z_0$ and hence in some neighbourhood $U \subseteq \Sigma^{cpt}$ of $z_0$. Choose any smooth cutoff function $\rho : \Sigma \to [0,1]$ with support in $U$ and such that $\rho(z_0) = 1$. Then the infinitesimal almost complex structure $z \mapsto \rho(z)Y_z$ violates the condition \eqref{eq:3.2.2}. This contradiction shows that $\eta$ vanishes almost everywhere and hence $\eta \equiv 0$. Thus we have proved that the linearized operator is surjective and so the universal moduli space is a separable $C^{\ell-1}$-Banach manifold. To prove surjectivity for general $k$ let $\eta \in W^{k-1,p}(\Sigma \times B, \wedge^{0,1} \otimes_{J^b} u^*TM)$ be given. Then, by surjectivity for $k = 1$, there exists a pair 
\begin{equation}
(\zeta,Y) \in W^{1,p}(\Sigma,u^* TM) \times C^\ell(M,End(TM,J,\omega))
\end{equation}
such that 
\begin{equation}
D\mathcal{S}_{(b,u,J)} = \zeta
\end{equation}
and, by elliptic regularity, $\zeta \in W^{k,p}$ (\cite[Theorem C.2.3]{MR2954391}). Hence $D\mathcal{S}_{(b,u,J)}$
is onto for every triple $(b,u,J)$. Because $D\mathcal{S}$ is a Fredholm operator it follows from \cite[Lemma A.3.6]{MR2954391} that the operator $D\mathcal{S}_{(b,u,J)}$ has a right inverse. Hence it follows from the infinite dimensional implicit function theorem that the space $\MM(A,\Sigma,E;\complexJ^\ell)$ is a $C^{\ell-k}$-Banach submanifold of $B \times \mathfrak{B}^{k,p} \times \complexJ^\ell$. Since $B \times \mathfrak{B}^{k,p} \times \complexJ^\ell$ is separable so is $\MM(A,\Sigma,\pi;\complexJ^\ell)$. 
\end{proof}

\subsection{Controlling bubbles}
Clearly, the Gromov compactification involves $J_{b,z,q}$-holorphic spheres in $E_b$ for a fixed $z \in \C$. We can picture them as vertical $\tilde{\textbf{J}}$-holomorphic spheres in the fibration $\tilde{\pi} : \tilde{E} \to B \times S^2$. 

\begin{definition}
The moduli space of simple vertical spheres representing the class $A$ is denoted by

\begin{equation} \label{eq:verticalmoduli}
\MM^{Vert,*}(A,\Sigma;\pi,\Omega;\textbf{J}) = \bigcup_{b \in B} \bigcup_{z \in \Sigma} \left\{(b,z)\right\} \times \MM^*(A,\Sigma,E_b;\textbf{J}^{b})
\end{equation}
\end{definition}

\begin{definition}
We say that $\textbf{J}$ is \textbf{parametric (V)-regular} if 
\begin{equation} 
D \mathcal{S}|_{b,z,u} \text{ is onto}
\end{equation}
for every homology class, every $z \in \Sigma$ and every $(b,v)$ where $b \in B$ and $v : \Sigma \to E_b$ is a \emph{simple} $\textbf{J}^{b,z}$-holomorphic curve. 
\end{definition}

This is equivalent to $\tilde{v}$ being parametric regular because if we consider the two corresponding linearized operators, then they are both Fredholm of index $\dim(E) + 2 + 2c_1(A)$, and they have the same kernel and have isomorphic cokernels (see the discussion in \cite[p.~177]{MR2954391}). \\

A basic observation is unlike (H)-regularity, it is impossible for $\tilde{\textbf{J}}$ to be parametric (V)-regular if $D \mathcal{S}|_{b,u}$ (which is independent of $z$) is not onto the original original $J^{base}$. 

\begin{definition} \label{def:good}
A domain-indpendent almost complex structure $J^{base}$ is called \textbf{(V)-good} if $D \mathcal{S}|_{b,u}$ is onto for every homology class $A \in H_2(M;\Z)$. 
\end{definition}

\begin{remark}
We can show that such structures exist using the the usual method. However, we note that $J^{base}$ is a non-generic, even among domain-independent almost complex structures (because of the condition that $(J^{base})\big|_b$ is (V)-good in every fiber)  
\end{remark}

However, assuming $J^{base}$ is (V)-good, we can apply Aronszajn's Theorem once more and the proof of Proposition 6.7.9 in \cite{MR2954391} immediately generalizes to give

\begin{proposition} \label{prop:Vregulariscomeager}
The subset $\JJ^{(V)}$ of (V)-regular almost complex structures is comeager in $\JJ$. Moreover, if $\textbf{J} \in \JJ^{(V)}$ then \eqref{eq:verticalmoduli} is a smooth oriented manifold of dimension 
\begin{equation}
\dim(E) + 2 + 2c_1(A)
\end{equation} \noproof
\end{proposition}

\subsection{Evaluation}
To ensure that we can make fiber product constructions, we must be able to control the evaluation map of all stratas of the bordification -- and not just on the primary moduli spaces. We briefly summarize some key features from the proof of Theorem 6.7.1 in \cite{MR2954391} and investigate how to adapt them to our setting. \\

Let $(M,\omega)$ be a symplectic manifold and let $T = (T,E,\Lambda)$ be a k-labelled tree with special vertex $0 \in T$ and $\left\{A_\alpha\right\}_{\alpha \in T}$ a homology decomposition. As in the previous subsection \ref{subsec:graphs}, it is convenient to think in terms of graphs of maps: We denote 
\begin{equation}
\tilde{A}_\alpha := \left\{\begin{array}{lr}
A_0  + [S^2 \times \left\{pt\right\}] \: , \: &\alpha =0 \\
A_\alpha \: , \: &\alpha \neq 0\\
\end{array}\right.
\end{equation}
Then 
\begin{equation}
\pi_* \tilde{A}_\alpha = \left\{\begin{array}{lr}
[S^2] \: , \: &\alpha =0 \\
0 \: , \: &\alpha \neq 0\\
\end{array}\right.
\end{equation}
where $\pi  : \tilde{M} \to M$ is the projection onto the first component. So for each $\alpha \in T \setminus \left\{0\right\}$, there is a one-to-one correspondence between $\tilde{J}$-holomorphic spheres $\tilde{u}_\alpha$ in the class $\tilde{A}_\alpha$ and pairs $(z_\alpha,u_\alpha)$, where $u_\alpha$ is a $J_{z_\alpha}$-holomorphic sphere in the class $A_\alpha$. Likewise, there is a one-to-one correspondence between $\tilde{J}$-holomorphic spheres $\tilde{u}_0$ in the class $\tilde{A}_0$ such 
that $\pi \circ \tilde{u}_0 = id$ and solutions $u_0$ of the domain-dependent del-bar equation in the class $A_0$. \\

For each connected component, there is a moduli space of \textbf{simple stable maps}:
\begin{equation}
\MM^*_{0,T}(\tilde{M},\underline{\tilde{A}};\tilde{J}) = \widetilde{\MM}^*_{0,T}(\tilde{M},\underline{\tilde{A}};\tilde{J}) / G_T = \bigcup_{b \in B} \widetilde{\MM}^*_{0,T}(\tilde{E_b},\underline{\tilde{A}};\tilde{J_b})/G_T.
\end{equation}

where we always assume $\pi \circ \tilde{u}_0 = id$ and with the reparametrization group $G_T$ acting on the main component as the identity. See \cite[p.~190--191]{MR2954391}. \\

Any submanifold $\WW \subset (S^2)^k$ (disjoint from the fat diagonal), defines a subset 
\begin{equation}
\MM^*_{0,T}(\tilde{M},\underline{\tilde{A}};\WW;\tilde{J}) := \tilde{\MM}^*_{0,T}(\tilde{M},\underline{\tilde{A}};\WW;\tilde{J}) /G_T.
\end{equation}

consisting of all equivalence classes $[\tilde{\u},\z] \in  \MM^*_{0,T}(\tilde{M},\underline{\tilde{A}};\tilde{J})$ of stable maps such that 
\begin{equation} \label{eq:6.7.10}
\left(\ldots \:,\:\pi(\tilde{u}_{\alpha_i}(z_i)) \: , \: \ldots\right) \in \WW. 
\end{equation}

This would allow us to consider pearls (by restricting the forgetful map to a proper submanifold with boundary and corners defined by certain cross-ratios) and also to fix marked points, if needed. \\

The moduli space can be described as the preimage of the evaluation map in the following way: As in \cite[Section 6.2]{MR2954391}, there is a moduli space 
\begin{equation}
\MM^*(\tilde{M}, \underline{\tilde{A}}; \tilde{J})
\end{equation}
of \textbf{simple tuples} $\tilde{u} = \left\{\tilde{u}_\alpha\right\}$ of $\tilde{J}$-holomorphic spheres representing the classes $\tilde{A}_\alpha$. We define 
\begin{equation}
Z(T) \subset (S^2)^E \times (S^2)^k
\end{equation}
as the set of all $\z = (\left\{z_{\alpha \beta}\right\}_{\alpha E \beta} ,\left\{z_i\right\}_{1 \leq i \leq k})$ such that the points $z_{\alpha \beta}$ for $\alpha E \beta$ and $z_i \in S^2$ for $\alpha_i = \alpha$ pairwise distinct for every $\alpha \in T$. 

\begin{definition}
We define the \textbf{stablization-evaluation map}

\begin{equation} \label{eq:se}
\tilde{se} := \tilde{ev}^E \times \pi : \MM^*(\tilde{M}, \underline{\tilde{A}}; \tilde{J}) \times Z(T) \to \tilde{M}^E \times (S^2)^k
\end{equation}
 by 

\begin{equation}
\tilde{ev}^E (\tilde{u},\z) = (\tilde{u}_\alpha)(z_{\alpha \beta})_{\alpha E \beta}\: , \: \pi(u,\z) = \tilde{\pi} ( \tilde{u}_{\alpha_i}(z_i))
\end{equation}
\end{definition}

The moduli space of simple stable tuples that satisfy \eqref{eq:6.7.10} is the preimage of $\tilde{\Delta}^E \times \WW$ under $\tilde{ev} \times \pi$. \\

As long as we are considering \emph{connected domains} the generalization to locally Hamiltonian fibrations is immediate. 

\begin{definition}
A smooth family of $\Omega$-tame almost complex structures is \textbf{parametric (EV)-regular} if the stablization-evaluation map \eqref{eq:se} is transverse to $\tilde{\Delta}^E \times \WW$ for every k-labelled tree with special vertex $0$, and every collection of homology classes $\left\{A_\alpha\right\}_{\alpha \in T}$. 
\end{definition}

The subset of almost complex structures that are parametric (V), (H) and (EV)-regular is denoted 
\begin{equation}
\JJ^{(V) + (H) + (EV)} \subset \JJ.
\end{equation}

\begin{theorem} \label{thm:6.7.11}
Let $\WW$ be any manifold with boundary and corners, disjoint from the fat diagonal. 
\begin{itemize}
\item[(a)]
If $\textbf{J} \in \JJ^{(V) + (H) + (EV)}(S^2;E,\Omega;\WW)$, then for every tree with special vertex $0$, and every collection of homology classes $\left\{A_\alpha\right\}_{\alpha \in T}$ of spherical homology classes in $H_2(M;\Z)$, the moduli space $\MM^*_{0,T}(\tilde{E},\underline{\tilde{A}};\WW;\tilde{J}) $ is a smooth oriented manifold with boundary and corners of the expected dimension
\begin{equation}
\dim(\MM^*_{0,T}(\tilde{E},\underline{\tilde{A}};\WW;\tilde{J}) ) = \dim(E) + 2c_1(A) + 2k - \codim(\WW) -2e(T)
\end{equation}
\item[(b)]
$\JJ^{(V) + (H) + (EV)}$ is comeager. 
\end{itemize}
\end{theorem}

\begin{proof}
This is a essentially identical to the proof of \cite[Theorem 6.7.11]{MR2954391}, but we give a brief overview since some of technical details play an important role in future discussion. 
\begin{description}
\item[Step 1] \emph{The universal moduli space}
\begin{equation}
\MM^*(\tilde{E},\left\{\tilde{A}_\alpha\right\};\JJ^\ell) 
\end{equation}
\emph{is a Banach manifold}. 
\begin{proof}
See the proof of \cite[Proposition 6.2.7]{MR2954391}. 
\end{proof}

\item[Step 2] \emph{Let $T = (T,E,\Lambda)$ be a k-labelled forest, suppose that the edge evaluation map }
\begin{equation}
\tilde{ev}^{E,\JJ} : \MM^*(\tilde{E},\left\{\tilde{A}_\alpha\right\};\JJ^\ell) \times Z(T) \to \tilde{E}^E 
\end{equation}                            
\emph{is transverse to $\tilde{\Delta}^E$, and consider the universal moduli space defined as} 
\begin{equation}
\tilde{\MM}^*_{0,T}(\tilde{E},\left\{\tilde{A}_\alpha\right\};\JJ^\ell)  := (\tilde{ev}^{E,\JJ})^{-1}(\tilde{\Delta}^E).
\end{equation}
\emph{Then each tuple }$\w = \left\{w_i\right\}$\emph{ of pairwise distinct points on} $S^2$ \emph{is a regular value of the universal projection}
\begin{equation}
\pi^{\JJ} :\tilde{\MM}^*_{0,T}(\tilde{E},\left\{\tilde{A}_\alpha\right\};\JJ^\ell) \to (S^2)^{k}.
\end{equation}
\begin{proof}
This is the same proof as Step 2 in \cite[Proposition 6.2.8]{MR2954391}, except we use Aronszajn. 


\end{proof}
\item[Step 3] \emph{For each k-labelled forest $T = (T,E,\Lambda)$ the edge evaluation map }
\begin{equation}
\tilde{ev}^{E,\JJ} : \MM^*(\tilde{E},\left\{\tilde{A}_\alpha\right\};\JJ^\ell) \times Z(T) \to \tilde{E}^E 
\end{equation}
\emph{is transverse to $\tilde{\Delta}^E$}. 
\begin{proof}
This is the same inductive process as \cite[Proposition 6.2.8]{MR2954391}, where we induct over the number of edges in a forest. At each stage we receive an input forest $T'$, and then remove an edge and add two marked points in its place to for the new forest $T$. We then show the corresponding evaluation map 
\begin{equation}
ev_{\alpha \beta} :  \MM_{0,T'}^*(\tilde{E},\left\{\tilde{A}_\alpha\right\};\JJ^\ell) \to \tilde{E} \times \tilde{E}
\end{equation}
is transverse to the diagonal. By Step 2, the composition of $\pi_{\alpha \beta}$ with the projection to $\Delta_{S^2} \subset S^2 \times S^2$ is transverse to the diagonal. Hence, by Exercise 6.3.2, it suffice to show that the map 
\begin{equation}
\pi_E \circ ev_{\alpha \beta} :  \MM_{0,T'}^*(\tilde{E},\left\{\tilde{A}_\alpha\right\};\JJ^\ell) : \pi_{\alpha \beta}^{-1}(\Delta_{S^2}) \to E \times E
\end{equation}
is transverse to the diagonal $\Delta_E \subset E \times E$. This follows from the same argument in the proof of Theorem 6.3.1.
\end{proof}
\item[Step 4] \emph{We prove (b)}.

By Proposition \ref{prop:Hregulariscomeager} and \ref{prop:Vregulariscomeager}, the set of all almost complex structures that satisfy (H) and (V) is comeager. Now let $T$ and $\underline{A}$ as in the statement of the Theorem. By Step 2 and 3, we can form the universal moduli space
\begin{equation}
(\pi^{\JJ})^{-1}(\WW) = \tilde{\MM}^*_{0,T}(\tilde{\underline{A}};\tilde{\pi},\tilde{\Omega};\WW;\JJ^\ell)
\end{equation}
It has a universal projection 
\begin{equation}
\pi_{T,\underline{A}} : \tilde{\MM}^*_{0,T}(\tilde{\underline{A}};\tilde{\pi},\tilde{\Omega};\WW;\JJ^\ell) \to \JJ^\ell
\end{equation}
This is Fredholm map and an almost complex structure is a regular value for $\pi_{T,\underline{A}}$ if and only if it satisfies (EV) for $T$ and $\underline{A}$. The result now follows from a standard Sard-Smale and a Taubes argument. We proceed to take the intersection over all possible $T$ and $\underline{A}$. 
\end{description}
\end{proof}

\begin{proof}[Proof of Theorem \ref{thm:pseudocycle}]
Follows immediately from Theorem \ref{thm:6.7.11} and the standard procedure for covering the image of an arbitrary stable map by a simple one (see \cite[Proposition 6.1.2]{MR2954391}.) We just need to observe that every stable map in the compactification is modelled after a tree with more then one edge, and that because of fiber monotonicity, every sphere bubble must have a nonnegative Chern number.
\end{proof}
The generalization to families of nodal curves over some base $\QQ$ with a fixed \emph{connected} combinatorial type is immediate. 

\subsection{Collision detection} \label{subsec:colldetect}

Unfortunately, as we have mentioned in the introduction, we will have to deal with disconnected domains (or families of a fixed disconnected combinatorial type.) \\

What happens when we remove the assumption that $T$ is connected from Theorem \ref{thm:6.7.11}? Consider a forest 
\begin{equation}
T = \bigcup_j T_j, 
\end{equation}
That is, we want to consider maps whose domain is a disconnected nodal Riemann surface. Then 

\begin{equation}
Z(T) := \bigtimes_j Z(T_j) \subset (S^2)^E \times (S^2)^k
\end{equation}

and we can still consider the stablization-evaluation map $\tilde{se}$, which is just the product of \eqref{eq:se} for every connected tree $T_j$. Note that the definition of a pearl tree is global in nature and not given locally at each pearl (i.e., the involution must act on all nodal curves in the tree at once), so one small way in which we deviate from McDuff-Salamon is that we want to consider manifolds $\WW \subset (S^2)^k$ with $k = \sum k_i$. That is a minute change and causes no problems. In particular, the analogue of Theorem \ref{thm:6.7.11} holds \emph{for each connected component separately}. However, we would like to obtain a result like Theorem \ref{thm:pseudocycle} for this setting. Here we encounter a much more serious obstacle, one which plagues all PSS-type constructions and moduli spaces. 

From an analytic point of view, what fails is that we might have energy concentration near the marked points, and we have no way of ensuring that two bubbles that belong to different main components do not map to the same image (even though the expected codimension of such phenomena is high.) Since \cite[Lemma 3.4.3]{MR2954391} is no longer valid, we lose control over the evaluation map, which no longer has to be a pseudocycle. 

\begin{remark}
The usual remedy in the monotone case, which is to cover the image with a simple map (Proposition 6.1.2 in \cite{MR2954391}) does not work here because the different main components might have different incidence conditions. 
\end{remark}
The first observation is that as long as the different families modelled after the individual connected components remain in different fibers, we can still apply \cite[Lemma 3.4.3]{MR2954391}, so the problem is local in nature (in the fibration $E$) and only occurs at the fat diagonal. \\

This requires us to introduce a relative of the graph pseudocycle (which usually appears as part of the proof of the gluing Theorem for holomorphic curves). For simplicity, assume that we are only dealing with two connected components and a fixed target symplectic manifold $(M,\omega)$. The general case is a straightforward generalization, where we replace the pseudocycle with a parametrized version as we did before, and allow the number of components to be more then two. \\

Fix marked points $w^0 \in S^2$ and $w^\infty \in S^2$. 

\begin{definition}
A pair $(J^0,J^\infty) \in \JJ_{S^2}(M ,\omega)  \times  \JJ_{S^2}(M ,\omega)$ is called \textbf{regular} for $(A^0,A^\infty) \in H^S_2(M;\Z)$ if:
\begin{itemize}
\item
$J^0$ is regular for $A_0$, and $J^\infty$ is regular for $A_\infty$.
\item
The evaluation map 
\begin{equation} \label{eq:pseudocycleofapair}
ev : \MM(A^0;J^0) \times \MM(A^0;J^0) \to M \times M
\end{equation}
given by $ev(u^0,u^\infty) = (u^0(0),u^\infty(\infty))$ is transverse to the diagonal. 
\end{itemize}
We denote the set of all regular pairs as $\JJ^{(H)+(V)+(EV)}(S^2 ,A^{0,\infty})$.
\end{definition}


We fix $k \geq 3$. Let $\S$ denote the set of all splitting $S = (S^0 , S^\infty)$ of the index set $\left\{1,\ldots,k\right\}$ such that
$k^0 = |S^0| \geq 2$ and  $k^\infty = |S^\infty| \geq 2$. Let
\begin{equation}
\sigma^0 : \left\{1,\ldots,k^0\right\} \to \left\{1,\ldots,k\right\} \: , \: \sigma^\infty : \left\{1,\ldots,k^\infty\right\} \to \left\{1,\ldots,k\right\} 
\end{equation}
be te unique order preserving maps such that $im(\sigma^0) = S^0$ and $im(\sigma^\infty) = S^\infty$. Consider the moduli space
\begin{equation}
\MM_{0,S}(A^{0,\infty};\WW;J^{0,\infty})
\end{equation}
of all tuples 
\begin{equation}
\end{equation}
where $(u^0,u^\infty) \in \MM(A^{0,\infty};J^{0,\infty})$; $\z^0  =\left\{z_i^0\right\}_{1 \leq i \leq k^0}$ is a $k^0$-tuple of pairwise distinct marked points on $S^2 \backslash \left\{0\right\}$, $\z^\infty  =\left\{z_i^\infty\right\}_{1 \leq i \leq k^\infty}$ is a $k^\infty$-tuple of pairwise distinct marked points on $S^2 \backslash \left\{\infty\right\}$ and their projection $\z = (\z^0,\z^\infty) \in \WW$. This space carries an evaluation map 
\begin{equation}
ev_{A^0,A^\infty,S} : \MM(A^{0,\infty};\WW;J^{0,\infty}) \to M^k
\end{equation}
given by 
\begin{equation}
\begin{split}
ev_S(u^0,u^\infty,\z^0,\z^\infty) &:= (x_1,\ldots,x_k) ,\\
x_{\sigma^0(i)} &= u^0(z_i^0) \: , \: i = 1,\ldots,k^0, \\
x_{\sigma^\infty(i)} &= u^\infty(z_i^\infty) \: , \: i = 1,\ldots,k^\infty. 
\end{split}
\end{equation}
The first steps in the proof of Theorem 10.8.3 show that the set of almost complex structures for which the evaluation map is a pseudocycle (see \cite[p.~408--410]{MR2954391}) is comeager. This is a generalization of the proof of \cite[Theorem 6.7.11]{MR2954391} obtained by changing the meaning of ''simple map" to fit a pair, and by a careful study of the structure of the limit set of the graphs of $\tilde{u}^0$ and $\tilde{u}^\infty$ and how they can interact. \\

This leads to the following definition (here the dependence on the entire tuple $\underline{b}$ is crucial). 

\begin{definition}
Let $\LHF$ be a locally Hamiltonian fibration. We say that $\textbf{J} = (J_{\underline{b},z,q})$ is \textbf{(B)-regular} if for every subsequence of indices $(i_1,\ldots,i_\ell)$ and $b \in B$ such that $b_{i_k} = b$ for every $1 \leq k \leq \ell$ the resulting evaluation map for the tuple is a pseudocycle. 
\end{definition}

Then we can summarize the discussion above as 

\begin{lemma}
The set $\JJ^{(H)+(V)+(EV)+(B)} \subset \JJ^{(H)+(V)+(EV)}$ is comeager.
\end{lemma}
\begin{proof}
We stratify the fat diagonal by sub-diagonal which correspond to the subsequences of indices $(i_1,\ldots,i_\ell)$ of components which map to the same fiber. Now we apply a parametrized version of the discussion in \cite[p.~408--410]{MR2954391} to each trivializing neighbourhood of the LHF, starting with the most degenerate sub-diagonal (i.e., the ordinary diagonal where all components are the same) and inducting on $\ell$.
\end{proof}

\section{Computation (IV): the main term $\tilde{\mu}_{2F}^3$} \label{sec:compute4}

In this section, we carry out the construction of a specific minimal $A_\infty$-pearl structure in which we can compute $\tilde{\mu}_{2F}^3$ by relating it to the $\tilde{\mu}^3$ we computed in the mapping torus.


\subsection{A relative Morse model for the complement} \label{subsec:relativemorsemodel}
Given a complete flag
\begin{equation}
\left\{pt\right\} \subset \ell \subset H \subset \P^3
\end{equation}
we decompose $\P^{3}$ in the following way: \vspace{0.5em}
\begin{enumerate}
\item
We choose a Morse function on $\ell$ with a unique minima at $\left\{pt\right\}$ and a unique maxima. This gives the standard handlebody structure 
\begin{equation}
\P^{1} = h_0 \cup h_2
\end{equation}
with a unique 0-handle and 2-handle. \vspace{0.5em}
\item
We decompose $H = N \cup h_4$ as where $N$ is the tubular neighborhood of $\ell$ (a $D^2$-bundle over $\P^{1}$ with Euler class $+1$) and $h_4 \iso B^4$ is its complement. 
Then $H$ inherits a handlebody decomposition 
\begin{equation}
\P^{2} = h_0 \cup h_2 \cup h_4 
\end{equation}
by ''thickening up" the handles of $\ell$. \vspace{0.5em}
\item
We write $\P^{3} = N' \cup h_6$ where $N'$ is the tubular neighborhood $H$ and $h_6 \iso B^6$ is the complement. As before, this gives a handle decomposition with a unique handle at every even index $i \in \left\{0,2,4,6\right\}$. 
\end{enumerate}
We take 
\begin{equation}
\hat{f}^{\P^3} : \P^3 \to \R 
\end{equation}
to be a Morse function inducing the ''upside down" version of this decomposition (equivalently, we take $\hat{f}^{\P^3}$ to be the Morse function that correspondes to this decomposition, but with the opposite sign) and write $p_i$ for the critical point corresponding to $h_i$. \\

We would like construct the Morse model for the 7-dimensional mapping torus via an relative handle decomposition with respect to the $2A_5$-curve $C_{max} \subset \P^3$ that appeared in Definition \ref{def:definitionofmaximallydegeneate2A5}. Each one of the irreducible component $L_i$, $i=1,2$ is a sphere which meets a generic hyperplane section at a single point (because such a hyperplane section gives a $(1,1)$-curve in the quadric.) Because $L_3$ is a rational $(1,3)$-curve it meets a generic hyperplane at four distinct points. By choosing a generic flag and ''pushing" $L_i$ off handles (this is a routine process in handlebody theory; For example, one can accomplish this by constructing a vector field which is directed away from the co-core of a handle in the radial direction and then integrating it to an ambient isotopy. See e.g., Propositions \cite[4.2.7]{MR1707327} and \cite[4.2.9]{MR1707327} for more arguments of the same flavour), we may assume without loss of generality that:  \vspace{0.5em}
\begin{itemize}
\item
Each $L_i$ lies in the sub-handlebody $\P^3_{\leq 2}$ which is the union of the zero- and two-handle.  \vspace{0.5em}
\item
Each $L_i$ intersects the unique 2-handle in a collection of disks parallel to the core. We number them
\begin{equation}
D_j = D^2 \times \left\{q_j\right\} \: , \: j = 1,\ldots,6
\end{equation}
where $D_1$ correspond to the unique intersection point of $L_1$ with the hyperplane, $D_2$ to to the intersection point of $L_2$ and the rest to $L_3$. The disks $D_j$ will be bounded by a (trivial) link in 
\begin{equation}
\partial^+ \P^3_0 = S^5 
\end{equation}
consisting of six parallel copies of the attaching $S^1$ of the 2-handle. \vspace{0.5em}
\item
The intersection points $q_{13} := L_1 \cap L_3$ and $q_{23} := L_{2} \cap L_3$ all lie inside the zero-handle. 
\end{itemize}
The situation is schematically depicted in Figure \ref{fig:fiber1}, where we write the intersection of a generic hyperplane section with the quadric as the plumbing of two spheres: a $(1,0)$ and a $(0,1)$-curve in the quadric. The blue circles indicate intersections with the first kind of ruling line, while green indicate intersections with the second. The black dotes are $q_{13}$ and $q_{23}$. We choose two $\epsilon$-balls $B_\epsilon(q_{13}),B_\epsilon(q_{23})$ around each $A_5$-intersection points (they appear in red), where $\epsilon>0$ is taken to be small enough so that the balls are disjoint from each other and contained in the zero-handle. \\

We consider a generic 1-parameter embedded smoothing 
\begin{equation}
\mathcal{C} \hookrightarrow B^2_\delta(0) \times \P^3
\end{equation}
For every $0<|t|<\delta$, the smooth curve $C_t$ is obtained from $C_{max}$ via the following process:
\begin{enumerate}
\item
We remove a small disc from $L_i$ around each intersection point. Denote the result $L'_i$. \vspace{0.5em}
\item
Take two copies of the $A_5$-Milnor fiber (which is a genus two surface with two boundary components), one for each intersection point. We denote them $\MM_{13}$ and $\MM_{23}$. \vspace{0.5em}
\item
We perform boundary connected sum (see Figure \ref{fig:fiber2})
\begin{equation}
L'_1 \#_\partial \MM_{13} \#_\partial L'_3 \#_\partial \MM_{23} \#_\partial L'_2.
\end{equation}
Observe that by taking $\delta$ to be sufficiently small, we can ensure that the Milnor fibers of $p_{13}$ and $p_{23}$ remain inside the $\epsilon$-balls. 
\end{enumerate}

\begin{figure}[htb]
\begin{minipage}[b]{.5\linewidth}
\centering
			\fontsize{0.25cm}{1em}
			\def\svgwidth{5cm}
			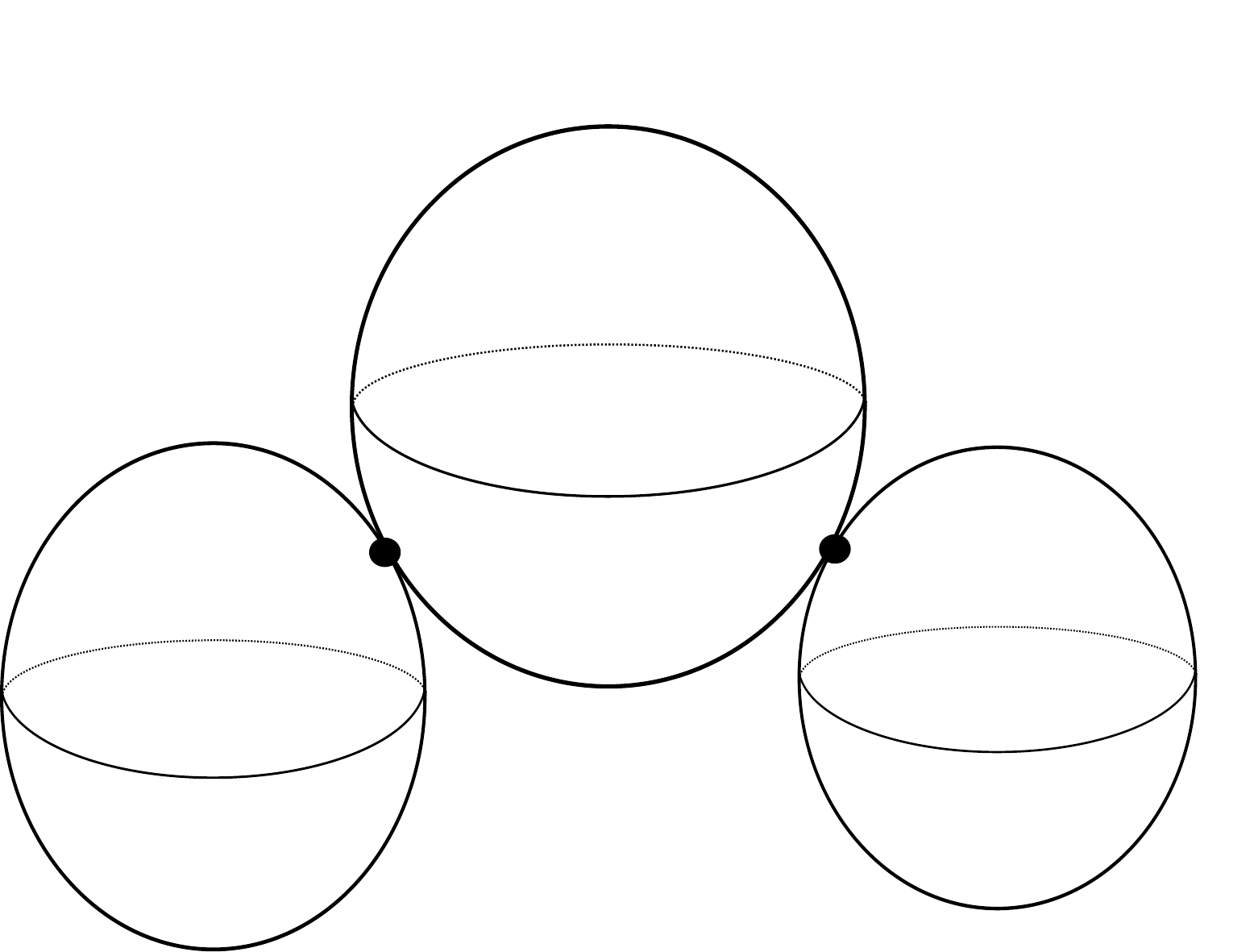
\subcaption{The singular fiber} \label{fig:fiber1}
\end{minipage}%
\begin{minipage}[b]{.5\linewidth}
\centering
			\fontsize{0.25cm}{1em}
			\def\svgwidth{7cm}
			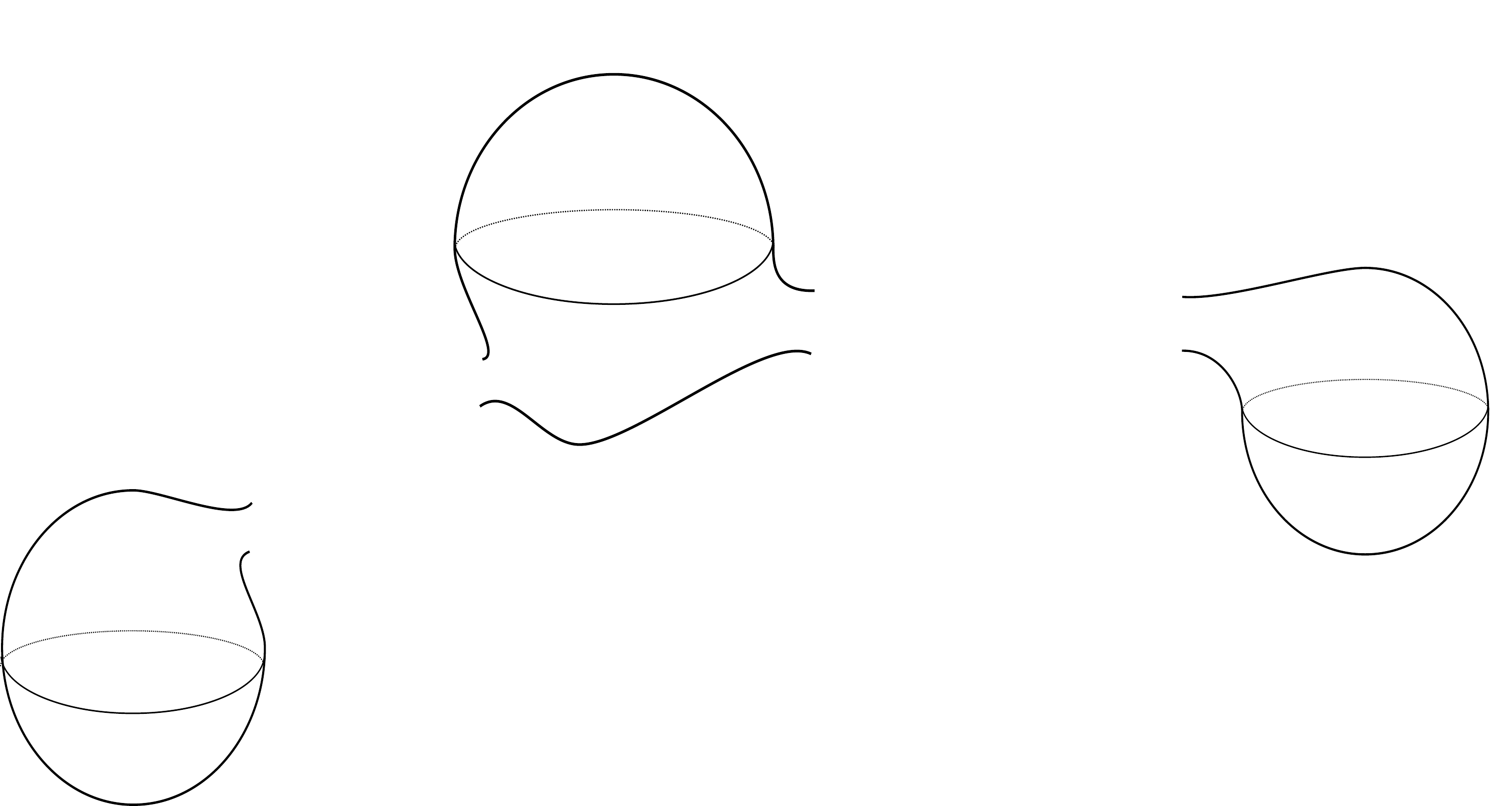
\subcaption{A generic fiber} \label{fig:fiber2}
\end{minipage}
			\end{figure} 

\begin{lemma}
We can choose an $\delta>0$, a 1-parameter smoothing $\mathcal{C} \hookrightarrow B^2_\delta(0) \times \P^3$, and a $C^2$-small perturbation of $\hat{f}^{\P^3}$ such that the resulting function, denoted $f^{\P^3}$, is Morse for the pair $(\P^3,C_t)$, $0 \neq |t|<\delta$. 
\end{lemma}
\begin{proof}
Follows from the standard density argument for Morse functions combined with the explicit description of the miniversal deformation space of an $A_n$-singularity. 
\end{proof}

Let us fix $C = C_t$ for some $t \neq 0$ and $\hat{C} = C \cap \P^3_0 = C \setminus \cup_j D_j$. Let $\hat{T}$ be a tubular neighborhood of $\hat{C}$ in $\P^3_0$, and $\hat{D} =\P^3_0 \setminus \hat{T} $. Then the codimension $q+1=4$, and by ambient re-arrangement and cancellation, we can find an ambient Morse function that induces a minimal Handle decomposition of $\hat{C}$ as well as a handlebody decomposition of the complement $\hat{D}$ with the following handles: \vspace{0.5em}
\begin{itemize}
\item
A type I 0-handle coming from the 0-handle of $\P^3$. \vspace{0.5em}
\item
For every r-handle of the surface, we have a type II $(q+r)$-handle. \vspace{0.5em}
\end{itemize}

Define $\hat{\CC} = \CC \setminus (\cup_j D_j) \times S^1$. Note that the same thing can be done for $\hat{\CC} \subset \P^3 \times S^1$: we start from the standard embedding and the Morse function $f^\CC : \CC \to \R$ constructed in Section \ref{sec:compute3dim}. We carry out an $S^1$-parametrized version of the same construction where Morse functions are replaced by Morse-Bott functions whose critical values are circles and handles by round handles. \\

To turn it into an actual Morse function, we then proceed to perturb the resulting Morse function using the standard model for the circle. Every round r-handle splits into two handles: an $r$-handle and an $(r+1)$-handle that lie in different fibers, say, over $0 \in S^1$ and $1 \in S^1$, respectively. \\

\textbf{Notation.} Given a round handle $p$, we denote the two resulting handles of indices $r$ and $r+1$ as $p$ and $\mathfrak{t}p$. \\

In each such fiber, we have six discs protruding from the ambient 0-handle and into the 2-handle, but this can be easily fixed by pushing the ribbons that connect them above the 0-handle as well. Looking at the handle decomposition upside down, we see six 0-handles being created and then connected by five 1-handles. It is always possible to cancel 1-handles, and the fact that the curve has high codimension allows us to do so by an ambient isotopy. Denote $T$ for a tubular neighbourhood of $C$ in $\P^3$, and $D$ for the complement. Now we only have a single disc $D^2 \times \left\{p\right\}$ of $C$ meeting the 2-handle of each $\P^3$ (in each fiber). Every punctured 2-handle would contribute a 2-handle (essentially the original type I one of $\P^3$) and a 5-handle to the complement. In summary, if we denote $\TT$ for a tubular neighbourhood of $C \times S^1$ in $\P^3 \times S^1$, and $\DD$ for the complement, then \vspace{0.5em}

\begin{itemize}
\item
$\CC$ has a handle decomposition with a unique round 0-handle, denoted $m_-$. Four pairs of round 1-handles, denoted $\left\{a_i,b_i\right\}$ and a unique round 2-handle, denoted $m_+$. \vspace{0.5em}
\item
For every such handle $p$, $\DD$ has a round Type II handle of index $|p|+3$, denoted $\mathfrak{a}p$. \vspace{0.5em}
\item
$\DD$ has round Type I handles of indices $0,2,4,6$, denoted with the same notation as those in $\P^3 \times S^1$. \vspace{0.5em}
\end{itemize}

 Note that the attaching sphere of each Type II handle is homologous to the normal bundle over the attaching sphere of the original handle (i.e., to $S^q \times S^r$). Dualizing this handle decomposition gives a recipe for constructing $\P^3 \times S^1$ from $\TT$ by attaching handles to the boundary $\partial \TT$. \\

Now we pass to the blowup: consider a tubular neighbourhood of the exceptional divisor 
\begin{equation}
\EE \subset \tilde{\TT} := b^{-1}(\TT) \subset \YY. 
\end{equation}
We fix an identification of $\tilde{\TT} \iso \NN_{\EE/\YY}$ and note that the projection $\tilde{\TT} \to \EE \stackrel{b}{\rightarrow} \CC$ is a (non-trivial) fiber bundle over the curve fibration, so we can: \vspace{0.5em}
\begin{itemize}
\item
Pullback our fixed choice $(f^\CC,g^\CC,\vec{\textbf{X}}^{\CC})$ of Morse $A_\infty$-datum from the previous Section \ref{sec:compute3dim}
(at this point we have kind of a ''Morse-Bott" $A_\infty$-datum for trees.) \vspace{0.5em}
\item
Add a small positive multiple of the standard model of $Tot_{\OO(-1)}$ from Section \ref{subsec:Otoy}. \vspace{0.5em}
\item
Modify the resulting triple $(f^\TT,g^\TT,\vec{\textbf{X}}^{\TT})$ using Lemma \ref{lem:consistentchoices} to a consistent universal choice of $A_\infty$-datum for trees. \vspace{0.5em}
\end{itemize}

\begin{remark} \label{rem:consistentuniversal}
Note that it is not clear if the perturbation can be taken to be $C^\infty$-small overall (due to the coherence condition and the fact that we have infinitely many steps, so any small change in, say, $\cStashefftree_{3}$ propagates to a big change in the higher Stasheff polytopes.) However, for every fixed level $d \geq 1$, we can make it as small as we please over $\cStashefftree_{d'}$ with $d'<d$; and we are only interested in triple Massey products that involve $\mu^2$ and $\mu^3$.
\end{remark}
We patch the Morse function on the complement and the Morse function on the tubular neighbourhood to obtain a new Morse-Smale pair $(f^{\YY},g^{\YY})$. Every original critical point $p$ of $\CC$ is now two critical points: $p$ and $\mathfrak{u}p$ (in our notation, that is the critical point in the pair that has \emph{degree} higher by two). Note that $\mathfrak{a}p$ and $p$ are a cancelling pair (because $\partial \tilde{\TT} \iso \partial \TT$, and the belt sphere meets the attaching sphere at a unique point, by construction.) We extend $\vec{\textbf{X}}^{\TT}$ to a consistent perturbation datum for trees $\vec{\textbf{X}}^{\YY}$, and then to consistent universal perturbation datum for pearl trees $\vec{\textbf{Y}}^{\TT} = (\vec{\textbf{X}}^{\TT},\textbf{J}^{\TT})$. Of course, in all stages of this process, we will possibly need to add perturbations, but they are all bounded in $C^\infty$-norm (at least up to finite order, see Remark \ref{rem:consistentuniversal}.)

\subsection{Proof of Proposition \ref{prop:mainterm}} \label{subsec:computemainterm}
Now everything is ready for the computation of $\tilde{\mu}^3_{2F}$. 

\begin{proposition} \label{prop:sameas3dim}
Let $(\AA,\mu^k)$ be the $A_\infty$-algebra defined using the Morse datum 
\begin{equation}
(f^{\CC},g^{\CC},\vec{\textbf{X}}^{\CC})
\end{equation}
described in Section \ref{sec:compute3dim} with critical points (labelled by the same letter as the corresponding cohomology class): 
\begin{equation}
1,\left\{a_i,b_i\right\}, c,  \mathfrak{t} , \left\{\mathfrak{t}a_i,\mathfrak{t}b_i\right\}, \mathfrak{t}c. 
\end{equation}
Denote $(\tilde{\AA},\tilde{\mu}^k)$ the $A_\infty$-algebra defined using the Morse datum 
\begin{equation}
(f^{\YY},g^{\YY},\vec{\textbf{X}}^{\YY})
\end{equation}
constructed in this subsection. Then 
\begin{equation}
\tilde{\mu}^3_{2F}(\mathfrak{u}z_3,\mathfrak{u}z_2,\mathfrak{u}z_1) =  \mathfrak{u} \mu^3(z_3,z_2,z_1)
\end{equation}
where $z_3,z_2,z_1 \in \left\{a_i,b_j\right\}$ for $1 \leq i,j \leq 4$. In more simple terms: there has to be a pair of 3-cycles coming from a standard pair critical points on the curve in the bottom part of the surface; as well as a critical point that belongs to the upper subsurface.
\end{proposition}

Note that together with Proposition \ref{prop:mainterm2}, this immediately proves Proposition \ref{prop:mainterm}. The reminder of the section is dedicated to the establishing this computation. 

\begin{remark}
We have precise control over the differentials that enter $\tilde{\TT}$, and the critical points $a_i$ and $\mathfrak{a} a_i$, $b_j$ and $\mathfrak{a} b_j$ are cancelling pairs that can even be cancelled simultaneously. However, we need to be a little careful: ''cancelling of critical points" is a non-unique operation and there is no way to insure in general that the function obtained at the end of this process is $C^2$ to the original function. The trick is to push the critical points off the exceptional divisor in a controlled way (e.g. via the same process described in the proof of Theorem 3.1 in \cite{MR3493413}) and only then cancel them. In this way we obtain a minimal Morse function with controlled behavior near the exceptional fibration. 
\end{remark}

The $\mu^3_{2F}$ term consists of three different contributions (recall Figure \ref{fig:mu3a}). \vspace{0.5em}

\begin{itemize}
\item
We immediately note that the central strata that correspond to the sphere with 4-marked points does not contribute, because it is defined as the fiber product with the Gromov-Witten pseudocycle with homology class $2F$, but the pushforward of the virtual fundamental class is zero. See Section \ref{subsec:topologyofE}. \vspace{0.5em}
\item
Next, we consider the left and right stratas. The only way to decompose $2F$ is: $F + F$, $2F + 0$ and $0+2F$. The same reason as before disqualifies the last two decompositions. Consider a rigid element $\tilde{u}$ in the left or right decomposition with homology decomposition $F + F$. $u$ is defines as the fiber product of a Gromov-Witten pseudocycle with Morse trajectory spaces. Since our $\textbf{J}$ was chosen sufficiently close, we can assume that the almost complex structure $\textbf{J}_{alg}$ is used in the computation of the Gromov-Witten pseudocycle. But then, the transverse intersection point of the Gromov-Witten pseudocycle blows-down uniquely to give a transverse intersection of a product of diagonals with Morse trajectories of the form that defines a rigid gradient tree. To determine the difference in sign, note that $GW_{F}(\mathfrak{u} \cdot a_i,\mathfrak{u} \cdot b_j,f) =  \delta_{ij}$, which is exactly the same sign as the intersection of the fibered triple diagonal $\Delta \subset \CC \times_\pi \CC \times_\pi \CC$ with $\overline{W^u}(a_i) \times \overline{W^u}(b_j) \times \CC$ -- which means that they are the same. 
\end{itemize}

\appendix
\section{Trees}
\subsubsection{Tree-minology}
\begin{definition} \label{def:tree1}
An (undirected, finite) \textbf{graph} $\Gamma$ is an ordered pair $\Gamma=(\Vert,\Edge)$ comprising a finite set $\Vert$ of vertices together with a set $\Edge$ of edges, which are 2-element subsets of $\Vert$. A \textbf{flag} $f = (v,e) \in \Flag$ is a pair consisting of a vertex and an edge adjacent to it. 
\end{definition}
\begin{definition}
The \textbf{geometric realization} $|\Gamma|$ of a graph is a one-dimensional CW complex obtained by taking a point for each vertex, a copy of the interval $[0, 1]$ for each edge, and gluing them together at coincident vertices. 
\end{definition}
In a slight abuse of notation, we will blur the distinction between a graph $\Gamma$ and its geometric realization $|\Gamma|$ and denote them both by the same letter.
\begin{definition}
A \textbf{forest} $T$ is a graph whose geometric realization has no non-trivial cycles. A connected forest is called a \textbf{tree}. 
\end{definition}
\begin{definition} 
A \textbf{tree homomorphism} $\tau : T \to T'$ is a map which collapses some subtrees of $T$ to vertices of $T'$. A tree homomorphism $\tau$ is called \textbf{tree isomorphism} if it is bijective and $\tau ^{-1}$ is a tree homomorphism as well. 
\end{definition} 
\begin{definition} A tree $T$ is called \textbf{stable} if each internal vertex has valency $|v| \geq 3$. 
\end{definition} 
Univalent vertices of $T$ and the edges which include them are called {\bf external} with all other vertices and edges called {\bf internal}. We denote $\Edge = \Edge^{ext} \sqcup \Edge^{int}$. The set of internal flags ($f = (v,e)$ where $e$ is internal) is denoted $\Flag^{int}$. 
\begin{remark}
Note that for $d<3$ every tree with $d$ external edges is unstable. For $d\geq 3$, every such tree $T$ can be stabilized in a canonical way to a stable tree $\st(T)$ with $d$-external edges by deleting the vertices with $|v|<3$ and modifying the edges in the obvious way (See e.g., \cite{MR2954391}).
\end{remark}

\subsubsection{Ribbon trees}
\begin{definition} 
An \textbf{orientation} of a tree is an assignment of a direction to each edge, turning the initial graph into a directed graph.
\end{definition} 
This allows us to think of $\Edge(T)$ as a subset of $(\Vert(T) \times \Vert(T)) \setminus \Delta$ , where $\Delta \subseteq \Vert(T) \times \Vert(T)$ is the diagonal. More precisely, for two adjacent vertices $v$ and $w$ with $v < w$, we denote by $(v, w)$ the edge going from $v$ to $w$; we call $v$ a child of $w$, and $w$ a parent of $v$. Conversely, given an edge $e \in \Edge(T)$, we write $e = (v_-,v_+)$. 
\begin{definition}
A rooted tree $T$ is a tree with a distinguished choice of a vertex (called the \textbf{root}). In that case, there is a canonical orientation on the tree, by letting the edges be oriented as going away from the root. The orientation on the edges induces a partial ordering on $V$, known as the \textbf{tree order}, where the root is the minimal element. 
\end{definition}
\begin{definition}
An \textbf{ordered tree morphism} $\tau: T \to T'$ is a tree morphism which takes the root of $T$ to the root of $T'$. 
\end{definition}
The following uniqueness statement for ordered tree isomorphism is well known.
\begin{lemma} \label{lem:noauto} If $\tau_1,\tau_2 : T \to T'$ are both ordered tree isomorphisms, then $\tau_1=\tau_2$. \noproof
\end{lemma}

\begin{definition}
A {\bf singular ribbon tree} $T$ consists of the following data:
\begin{itemize}
\item A rooted tree $T$, with a cyclic ordering of the edges adjacent to each vertex,
\item A labeling of all vertices as either of finite or infinite type.
\end{itemize}
These are subject to the conditions:
\begin{enumerate}
\item
The root must be an external vertex. 
\item
All vertices of valence greater than 2 are of finite type.
\item
All vertices of valence 2 are of infinite type.
\end{enumerate}
\end{definition}
Let $T$ be a singular tree with $d+1$ external edges. Once we have chosen the root $v_0$, the remaining external vertices (called the \textbf{leaves}) of $T$ acquire a linear ordering and would be denoted $\left\{v_1,\ldots,v_d\right\}$. $\left\{v_1,\ldots,v_d\right\}$ are the {\bf incoming} vertices (similarly for incoming edges). We equip the edges of a ribbon tree with a natural orientation as follows:
\begin{itemize}
\item The orientation on the incoming external edges points towards the interior of the tree. The outgoing external edge is oriented away from the interior.
\item The orientations on all but one of the edges adjacent to a fixed vertex $v$ point towards the vertex $v$ itself.
\end{itemize}
One can easily prove by induction that an orientation satisfying the above properties exists and is unique.  We shall therefore freely speak of the (unique) ''outgoing edge at a vertex" (the one oriented away from $v$) or of the incoming vertices of a subtree. Another way to look at that, is to note that the data we described is equivalent to specifying an isotopy class of embedding of the geometric realization of $T$ in the closed unit disc $\Delta \subset \C$, where the only points on the boundary $\partial \Delta$ are the external vertices. We think of the root $v_0$ as ''the bottom" and arrange the leaves $\left\{v_1,\ldots,v_d\right\}$ so that their linear ordering matches the order on which they fall on $\partial D$, counterclockwise. Then the natural orientation on $T$ is given by flowing ''upwards" from the bottom to the top. \\

A flag $f = (v,e)$ is called positive if the orientation of $e$ points away from $v$, and
negative otherwise. To every interior edge $e$ correspond two flags $f^\pm(e)$, one positive
and the other negative. Similarly, an exterior edge has a single flag belonging to it,
which is positive or negative depending on whether the edge is a leaf or the root. On
occasion, we will find it useful to number the flags adjacent to a given vertex $v$ by
$f_0,\ldots,f_{|v|-1}$, where $|v|$ is the valency. This numbering always starts with
the unique negative flag $f_0$ and continues anticlockwise with respect to the given
embedding into $\Delta$.
\begin{definition}
A singular ribbon tree is {\bf reducible} if it contains a bivalent vertex. It is otherwise called irreducible, or, more simply, a {\bf ribbon tree}. If $T$ is a singular ribbon tree then a subtree $S \subset T$ is called an {\bf irreducible component} if it is maximal among irreducible subtrees.
\end{definition} 
\begin{lemma}
Every singular ribbon tree admits a unique decomposition into irreducible components. \noproof
\end{lemma}
If a vertex lies on an infinite edge, the labeling of vertices as either of finite or infinite type determines whether the edge is isometric to $[0 , + \infty)$ or $(- \infty, + \infty)$, with $0$ corresponding to the vertex of finite type. Hence, there is a natural notion of (isometric) automorphisms.
\begin{definition} \label{def:ribbonstability}
A ribbon tree is {\bf stable} if it has a finite isometry group. A singular ribbon tree is stable if so are all its irreducible components.
\end{definition}
\begin{lemma}
All ribbon trees which are not stable are isometric to $( - \infty , + \infty)$.  In particular, a singular ribbon tree is stable unless it includes an edge with two bivalent endpoints. \noproof
\end{lemma}
\begin{definition}
A \textbf{metric} on $T$ is a map that assigns a positive length to each edge
\begin{align*}
g_T : \Edge \to (0 , + \infty]
\end{align*}
such that $g_T(e) = + \infty$ if and only if it has an endpoint of infinite type. 
\end{definition}
Let $T_1$ and $T_2$ be singular ribbon trees with $d_1+1$ and $d_2+1$ external edges, respectively. There are natural \textbf{partial compositions} defined by ''grafting" one tree into the other,   
\begin{definition} \label{def:grafting}
Given any $1 \leq j \leq d_1$, we may construct a new singular ribbon tree 
\begin{equation}
T = T_1 \circ_j T_2
\end{equation}
with $d=d_1+d_2-1$ inputs by attaching the output of $T_2$ to the j-th input of $T_1$. 
\end{definition}
We often refer to singular ribbon trees $T$ as ''broken trees", meaning they were obtained from two subtrees $T_1$ and $T_2$ by merging two infinite type univalent vertices to a single bivalent vertex.

\subsubsection{As an equivalence relation}
There is a different convention in the literature.
\begin{definition}
A \textbf{$k$-labelled tree} is a triple 
\begin{equation}
T=(V,E,\Lambda), 
\end{equation}
where $(V,E)$ is a (connected) tree with set of vertices $V$ and edge relation $E\subset V\times V$, and $\Lambda=\{\Lambda_\alpha\}_{\alpha\in V}$ is a decomposition of the index set $\{1,\dots,k\}=\amalg_{\alpha\in V}\Lambda_\alpha$. 
\end{definition}
We write $\alpha E\beta$ if $(\alpha,\beta)\in E$. Let
\begin{equation}
   e(V) = |V|-1
\end{equation}
be the number of edges.
\begin{remark}
Note that the labelling $\Lambda$ defines a unique map $\{1,\dots,k\}\to V$, $i\mapsto\alpha_i$ by the requirement $i\in\Lambda_{\alpha_i}$.
\end{remark}
This would be useful to us because we will often be referring to \cite{MR2954391} and \cite{MR2399678}. 
\begin{remark} 
It is easy to switch between the different conventions: Given a tree as in Definition \ref{def:tree1}, let $\Label(T)$ denote the set of external flags with $f = (v,e)$ with $v \in \Vert^{int}(T)$. Since the other end point of $e$ has to be an external vertex, we simply suppress the external edges in the notation and think of these external edges as labels, i.e., as a collection of sets
\begin{equation}
\Label(T) = \left\{\Label_v\right\}_{v \in \Vert^{int}(T)},
\end{equation}
indexed by the interval vertices which give a decomposition of the set of external edges
\begin{equation}
\Edge^{ext}(T)=\coprod_{v} \Label_v.
\end{equation}
The converse is also clear.
\end{remark}

\section{Signs and orientation} \label{sec:signs}
We describe the signs assigned to each of the relevant manifold and moduli spaces discussed, thereby obtaining operations defined over $\Z$. We will draw heavily from the similar discussions in \cite[Subsection (11l)]{MR2441780} and \cite[Section 8]{MR2786590}. \\

\subsubsection{General conventions}
We assume that we have fixed a choice of orientation $\lambda(M)$ for the ambient manifold $M$, where $\lambda$ denotes the top power of the tangent bundle, taken as a graded vector space. 
\begin{itemize}
\item
If $v_1 \wedge \ldots \wedge v_n$ is an orientation of $\lambda(V)$, we define the \textbf{opposite orientation} $\lambda^{-1}(V)$ as the one defined by $v_n^{-1} \wedge \ldots \wedge v_1^{-1}$. \vspace{0.5em}
\item
Whenever a space $X$ is given as a transverse fibre product of $X_1$ and $X_2$ over $X_0$, the short exact sequence
\begin{equation} \label{eq:shortexactsequence}
0 \to TX \to TX_1 \oplus TX_2 \to TX_0 \to 0
\end{equation}
determines a fixed isomorphism
\begin{equation} \label{eq:fibeproductconvention}
\lambda(X) \otimes \lambda(X_0) \iso \lambda(X_1) \otimes \lambda(X_2). \vspace{0.5em}
\end{equation}
\item
Our boundary orientation follow the \textbf{outward normal convention}: 
\begin{equation}
\text{(Outward normal vector $\nu$)} \times \partial M = M. \vspace{0.5em}
\end{equation}
\end{itemize}
\begin{remark}
For simplicity we work in a fixed manifold, but everything we say applies just as well for gradient trees in the total space of a locally Hamiltonian fibration $\LHF$, vertical maps etc. 
\end{remark}

Finally we remark that the signs of the multi-linear operations $\mu^d$ we construct using these coherent orientations will need to coincide with those introduced in Getzler and Jones, which is the standard convention for $A_\infty$-algebras (see Section \ref{sec:algebraic} for more.) So in the actual definitions of $\mu^d_A$ we will usually incorporate some global sign twisting. 

\subsubsection{Morse theory} \label{subsec:morseorientation}
Let $(f,g)$ be a Morse-Smale pair. We begin by describing the differential in the Morse complex $CM^\bullet(f,g)$. \\

Recall that we defined the space of gradient trajectories $\TT (p_1; p_0)$ to be the intersection of $W^s(p_0)$ and $W^u(p_1)$, i.e., the fibre product of $W^s(p_0)$ and $W^u(p_1)$ over their common inclusion in $M$, yields a natural short exact sequence
\begin{equation} \label{eq:8.1}
0 \to T \TT (p_1; p_0) \to T(W^s(p_0) \times W^u(p_1)) \to TM \to 0
\end{equation}
and therefore an isomorphism of top exterior powers 
\begin{equation} \label{eq:8.2}
\lambda(\TT (p_1; p_0)) \otimes \lambda(M) \iso \lambda(W^s(p_0)) \otimes \lambda(W^u(p_1)).
\end{equation}
Whenever $\deg(p_0) = \deg(p_1)+1$, the intersection between $W^s(p_0)$ and $W^u(p_1)$ is a segment,
mapping to $\R$ immersively under $f$, so the tangent space of
\begin{equation} \label{eq:8.3}
\psi \in \TT (p_1; p_0)
\end{equation}
is canonically oriented, with our conventions dictating that a rigid gradient flow line is given the opposite orientation of the one induced by projection to $\R$ (we pick the orientation identifies $p_1$ with $-\infty$). Decomposing the tangent space $T_{p_1} M$, this gives an isomorphism
\begin{equation} \label{eq:decom_tangent_space_crit_points}
\lambda(M) \iso \lambda(W^s(p_1)) \otimes \lambda(W^u(p_1))
\end{equation}
Together with the map in \eqref{eq:8.2} determines an isomorphism
\begin{equation} \label{eq:8.6}
\lambda(W^s(p_1)) \otimes \lambda(W^u(p_1)) \otimes \lambda(W^s(p_0)) \otimes \lambda(W^u(p_1)),
\end{equation}
hence
\begin{equation} \label{eq:8.6}
\lambda(W^s(p_1)) \iso \lambda(W^s(p_0)). 
\end{equation}
\begin{definition}
Isomorphism \eqref{eq:8.6} determines an induced map on orientation lines, denoted
\begin{equation}
\mu^M_\psi : |\oo_{p_1}| \mapsto |\oo_{p_0}|. 
\end{equation}
\end{definition}
Since the Morse trajectories live in the total space, the same applies to a Morse theory in an LHF (just replace $M$ with $E$.) \\

The differential is defined to be the sum of all such terms multiplied by a global sign:
\begin{equation} \label{eq:formula_differential_morse}
\mu^M_1([p_1]) = \sum_{\stackrel{\deg(p_0) = \deg(p_1) +1}{\psi \in \TT(p_1;p_0)}} (-1)^{\dim(M)} \mu^M_\psi([p_0])
\end{equation}

We remark that each of the Morse of the trajectory spaces from Section \ref{subsec:toymodeltrajectoryspace} can be orienated in this fashion: the orientation of $\Morse(p_-,p_+)$ is induced from that of $\TT (p_-; p_+)$ and that of a level set $f^{-1}(c)$ by our fiber product convention, the half-infinite trajectory spaces have the same orientation as the stable/unstable manifold. Finally, the space of finite trajectories $\Morse(U_-,U_+)$ is globally oriented; and if we restrict to non-constant trajectories, the space $\Morse^*(U_-,U_+)$ embedds into $M \times M$ as an open subset, and therefore has the same orientation as $\lambda(M) \otimes \lambda(M)$. 

\subsubsection{Coherent orientation for abstract moduli spaces}

\subsubsection{Moduli of discs} \label{subsubsection:pearl}
To orient the d-th moduli of pearls, we fix a slice of $\Stasheffdisc_d$ such that the first three boundary marked points $z_0$ (negative), and $z_1,z_2$ (positive) are fixed; and consider the positions of the remaining points $(z_3,\ldots,z_d$) with respect to the counterclockwise boundary orientation as a local chart. With respect to this chart, orient $\Stasheffdisc_d$ by
the top form
\begin{equation}
dz_3 \wedge \ldots \wedge dz_d
\end{equation}
\begin{remark}
This agrees with the conventions in \cite{MR2441780} (as well as the realization as the real part in Deligne-mumford space).
\end{remark}

\subsubsection{Stasheff trees} \label{subsubsection:orient_stasheff_tree}
Let $T$ be a trivalent Stasheff tree with $d$ inputs. Given an integer $k$, let $A_k$ denote the descending arc starting at the k-th incoming leaf, and let $b_k$ denote the vertex of $T$ where $A_k$ and $A_{k−1}$ meet. We shall label the edges of this trivalent vertex $e^l_k$, $e^r_k$ and $e^d_k$ as in Figure \ref{fig:stasheffmoduliabouzaid}.
\begin{figure} 	
    \centering
		\fontsize{0.3cm}{1em}
		\def\svgwidth{5cm}
		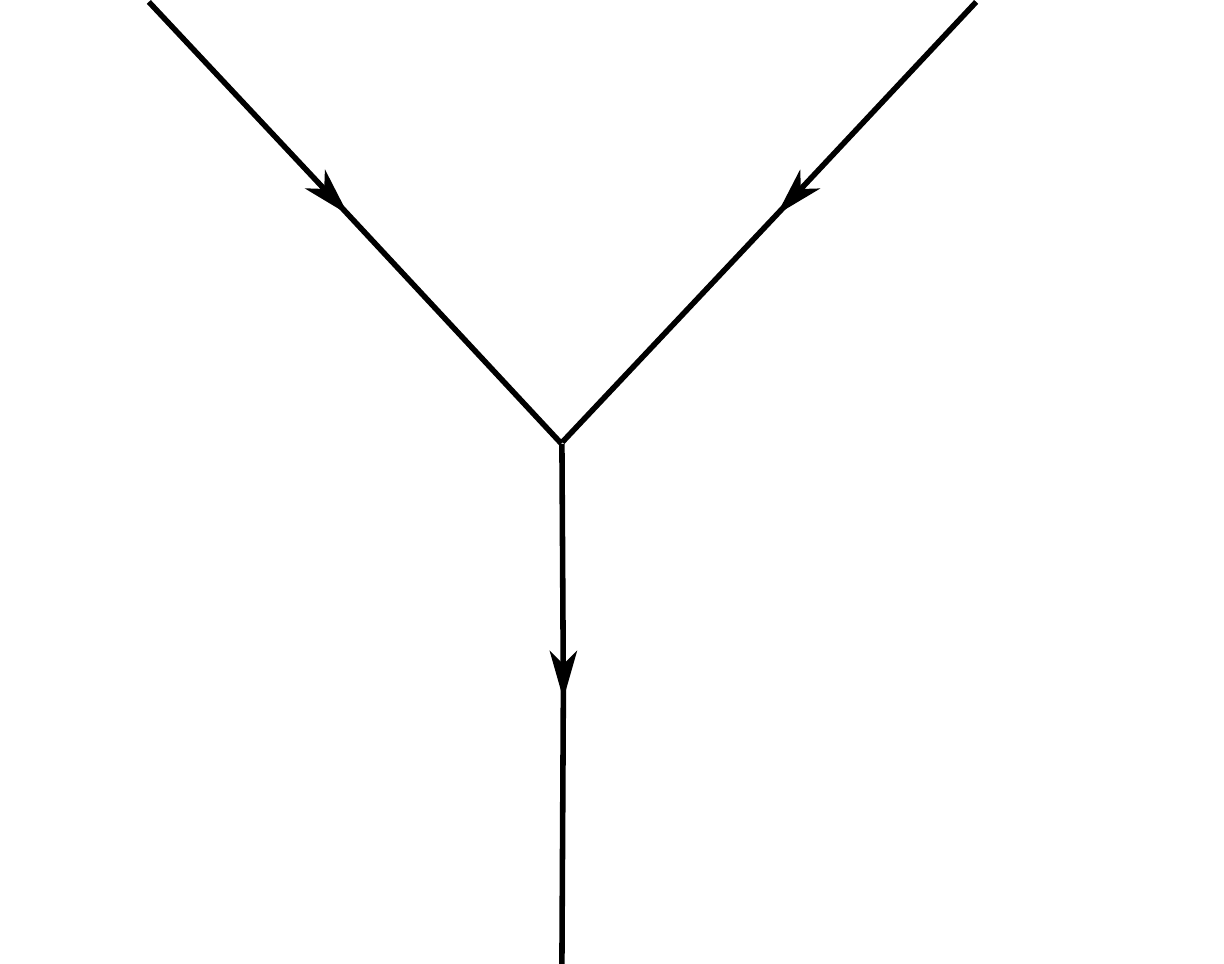
		\caption{Moduli of Stasheff trees}
		\label{fig:stasheffmoduliabouzaid}
\end{figure}

We follow the procedure described in \cite{MR2529936} and \cite{MR2786590}; Namely, order the (finite) edges of $T$ using the following inductive process:
\begin{enumerate}
\item
Set $T_2$ to be the union of all the external edges of $T$.
\item
If $e^r_k$ lies in $T_k$, then define $e_k = e_k^d$. Otherwise, set $e_k = e^r_k$. 
\item
Set $T_k$ to be the union of $T_{k−1}$ with $e_k$.
\end{enumerate}
Then if $j \neq k$, $e_j \neq e_k$ (Lemma 8.1 in \cite{MR2786590}) and 
\begin{definition}
We define the orientation on $\TT_d$ to be given by
\begin{equation} \label{eq:orientation_stasheff}
(-1)^{r(T)} dt_{e_3} \wedge \ldots \wedge dt_{e_d}
\end{equation}
where $t_e$ is the length of $e$, and $r(T )$ is the number of integers between $3$ and $d$ for which $e_k = e^r_k$.  
\end{definition}
Then the orientation given by formula \eqref{eq:orientation_stasheff} is independent of the topological type of T (Lemma 8.2 in \cite{MR2786590}) and 
\begin{lemma}[8.3 in \cite{MR2786590}] \label{lem:boundary_orientation_stasheff}
Consider the boundary stratum of $\cStashefftree_{d}$ which is the image of the grafting map
\begin{equation}  
\cStashefftree_{d_1} \times \cStashefftree_{d_2} \to \cStashefftree_{d} 
\end{equation}
given by attaching a tree with $d_2$ inputs to the $(k+1)^{\mathrm{st}}$ incoming leaf of a tree with $d_1$ inputs (where $d_1 + d_2 = d +1$)
The product orientation on
\begin{equation}  (-1, 0] \times \cStashefftree_{d_1} \times \cStashefftree_{d_2} \hookrightarrow \cStashefftree_{d}  
\end{equation}
differs from the boundary orientation by a sign of
\begin{equation}  
(d_1 -k) d_2  + d_2 + k.
\end{equation}  \noproof
\end{lemma}

Now, applying \eqref{eq:fibeproductconvention} to the fiber product description of the moduli space of gradient trees determines an isomorphism
\begin{equation}   \label{eq:iso_det_bunldes_stasheff} 
\lambda(\Stashefftree(\morseLabel;p_0)) \otimes \lambda(Q^{d+1}) \cong \lambda(M) \otimes \lambda(\Stashefftree_{d}) \otimes \lambda(W^{s}(p_0)) \otimes \lambda (W^{u}(\morseLabel)) 
\end{equation}
where $W^{u}(\morseLabel^{in})$ is the product of the descending manifolds of the critical points $\morseLabel^{in} = (p_d,\ldots,p_1)$. 

\begin{definition}
Given a rigid gradient tree $\psi$, the above isomorphism and Equation \eqref{eq:decom_tangent_space_crit_points} give a natural map
\begin{equation}  \label{eq:iso_det_bunldes_stasheff}
\mu_\psi : |\oo_{p_d}| \otimes \cdots \otimes |\oo_{p_1}| \mapsto |\oo_{p_0}|. 
\end{equation}
\end{definition}

\subsubsection{Pearl trees} \label{subsubsection:orient_pearl_tree}
It is now easy to define to construct orientations on the strata in the moduli space of pearl trees. For the strata of pearl trees that only have one main component and d leaves we define the orientations to be the same as that of the d-th moduli of pearls. Now note that any other strata can be constructed iteratively by grafting two pearl trees. So we simply define the orientation of any strata obtained by grafting as the tensor product of the orientations of its components.

\begin{lemma} \label{lem:boundary_orientation_pearl}
Consider the boundary stratum of $\cPearltree_{d}$ which is the image of the grafting map
\begin{equation}  
\cPearltree_{d_1} \times \cPearltree_{d_2} \to \cPearltree_{d} 
\end{equation}
given by attaching a pearl tree with $d_2$ inputs to the $(k+1)^{st}$ incoming leaf of a pearl tree with $d_1$ inputs (where $d_1 + d_2 = d +1$). The product orientation on
\begin{equation}  
(-1, 0] \times \cPearltree_{d_1} \times \cPearltree_{d_2} \hookrightarrow \cPearltree_{d} 
\end{equation}
differs from the boundary orientation by a sign of
\begin{equation}  
(d_1 -k) d_2  + d_2 + k.
\end{equation}  
\end{lemma}
\begin{proof}
Follows as a special case of a more general result (Lemma 1.72 in \cite{MR3153325}) proved by Charest for cluster $\otimes$-complexes (even though it was stated for trees of marked \emph{discs}, the combinatorics is identical). 
\end{proof}

\subsubsection{Orienting moduli spaces of pearl tree maps} 

One naturally obtains an orientation of the determinant line of the linearized operator from the isomorphism

Let $d \geq 1$, $\Upsilon$ a d-leafed pearl type. Applying \eqref{eq:fibeproductconvention} to the fiber product diagram \eqref{eq:flag_correspondence} determined by a homology decomposition and a stable pearl combinatorial type $\Upsilon$ gives an isomorphism
\begin{equation} \label{eq:iso_det_bunldes_pearl} 
\lambda(\PP_\Upsilon(\underline{A},\morseLabel)) \otimes  \lambda(\Flag_\Upsilon)= \lambda(\Vert_\Upsilon) \otimes \lambda(\Edge_\Upsilon) 
\end{equation}

However, this can simplified: since the orientation of every finite edge corresponds to a finite Morse trajectory spaces which contributes $\lambda(M) \otimes \lambda(M)$. Inductive fiber products over the diagonal is exactly how the space of nodal maps is defined and oriented, so it is enough to apply a type (II) degeneration and contract every combinatorially finite edge to zero and consider the resulting strata. But the moduli spaces of holomorphic curves from a closed domain are canonically oriented by the complex structure (see e.g., Theorem 3.1.5 and in \cite{MR2954391}). So we apply a type (I) degeneration and deduce that equation \eqref{eq:iso_det_bunldes_pearl} (neglecting even powers) is equivalent to  
\begin{equation} \label{eq:iso_det_bunldes_pearl2} 
\lambda(\PP_\Upsilon(\underline{A},\morseLabel)) \otimes  \lambda(M^{d+1})= \lambda(M) \otimes \lambda(W^{u}(\morseLabel)) \otimes \lambda(W^{s}(p_0)) \otimes \lambda(\PP_d)
\end{equation}

where $W^{u}(\morseLabel)$ is the product of the descending manifolds of the critical points $\morseLabel = (p_d,\ldots,p_1)$. \\

As before, 
\begin{definition} \label{def:inducedmaponorientationlines}
Given a rigid holomorphic pearl tree $u$, the isomorphism \eqref{eq:iso_det_bunldes_pearl} and Equation \eqref{eq:decom_tangent_space_crit_points} give a natural map
\begin{equation}  \label{eq:iso_det_bunldes_stasheff}
\mu_u : |\oo_{p_d}| \otimes \cdots \otimes |\oo_{p_1}| \mapsto |\oo_{p_0}|. 
\end{equation}
\end{definition}
\section{Topology of Mapping tori}
We summarize some facts about the topological invariants associated to mapping tori. All the material is well-known, see e.g., \cite[Appendix B]{perutzthesis}. Let $M$ be a manifold and $\phi : M \to M$ a diffeomorphism. 

\subsubsection{Exact triangles}
The first thing to note is that we can realize the mapping torus operation within the category of chain complexes. 
\begin{definition}
Let $(C_\bullet ,\partial)$ be a chain complex and let $\Phi : C_\bullet \to C_\bullet$ be a chain-map. We define the \textbf{algebraic mapping cone} to be the complex 
\begin{equation}
cone(\Phi):= (C_\bullet[-1] \oplus C_\bullet , \partial_{cone}),
\end{equation}
where the differential is defined by
\begin{equation}
(a,b) \mapsto (-\partial a , \partial b + \Phi(a)). 
\end{equation}
\end{definition}

There is an obvious short exact sequence

\begin{equation} \label{eq:shortexact}
 0 \to  C_\bullet \to cone(\Phi) \to C_{\bullet}[-1] \to 0. 
\end{equation}

\begin{lemma}
The homology of the algebraic mapping cone of
\begin{equation}
id - \phi_* : C_*(M) \to C_*(M)
\end{equation}
fits into an exact triangle 
\begin{equation} \label{eq:triangle1}
\begin{tikzpicture}
  \matrix (m) [matrix of math nodes,row sep=3em,column sep=4em,minimum width=2em]
  {
     H_*(M) & H_*(M) \\
      & cone(id-\phi) \\};
  \path[-stealth]
    (m-1-1) edge node [above] {$id - \phi_*$} (m-1-2)
		(m-1-2) edge node [above] {$$} (m-2-2)
		(m-2-2) edge node [auto] {$[-1]$} (m-1-1);
\end{tikzpicture}
\end{equation}
\end{lemma}
\begin{proof}
Follows from \eqref{eq:shortexact}, and the maps $(id - \phi_*)$ in the triangle are the connecting maps.
\end{proof}

\begin{lemma}
The homology of $M_\phi$ fits into an exact triangle
\begin{equation} \label{eq:triangle2}
\begin{tikzpicture}
  \matrix (m) [matrix of math nodes,row sep=3em,column sep=4em,minimum width=2em]
  {
     H_*(M) & H_*(M) \\
      & H_*(M_\phi) \\};
  \path[-stealth]
    (m-1-1) edge node [above] {$id - \phi_*$} (m-1-2)
		(m-1-2) edge node [right] {$i_*$} (m-2-2)
		(m-2-2) edge node [auto] {$[-1]$} (m-1-1);
\end{tikzpicture}
\end{equation}
\end{lemma}
\begin{proof}
We give $S^1$ the cell structure that has just two cells, and consider the cellular version of the Leray-Serre spectral sequence of the fibration $M_\phi \to S^1$. Then the $E^1_{p,q}$-term is zero for any $p \neq 0,1$. When $p=0,1$, the entry is $H_q(M)$ with differential 
\begin{equation}
d^1_{1,q} = id - \phi_\bullet : H_\bullet(M) \to H_\bullet(M)
\end{equation}
where $\phi_\bullet$ is the induced map on cohomology. Clearly the sequence collapse at the $E^2$-page to give short exact sequences
\begin{equation}
0 \to E_{1,q}^\infty \to H_q(M) \stackrel{id-\phi_*}{\longrightarrow} H_q(M) \to E_{0,q}^\infty \to 0
\end{equation}
The exact triangle comes from splicing with the short exact sequences
\begin{equation}
0 \to H_{p+q}(M) \to H_{p+q}(M_\phi) \to E^\infty_{p,q} \to 0. 
\end{equation}
\end{proof}
The connecting maps are the transfer operations 
\begin{equation}
i^! : H_q(M_\phi) \to H_{q-1}(M) 
\end{equation}
which represents intersection with the fibre, and 
\begin{equation}
i_! : H_q(M) \to H_{q+1}(M_\phi).
\end{equation}
\begin{lemma}
There is a quasi-isomorphism
\begin{equation}
Q : cone(id - \phi_*) \to C_*(M_\phi),
\end{equation}
and the isomorphism $Q_*$ identifies the two exact triangles \eqref{eq:triangle1} and \eqref{eq:triangle2}. 
\end{lemma}
\begin{proof}
Using cubical singular chains, define
\begin{equation}
Q : cone(id - \phi_*)_q \to C_q(M_\phi) \: , \: (a, b) \to \tilde{a} + i^* b. 
\end{equation}
Here $a \mapsto \tilde{a}$ is the linear extension of the map which sends a singular $(q - 1)$-cube $\sigma$ to the $q$-cube $\sigma \times id : [0, 1]^{q - 1} \times [0, 1] \to M \times [0, 1]$, projected to $M_\phi$. The map $Q$ is a chain map because
\begin{equation}
\partial \tilde{a} = -\widetilde{\partial a} + i_\bullet(a - \phi_\bullet(a))
\end{equation}
Using the five-lemma it is easy to show that the induced map $Q_\bullet$ is an isomorphism.
\end{proof}

\subsubsection{A De-Rham model}
Dualizing the same argument as above sets up an exact triangle for the cohomology of the mapping torus
\begin{equation} 
\label{eq:triangle3}
\begin{tikzpicture}
  \matrix (m) [matrix of math nodes,row sep=3em,column sep=4em,minimum width=2em]
  {
     H^*(M) & H^*(M) \\
      H^*(M_\phi) & \\};
  \path[-stealth]
    (m-1-1) edge node [above] {$id - \phi^*$} (m-1-2)
		(m-1-2) edge node [auto] {$[+1]$} (m-2-1)
		(m-2-1) edge node [left] {$i^*$} (m-1-1);
\end{tikzpicture}
\end{equation}
as well as a quasi-isomorphism
\begin{equation}
C^*(M_\phi) \to cone(id - \phi^*).
\end{equation}
This have an explicit representations in terms of differential forms: Fix a Thom form $\tau$ on $S^1$ - a one-form, supported in a small interval, with integral $1$. For a form $\alpha \in \Omega^q(M)$, define $i_!(\alpha) \in \Omega^{q+1}(M_\phi)$ to be the form whose pullback to $\R \times M$ is $\tau \wedge \alpha$. The induced map in cohomology, 
\begin{equation}
i_! : H^q(M; \R) \to H^{q+1}(M; \R), 
\end{equation}
is the transfer operation - the connecting map in the exact triangle. A typical element of $\Omega^q(M_\phi)$ has shape $\omega_t + dt \wedge \eta_t$. Here, for any $t \in \R$, $\omega_t \in \Omega^q(M)$, $\eta_t \in \Omega^{q-1}(M)$, and $\phi^\bullet \omega_t = \omega_{t+1}$, $\phi^\bullet \eta_t = \eta_{t+1}$. The mapping cone of $id  - \phi^\bullet : \Omega^\bullet(M) \to \Omega^\bullet(M)$ is the complex
\begin{equation}
cone(id  - \phi^*) := \left(\Omega^*(M)[-1] \oplus \Omega^*(M) \: , \: d(\alpha,\beta) := (−d\alpha + \beta - \phi^\bullet \beta , d \beta)\right). 
\end{equation}
It is easy to see that the map
\begin{equation}
\begin{split}
\Omega^*(M_\phi) &\to cone(id - \phi^*) \\
\omega_t + dt \wedge \eta_t &\mapsto \int_0^1 \eta_t dt
\end{split}
\end{equation}
is a chain map. By the the five-lemma it is a quasi-isomorphism.

\section{Symmetric products of curves}
As preparation for the computation in Section \ref{sec:cohomologylevelcomputations} of the GW invariants of lines in the 3-fold and the mapping tori in homology class $L-cF$, we summarize in this appendix some standard facts (taken from \cite{MR770932}) about symmetric products. \\

Let $C$ be a smooth irreducible complex projective curve of arbitrary genus $g$. 
\begin{definition}
Given $k \geq 1$, we define the \textbf{k-fold symmetric product of C} to be the quotient
\begin{equation}
C^{(k)} := C^{\times k}/\Sigma_k, 
\end{equation}
where $\Sigma_k$ is the symmetric group acting on $C^{\times k}$ by permuting coordinates. It is a smooth projective variety of complex dimension $d$ (see \cite[p. 18]{MR770932}) parametrizing unordered k-tuples of points on the curve or, equivalently, effective divisors of degree $k$ on $C$. Adopting this point of view we will write a point $D \in C^{(k)}$ as a formal linear combination 
\begin{equation}
D = \sum k_i p_i 
\end{equation}
where $p_i$ are points in the curve, and $k_i$ positive integers whose sum is $k$. 
\end{definition}
Associated to a smooth curve there are also the Jacobian of the curve and the Picard varieties. 
\begin{definition}
We define the \textbf{Jacobian} to be the g-dimensional complex torus obtained as the quotient
\begin{equation}
J(C) := H^0(C;\KK_C)^\vee / H_1(C;\Z)
\end{equation}
\end{definition}
\begin{definition}
We define the \textbf{Picard group} of $C$ to be 
\begin{equation}
Pic(C) := H^1(C,\OO_C^*). 
\end{equation}
We denote by $Pic^k(C)$ to be the connected component of $Pic(C)$ which parametrizes the isomorphism classes of line bundles of degree $k$ on $C$. 
\end{definition}
Fix once and for all a basepoint $p \in C$. 
\begin{itemize}
\item
By choosing a basis of holomorphic 1-forms
\begin{equation}
\omega_1,\ldots,\omega_{g} \in H^0(C;\KK_C)^\vee, 
\end{equation}
we define the \textbf{Abel-Jacobi map}
\begin{equation}
\begin{split}
C^{(k)} &\stackrel{u}{\longrightarrow} J(C) \\
p_1 + \ldots + p_k &\mapsto \left(\sum_{i=1}^k \int_p^{p_1} \omega_1,\ldots,\sum_{i=1}^k \int_p^{p_1} \omega_k\right)
\end{split}
\end{equation}
\item
We define the \textbf{Picard map} to be
\begin{equation}
\begin{split}
C^{(k)} &\stackrel{v}{\longrightarrow} Pic^k(C) \\
D &\mapsto \OO(D)
\end{split}
\end{equation}
\item
It is possible to prove that there is an isomorphism $Pic^k(C) \to J(C)$ such that the resulting Abel-Jacobi and Picard maps make the following diagram commute 
\begin{equation}
\begin{tikzpicture}
  \matrix (m) [matrix of math nodes,row sep=3em,column sep=4em,minimum width=2em]
  {
      & C^{(k)} & \\
     Pic^k(C) & & J(C) \\};
  \path[-stealth]
    (m-1-2) edge node [above] {$v$} (m-2-1)
		(m-1-2) edge node [above] {$u$} (m-2-3)
		(m-2-1) edge node [above] {$\iso$} (m-2-3)		;
\end{tikzpicture}
\end{equation}
\end{itemize}
\textbf{Notation.} We shall define the binomial coefficients $\binom{n}{i}$, regardless of the sign of the integers $n$ and $i$, by the following convention
\begin{equation}
\binom{n}{i} = \begin{cases}
\frac{n\cdot(n-1)\cdot \ldots \cdot (n-i+1)}{i!}   &\text{, If } i > 0 \\
1  &\text{, If } i=0 \\
0 &\text{, If } i<0 \\
\end{cases}
\end{equation}
\subsubsection{Homology}
We choose a standard basis of 1-cycles $A_i,B_i$ and denote 
\begin{equation}
E_i := \begin{cases}
A_i,		& \text{if } 1 \leq i \leq g \\
B_{i-g},		& \text{if } g+1 \leq i \leq 2g \\
\end{cases}
\end{equation}
In this notation, the homology of the curve is
\begin{equation}
H_\bullet(C;\Z) = 
\begin{cases}
\Z \cdot PT,		& \text{if } \bullet = 0 \\
\bigoplus_{i=1}^{2g} (\Z \cdot E_i) 	& \text{if } \bullet = 1 \\
\Z \cdot C,		& \text{if } \bullet = 2 
\end{cases}
\end{equation}
The following is well-known property of symmetric products.
\begin{proposition}
The \textbf{basepoint inclusion map}, defined by
\begin{equation}
\begin{split}
\iota_p &: C^{(k-1)} \to C^{(k)}, \\
Q &\mapsto p + Q.
\end{split}
\end{equation}
induces a homology monomorphism. \noproof
\end{proposition}

As a standard abuse of notation, we identify all $C^{(r)}$ with its image in $C^{(s)}$ under the basepoint embeddings for $1 \leq r \leq s$. We also identify $H_\bullet(C^{(r)})$ with its image in $H_\bullet(C^{(s)})$, and denote $C_r$ for the image of the fundamental class $[C^{(r)}] \in H_\bullet(C^{(s)})$, with $C = C_1$. Note that the operation defined by concatenation of points, which we will denote as 
\begin{equation}
\cdot : C^{(r)} \times C^{(s)} \to C^{(r+s)} 
\end{equation}
gives a commutative multiplication. Moreover,
\begin{lemma}
The following diagram commutes
\begin{equation}
\begin{tikzpicture}
  \matrix (m) [matrix of math nodes,row sep=3em,column sep=4em,minimum width=2em]
  {
      C^{(r)} \times C^{(s)} & C^{(r+s)}  \\
      J(C) \times J(C) & J(C) \\};
  \path[-stealth]
    (m-1-1) edge node [above]  {} (m-1-2)
		(m-1-1) edge node [left] {$u$} (m-2-1)
		(m-1-2) edge node [right] {$u$} (m-2-2)		
    (m-2-1) edge node [above] {} (m-2-2);
\end{tikzpicture}
\end{equation}
where: the bottom map is addition in the Jacobian torus, and the top map is concatenation. \noproof
\end{lemma}
We write
\begin{equation}
\cdot : H_i(C^{(r)}) \times H_j(C^{(s)}) \to H_{i+j}(C^{(r+s)})
\end{equation}
for the corresponding Pontryagin product on homology. In particular, the covering projection $\pi : C^{\times k} \to C^{(k)}$ corresponds to the iterated
concatenation 
\begin{equation}
\cdot : (C^{(1)})^{ \times k} \to C^{(k)}
\end{equation}
and under $\pi$ the class $[C]^{\otimes k}$ maps to a scalar multiple of the orientation class of $C^{(k)}$, i.e., 
\begin{equation} \label{eq:pullingback}
\pi_\bullet: [C]^{\otimes k} \mapsto k! \cdot C_k
\end{equation}
where $k! = \deg(\pi)$. \\

Under this product, and $H_\bullet(C^{(k)})$ can be thought of as a ''truncated Pontryagin ring". More precisely, 
\begin{proposition}
$H_*(C^{(k)})$ has generators $C_i$, $1 \leq i \leq k$ and $E_j$, $1 \leq j \leq 2g$, and all products of such classes
\begin{equation}
E_{i_1} \cdot \ldots \cdot E_{i_r} \cdot C_{s} \: , \: i_1<\ldots<i_r
\end{equation}
of length at most $k$ (i.e. $r + s \leq k$). The multiplication satisfies the rule
\begin{equation}
C_i \cdot C_j = \binom{i+j}{i} C_{i+j} 
\end{equation} \noproof
\end{proposition}
\begin{example}
For $n=2$, the homology $H_\bullet(C^{(2)})$ has generators $E_i$ (in dimension one), $C$ and all two fold products of the dimension one generators (in dimension two), $E_i \cdot C$ in dimension three and then $C \cdot C$ in dimension four.
\end{example}

\subsubsection{Divisors}
Note that there is an isomorphism between the groups of Cartier and Weil divisors on the k-fold symmetric product  (see e.g., \cite[Chapter 2.1]{MR1644323}), so in a standard abuse of notation we will not distinguish between them from here on. \\

Our task now is to recall the definition of some important divisors on $C^{(k)}$. \vspace{0.5em}
\begin{itemize}
\item
Given a point $p \in C$, we define the divisor $X_p$ as the image of the inclusion map 
\begin{equation}
X_p := \left\{p + Q \: \big| \: Q \in C^{(k-1)}\right\}
\end{equation}
We note that the numerical equivalence class of $X_p$ is independent of $p$. We denote $x \in N^1(C^{(k)})_\Z$ for this class. \vspace{0.5em}
\item
Let $\Theta$ be the theta divisor on $J(C)$ and let $\theta$ be its class in the N\'{e}ron-Severi group of the Jacobian. For simplicity, we denote the pullback $u^* \theta$ again as $\theta \in N^1(C^{(k)})_\Z$. \vspace{0.5em}
\item
We consider the diagonal map
\begin{equation}
\begin{split}
C^{(k-2)} \times C &\stackrel{d_k}{\longrightarrow} C^{(k)} \\
(Q,q) &\mapsto Q + 2q
\end{split}
\end{equation}
and we define the diagonal divisor $\Delta_k$ to be the image of this map. Then we denote by $\delta \in N^1(C^{(k)})_\Z$ the numerical equivalence class of $\Delta_k$. 
\begin{lemma} \label{lem:diagonalclass}
$\Delta_k$ is divisible by $2$, and 
\begin{equation}
\delta = 2((k+g-1)x - \theta)
\end{equation}
\end{lemma}
\begin{proof}
This is a special case of a more general computation. See \cite[Chapter VIII, Proposition 5.1]{MR770932}. 
\end{proof} \vspace{0.5em}
\item
Let $d \geq k>r \geq 1$ be some integers. \vspace{0.5em}
\begin{definition}
We define $\GG^r_d(C)$ to be the variety of linear series on $C$ of degree $d$ and dimension exactly $r$ (whose elements are called $g^r_d$'s in classical language). 
That is, given $\DD \in \GG^r_d(C)$ there exist a complete linear series $L$ of degree $d$ and a $r$-dimensional vector
space $V \subset H^0(C;L)$ such that $\DD = \P(V)$.
\end{definition}
Let $\DD$ be such a $g^r_d$ on the curve. The cycle of all divisors on $C$ that are subordinate to the linear series $\DD$ is defined to be
\begin{equation}
\Gamma_{(k)}(\DD) = \left\{P \in C^{(k)} \: \big| \: E - P \geq 0 \text{ for some } E \in \DD \right\}. 
\end{equation}
Notice that for any $\DD \in \GG^r_d(C)$, the variety $\Gamma_{r+1}(\DD)$ is a divisor on
$C^{(r+1)}$. 
\begin{lemma} \label{lem:grdclass}
The fundamental class $\gamma_k(\DD)$ of the cycle $\Gamma_{(k)}(\DD)$ in $C^{(k)}$ is given by
\begin{equation}
\gamma_k(\DD) := \sum_{i=0}^{k-r} \binom{d-g-r}{i} \frac{x^i \theta^{k-r-i}}{(k-r-i)!}. 
\end{equation}
\end{lemma}
\begin{proof}
See \cite[Chapter VIII, Lemma 3.2]{MR770932}. 
\end{proof} 
\end{itemize} \vspace{0.5em}
\begin{lemma}
In $C^{(2)}$, the intersection numbers of the classes $x$ and $\delta/2$ are 
\begin{equation} \label{eq:intersectionnumbers}
x \cdot x =1 \: , \: x \cdot (\delta/2) =1-g\:,\: (\delta/2)\cdot (\delta/2) = 1
\end{equation} \noproof
\end{lemma} 
Finally, there exists a universal divisor 
\begin{equation}
\Delta \subset C^{(k)} \times C 
\end{equation}
which, for any $D \in C^{(k)}$, cuts on
\begin{equation}
C \iso \left\{D\right\} \times C 
\end{equation}
exactly the divisor $D$ (see \cite[p. 164]{MR770932}.)
\subsubsection{Cohomology and duality}
The Poincar\'{e} dual of $\Delta$ is a class in $H^2(C^{(k)} \times C; \Z)$, which determines a map 
\begin{equation}
\mu : H_\bullet(C;\Z) \to H^\bullet(C^{(k)};\Z)
\end{equation}
The following famous result describes a presentation for the cohomology ring of $C^{(n)}$.
\begin{theorem}[Macdonald] Let $C$ be a genus $g$ Riemann surface, $pt \in H^2(C;\Z)$ the Poincare dual of a point and 
\begin{equation}
\left\{a_i,b_i\right\}_{i=1}^g \in H^1(C)
\end{equation}
the standard one dimensional generators. Then the cohomology ring $H^\bullet(C^{(n)};\Z)$ is generated by $\left\{a_i,b_j,pt\right\}$ where $i,j=1,\ldots,g$ subject to the following relations:
\begin{enumerate}
\item
All the one dimensional generators anti-commute with each other and commute with $pt$;
\item
If $1 \leq i_1, \ldots , i_\alpha, j_1, \ldots, j_\beta, k_1, \ldots , k_\gamma \leq g$ are distinct integers, then
\begin{equation}
\left(a_{i_1} \cdot \ldots \cdot a_{i_\alpha} \right) \cdot \left(b_{j_1} \cdot \ldots \cdot b_{j_\beta} \right) \cdot \left(a_{k_1} b_{k_1} -pt\right) \cdot \ldots \cdot  \left(a_{k_\gamma} b_{k_\gamma} -pt\right) \cdot (pt)^\delta = 0
\end{equation}
provided that $\alpha + \beta + 2\gamma + \delta = n + 1$. 
\end{enumerate}
Moreover, if $n \leq 2g-2$ all the relations above are consequences of those for which $\delta = 0, 1$, and if $n > 2g - 2$
all the relations are consequences of the single relation
\begin{equation}
(pt)^{n-2g+1} \cdot \prod_{i=1}^g (a_i b_i - pt) = 0.
\end{equation}
\end{theorem}

\begin{lemma}
The following relations hold: 
\begin{equation}
\begin{split}
PD(A_i) &= (pt)^{n-1} \cdot a_i , \\
PD(B_i) &= (pt)^{n-1} \cdot b_i ,\\
PD(A_i \cdot A_j) &= (pt)^{n-2} \cdot a_i \cdot a_j, \\
PD(A_i \cdot B_j) &= \begin{cases}
(pt)^{n-2} \cdot a_i \cdot b_j,		& \text{if }i \neq j\\
(pt)^{n-1} + (pt)^{n-2} \cdot a_i \cdot b_j,		& \text{if }i = j\\
\end{cases} \\
PD(B_i \cdot A_j) &= \begin{cases}
(pt)^{n-2} \cdot b_i \cdot a_j,		& \text{if }i \neq j\\
(pt)^{n-1} + (pt)^{n-2} \cdot b_i \cdot a_j,		& \text{if }i = j\\
\end{cases} \\
PD(B_i \cdot B_j) &= (pt)^{n-2} \cdot b_i \cdot b_j, \\
\end{split}
\end{equation}
with the product on the left side being the symmetric product pairing and on the right the cup product.
\end{lemma}

\end{document}